\documentclass{book}

\usepackage{etex}
\usepackage{savesym}
\usepackage{amssymb, amsmath,amsthm,amscd}
\usepackage{txfonts}
\savesymbol{lrcorner}\savesymbol{llcorner}\savesymbol{urcorner}\savesymbol{ulcorner}
\usepackage{wasysym}
\savesymbol{Sun}\savesymbol{Mercury}\savesymbol{Venus}\savesymbol{Earth}\savesymbol{Mars}\savesymbol{Jupiter}\savesymbol{Saturn}\savesymbol{Uranus}\savesymbol{Neptune}\savesymbol{Pluto}\savesymbol{leftmoon}\savesymbol{rightmoon}\savesymbol{fullmoon}\savesymbol{newmoon}\savesymbol{Aries}\savesymbol{Taurus}\savesymbol{Gemini}\savesymbol{Leo}\savesymbol{Libra}\savesymbol{Scorpio}\savesymbol{diameter}
\usepackage{mathabx}
\usepackage{setspace}
\usepackage{chngcntr}
\usepackage[tt]{titlepic}
\usepackage{enumerate,makecell}
\usepackage{makeidx,tabularx,dashbox}
\usepackage[usenames,dvipsnames]{xcolor}
\usepackage[bookmarks=true,colorlinks=true, linkcolor=MidnightBlue, citecolor=cyan]{hyperref}
\usepackage{lmodern}
\usepackage{graphicx,float}
\usepackage{multirow}
\usepackage{geometry}
\newgeometry{left=1.6in,right=1.6in,top=1.4in,bottom=1.4in}
\restoresymbol{txfonts}{lrcorner}\restoresymbol{txfonts}{llcorner}\restoresymbol{txfonts}{urcorner}\restoresymbol{txfonts}{ulcorner}

\usepackage{color}
\usepackage[all,poly,color,matrix,arrow]{xy}
\makeindex

\newcommand{\comment}[1]{}

\newcommand{\cfbox}[2]{
    \colorlet{currentcolor}{.}
    {\color{#1}
    \fbox{\color{currentcolor}#2}}
}

\def\tn{\textnormal}

\def\mc{\mathcal}
\newcommand{\qt}[1]{\tn{``}#1\tn{"}}

\def\ZZ{{\mathbb Z}}

\def\RR{{\mathbb R}}
\def\CC{{\mathbb C}}

\def\PP{{\mathbb P}}
\def\NN{{\mathbb N}}

\def\Hom{\tn{Hom}}
\def\Path{\tn{Path}}
\def\Paths{\tn{Paths}}
\def\List{\tn{List}}
\def\Aut{\tn{Aut}}
\def\im{\tn{im}}
\def\Fun{\tn{Fun}}
\def\Ob{\tn{Ob}}
\def\Skel{\tn{Skel}}
\def\Op{\tn{Open}}
\def\PK{\tn{PK}}
\def\FK{\tn{FK}}
\def\SEL*{\tn{SEL*}}

\def\hsp{\hspace{.3in}}

\newcommand{\tin}[1]{\text{\tiny #1}}

\def\singleton{\{\smiley\}}
\newcommand{\boxtitle}[1]{\begin{center}#1\end{center}\vspace{-.1in}}

\def\lcone{^\triangleleft}
\def\rcone{^\triangleright}
\def\to{\rightarrow}
\def\from{\leftarrow}
\def\down{\downarrrow}
\def\Down{\Downarrow}
\def\taking{\colon}
\def\inj{\hookrightarrow}

\def\too{\longrightarrow}

\def\tto{\rightrightarrows}

\def\ss{\subseteq}

\def\iso{\cong}
\def\down{\downarrow}
\def\|{{\;|\;}}
\def\m1{{-1}}
\def\op{^\tn{op}}

\def\la{\langle}
\def\ra{\rangle}

\def\ol{\overline}
\def\ul{\underline}
\def\plpl{+\!\!+\hspace{1pt}}
\def\acts{\lefttorightarrow}

\def\rr{\raggedright}

\newcommand{\LMO}[1]{\stackrel{#1}{\bullet}}
\newcommand{\LTO}[1]{\stackrel{\tt{#1}}{\bullet}}
\newcommand{\LA}[2]{\ar[#1]^-{\tn {#2}}}
\newcommand{\LAL}[2]{\ar[#1]_-{\tn {#2}}}
\newcommand{\obox}[3]{\stackrel{#1}{\fbox{\parbox{#2}{#3}}}}
\newcommand{\labox}[2]{\obox{#1}{1.6in}{#2}}
\newcommand{\mebox}[2]{\obox{#1}{1in}{#2}}
\newcommand{\smbox}[2]{\stackrel{#1}{\fbox{#2}}}
\newcommand{\fakebox}[1]{\tn{$\ulcorner$#1$\urcorner$}}

\def\monOb{\blacktriangle}

\def\ullimit{\ar@{}[rd]|(.3)*+{\lrcorner}}
\def\urlimit{\ar@{}[ld]|(.3)*+{\llcorner}}
\def\lllimit{\ar@{}[ru]|(.3)*+{\urcorner}}
\def\lrlimit{\ar@{}[lu]|(.3)*+{\ulcorner}}
\def\ulhlimit{\ar@{}[rd]|(.3)*+{\diamond}}
\def\urhlimit{\ar@{}[ld]|(.3)*+{\diamond}}
\def\llhlimit{\ar@{}[ru]|(.3)*+{\diamond}}
\def\lrhlimit{\ar@{}[lu]|(.3)*+{\diamond}}
\newcommand{\clabel}[1]{\ar@{}[rd]|(.5)*+{#1}}
\newcommand{\TriRight}[7]{\xymatrix{#1\ar[dr]_{#2}\ar[rr]^{#3}&&#4\ar[dl]^{#5}\\&#6\ar@{}[u] |{\Longrightarrow}\ar@{}[u]|>>>>{#7}}}
\newcommand{\TriLeft}[7]{\xymatrix{#1\ar[dr]_{#2}\ar[rr]^{#3}&&#4\ar[dl]^{#5}\\&#6\ar@{}[u] |{\Longleftarrow}\ar@{}[u]|>>>>{#7}}}
\newcommand{\TriIso}[7]{\xymatrix{#1\ar[dr]_{#2}\ar[rr]^{#3}&&#4\ar[dl]^{#5}\\&#6\ar@{}[u] |{\Longleftrightarrow}\ar@{}[u]|>>>>{#7}}}

\newcommand{\arr}[1]{\ar@<.5ex>[#1]\ar@<-.5ex>[#1]}
\newcommand{\arrr}[1]{\ar@<.7ex>[#1]\ar@<0ex>[#1]\ar@<-.7ex>[#1]}
\newcommand{\arrrr}[1]{\ar@<.9ex>[#1]\ar@<.3ex>[#1]\ar@<-.3ex>[#1]\ar@<-.9ex>[#1]}
\newcommand{\arrrrr}[1]{\ar@<1ex>[#1]\ar@<.5ex>[#1]\ar[#1]\ar@<-.5ex>[#1]\ar@<-1ex>[#1]}

\newcommand{\To}[1]{\xrightarrow{#1}}
\newcommand{\Too}[1]{\xrightarrow{\ \ #1\ \ }}
\newcommand{\From}[1]{\xleftarrow{#1}}
\newcommand{\Fromm}[1]{\xleftarrow{\ \ #1\ \ }}

\newcommand{\Adjoint}[4]{\xymatrix@1{#2 \ar@<.5ex>[r]^-{#1} & #3 \ar@<.5ex>[l]^-{#4}}}
\newcommand{\adjoint}[4]{\xymatrix{#1\taking #2\ar@<.5ex>[r]& #3\hspace{1pt}:\hspace{-2pt} #4\ar@<.5ex>[l]}}

\newcommand{\prodmap}[2]{\la#1,#2\ra}
\newcommand{\pb}[3]{\prodmap{#1}{#1}_{#3}}
\newcommand{\coprodmap}[2]{{\left\{\parbox{.1in}{#1\\#2}\right.}}
\newcommand{\po}[3]{\coprodmap{#1}{#2}}

\def\id{\tn{id}}
\def\ids{\tn{ids}}
\def\Top{{\bf Top}}
\def\Cat{{\bf Cat}}
\def\Oprd{{\bf Oprd}}

\def\Mon{{\bf Mon}}
\def\Grp{{\bf Grp}}
\def\Grph{{\bf Grph}}

\def\Supp{{\bf Supp}}
\def\Dist{{\bf Dist}}
\def\Vect{{\bf Vect}}
\def\Kls{{\bf Kls}}
\def\Prop{{\bf Prop}}
\def\FLin{{\bf FLin}}
\def\Set{{\bf Set}}
\def\Sets{{\bf Sets}}
\def\PrO{{\bf PrO}}
\def\Star{{\bf Star}}
\def\Cob{{\bf Cob}}

\def\set{{\text \textendash}{\bf Set}}
\def\sSet{{\bf sSet}}

\def\Grpd{{\bf Grpd}}

\def\bD{{\bf \Delta}}

\def\bhline{\Xhline{2\arrayrulewidth}}
\def\bbhline{\Xhline{2.5\arrayrulewidth}}

\def\colim{\mathop{\tn{colim}}}

\def\mcA{\mc{A}}
\def\mcB{\mc{B}}
\def\mcC{\mc{C}}
\def\mcD{\mc{D}}
\def\mcE{\mc{E}}
\def\mcF{\mc{F}}
\def\mcG{\mc{G}}

\def\mcK{\mc{K}}
\def\mcL{\mc{L}}
\def\mcM{\mc{M}}
\def\mcN{\mc{N}}
\def\mcO{\mc{O}}
\def\mcP{\mc{P}}

\def\mcR{\mc{R}}
\def\mcS{\mc{S}}
\def\mcT{\mc{T}}

\def\mcW{\mc{W}}
\def\mcX{\mc{X}}
\def\mcY{\mc{Y}}

\def\Loop{{\mcL oop}}
\def\LoopSchema{{\parbox{.5in}{\fbox{\xymatrix{\LMO{s}\ar@(l,u)[]^f}}}}}

\newtheorem{theorem}[subsubsection]{Theorem}
\newtheorem{lemma}[subsubsection]{Lemma}
\newtheorem{proposition}[subsubsection]{Proposition}
\newtheorem{corollary}[subsubsection]{Corollary}

\theoremstyle{remark}
\newtheorem{remark}[subsubsection]{Remark}
\newtheorem{example}[subsubsection]{Example}
\newtheorem{warning}[subsubsection]{Warning}
\newtheorem{question}[subsubsection]{Question}
\newtheorem{guess}[subsubsection]{Guess}

\newtheorem{construction}[subsubsection]{Construction}
\newtheorem{rules}[subsubsection]{Rules of good practice}
\newtheorem{exc}[subsubsection]{Exercise}
\newenvironment{exercise}{\begin{exc}}{\hspace*{\fill}$\lozenge$\end{exc}}
\newtheorem{app}[subsubsection]{Application}
\newenvironment{application}{\begin{app}}{\hspace*{\fill}$\lozenge\lozenge$\end{app}}

\newenvironment{slogan}{\addtocounter{subsubsection}{1}\vspace{.1in}\begin{sloppypar}\noindent{\em Slogan}\;\arabic{chapter}.\arabic{section}.\arabic{subsection}.\arabic{subsubsection}. \begin{quote}``\slshape}{"\end{quote}\end{sloppypar}\vspace{.1in}}

\makeatletter\let\c@figure\c@equation\makeatother

\theoremstyle{definition}
\newtheorem{definition}[subsubsection]{Definition}
\newtheorem{notation}[subsubsection]{Notation}

\def\Sch{{\bf Sch}}
\def\Fin{{\bf Fin}}

\newcommand{\MainCatLarge}[1]{ 
	\stackrel{#1}{
		\parbox{4.5in}{\fbox{\parbox{4.4in}{\begin{center}\underline{{\tt Employee} manager worksIn $\simeq$ {\tt Employee} worksIn}\hsp  \underline{{\tt Department} secretary worksIn $\simeq$ {\tt Department}}\end{center}~\\\\\\
			\xymatrix@=8pt{&\LTO{Employee}\ar@<.5ex>[rrrrr]^{\tn{worksIn}}\ar@(l,u)[]+<5pt,10pt>^{\tn{manager}}\ar[dddl]_{\tn{first}}\ar[dddr]^{\tn{last}}&&&&&\LTO{Department}\ar@<.5ex>[lllll]^{\tn{secretary}}\ar[ddd]^{\tn{name}}\\\\\\\LTO{FirstNameString}&&\LTO{LastNameString}&~&~&~&\LTO{DepartmentNameString}
			}
		}}}
	}
}
\def\GrIn{{\bf GrIn}}
\def\GrInSchema{\xymatrix{\LMO{Ar}\ar@<.5ex>[r]^{src}\ar@<-.5ex>[r]_{tgt}&\LMO{V\!e}}}
\pdfoutput=1

\def\sub{\begin{itemize}\item}
\def\sexc{\begin{enumerate}[a.)]\setlength{\itemsep}{.1cm}\setlength{\parskip}{.1cm}\item}
\def\next{\item}
\def\endsub{\end{itemize}}
\def\endsexc{\end{enumerate}}

\begin{document}

\title{~\\~\\Category Theory for Scientists\\(Old Version)}
\author{David I. Spivak}
\titlepic{\vspace{1.3in}\dashbox{\includegraphics[width=.8\textwidth]{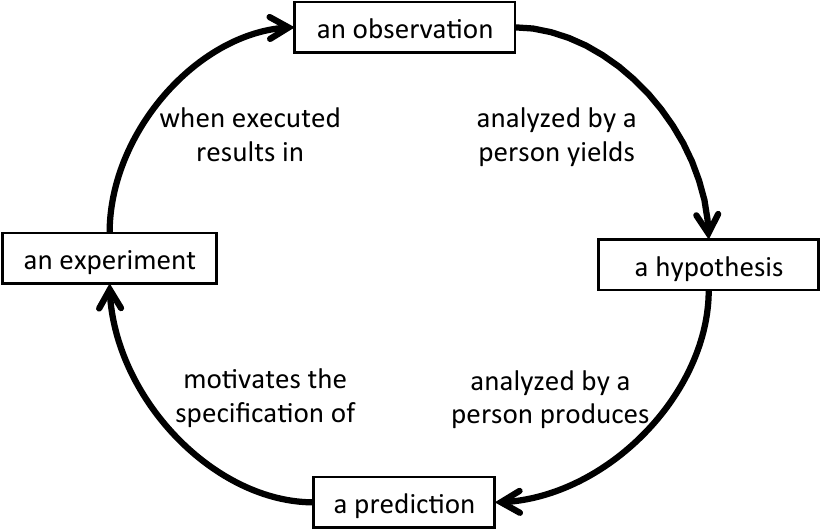}}\\\vspace{.3in}\Large How can mathematics make this diagram meaningful?
\normalsize}
\maketitle

\chapter*{Preface}

An early version of this book was put on line in February 2013 to serve as the textbook for my course \href{http://math.mit.edu/~dspivak/teaching/sp13/}{\text Category Theory for Scientists} taught in the spring semester of 2013 at MIT. During that semester, students provided me with hundreds of comments and questions, which led to a substantial improvement (and the addition of 50 pages) to the original document. 

In the summer of 2013 I signed a contract with the MIT Press to publish a new version of this work under the title {\em Category Theory for the Sciences}. Because I am committed to the open source development model I insisted that a version of this book, namely the one you are reading, remain freely available online. The MIT Press version will of course not be free.

Other than the title, there are two main differences between the present version and the MIT Press version. The first difference is that I will do a full edit with the help of professional editors from the Press. The second difference is that I will write up solutions to the book's (approximately 280) exercises; some of these will be included in the published version, whereas the rest will be available by way of a password-protected page, accessible only to professors who teach the subject.

\tableofcontents


\chapter{Introduction}

The title page of this book contains a graphic that we reproduce here. 
\begin{align}\label{dia:scientific method}
\dashbox{\includegraphics[width=.7\textwidth]{ScientificMethod}}
\end{align}
It is intended to evoke thoughts of the scientific method. 
\begin{quote}
A hypothesis analyzed by a person produces a prediction, which motivates the specification of an experiment, which when executed results in an observation, which analyzed by a person yields a hypothesis.
\end{quote}
This sounds valid, and a good graphic can be exceptionally useful for leading a reader through the story that the author wishes to tell. 

Interestingly, a graphic has the power to evoke feelings of understanding, without really meaning much. The same is true for text: it is possible to use a language such as English to express ideas that are never made rigorous or clear. When someone says ``I believe in free will," what does she believe in? We may all have some concept of what she's saying---something we can conceptually work with and discuss or argue about. But to what extent are we all discussing the same thing, the thing she intended to convey?

Science is about agreement. When we supply a convincing argument, the result of this convincing is agreement. When, in an experiment, the observation matches the hypothesis---success!---that is agreement. When my methods make sense to you, that is agreement. When practice does not agree with theory, that is disagreement. Agreement is the good stuff in science; it's the high fives.

But it is easy to think we're in agreement, when really we're not. Modeling our thoughts on heuristics and pictures may be convenient for quick travel down the road, but we're liable to miss our turnoff at the first mile. The danger is in mistaking our convenient conceptualizations for what's actually there. It is imperative that we have the ability at any time to ground out in reality. What does that mean?

Data. Hard evidence. The physical world. It is here that science touches down and heuristics evaporate. So let's look again at the diagram on the cover. It is intended to evoke an idea of how science is performed. Is there hard evidence and data to back this theory up? Can we set up an experiment to find out whether science is actually performed according to such a protocol? To do so we have to shake off the stupor evoked by the diagram and ask the question: ``what does this diagram intend to communicate?"

In this course I will use a mathematical tool called {\em ologs}, or ontology logs, to give some structure to the kinds of ideas that are often communicated in pictures like the one on the cover. Each olog inherently offers a framework in which to record data about the subject. More precisely it encompasses a {\em database schema}, which means a system of interconnected tables that are initially empty but into which data can be entered. For example consider the olog below
$$\xymatrix{
\obox{}{.5in}{a mass}&&\obox{}{1.1in}{an object of mass $m$ held at height $h$ above the ground}\LAL{ll}{\footnotesize has as mass}\LA{rrdd}{\hspace{.4in}\parbox{1in}{\singlespace \footnotesize when dropped has as number of seconds till hitting the ground}}\LAL{dd}{\parbox{.7in}{\singlespace\footnotesize has as height in meters}}&&\\\\
&&\obox{}{1in}{a real number $h$}\ar@{}[uurr]|(.35){?}\ar[rr]_-{\sqrt{2h\div 9.8}}&\hspace{.3in}&\obox{}{.9in}{a real number}
}
$$
This olog represents a framework in which to record data about objects held above the ground, their mass, their height, and a comparison (the ?-mark in the middle) between the number of seconds till they hit the ground and a certain real-valued function of their height. We will discuss ologs in detail throughout this course. 

The picture in (\ref{dia:scientific method}) looks like an olog, but it does not conform to the rules that we lay out for ologs in Section \ref{sec:ologs}. In an olog, every arrow is intended to represent a mathematical function. It is difficult to imagine a function that takes in predictions and outputs experiments, but such a function is necessary in order for the arrow
$$\fbox{a prediction}\To{\tn{motivates the specification of}}\fbox{an experiment}
$$
in (\ref{dia:scientific method}) to make sense. To produce an experiment design from a prediction probably requires an expert, and even then the expert may be motivated to specify a different experiment on Tuesday than he is on Monday. But perhaps our criticism has led to a way forward: if we say that every arrow represents a function {\em when in the context of a specific expert who is actually doing the science at a specific time}, then Figure (\ref{dia:scientific method}) begins to make sense. In fact, we will return to the figure in Section \ref{sec:monads} (specifically Example \ref{ex:scientific method}), where background methodological context is discussed in earnest.

This course is an attempt to extol the virtues of a new branch of mathematics, called {\em category theory}, which was invented for powerful communication of ideas between different fields and subfields within mathematics. By powerful communication of ideas I actually mean something precise. Different branches of mathematics can be formalized into categories. These categories can then be connected together by functors. And the sense in which these functors provide powerful communication of ideas is that facts and theorems proven in one category can be transferred through a connecting functor to yield proofs of analogous theorems in another category. A functor is like a conductor of mathematical truth.

I believe that the language and toolset of category theory can be useful throughout science. We build scientific understanding by developing models, and category theory is the study of basic conceptual building blocks and how they cleanly fit together to make such models. Certain structures and conceptual frameworks show up again and again in our understanding of reality. No one would dispute that vector spaces are ubiquitous. But so are hierarchies, symmetries, actions of agents on objects, data models, global behavior emerging as the aggregate of local behavior, self-similarity, and the effect of methodological context. 

Some ideas are so common that our use of them goes virtually undetected, such as set-theoretic intersections. For example, when we speak of a material that is both lightweight and ductile, we are intersecting two sets. But what is the use of even mentioning this set-theoretic fact? The answer is that when we formalize our ideas, our understanding is almost always clarified. Our ability to communicate with others is enhanced, and the possibility for developing new insights expands. And if we are ever to get to the point that we can input our ideas into computers, we will need to be able to formalize these ideas first.

It is my hope that this course will offer scientists a new vocabulary in which to think and communicate, and a new pipeline to the vast array of theorems that exist and are considered immensely powerful within mathematics. These theorems have not made their way out into the world of science, but they are directly applicable there. Hierarchies are partial orders, symmetries are group elements, data models are categories, agent actions are monoid actions, local-to-global\index{local-to-global} principles are sheaves, self-similarity is modeled by operads, context can be modeled by monads.


\section{A brief history of category theory}

The paradigm shift brought on by Einstein's theory of relativity brought on the realization that there is no single perspective from which to view the world. There is no background framework that we need to find; there are infinitely many different frameworks and perspectives, and the real power lies in being able to translate between them. It is in this historical context that category theory got its start.
\footnote{The following history of category theory is far too brief, and perhaps reflects more of the author's aesthetic than any kind of objective truth, whatever that may mean. Here are some much better references: \cite{Kro}, \cite{Mar1}, \cite{LM}.}

Category theory was invented in the early 1940s by Samuel Eilenberg\index{Eilenberg, Samuel} and Saunders Mac Lane.\index{Mac Lane, Saunders} It was specifically designed to bridge what may appear to be two quite different fields: topology and algebra. Topology is the study of abstract shapes such as 7-dimensional spheres; algebra is the study of abstract equations such as $y^2z=x^3-xz^2$. People had already created important and useful links (e.g. cohomology theory) between these fields, but Eilenberg and Mac Lane needed to precisely compare different links with one another. To do so they first needed to boil down and extract the fundamental nature of these two fields. But the ideas they worked out amounted to a framework that fit not only topology and algebra, but many other mathematical disciplines as well.

At first category theory was little more than a deeply clarifying language for existing difficult mathematical ideas. However, in 1957 Alexander Grothendieck\index{Grothendieck!in history} used category theory to build new mathematical machinery (new cohomology theories) that granted unprecedented insight into the behavior of algebraic equations. Since that time, categories have been built specifically to zoom in on particular features of mathematical subjects and study them with a level of acuity that is simply unavailable elsewhere.

Bill Lawvere\index{Lawvere, William} saw category theory as a new foundation for all mathematical thought. Mathematicians had been searching for foundations in the 19th century and were reasonably satisfied with set theory as {\em the foundation}. But Lawvere showed that the category of sets is simply a category with certain nice properties, not necessarily the center of the mathematical universe. He explained how whole algebraic theories can be viewed as examples of a single system. He and others went on to show that higher order logic was beautifully captured in the setting of category theory (more specifically toposes). It is here also that Grothendieck and his school worked out major results in algebraic geometry. 

In 1980 Joachim Lambek\index{Lambek, Joachim} showed that the types and programs used in computer science form a specific kind of category. This provided a new semantics for talking about programs, allowing people to investigate how programs combine and compose to create other programs, without caring about the specifics of implementation. Eugenio Moggi\index{Moggi, Eugenio} brought the category theoretic notion of monads into computer science to encapsulate ideas that up to that point were considered outside the realm of such theory.

It is difficult to explain the clarity and beauty brought to category theory by people like Daniel Kan\index{Kan, Daniel} and Andr\'{e} Joyal\index{Joyal, Andr\'{e}}. They have each repeatedly extracted the essence of a whole mathematical subject to reveal and formalize a stunningly simple yet extremely powerful pattern of thinking, revolutionizing how mathematics is done. 

All this time, however, category theory was consistently seen by much of the mathematical community as ridiculously abstract. But in the 21st century it has finally come to find healthy respect within the larger community of pure mathematics. It is the language of choice for graduate-level algebra and topology courses, and in my opinion will continue to establish itself as the basic framework in which mathematics is done.

As mentioned above category theory has branched out into certain areas of science as well. Baez\index{Baez, John} and Dolan\index{Dolan, James} have shown its value in making sense of quantum physics, it is well established in computer science, and it has found proponents in several other fields as well. But to my mind, we are the very beginning of its venture into scientific methodology. Category theory was invented as a bridge and it will continue to serve in that role. 


\section{Intention of this book}

The world of {\em applied mathematics} is much smaller than the world of {\em applicable mathematics}. As alluded to above, this course is intended to create a bridge between the vast array of mathematical concepts that are used daily by mathematicians to describe all manner of phenomena that arise in our studies, and the models and frameworks of scientific disciplines such as physics, computation, and neuroscience. 

To the pure mathematician I'll try to prove that concepts such as categories, functors, natural transformations, limits, colimits, functor categories, sheaves, monads, and operads---concepts that are often considered too abstract for even math majors---can be communicated to scientists with no math background beyond linear algebra. If this material is as teachable as I think, it means that category theory is not esoteric but somehow well-aligned with ideas that already make sense to the scientific mind. Note, however, that this book is example-based rather than proof-based, so it may not be suitable as a reference for students of pure mathematics.

To the scientist I'll try to prove the claim that category theory includes a formal treatment of conceptual structures that the scientist sees often, perhaps without realizing that there is well-oiled mathematical machinery to be employed. We will work on the structure of information; how data is made meaningful by its connections, both internal and outreaching, to other data. Note, however, that this book should most certainly not be taken as a reference on scientific matters themselves. One should assume that any account of physics, materials science, chemistry, etc. has been oversimplified.\index{a warning!oversimplified science} The intention is to give a flavor of how category theory may help us model scientific ideas, not to explain these ideas in a serious way. 

Data gathering is ubiquitous in science. Giant databases are currently being mined for unknown patterns, but in fact there are many (many) known patterns that simply have not been catalogued. Consider the well-known case of medical records. A patient's medical history is often known by various individual doctor-offices but quite inadequately shared between them. Sharing medical records often means faxing a hand-written note or a filled-in house-created form between offices. 

Similarly, in science there exists substantial expertise making brilliant connections between concepts, but it is being conveyed in silos of English prose known as journal articles. Every scientific journal article has a methods section, but it is almost impossible to read a methods section and subsequently repeat the experiment---the English language is inadequate to precisely and concisely convey what is being done.

The first thing to understand in this course is that reusable methodologies can be formalized, and that doing so is inherently valuable. Consider the following analogy. Suppose you want to add up the area of a region in space (or the area under a curve). You break the region down into small squares, each of which you know has area $A$; then you count the number of squares, say $n$, and the result is that the region has an area of about $nA$. If you want a more precise and accurate result you repeat the process with half-size squares. This methodology can be used for any area-finding problem (of which there are more than a first-year calculus student generally realizes) and thus it deserves to be formalized. But once we have formalized this methodology, it can be taken to its limit and out comes integration by Riemann sums. 

I intend to show that category theory is incredibly efficient as a language for experimental design patterns, introducing formality while remaining flexible. It forms a rich and tightly woven conceptual fabric that will allow the scientist to maneuver between different perspectives whenever the need arises. Once one builds that fabric for oneself, he or she has an ability to think about models in a way that simply would not occur without it.  Moreover, putting ideas into the language of category theory forces a person to clarify their assumptions. This is highly valuable both for the researcher and for his or her audience.

What must be recognized in order to find value in this course is that conceptual chaos is a major problem. Creativity demands clarity of thinking, and to think clearly about a subject requires an organized understanding of how its pieces fit together. Organization and clarity also lead to better communication with others. Academics often say they are paid to think and understand, but that is not true. They are paid to think, understand, and {\em communicate their findings}. Universal languages for science---languages such as calculus and differential equations, matrices, or simply graphs and pie-charts---already exist, and they grant us a cohesiveness that makes scientific research worthwhile. In this book I will attempt to show that category theory can be similarly useful in describing complex scientific understandings.


\section{What is requested from the student}

I will do my best to make clear the value of category theory in science, but I am not a scientist. To that end I am asking for your help in exploring how category theory may be useful in your specific field. 

I also want you to recognize that the value of mathematics is not generally obvious at first. A good student learning a good subject with a good teacher will see something compelling almost immediately, but may not see how it will be useful in real life. This will come later. I hope you will work hard to understand even without yet knowing what its actual value in your life and research will be. Like a student of soccer is encouraged to spend hours juggling the ball when he or she could be practicing penalty shots, it is important to gain facility with the materials you will be using. Doing exercises is imperative for learning mathematics.


\section{Category theory references}

I wrote this book because the available books on category theory are almost all written for mathematicians (the rest are written for computer scientists). There is one book by Lawvere and Schanuel, called {\em Conceptual Mathematics} \cite{LS}, that offers category theory to a wider audience, but its style is not appropriate for this course. Still, it is very well written and clear. 

The ``bible" of category theory is {\em Categories for the working mathematician} by Mac Lane \cite{Mac}. But as the title suggests, it was written for working mathematicians and will be quite opaque to my target audience. However, once a person has read my book, Mac Lane's book may become a valuable reference. 

Other good books include Steve Awodey's book {\em Category theory} \cite{Awo} and Barr and Wells book {\em Category theory for computing science}, \cite{BW}.  A paper by Brown and Porter called  \href{http://pages.bangor.ac.uk/\%7Emas010/pdffiles/Analogy-and-Comparison.pdf}{\text Category Theory: an abstract setting for analogy and comparison} \cite{BP1} is more in line with the style of this book, only much shorter. Online, I find \href{http://www.wikipedia.org}{\text wikipedia} and a site called \href{http://ncatlab.org/nlab/show/HomePage}{\em the $n$lab} to be quite useful. 

This book attempts to explain category theory by examples and exercises rather than by theorems and proofs. I hope this approach will be valuable to the working scientist.

\section{Acknowledgments}

I would like to express my deep appreciation for the many scientists who I have worked with over the past five years. It all started with Paea LePendu who first taught me about databases when I was naively knocking on doors in the University of Oregon computer science department. This book would never have been written if Tristan Nguyen and Dave Balaban had not noticed my work and encouraged me to continue. Dave Balaban and Peter Gates have been my scientific partners since the beginning, working hard to understand what I'm offering and working just as hard to help me understand all that I'm missing. Peter Gates has deepened my understanding of data in profound ways. 

I have also been tremendously lucky to know Haynes Miller, who made it possible for me set down at MIT, with the help of Clark Barwick and Jacob Lurie. I knew that MIT would be the best place in the world for me to pursue this type of research, and it has really come through. Researchers like Markus Buehler and his graduate students Tristan Giesa and Dieter Brommer have been a pleasure to work with, and the many materials science examples scattered throughout this book is a testament to how much our work together has influenced my thinking. 

I'd also like to thank my collaborators and conversation partners with whom I have discussed subjects written about in this book. Other than people mentioned above, these include Steve Awodey, Allen Brown, Adam Chlipala, Carlo Curino, Dan Dugger, Henrik Forssell, David Gepner, Jason Gross, Bob Harper, Ralph Hutchison, Robert Kent, Jack Morava, Scott Morrison, David Platt, Joey Perricone, Dylan Rupel, Guarav Singh, Sam Shames, Nat Stapleton, Patrick Schultz, Ka Yu Tam, Ryan Wisnesky, Jesse Wolfson, and Elizabeth Wood.

I would like to thank Peter Kleinhenz and Peter Gates for reading this book and providing invaluable feedback before I began teaching the 18-S996 class at MIT in Spring 2013. In particular the cover image is a mild alteration of something Gates sent me to help motivate the book to scientists. I would also like to greatly thank the 18-S996 course grader Darij Grinberg, who was not only the best grader I've had in my 14 years of teaching, but gave me more comments than anyone else on the book itself. I'd also like to thank the students from the 18-S996 class at MIT who helped me find typos, pointed me to unclear explanations, and generally helped me improve the book in many ways. Other than the people listed above, these include Aaron Brookner, Leon Dimas, Dylan Erb, Deokhwan Kim, Taesoo Kim, Owen Lewis, Yair Shenfeld, and Adam Strandberg.

I would like to thank my teacher, Peter Ralston, who taught me to repeatedly question the obvious. My ability to commit to a project like this one and to see it to fruition has certainly been enhanced since studying with him.

Finally, I acknowledge my appreciation for support from the Office of Naval Research
\footnote{Grant numbers: N000140910466, N000141010841, N000141310260} 
without which this book would not have been remotely possible. I believe that their funding of basic research is an excellent way of ensuring that the US remains a global leader in the years to come.


\chapter{The category of sets}\label{chap:sets}

The theory of sets was invented as a foundation for all of mathematics. The notion of sets and functions serves as a basis on which to build our intuition about categories in general. In this chapter we will give examples of sets and functions and then move on to discuss commutative diagrams. At this point we can introduce ologs which will allow us to use the language of category theory to speak about real world concepts. Then we will introduce limits and colimits, and their universal properties. All of this material is basic set theory, but it can also be taken as an investigation of our first category, the {\em category of sets}, which we call $\Set$. We will end this chapter with some other interesting constructions in $\Set$ that do not fit into the previous sections.


\section{Sets and functions}\index{set} 


\subsection{Sets}

In this course I'll assume you know what a set is. We can think of a set $X$ as a collection of things $x\in X$, each of which is recognizable as being in $X$ and such that for each pair of named elements $x,x'\in X$ we can tell if $x=x'$ or not.
\footnote{Note that the symbol $x'$, read ``x-prime", has nothing to do with calculus or derivatives. It is simply notation that we use to name a symbol that is suggested as being somehow like $x$. This suggestion of kinship between $x$ and $x'$ is meant only as an aid for human cognition, and not as part of the mathematics.}
The set of pendulums is the collection of things we agree to call pendulums, each of which is recognizable as being a pendulum, and for any two people pointing at pendulums we can tell if they're pointing at the same pendulum or not. 

\begin{figure}
\begin{center}
\includegraphics[height=2in]{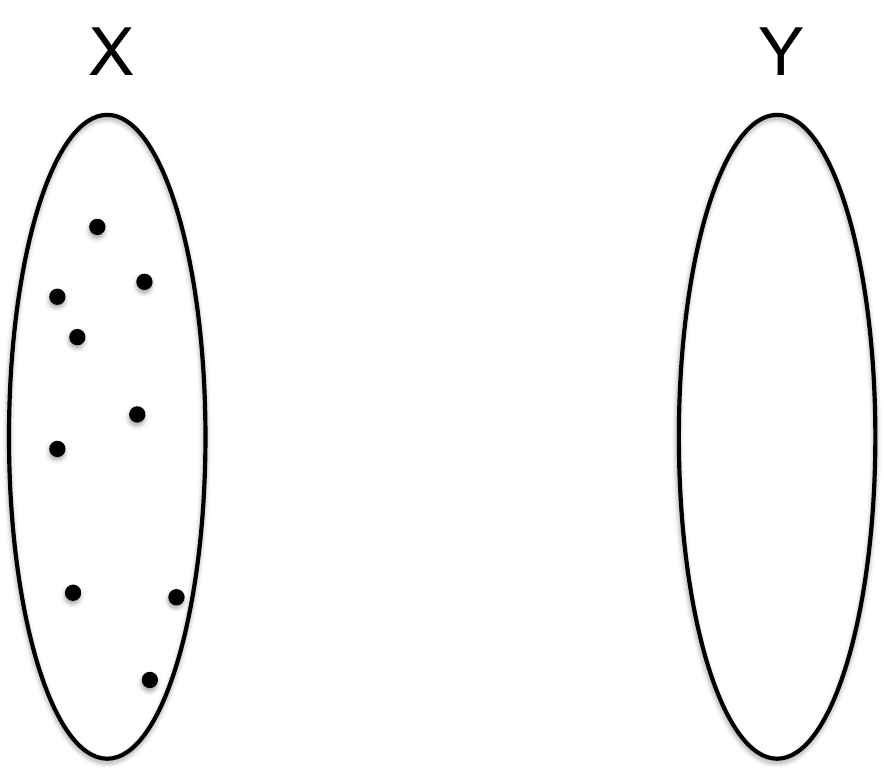}
\end{center}
\caption{A set $X$ with $9$ elements and a set $Y$ with no elements, $Y=\emptyset$.}
\end{figure}

\begin{notation}\label{not:basic math notation}

The symbol $\emptyset$\index{a symbol!$\emptyset$} denotes the set with no elements. The symbol $\NN$\index{a symbol!$\NN$} denotes the set of natural numbers, which we can write as 
$$\NN:=\{0,1,2,3,4,\ldots,877,\ldots\}.$$
The symbol $\ZZ$\index{a symbol!$\ZZ$} denotes the set of integers, which contains both the natural numbers and their negatives, 
$$\ZZ:=\{\ldots,-551,\ldots,-2,-1,0,1,2,\ldots\}.$$ 

If $A$ and $B$ are sets, we say that $A$ is a {\em subset}\index{subset} of $B$, and write $A\ss B$, if every element of $A$ is an element of $B$. So we have $\NN\ss\ZZ$. Checking the definition, one sees that for any set $A$, we have (perhaps uninteresting) subsets $\emptyset\ss A$ and $A\ss A$. We can use {\em set-builder notation}\index{set!set builder notation} to denote subsets. For example the set of even integers can be written $\{n\in\ZZ\|n\tn{ is even}\}$. The set of integers greater than $2$ can be written in many ways, such as $$\{n\in\ZZ\|n>2\} \hsp\tn{or}\hsp\{n\in\NN\|n>2\}\hsp\tn{or}\hsp\{n\in\NN\|n\geq 3\}.$$

The symbol $\exists$ means ``there exists".\index{a symbol!$\exists$} So we could write the set of even integers as $$\{n\in\ZZ\|n\tn{ is even}\}\hsp=\hsp\{n\in\ZZ\|\exists m\in\ZZ\tn{ such that } 2m=n\}.$$ The symbol $\exists!$\index{a symbol!$\exists$"!} means ``there exists a unique". So the statement ``$\exists! x\in\RR\tn{ such that } x^2=0$" means that there is one and only one number whose square is 0. Finally, the symbol $\forall$ means ``for all".\index{a symbol!$\forall$} So the statement ``$\forall m\in\NN\;\exists n\in\NN\tn{ such that } m<n$" means that for every number there is a bigger one.

As you may have noticed, we use the colon-equals notation `` $A:=XYZ$ " to mean something like ``define $A$ to be $XYZ$".\index{a symbol!:=} That is, a colon-equals declaration is not denoting a fact of nature (like $2+2=4$), but a choice of the speaker. It just so happens that the notation above, such as $\NN:=\{0,1,2,\ldots\}$, is a widely-held choice.

\end{notation}

\begin{exercise}
Let $A=\{1,2,3\}$. What are all the subsets of $A$? Hint: there are 8.
\end{exercise}


\subsection{Functions}\label{sec:functions}

If $X$ and $Y$ are sets, then a {\em function $f$ from $X$ to $Y$},\index{function} denoted $f\taking X\to Y$, is a mapping that sends each element $x\in X$ to an element of $Y$, denoted $f(x)\in Y$. We call $X$ the {\em domain}\index{function!domain} of the function $f$ and we call $Y$ the {\em codomain}\index{function!codomain} of $f$. 

\begin{align}\label{dia:setmap}
\parbox{2.3in}{\includegraphics[height=2in]{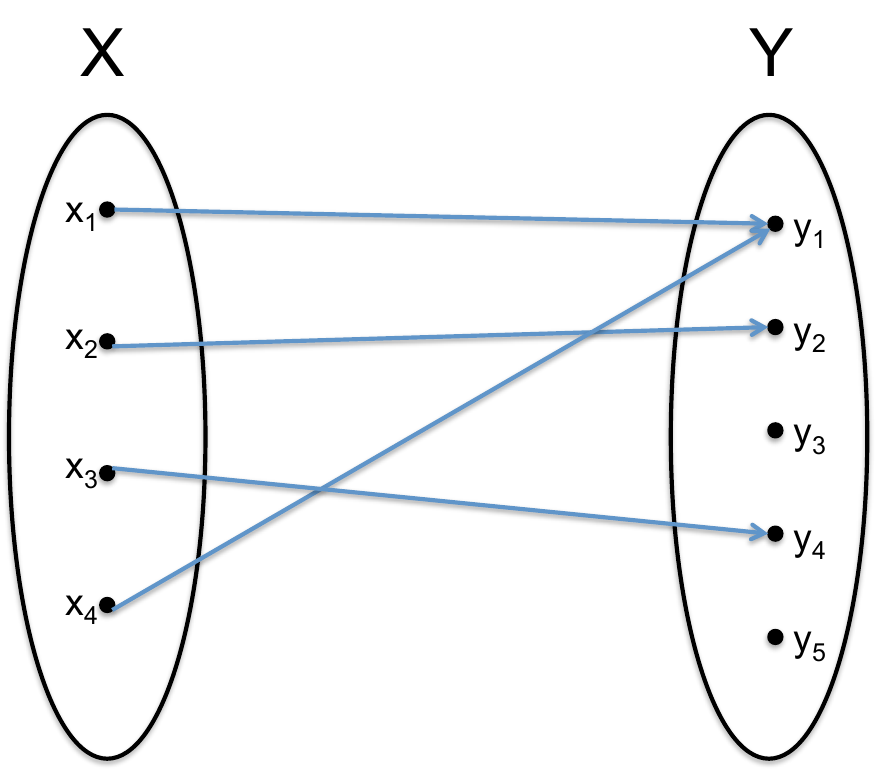}}
\end{align}

Note that for every element $x\in X$, there is exactly one arrow emanating from $x$, but for an element $y\in Y$, there can be several arrows pointing to $y$, or there can be no arrows pointing to $y$. 

\begin{application}\label{app:force-extension}\index{materials!force-extension curves}

In studying the mechanics of materials, one wishes to know how a material responds to tension. For example a rubber band responds to tension differently than a spring does. To each material we can associate a \href{http://en.wikipedia.org/wiki/StressÐstrain_curve}{\text force-extension curve}, recording how much force the material carries when extended to various lengths. Once we fix a methodology for performing experiments, finding a material's force-extension curve would ideally constitute a function from the set of materials to the set of curves.
\footnote{In reality, different samples of the same material, say samples of different sizes or at different temperatures, may have different force-extension curves. If we want to see this as a true function whose codomain is curves it should have as domain something like the set of material samples.}

\end{application}

\begin{exercise}

Here is a simplified account of how the \href{http://en.wikipedia.org/wiki/Retina}{\text brain receives light}. The eye contains about 100 million photoreceptor (PR) cells. Each connects to a retinal ganglion (RG) cell. No PR cell connects to two different RG cells, but usually many PR cells can attach to a single RG cell. 

Let $PR$ denote the set of photoreceptor cells and let $RG$ denote the set of retinal ganglion cells. 
\sexc According to the above account, does the connection pattern constitute a function $RG\to PR$, a function $PR\to RG$ or neither one? 
\next Would you guess that the connection pattern that exists between other areas of the brain are ``function-like"?
\endsexc
\end{exercise}

\begin{example}\label{ex:subset as function}\index{subset!as function}

Suppose that $X$ is a set and $X'\ss X$ is a subset. Then we can consider the function $X'\to X$ given by sending every element of $X'$ to ``itself" as an element of $X$. For example if $X=\{a,b,c,d,e,f\}$ and $X'=\{b,d,e\}$ then $X'\ss X$ and we turn that into the function $X'\to X$ given by $b\mapsto b, d\mapsto d, e\mapsto e$.
\footnote{This kind of arrow,\;\;$\mapsto$\;\;, is read aloud as ``maps to". A function $f\taking X\to Y$ means a rule for assigning to each element $x\in X$ an element $f(x)\in Y$. We say that ``$x$ maps to $f(x)$" and write $x\mapsto f(x)$.}\index{a symbol!$\mapsto$}

As a matter of notation, we may sometimes say something like the following: Let $X$ be a set and let $i\taking X'\ss X$ be a subset. Here we are making clear that $X'$ is a subset of $X$, but that $i$ is the name of the associated function.

\end{example}

\begin{exercise}
Let $f\taking\NN\to\NN$ be the function that sends every natural number to its square, e.g. $f(6)=36$. First fill in the blanks below, then answer a question.
\sexc $2\mapsto\ul{\hspace{.5in}}$
\next $0\mapsto\ul{\hspace{.5in}}$
\next $-2\mapsto\ul{\hspace{.5in}}$
\next $5\mapsto\ul{\hspace{.5in}}$
\next Consider the symbol $\to$ and the symbol $\mapsto$. What is the difference between how these two symbols are used in this book?
\endsexc
\end{exercise}

Given a function $f\taking X\to Y$, the elements of $Y$ that have at least one arrow pointing to them are said to be {\em in the image} of $f$; that is we have \index{image}
\begin{align}\label{dia:image}
\im(f):=\{y\in Y\| \exists x\in X \tn{ such that } f(x)=y\}.
\end{align} 

\begin{exercise}
If $f\taking X\to Y$ is depicted by (\ref{dia:setmap}) above, write its image, $\im(f)$ as a set.
\end{exercise}

Given a function $f\taking X\to Y$ and a function $g\taking Y\to Z$, where the codomain of $f$ is the same set as the domain of $g$ (namely $Y$), we say that $f$ and $g$ are composable 
$$X\Too{f}Y\Too{g}Z.$$ The {\em composition of $f$ and $g$}\label{function composition}\index{function!composition}\index{composition!of functions}\index{a symbol!$\circ$} is denoted by $g\circ f\taking X\to Z$. 

\begin{figure}[h]
\begin{center}
\includegraphics[height=2in]{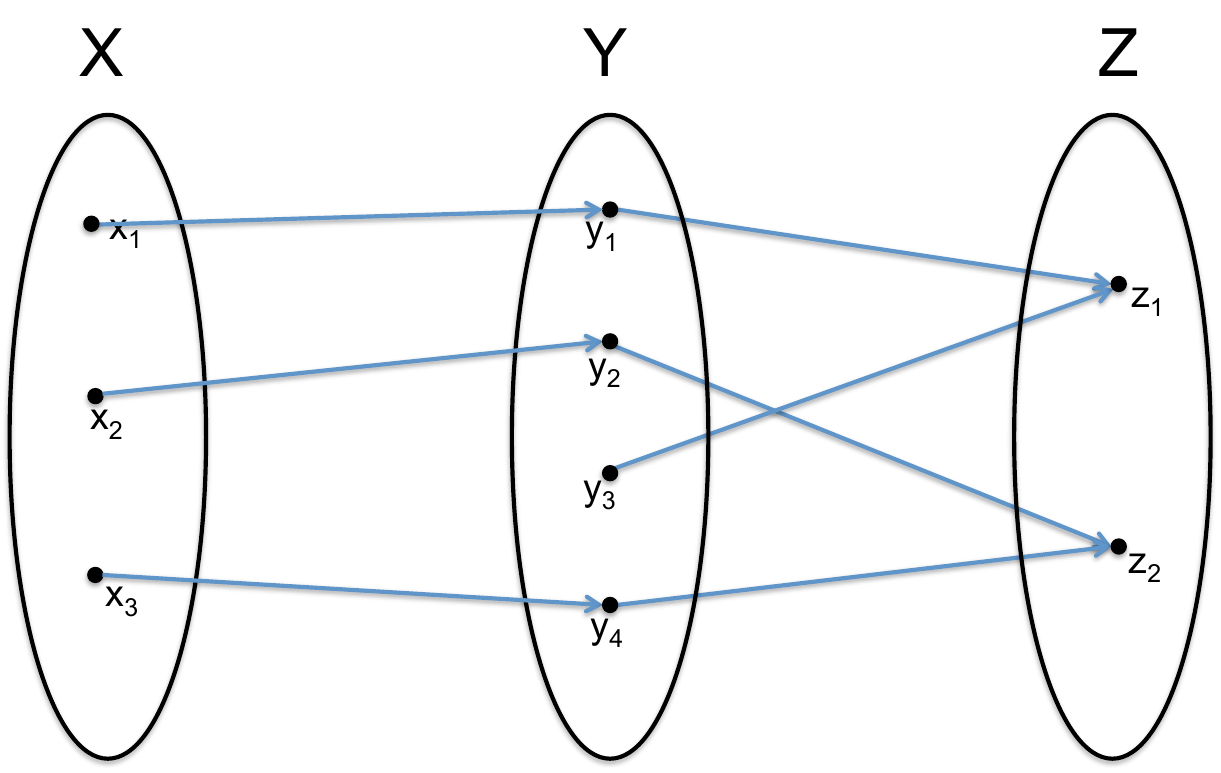}
\end{center}
\caption{Functions $f\taking X\to Y$ and $g\taking Y\to Z$ compose to a function $g\circ f\taking X\to Z$; just follow the arrows.}
\end{figure}

Let $X$ and $Y$ be sets. We write $\Hom_\Set(X,Y)$\index{a symbol!$\Hom_\Set$} to denote the set of functions $X\to Y$.
\footnote{The strange notation $\Hom_\Set(-,-)$ will make more sense later, when it is seen as part of a bigger story.} 
Note that two functions $f,g\taking X\to Y$ are equal\index{function!equality of} if and only if for every element $x\in X$ we have $f(x)=g(x)$. 

\begin{exercise}
Let $A=\{1,2,3,4,5\}$ and $B=\{x,y\}.$ 
\sexc How many elements does $\Hom_\Set(A,B)$ have? 
\next How many elements does $\Hom_\Set(B,A)$ have?
\endsexc
\end{exercise}

\begin{exercise}~
\sexc Find a set $A$ such that for all sets $X$ there is exactly one element in $\Hom_\Set(X,A)$. Hint: draw a picture of proposed $A$'s and $X$'s.
\next Find a set $B$ such that for all sets $X$ there is exactly one element in $\Hom_\Set(B,X)$.
\endsexc 
\end{exercise}

For any set $X$, we define the {\em identity function on $X$}\index{function!identity}, denoted $\id_X\taking X\to X$, to be the function such that for all $x\in X$ we have $\id_X(x)=x$.\index{a symbol!$\id_X$}

\begin{definition}[Isomorphism]\label{def:iso in set}

Let $X$ and $Y$ be sets. A function $f\taking X\to Y$ is called an {\em isomorphism}\index{function!isomorphism}\index{isomorphism!of sets}, denoted $f\taking X\To{\iso}Y$, if there exists a function $g\taking Y\to X$ such that $g\circ f=\id_X$ and $f\circ g=\id_Y$. We also say that $f$ is {\em invertible} and we say that $g$ is {\em the inverse}\index{function!inverse} of $f$. If there exists an isomorphism $X\To\iso Y$ we say that $X$ and $Y$ are {\em isomorphic} sets and may write $X\iso Y$. \index{a symbol!$\iso$}

\end{definition}

\begin{example}

If $X$ and $Y$ are sets and $f\taking X\to Y$ is an isomorphism then the analogue of Diagram \ref{dia:setmap} will look like a perfect matching, more often called a {\em one-to-one correspondence}\index{one-to-one correspondence}\index{correspondence!one-to-one}. That means that no two arrows will hit the same element of $Y$, and every element of $Y$ will be in the image. For example, the following depicts an isomorphism $X\To{\iso}Y$.

\begin{align}\label{dia:setmapiso}
\parbox{2.3in}{\includegraphics[height=2in]{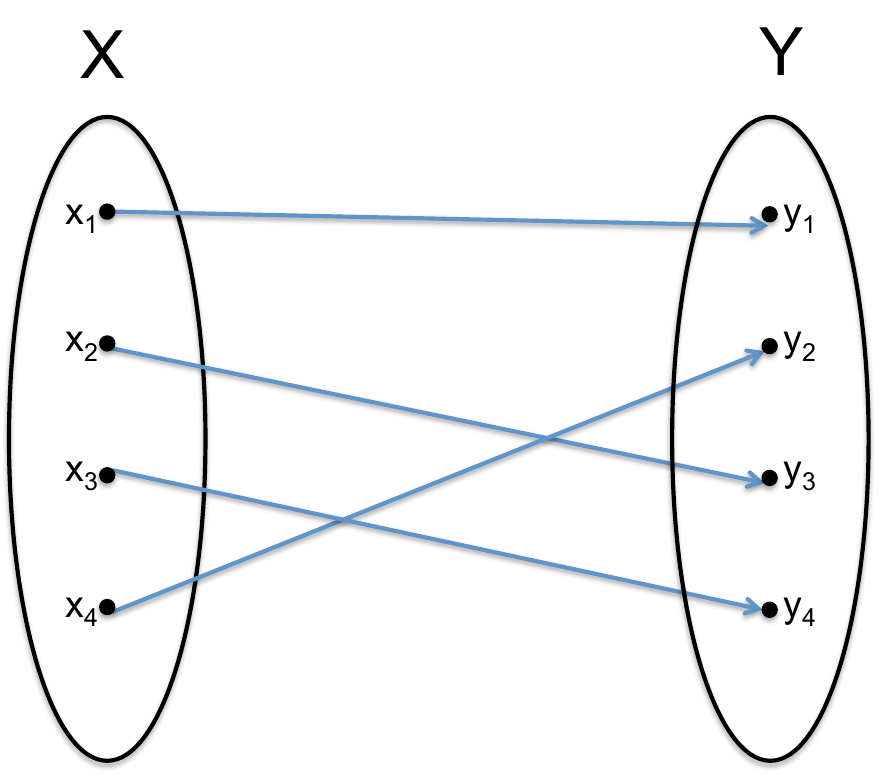}}
\end{align}

\end{example}

\begin{application}\label{app:DNA RNA}\index{RNA transcription}

There is an isomorphism between the set $\tn{Nuc}_\tn{DNA}$ of \href{http://en.wikipedia.org/wiki/Nucleotides}{\text nucleotides} found in DNA and the set $\tn{Nuc}_\tn{RNA}$ of nucleotides found in RNA. Indeed both sets have four elements, so there are 24 different isomorphisms. But only one is useful. Before we say which one it is, let us say there is also an isomorphism $\tn{Nuc}_\tn{DNA}\iso\{A,C,G,T\}$ and an isomorphism $\tn{Nuc}_\tn{RNA}\iso\{A,C,G,U\}$, and we will use the letters as abbreviations for the nucleotides. 

The convenient isomorphism $\tn{Nuc}_\tn{DNA}\To{\iso}\tn{Nuc}_\tn{RNA}$ is that given by RNA transcription; it sends 
$$A\mapsto U, C\mapsto G, G\mapsto C, T\mapsto A.$$ 
(See also Application \ref{app:polymerase}.) There is also an isomorphism $\tn{Nuc}_\tn{DNA}\To{\iso}\tn{Nuc}_\tn{DNA}$ (the matching in the double-helix) given by 
$$A\mapsto T, C\mapsto G, G\mapsto C, T\mapsto A.$$

Protein production can be modeled as a function from the set of 3-nucleotide sequences to the set of eukaryotic amino acids. However, it cannot be an isomorphism because there are $4^3=64$ triplets of RNA nucleotides, but only 21 eukaryotic amino acids. 

\end{application}

\begin{exercise}
Let $n\in\NN$ be a natural number and let $X$ be a set with exactly $n$ elements. 
\sexc How many isomorphisms are there from $X$ to itself? 
\next Does your formula from part a.) hold when $n=0$?
\endsexc
\end{exercise}

\begin{lemma}\label{lemma:isomorphic ER in Set}

The following facts hold about isomorphism.
\begin{enumerate}
\item Any set $A$ is isomorphic to itself; i.e. there exists an isomorphism $A\To{\iso} A$.
\item For any sets $A$ and $B$, if $A$ is isomorphic to $B$ then $B$ is isomorphic to $A$.
\item For any sets $A, B,$ and $C$, if $A$ is isomorphic to $B$ and $B$ is isomorphic to $C$ then $A$ is isomorphic to $C$.
\end{enumerate}

\end{lemma}

\begin{proof}

\begin{enumerate}
\item The identity function $\id_A\taking A\to A$ is invertible; its inverse is $\id_A$ because $\id_A\circ\id_A=\id_A$.
\item If $f\taking A\to B$ is invertible with inverse $g\taking B\to A$ then $g$ is an isomorphism with inverse $f$.
\item If $f\taking A\to B$ and $f'\taking B\to C$ are each invertible with inverses $g\taking B\to A$ and $g'\taking C\to B$ then the following calculations show that $f'\circ f$ is invertible with inverse $g\circ g'$: 
\begin{align*}
(f'\circ f)\circ(g\circ g')=f'\circ(f\circ g)\circ g'=f'\circ\id_B\circ g'=f'\circ g'=\id_C\\
(g\circ g')\circ(f'\circ f)=g\circ(g'\circ f')\circ f=g\circ\id_B\circ f=g\circ f=\id_A
\end{align*}
\end{enumerate}

\end{proof}

\begin{exercise}\label{exc:functions are not iso invariant}
Let $A$ and $B$ be the sets drawn below:
$$
\parbox{1.1in}{\boxtitle{A:=}\fbox{\xymatrix@=1pt{\\&\LMO{\;a\;}&&&\LMO{\;\;\;7\;\;}&\\\\\\&&&\LMO{Q}\\&}}}
\hspace{.8in}
\parbox{1.2in}{\boxtitle{B:=}\fbox{\xymatrix@=1pt{&&&\LMO{r8}&&\\\\\\\\&\LMO{``Bob"}\\&&\LMO{\clubsuit}}}}
$$
Note that the sets $A$ and $B$ are isomorphic. Supposing that $f\taking B\to\{1,2,3,4,5\}$ sends ``Bob" to $1$, sends $\clubsuit$ to $3$, and sends $r8$ to $4$, is there a canonical function $A\to\{1,2,3,4,5\}$ corresponding to $f$?
\footnote{Canonical means something like ``best choice", a choice that stands out as the only reasonable one.}\index{canonical}
\end{exercise}

\begin{exercise}\label{exc:generator for set}
Find a set $A$ such that for any set $X$ there is a isomorphism of sets $$X\iso\Hom_\Set(A,X).$$ Hint: draw a picture of proposed $A$'s and $X$'s.
\end{exercise}

For any natural number $n\in\NN$, define a set 
\begin{align}\label{dia:underline n}\index{a symbol!$\ul{n}$}
\ul{n}:=\{1,2,3,\ldots,n\}.
\end{align}
So, in particular, $\ul{2}=\{1,2\}, \ul{1}=\{1\}$, and $\ul{0}=\emptyset$. 

Let $A$ be any set. A function $f\taking\ul{n}\to A$ can be written as a sequence $$f=(f(1),f(2),\ldots,f(n)).$$

\begin{exercise}\label{exc:sequence}~
\sexc Let $A=\{a,b,c,d\}$. If $f\taking\ul{10}\to A$ is given by $(a,b,c,c,b,a,d,d,a,b)$, what is $f(4)$?
\next Let $s\taking\ul{7}\to\NN$ be given by $s(i)=i^2$. Write $s$ out as a sequence.
\endsexc
\end{exercise}

\begin{definition}Cardinality of finite sets]\label{def:cardinality}[

Let $A$ be a set and $n\in\NN$ a natural number. We say that $A$ is {\em has cardinality $n$}\index{cardinality}, denoted $$|A|=n,$$ if there exists an isomorphism of sets $A\iso\ul{n}$. If there exists some $n\in\NN$ such that $A$ has cardinality $n$ then we say that $A$ is {\em finite}. Otherwise, we say that $A$ is {\em infinite} and write $|A|\geq\infty$.

\end{definition}

\begin{exercise}~
\sexc Let $A=\{5,6,7\}$. What is $|A|$? 
\next What is $|\NN|$? 
\next What is $|\{n\in\NN\|n\leq 5\}|$?
\endsexc
\end{exercise}

\begin{lemma}

Let $A$ and $B$ be finite sets. If there is an isomorphism of sets $f\taking A\to B$ then the two sets have the same cardinality, $|A|=|B|$.

\end{lemma}

\begin{proof}

Suppose $f\taking A\to B$ is an isomorphism. If there exists natural numbers $m,n\in\NN$ and isomorphisms $a\taking\ul{m}\To\iso A$ and $b\taking\ul{n}\To\iso B$ then $\ul{m}\To{a^\m1}A\To{f}B\To{b}\ul{n}$ is an isomorphism. One can prove by induction that the sets $\ul{m}$ and $\ul{n}$ are isomorphic if and only if $m=n$. 

\end{proof}


\section{Commutative diagrams}\label{sec:comm diag}
\addtocounter{subsection}{1}\setcounter{subsubsection}{0}

At this point it is difficult to precisely define diagrams or commutative diagrams in general, but we can give the heuristic idea.
\footnote{We will define commutative diagrams precisely in Section \ref{sec:diagrams in a category}.}
Consider the following picture: 
\begin{align}\label{dia:triangle}
\xymatrix{A\ar[r]^f\ar[rd]_h&B\ar[d]^g\\&C}
\end{align}
We say this is a {\em diagram of sets}\index{diagram!in $\Set$} if each of $A,B,C$ is a set and each of $f,g,h$ is a function. We say this diagram {\em commutes}\index{commuting diagram}\index{diagam!commutes} if $g\circ f = h$. In this case we refer to it as a commutative triangle of sets.

\begin{application}

\href{http://en.wikipedia.org/wiki/Central_dogma_of_molecular_biology}{\text The central dogma of molecular biology} is that ``DNA codes for RNA codes for protein". That is, there is a function from DNA triplets to RNA triplets and a function from RNA triplets to amino acids. But sometimes we just want to discuss the translation from DNA to amino acids, and this is the composite of the other two. The commutative diagram is a picture of this fact.

\end{application}

Consider the following picture:
$$\xymatrix{A\ar[r]^f\ar[d]_h&B\ar[d]^g\\C\ar[r]_i&D}$$
We say this is a {\em diagram of sets} if each of $A,B,C,D$ is a set and each of $f,g,h,i$ is a function. We say this diagram {\em commutes} if $g\circ f=i\circ h$. In this case we refer to it as a commutative square of sets.

\begin{application}

Given a physical system $S$, there may be two mathematical approaches $f\taking S\to A$ and $g\taking S\to B$ that can be applied to it. Either of those results in a prediction of the same sort, $f'\taking A\to P$ and $g'\taking B\to P$. For example, in \href{http://en.wikipedia.org/wiki/Hamiltonian_mechanics#As_a_reformulation_of_Lagrangian_mechanics}{\text mechanics} we can use either Lagrangian approach or the Hamiltonian approach to predict future states. To say that the diagram 
$$
\xymatrix{S\ar[r]\ar[d]&A\ar[d]\\B\ar[r]&P}
$$
commutes would say that these approaches give the same result.

\end{application}

And so on. Note that diagram (\ref{dia:triangle}) is considered to be the same diagram as each of the following:
$$
\xymatrix{A\ar[r]^f\ar[d]_h&B\ar[dl]^g\\C}\hspace{.8in}
\xymatrix{A\ar[r]^f\ar@/_1pc/[rr]_h&B\ar[r]^g&C}\hspace{.8in}
\xymatrix{B\ar[rd]^g\\&C\\A\ar[ru]_h\ar[uu]^f}$$


\section{Ologs}\label{sec:ologs}\index{olog}

In this course we will ground the mathematical ideas in applications whenever possible. To that end we introduce ologs, which will serve as a bridge between mathematics and various conceptual landscapes. The following material is taken from \cite{SK}, an introduction to ologs.
\begin{align}\label{dia:arginine}\fbox{\xymatrixnocompile{\obox{D}{1in}{\rr an amino acid found in dairy}\LAL{dr}{is}&\obox{A}{.5in}{arginine}\ar@{}[dl]|(.3){\checkmark}\ar@{}[dr]|(.3){\checkmark}\LA{r}{has}\LAL{l}{is}\LA{d}{is}&\obox{E}{.9in}{\rr an electrically-charged side chain}\LA{d}{is}\\&\obox{X}{.9in}{an amino acid}\LAL{dl}{has}\LA{dr}{has}\LA{r}{has}&\smbox{R}{a side chain}\\\mebox{N}{an amine group}&&\mebox{C}{a carboxylic acid}}}\end{align}  



\subsection{Types}\index{olog!types}

A type is an abstract concept, a distinction the author has made.  We represent each type as a box containing a {\em singular indefinite noun phrase.}   Each of the following four boxes is a type: \begin{align}\label{dia:types}\xymatrixnocompile{\fbox{a man}&\fbox{an automobile}\\\obox{}{1.5in}{a pair $(a,w)$, where $w$ is a woman and $a$ is an automobile}&\obox{}{1.5in}{a pair $(a,w)$ where $w$ is a woman and $a$ is a blue automobile owned by $w$}}\end{align}

Each of the four boxes in (\ref{dia:types}) represents a type of thing, a whole class of things, and the label on that box is what one should call {\em each example} of that class.  Thus \fakebox{a man} does not represent a single man, but the set of men, each example of which is called ``a man".  Similarly, the bottom right box represents an abstract type of thing, which probably has more than a million examples, but the label on the box indicates the common name for each such example.  

Typographical problems emerge when writing a text-box in a line of text, e.g. the text-box \fbox{a man} seems out of place here, and the more in-line text-boxes there are, the worse it gets.  To remedy this, I will denote types which occur in a line of text with corner-symbols; e.g. I will write \fakebox{a man} instead of \fbox{a man}.


\subsubsection{Types with compound structures}

Many types have compound structures; i.e. they are composed of smaller units.  Examples include \begin{align}\label{dia:compound}\xymatrixnocompile{\obox{}{.7in}{\rr a man and a woman}&\obox{}{1.3in}{\rr a food portion $f$ and a child $c$ such that $c$ ate all of $f$}&\labox{}{a triple $(p,a,j)$ where $p$ is a paper, $a$ is an author of $p$, and $j$ is a journal in which $p$ was published}}\end{align}  It is good practice to declare the variables in a ``compound type", as I did in the last two cases of (\ref{dia:compound}).  In other words, it is preferable to replace the first box above with something like $$\obox{}{.8in}{a man $m$ and a woman $w$}\hsp\tn{or}\hsp\obox{}{1.1in}{\rr a pair $(m,w)$ where $m$ is a man and $w$ is a woman}$$ so that the variables $(m,w)$ are clear.

\begin{rules}\label{rules:types}\index{olog!rules}

A type is presented as a text box.  The text in that box should 
\begin{enumerate}[(i)]
\item begin with the word ``a" or ``an";
\item refer to a distinction made and recognizable by the olog's author;
\item refer to a distinction for which instances can be documented;
\item declare all variables in a compound structure. 
\end{enumerate}

\end{rules}

The first, second, and third rules ensure that the class of things represented by each box appears to the author as a well-defined set.  The fourth rule encourages good ``readability" of arrows, as will be discussed next in Section \ref{sec:aspects}.  

I will not always follow the rules of good practice throughout this document.  I think of these rules being followed ``in the background" but that I have ``nicknamed" various boxes.  So \fakebox{Steve} may stand as a nickname for \fakebox{a thing classified as Steve} and \fakebox{arginine} as a nickname for \fakebox{a molecule of arginine}. However, when pressed, one should always be able to rename each type according to the rules of good practice.


\subsection{Aspects}\label{sec:aspects}\index{olog!aspects}

An aspect of a thing $x$ is a way of viewing it, a particular way in which $x$ can be regarded or measured.  For example, a woman can be regarded as a person; hence ``being a person" is an aspect of a woman.  A molecule has a molecular mass (say in daltons), so ``having a molecular mass" is an aspect of a molecule.  In other words, by {\em aspect} we simply mean a function. The domain $A$ of the function $f\taking A\to B$ is the thing we are measuring, and the codomain is the set of possible ``answers" or results of the measurement. 
\begin{align}\label{dia:aspect 1}\xymatrixnocompile{\fbox{a woman}\LA{r}{is}&\fbox{a person}}\end{align}\begin{align}\label{dia:aspect 2}\xymatrixnocompile{\fbox{a molecule}\LA{rr}{has as molecular mass (Da)}&\hspace{.7in}&\fbox{a positive real number}}\end{align}

So for the arrow in (\ref{dia:aspect 1}), the domain is the set of women (a set with perhaps 3 billion elements); the codomain is the set of persons (a set with perhaps 6 billion elements).   We can imagine drawing an arrow from each dot in the ``woman" set to a unique dot in the ``person" set, just as in (\ref{dia:setmap}).  No woman points to two different people, nor to zero people --- each woman is exactly one person --- so the rules for a function are satisfied.  Let us now concentrate briefly on the arrow in (\ref{dia:aspect 2}).  The domain is the set of molecules, the codomain is the set $\RR_{>0}$ of positive real numbers.  We can imagine drawing an arrow from each dot in the ``molecule" set to a single dot in the ``positive real number" set.  No molecule points to two different masses, nor can a molecule have no mass: each molecule has exactly one mass.  Note however that two different molecules can point to the same mass.


\subsubsection{Invalid aspects}\label{sec:invalid aspect}\index{olog!invalid aspects}

I tried above to clarify what it is that makes an aspect ``valid", namely that it must be a ``functional relationship."  In this subsection I will show two arrows which on their face may appear to be aspects, but which on closer inspection are not functional (and hence are not valid as aspects).  
 
Consider the following two arrows:
\begin{align}\tag{\arabic{subsection}.\arabic{equation}*}\addtocounter{equation}{1}\label{dia:invalid 1}
\xymatrixnocompile{\fbox{a person}\LA{r}{has}&\fbox{a child}}
\end{align}
\vspace{-.13in}
\begin{align}\tag{\arabic{subsection}.\arabic{equation}*}\addtocounter{equation}{1}\label{dia:invalid 2}
\xymatrixnocompile{\fbox{a mechanical pencil}\LA{r}{uses}&\fbox{a piece of lead}}
\end{align}  
A person may have no children or may have more than one child, so the first arrow is invalid: it is not a function.  Similarly, if we drew an arrow from each mechanical pencil to each piece of lead it uses, it would not be a function.

\begin{warning}\label{warn:worldview}\index{a warning!different worldviews}

The author of an olog has a world-view, some fragment of which is captured in the olog.  When person A examines the olog of person B, person A may or may not ``agree with it."  For example, person B may have the following olog $$\fbox{\xymatrix{&\fbox{a marriage}\LA{dr}{ includes}\LAL{dl}{includes }\\\fbox{a man}&&\fbox{a woman}}}$$ which associates to each marriage a man and a woman.  Person A may take the position that some marriages involve two men or two women, and thus see B's olog as ``wrong."  Such disputes are not ``problems" with either A's olog or B's olog, they are discrepancies between world-views.  Hence, throughout this paper, a reader R may see a displayed olog and notice a discrepancy between R's world-view and my own, but R should not worry that this is a problem.  This is not to say that ologs need not follow rules, but instead that the rules are enforced to ensure that an olog is structurally sound, rather than that it ``correctly reflects reality," whatever that may mean.

Consider the aspect $\fakebox{an object}\Too{\tn{has}}\fakebox{a weight}$. At some point in history, this would have been considered a valid function. Now we know that the same object would have a different weight on the moon than it has on earth. Thus as world-views change, we often need to add more information to our olog. Even the validity of $\fakebox{an object on earth}\Too{\tn{has}}\fakebox{a weight}$ is questionable. However to build a model we need to choose a level of granularity and try to stay within it, or the whole model evaporates into the nothingness of truth!

\end{warning}

\begin{remark}

In keeping with Warning \ref{warn:worldview}, the arrows (\ref{dia:invalid 1}) and (\ref{dia:invalid 2}) may not be wrong but simply reflect that the author has a strange world-view or a strange vocabulary.  Maybe the author believes that every mechanical pencil uses exactly one piece of lead.  If this is so, then $\fakebox{a mechanical pencil}\To{\tn{uses}}\fakebox{a piece of lead}$ is indeed a valid aspect!   Similarly, suppose the author meant to say that each person {\em was once} a child, or that a person has an inner child.  Since every person has one and only one inner child (according to the author), the map $\fakebox{a person}\To{\tn{has as inner child}}\fakebox{a child}$ is a valid aspect.  We cannot fault the olog if the author has a view, but note that we have changed the name of the label to make his or her intention more explicit.

\end{remark}


\subsubsection{Reading aspects and paths as English phrases}

Each arrow (aspect) $X\To{f} Y$ can be read by first reading the label on its source box (domain of definition) $X$, then the label on the arrow $f$, and finally the label on its target box (set of values) $Y$.  For example, the arrow \begin{align}\label{dia:first author}\fbox{\xymatrixnocompile{\smbox{}{a book}\LA{rrr}{has as first author}&&&\smbox{}{a person}}}\end{align} is read ``a book has as first author a person".  

\begin{remark}

Note that the map in (\ref{dia:first author}) is a valid aspect, but that a similarly benign-looking map $\fakebox{a book}\To{\tn{has as author}}\fakebox{a person}$ would not be valid, because it is not functional.  The authors of an olog must be vigilant about this type of mistake because it is easy to miss and it can corrupt the olog.

\end{remark}

Sometimes the label on an arrow can be shortened or dropped altogether if it is obvious from context.  We will discuss this more in Section \ref{sec:facts} but here is a common example from the way I write ologs. \begin{align}\label{dia:pair of integers}\fbox{\xymatrixnocompile{&\obox{A}{1.2in}{\rr a pair $(x,y)$ where $x$ and $y$ are integers}\ar[dl]_x\ar[dr]^y\\\smbox{B}{an integer}&&\smbox{B}{an integer}}}\end{align}  Neither arrow is readable by the protocol given above (e.g. ``a pair $(x,y)$ where $x$ and $y$ are integers $x$ an integer" is not an English sentence), and yet it is obvious what each map means.  For example, given $(8,11)$ in $A$, arrow $x$ would yield $8$ and arrow $y$ would yield $11$.  The label $x$ can be thought of as a nickname for the full name ``yields, via the value of $x$," and similarly for $y$.  I do not generally use the full name for fear that the olog would become cluttered with text.

One can also read paths through an olog by inserting the word ``which" after each intermediate box.
\footnote{If the intended elements of an intermediate box are humans, it is polite to use ``who" rather than ``which", and other such conventions may be upheld if one so desires.}
For example the following olog has two paths of length 3 (counting arrows in a chain): \small\begin{align}\label{olog:paths}\fbox{\xymatrixnocompile{\fbox{a child}\LA{r}{is}&\fbox{a person}\LA{rr}{has as parents}\LAL{dr}{has, as birthday}&&\obox{}{.8in}{\rr a pair $(w,m)$ where $w$ is a woman and $m$ is a man}\LA{r}{$w$}&\fbox{a woman}\\&&\fbox{a date}\LA{r}{includes}&\fbox{a year}}}\end{align}  \normalsize The top path is read ``a child is a person, who has as parents a pair $(w,m)$ where $w$ is a woman and $m$ is a man, which yields, via the value of $w$, a woman."  The reader should read and understand the content of the bottom path, which associates to every child a year.


\subsubsection{Converting non-functional relationships to aspects}\label{sec:relations}

There are many relationships that are not functional, and these cannot be considered aspects.  Often the word ``has" indicates a relationship --- sometimes it is functional as in $\fakebox{a person}\To{\tn{ has }}\fakebox{a stomach}$, and sometimes it is not, as in $\fakebox{a father}\To{\tn{has}}\fakebox{a child}$. Obviously, a father may have more than one child. This one is easily fixed by realizing that the arrow should go the other way: there is a function $\fakebox{a child}\To{\tn{has}}\fakebox{a father}$. 

What about $\fakebox{a person}\To{\tn{owns}}\fakebox{a car}$. Again, a person may own no cars or more than one car, but this time a car can be owned by more than one person too. A quick fix would be to replace it by $\fakebox{a person}\To{\tn{owns}}\fakebox{a set of cars}$.   This is ok, but the relationship between \fakebox{a car} and \fakebox{a set of cars} then becomes an issue to deal with later.  There is another way to indicate such ``non-functional" relationships. In this case it would look like this:
$$
\fbox{\xymatrix{&\obox{}{1.15in}{a pair $(p,c)$ where $p$ is a person, $c$ is a car, and $p$ owns $c$.}\ar[ddl]_p\ar[ddr]^c\\\\
\obox{}{.5in}{a person}&&\obox{}{.3in}{a car}}}
$$
This setup will ensure that everything is properly organized. In general, relationships can involve more than two types, and the general situation looks like this $$\fbox{\xymatrixnocompile{&&\fbox{$R$}\ar[ddll]\ar[ddl]\ar[ddr]\\\\\fbox{$A_1$}&\fbox{$A_2$}&\cdots&\fbox{$A_n$}}}$$  For example, $$\fbox{\xymatrixnocompile{&\labox{R}{a sequence $(p,a,j)$ where $p$ is a paper, $a$ is an author of $p$, and $j$ is a journal in which $p$ was published}\ar[ddl]_p\ar[dd]_a\ar[ddr]^j\\\\\smbox{A_1}{a paper}&\smbox{A_2}{an author}&\smbox{A_3}{a journal}}}$$ 

\begin{exercise}
On page \pageref{dia:invalid 1} we indicate a so-called invalid aspect, namely 
\begin{align}\tag{\ref{dia:invalid 1}}\xymatrixnocompile{\fbox{a person}\LA{r}{has}&\fbox{a child}}
\end{align}
Create a (valid) olog that captures the parent-child relationship; your olog should still have boxes \fakebox{a person} and \fakebox{a child} but may have an additional box.
\end{exercise}

\begin{rules}\label{rules:aspects}\index{olog!rules}

An aspect is presented as a labeled arrow, pointing from a source box to a target box.  The arrow text should

\begin{enumerate}[(i)]
\item begin with a verb;
\item yield an English sentence, when the source-box text followed by the arrow text followed by the target-box text is read; and
\item refer to a functional relationship: each instance of the source type should give rise to a specific instance of the target type.
\end{enumerate}

\end{rules}


\subsection{Facts}\label{sec:facts}\index{olog!facts}

In this section I will discuss facts, which are simply ``path equivalences" in an olog. It is the notion of path equivalences that make category theory so powerful. 

A {\em path}\index{olog!path in} in an olog is a head-to-tail sequence of arrows. That is, any path starts at some box $B_0$, then follows an arrow emanating from $B_0$ (moving in the appropriate direction), at which point it lands at another box $B_1$, then follows any arrow emanating from $B_1$, etc, eventually landing at a box $B_n$ and stopping there. The number of arrows is the {\em length} of the path. So a path of length 1 is just an arrow, and a path of length 0 is just a box. We call $B_0$ the {\em source} and $B_n$ the {\em target} of the path.

Given an olog, the author may want to declare that two paths are equivalent.  For example consider the two paths from $A$ to $C$ in the olog 
\begin{align}\label{olog:commute}\fbox{\xymatrixnocompile{\smbox{A}{a person}\LA{rr}{has as parents}\LAL{drr}{\parbox{.8in}{has as mother}}&&\obox{B}{.8in}{\rr a pair $(w,m)$ where $w$ is a woman and $m$ is a man}\ar@{}[dll]|(.4){\checkmark}\LA{d}{yields as $w$}\\&&\smbox{C}{a woman}}}\end{align}  We know as English speakers that a woman parent is called a mother, so these two paths $A\to C$ should be equivalent.  A more mathematical way to say this is that the triangle in Olog (\ref{olog:commute}) {\em commutes}. That is, path equivalences are simply commutative diagrams as in Section \ref{sec:comm diag}. In the example above we concisely say ``a woman parent is equivalent to a mother."  We declare this by defining the diagonal map in (\ref{olog:commute}) to be {\em the composition} of the horizontal map and the vertical map. 

I generally prefer to indicate a commutative diagram by drawing a check-mark, $\checkmark$, in the region bounded by the two paths, as in Olog (\ref{olog:commute}).  Sometimes, however, one cannot do this unambiguously on the 2-dimensional page.  In such a case I will indicate the commutative diagrams (fact) by writing an equation.  For example to say that the diagram $$\xymatrix{A\ar[r]^f\ar[d]_h&B\ar[d]^g\\C\ar[r]_i&D}$$ commutes, we could either draw a checkmark inside the square or write the equation $A\;f\;g\simeq A\;h\;i$ above it\index{a symbol!$\simeq$}.
\footnote{We defined function composition on page \ref{function composition}, but here we're using a different notation.\index{a warning!notation for composition} There we would have said $g\circ f = i\circ h$, which is in the backwards-seeming {\em classical order}.\index{composition!classical order} Category theorists and others often prefer the {\em diagrammatic order}\index{composition!diagrammatic order} for writing compositions, which is $f;g = h;i$. For ologs, we follow the latter because it makes for better English sentences, and for the same reason we add the source object to the equation, writing $A f g \simeq A h i$.}
  Either way, it means that ``$f$ then $g$" is equivalent to ``$h$ then $i$".  

Here is another, more scientific example:
\begin{align*}
\fbox{\xymatrix{
\obox{}{1in}{a DNA sequence}\LA{rr}{is transcribed to}\LAL{drr}{codes for}&\hspace{.1in}&\obox{}{1.1in}{an RNA sequence}\ar@{}[dll]|(.35){\checkmark}\LA{d}{is translated to}\\
&&\obox{}{.6in}{a protein}}}
\end{align*}
Note how this diagram gives us the established terminology for the various ways in which DNA, RNA, and protein are related in this context.

\begin{exercise}\label{exc:family olog}

Create an olog for human nuclear biological families that includes the concept of person, man, woman, parent, father, mother, and child. Make sure to label all the arrows, and make sure each arrow indicates a valid aspect in the sense of Section \ref{sec:invalid aspect}. Indicate with check-marks ($\checkmark$) the diagrams that are intended to commute. If the 2-dimensionality of the page prevents a check-mark from being unambiguous, indicate the intended commutativity with an equation.
\end{exercise}

\begin{example}[Non-commuting diagram]

In my conception of the world, the following diagram does not commute:
\begin{align}\label{dia:non-commuting}
\xymatrixnocompile@=50pt{\obox{}{.5in}{a person}\LA{r}{has as father}\LAL{dr}{lives in}&\obox{}{.4in}{a man}\LA{d}{lives in}\\&\obox{}{.4in}{a city}}
\end{align}
The non-commutativity of Diagram (\ref{dia:non-commuting}) does not imply that, in my conception, no person lives in the same city as his or her father. Rather it implies that, in my conception, it is not the case that {\em every} person lives in the same city as his or her father.

\end{example}

\begin{exercise}
Create an olog about a scientific subject, preferably one you think about often. The olog should have at least five boxes, five arrows, and one commutative diagram. 
\end{exercise}


\subsubsection{A formula for writing facts as English}\index{olog!facts in English}

Every fact consists of two paths, say $P$ and $Q$, that are to be declared equivalent. The paths $P$ and $Q$ will necessarily have the same source, say $s$, and target, say $t$, but their lengths may be different, say $m$ and $n$ respectively.
\footnote{If the source equals the target, $s=t$, then it is possible  to have $m=0$ or $n=0$, and the ideas below still make sense.} 
We draw these paths as 
\begin{align}\label{dia:two paths for equivalence}
P:&\hsp\xymatrix@=22pt{\LMO{a_0=s}\ar[r]^{f_1}&\LMO{a_1}\ar[r]^{f_2}&\LMO{a_2}\ar[r]^{f_3}&\cdots\ar[r]^{f_{m-1}}&\LMO{a_{m-1}}\ar[r]^{f_m}&\LMO{a_m=t}}\\\nonumber
Q:&\hsp\xymatrix@=23pt{\LMO{b_0=s}\ar[r]^{g_1}&\LMO{b_1}\ar[r]^{g_2}&\LMO{b_2}\ar[r]^{g_3}&\cdots\ar[r]^{g_{n-1}}&\LMO{b_{n-1}}\ar[r]^{g_n}&\LMO{b_n=t}}
\end{align}
Every part $\ell$ of an olog (i.e. every box and every arrow) has an associated English phrase, which we write as $\qt{\ell}$. Using a dummy variable $x$ we can convert a fact into English too. The following general formula is a bit difficult to understand, see Example \ref{ex:English fact}, but here goes. The fact $P\simeq Q$ from (\ref{dia:two paths for equivalence}) can be Englishified as follows:

\begin{align}\label{dia:Englishification}\index{Englishification}
&\tn{Given }x,\qt{s},\tn{ consider the following. We know that }x\tn{ is }\qt{s}, \\
\nonumber&\tn {which } \qt{f_1}\;\qt{a_1}, \tn{ which } \qt{f_2}\;\qt{a_2}, \tn { which }\ldots \; \qt{f_{m-1}}\;\qt{a_{m-1}}, \tn { which } \qt{f_m}\;\qt{t}\\
\nonumber&\tn{that we'll call } P(x).\\
\nonumber&\tn{We also know that }x\tn{ is } \qt{s},\\
\nonumber&\tn {which } \qt{g_1}\;\qt{b_1}, \tn{ which }\qt{g_2}\;\qt{b_2}, \tn { which }\ldots\;\qt{g_{n-1}}\;\qt{b_{n-1}}, \tn { which } \qt{g_n}\;\qt{t}\\
\nonumber&\tn{that we'll call } Q(x).\\
\nonumber&\tn{Fact: whenever }x\tn{ is }``s",\tn{ we will have }P(x)=Q(x).
\end{align}

\begin{example}\label{ex:English fact}

Consider the olog
\begin{align}\label{olog:commute2}\fbox{\xymatrixnocompile{\smbox{A}{a person}\LA{rr}{has}\LAL{drr}{\parbox{.8in}{lives in}}&&\obox{B}{.7in}{\rr an address}\ar@{}[dll]|(.4){\checkmark}\LA{d}{is in}\\&&\smbox{C}{a city}}}
\end{align}
To put the fact that Diagram \ref{olog:commute2} commutes into English, we first Englishify the two paths: $F$=``a person has an address which is in a city" and $G$=``a person lives in a city". The source of both is $s$=``a person" and the target of both is $t$=``a city".
write:
\begin{align*}
&\tn{Given }x,\tn{a person, consider the following. We know that } x\tn{ is a person,}\\
&\tn{which has an address, which is in a city}\\
&\tn{that we'll call } P(x).\\
&\tn{We also know that }x\tn{ is a person,}\\
&\tn{which lives in a city}\\
&\tn{that we'll call } Q(x).\\
&\tn{Fact: whenever }x\tn{ is a person, we will have }P(x)=Q(x).
\end{align*}

\end{example}

\begin{exercise}
This olog was taken from \cite{Sp1}.
\begin{align}\label{dia:phone paths}\xymatrix{&\obox{N}{1in}{a phone number}\LA{rr}{has}&&\obox{C}{.8in}{an area code}\ar@{}[dll]|{\checkmark}\LA{d}{corresponds to}\\\obox{OLP}{1.2in}{an operational landline phone}\LA{ru}{is assigned}\LAL{r}{is}&\obox{P}{1in}{a physical phone}\LAL{rr}{\parbox{.55in}{\scriptsize is currently located in}}&&\obox{R}{.5in}{a region}}
\end{align} 
It says that a landline phone is physically located in the region that its phone number is assigned. Translate this fact into English using the formula from \ref{dia:Englishification}.
\end{exercise}

\begin{exercise}
In the above olog (\ref{dia:phone paths}), suppose that the box \fakebox{an operational landline phone} is replaced with the box \fakebox{an operational mobile phone}. Would the diagram still commute?
\end{exercise}


\subsubsection{Images}\label{sec:images}\index{olog!images}\index{image!in olog}

In this section we discuss a specific kind of fact, generated by any aspect. Recall that every function has an image, meaning the subset of elements in the codomain that are ``hit" by the function. For example the function $f(x)=2*x\taking \ZZ\to\ZZ$ has as image the set of all even numbers.

Similarly the set of mothers arises as is the image of the ``has as mother" function, as shown below 
$$
\xymatrix{\obox{P}{.5in}{a person}\LAL{rd}{has}\LA{rr}{$\stackrel{f\taking P\to P}{\tn{has as mother}}$}&&\obox{P}{.5in}{a person}\\
&\obox{M=\im(f)}{.6in}{a mother}\LAL{ur}{is}\ar@{}[u]|(.6){\checkmark}
}$$

\begin{exercise}
For each of the following types, write down a function for which it is the image, or say ``not clearly an image type" 
\sexc \fakebox{a book}
\next \fakebox{a material that has been fabricated by a process of type $T$}
\next \fakebox{a bicycle owner}
\next \fakebox{a child}
\next \fakebox{a used book}
\next \fakebox{an inhabited residence}
\endsexc
\end{exercise}


\section{Products and coproducts}\label{sec:prods and coprods in set}

In this section we introduce two concepts that are likely to be familiar, although perhaps not by their category-theoretic names, product and coproduct. Each is an example of a large class of ideas that exist far beyond the realm of sets.


\subsection{Products}\label{sec:products}\index{products!of sets}

\begin{definition}

Let $X$ and $Y$ be sets. The {\em product of $X$ and $Y$}, denoted $X\times Y$,\index{a symbol!$\times$} is defined as the set of ordered pairs $(x,y)$ where $x\in X$ and $y\in Y$. Symbolically, $$X\times Y=\{(x,y)\|x\in X,\;\; y\in Y\}.$$ There are two natural {\em projection functions} $\pi_1\taking X\times Y\to X$ and $\pi_2\taking X\times Y\to Y$.\index{projection functions}\index{product!projection functions}
$$\xymatrix@=15pt{&X\times Y\ar[ddr]^{\pi_2}\ar[ddl]_{\pi_1}\\\\X&&Y}$$

\end{definition}

\begin{example}\label{ex:grid1}[Grid of dots]\index{product!as grid}

Let $X=\{1,2,3,4,5,6\}$ and $Y=\{\clubsuit,\diamondsuit,\heartsuit,\spadesuit\}$. Then we can draw $X\times Y$ as a 6-by-4 grid of dots, and the projections as projections
\begin{align}
\parbox{2.9in}{\begin{center}\small $X\times Y$\vspace{-.1in}\end{center}\fbox{
\xymatrix@=10pt{
\LMO{(1,\clubsuit)}&\LMO{(2,\clubsuit)}&\LMO{(3,\clubsuit)}&\LMO{(4,\clubsuit)}&\LMO{(5,\clubsuit)}&\LMO{(6,\clubsuit)}\\
\LMO{(1,\diamondsuit)}&\LMO{(2,\diamondsuit)}&\LMO{(3,\diamondsuit)}&\LMO{(4,\diamondsuit)}&\LMO{(5,\diamondsuit)}&\LMO{(6,\diamondsuit)}\\
\LMO{(1,\heartsuit)}&\LMO{(2,\heartsuit)}&\LMO{(3,\heartsuit)}&\LMO{(4,\heartsuit)}&\LMO{(5,\heartsuit)}&\LMO{(6,\heartsuit)}\\
\LMO{(1,\spadesuit)}&\LMO{(2,\spadesuit)}&\LMO{(3,\spadesuit)}&\LMO{(4,\spadesuit)}&\LMO{(5,\spadesuit)}&\LMO{(6,\spadesuit)}\\
}}}
\parbox{.9in}{
\xymatrix{~\ar[rr]^{\pi_2}&&~}
}
\parbox{.3in}{\begin{center}\small $Y$\vspace{-.1in}\end{center}\fbox{
\xymatrix@=10pt{
\LMO{\clubsuit}\\\LMO{\diamondsuit}\\\LMO{\heartsuit}\\\LMO{\spadesuit}
}}}
\\\nonumber
\parbox{1in}{\hspace{-1.95in}\xymatrix{~\ar[dd]_{\pi_1}\\\\~}}
\\\nonumber
\parbox{2.9in}{\hspace{-1.2in}\fbox{
\xymatrix@=24pt{
\LMO{1}&\LMO{2}&\LMO{3}&\LMO{4}&\LMO{5}&\LMO{6}
}}\begin{center}\hspace{-2.6in}\small$X$\end{center}}
\end{align}

\end{example}

\begin{application}
A traditional (Mendelian) way to predict the genotype of offspring based on the genotype of its parents is by the use of \href{http://en.wikipedia.org/wiki/Punnett_square}{Punnett squares}. If $F$ is the set of possible genotypes for the female parent and $M$ is the set of possible genotypes of the male parent, then $F\times M$ is drawn as a square, called a Punnett square, in which every combination is drawn. 
\end{application}

\begin{exercise}
How many elements does the set $\{a,b,c,d\}\times\{1,2,3\}$ have?
\end{exercise}

\begin{application}

Suppose we are conducting experiments about the mechanical properties of materials, as in Application \ref{app:force-extension}. For each material sample we will produce multiple data points in the set $\fakebox{extension}\times\fakebox{force}\iso\RR\times\RR$.

\end{application}

\begin{remark}

It is possible to take the product of more than two sets as well. For example, if $A,B,$ and $C$ are sets then $A\times B\times C$ is the set of triples, 
$$A\times B\times C:=\{(a,b,c)\|a\in A, b\in B, c\in C\}.$$

This kind of generality is useful in understanding multiple dimensions, e.g. what physicists mean by 10-dimensional space. It comes under the heading of {\em limits}, which we will see in Section \ref{sec:lims and colims in a cat}.

\end{remark}

\begin{example}\label{ex:R2}

Let $\RR$\index{a symbol!$\RR$} be the set of real numbers. By $\RR^2$ we mean $\RR\times\RR$ (though see Exercise \ref{exc:two R2s}). Similarly, for any $n\in\NN$, we define $\RR^n$ to be the product of $n$ copies of $\RR$. 

According to \cite{Pen}, Aristotle seems to have conceived of space as something like $S:=\RR^3$ and of time as something like $T:=\RR$. Spacetime, had he conceived of it, would probably have been $S\times T\iso\RR^4$. He of course did not have access to this kind of abstraction, which was probably due to Descartes. 

\end{example}

\begin{exercise}
Let $\ZZ$ denote the set of integers, and let $+\taking\ZZ\times\ZZ\to\ZZ$ denote the addition function and $\cdot\taking\ZZ\times\ZZ\to\ZZ$ denote the multiplication function. Which of the following diagrams commute?
\sexc $$\xymatrix{
\ZZ\times\ZZ\times\ZZ\ar[rr]^-{(a,b,c)\mapsto(a\cdot b,a\cdot c)}\ar[d]_{(a,b,c)\mapsto(a+b,c)}&\hsp&\ZZ\times\ZZ\ar[d]^{(x,y)\mapsto x+y}\\
\ZZ\times\ZZ\ar[rr]_{(x,y)\mapsto xy}&&\ZZ}
$$
\next $$
\xymatrix{
\ZZ\ar[rr]^{x\mapsto (x,0)}\ar[drr]_{\id_\ZZ}&&\ZZ\times\ZZ\ar[d]^{(a,b)\mapsto a\cdot b}\\&&\ZZ}
$$
\next$$
\xymatrix{
\ZZ\ar[rr]^{x\mapsto (x,1)}\ar[drr]_{\id_\ZZ}&&\ZZ\times\ZZ\ar[d]^{(a,b)\mapsto a\cdot b}\\&&\ZZ}
$$
\endsexc
\end{exercise}


\subsubsection{Universal property for products}\index{products!universal property of}\index{universal property!products}

\begin{lemma}[Universal property for product]\label{lemma:up for prod}

Let $X$ and $Y$ be sets. For any set $A$ and functions $f\taking A\to X$ and $g\taking A\to Y$, there exists a unique function $A\to X\times Y$ such that the following diagram commutes \footnote{The symbol $\forall$ is read ``for all"; the symbol $\exists$ is read ``there exists", and the symbol $\exists!$ is read ``there exists a unique". So this diagram is intended to express the idea that for any functions $f\taking A\to X$ and $g\taking A\to Y$, there exists a unique function $A\to X\times Y$ for which the two triangles commute.}
\begin{align}\label{dia:univ prop for products}
\xymatrix@=15pt{&X\times Y\ar[ldd]_{\pi_1}\ar[rdd]^{\pi_2}\\\\X\ar@{}[r]|{\checkmark}&&Y\ar@{}[l]|{\checkmark}\\\\&A\ar[luu]^{\forall f}\ar[ruu]_{\forall g}\ar@{-->}[uuuu]^{\exists !}}
\end{align}
We might write the unique function as $$\prodmap{f}{g}\taking A\to X\times Y.$$

\end{lemma}

\begin{proof}

Suppose given $f,g$ as above. To provide a function $\ell\taking A\to X\times Y$ is equivalent to providing an element $\ell(a)\in X\times Y$ for each $a\in A$. We need such a function for which $\pi_1\circ \ell=f$ and $\pi_2\circ \ell=g$. An element of $X\times Y$ is an ordered pair $(x,y)$, and we can use $\ell(a)=(x,y)$ if and only if $x=\pi_1(x,y)=f(a)$ and $y=\pi_2(x,y)=g(a)$. So it is necessary and sufficient to define $$\prodmap{f}{g}(a):=(f(a),g(a))$$ for all $a\in A$.

\end{proof}

\begin{example}[Grid of dots, continued]\label{ex:grid2}

We need to see the universal property of products as completely intuitive. Recall that if $X$ and $Y$ are sets, say of cardinalities $|X|=m$ and $|Y|=n$ respectively, then $X\times Y$ is an $m\times n$ grid of dots, and it comes with two canonical projections $X\From{\pi_1}X\times Y\To{\pi_2}Y$. These allow us to extract from every grid element $z\in X\times Y$ its column $\pi_1(z)\in X$ and its row $\pi_2(z)\in Y$.

Suppose that each person in a classroom picks an element of $X$ and an element of $Y$. Thus we have functions $f\taking C\to X$ and $g\taking C\to Y$. But isn't picking a column and a row the same thing as picking an element in the grid? The two functions $f$ and $g$ induce a unique function $C\to X\times Y$. And how does this function $C\to X\times Y$ compare with the original functions $f$ and $g$? The commutative diagram (\ref{dia:univ prop for products}) sums up the obvious connection. 

\end{example}

\begin{example}

Let $\RR$ be the set of real numbers. The origin in $\RR$ is an element of $\RR$. As you showed in Exercise \ref{exc:generator for set}, we can view this (or any) element of $\RR$ as a function $z\taking\singleton\to\RR$, where $\singleton$ is any set with one element. Our function $z$ ``picks out the origin". Thus we can draw functions 
$$\xymatrix@=15pt{&\singleton\ar[ddr]^z\ar[ddl]_z\\\\\RR&&\RR}
$$
The universal property for products guarantees a function $\singleton\to\RR\times\RR$, which will be the origin in $\RR^2.$

\end{example}

\begin{remark}

Given sets $X, Y,$ and $A$, and functions $f\taking A\to X$ and $g\taking A\to Y$, there is a unique function $A\to X\times Y$ that commutes with $f$ and $g$. We call it {\em the induced function $A\to X\times Y$},\index{induced function} meaning the one that arises in light of $f$ and $g$.

\end{remark}

\begin{exercise}
For every set $A$ there is some nice relationship between the following three sets: $$\Hom_{\Set}(A,X), \hsp \Hom_\Set(A,Y), \hsp \text{and} \hsp\Hom_\Set(A,X\times Y).$$ What is it?

Hint: Do not be alarmed: this problem is a bit ``recursive" in that you'll use products in your formula.
\end{exercise}

\begin{exercise}~
\sexc Let $X$ and $Y$ be sets. Construct the ``swap map" $s\taking X\times Y\to Y\times X$ using only the universal property for products. If $\pi_1\taking X\times Y\to X$ and $\pi_2\taking X\times Y\to Y$ are the projection functions, write $s$ in terms of the symbols $``\pi_1",``\pi_2", ``(\ ,\ )",$ and $``\circ"$. 
\next Can you prove that $s$ is a isomorphism using only the universal property for product?
\endsexc
\end{exercise}

\begin{example}\label{ex:product to product}
Suppose given sets $X,X', Y, Y'$ and functions $m\taking X\to X'$ and $n\taking Y\to Y'$. We can use the universal property of products to construct a function $s\taking X\times Y\to X'\times Y'$.  Here's how.

The universal property (Lemma \ref{lemma:up for prod}) says that to get a function from any set $A$ to $X'\times Y'$, we need two functions, namely some $f\taking A\to X'$ and some $g\taking A\to Y'$. Here $A=X\times Y$. 

What we have readily available are the two projections $\pi_1\taking X\times Y\to X$ and $\pi_2\taking X\times Y\to Y$. But we also have $m\taking X\to X'$ and $n\taking Y\to Y'$. Composing, we set $f:=m\circ \pi_1$ and $g:=n\circ\pi_2$.
$$\xymatrix{
&X'\times Y'\ar[dl]_{\pi_1'}\ar[dr]^{\pi_2'}\\
X'&&Y'\\
X\ar[u]^m&&Y\ar[u]_n\\
&X\times Y\ar[ul]^{\pi_1}\ar[ur]_{\pi_2}\ar@{-->}[uuu]
}
$$
The dotted arrow is often called the {\em product} of $m\taking X\to X'$ and $n\taking Y\to Y'$ and is denoted simply by 
$$m\times n\taking X\times Y\to X'\times Y'.$$

\end{example}


\subsubsection{Ologging products}\label{sec:ologging products}

Given two objects $c,d$ in an olog, there is a canonical label $\qt{c\times d}$ for their product $c\times d$, written in terms of the labels $\qt{c}$ and $\qt{d}$. Namely, $$\qt{c\times d}:=\tn{a pair }(x,y)\tn{ where }x\tn{ is }\qt{c}\tn{ and }y\tn{ is }\qt{d}.$$ The projections $c\from c\times d\to d$ can be labeled ``yields, as $x$," and ``yields, as $y$," respectively.

Suppose that $e$ is another object and $p\taking e\to c$ and $q\taking e\to d$ are two arrows. By the universal property of products (Lemma \ref{lemma:up for prod}), $p$ and $q$ induce a unique arrow $e\to c\times d$ making the evident diagrams commute. This arrow can be labeled
\begin{center}
yields, insofar as it $\qt{p}\;\qt{c}$ and $\qt{q}\;\qt{d}$, 
\end{center}

\begin{example}

Every car owner owns at least one car, but there is no obvious function $\fakebox{a car owner}\to\fakebox{a car}$ because he or she may own more than one. One good choice would be the car that the person drives most often, which we'll call his or her primary car. Also, given a person and a car, an economist could ask how much utility the person would get out of the car. From all this we can put together the following olog involving products:
$$
\xymatrixnocompile{\obox{O}{.7in}{a car owner}\LAL{dd}{is}\ar[ddrr]_(.35){\parbox{.45in}{\scriptsize owns, as primary,}}\LA{rr}{\parbox{.8in}{\rr\scriptsize yields, insofar as it is a person and owns, as primary, a car,}}&\ar@{}[d]^(.4){\checkmark}&
\obox{P\times C}{1in}{a pair $(x,y)$ where $x$ is a person and $y$ is a car}\ar@/^1pc/[ddll]^(.7){\tn{yields, as }x,}\LA{dd}{\tn{yields, as }y,}\LA{rr}{\parbox{.7in}{\scriptsize has as associated utility}}&&\obox{V}{.8in}{a dollar value}\\&&\\\obox{P}{.5in}{a person}&&\obox{C}{.4in}{a car}
}
$$

\end{example}


\subsection{Coproducts}\label{sec:coproducts}\index{coproducts!of sets}

\begin{definition}\label{def:coproduct}

Let $X$ and $Y$ be sets. The {\em coproduct of $X$ and $Y$}, denoted $X\sqcup Y$,\index{a symbol!$\sqcup$} is defined as the ``disjoint union" of $X$ and $Y$, i.e. the set for which an element is either an element of $X$ or an element of $Y$. If something is an element of both $X$ and $Y$ then we include both copies, and distinguish between them, in $X\sqcup Y$. See Example \ref{ex:coproduct}

There are two natural inclusion functions $i_1\taking X\to X\sqcup Y$ and $i_2\taking Y\to X\sqcup Y$.\index{inclusion functions}\index{coproduct!inclusion functions}
$$\xymatrix@=15pt{X\ar[ddr]_{i_1}&&Y\ar[ddl]^{i_2}\\\\&X\sqcup Y}$$

\end{definition}

\begin{example}\label{ex:coproduct}

The coproduct of $X:=\{a,b,c,d\}$ and $Y:=\{1,2,3\}$ is $$X\sqcup Y\iso\{a,b,c,d,1,2,3\}.$$ The coproduct of $X$ and itself is $$X\sqcup X\iso\{i_1a,i_1b,i_1c,i_1d,i_2a,i_2b,i_2c,i_2d\}$$ 
The names of the elements in $X\sqcup Y$ are not so important. What's important are the inclusion maps $i_1,i_2$, which ensure that we know where each element of $X\sqcup Y$ came from.

\end{example}

\begin{example}[Airplane seats]\label{ex:airplanes}

\begin{align}\label{dia:airplane}
\xymatrix@=15pt{
\obox{X}{.8in}{an economy-class seat in an airplane}\LAL{ddr}{is}&&\obox{Y}{.7in}{a first-class seat in an airplane}\LA{ddl}{is}\\\\
&\obox{X\sqcup Y}{.7in}{a seat in an airplane}
}
\end{align}

\end{example}

\begin{exercise}
Would you say that \fakebox{a phone} is the coproduct of \fakebox{a cellphone} and \fakebox{a landline phone}? 
\end{exercise}

\begin{example}[Disjoint union of dots]\label{ex:coprod of dots}

\begin{align}
\parbox{2.4in}{\begin{center}\small $X\sqcup Y$\vspace{-.1in}\end{center}\fbox{
\xymatrix@=15pt{
\LMO{\clubsuit}&\LMO{1}&\LMO{2}&\LMO{3}&\LMO{4}&\LMO{5}&\LMO{6}\\\LMO{\diamondsuit}\\\LMO{\heartsuit}\\\LMO{\spadesuit}
}}}
\parbox{.9in}{
\xymatrix{~&&\ar[ll]_{i_2}~}
}
\parbox{.3in}{\begin{center}\small $Y$\vspace{-.1in}\end{center}\fbox{
\xymatrix@=15pt{
\LMO{\clubsuit}\\\LMO{\diamondsuit}\\\LMO{\heartsuit}\\\LMO{\spadesuit}
}}}
\\\nonumber
\parbox{1in}{\hspace{-1.4in}\xymatrix{~\\\\\ar[uu]_{i_1}}}
\\\nonumber
\parbox{2.1in}{\hspace{-1.3in}\fbox{
\xymatrix@=15pt{
\LMO{1}&\LMO{2}&\LMO{3}&\LMO{4}&\LMO{5}&\LMO{6}
}}\begin{center}\hspace{-2.6in}\small$X$\end{center}}
\end{align}

\end{example}


\subsubsection{Universal property for coproducts}\index{coproducts!universal property of}

\begin{lemma}[Universal property for coproduct]\label{lemma:up for coprod}

Let $X$ and $Y$ be sets. For any set $A$ and functions $f\taking X\to A$ and $g\taking Y\to A$, there exists a unique function $X\sqcup Y\to A$ such that the following diagram commutes
$$
\xymatrix@=15pt{&A\\\\X\ar[uur]^{\forall f}\ar[ddr]_{i_1}&&Y\ar[uul]_{\forall g}\ar[ddl]^{i_2}\\\\&X\sqcup Y\ar@{-->}[uuuu]^{\exists!}}
$$
We might write the unique function as 
\footnote{We are about to use a two-line symbol, which is a bit unusual. In what follows a certain function $X\sqcup Y\to A$ is being denoted by the symbol $\coprodmap{f}{g}$.}
$$\coprodmap{f}{g}\taking X\sqcup Y\to A.$$

\end{lemma}

\begin{proof}

Suppose given $f,g$ as above. To provide a function $\ell\taking X\sqcup Y\to A$ is equivalent to providing an element $f(m)\in A$ is for each $m\in X\sqcup Y$. We need such a function such that $\ell\circ i_1=f$ and $\ell\circ i_2=g$. But each element $m\in X\sqcup Y$ is either of the form $i_1x$ or $i_2y$, and cannot be of both forms. So we assign 
$$\coprodmap{f}{g}(m)=\begin{cases}f(x)&\tn{if } m=i_1x,\\ g(y) &\tn{if }m=i_2y.\end{cases}$$
This assignment is necessary and sufficient to make all relevant diagrams commute.

\end{proof}

\begin{example}[Airplane seats, continued]

The universal property of coproducts says the following. Any time we have a function $X\to A$ and a function $Y\to A$, we get a unique function $X\sqcup Y\to A$. For example, every economy class seat in an airplane and every first class seat in an airplane is actually {\em in a particular airplane}. Every economy class seat has a price, as does every first class seat.
\begin{align}
\xymatrix{
&\obox{A}{.9in}{a dollar figure}&\\
\obox{X}{.8in}{an economy-class seat in an airplane}\LA{ru}{has as price}\LA{r}{is}\LAL{dr}{is in}&\obox{X\sqcup Y}{.7in}{a seat in an airplane}\ar@{-->}[d]_{\exists!}\ar@{-->}[u]^{\exists!}\ar@{}[ur]|(.35){\checkmark}\ar@{}[dl]|(.35){\checkmark}\ar@{}[dr]|(.35){\checkmark}\ar@{}[ul]|(.35){\checkmark}&\obox{Y}{.7in}{a first-class seat in an airplane}\LAL{l}{is}\LAL{lu}{has as price}\LA{dl}{is in}\\
&\obox{B}{.7in}{an airplane}&
}
\end{align}
The universal property of coproducts formalizes the following intuitively obvious fact:
\begin{quote}
If we know how economy class seats are priced and we know how first class seats are priced, and if we know that every seat is either economy class or first class, then we automatically know how all seats are priced.
\end{quote}
To say it another way (and using the other induced map):
\begin{quote}
If we keep track of which airplane every economy class seat is in and we keep track of which airplane every first class seat is in, and if we know that every seat is either economy class or first class, then we require no additional tracking for any airplane seat whatsoever.
\end{quote}

\end{example}

\begin{application}[Piecewise defined curves]

In science, curves are often defined or considered piecewise. For example in testing the mechanical properties of a material, we might be interested in various regions of \href{http://en.wikipedia.org/wiki/Deformation_(engineering)}{deformation}, such as elastic, plastic, or post-fracture. These are three intervals on which the material displays different kinds of properties. 

For real numbers $a<b\in\RR$, let $[a,b]:=\{x\in\RR\|a\leq x\leq b\}$ denote the closed interval. Given a function $[a,b]\to\RR$ and a function $[c,d]\to\RR$, the universal property of coproducts implies that they extend uniquely to a function $[a,b]\sqcup[c,d]\to\RR$, which will appear as a piecewise defined curve.

Often we are given a curve on $[a,b]$ and another on $[b,c]$, where the two curves agree at the point $b$. This situation is described by pushouts, which are mild generalizations of coproducts; see Section \ref{sec:pushouts}.

\end{application}

\begin{exercise}\label{exc:coprod}

Write the universal property for coproduct in terms of a relationship between the following three sets: $$\Hom_{\Set}(X,A), \hsp \Hom_\Set(Y,A), \hsp \text{and} \hsp\Hom_\Set(X\sqcup Y,A).$$ 
\end{exercise}

\begin{example}\label{ex:coproduct1}

In the following olog the types $A$ and $B$ are disjoint, so the coproduct $C=A\sqcup B$ is just the union. $$\fbox{\xymatrix{\smbox{A}{a person}\LA{r}{is}&\smbox{C=A\sqcup B}{a person or a cat}&\smbox{B}{a cat}\LAL{l}{is}}}$$

\end{example}

\begin{example}\label{ex:coproduct2}

In the following olog, $A$ and $B$ are not disjoint, so care must be taken to differentiate common elements. $$\fbox{\xymatrixnocompile{\obox{A}{.7in}{\rr an animal that can fly}\LA{rr}{labeled ``A" is}&&\obox{C=A\sqcup B}{1.3in}{an animal that can fly (labeled ``A") or an animal that can swim (labeled ``B")}&&\obox{B}{.9in}{\rr an animal that can swim}\LAL{ll}{labeled ``B" is}}}$$  Since ducks can both swim and fly, each duck is found twice in $C$, once labeled as a flyer and once labeled as a swimmer.  The types $A$ and $B$ are kept disjoint in $C$, which justifies the name ``disjoint union."

\end{example}

\begin{exercise}

Understand Example \ref{ex:coproduct2} and see if a similar idea would make sense for particles and waves. Make an olog, and choose your wording in accordance with Rules \ref{rules:types}. How do photons, which exhibit properties of both waves and particles, fit into the coproduct in your olog?

\end{exercise}

\begin{exercise}
Following the section above, ``Ologging products" page \pageref{sec:ologging products}, come up with a naming system for coproducts, the inclusions, and the universal maps. Try it out by making an olog (involving coproducts) discussing the idea that both a .wav file and a .mp3 file can be played on a modern computer. Be careful that your arrows are valid in the sense of Section \ref{sec:invalid aspect}.
\end{exercise}


\section{Finite limits in $\Set$}\label{sec:finite limits}

In this section we discuss what are called {\em limits} of variously-shaped diagrams of sets. We will make all this much more precise when we discuss limits in arbitrary categories in Section \ref{sec:lims and colims in a cat}.


\subsection{Pullbacks}

\begin{definition}[Pullback]\label{def:pullback}\index{pullback!of sets}

Suppose given the diagram of sets and functions below.
\begin{align}\label{dia:fp sets}
\xymatrix{&Y\ar[d]^g\\
X\ar[r]_f&Z}
\end{align}
Its {\em fiber product}\index{fiber product} is the set 
$$X\times_ZY:=\{(x,w,y)\|f(x)=w=g(y)\}.$$ There are obvious projections $\pi_1\taking X\times_ZY\to X$ and $\pi_2\taking X\times_ZY\to Y$ (e.g. $\pi_2(x,w,y)=y$). Note that if $W=X\times_ZY$ then the diagram 
\begin{align}\label{dia:pullback sets}
\xymatrix{W\ullimit\ar[r]^-{\pi_2}\ar[d]_{\pi_1}&Y\ar[d]^g\\
X\ar[r]_f&Z}
\end{align}
commutes. Given the setup of Diagram \ref{dia:fp sets} we define the {\em pullback of $X$ and $Y$ over $Z$} to be any set $W$ for which we have an isomorphism $W\To{\iso}X\times_ZY$. The corner symbol $\lrcorner$ in Diagram \ref{dia:pullback sets} indicates that $W$ is the pullback.\index{a symbol!$\lrcorner$}

\end{definition}

\begin{exercise}
Let $X,Y,Z$ be as drawn and $f\taking X\to Z$ and $g\taking Y\to Z$ the indicated functions. 
\begin{center}
\includegraphics[height=2in]{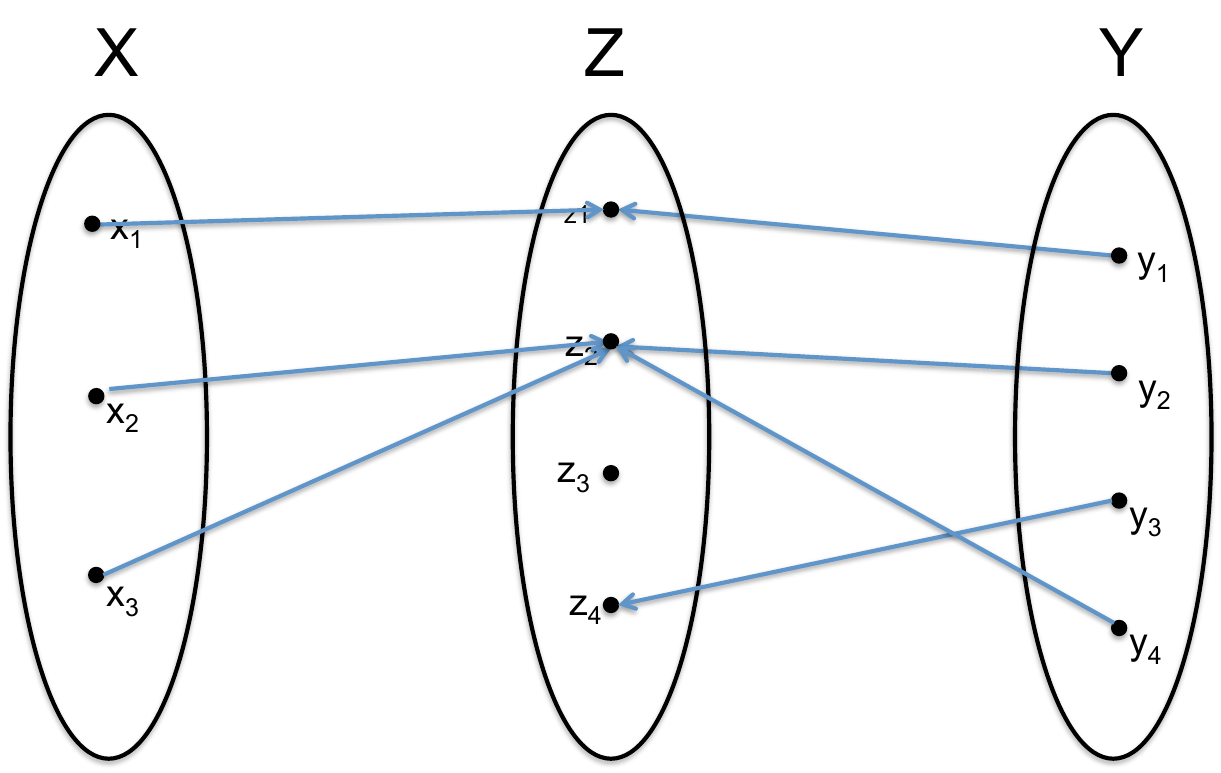}
\end{center}
What is the pullback of the diagram $X\Too{f}Z\Fromm{g}Y$?
\end{exercise}

\begin{exercise}~
\sexc Draw a set $X$ with five elements and a set $Y$ with three elements. Color each element of $X$ and each element of $Y$ either red, blue, or yellow,
\footnote{You can use shadings rather than coloring, if coloring would be annoying.}
and do so in a ``random-looking" way. Considering your coloring of $X$ as a function $X\to C$, where $C=\{\tn{red, blue, yellow}\}$, and similarly obtaining a function $Y\to C$, draw the fiber product $X\times_CY$. Make sure it is colored appropriately.
\next The universal property for products guarantees a function $X\times_CY\to X\times Y$, which I can tell you will be an injection. This means that the drawing you made of the fiber product can be imbedded into the $5\times 3$ grid; please draw the grid and indicate this subset.
\endsexc
\end{exercise}

\begin{remark}

Some may prefer to denote this fiber product by $f\times_Zg$ rather than $X\times_ZY$. The former is  mathematically better notation, but human-readability is often enhanced by the latter, which is also more common in the literature. We use whichever is more convenient.

\end{remark}

\begin{exercise}~
\sexc Suppose that $Y=\emptyset$; what can you say about $X\times_ZY$? 
\next Suppose now that $Y$ is any set but that $Z$ has exactly one element; what can you say about $X\times_ZY$?
\endsexc
\end{exercise}

\begin{exercise}
Let $S=\RR^3, T=\RR$, and think of them as (Aristotelian) space and time, with the origin in $S\times T$ given by the center of mass of MIT at the time of its founding. Let $Y=S\times T$ and let $g_1\taking Y\to S$ be one projection and $g_2\taking Y\to T$ the other projection. Let $X=\singleton$ be a set with one element and let $f_1\taking X\to S$ and $f_2\taking X\to T$ be given by the origin in both cases. 
\sexc What are the fiber products $W_1$ and $W_2$:
$$
\xymatrix{W_1\ar[r]\ar[d]\ullimit&Y\ar[d]^{g_1}\\X\ar[r]_{f_1}&S}\hspace{1in}
\xymatrix{W_2\ar[r]\ar[d]\ullimit&Y\ar[d]^{g_2}\\X\ar[r]_{f_2}&T}
$$
\next Interpret these sets in terms of the center of mass of MIT at the time of its founding.
\endsexc
\end{exercise}


\subsubsection{Using pullbacks to define new ideas from old}

In this section we will see that the fiber product of a diagram can serve to define a new concept. For example, in (\ref{dia:bad battery}) we define what it means for a cellphone to have a bad battery, in terms of the length of time for which it remains charged. By being explicit, we reduce the chance of misunderstandings between different groups of people. This can be useful in situations like audits and those in which one is trying to reuse or understand data gathered by others.

\begin{example}

Consider the following two ologs. The one on the right is the pullback of the one on the left. 
\begin{align}\label{dia:wealthy and loyal}
\fbox{\xymatrixnocompile{&\obox{C}{.7in}{\rr a loyal customer}\LA{d}{is}\\\obox{B}{.7in}{\rr a wealthy customer}\LA{r}{is}&\smbox{D}{a customer}}}\hsp&\fbox{\xymatrix{\obox{A=B\times_DC}{.9in}{\rr a customer that is wealthy and loyal}\LAL{d}{is}\LA{r}{is}&\obox{C}{.7in}{\rr a loyal customer}\LA{d}{is}\\\obox{B}{.7in}{\rr a wealthy customer}\LA{r}{is}&\smbox{D}{a customer}}}
\end{align}
Check from Definition \ref{def:pullback} that the label, ``a customer that is wealthy and loyal", is fair and straightforward as a label for the fiber product $A=B\times_DC$, given the labels on $B,C$, and $D$.

\end{example}

\begin{remark}\label{rem:defining using pullbacks}

Note that in Diagram (\ref{dia:wealthy and loyal}) the top-left box could have been (non-canonically named) \fakebox{a good customer}. If it was taken to be the fiber product, then the author would be effectively {\em defining} a good customer to be one that is wealthy and loyal. 

\end{remark}

\begin{exercise}
For each of the following, an author has proposed that the diagram on the right is a pullback. Do you think their labels are appropriate or misleading; that is, is the label on the upper-left box reasonable given the rest of the olog, or is it suspect in some way?
\sexc\begin{align*}\footnotesize\fbox{\xymatrix{&&\smbox{C}{blue}\LA{d}{is}\\\smbox{B}{a person}\LA{rr}{\parbox{.7in}{\rr has as favorite color}}&&\smbox{D}{a color}}}\hsp&
\footnotesize\fbox{\xymatrixnocompile{\obox{A=B\times_DC}{1.1in}{\rr a person whose favorite color is blue}\LAL{d}{is}\LA{rr}{\parbox{.7in}{\rr has as favorite color}}&&\smbox{C}{blue}\LA{d}{is}\\\smbox{B}{a person}\LA{rr}{\parbox{.7in}{\rr has as favorite color}}&&\smbox{D}{a color}}}
\end{align*}
\next\begin{align*}
\footnotesize\fbox{\xymatrixnocompile{&&\smbox{C}{a woman}\LA{d}{is}\\\smbox{B}{a dog}\LA{rr}{\parbox{.7in}{\rr has as owner}}&&\smbox{D}{a person}}}\hsp&
\footnotesize\fbox{\xymatrixnocompile{\obox{A=B\times_DC}{1in}{\rr a dog whose owner is a woman}\LAL{d}{is}\LA{rr}{\parbox{.7in}{\rr has as owner}}&&\smbox{C}{a woman}\LA{d}{is}\\\smbox{B}{a dog}\LA{rr}{\parbox{.7in}{\rr has as owner}}&&\smbox{D}{a person}}}
\end{align*}
\next\begin{align*}
\footnotesize\fbox{\xymatrixnocompile{&\obox{C}{.5in}{\rr a piece of furniture}\LA{d}{has}\\\obox{B}{.6in}{\rr a space in our house}\LA{r}{has}&\smbox{D}{a width}}}\hsp&
\footnotesize\fbox{\xymatrixnocompile{\obox{A=B\times_DC}{.5in}{\rr a good fit}\LAL{d}{$s$}\LA{r}{$f$}&\obox{C}{.5in}{\rr a piece of furniture}\LA{d}{has}\\\obox{B}{.6in}{\rr a space in our house}\LA{r}{has}&\smbox{D}{a width}}}
\end{align*}
\endsexc
\end{exercise}

\begin{exercise}~
\sexc Consider your olog from Exercise \ref{exc:family olog}. Are any of the commutative squares there actually pullback squares? 
\next Now use ologs with products and pullbacks to define what a brother is and what a sister is (again in a human biological nuclear family), in terms of types such as \fakebox{an offspring of mating pair $(a,b)$}, \fakebox{a person}, \fakebox{a male person}, \fakebox{a female person}, and so on.
\endsexc
\end{exercise}

\begin{definition}[Preimage]\label{def:preimage}

Let $f\taking X\to Y$ be a function and $y\in Y$ an element. The {\em preimage of y under $f$}\index{preimage}, denoted $f^\m1(y)$,\index{a symbol!$f^\m1$} is  the subset $f^\m1(y):=\{x\in X\|f(x)=y\}$. If $Y'\ss Y$ is any subset, the {\em preimage of $Y'$ under $f$}, denoted $f^\m1(Y')$, is the subset $f^\m1(Y')=\{x\in X\|f(x)\in Y'\}$.

\end{definition}

\begin{exercise}
Let $f\taking X\to Y$ be a function and $y\in Y$ an element. Draw a pullback diagram in which the fiber product is isomorphic to the preimage $f^\m1(y)$.
\end{exercise}

\begin{lemma}[Universal property for pullback]\label{lemma:up for fp}

Suppose given the diagram of sets and functions as below.
\begin{align*}
\xymatrix{&Y\ar[d]^u\\
X\ar[r]_t&Z}
\end{align*}
For any set $A$ and commutative solid arrow diagram as below (i.e. functions $f\taking A\to X$ and $g\taking A\to Y$ such that $t\circ f=u\circ g$), 
\begin{align}\label{dia:universal property of fp}
\xymatrix{
&X\times_ZY\ar@/_1pc/[lddd]_{\pi_1}\ar@/^1pc/[rddd]^{\pi_2}\\\\
&A\ar@{-->}[uu]^{\exists!}\ar[dl]_{\forall f}\ar[dr]^{\forall g}&\\
X\ar[rd]_t&&Y\ar[ld]^u\\
&Z&}
\end{align}
there exists a unique arrow $\pb{f}{g}{Z}\taking A\to X\times_ZY$ making everything commute, i.e. 
$$f=\pi_1\circ \pb{f}{g}{Z}\hsp\text{and}\hsp g=\pi_2\circ\pb{f}{g}{Z}.$$

\end{lemma}

\begin{exercise}
Create an olog whose underlying shape is a commutative square. Now add the fiber product so that the shape is the same as that of Diagram (\ref{dia:universal property of fp}). Assign English labels to the projections $\pi_1,\pi_2$ and to the dotted map $A\To{\pb{f}{g}{Z}}X\times_ZY$, such that these labels are as canonical as possible.
\end{exercise}


\subsubsection{Pasting diagrams for pullback}

Consider the diagram drawn below, which includes a left-hand square, a right-hand square, and a big rectangle.
$$
\xymatrix{
A'\ar[r]^{f'}\ar[d]_i\ullimit&B'\ar[r]^{g'}\ar[d]_j\ullimit&C'\ar[d]^k\\
A\ar[r]_f&B\ar[r]_g&C}
$$
The right-hand square has a corner symbol indicating that $B'\iso B\times_CC'$ is a pullback. But the corner symbol on the left is ambiguous; it might be indicating that the left-hand square is a pullback, or it might be indicating that the big rectangle is a pullback. It turns out that if $B'\iso B\times_CC'$ then it is not ambiguous because the left-hand square is a pullback if and only if the big rectangle is.

\begin{proposition}\label{prop:pasting}

Consider the diagram drawn below
$$
\xymatrix{
&B'\ar[r]^{g'}\ar[d]_j\ullimit&C'\ar[d]^k\\
A\ar[r]_f&B\ar[r]_g&C}
$$
where $B'\iso B\times_CC'$ is a pullback. Then there is an isomorphism $A\times_BB'\iso A\times_CC'$. Said another way, $$A\times_B(B\times_CC')\iso A\times_CC'.$$

\end{proposition}

\begin{proof}

We first provide a map $\phi\taking A\times_B(B\times_CC')\to A\times_CC'$. An element of $A\times_B(B\times_CC')$ is of the form $(a,b,(b,c,c'))$ such that $f(a)=b, g(b)=c$ and $k(c')=c$. But this implies that $g\circ f(a)=c=k(c')$ so we put $\phi(a,b,(b,c,c')):=(a,c,c')\in A\times_CC'$. Now we provide a proposed inverse, $\psi\taking A\times_CC'\to A\times_B(B\times_CC')$. Given $(a,c,c')$ with $g\circ f(a)=c=k(c')$, let $b=f(a)$ and note that $(b,c,c')$ is an element of $B\times_CC'$. So we can define $\psi(a,c,c')=(a,b,(b,c,c'))$. It is easy to see that $\phi$ and $\psi$ are inverse.
 
\end{proof}

Proposition \ref{prop:pasting} can be useful in authoring ologs. For example, the type \fakebox{a cellphone that has a bad battery} is vague, but we can lay out precisely what it means using pullbacks:
\small
\begin{align}\label{dia:bad battery}
\fbox{\xymatrixnocompile{\obox{A\iso B\times_DC}{1in}{a cellphone that has a bad battery}\ar[r]\ar[d]&\smbox{C\iso D\times_FE}{a bad battery}\ar[r]\ar[d]&\obox{E\iso F\times_HG}{.5in}{less than 1 hour}\ar[r]\ar[d]&\obox{G}{.5in}{between 0 and 1}\ar[d]\\\smbox{B}{a cellphone}\LA{r}{has}&\smbox{D}{a battery}\LA{r}{\parbox{.4in}{\rr remains charged for}}&\obox{F}{.6in}{a duration of time}\LA{r}{\hspace{.07in}\parbox{.4in}{\rr in hours yields}}&\obox{H}{.6in}{a range of numbers}}}
\end{align}\normalsize

The category-theoretic fact described above says that since $A\iso B\times_DC$ and $C\iso D\times_FE$, it follows that $A\iso B\times_FE$.  That is, we can deduce the definition ``a cellphone that has a bad battery is defined as a cellphone that has a battery which remains charged for less than one hour."  

\begin{exercise}~
\sexc Create an olog that defines two people to be ``of approximately the same height" if and only if their height difference is less than half an inch, using a pullback. Your olog can include the box \fakebox{a real number $x$ such that $-.5<x<.5$}. 
\next In the same olog, make a box for those people whose height is approximately the same as a person named ``The Virgin Mary". You may need to use images, as in Section \ref{sec:images}.
\endsexc
\end{exercise}

\begin{exercise}\label{exc:pointwise map of fp}
Consider the diagram on the left below, where both squares commute. 
$$
\xymatrix@=15pt{
&&&Y'\ar[dd]\\
&&Y\ar[ru]\ar[dd]\\
&X'\ar'[r][rr]&&Z'\\
X\ar[rr]\ar[ru]&&Z\ar[ru]
}
\hspace{1in}
\xymatrix@=15pt{
&W'\ar[rr]\ar'[d][dd]\ullimit&&Y'\ar[dd]\\
W\ar[rr]\ar[dd]\ullimit&&Y\ar[ru]\ar[dd]\\
&X'\ar'[r][rr]&&Z'\\
X\ar[rr]\ar[ru]&&Z\ar[ru]
}
$$
Let $W=X\times_ZY$ and $W'=X'\times_{Z'}Y'$, and form the diagram to the right. Use the universal property of fiber products to construct a map $W\to W'$ such that all squares commute.
\end{exercise}


\subsection{Spans, experiments, and matrices}

\begin{definition}\label{def:span}\index{span}

Given sets $A$ and $B$, a {\em span on $A$ and $B$} is a set $R$ together with functions $f\taking R\to A$ and $g\taking R\to B$. 
$$\xymatrix@=15pt{&R\ar[ddl]_f\ar[ddr]^g\\\\A&&B}$$

\end{definition}

\begin{application}\label{app:exp temp press}

Think of $A$ and $B$ as observables and $R$ as a set of experiments performed on these two variables. For example, let's say $T$ is the set of possible temperatures of a \href{http://en.wikipedia.org/wiki/Ideal_gas_law}{\text gas} in a fixed container and let's say $P$ is the set of possible pressures of the gas. We perform 1000 experiments in which we change and record the temperature and we simultaneously also record the pressure; this is a span $T\From{f}E\To{g}P$. The results might look like this:
$$
\begin{tabular}{| l || l | l |}
\bhline
\multicolumn{3}{| c |}{Experiment}\\\bhline
{\bf ID}&{\bf Temperature}&{\bf Pressure}\\\bbhline
1&100& 72\\\hline
2&100&73\\\hline
3&100&72\\\hline
4&200&140\\\hline
5&200&138\\\hline
6&200&141\\\hline
\vdots&\vdots&\vdots\\\bhline
\end{tabular}
$$

\end{application}

\begin{definition}\label{def:composite span}

Let $A,B,$ and $C$ be sets, and let $A\From{f}R\To{g}B$ and $B\From{f'}R'\To{g'}C$ be spans. Their {\em composite span}\index{span!composite} is given by the fiber product $R\times_BR'$ as in the diagram below:
$$
\xymatrix@=10pt{&&R\times_BR'\ar[ldd]\ar[rdd]\\\\&R\ar[ddl]_f\ar[ddr]^g&&R'\ar[ddl]_{f'}\ar[ddr]^{g'}\\\\A&&B&&C
}$$

\end{definition}

\begin{application}\label{app:exp temp press 2}

Let's look back at our lab's experiment from Application \ref{app:exp temp press}, which resulted in a span $T\From{f}E\To{g}P$. Suppose we notice that something looks a little wrong. The pressure should be linear in the temperature but it doesn't appear to be. We hypothesize that the volume of the container is increasing under pressure. We look up this container online and see that experiments have been done to measure the volume as the interior pressure changes. The data has generously been made available online, which gives us a span $P\From{f'}E'\To{g'}V$. 

The composite of our lab's span with the online data span yields a span $T\from E''\to V$, where $E'':=E\times_PE'$. What information does this span give us? In explaining it, one might say ``whenever an experiment in our lab yielded the same pressure as one they recorded, let's call that a data point. Every data point has an associated temperature (from our lab) and an associated volume (from their experiment). This is the best we can do." 

The information we get this way might be seen by some as unscientific, but it certainly is the kind of information people use in business and in every day life calculation---we get our data from multiple sources and put it together. Moreover, it is scientific in the sense that it is reproducible. The way we obtained our $T$-$V$ data is completely transparent.

\end{application}

We can relate spans to matrices of natural numbers, and see a natural ``categorification" of matrix addition and matrix multiplication. If our spans come from experiments as in Applications \ref{app:exp temp press} and \ref{app:exp temp press 2} the matrices involved will look like huge but sparse matrices. Let's go through that.

Let $A$ and $B$ be sets and let $A\from R\to B$ be a span. By the universal property of products, we have a unique map $R\To{p}A\times B$. 

We make a matrix of natural numbers out of this data as follows. The set of rows is $A$, the set of columns is $B$. For elements $a\in A$ and $b\in B$, the $(a,b)$-entry is the cardinality of its preimage, $|p^\m1(a,b)|$, i.e. the number of elements in $R$ that are sent by $p$ to $(a,b)$. 

Suppose we are given two $(A,B)$-spans, i.e. $A\from R\to B$ and $A\from R'\to B$; we might think of these has having the same {\em dimensions}, i.e. they are both $|A|\times|B|$-matrices. We can take the disjoint union $R\sqcup R'$ and by the universal property of coproducts we have a unique span $A\from R\sqcup R'\to B$ making the requisite diagram commute.
\footnote{
$$\xymatrix{
&R\ar[dl]\ar[dr]\ar[d]\\
A&R\sqcup R'\ar[l]\ar[r]&B\\
&R'\ar[ur]\ar[ul]\ar[u]}
$$
}
The matrix corresponding to this new span will be the sum of the matrices corresponding to the two previous spans out of which it was made.

Given a span $A\from R\to B$ and a span $B\from S\to C$, the composite span can be formed as in Definition \ref{def:composite span}. It will correspond to the usual multiplication of matrices.

\begin{construction}\label{const:bipartite}\index{graph!bipartite}

Given a span $A\From{f} R\To{g} B$, one can draw a {\em bipartite graph} with each element of $A$ drawn as a dot on the left, each element of $B$ drawn as a dot on the right, and each element $r\in R$ drawn as an arrow connecting vertex $f(r)$ on the left to vertex $g(r)$ on the right.

\end{construction}

\begin{exercise}~
\sexc Draw the bipartite graph (as in Construction \ref{const:bipartite}) corresponding to the span $T\From{f}E\To{g}P$ in Application \ref{app:exp temp press}.
\next Now make up your own span $P\From{f'}E'\To{g'}V$ and draw it. Finally, draw the composite span below. 
\next Can you say how the composite span graph relates to the graphs of its factors?
\endsexc
\end{exercise}


\subsection{Equalizers and terminal objects}

\begin{definition}\label{def:equalizer}\index{equalizer}

Suppose given two parallel arrows 
\begin{align}\label{dia:equalizer}
\xymatrix{X\ar@<.5ex>[r]^f\ar@<-.5ex>[r]_g&Y.}\hspace{1in}\xymatrix{Eq(f,g)\ar[r]^-p&X\ar@<.5ex>[r]^f\ar@<-.5ex>[r]_g&Y}
\end{align}
The {\em equalizer of $f$ and $g$} is the commutative diagram as to the right in (\ref{dia:equalizer}), where we define $$Eq(f,g):=\{x\in X\|f(x)=g(x)\}$$ and where $p$ is the canonical inclusion.

\end{definition}

\begin{example}

Suppose one has designed an experiment to test a theoretical prediction. The question becomes, ``when does the theory match the experiment?" The answer is given by the equalizer of the following diagram:
$$\xymatrix{
\obox{}{.5in}{an input}\ar@<1ex>[rr]^{\tn{should, according to theory, yield}}\ar@<-1ex>[rr]_{\tn{according to experiment yields}}&\hspace{1in}&\obox{}{.6in}{an output}
}$$
The equalizer is the set of all inputs for which the theory and the experiment yield the same output.

\end{example}

\begin{exercise}
Come up with an olog that uses equalizers in a reasonably interesting way. Alternatively, use an equalizer to specify those published authors who have published exactly one paper. Hint: find a function from authors to papers; then find another.
\end{exercise}

\begin{exercise}
Find a universal property enjoyed by the equalizer of two arrows, and present it in the style of Lemmas \ref{lemma:up for prod}, \ref{lemma:up for coprod}, and \ref{lemma:up for fp}.
\end{exercise}

\begin{exercise}\index{terminal object!in $\Set$}~
\sexc A terminal set is a set $S$ such that for every set $X$, there exists a unique function $X\to S$. Find a terminal set. 
\next Do you think that the notion {\em terminal set} belongs in this section (Section \ref{sec:finite limits})? How so? If products, pullbacks, and equalizers are all limits, what do limits have in common?
\endsexc
\end{exercise}


\section{Finite colimits in $\Set$}\label{sec:finite colimits}

This section will parallel Section \ref{sec:finite limits}---I will introduce several types of finite colimits and hope that this gives the reader some intuition about them, without formally defining them yet. Before doing so, I must define equivalence relations and quotients.


\subsection{Background: equivalence relations}\index{equivalence relation}\index{relation!equivalence}

\begin{definition}[Equivalence relations and equivalence classes]

Let $X$ be a set. An {\em equivalence relation on $X$} is a subset $R\ss X\times X$ satisfying the following properties for all $x,y,z\in X$:
\begin{description}
\item[Reflexivity:] $(x,x)\in R$;
\item[Symmetry:] $(x,y)\in R$ if and only if $(y,x)\in R$; and
\item[Transitivity:] if $(x,y)\in R$ and $(y,z)\in R$ then $(x,z)\in R$.
\end{description}
If $R$ is an equivalence relation, we often write $x\sim_R y$, or simply $x\sim y$, to mean $(x,y)\in R$. For convenience we may refer to the equivalence relation by the symbol $\sim$, saying that $\sim$ is an equivalence relation on $X$.\index{a symbol!$\sim$}

An {\em equivalence class of $\sim$}\index{equivalence relation!equivalence classes} is a subset $A\ss X$ such that
\begin{itemize}
\item $A$ is nonempty, $A\neq\emptyset$;
\item if $x\in A$ and $x'\in A$, then $x\sim x'$; and 
\item if $x\in A$ and $x\sim y$, then $y\in A$.
\end{itemize}
Suppose that $\sim$ is an equivalence relation on $X$. The {\em quotient of $X$ by $\sim$}\index{equivalence relation!quotient by}, denoted $X/\sim$\index{a symbol!$X/\sim$} is the set of equivalence classes of $\sim$.

\end{definition}

\begin{example}

Let $\ZZ$ denote the set of integers. Define a relation $R\ss\ZZ\times\ZZ$ by $$R=\{(x,y)\|\exists n\in\ZZ \tn{ such that } x+7n=y\}.$$ Then $R$ is an equivalence relation because $x+7*0=x$ (reflexivity); $x+7*n=y$ if and only if $y+7*(-n)= x$ (symmetry); and $x+7n=y$ and $y+7m=z$ together imply that $x+7(m+n)=z$ (transitivity).

\end{example}

\begin{exercise}
Let $X$ be the set of people on earth; define a binary relation $R\ss X\times X$ on $X$ as follows. For a pair $(x,y)$ of people, say $(x,y)\in R$ if $x$ spends a lot of time thinking about $y$. 
\sexc Is this relation reflexive? 
\next Is it symmetric? 
\next Is it transitive?
\endsexc
\end{exercise}

\begin{example}[Partitions]\label{ex:partition}

An equivalence relation on a set $X$ can be thought of as a way of partitioning $X$. A {\em partition of $X$}\index{equivalence relation!as partition} consists of a set $I$, called {\em the set of parts}, and for every element $i\in I$ a subset $X_i\ss X$ such that two properties hold:
\begin{itemize}
\item every element $x\in X$ is in some part (i.e. for all $x\in X$ there exists $i\in I$ such that $x\in X_i$); and
\item no element can be found in two different parts (i.e. if $x\in X_i$ and $x\in X_j$ then $i=j$).
\end{itemize}

Given a partition of $X$, we define an equivalence relation $\sim$ on $X$ by saying $x\sim x'$ if $x$ and $x'$ are in the same part (i.e. if there exists $i\in I$ such that $x,x'\in X_i$). The parts become the equivalence classes of this relation. Conversely, given an equivalence relation, one makes a partition on $X$ by taking $I$ to be the set of equivalence classes and for each $i\in I$ letting $X_i$ be the elements in that equivalence class.

\end{example}

\begin{exercise}
Let $X$ and $B$ be sets and let $f\taking X\to B$ be a function. Define a subset $R\ss X\times X$ by $$R=\{(x,y)\|f(x)=f(y)\}.$$ 
\sexc Is $R$ an equivalence relation? 
\next Are all equivalence relations on $X$ obtainable in this way (as the fibers of some function having domain $X$)?
\next Does this viewpoint on equivalence classes relate to that of Example \ref{ex:partition}?
\endsexc
\end{exercise}

\begin{exercise}
Take a set $I$ of sets; i.e. suppose that for each element $i\in I$ you are given a set $X_i$. For every two elements $i,j\in I$ say that $i\sim j$ if $X_i$ and $X_j$ are isomorphic. Is this relation an equivalence relation on $I$?  
\end{exercise}

\begin{lemma}[Generating equivalence relations]\label{lemma:generating ERs}

Let $X$ be a set and $R\ss X\times X$ a subset. There exists a relation $S\ss X\times X$ such that
\begin{itemize}
\item $S$ is an equivalence relation,
\item $R\ss S$, and
\item for any equivalence relation $S'$ such that $R\ss S'$, we have $S\ss S'$.
\end{itemize}
The relation $S'$ will be called {\em the equivalence relation generated by $R$}.\index{equivalence relation!generated}

\end{lemma}

\begin{proof}

Let $L_R$ be the set of all equivalence relations on $X$ that contain $R$; in other words, each element $\ell\in L_R$ is an equivalence relation, $\ell\in X\times X$. The set $L_R$ is non-empty because $X\times X\ss X\times X$ is an equivalence relation. Let $S$ denote the set of pairs $(x_1,x_2)\in X\times X$ that appear in every element of $L_R$. Note that $R\ss S$ by definition. We need only show that $S$ is an equivalence relation.

It is clearly reflexive, because $R$ is. If $(x,y)\in S$ then $(x,y)\in\ell$ for all $\ell\in L_R$. But since each $\ell$ is an equivalence relation, $(y,x)\in\ell$ too, so $(y,x)\in S$. This shows that $S$ is symmetric. The proof that it is transitive is similar: if $(x,y)\in S$ and $(y,z)\in S$ then they are both in each $\ell$ which puts $(x,z)$ in each $\ell$, which puts it in $S$.

\end{proof}

\begin{remark}

Let $X$ be a set and $R\ss X\times X$ a relation. The proof of Lemma \ref{lemma:generating ERs} has the benefit of working even if $|X|\geq\infty$, but it has the cost that it is not very intuitive, nor useful in practice when $X$ is finite. The intuitive way to think about the idea of equivalence relation generated by $R$ is as follows.
\begin{enumerate}
\item First add to $R$ what is demanded by reflexivity, $R_1:=R\cup\{(x,x)\|x\in X\}$.
\item Then add to $R$ what is demanded by symmetry, $R_2:=R_1\cup\{(x,y)\|(y,x)\in R_1\}.$
\item Finally, add to $R$ what is demanded by transitivity, $$S=R_2\cup\{(x,z)\|(x,y)\in R_2, \tn{ and } (y,z)\in R_2\}.$$
\end{enumerate}

\end{remark}

\begin{exercise}
Consider the set $\RR$ of real numbers. Draw the coordinate plane $\RR\times\RR$, give it coordinates $x$ and $y$. A binary relation on $\RR$ is a subset $S\ss\RR\times\RR$, which can be drawn as a set of points in the plane. 
\sexc Draw the relation $\{(x,y)\|y=x^2\}$. 
\next Draw the relation $\{(x,y)\|y\geq x^2\}.$
\next Let $S_0$ be the equivalence relation on $\RR$ generated (in the sense of Lemma \ref{lemma:generating ERs}) by the empty set. Draw $S$ as a subset of the plane.
\next Consider the equivalence relation $S_1$ generated by $\{(1,2),(1,3)\}$. Draw $S_1$ in the plane. Highlight the equivalence class containing $(1,2)$.
\next The reflexivity property and the symmetry property have pleasing visualizations in $\RR\times\RR$; what are they? 
\next Is there a nice heuristic for visualizing the transitivity property?
\endsexc
\end{exercise}

\begin{exercise}
Consider the binary relation $R=\{(n,n+1)\|n\in\ZZ\}\ss\ZZ\times\ZZ$. 
\sexc What is the equivalence relation generated by $R$? 
\next How many equivalence classes are there?
\endsexc
\end{exercise}

\begin{exercise}
Suppose $N$ is a network (or graph). Let $X$ be the nodes of the network, and let $R\ss X\times X$ denote the relation such that $(x,y)\in R$ iff there exists an arrow connecting $x$ to $y$.
\footnote{The word {\em iff} means ``if and only if". In this case we are saying that the pair $(x,y)$ is in $R$ if and only if there exists an arrow connecting $x$ and $y$.\index{iff}}
\sexc What is the equivalence relation $\sim$ generated by $R$? 
\next What is the quotient $X/\sim$?
\endsexc
\end{exercise}


\subsection{Pushouts}\label{sec:pushouts}

\begin{definition}[Pushout]\label{def:pushout}

Suppose given the diagram of sets and functions below:
\begin{align}\label{dia:pushout}
\xymatrix{W\ar[r]^f\ar[d]_g&X\\Y}
\end{align}
Its {\em fiber sum},\index{fiber sum} denoted $X\sqcup_WY$, is defined as the quotient of $X\sqcup W\sqcup Y$ by the equivalence relation $\sim$ generated by $w\sim f(w)$ and $w\sim g(w)$ for all $w\in W$.
$$X\sqcup_WY:=(X\sqcup W\sqcup Y)/\sim \hsp\tn{where } \forall w\in W,\;\;  w\sim f(w)\;\;\tn{ and }\;\; w\sim g(w).$$ 
There are obvious inclusions $i_1\taking X\to X\sqcup_WY$ and $i_2\taking Y\to X\sqcup_WY$.
\footnote{Note that our term inclusions is not too good, because it seems to suggest that $i_1$ and $i_2$ are injective (see Definition \ref{def:inj,surj,bij}) and this is not always the case.}
Note that if $Z=X\sqcup_WY$ then the diagram
\begin{align}\label{dia:pushout sets}
\xymatrix{W\ar[r]^g\ar[d]_f&Y\ar[d]^{i_2}\\X\ar[r]_-{i_1}&Z\lrlimit}
\end{align} 
commutes. Given the setup of Diagram \ref{dia:pushout} we define the {\em pushout of $X$ and $Y$ over $W$} to be any set $Z$ for which we have an isomorphism $Z\To{\iso}X\sqcup_WY$. The corner symbol $\ulcorner$ in Diagram \ref{dia:pushout sets} indicates that $Z$ is the pushout.\index{a symbol!$\ulcorner$}

\end{definition}

\begin{example}

Let $X=\{x\in\RR\|0\leq x\leq1\}$ be the set of numbers between 0 and 1, inclusive, let $Y=\{y\in\RR\|1\leq y\leq 2\}$ by the set of numbers between 1 and 2, inclusive, and let $W=\{1\}$. Then the pushout $X\From{f} W\To{g} Y$, where $f$ and $g$ are the ``obvious" functions ($1\mapsto 1$) is $X\sqcup_WY\iso\{z\in\RR\|0\leq z\leq 2\}$, as expected. When we eventually get to general colimits, one can check that the whole real line can be made by patching together intervals in this way.

\end{example}

\begin{example}[Pushout]\label{ex:pushout}

In each example below, the diagram to the right is intended to be a pushout of the diagram to the left.  The new object, $D$, is the union of $B$ and $C$, but instances of $A$ are equated to their $B$ and $C$ aspects.  This will be discussed after the two diagrams.

\begin{align}
\label{dia:po1}\fbox{\xymatrixnocompile{\obox{A}{.7in}{a cell in the shoulder}\LA{r}{is}\LAL{d}{is}&\obox{C}{.6in}{a cell in the arm}\\\obox{B}{.7in}{a cell in the torso}}}\hsp&\fbox{\xymatrix{\obox{A}{.7in}{a cell in the shoulder}\LA{r}{is}\LAL{d}{is}&\obox{C}{.6in}{a cell in the arm}\LA{d}{}\\\obox{B}{.7in}{a cell in the torso}\LA{r}{}&\obox{D=B\sqcup_AC}{.8in}{a cell in the torso or arm}}}
\end{align}
In the left-hand olog (\ref{dia:po1}, the two arrows are inclusions: the author considers every cell in the shoulder to be both in the arm and in the torso. The pushout is then just the union, where cells in the shoulder are not double-counted.

\begin{align}\label{dia:po2}\fbox{\xymatrixnocompile@=18pt{\obox{A}{.8in}{\rr a college mathematics course}\LA{r}{yields}\LAL{d}{is}&\obox{C}{.8in}{an utterance of the phrase ``too hard"}\\\obox{B}{.6in}{\rr a college course}}}\hsp&\fbox{\xymatrixnocompile@=18pt{\obox{A}{.8in}{\rr a college mathematics course}\LA{r}{yields}\LAL{d}{is}&\obox{C}{.8in}{an utterance of the phrase ``too hard"}\LA{d}{}\\\obox{B}{.6in}{\rr a college course}\LA{r}{}&\obox{\parbox{.6in}{\vspace{.1in}\tiny$D=B\!\sqcup_A\!C$}}{1in}{\rr a college course, where every mathematics course is replaced by an utterance of the phrase ``too hard"}}}
\end{align}

In Olog (\ref{dia:po1}), the shoulder is seen as part of the arm and part of the torso.  When taking the union of these two parts, we do not want to ``double-count" the shoulder (as would be done in the coproduct $B\sqcup C$, see Example \ref{ex:coproduct2}).  Thus we create a new type $A$ for cells in the shoulder, which are considered the same whether viewed as cells in the arm or cells in the torso.  In general, if one wishes to take two things and glue them together, with $A$ as the glue and with $B$ and $C$ as the two things to be glued, the union is the pushout $B\sqcup_AC$. (A nice image of this can be seen in the setting of topological spaces, see Example \ref{ex:pushout in Top}.)

In Olog (\ref{dia:po2}), if every mathematics course is simply ``too hard," then when reading off a list of courses, each math course will not be read aloud but simply read as ``too hard."  To form $D$ we begin by taking the union of $B$ and $C$, and then we consider everything in $A$ to be the same whether one looks at it as a course or as the phrase ``too hard."  The math courses are all blurred together as one thing.  Thus we see that the power to equate different things can be exercised with pushouts.

\end{example}

\begin{exercise}
Let $W,X,Y$ be as drawn and $f\taking W\to X$ and $g\taking W\to Y$ the indicated functions. 
\begin{center}
\includegraphics[height=2in]{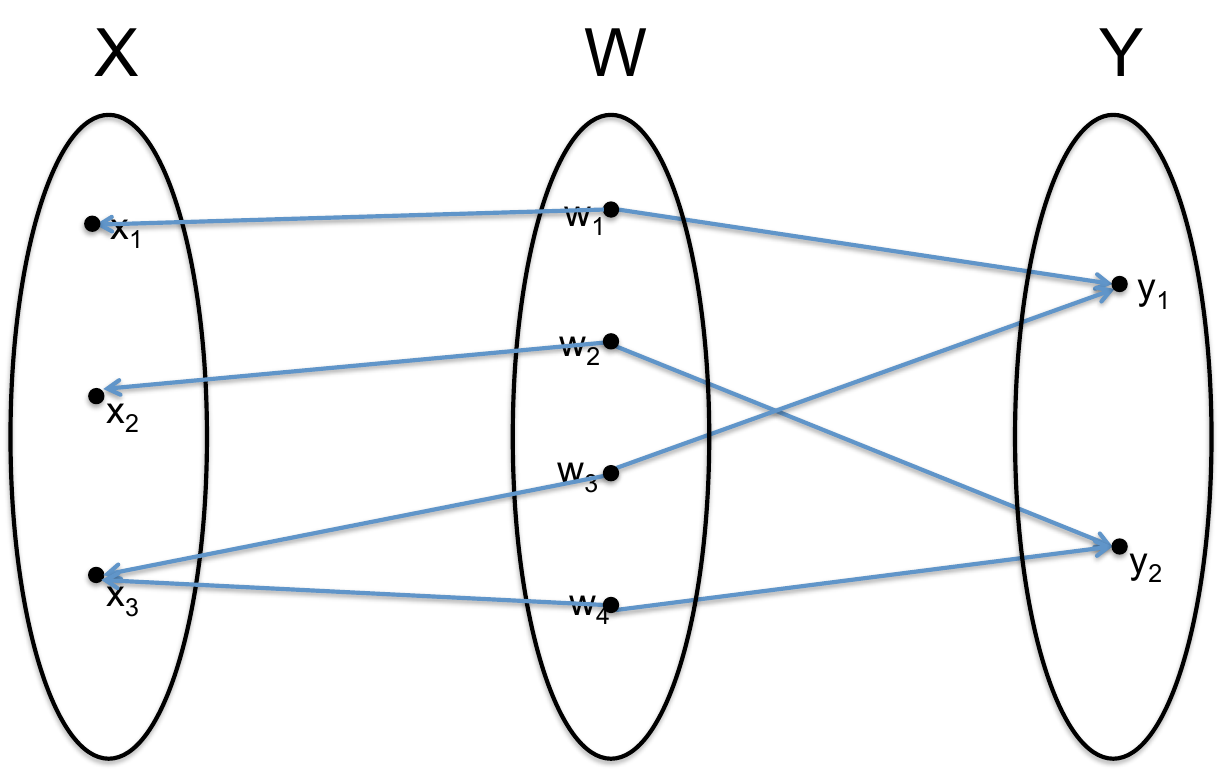}
\end{center}
The pushout of the diagram $X\Fromm{f}W\Too{g}Y$ is a set $P$. Write down the cardinality of $P\iso\ul{n}$ as a natural number $n\in\NN$.  
\end{exercise}

\begin{exercise}
Suppose that $W=\emptyset$; what can you say about $X\sqcup_WZ$? 
\end{exercise}

\begin{exercise}
Let $W:=\NN=\{0,1,2,\ldots\}$ denote the set of natural numbers, let $X=\ZZ$ denote the set of integers, and let $Y=\singleton$ denote a one-element set. Define $f\taking W\to X$ by $f(w)= -(w+1)$, and define $g\taking W\to Y$ to be the unique map. Describe the set $X\sqcup_WY$.
\end{exercise}

\begin{exercise}
Let $i\taking R\ss X\times X$ be an equivalence relation (see Example \ref{ex:subset as function} for notation). Composing with the projections $\pi_1,\pi_2\taking X\times X\to X$, we have two maps $\pi_1\circ i,\taking R\to X$ and $\pi_2\circ i\taking R\to X$. 
\sexc What is the pushout $$X\From{\pi_1\circ i}R\To{\pi_2\circ i}X?$$ 
\next If $i\taking R\ss X\times X$ is not assumed to be an equivalence relation, we can still define the pushout above. Is there a relationship between the pushout $X\From{\pi_1\circ i}R\To{\pi_2\circ i}X$ and the equivalence relation generated by $R\ss X\times X$?
\endsexc
\end{exercise}

\begin{lemma}[Universal property for pushout]\label{lemma:up for po}

Suppose given the diagram of sets and functions as below.
\begin{align*}
\xymatrix{W\ar[r]^u\ar[d]_t&Y\\
X}
\end{align*}
For any set $A$ and commutative solid arrow diagram as below (i.e. functions $f\taking X\to A$ and $g\taking Y\to A$ such that $f\circ t=g\circ u$), 
\begin{align}\label{dia:universal property of po}
\xymatrix{
&W\ar[dr]^u\ar[dl]_t\\
X\ar@/_1pc/[dddr]_{i_1}\ar[rd]_f&&Y\ar@/^1pc/[dddl]^{i_2}\ar[dl]^g\\
&A\\\\
&X\sqcup_WY\ar@{-->}[uu]^{\exists!}
}
\end{align}
there exists a unique arrow $\po{f}{g}{W}\taking X\sqcup_WY\to A$ making everything commute, $$f=\po{f}{g}{W}\circ i_1\hsp\text{and}\hsp g=\po{f}{g}{W}\circ i_2.$$

\end{lemma}


\subsection{Other finite colimits}

\begin{definition}\label{def:coequalizer}[Coequalizer]\index{coequalizer}

Suppose given two parallel arrows 
\begin{align}\label{dia:coequalizer}
\xymatrix{X\ar@<.5ex>[r]^f\ar@<-.5ex>[r]_g&Y.}\hspace{1in}\xymatrix{X\ar@<.5ex>[r]^f\ar@<-.5ex>[r]_g&Y\ar[r]^-q&Coeq(f,g)}
\end{align}
The {\em coequalizer of $f$ and $g$} is the commutative diagram as to the right in (\ref{dia:coequalizer}), where we define $$Coeq(f,g):=Y\;/\;f(x)\sim g(x)$$ i.e. the coequalizer of $f$ and $g$ is the quotient of $Y$ by the equivalence relation generated by $\{(f(x),g(x))\|x\in X\}\ss Y\times Y$

\end{definition}

\begin{exercise}
Let $X=\RR$ be the set of real numbers. What is the coequalizer of the two maps $X\to X$ given by $x\mapsto x$ and $x\mapsto (x+1)$ respectively?
\end{exercise}

\begin{exercise}
Find a universal property enjoyed by the coequalizer of two arrows.
\end{exercise}

\begin{exercise}[Initial object]\label{exc:initial set}
An initial set is a set $S$ such that for every set $A$, there exists a unique function $S\to A$. 
\sexc Find an initial set. 
\next Do you think that the notion {\em initial set} belongs in this section (Section \ref{sec:finite colimits})? How so? If coproducts, pushouts, and coequalizers are all colimits, what do colimits have in common?
\endsexc
\end{exercise}


\section{Other notions in $\Set$}

In this section we discuss some left-over notions in the category of Sets.


\subsection{Retractions}

\begin{definition}

Suppose we have a function $f\taking X\to Y$ and a function $g\taking Y\to X$ such that $g\circ f=\id_X$. In this case we call $f$ a {\em retract section} and we call $g$ a {\em retract projection}. \index{retraction}

\end{definition}

\begin{exercise}
Create an olog that includes sets $X$ and $Y$, and functions $f\taking X\to Y$ and $g\taking Y\to X$ such that $g\circ f=\id_X$ but such that $f\circ g\neq\id_Y$; that is, such that $f$ is a retract section but not an isomorphism.
\end{exercise}


\subsection{Currying}\label{sec:currying}\index{currying}\index{materials!force extension curves}

Currying is the idea that when a function takes many inputs, we can input them one at a time or all at once. For example, consider the function that takes a material $M$ and an extension $E$ and returns the force transmitted through the material when it is pulled to that extension. This is a function $e\taking \fakebox{a material}\times\fakebox{an extension}\to\fakebox{a force}$. This function takes two inputs at once, but it is convenient to ``curry" the second input. Recall that $\Hom_\Set(\fakebox{an extension},\fakebox{a force})$ is the set of theoretical force-extension curves. Currying transforms $e$ into a function $$e'\taking\fakebox{a material}\to\Hom_\Set(\fakebox{an extension},\fakebox{a force}).$$ This is a more convenient way to package the same information. 

In fact, it may be convenient to repackage this information another way. For any extension, we may want the function that takes a material and returns how much force it can transmit at that extension. This is a function $$e''\taking\fakebox{an extension}\to\Hom_\Set(\fakebox{a material},\fakebox{a force}).$$ 

\begin{notation}\index{exponentials ! in $\Set$}
Let $A$ and $B$ be sets. We sometimes denote the set of functions from $A$ to $B$ by 
\begin{align}\label{dia:exponential sets}
B^A:=\Hom_\Set(A,B).
\end{align}
\end{notation}

\begin{exercise}
For a finite set $A$, let $|A|\in\NN$ denote the cardinality of (number of elements in) $A$. If $A$ and $B$ are both finite (including the possibility that one or both are empty), is it always true that $|B^A|=|B|^{|A|}$?
\end{exercise}

\begin{proposition}[Currying]\label{prop:curry}

Let $A$ denote a set. For any sets $X,Y$ there is a bijection 
\begin{align}\label{dia:curry bijection}
\phi\taking\Hom_\Set(X\times A,Y)\To{\iso}\Hom_\Set(X,Y^A).
\end{align}

\end{proposition}

\begin{proof}

Suppose given $f\taking X\times A\to Y$. Define $\phi(f)\taking X\to Y^A$ as follows: for any $x\in X$ let $\phi(f)(x)\taking A\to Y$ be defined as follows: for any $a\in A$, let $\phi(f)(x)(a):=f(x,a)$. 

We now construct the inverse, $\psi\taking\Hom_\Set(X,Y^A)\to\Hom_\Set(X\times A,Y)$. Suppose given $g\taking X\to Y^A$. Define $\psi(g)\taking X\times A\to Y$ as follows: for any pair $(x,a)\in X\times A$ let $\psi(g)(x,a):=g(x)(a)$. 

Then for any $f\in\Hom_\Set(X\times A,Y)$ we have $\psi\circ\phi(f)(x,a)=\phi(f)(x)(a)=f(x,a)$, and for any $g\in\Hom_\Set(X,Y^A)$ we have $\phi\circ\psi(g)(x)(a)=\psi(g)(x,a)=g(x)(a)$, Thus we see that $\phi$ is an isomorphism as desired.

\end{proof}

\begin{exercise}
Let $X=\{1,2\}, A=\{a,b\}$, and $Y=\{x,y\}$. 
\sexc\label{part:three distinct} Write down three distinct elements of $L:=\Hom_\Set(X\times A,Y)$. 
\next Write down all the elements of $M:=\Hom_\Set(A,Y)$. 
\next For each of the three elements $\ell\in L$ you chose in part (\ref{part:three distinct}), write down the corresponding function $\phi(\ell)\taking X\to M$ guaranteed by Proposition \ref{prop:curry}.
\endsexc
\end{exercise}

\begin{exercise}\label{exc:evaluation}\index{exponentials!evaluation of}
Let $A$ and $B$ be sets. We know that $\Hom_\Set(A,B)=B^A$, so we have a function $\id_{B^A}\taking\Hom_\Set(A,B)\to B^A$. Look at Proposition \ref{prop:curry}, making the substitutions $X=\Hom_\Set(A,B)$, $Y=B$, and  $A=A$. Consider the function $$\phi^\m1\taking\Hom_\Set(\Hom_\Set(A,B),B^A)\to\Hom_\Set(\Hom_\Set(A,B)\times A,B)$$ obtained as the inverse of (\ref{dia:curry bijection}). We have a canonical element $\id_{B^A}$ in the domain of $\phi^\m1$. We can apply the function $\phi^\m1$ and obtain an element $ev=\phi^\m1(\id_{B^A})\in\Hom_\Set(\Hom_\Set(A,B)\times A,B)$, which is itself a function, $$ev\taking\Hom_\Set(A,B)\times A\to B.$$ 
\sexc Describe the function $ev$ in terms of how it operates on elements in its domain. 
\next Why might one be tempted to denote this function by $ev$?
\endsexc
\end{exercise}

If $n\in\NN$ is a natural number, recall from (\ref{dia:underline n}) that there is a nice set $\ul{n}=\{1,2,\ldots,n\}$. If $A$ is a set, we often make the abbreviation 
\begin{align}\label{dia:exponential abbrev}
A^n:=A^{\ul{n}}.
\end{align}

\begin{exercise}\label{exc:two R2s}

In Example \ref{ex:R2} we said that $\RR^2$ is an abbreviation for $\RR\times\RR$, but in (\ref{dia:exponential abbrev}) we say that $\RR^2$ is an abbreviation for $\RR^{\ul{2}}$. Use Exercise \ref{exc:generator for set}, Proposition \ref{prop:curry}, Exercise \ref{exc:coprod}, and the fact that 1+1=2, to prove that these are isomorphic, $\RR^{\ul{2}}\iso\RR\times\RR$.

(The answer to Exercise \ref{exc:generator for set} was $A=\singleton$: i.e. $\Hom_\Set(\singleton,X)\iso X$ for all $X$.)
\end{exercise}


\subsection{Arithmetic of sets}\label{sec:arithmetic of sets}\index{set!arithmetic of}

Proposition \ref{prop:arithmetic of sets} summarizes the properties of products, coproducts, and exponentials, and shows them all in a familiar light, namely that of arithmetic. In fact, one can think of the natural numbers as literally being the isomorphism classes of finite sets---that's what they are used for in counting. Consider the standard procedure for counting the elements of a set $S$, say cows in a field: one points to an element in $S$ and simultaneously says ``1", points to another element in $S$ and simultaneously says ``2", and so on until finished. This procedure amounts to nothing more than creating an isomorphism (one-to-one mapping) between $S$ and some set $\ul{n}$. 

Again, the natural numbers are the isomorphism classes of finite sets. Their behavior, i.e. the arithmetic of natural numbers, reflects the behavior of sets. For example the fact that multiplication distributes over addition is a fact about grids of dots as in Example \ref{ex:grid1}. The following proposition lays out such arithmetic properties of sets.

In this proposition, we denote the coproduct of two sets $A$ and $B$ by the notation $A+B$ rather than $A\sqcup B$. It is a reasonable notation in general, and one that is often used. 

\begin{proposition}\label{prop:arithmetic of sets}

The following isomorphisms exist for any sets $A,B,$ and $C$ (except for one caveat, see Exercise \ref{exc:0 to the 0}). 
\begin{itemize}
\item $A+\ul{0}\iso A$
\item $A + B\iso B + A$
\item $(A + B) + C \iso A + (B + C)$
\item $A\times\ul{0}\iso\ul{0}$
\item $A\times\ul{1}\iso A$
\item $A\times B\iso B\times A$
\item $(A\times B)\times C \iso A\times (B\times C)$
\item $A\times(B+C)\iso (A\times B)+(A\times C)$
\item $A^{\ul{0}}\iso \ul{1}$
\item $A^{\ul{1}}\iso A$
\item $\ul{0}^A\iso\ul{0}$
\item $\ul{1}^A\iso\ul{1}$
\item $A^{B+C}\iso A^B\times A^C$
\item $(A^B)^C\iso A^{B\times C}$
\end{itemize}

\end{proposition}

\begin{exercise}\label{exc:0 to the 0}
Everything in Proposition \ref{prop:arithmetic of sets} is true except in one case, namely that of $$\ul{0}^{\ul{0}}.$$ In this case, we get conflicting answers, because for any set $A$, including $A=\emptyset=\ul{0}$, we have claimed both that $A^{\ul{0}}\iso\ul{1}$ and that $\ul{0}^A\iso\ul{0}.$ 

What is the correct answer for $\ul{0}^{\ul{0}}$, based on the definitions of $\ul{0}$ and $\ul{1}$, given in (\ref{dia:underline n}), and of $A^B$, given in (\ref{dia:exponential sets})?
\end{exercise}

\begin{exercise}
It is also true of natural numbers that if $a,b\in\NN$ and $ab=0$ then either $a=0$ or $b=0$. Is the analogous statement true of all sets?
\end{exercise}

Proposition \ref{prop:arithmetic of sets} is in some sense about isomorphisms. It says that understanding isomorphisms of sets reduces to understanding natural numbers. But note that there is much more going on in $\Set$ than isomorphisms; in particular there are functions that are not invertible. 

In grade school you probably never saw anything that looked like this:
$$5^3\times 3\too 5$$
And yet in Exercise \ref{exc:evaluation} we found a function $ev\taking B^A\times A\to B$ that exists for any sets $A,B$. This function $ev$ is not an isomorphism so it somehow does not show up as an equation of natural numbers. But it still has important meaning.
\footnote{Roughly, the existence of $ev\taking\ul{5}^{\ul{3}}\times\ul{3}\too \ul{5}$ says that given a dot in a $5\times 5\times 5$ grid of dots, and given one of the three axes, you can tell me the coordinate of that dot along that axis.} In terms of mere number, it looks like we are being told of an important function $\ul{575}\to\ul{5}$, which is bizarre. The issue here is precisely the one you confronted in Exercise \ref{exc:functions are not iso invariant}.

\begin{exercise}
Explain why there is a canonical function $\ul{5}^{\ul{3}}\times\ul{3}\too \ul{5}$ but not a canonical function $\ul{575}\to\ul{5}$.
\end{exercise}

\begin{slogan}
It is true that a set is isomorphic to any other set with the same number of elements, but don't be fooled into thinking that the study of sets reduces to the study of numbers. Functions that are not isomorphisms cannot be captured within the framework of numbers. 
\end{slogan}


\subsection{Subobjects and characteristic functions}

\begin{definition}\label{def:power set}

For any set $B$, define the {\em power set of $B$}\index{power set}, denoted $\PP(B)$,\index{a symbol!$\PP$} to be the set of subsets of $B$.

\end{definition}

\begin{exercise}\label{exc:size of power sets}~
\sexc How many elements does $\PP(\emptyset)$ have? 
\next How many elements does $\PP(\singleton)$ have? 
\next How many elements does $\PP(\{1,2,3,4,5,6\})$ have? 
\next Any idea why they may have named it ``power set"?
\endsexc
\end{exercise}


\subsubsection{Simplicial complexes}\label{sec:simplicial complex}

\begin{definition}\label{def:simplicial complex}\index{simplicial complex}

Let $V$ be a set and let $\PP(V)$ be its powerset. A subset $X\ss\PP(V)$ is called {\em downward-closed} if, for every $u\in X$ and every $u'\ss u$, we have $u'\in X$. We say that $X$ {\em contains all atoms} if for every $v\in V$ the singleton set $\{v\}$ is an element of $X$. 

A {\em simplicial complex} is a pair $(V,X)$ where $V$ is a set and $X\ss\PP(V)$ is a downward-closed subset that contains all atoms. The elements of $X$ are called {\em simplices} (singular: {\em simplex}).\index{simplex} Any subset $u\ss V$ has a cardinality $|u|$, so we have a function $X\to\NN$ sending each simplex to its cardinality. The set of simplices with cardinality $n+1$ is denoted $X_n$ and each element $x\in X_n$ is called an {\em $n$-simplex}.
\footnote{It is annoying at first that the set of subsets with cardinality 1 is denoted $X_0$, etc. But this is standard convention because as we will see, $X_n$ will be $n$-dimensional.}
Since $X$ contains all atoms (subsets of cardinality 1), we have $X_0\iso V$, and we may also call the 0-simplices {\em vertices}. We sometimes call the 1-simplices {\em edges}.
\footnote{The reason we wrote $X_0\iso V$ rather than $X_0=V$ is that $X_0$ is the set of 1-element subsets of $V$. So if $V=\{a,b,c\}$ then $X_0=\{\{a\},\{b\},\{c\}\}$. This is really just pedantry.}

Since $X_0\iso V$, we may denote a simplicial complex $(V,X)$ simply by $X$.

\end{definition}

\begin{example}

Let $n\in\NN$ be a natural number and let $V=\ul{n+1}$. Define {\em the $n$-simplex}, denoted $\Delta^n$, to be the simplicial complex $\PP(V)\ss\PP(V)$, i.e. the whole power set, which indeed is downward-closed and contains all atoms. 

\end{example}

We can draw a simplicial complex $X$ by first putting all the vertices on the page as dots. Then for every $x\in X_1$, we see that $x=\{v,v'\}$ consists of 2 vertices, so we draw an edge connecting $v$ and $v'$. For every $y\in X_2$ we see that $y=\{w,w',w''\}$ consists of 3 vertices, so we draw a (filled-in) triangle connecting them. All three edges will be drawn too because $X$ is assumed to be downward closed.

Thus, the 0-simplex $\Delta^0$, the 1-simplex $\Delta^1$, the 2-simplex $\Delta^2$, and the 3-simplex $\Delta^3$ are drawn here:
\begin{center}
\includegraphics[height=1.1in]{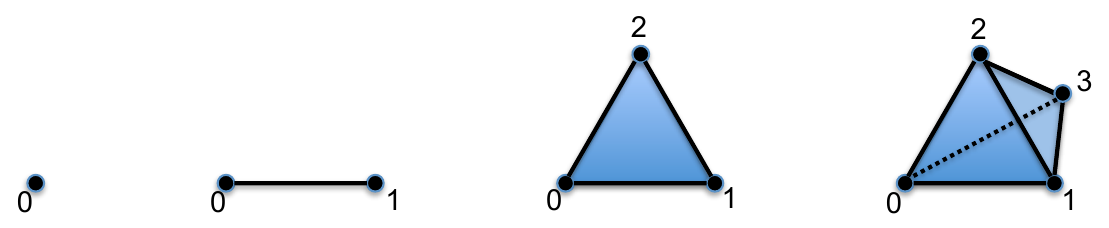}
\end{center} 

The $n$-simplices for various $n$'s are in no way all of the simplicial complexes. In general a simplicial complex is a union or ``gluing together" of simplices in a prescribed manner. For example, consider the simplicial complex $X$ with vertices $X_0=\{1,2,3,4\},$ edges $X_1=\{\{1,2\},\{2,3\},\{2,4\}\},$ and no higher simplices $X_2=X_3=\cdots=\emptyset$. We might draw $X$ as follows:
$$\xymatrix{\LMO{1}\ar@{-}[r]&\LMO{2}\ar@{-}[r]\ar@{-}[d]&\LMO{3}\\&\LMO{4}}$$

\begin{exercise}
Let $X$ be the following simplicial complex, so that $X_0=\{A,B,\ldots,M\}$. 
\begin{center}
\includegraphics[height=3in]{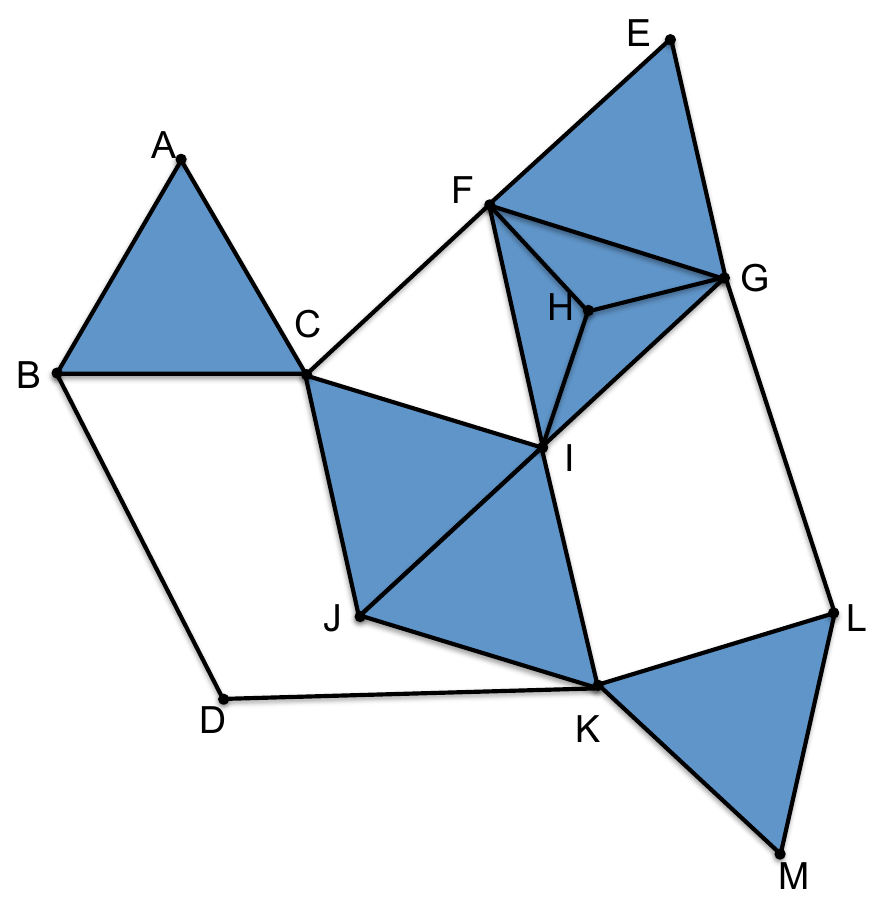}
\end{center} 
In this case $X_1$ consists of elements like $\{A,B\}$ and $\{D,K\}$ but not $\{D,J\}$. 

Write out $X_2$ and $X_3$ (hint: the drawing of $X$ indicates that $X_3$ should have one element).
\end{exercise}

\begin{exercise}
The 2-simplex $\Delta^2$ is drawn as a filled-in triangle with vertices $V=\{1,2,3\}$. There is a simplicial complex $X=\partial\Delta^2$ that would be drawn as an empty triangle with the same set of vertices. 
\sexc Draw $\Delta^2$ and $X$ side by side and make clear the difference.
\next Write down the data for $X$ as a simplicial complex. In other words what are the sets $X_0, X_1, X_2, X_3,\ldots$?
\endsexc
\end{exercise}


\subsubsection{Subobject classifier}

\begin{definition}\label{def:subobject classifier}\index{subobject classifier!in $\Set$}

Define the {\em subobject classifier} for $\Set$, denoted $\Omega$\index{a symbol!$\Omega$}, to be the set $\Omega:=\{True,False\}$, together with the function $\singleton\to\Omega$ sending the unique element to $True$.

\end{definition}

\begin{proposition}\label{prop:characteristic function}

Let $B$ be a set. There is an isomorphism $$\phi\taking\Hom_\Set(B,\Omega)\To{\iso}\PP(B).$$

\end{proposition}

\begin{proof}

Given a function $f\taking B\to\Omega$, let $\phi(f)=\{b\in B\|f(b)=True\}\ss B$. We now construct a function $\psi\taking\PP(B)\to\Hom_\Set(B,\Omega)$ to serve as the inverse of $\phi$. Given a subset $B'\ss B$, define $\psi(B')\taking B\to\Omega$ as follows: 
$$\psi(i)(b)=\begin{cases}
True&\tn{ if } b\in B',\\
False&\tn{ if } b\not\in B'.
\end{cases}
$$
One checks easily that $\phi$ and $\psi$ are mutually inverse.

\end{proof}

\begin{definition}[Characteristic function]\index{characteristic function}\index{subset!characteristic function of}

Given a subset $B'\ss B$, we call the corresponding function $B\to\Omega$ the {\em characteristic function of $B'$ in $B$.}

\end{definition}

Let $B$ be any set and let $\PP(B)$ be its power set. By Proposition \ref{prop:characteristic function} there is a bijection between $\PP(B)$ and $\Omega^B$. Since $\Omega$ has cardinality 2, the cardinality of $\PP(B)$ is $2^{|B|}$, which explains the correct answer to Exercise \ref{exc:size of power sets}.

\begin{exercise}
Let $f\taking A\to\Omega$ denote the characteristic function of some $A'\ss A$, and define $A''\ss A$ to be its complement, $A'':=A-A'$ (i.e. $a\in A''$ if and only if $a\not\in A'$). 
\sexc What is the characteristic function of $A''\ss A$? 
\next Can you phrase it in terms of some function $\Omega\to\Omega$?
\endsexc
\end{exercise}


\subsection{Surjections, injections}

The classical definition of injections and surjections involves elements, which we give now. But a more robust notion involves all maps and will be given in Proposition \ref{prop:inj and surj}.

\begin{definition}\label{def:inj,surj,bij}\index{function!injection}\index{function!surjection}\index{function!bijection}

Let $f\taking X\to Y$ be a function. We say that $f$ is {\em surjective} if, for all $y\in Y$ there exists some $x\in X$ such that $f(x)=y$. We say that $f$ is {\em injective} if, for all $x\in X$ and all $x'\in X$ with $f(x)=f(x')$ we have $x=x'$.

A function that is both injective and surjective is called {\em bijective}.

\end{definition}

\begin{remark}

It turns out that a function that is bijective is always an isomorphism and that all isomorphisms are bijective. We will not show that here, but it is not too hard; see for example \cite[Theorem 5.4]{Big}.

\end{remark}

\begin{definition}[Monomorphisms, epimorphisms]\label{def:mono, epi in set}\index{epimorphism!in $\Set$}\index{monomorphism!in $\Set$}

Let $f\taking X\to Y$ be a function. 

We say that $f$ is a {\em monomorphism} if for all sets $A$ and pairs of functions $g,g'\taking A\to X$,
$$
\xymatrix{A\ar@/^1pc/[r]^g\ar@/_1pc/[r]_{g'}&X\ar[r]^f&Y}
$$
if $f\circ g=f\circ g'$ then $g=g'$.

We say that $f$ is an {\em epimorphism} if for all sets $B$ and pairs of functions $h,h'\taking Y\to B$, 
$$
\xymatrix{X\ar[r]^f&Y\ar@/^1pc/[r]^h\ar@/_1pc/[r]_{h'}&B}
$$
if $h\circ f=h'\circ f$ then $h=h'$.

\end{definition}

\begin{proposition}\label{prop:inj and surj}

Let $f\taking X\to Y$ be a function. Then $f$ is injective if and only if it is a monomorphism; $f$ is surjective if and only if it is an epimorphism.

\end{proposition}

\begin{proof}

If $f$ is a monomorphism it is clearly injective by putting $A=\singleton$. Suppose that $f$ is injective and let $g,g'\taking A\to X$ be functions such that $f\circ g=f\circ g'$, but suppose for contradiction that $g\neq g'$. Then there is some element $a\in A$ such $g(a)\neq g'(a)\in X$. But by injectivity $f(g(a))\neq f(g'(a))$, contradicting $f\circ g=f\circ g'$.

Suppose that $f\taking X\to Y$ is an epimorphism and choose some $y_0\in Y$ (noting that if $Y$ is empty then the claim is vacuously true). Let $h\taking Y\to\Omega$ denote the characteristic function of the subset $\{y_0\}\ss Y$ and let $h'\taking Y\to\Omega$ denote the characteristic function of $\emptyset\ss Y$; note that $h(y)=h'(y)$ for all $y\neq y_0$. Then since $f$ is an epimorphism and $h\neq h'$, we must have $h\circ f\neq h'\circ f$, so there exists $x\in X$ with $h(f(x))\neq h'(f(x))$, which implies that $f(x)=y_0$. This proves that $f$ is surjective.

Finally, suppose that $f$ is surjective, and let $h,h'\taking Y\to B$ be functions with $h\circ f=h'\circ f$. For any $y\in Y$, there exists some $x\in X$ with $f(x)=y$, so $h(y)=h(f(x))=h'(f(x))=h'(y)$. This proves that $f$ is an epimorphism.

\end{proof}

\begin{proposition}\label{prop:pb preserve mono}

Let $f\taking X\to Y$ be a monomorphism. Then for any function $g\taking A\to Y$, the top map $f'\taking X\times_YA\to A$ in the diagram
$$
\xymatrix{X\times_YA\ar[r]^-{f'}\ar[d]_{g'}\ullimit&A\ar[d]^g\\X\ar[r]_f&Y}
$$
is a monomorphism.

\end{proposition}

\begin{proof}

To show that $f'$ is a monomorphism, we take an arbitrary set $B$ and two maps $m,n\taking B\to X\times_YA$ such that $f'\circ m=f'\circ n$, denote that function by $p:=f'\circ m\taking B\to A$. Now let $q=g'\circ m$ and $r=g'\circ n$. The diagram looks like this:
$$
\xymatrix{B\ar@<.5ex>[rr]^(.4)m\ar@<-.5ex>[rr]_(.4)n\ar@/^2pc/[rrr]^p\ar@<.5ex>[drr]^(.6)q\ar@<-.5ex>[drr]_(.6)r&&X\times_YA\ar[r]^{f'}\ar[d]_{g'}\ullimit&A\ar[d]^g\\&&X\ar[r]_f&Y}
$$
We have that 
\begin{align*}f\circ q=f\circ g'\circ m=g\circ f'\circ m=g\circ f'\circ n=f\circ g'\circ n=f\circ r\end{align*} 
But we assumed that $f$ is a monomorphism so this implies that $q=r$. By the universal property of pullbacks, Lemma \ref{lemma:up for fp}, we have $m=n$.

\end{proof}

\begin{exercise}
Show, in analogy to Proposition \ref{prop:pb preserve mono}, that pushouts preserve epimorphisms.
\end{exercise}

\begin{example}\label{exc:olog pullbacks}

Suppose an olog has a fiber product square
$$\xymatrix{X\times_ZY\ar[r]^-{g'}\ar[d]_{f'}&Y\ar[d]^f\\X\ar[r]_g&Z}$$ such that $f$ is intended to be an injection and $g$ is any map.
\footnote{Of course, this diagram is symmetrical, so the same ideas hold if $g$ is an injection and $f$ is any map.} 
In this case, there are nice labeling systems for $f', g'$, and $X\times_ZY$. Namely:
\begin{itemize}
\item ``is" is an appropriate label for $f'$, 
\item the label for $g$ is an appropriate label for $g'$,
\item (the label for $X$, then ``which", then the label for $g$, then the label for $Y$) is an appropriate label for $X\times_ZY$.
\end{itemize}

To give an explicit example, 
$$\xymatrix{
\obox{X\times_ZY}{.9in}{a rib which is made by a cow}\LA{rr}{is made by}\LAL{d}{is}&&\obox{Y}{.4in}{a cow}\LA{d}{is}\\
\obox{X}{.3in}{a rib}\LAL{rr}{is made by}&&\obox{Z}{.6in}{an animal}
}
$$

\end{example}

\begin{corollary}\label{cor:monos are pullbacks of true}

Let $i\taking A\to X$ be a monomorphism. Then there is a fiber product square of the form 
\begin{align}\label{dia:monos are pbs of true}
\xymatrix{A\ar[r]^{f'}\ar[d]_i\ullimit&\singleton\ar[d]^{True}\\X\ar[r]_f&\Omega.}
\end{align}

\end{corollary}

\begin{proof}

Let $X'\ss X$ denote the image of $i$ and let $f\taking X\to\Omega$ denote the characteristic function of $X'\ss X$. Then it is easy to check that Diagram \ref{dia:monos are pbs of true} is a pullback.

\end{proof}

\begin{exercise}
Consider the subobject classifier $\Omega$, the singleton $\singleton$ and the map $\singleton\To{True}\Omega$ from Definition \ref{def:subobject classifier}. Look at diagram \ref{dia:monos are pbs of true} and in the spirit of Exercise \ref{exc:olog pullbacks}, come up with a label for $\Omega$, a label for $\singleton$, and a label for $True$. Given a label for $X$ and a label for $f$, come up with a label for $A$, a label for $i$ and a label for $f'$, such that the English smoothly fits the mathematics.
\end{exercise}


\subsection{Multisets, relative sets, and set-indexed sets}

In this section we prepare ourselves for considering categories other than $\Set$, by looking at some categories related to $\Set$. 


\subsubsection{Multisets}\index{multiset}

Consider the set $X$ of words in a given document. If $WC(X)$ is the wordcount of the document, we will not generally have $WC(X)=|X|$. The reason is that a set cannot contain the same element more than once, so words like ``the" might be undercounted in $|X|$. A {\em multiset} is a set in which elements can be assigned a multiplicity, i.e. a number of times they are to be counted. 

But if $X$ and $Y$ are multisets, what is the appropriate type of mapping from $X$ to $Y$? Since every set is a multiset (in which each element has multiplicity 1), let's restrict ourselves to notions of mapping that agree with the usual one on sets. That is, if multisets $X$ and $Y$ happen to be sets then our mappings $X\to Y$ should just be functions.

\begin{exercise}\label{exc:multiset 1}~
\sexc Come up with some notion of mapping for multisets that generalizes functions when the notion is restricted to sets. 
\next Suppose that $X=(1,1,2,3)$ and $Y=(a,b,b,b)$, i.e. $X=\{1,2,3\}$ with $1$ having multiplicity 2, and $Y=\{a,b\}$ with $b$ having multiplicity 3. What are all the maps $X\to Y$ in your notion?
\endsexc
\end{exercise}

In Chapter \ref{chap:categories} we will be getting to the definition of category, and you can test whether your notion of mapping in fact defines a category. Here is my definition of mapping for multisets.

\begin{definition}\label{def:multiset}

A {\em multiset} is a sequence $X:=(E,B,\pi)$ where $E$ and $B$ are sets and $\pi\taking E\to B$ is a surjective function. We refer to $E$ as the set of {\em element instances of $X$}, we refer to $B$ as the set of {\em element names of $X$}, and we refer to $\pi$ as the {\em naming function for $X$}. Given an element name $x\in B$, let $\pi^\m1(x)\ss E$ be the preimage; the number of elements in $\pi^\m1(x)$ is called the {\em multiplicity of $x$}.

Suppose that $X=(E,B,\pi)$ and $X'=(E',B',\pi')$ are multisets. A {\em mapping from $X$ to $Y$}, denoted $f\taking X\to Y$, consists of a pair $(f_1,f_0)$ such that $f_1\taking E\to E'$ and $f_0\taking B\to B'$ are functions and such that the following diagram commutes:
\begin{align}\label{dia:multiset map}
\xymatrix{E\ar[r]^{f_1}\ar[d]_{\pi}&E'\ar[d]^{\pi'}\\B\ar[r]_{f_0}&B'.}
\end{align}

\end{definition}

\begin{exercise}
Suppose that a pseudo-multiset is defined to be almost the same as a multiset, except that $\pi$ is not required to be surjective. 
\sexc Write down a pseudo-multiset that is not a multi-set. 
\next Describe the difference between the two notions in terms of multiplicities. 
\next Complexity of names aside, which do you think is a more useful notion: multiset or pseudo-multisets? 
\endsexc
\end{exercise}

\begin{exercise}
Consider the multisets described in Exercise \ref{exc:multiset 1}. 
\sexc Write each of them in the form $(E,B,\pi)$, as in Definition \ref{def:multiset}. 
\next In terms of the same definition, what are the mappings $X\to Y$? 
\next If we remove the restriction that diagram \ref{dia:multiset map} must commute, how many mappings $X\to Y$ are there?
\endsexc
\end{exercise}

\subsubsection{Relative sets}\label{sec:relative sets}\index{relative set}

Let's continue with our ideas from multisets, but now suppose that we have a fixed set $B$ of names that we want to keep once and for all. Whenever someone discusses a set, each element must have a name in $B$. And whenever someone discusses a mapping, it must preserve the names. For example, if $B$ is the set of English words, then every document consists of an ordered set mapping to $B$ (e.g. $1\mapsto \tn{Suppose}, 2\mapsto\tn{that}, 3\mapsto\tn{we},$ etc.) A mapping from document $A$ to document $B$ would send each word found somewhere in $A$ to the same word found somewhere in $B$. This notion is defined carefully below.

\begin{definition}[Relative set]\label{def:relative sets}\index{relative set}

Let $B$ be a set. A {\em relative set over $B$}, or simply a {\em set over $B$}, is a pair $(E,\pi)$ such that $E$ is a set and $\pi\taking E\to B$ is a function. A {\em mapping of relative sets over $B$}, denoted $f\taking (E,\pi)\to(E',\pi')$, is a function $f\taking E\to E'$ such that the triangle below commutes, i.e. $\pi=\pi'\circ f$,
$$
\xymatrix@=10pt{E\ar[rr]^f\ar[rdd]_{\pi}&&E'\ar[ldd]^{\pi'}\\\\&B
}
$$

\end{definition}

\begin{exercise}
Given sets $X,Y,Z$ and functions $f\taking X\to Y$ and $g\taking Y\to Z$, we can compose them to get a function $X\to Z$. If $B$ is a set, if $(X,p), (Y,q),$ and $(Z,r)$ are relative sets over $B$, and if $f\taking (X,p)\to (Y,q)$ and $g\taking (Y,q)\to (Z,r)$ are mappings, is there a reasonable notion of composition such that we get a mapping of relative sets $(X,p)\to (Z,r)$? Hint: draw diagrams.
\end{exercise}

\begin{exercise}~
\sexc Let $\singleton$ denote a set with one element. What is the difference between sets over $\singleton$ and simply sets?
\next Describe the sets relative to $\emptyset$. How many are there?
\endsexc
\end{exercise}


\subsubsection{Indexed sets}\label{sec:indexed sets}\index{indexed set}

Let $A$ be a set. Suppose we want to assign to each element $a\in A$ a set $S_a$. This is called an $A$-indexed set. In category theory we are always interested in the legal mappings between two different structures of the same sort, so we need a notion of $A$-indexed mappings; we do the ``obvious thing".

\begin{example}\label{ex:classroom seats}

Let $C$ be a set of classrooms. For each $c\in C$ let $P_c$ denote the set of people in classroom $c$, and let $S_c$ denote the set of seats (chairs) in classroom $c$. Then $P$ and $S$ are $C$-indexed sets. The appropriate kind of mapping between them respects the indexes. That is, a mapping of multi-sets $P\to S$ should, for each classroom $c\in C$, be a function $P_c\to S_c$.\footnote{If we wanted to allow people from any classroom to choose a chair from just any classroom, category theory would tell us to reconsider $P$ and $S$ as sets, forgetting their indices. See Section \ref{sec:left push}.}

\end{example}

\begin{definition}\label{def:indexed sets}\index{indexed set}

Let $A$ be a set. An {\em $A$-indexed set} is a collection of sets $S_a$, one for each element $a\in A$; for now we denote this by $(S_a)_{a\in A}$. If $(S'_a)_{a\in A}$ is another $A$-indexed set, a {\em mapping of $A$-indexed sets from $(S_a)_{a\in A}$ to $(S'_a)_{a\in A}$}, denoted $$(f_a)_{a\in A}\taking(S_a)_{a\in A}\to (S'_a)_{a\in A}$$ is a collection of functions $f_a\taking S_a\to S'_a$, one for each element $a\in A$.

\end{definition}

\begin{exercise}
Let $\singleton$ denote a one element set. What are $\singleton$-indexed sets and mappings between them?
\end{exercise}

\begin{exercise}
There is a strong relationship between $A$-indexed sets and relative sets over $A$. What is it? 
\end{exercise}


\chapter{Categories and functors, without admitting it}\label{chap:categories and functors without admitting it}

In this chapter we begin to use our understanding of sets to build more interesting mathematical devices, each of which organizes our understanding of a certain kind of domain. For example, monoids organize our thoughts about agents acting on objects; groups are monoids except restricted to only allow agents to act reversibly. We will then study graphs, which are systems of nodes and arrows that can capture ideas like information flow through a network or model connections between building blocks in a material. We will discuss orders, which can be used to study taxonomies or hierarchies. Finally we take a mathematical look at databases, which actually subsume everything else in the chapter. Databases are connection patterns for structuring information.

We will see in Chapter \ref{chap:categories} that everything we study in the present chapter is an example of a category. So is $\Set$, the category of sets studied in Chapter \ref{chap:sets}. One way to think of a category is as a set of objects and a connection pattern between them; sets are objects (ovals full of dots if you wish) connected by functions. But each set is itself a category: the objects inside it are just disconnected! Just like a set has an interior view and an exterior view, so will all the categories in this chapter. Each monoid {\em is} a category, but there is also a category {\em of} monoids. 

However, we will not really say the word ``category" much if at all in this chapter. It seems preferable to let the ideas rise on their own accord as interesting structures in their own right before explaining that everything in site fits into a single framework. That will be the pleasant reward to come in Chapter \ref{chap:categories}.


\section{Monoids}\label{sec:monoids}\index{monoid}

A common way to interpret phenomena we see around us is to say that agents are acting on objects. For example, in a computer drawing program, the user {\em acts on} the canvas in certain prescribed ways. Choices of actions from an available list can be performed in sequence to transform one image into another. As another example, one might investigate the notion that time {\em acts on} the position of hands on a clock in a prescribed way. A first rule for actions is this: the performance of a sequence of several actions is itself the performance of an action---a more complex action, but an action nonetheless.

Mathematical objects called {\em monoids} and {\em groups} are tasked with encoding the agent's perspective in all this, i.e. what the agent can do, and what happens when different actions are done in succession. A monoid can be construed as a set of actions, together with a formula that encodes how a sequence of actions is itself considered an action. A group is the same as a monoid, except that every action is required to be reversible. In this section we concentrate on monoids; we will get to groups in Section \ref{sec:groups}.


\subsection{Definition and examples}

\begin{definition}[Monoid]\label{def:monoid}

A {\em monoid} is a sequence $(M,e,\star)$, where $M$ is a set, $e\in M$ is an element, and $\star\taking M\times M\to M$ is a function, such that the following conditions hold for all $m,n,p\in M$:
\begin{itemize}
\item $m\star e=m$,
\item $e\star m=m$, and
\item $(m\star n)\star p=m\star(n\star p)$.
\end{itemize}
We refer to $e$ as the {\em identity element}\index{monoid!identity element of} and to $\star$ as the {\em multiplication formula} for the monoid.\index{monoid!multiplication formula}
\footnote{Although the function $\star\taking M\times M\to M$ is called the multiplication formula, it may have nothing to do with multiplication. It is nothing more than a formula for taking two inputs and returning an output; calling it ``multiplication" is suggestive of its origins, rather than prescriptive of its behavior.} 
We call the first two rules {\em identity laws} and the third rule the {\em associativity law} for monoids. 

\end{definition}

\begin{remark}

To be pedantic, the conditions from Definition \ref{def:monoid} should be stated 
\begin{itemize}
\item $\star(m,e)=m$,
\item $\star(e,m)=m$, and 
\item $\star(\star(m,n),p)=\star(m,(\star(n,p))$.
\end{itemize} The way they are written in Definition \ref{def:monoid} is called {\em infix notation},\index{infix notation} and we often use infix notation without mentioning it. That is, given a function $\cdot\taking A\times B\to C$, we may write $a\cdot b$ rather than $\cdot(a,b)$.

\end{remark}

\begin{example}[Additive monoid of natural numbers]\label{ex:monoid 0}\index{monoid!additive natural numbers}

Let $M=\NN$ be the set of natural numbers. Let $e=0$ and let $\star\taking M\times M\to M$ denote addition, so that $\star(4,18)=22$. Then the equations $m\star 0=m$ and $0\star m=m$ hold, and $(m\star n)\star p=m\star (n\star p)$. By assigning $e$ and $\star$ in this way, we have ``given $\NN$ the structure of a monoid".

\end{example}

\begin{remark}

Sometimes we are working with a monoid $(M,e,\star)$, and the identity $e$ and multiplication $\star$ are somehow clear from context. In this case we might refer to the set $M$ as though it were the whole monoid. For example, if we were discussing the monoid from Example \ref{ex:monoid 0}, we might refer to it as $\NN$. The danger comes because sets may have multiple monoid structures, as we see below in Exercise \ref{exc:monoid 1}. 

\end{remark}

\begin{example}[Non-monoid]

If $M$ is a set, we might call a function $f\taking M\times M\to M$ an {\em operation on $M$}. For example, if $M=\NN$ is the set of natural numbers, we can consider the operation $f\taking\NN\to\NN$ called exponentiation. For example $f(2,5)=2*2*2*2*2=32$ and $f(7,2)=49.$ This is indeed an operation, but it is not part of any monoid. For one thing there is no possible unit. Trying the obvious choice of $e=1$, we see that $a^1=a$ (good), but that $1^a=1$ (bad: we need it to be $a$). For another thing, this operation is not associative because in general $a^{b^c}\neq (a^b)^c$. For example, $2^{1^2}=2$ but $(2^1)^2=4$. 

One might also attempt to consider an operation $f\taking M\times M\to M$ that, upon closer inspection, aren't even operations. For example, if $M=\ZZ$ then exponentiation is not even an operation. Indeed, $f(2,-1)=2^{-1}=\frac{1}{2}$, and this is not an integer. To have a function $f\taking M\times M\to M$, we need that every element of the domain, in this case every pair of integers, has an output under $f$. So there is no such function $f$. 

\end{example}

\begin{exercise}\label{exc:monoid 1}
Let $M=\NN$ be the set of natural numbers. Taking $e=1$, come up with a formula for $\star$ that gives $\NN$ the structure of a monoid.
\end{exercise}

\begin{exercise}
Come up with an operation on the set $M=\{1,2,3,4\}$, i.e. a legitimate function $f\taking M\times M\to M$, such that $f$ cannot be the multiplication formula for a monoid on $M$. That is, either it is not associative, or no element of $M$ can serve as a unit.
\end{exercise}

\begin{exercise}\label{ex:commutative monoid}
In both Example \ref{ex:monoid 0} and Exercise \ref{exc:monoid 1}, the monoids $(M,e,\star)$ satisfied an additional rule called {\em commutativity},\index{monoid!commutative} namely $m\star n=n\star m$ for every $m,n\in M$. There is a monoid $(M,e,\star)$ lurking in linear algebra textbooks that is not commutative; if you have background in linear algebra try to answer this: what $M, e$, and $\star$ might I be referring to?
\end{exercise}

\begin{exercise}
Recall the notion of commutativity for monoids from Exercise \ref{ex:commutative monoid}. 
\sexc What is the smallest set $M$ that you can give the structure of a non-commutative monoid? 
\next What is the smallest set $M$ that you can give the structure of a monoid?
\endsexc
\end{exercise}

\begin{example}[Trivial monoid]\label{ex:trivial monoid}

There is a monoid with only one element, $M=(\{e\},e,\star)$ where $\star\taking\{e\}\times\{e\}\to\{e\}$ is the unique function. We call this monoid {\em the trivial monoid},\index{monoid!trivial} and sometimes denote it $\ul{1}$.

\end{example}

\begin{example}

Suppose that $(M,e,\star)$ is a monoid. Given elements $m_1,m_2,m_3,m_4$ there are five different ways to parenthesize the product $m_1\star m_2\star m_3\star m_4$, and the associativity law for monoids will show them all to be the same. We have
\begin{align*}
((m_1\star m_2)\star m_3)\star m_4&=(m_1\star m_2)\star (m_3\star m_4)\\
&=(m_1\star(m_2\star m_3))\star m_4\\
&=m_1\star(m_2\star (m_3\star m_4))\\
&=m_1\star((m_2\star m_3)\star m_4)
\end{align*}

In fact, the product of any list of monoid elements is the same, regardless of parenthesization. Therefore, we can unambiguously write $m_1m_2m_3m_4m_5$ rather than any given parenthesization of it. This is known as the \href{http://en.wikipedia.org/wiki/Coherence_theorem}{\text coherence theorem} and can be found in \cite{Mac}.

\end{example}


\subsubsection{Free monoids and finitely presented monoids}\label{sec:free monoid}

\begin{definition}\label{def:list}\index{list}

Let $X$ be a set. A {\em list in $X$} is a pair $(n,f)$ where $n\in\NN$ is a natural number (called the {\em length of the list}) and $f\taking\ul{n}\to X$ is a function, where $\ul{n}=\{1,2,\ldots,n\}$. We may denote such a list by 
$$(n,f)=[f(1),f(2),\ldots,f(n)].$$ 
The {\em empty list} is the unique list in which $n=0$; we may denote it by $[\;]$. Given an element $x\in X$ the {\em singleton list on $x$} is the list $[x]$. Given a list $L=(n,f)$ and a number $i\in\NN$ with $i\leq n$, the {\em $i$th entry of $L$} is the element $f(i)\in X$. \index{entry!in list}

Given two lists $L=(n,f)$ and $L'=(n',f')$, define the {\em concatenation of $L$ and $L'$}\index{list!concatenation}\index{concatenation!of lists}, denoted $L\plpl L'$,\index{a symbol!$\plpl$} to be the list $(n+n',f\plpl f')$, where $f\plpl f'\taking \ul{n+n'}\to X$ is given on $i\leq n+n'$ by
$$(f\plpl f')(i):=
\begin{cases}
f(i)&\tn{ if }i\leq n\\
f'(i-n)&\tn{ if }i\geq n+1
\end{cases}
$$
\end{definition}

\begin{example}

Let $X=\{a,b,c,\ldots,z\}$. The following are elements of $\List(X)$: $$[a,b,c],\;\; [p],\;\; [p,a,a,a,p],\;\; [\;],\;\;\dots$$ The concatenation of $[a,b,c]$ and $[p,a,a,a,p]$ is $[a,b,c,p,a,a,a,p]$. The concatenation of any list $A$ with $[\;]$ is just $A$.

\end{example}

\begin{definition}\label{def:free monoid}\index{monoid!free}

Let $X$ be a set. The {\em free monoid generated by $X$} is the sequence $M:=(\List(X),[\;],\plpl)$, where $\List(X)$ is the set of lists of elements in $X$, where $[\;]\in\List(X)$ is the empty list, and where $\plpl$ is the operation of list concatenation. We refer to $X$ as the set of generators for the monoid $M$.

\end{definition}

\begin{exercise}
Let $\singleton$ denote a one-element set. 
\sexc What is the free monoid generated by $\singleton$? 
\next What is the free monoid generated by $\emptyset$?
\endsexc
\end{exercise}

In the definition below, we will define a monoid $M$  by specifying some generators and some relations. Lists of generators provide us all the possible ways to write elements of $M$. The relations allow us to have two such ways of writing the same element. The following definition is a bit dense, so see Example \ref{ex:presented monoid} for a concrete example.

\begin{definition}[Presented monoid]\label{def:presented monoid}\index{monoid!presented}

Let $G$ be a finite set, let $n\in\NN$ be a natural number,
\footnote{The number $n\in\NN$ is going to stand for the number of relations we declare.} 
and for each $1\leq i\leq n$, let $m_i$ and $m_i'$ be elements of $\List(G)$.
\footnote{Each $m_i$ and $m_i'$ are going to be made equal in the set $M$.} 
The {\em monoid presented by generators $G$ and relations $\{(m_i,m_i')\|1\leq i\leq n\}$} is the monoid $\mcM=(M,e,\star)$ defined as follows. Let $\sim$ denote the equivalence relation on $\List(G)$ generated by $\{(xm_iy\sim xm_i'y)\|x,y\in\List(G), 1\leq i\leq n\}$, and define $M=\List(G)/\sim$. Let $e=[\;]$ and let $a * b$ be obtained by concatenating representing lists. 

\end{definition}

\begin{remark}

Every free monoid is a presented monoid, because we can just take the set of relations to be empty.

\end{remark}

\begin{example}\label{ex:presented monoid}

Let $G=\{a,b,c,d\}$. Think of these as buttons that can be pressed. The free monoid $\List(G)$ is the set of all ways of pressing buttons, e.g. pressing $a$ then $a$ then $c$ then $c$ then $d$ corresponds to the list $[a,a,c,c,d]$. The idea of presented monoids is that you notice that pressing $[a,a,c]$ always gives the same result as pressing $[d,d]$. You also notice that pressing $[c,a,c,a]$ is the same thing as doing nothing. 

In this case, we would have $m_1=[a,a,c]$, $m_1'=[d,d]$, and $m_2=[c,a,c,a], m_2'=[\;]$ and relations $\{(m_1,m_1'), (m_2,m_2')\}$. Really this means that we're equating $m_1$ with $m_1'$ and $m_2$ with $m_2'$, which for convenience we'll write out:
$${\color{blue}{[a,a,c]}}={\color{blue}{[d,d]}}\hsp\tn{and}\hsp{\color{red}{[a,c,a,c]}}={\color{red}{[\;]}}
$$ 

To see how this plays out, we give an example of a calculation in $M=\List(G)/\sim$. Namely, 
\begin{align*}
[b,c,b,{\color{blue}{d,d}},a,c,a,a,c,d] = [b,c,b,a,a,{\color{red}{c,a,c,a}},a,c,d] &= [b,c,b,a,{\color{blue}{a,a,c}},d]\\
&= [b,c,b,a,d,d,d].
\end{align*}

\end{example}

\begin{application}[Buffer]\label{app:buffer}

Let $G=\{a,b,c,\ldots\,z\}$. Suppose we have a \href{http://en.wikipedia.org/wiki/Data_buffer}{\text buffer} of 32 characters and we want to consider the set of lists of length at most 32 to be a monoid. We simply have to decide what happens when someone types a list of length more than 32. 

One option is to say that the last character typed overwrites the 32nd entry, $$[a_1,a_2,\ldots,a_{31},a_{32},b]\sim_1[a_1,a_2,\ldots,a_{31},b].$$ Another option is to say that any character typed after\_32 entries is discarded, $$[a_1,a_2,\ldots,a_{31},a_{32},b]\sim_2[a_1,a_2,\ldots,a_{31},a_{32}].$$ Both of these yield finitely presented monoids, generated by $G$. (In case it's useful, the number of necessary relations in both cases is $26^{33}$.)

\end{application}

\begin{exercise}\label{exc:buffer3}

Let's consider the buffer concept again (see Application \ref{app:buffer}), but this time only having size 3 rather than size 32. Show using Definition \ref{def:presented monoid} that with relations given by $\sim_1$ we indeed have $[a,b,c,d,e,f]=[a,b,f]$ and that with relations given by $\sim_2$ we indeed have $[a,b,c,d,e,f]=[a,b,c]$.
\end{exercise}

\begin{exercise}
Let $K:=\{BS,a,b,c,\ldots,z\}$, a set having 27 elements. Suppose you want to think of $BS\in K$ as the ``backspace key" and the elements $a,b,\ldots z\in K$ as the letter keys on a keyboard. Then the free monoid $\List(K)$ is not quite appropriate as a model because we want $[a,b,d,BS]=[a,b]$. 
\sexc Choose a set of relations for which the monoid presented by generators $K$ and the chosen relations is appropriate to this application. 
\next Under your relations, how does $[BS]$ compare with $[\;]$? Is that suitable?
\endsexc
\end{exercise}


\subsubsection{Cyclic monoids}

\begin{definition}

A monoid is called {\em cyclic}\index{monoid!cyclic} if it has a presentation involving only one generator. 

\end{definition}

\begin{example}\label{ex:cyclic}

Let $Q$ be a symbol; we look at some cyclic monoids generated by $\{Q\}$. With no relations the monoid would be the free monoid on one generator, and would have underlying set $\{[\;],[Q],[Q,Q],[Q,Q,Q],\ldots\}$, with identity element $[\;]$ and multiplication given by concatenation (e.g. $[Q,Q,Q]\plpl[Q,Q]=[Q,Q,Q,Q,Q]$). This is just $\NN$, the additive monoid of natural numbers.

With the really strong relation $[Q]\sim[\;]$ we would get the trivial monoid, a monoid having only one element (see Example \ref{ex:trivial monoid}).

Another possibility is given in the first part of Example \ref{ex:clocks}, where the relation $Q^{12}\sim[\;]$ is used, where $Q^{12}$ is shorthand for $[Q,Q,Q,Q,Q,Q,Q,Q,Q,Q,Q,Q]$.

\end{example}

\begin{example}\label{ex:cyclic monoid (7,4)}

Consider the cyclic monoid with generator $Q$ and relation $Q^7=Q^4$. This monoid has seven elements, $\{e=Q^0,Q=Q^1, Q^2, Q^3, Q^4, Q^5, Q^6\}$, and we know that $Q^6\star Q^5=Q^7*Q^4=Q^4*Q^4=Q^7*Q=Q^5.$ One might depict this monoid as follows
$$\xymatrix@=15pt{
\LMO{e}\ar[rr]&&\LMO{Q}\ar[rr]&&\LMO{Q^2}\ar[rr]&&\LMO{Q^3}\ar[rr]&&\LMO{Q^4}\ar[dr]\\
&&&&&&&\LMO{Q^6}\ar[ur]&&\LMO{Q^5}\ar[ll]
}
$$
To see the mathematical source of this intuitive depiction, see Example \ref{ex:yoneda for cyclic monoid}.

\end{example}

\begin{exercise}[Classify the cyclic monoids]\label{exc:classify cyclic}

Classify all the cyclic monoids up to isomorphism. That is, come up with a naming system such that every cyclic monoid can be given a name in your system, such that no two non-isomorphic cyclic monoids have the same name, and such that no name exists in the system unless it refers to a cyclic monoid. 

Hint: one might see a pattern in which the three monoids in Example \ref{ex:cyclic} correspond respectively to $\infty$, $1$, and $12$, and then think ``Cyclic monoids can be classified by (i.e. systematically named by elements of) the set $\NN\sqcup\{\infty\}$." That idea is on the right track, but is not correct.
\end{exercise}


\subsection{Monoid actions}

\begin{definition}[Monoid action]\label{def:monoid action}\index{monoid!action}\index{action!of a monoid}

Let $(M,e,\star)$ be a monoid and let $S$ be a set. An {\em action of $(M,e,\star)$ on $S$}, or simply an {\em action of $M$ on $S$} or an {\em $M$-action on $S$}, is a function $$\acts\;\;\taking M\times S\to S$$\index{a symbol!$\acts$} such that the following conditions hold for all $m,n\in M$ and all $s\in S$:
\begin{itemize}
\item $e\acts s=s$
\item $m\acts(n\acts s)=(m\star n)\acts s$.
\footnote{
Definition \ref{def:monoid action} actually defines a {\em left action}\index{action!left} of $(M,e,\star)$ on $S$. A {\em right action}\index{action!right} is like a left action except the order of operations is somehow reversed. We will not really use right-actions in this text, but we briefly define it here for completeness. With notation as above, the only difference is in the second condition. We replace it by the condition that for all $m,n\in M$ and all $s\in S$ we have 
$$m\acts(n\acts s)=(n\star m)\acts s
$$}
\end{itemize}

\end{definition}

\begin{remark}\label{rmk:monoid action}

To be pedantic (and because it's sometimes useful), we may rewrite $\acts$ as $\alpha\taking M\times S\to S$ and restate the conditions from Definition \ref{def:monoid action} as
\begin{itemize}
\item $\alpha(e,s)=s$, and
\item $\alpha(m,\alpha(n,s))=\alpha(m\star n,s)$.
\end{itemize}

\end{remark}

\begin{example}\label{ex:clocks}

Let $S=\{0,1,2,\ldots,11\}$ and let $N=(\NN,0,+)$ be the additive monoid of natural numbers (see Example \ref{ex:monoid 0}).  We define a function $\acts\taking\NN\times S\to S$ by taking a pair $(n,s)$ to the remainder that appears when $n+s$ is divided by 12. For example $4\acts 2=6$ and $8\acts 9 = 5$. This function has the structure of a monoid action because the two rules from Definition \ref{def:monoid action} hold.

Similarly, let $T$ denote the set of points on a circle, elements of which are denoted by a real number in the interval $[0,12)$, i.e. $$T=\{x\in\RR\|0\leq x< 12\}$$ and let $R=(\RR,0,+)$ denote the additive monoid of real numbers. Then there is an action $R\times T\to T$, similar to the one above (see Exercise \ref{exc:clock}).

One can think of this as an action of the monoid of time on the clock.

\end{example}

\begin{exercise}\label{exc:clock}~
\sexc Realize the set $T:=[0,12)\ss\RR$ as the coequalizer of a pair of arrows $\RR\tto\RR$. 
\next For any $x\in\RR$, realize the mapping $x\cdot-\taking T\to T$, implied by Example \ref{ex:clocks}, using the universal property of coequalizers. 
\next Prove that it is an action.
\endsexc
\end{exercise}

\begin{exercise}
Let $B$ denote the set of buttons (or positions) of a video game controller (other than, say `start' and `select'), and consider the free monoid $\List(B)$ on $B$. 
\sexc What would it mean for $\List(B)$ to act on the set of states of some game? Imagine a video game $G'$ that uses the controller, but for which $\List(B)$ would not be said to act on the states of $G'$. Now imagine a simple game $G$ for which $\List(B)$ would be said to act.
\next Can you think of a state $s$ of $G$, and two distinct elements $\ell,\ell'\in\List(B)$ such that $\ell\acts s=\ell'\acts s$? In video game parlance, what would you call an element $b\in B$ such that, for every state $s\in G$, one has $b\acts s=s$? 
\next In video game parlance, what would you call a state $s\in S$ such that, for every sequence of buttons $\ell\in\List(B)$, one has $\ell\acts s=s$?
\endsexc
\end{exercise}

\begin{application}

Let $f\taking\RR\to\RR$ be a differentiable function of which we want to find roots (points $x\in\RR$ such that $f(x)=0$). Let $x_0\in\RR$ be a starting point. For any $n\in\NN$ we can apply \href{http://en.wikipedia.org/wiki/Newton's_method}{\text Newton's method} to $x_n$ to get $$x_{n+1}=x_n-\frac{f(x_n)}{f'(x_n)}.$$ 
This is a monoid (namely $\NN$, the free monoid on one generator) acting on a set (namely $\RR$). 

However, Newton's method can get into trouble. For example at a critical point it causes division by 0, and sometimes it can oscillate or overshoot. In these cases we want to perturb a bit to the left or right. To have these actions available to us, we would add ``perturb" elements to our monoid. Now we have more available actions at any point, but at the cost of using a more complicated monoid.

When publishing an experimental finding, there may be some deep methodological questions that are not considered suitably important to mention. For example, one may not publish the kind solution finding method (e.g. Newton's method or Runge-Kutta) that was used, nor the set of available actions, e.g. what kinds of perturbation were used by the researcher. However, these may actually influence the reproducibility of results. By using a language such as that of monoid actions, we can align our data model with our unspoken assumptions about how functions are analyzed.

\end{application}

\begin{remark}

A monoid is useful for understanding how an agent acts on the set of states of an object, but there is only one {\em kind} of action. At any point, all actions are available. In reality it is often the case that contexts can change and different actions are available at different times. For example on a computer, the commands available in one application have no meaning in another. This will get us to categories in the next chapter. 

\end{remark}


\subsubsection{Monoids actions as ologs}

If monoids are understood in terms of how they act on sets, then it is reasonable to think of them in terms of ologs. In fact, the ologs associated to monoids are precisely those ologs that have exactly one type (and possibly many arrows and commutative diagrams). 

\begin{example}\label{ex:monoid as olog}\index{monoid!olog of}

In this example we show how to associate an olog to a monoid action. Consider the monoid $M$ generated by the set $\{u,d,r\}$, standing for ``up, down, right", and subject to the relations $$[u,d]\sim[\;],\hsp[d,u]\sim[\;],\hsp[u,r]=[r,u],\hsp \tn{and}\hsp [d,r]=[r,d].$$
We might imagine that $M$ acts on the set of positions for a character in an old video game. In that case the olog corresponding to this action should look something like the following:
\begin{center}
\includegraphics[width=\textwidth]{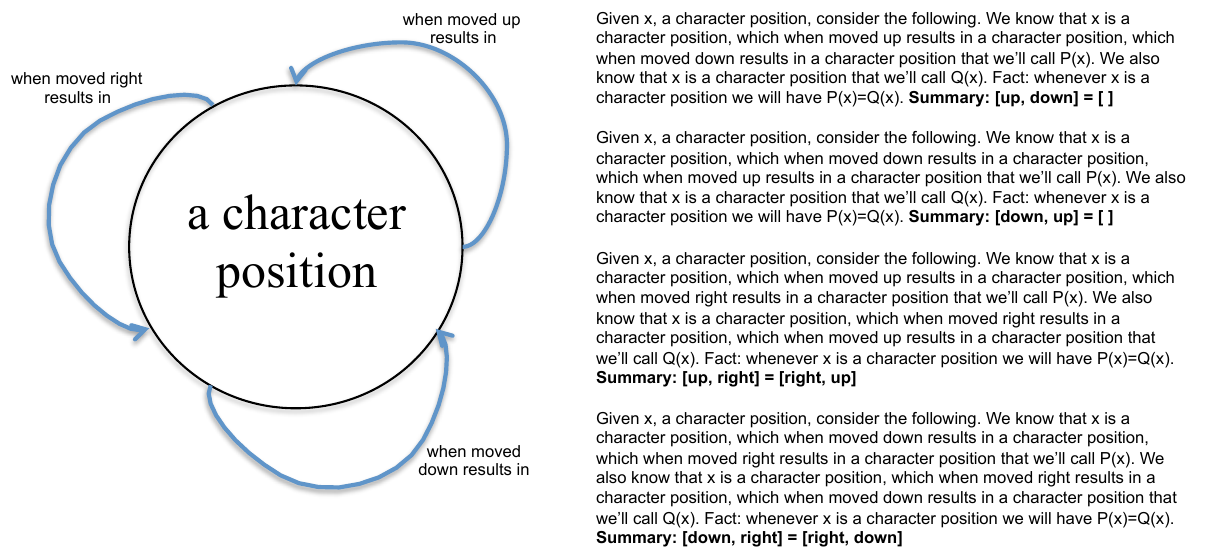}
\end{center}

\end{example}


\subsubsection{Finite state machines}\label{sec:FSMs}\index{finite state machine}

According to Wikipedia, a \href{http://en.wikipedia.org/wiki/Finite_state_machine#Mathematical_model}{\em deterministic finite state machine} is a quintuple $(\Sigma,S,s_0,\delta,F)$, where
\begin{enumerate}
\item $\Sigma$ is a finite non-empty set of symbols, called the {\em input alphabet},
\item $S$ is a finite, non-empty set, called {\em the state set},
\item $\delta\taking \Sigma\times S\to S$ is a function, called the {\em state-transition function}, and
\item $s_0\in S$ is an element, called {\em the initial state},
\item $F\ss S$ is a subset, called the {\em set of final states}.
\end{enumerate}

In this book we will not worry about the initial state and the set of final states, concerning ourselves more with the interaction via $\delta$ of the alphabet $\Sigma$ on the set $S$ of states.

\begin{figure}[h]
\begin{center}
\includegraphics[height=2in]{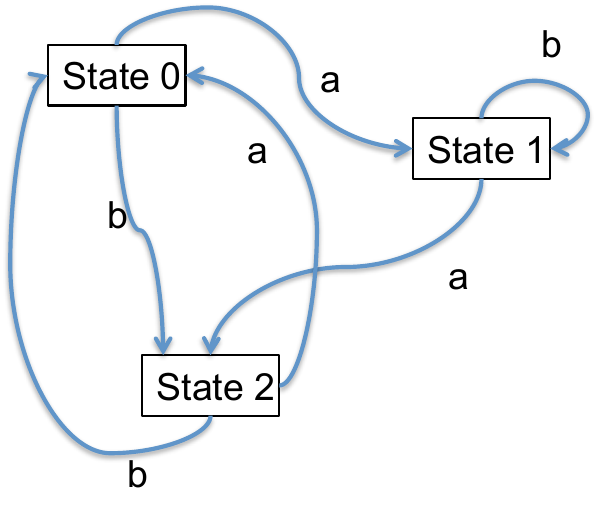}
\end{center}
\caption{A finite state machine with alphabet $\Sigma=\{a,b\}$ and state set $S=\{\tn{State 0, State 1, State 2}\}$. If pressed, we will make State 0 the initial state and \{State 2\} the set of final states.}\label{fig:fsa}
\end{figure}
The following proposition expresses the notion of finite state automata in terms of free monoids and their actions on finite sets.

\begin{proposition}

Let $\Sigma, S$ be finite non-empty sets. Giving a function $\delta\taking\Sigma\times S\to S$ is equivalent to giving an action of the free monoid $\List(\Sigma)$ on $S$. 

\end{proposition}

\begin{proof}

By Definition \ref{def:monoid action}, we know that function $\epsilon\taking\List(\Sigma)\times S\to S$ constitutes an action of the monoid $\List(\Sigma)$ on the set $S$ if and only if, for all $s\in S$ we have $\epsilon([\;],s)=s$, and for any two elements $m,m'\in\List(\Sigma)$ we have $\epsilon(m,\epsilon(m',s))=\epsilon(m\star m',s)$, where $m\star m'$ is the concatenation of lists. Let $$A=\{\epsilon\taking \List(\Sigma)\times S\to S\|\epsilon\tn{ constitutes an action}\}.$$ We need to prove that there is an isomorphism of sets $$\phi\taking A\To{\iso}\Hom_\Set(\Sigma\times S,S).$$

Given an element $\epsilon\taking\List(\Sigma)\times S\to S$ in $A$, define $\phi(\epsilon)$ on an element $(\sigma,s)\in\Sigma\times S$ by $\phi(\epsilon)(\sigma,s):=\epsilon([\sigma],s)$, where $[\sigma]$ is the one-element list. We now define $\psi\taking\Hom_\Set(\Sigma\times S,S)\to A$.

Given an element $f\in\Hom_\Set(\Sigma\times S,S)$, define $\psi(f)\taking\List(\Sigma)\times S\to S$ on a pair $(L,s)\in\List(\Sigma)\times S$, where $L=[\epsilon_1,\ldots,\epsilon_n]$ as follows. By induction, if $n=0$, put $\psi(f)(L,s)=s$; if $n\geq 1$, let $L'=[\epsilon_1,\ldots,\epsilon_{n-1}]$ and put $\psi(f)(L,s)=\psi(f)(L',f(\epsilon_n,s))$. One checks easily that $\psi(f)$ satisfies the two rules above, making it an action of $\List(\Sigma)$ on $S$. It is also easy to check that $\phi$ and $\psi$ are mutually inverse, completing the proof.

\end{proof}

We sum up the idea of this section as follows:
\begin{slogan}
A finite state machine is an action of a free monoid on a finite set.
\end{slogan}

\begin{exercise}
Consider the functions $\phi$ and $\psi$ above. 
\sexc Show that for any $f\taking\Sigma\times S\to S$, the map $\psi(f)\taking\List(\Sigma)\times S\to S$ constitutes an action. 
\next Show that $\phi$ and $\psi$ are mutually inverse functions (i.e. $\phi\circ\psi=\id_{\Hom(\Sigma\times S,S)}$ and $\psi\circ\phi=\id_{A}.$)
\endsexc
\end{exercise}


\subsection{Monoid action tables}\label{sec:monoid action table}\index{action table}

Let $M$ be a monoid generated by the set $G=\{g_1,\ldots,g_m\}$, and with some relations, and suppose that $\alpha\taking M\times S\to S$ is an action of $M$ on a set $S=\{s_1,\ldots,s_n\}$. We can represent the action $\alpha$ using an {\em action table} whose columns are the elements of $G$ and whose rows are the elements of $S$. In each cell $(row,col)$, where $row\in S$ and $col\in G$, we put the element $\alpha(col,row)\in S$. 

\begin{example}[Action table]\label{ex:action table}

If $\Sigma$ and $S$ are the sets from Figure \ref{fig:fsa}, the displayed action of $\List(\Sigma)$ on $S$ would be given by the action table
\begin{align}\label{dia:action table for FSM}
\begin{tabular}{| l || l | l |}\bhline
\multicolumn{3}{|c|}{Action from \ref{fig:fsa}}\\\bhline
{\bf ID}&{\bf a}&{\bf b}\\\bbhline
State 0&State 1&State 2\\\hline
State 1& State 2& State 1\\\hline
State 2&State 0&State 0\\\bhline
\end{tabular}
\end{align}

\end{example}

\begin{example}[Multiplication action table]\label{ex:multiplication table}

Every monoid acts on itself by its multiplication formula, $M\times M\to M$. If $G$ is a generating set for $M$, we can write the elements of $G$ as the columns and the elements of $M$ as rows, and call this a multiplication table. For example, let $(\NN,1,*)$ denote the multiplicative monoid of natural numbers. The multiplication table is as follows:
\begin{align}
\begin{tabular}{| l || l | l | l | l | l | l | l |}\bhline
\multicolumn{8}{|c|}{Multiplication of natural numbers}\\\bhline
{\bf $\NN$}&{\bf 0}&{\bf 1}&{\bf 2}&{\bf 3}&{\bf 4}&{\bf 5}&{\bf $\cdots$}\\\bbhline
0&0&0&0&0&0&0&$\cdots$\\\hline
1&0&1& 2& 3 & 4&5&$\cdots$\\\hline
2&0&2&4&6&8&10&$\cdots$\\\hline
3&0&3&6&9&12&15&$\cdots$\\\bhline
4&0&4&8&12&16&20&$\cdots$\\\bhline
\vdots&\vdots&\vdots&\vdots&\vdots&\vdots&\vdots&$\ddots$\\\hline
21&0&21&42&63&84&105&$\cdots$\\\hline
\vdots&\vdots&\vdots&\vdots&\vdots&\vdots&\vdots&$\ddots$\\\bhline
\end{tabular}
\end{align}
Try to understand what is meant by this: ``applying column $2$ and then column $2$ returns the same thing as applying column $4$."

In the above table, we were implicitly taking every element of $\NN$ as a generator (since we had a column for every natural number). In fact, there is a smallest generating set for the monoid $(\NN,1,*)$, so that every element of the monoid is a product of some combination of these generators, namely the primes and 0.
\begin{align*}
\begin{tabular}{| l || l | l | l | l | l | l | l |}\bhline
\multicolumn{8}{|c|}{Multiplication of natural numbers}\\\bhline
{\bf $\NN$}&{\bf 0}&{\bf 2}&{\bf 3}&{\bf 5}&{\bf 7}&{\bf 11}&{\bf $\cdots$}\\\bbhline
0&0&0&0&0&0&0&$\cdots$\\\hline
1&0&2& 3& 5 & 7&11&$\cdots$\\\hline
2&0&4&6&10&14&22&$\cdots$\\\hline
3&0&6&9&15&21&33&$\cdots$\\\bhline
4&0&8&12&20&28&44&$\cdots$\\\bhline
\vdots&\vdots&\vdots&\vdots&\vdots&\vdots&\vdots&$\ddots$\\\hline
21&0&42&63&105&147&231&$\cdots$\\\hline
\vdots&\vdots&\vdots&\vdots&\vdots&\vdots&\vdots&$\ddots$\\\bhline
\end{tabular}
\end{align*}

\end{example}

\begin{exercise}
Let $\NN$ be the additive monoid of natural numbers, let $S=\{0,1,2,\ldots,11\}$, and let $\cdot\taking\NN\times S\to S$ be the action given in Example \ref{ex:clocks}. Using a nice small generating set for the monoid, write out the corresponding action table.
\end{exercise}


\subsection{Monoid homomorphisms}

A monoid $(M,e,\star)$ involves a set, an identity element, and a multiplication formula. For two monoids to be comparable, their sets, their identity elements, and their multiplication formulas should be appropriately comparable.\index{appropriate comparison} For example the additive monoids $\NN$ and $\ZZ$ should be comparable because $\NN\ss\ZZ$ is a subset, the identity elements in both cases are the same $e=0$, and the multiplication formulas are both integer addition. 

\begin{definition}\label{def:monoid hom}\index{monoid!homomorphism}

Let $\mcM:=(M,e,\star)$ and $\mcM':=(M',e',\star')$ be monoids. A {\em monoid homomorphism $f$ from $\mcM$ to $\mcM'$}, denoted $f\taking\mcM\to\mcM'$, is a function $f\taking M\to M'$ satisfying two conditions:
\begin{itemize}
\item $f(e)=e'$, and 
\item $f(m_1\star m_2)=f(m_1)\star'f(m_2)$, for all $m_1,m_2\in M$.
\end{itemize}

The set of monoid homomorphisms from $\mcM$ to $\mcM'$ is denoted $\Hom_{\Mon}(\mcM,\mcM')$.

\end{definition}

\begin{example}[From $\NN$ to $\ZZ$]\label{ex:nat to int}

As stated above, the inclusion map $i\taking\NN\to\ZZ$ induces a monoid homomorphism $(\NN,0,+)\to(\ZZ,0,+)$ because $i(0)=0$ and $i(n_1+n_2)=i(n_1)+i(n_2)$. 

Let $i_5\taking\NN\to\ZZ$ denote the function $i_5(n)=5*n$, so $i_5(4)=20$. This is also a monoid homomorphism because $i_5(0)=5*0=0$ and $i_5(n_1+n_2)=5*(n_1+n_2)=5*n_1+5*n_2=i_5(n_1)+i_5(n_2).$

\end{example}

\begin{application}\label{app:RNA reader 1}

Let $R=\{a,c,g,u\}$ and let $T=R^3$, the set of triplets in $R$. Let $\mcR=\List(R)$ be the free monoid on $R$ and let $\mcT=\List(T)$ denote the free monoid on $T$. There is a monoid homomorphism $F\taking\mcT\to\mcR$ given by sending $t=(r_1,r_2,r_3)$ to the list $[r_1,r_2,r_3]$.
\footnote{More precisely, the monoid homomorphism $F$ sends a list $[t_1,t_2,\ldots,t_n]$ to the list $[r_{1,1},r_{1,2},r_{1,3},r_{2,1},r_{2,2},r_{2,3},\ldots,r_{n,1},r_{n,2},r_{n,3}]$, where for each $0\leq i\leq n$ we have $t_i=(r_{i,1},r_{i,2},r_{i,3})$.}

If $A$ be the set of amino acids and $\mcA=\List(A)$ the free monoid on $A$, the process of \href{http://en.wikipedia.org/wiki/Translation_(biology)}{\text translation} gives a monoid homomorphism $G\taking\mcT\to\mcA$, turning a list of RNA triplets into a polypeptide. But how do we go from a list of RNA nucleotides to a polypeptide? The answer is that there is no good way to do this mathematically. So what is going wrong?

The answer is that there should not be a monoid homomorphism $\mcR\to\mcA$ because not all sequences of nucleotides produce a polypeptide; for example if the sequence has only two elements, it does not code for a polypeptide. There are several possible remedies to this problem. One is to take the image of $F$, which is a submonoid $\mcR'\ss\mcR$. It is not hard to see that there is a monoid homomorphism $F'\taking\mcR'\to\mcT$, and we can compose it with $G$ to get our desired monoid homomorphism $G\circ F'\taking\mcR'\to\mcA$. 
\footnote{Adding stop-codons to the mix we can handle more of $\mcR$, e.g. sequences that don't have a multiple-of-three many nucleotides.}

\end{application}

\begin{example}\label{ex:trivial monoid homomorphism}\index{monoid!trivial homomorphism}

Given any monoids $\mcM$ there is a unique monoid homomorphism from $\mcM$ to the trivial monoid $\ul{1}$ (see Example \ref{ex:trivial monoid}). There is also a unique homomorphism $\ul{1}\to\mcM$. These facts together have an upshot: between any two monoids $\mcM$ and $\mcM'$ we can always construct a homomorphism 
$$\mcM\Too{!}\ul{1}\Too{!}\mcM'$$
which we call the {\em trivial homomorphism $\mcM\to\mcM'$}.\index{trivial homomorphism!of monoids} A morphism $\mcM\to\mcM'$ that is not trivial is called a {\em nontrivial homomorphism}.

\end{example}

\begin{proposition}\label{prop:int to nat trivial}

Let $\mcM=(\ZZ,0,+)$ and $\mcM'=(\NN,0,+)$. The only monoid homomorphism $f\taking\mcM\to\mcM'$ sends every element $m\in\ZZ$ to $0\in\NN$.

\end{proposition}

\begin{proof}

Let $f\taking\mcM\to\mcM'$ be a monoid homomorphism, and let $n=f(1)$ and $n'=f(-1)$ in $\NN$. Then we know that since $0=1+(-1)$ in $\ZZ$ we must have $0=f(0)=f(1+(-1))=f(1)+f(-1)=n+n'\in\NN$. But if $n\geq 1$ then this is impossible, so $n=0$. Similarly $n'=0$. Any element $m\in\ZZ$ can be written $m=1+1+\cdots+1$ or as $m=-1+-1+\cdots+-1$, and it is easy to see that $f(1)+f(1)+\cdots+f(1)=0=f(-1)+f(-1)+\cdots+f(-1)$. Therefore, $f(m)=0$ for all $m\in\ZZ$. 

\end{proof}

\begin{exercise}
For any $m\in\NN$ let $i_m\taking\NN\to\ZZ$ be the function $i_m(n)=m*n$. All such functions are monoid homomorphisms $(\NN,0,+)\to(\ZZ,0,+)$. Do any monoid homomorphisms $(\NN,0,+)\to(\ZZ,0,+)$ not come in this way? For example, what about using $n\mapsto 5*n-1$ or $n\mapsto n^2$, or some other function? 
\end{exercise}

\begin{exercise}
Let $\mcM:=(\NN,0,+)$ be the additive monoid of natural numbers, let $\mcN=(\RR_{\geq0},0,+)$ be the additive monoid of nonnegative real numbers, and let $\mcP:=(\RR_{>0},1,*)$ be the multiplicitive monoid of positive real numbers. Can you think of any nontrivial monoid homomorphisms of the following sorts: $$\mcM\to\mcN,\hsp\mcM\to\mcP,\hsp\mcN\to\mcP,\hsp \mcN\to\mcM,\hsp\mcP\to\mcN?$$
\end{exercise}


\subsubsection{Homomorphisms from free monoids}

Recall that $(\NN,0,+)$ is the free monoid on one generator. It turns out that for any other monoid $\mcM=(M,e,\star)$, the set of monoid homomorphisms $\NN\to\mcM$ is in bijection with the set $M$. This is a special case (in which $G$ is a set with one element) of the following proposition.

\begin{proposition}\label{prop:free monoid}

Let $G$ be a set, let $F(G):=(\List(G),[\;],\plpl)$ be the free monoid on $G$, and let $\mcM:=(M,e,\star)$ be any monoid. There is a natural bijection
$$\Hom_\Mon(F(G),\mcM)\To{\iso}\Hom_\Set(G,M).$$

\end{proposition}

\begin{proof}

We provide a function $\phi\taking\Hom_\Mon(F(G),\mcM)\to\Hom_\Set(G,M)$ and a function $\psi\taking\Hom_\Set(G,M)\to\Hom_\Mon(F(G),\mcM)$ and show that they are mutually inverse. Let us first construct $\phi$. Given a monoid homomorphism $f\taking F(G)\to\mcM$, we need to provide $\phi(f)\taking G\to M$. Given any $g\in G$ we define $\phi(f)(g):=f([g]).$ 

Now let us construct $\psi$. Given $p\taking G\to M$, we need to provide $\psi(p)\taking\List(G)\to\mcM$ such that $\psi(p)$ is a monoid homomorphism. For a list $L=[g_1,\ldots,g_n]\in\List(G)$, define $\psi(p)(L):=p(g_1)\star\cdots\star p(g_n)\in M$. In particular, $\psi(p)([\;])=e$. It is not hard to see that this is a monoid homomorphism. It is also easy to see that $\phi\circ\psi(p)=p$ for all $p\in\Hom_\Set(G,M)$. We show that $\psi\circ\phi(f)=f$ for all $f\in\Hom_\Mon(F(G),\mcM)$. Choose $L=[g_1,\ldots,g_n]\in\List(G)$. Then 
$$\psi(\phi f)(L)=(\phi f)(g_1)\star\cdots\star(\phi f)(g_n)=f[g_1]\star\cdots\star f[g_n]=f([g_1,\ldots,g_n])=f(L).$$

\end{proof}

\begin{exercise}
Let $G=\{a,b\}$, let $\mcM:=(M,e,\star)$ be any monoid, and let $f\taking G\to M$ be given by $f(a)=m$ and $f(b)=n$, where $m,n\in M$. If $\psi\taking\Hom_\Set(G,M)\to\Hom_\Mon(F(G),\mcM)$ is the function from the proof of Proposition \ref{prop:free monoid} and $L=[a,a,b,a,b]$, what is $\psi(f)(L)$ ?
\end{exercise}


\subsubsection{Restriction of scalars}

A monoid homomorphism $f\taking M\to M'$ (see Definition \ref{def:monoid hom}) ensures that the elements of $M$ have a reasonable interpretation in $M'$; they act the same way over in $M'$ as they did back home in $M$. If we have such a homomorphism $f$ and we have an action $\alpha\taking M'\times S\to S$ of $M'$ on a set $S$, then we have a method for allowing $M$ to act on $S$ as well. Namely, we take an element of $M$, send it over to $M'$, and act on $S$. In terms of functions, we compose $\alpha$ with the function $f\times\id_S\taking M\times S\to M'\times S$, to get a function we'll denote $$\Delta_f(\alpha)\taking M\times S\to S.$$ After Proposition \ref{prop:restriction of scalars} we will know that $\Delta_f(\alpha)$ is indeed a monoid action, and we say that it is given by {\em restriction of scalars along $f$}.\index{restriction of scalars}

\begin{proposition}\label{prop:restriction of scalars}

Let $\mcM:=(M,e,\star)$ and $\mcM':=(M',e',\star')$ be monoids, $f\taking\mcM\to\mcM'$ a monoid homomorphism, $S$ a set, and suppose that $\alpha\taking M'\times S\to S$ is an action of $\mcM'$ on $S$. Then $\Delta_f(\alpha)\taking M\times S\to S$, defined as above, is a monoid action as well.

\end{proposition}

\begin{proof}

Refer to Remark \ref{rmk:monoid action}; we assume $\alpha$ is a monoid action and want to show that $\Delta_f(\alpha)$ is too. We have $\Delta_f(\alpha)(e,s)=\alpha(f(e),s)=\alpha(e',s)=s$. We also have 
\begin{align*}
\Delta_f(\alpha)(m,\Delta_f(\alpha)(n,s))=\alpha(f(m),\alpha(f(n),s))&=\alpha(f(m)\star' f(n),s)\\
&=\alpha(f(m\star n),s)\\
&=\Delta_f(\alpha)(m\star n,s).
\end{align*}

\end{proof}

\begin{example}

Let $\NN$ and $\ZZ$ denote the additive monoids of natural numbers and integers, respectively, and let $i\taking\NN\to\ZZ$ be the inclusion, which we saw in Example \ref{ex:nat to int} is a monoid homomorphism. There is an action $\alpha\taking\ZZ\times\RR\to\RR$ of the monoid $\ZZ$ on the set $\RR$ of real numbers, given by $\alpha(n,x)=n+x$. Clearly, this action works just as well if we restrict our scalars to $\NN\ss\ZZ$, allowing ourselves only to add natural numbers to reals. The action $\Delta_i\alpha\taking\NN\times\RR\to\RR$ is given on $(n,x)\in\NN\times\RR$ by $\Delta_i\alpha(n,x)=\alpha(i(n),x)=\alpha(n,x)=n+x$, just as expected.

\end{example}

\begin{example}

Suppose that $V$ is a complex vector space. In particular, this means that the monoid $\CC$ of complex numbers (under multiplication) acts on the elements of $V$. If $i\taking\RR\to\CC$ is the inclusion of the real line inside $\CC$, then $i$ is a monoid homomorphism. Restriction of scalars in the above sense turns $V$ into a real vector space, so the name ``restriction of scalars" is apt.

\end{example}

\begin{exercise}
Let $\NN$ be the free monoid on one generator, let $\Sigma=\{a,b\}$, and let $S=\{\tn{State 0, State 1, State 2}\}$. Consider the map of monoids $f\taking\NN\to\List(\Sigma)$ given by sending $1\mapsto [a,b,b]$. The monoid action $\alpha\taking\List(\Sigma)\times S\to S$ given in Example \ref{ex:action table} can be transformed by restriction of scalars along $f$ to an action $\Delta_f(\alpha)$ of $\NN$ on $S$. Write down its action table.
\end{exercise}


\section{Groups}\label{sec:groups}

Groups are monoids in which every element has an inverse. If we think of these structures in terms of how they act on sets, the difference between groups and monoids is that the action of every group element can be undone. One way of thinking about groups is in terms of symmetries. For example, the rotations and reflections of a square form a group. 

Another way to think of the difference between monoids and groups is in terms of time. Monoids are likely useful in thinking about diffusion, in which time plays a role and things cannot be undone. Groups are more likely useful in thinking about mechanics, where actions are time-reversible. 


\subsection{Definition and examples}

\begin{definition}\label{def:group}\index{group}\index{monoid!inverse of an element in}

Let $(M,e,\star)$ be a monoid. An element $m\in M$ is said to {\em have an inverse} if there exists an $m'\in M$ such that $mm'=e$ and $m'm=e$. A {\em group} is a monoid $(M,e,\star)$ in which every element $m\in M$ has an inverse.

\end{definition}

\begin{proposition}

Suppose that $\mcM:=(M,e,\star)$ is a monoid and let $m\in M$ be an element. Then $m$ has at most one inverse.
\footnote{If $\mcM$ is a group then every element $m$ has exactly one inverse.}

\end{proposition}

\begin{proof}

Suppose that both $m'$ and $m''$ are inverses of $m$; we want to show that $m'=m''$. This follows by the associative law for monoids:
$$m'=m'(mm'')=(m'm)m''=m''.$$

\end{proof}

\begin{example}

The additive monoid $(\NN,0,+)$ is not a group because none of its elements are invertible, except for $0$. However, the monoid of integers $(\ZZ,0,+)$ is a group. The monoid of clock positions from Example \ref{ex:cyclic} is also a group. For example the inverse of $Q^5$ is $Q^7$ because $Q^5\star Q^7=e=Q^7\star Q^5$.

\end{example}

\begin{example}

Consider a square centered at the origin in $\RR^2$. It has rotational and mirror symmetries. There are eight of these, which we denote $$\{e,\rho,\rho^2,\rho^3,\phi,\phi\rho,\phi\rho^2,\phi\rho^3\},$$ where $\rho$ stands for $90^\circ$ counterclockwise rotation and $\phi$ stands for horizontal-flip (across the vertical axis). So relations include $\rho^4=e$, $\phi^2=e$, and $\rho^3\phi=\phi\rho$.

\end{example}

\begin{example}\label{ex:important groups}

The set of $3\times 3$ matrices can be given the structure of a monoid, where the identity element is the $3\times 3$ identity matrix, the multiplication is matrix multiplication. The subset of invertible matrices forms a group, called {\em the general linear group of dimension 3}\index{a group!$GL_3$} and denoted $GL_3$. Inside of $GL_3$ is the so-called {\em orthogonal group}, denoted $O_3$, of matrices $M$ such that $M^\m1=M^\top$. These matrices correspond to symmetries of the sphere centered at the origin.

Another interesting group is the Euclidean group\index{a group!$E_3$} $E(3)$ which consists of all {\em isometries} of $\RR^3$, i.e. all functions $\RR^3\to\RR^3$ that preserve distances.  

\end{example}

\begin{application}\label{app:groups for symmetry}\index{symmetry}

In \href{http://en.wikipedia.org/wiki/Crystallography}{\text crystallography} one is often concerned with the symmetries that arise in the arrangement $A$ of atoms in a molecule. To think about symmetries in terms of groups, we first define an {\em atom-arrangement} to be a finite subset $i\taking A\ss\RR^3$. A symmetry in this case is an isometry of $\RR^3$ (see Example \ref{ex:important groups}), say $f\taking\RR^3\to\RR^3$ such that there exists a dotted arrow making the diagram below commute:
$$
\xymatrix{A\ar@{-->}[r]\ar[d]_i&A\ar[d]^i\\\RR^3\ar[r]_f&\RR^3}
$$
That is, it's an isometry of $\RR^3$ such that each atom of $A$ is sent to a position currently occupied by an atom of $A$. It is not hard to show that the set of such isometries forms a group, called the \href{http://en.wikipedia.org/wiki/Space_group}{\em space group}\index{space group} of the crystal.

\end{application}

\begin{exercise}\label{exc:permutation}\index{set!permutation of}
Let $S$ be a finite set. A {\em permutation of $S$}\index{permutation} is an isomorphism $f\taking S\To{\iso}S$. 
\begin{center}
\parbox{2.3in}{
\includegraphics[height=2in]{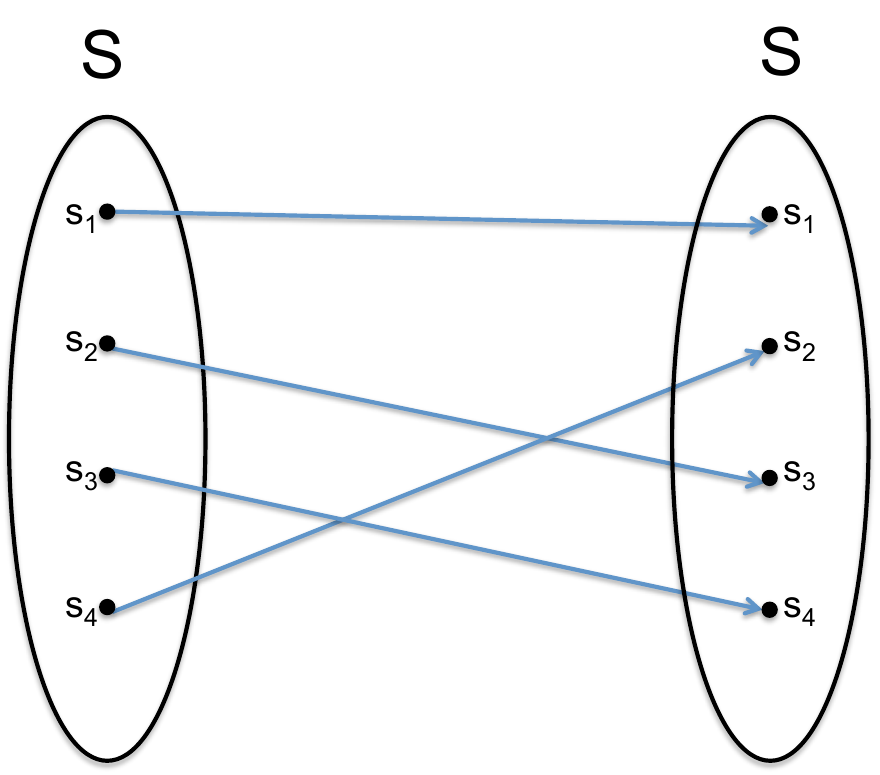}}
\end{center}
\sexc Come up with an identity, and a  multiplication formula, such that the set of permutations of $S$ forms a monoid. 
\next Is it a group?
\endsexc
\end{exercise}

\begin{exercise}
In Exercise \ref{exc:classify cyclic} you classified the cyclic monoids. Which of them are groups? 
\end{exercise}

\begin{definition}[Group action]\label{def:group action}\index{group!action}\index{action!of a group}

Let $(G,e,\star)$ be a group and $S$ a set. An {\em action} of $G$ on $S$ is a function $\acts\taking G\times S\to S$ such that for all $s\in S$ and $g,g'\in G$, we have
\begin{itemize}
\item $e\acts s=s$ and
\item $g\acts(g'\acts s)=(g\star g')\acts s.$
\end{itemize}
In other words, considering $G$ as a monoid, it is an action in the sense of Definition \ref{def:monoid action}.

\end{definition}

\begin{example}\label{ex:U(1)}\index{a group!$U(1)$}

When a group acts on a set, it has the character of \href{http://en.wikipedia.org/wiki/Symmetry}{\text symmetry}. For example, consider the group whose elements are angles $\theta$. This group may be denoted $U(1)$ and is often formalized as the unit circle in $\CC$ of complex numbers $z=a+bi$ such that $|z|=a^2+b^2=1$. The set of such points is given the structure of a group $(U(1),e,\star)$ by defining the identity element to be $e:=1+0i$ and the group law to be complex multiplication. But for those unfamiliar with complex numbers, this is simply angle addition where we understand that $360^\circ=0^\circ$. If $\theta_1=190^\circ$ and $\theta_2=278^\circ$, then $\theta_1\star\theta_2=468^\circ=108^\circ.$ In the language of complex numbers, $z=e^{i\theta}$.

The group $U(1)$ acts on any set that we can picture as having rotational symmetry about a fixed axis, such as the earth around the north-south axis. We will define $S=\{(x,y,z)\in\RR^3\|x^2+y^2+z^2=1\}$, the unit sphere, and understand the rotational action of $U(1)$ on $S$.\index{orbit!rotating earth}

We first show that $U(1)$ acts on $\RR^3$ by $\theta\acts(x,y,z)=(x\cos\theta+y\sin\theta, -x\sin\theta+y\cos\theta,z)$, or with matrix notation as 
$$\theta\acts(x,y,z)
:=(x,y,z)\left(\begin{array}{ccc}
\cos(\theta)&-\sin(\theta)&0\\
\sin(\theta)&\cos(\theta)&0\\
0&0&1\end{array}\right)
$$
\href{http://en.wikipedia.org/wiki/List_of_trigonometric_identities#Matrix_form}{\text Trigonometric identities} ensure that this is indeed an action.

In terms of action tables, we would need infinitely many columns to express this action. Here is a sample
$$
\begin{tabular}{| l || l | l | l |}
\bhline
\multicolumn{4}{|c|}{Action of $U(1)$ on $\RR^3$}\\\bhline
{$\RR^3$}&{$\theta=45^\circ$}&{$\theta=90^\circ$}&{$\theta=100^\circ$}\\\bbhline
(0,0,0)&(0,0,0)&(0,0,0)&(0,0,0)\\\hline
(1,0,0)&(.71,.71,0)&(0,1,0)&(-.17,.98,0)\\\hline
(0,1,-4.2)&(-.71,.71,-4.2)&(-1,0,-4.2)&(-.98,-.17,-4.2)\\\hline
(3,4,2)&(4.95,.71,2)&(-4,3,2)&(3.42,-3.65,2)\\\hline
$\vdots$&$\vdots$&$\vdots$&$\vdots$\\\bhline
\end{tabular}
$$

Finally, we are looking to see that the action preserves length so that if $(x,y,z)\in S$ then $\theta\acts(x,y,z)\in S$; this way we will have confirmed that $U(1)$ indeed acts on $S$. The calculation begins by assuming $x^2+y^2+z^2=1$ and checks 
$$
(x\cos\theta+y\sin\theta)^2+(-x\sin\theta+y\cos\theta)^2+z^2=x^2+y^2+z^2=1.
$$

\end{example}

\begin{exercise}\label{exc:permutation group}
Let $X$ be a set and consider the group of permutations of $X$ (see Exercise \ref{exc:permutation}), which we will denote $\Sigma_X$\index{a group!$\Sigma_X$}. Find a canonical action of $\Sigma_X$ on $X$.
\end{exercise}

\begin{definition}

Let $G$ be a group acting on a set $X$. For any point $x\in X$, the {\em orbit of $x$},\index{orbit}\index{action!orbit of} denoted $Gx$, is the set 
$$Gx:=\{x'\in X\|\exists g\in G \tn{ such that }gx=x'\}.$$

\end{definition}

\begin{application}

Let $S$ be the surface of the earth, understood as a sphere, and let $G=U(1)$ be the group of angles acting on $S$ as in Example \ref{ex:U(1)}. The orbit of any point $p=(x,y,z)\in S$ is the set of points on the same latitude line as $p$.

One may also consider a small band around the earth, i.e. the set $A=\{(x,y,z)\|1.0\leq x^2+y^2+z^2\leq 1.05\}$. The action of $U(1)\acts S$ extends to an action $U(1)\acts A$. The orbits are latitude-lines-at-altitude. A simplifying assumption in \href{http://en.wikipedia.org/wiki/Climatology}{\text climatology} may be given by assuming that $U(1)$ acts on all currents in the atmosphere in an appropriate sense. That way, instead of considering movement within the whole space $A$, we only allow movement that behaves the same way throughout each orbit of the group action.

\end{application}

\begin{exercise}~
\sexc Consider the $U(1)$ action on $\RR^3$ given in Example \ref{ex:U(1)}. Describe the set of orbits of this action.
\next What are the orbits of the action of the permutation group $\Sigma_{\{1,2,3\}}$ on the set $\{1,2,3\}$? (See Exercise \ref{exc:permutation group}.)
\endsexc
\end{exercise}

\begin{exercise}
Let $G$ be a group and $X$ a set on which $G$ acts by $\acts\taking G\times X\to X$. Is ``being in the same orbit" an equivalence relation on $X$? 
\end{exercise}

\begin{definition}\label{def:group homomorphism}\index{group!homomorphism of}

Let $G$ and $G'$ be groups. A {\em group homomorphism} $f\taking G\to G'$ is defined to be a monoid homomorphism $G\to G'$, where $G$ and $G'$ are being regarded as monoids in accordance with Definition \ref{def:group}.

\end{definition}


\section{Graphs}\label{sec:graphs}

In this course, unless otherwise specified, whenever we speak of graphs we are not talking about curves in the plane, such as parabolas, or pictures of functions generally. We are speaking of systems of vertices and arrows.

We will take our graphs to be {\em directed}, meaning that every arrow points {\em from} a vertex {\em to} a vertex; rather than merely connecting vertices, arrows will have direction. If $a$ and $b$ are vertices, there can be many arrows from $a$ to $b$, or none at all. There can be arrows from $a$ to itself. Here is the formal definition in terms of sets and functions.


\subsection{Definition and examples}

\begin{definition}\label{def:graph}\index{graph}

A {\em graph} $G$ consists of a sequence $G:=(V,A,src,tgt)$ where 
\begin{itemize}
\item $V$ is a set, called {\em the set of vertices of $G$} (singular:{\em vertex}),\index{vertex}
\item $A$ is a set, called {\em the set of arrows of $G$},\index{arrow}
\item $src\taking A\to V$ is a function, called {\em the source function for $G$}, and
\item $tgt\taking A\to V$ is a function, called {\em the target function for $G$}.
\end{itemize}
Given an arrow $a\in A$ we refer to $src(a)$ as the {\em source vertex} of $a$ and to $tgt(a)$ as the {\em target vertex} of $a$.

\end{definition}

To draw a graph, first draw a dot for every element of $V$. Then for every element $a\in A$, draw an arrow connecting dot $src(a)$ to dot $tgt(a)$.

\begin{example}[Graph]\label{ex:graph}

Here is a picture of a graph $G=(V,A,src,tgt)$:
\begin{align}\label{dia:graph}
G:=\parbox{2in}{\fbox{\xymatrix{\bullet^v\ar[r]^f&\bullet^w\ar@/_1pc/[r]_h\ar@/^1pc/[r]^g&\bullet^x\\\bullet^y\ar@(l,u)[]^i\ar@/^1pc/[r]^j&\bullet^z\ar@/^1pc/[l]^k}}}
\end{align} 
We have $V=\{v,w,x,y,z\}$ and $A=\{f,g,h,i,j,k\}$. The source and target functions $src,tgt\taking A\to V$ can be captured in the table to the left below:
\begin{align*}
\begin{array}{| l || l | l |}\bhline
{\bf A}&{\bf src}&{\bf tgt}\\\bbhline
f&v&w\\\hline
g&w&x\\\hline
h&w&x\\\hline
i&y&y\\\hline
j&y&z\\\hline
k&z&y\\\bhline
\end{array}
\hspace{1in}
\begin{array}{| l |}\bhline
{\bf V}\\\bbhline
v\\\hline
w\\\hline
x\\\hline
y\\\hline
z\\\bhline
\end{array}
\end{align*}
In fact, all of the data of the graph $G$ is captured in the two tables above---together they tell us the sets $A$ and $V$ and the functions $src$ and $tgt$.
\end{example}

\begin{example}

Every olog has an underlying graph. The additional information in an olog has to do with which pairs of paths are declared equivalent, as well as text that has certain English-readability rules.\index{olog!underlying graph}

\end{example}

\begin{exercise}
\sexc Draw the graph corresponding to the following tables:
\begin{align*}
\begin{array}{| l || l | l |}\bhline
{\bf A}&{\bf src}&{\bf tgt}\\\bbhline
f&v&w\\\hline
g&v&w\\\hline
h&v&w\\\hline
i&x&w\\\hline
j&z&w\\\hline
k&z&z\\\bhline
\end{array}
\hspace{1in}
\begin{array}{| l |}\bhline
{\bf V}\\\bbhline
u\\\hline
v\\\hline
w\\\hline
x\\\hline
y\\\hline
z\\\bhline
\end{array}
\end{align*}
\next Write down two tables, as above, corresponding to the following graph:
$$\fbox{\xymatrix{
\LMO{a}\ar[r]^{1}&\LMO{b}\ar[r]^2\ar@/^1pc/[r]^3&\LMO{c}\ar@/^1pc/[l]^4\ar[r]^5&\LMO{d}\\
\LMO{e}&\LMO{f}\ar[l]^6\ar[r]_7&\LMO{g}\ar[ur]_8}}
$$
\endsexc
\end{exercise}

\begin{exercise}
Let $A=\{1,2,3,4,5\}$ and $B=\{a,b,c\}$. Draw them and choose an arbitrary function $f\taking A\to B$ and draw it. Let $A\sqcup B$ be the coproduct of $A$ and $B$ (Definition \ref{def:coproduct}) and let $A\To{i_1}A\sqcup B\From{i_2}B$ be the two inclusions. Consider the two functions $src,tgt\taking A\to A\sqcup B$, where $src=i_1$ and $tgt$ is the composition $A\To{f}B\To{i_2}A\sqcup B$. Draw the associated graph $(A\sqcup B,A,src,tgt)$.
\end{exercise}

\begin{exercise}~
\sexc Let $V$ be a set. Suppose we just draw the elements of $V$ as vertices and have no arrows between them. Is this a graph?
\next Given $V$, is there any other ``canonical" or somehow automatic non-random procedure for generating a graph with those vertices? 
\endsexc
\end{exercise}

\begin{example}

Recall from Construction \ref{const:bipartite} the notion of bipartite graph, which we defined to be a span (i.e. pair of functions, see Definition \ref{def:span}) $A\From{f}R\To{g}B$. Now that we have a formal definition of graph, we might hope that bipartite graphs fit in, and they do. Let $V=A\sqcup B$ and let $i\taking A\to V$ and $j\taking B\to V$ be the inclusions. Let $src=i\circ f\taking R\to V$ and let $tgt=j\circ g\taking R\to V$ be the composites.
$$
\xymatrix{&A\ar[dr]^i\\R\ar@/^1pc/[rr]_{src}\ar@/_1pc/[rr]^{tgt}\ar[ur]^f\ar[dr]_g&&V\\&B\ar[ur]_j}
$$ 
Then $(V,R,src,tgt)$ is a graph that would be drawn exactly as we specified the drawing of spans in Construction \ref{const:bipartite}.

\end{example}

\begin{example}\label{ex:[n] as graph}

Let $n\in\NN$ be a natural number. The {\em chain graph of length $n$},\index{graph!chain} denoted $[n]$ is the graph depicted here:
$$
\xymatrix{
\LMO{0}\ar[r]&\LMO{1}\ar[r]&\cdots\ar[r]&\LMO{n}
}
$$
In general $[n]$ has $n$ arrows and $n+1$ vertices. In particular, when $n=0$ we have that $[0]$ is the graph consisting of a single vertex and no arrows. 

\end{example}

\begin{example}\label{ex:ZxG}

Let $G=(V,A,src,tgt)$ be a graph; we want to spread it out over discrete time, so that each arrow does not occur within a given time-slice but instead over a quantum unit of time. 

Let $N=(\NN,\NN,n\mapsto n,n\mapsto n+1)$ be the graph depicted 
$$\xymatrix{\LMO{0}\ar[r]^0&\LMO{1}\ar[r]^1&\LMO{2}\ar[r]^2&\cdots}$$
When we get to limits in a category, we will understand that products can be taken in the category of graphs (see  Example \ref{ex:product of graphs}), and $N\times G$ will make sense. For now, we construct it by hand.

Let $T(G)=(V\times \NN,A\times\NN,src',tgt')$ be a new graph, where for $a\in A$ and $n\in\NN$ we have $src'(a,n):=(src(a),n)$ and $tgt'(a,n)=(tgt(a),n+1)$. This may be a bit much to swallow, so try to simply understand what is being done in the following example. 

Let $G$ be the graph drawn below 
$$\xymatrix{\LMO{a}\ar@(ul,ur)[]^f\ar[d]_g\\\LMO{b}}$$
Then $T(G)$ will be the graph 
$$\xymatrix@=30pt{
\LMO{a0}\ar[r]^{f0}\ar[rd]_{g0}&\LMO{a1}\ar[r]^{f1}\ar[rd]_{g1}&\LMO{a2}\ar[r]^{f2}\ar[rd]_{g2}&\cdots\\
\LMO{b0}&\LMO{b1}&\LMO{b2}&\cdots
}
$$
As you can see, $f$-arrows still take $a$'s to $a$'s and $g$-arrows still take $a$'s to $b$'s, but they always march forward in time.

\end{example}

\begin{exercise}\label{exc:secret turing}
Let $G$ be the graph depicted below:
$$
\xymatrix{\LMO{a}\ar@/^1pc/[rr]^w\ar@(lu,ld)[]_x&&\LMO{b}\ar@/^1pc/[ll]^y\ar@(ur,dr)[]^z}
$$
Draw (using ellipses ``$\cdots$" if necessary) the graph $T(G)$ defined in Example \ref{ex:ZxG}.
\end{exercise}

\begin{exercise}\label{exc:lattice}
Consider the infinite graph $G=(V,A,src,tgt)$ depicted below,
$$
\xymatrix{
\vdots&\vdots&\vdots\\
(0,2)\ar[r]\ar[u]&(1,2)\ar[r]\ar[u]&(2,2)\ar[r]\ar[u]&\cdots\\
(0,1)\ar[r]\ar[u]&(1,1)\ar[r]\ar[u]&(2,1)\ar[r]\ar[u]&\cdots\\
(0,0)\ar[r]\ar[u]&(1,0)\ar[r]\ar[u]&(2,0)\ar[r]\ar[u]&\cdots}
$$
\sexc Write down the sets $A$ and $V$. 
\next What are the source and target function $A\to V$?  
\endsexc
\end{exercise}

\begin{exercise}\label{exc:(co)equalizer of graph}
A graph is a pair of functions $A\tto V$. This sets up the notion of equalizer and coequalizer (see Definitions \ref{def:equalizer} and \ref{def:coequalizer}). 
\sexc What feature of a graph is captured by the equalizer of its source and target functions? 
\next What feature of a graph is captured by the coequalizer of its source and target functions?
\endsexc
\end{exercise}


\subsection{Paths in a graph}\label{sec:paths in graph}\index{graph!paths}

We all know what a path in a graph is, especially if we understand that a path must always follow the direction of arrows. The following definition makes this idea precise. In particular, one can have paths of any finite length $n\in\NN$, even length $0$ or $1$. Also, we want to be able to talk about the source vertex and target vertex of a path, as well as concatenation of paths.

\begin{definition}\label{def:paths in graph}

Let $G=(V,A,src,tgt)$ be a graph. A {\em path of length $n$}\index{path} in $G$, denoted $p\in\Path_G^{(n)}$\index{a symbol!$\Path$} is a head-to-tail sequence \begin{align}\label{dia:path}p=(v_0\To{a_1}v_1\To{a_2}v_2\To{a_3}\ldots\To{a_n}v_n)\end{align} of arrows in $G$, which we denote by $v_0 a_1 a_2 \ldots a_n$. In particular we have canonical isomorphisms $\Path_G^{(1)}\iso A$ and $\Path_G^{(0)}\iso V$; we refer to the path of length 0 on vertex $v$ as the {\em trivial path on $v$} and denote it simply by $v$. We denote by $\Path_G$ the set of paths in $G$, $$\Path_G:=\bigcup_{n\in\NN}\Path_G^{(n)}.$$ Every path $p\in\Path_G$ has a source vertex and a target vertex, and we may denote these by $\ol{src},\ol{tgt}\taking\Path_G\to V$. If $p$ is a path with $\ol{src}(p)=v$ and $\ol{tgt}(p)=w$, we may denote it by $p\taking v\to w$. Given two vertices $v,w\in V$, we write $\Path_G(v,w)$ to denote the set of all paths $p\taking v\to w$.

There is a concatenation operation on paths.\index{concatenation!of paths} Given a path $p\taking v\to w$ and $q\taking w\to x$, we define the concatenation, denoted $p q\taking v\to x$ in the obvious way. If $p=va_1,a_2\ldots a_m$ and $q= wb_1b_2\ldots b_n$ then $pq=va_1\ldots a_mb_1\ldots b_n.$ In particular, if $p$ (resp. $r$) is the trivial path on vertex $v$ (resp. vertex $w$) then for any path $q\taking v\to w$, we have $pq=q$ (resp. $qr=q$). 

\end{definition}

\begin{example}

In Diagram (\ref{dia:graph}), page \pageref{dia:graph}, there are no paths from $v$ to $y$, one path ($f$) from $v$ to $w$, two paths ($fg$ and $fh$) from $v$ to $x$, and infinitely many paths $$\{y i^{p_1}(jk)^{q_1}\cdots i^{p_n}(jk)^{q_n}\;|\;n,p_1,q_1,\ldots,p_n,q_n\in\NN\}$$ from $y$ to $y$. There are other paths as well, including the five trivial paths.

\end{example}

\begin{exercise}
How many paths are there in the following graph? 
$$\xymatrix{\LMO{1}\ar[r]^{f}&\LMO{2}\ar[r]^{g}&\LMO{3}}$$
\end{exercise}

\begin{exercise}
Let $G$ be a graph and consider the set $\Path_G$ of paths in $G$. Suppose someone claimed that there is a monoid structure on the set $\Path_G$, where the multiplication formula is given by concatenation of paths. Are they correct? Why or why not? Hint: what should be the identity element?
\end{exercise}


\subsection{Graph homomorphisms}

A graph $(V,A,src,tgt)$ involves two sets and two functions. For two graphs to be comparable, their two sets and their two functions should be appropriately comparable.\index{appropriate comparison}

\begin{definition}\label{def:graph homomorphism}\index{graph!homomorphism}

Let $G=(V,A,src,tgt)$ and $G'=(V',A',src',tgt')$ be graphs. A {\em graph homomorphism $f$ from $G$ to $G'$}, denoted $f\taking G\to G'$, consists of two functions $f_0\taking V\to V'$ and $f_1\taking A\to A'$ such that the two diagrams below commute:
\begin{align}\label{dia:graph hom}
\xymatrix{A\ar[r]^{f_1}\ar[d]_{src}&A'\ar[d]^{src'}\\V\ar[r]_{f_0}&V'
}\hspace{1in}
\xymatrix{A\ar[r]^{f_1}\ar[d]_{tgt}&A'\ar[d]^{tgt'}\\V\ar[r]_{f_0}&V'
}
\end{align}

\end{definition}

\begin{remark}

The above conditions (\ref{dia:graph hom}) may look abstruse at first, but they encode a very important idea, roughly stated ``arrows are bound to their vertices". Under a map of graphs $G\to G'$ , one cannot flippantly send an arrow of $G$ any old arrow of $G'$: it must still connect the vertices it connected before. Below is an example of a mapping that does not respect this condition: $a$ connects $1$ and $2$ before, but not after:
$$
\fbox{\xymatrix{\LMO{\color{red}{1}}\ar[r]^{a}&\LMO{\color{blue}{2}}}}
\xymatrix{~\ar[rr]^{1\mapsto 1',2\mapsto 2', a\mapsto a'}&\hsp&~}
\fbox{\xymatrix{\LMO{\color{red}{1'}}&\LMO{\color{blue}{2'}}\ar[r]^{a'}&\LMO{\color{ForestGreen}{3'}}}}
$$
The commutativity of the diagrams in (\ref{dia:graph hom}) is exactly what is needed to ensure that arrows are handled in the expected way by a proposed graph homomorphism.
 
\end{remark}

\begin{example}[Graph homomorphism]\label{ex:graph hom}

Let $G=(V,A,src,tgt)$ and $G'=(V',A',src',tgt')$ be the graphs drawn to the left and right (respectively) below:
\begin{align}\label{dia:graph hom example}
\parbox{1.5in}{\fbox{\xymatrix{\LMO{\color{red}{1}}\ar[r]^a\ar@/^1pc/[d]^d\ar@/_1pc/[d]_c&\LMO{\color{ForestGreen}{2}}\ar[r]^b&\LMO{\color{red}{3}}\\\LMO{4}&\LMO{\color{blue}{5}}\ar[r]^e&\LMO{\color{blue}{6}}}}}
\parbox{1in}{\xymatrix{~\ar[rr]^{\parbox{.8in}{\vspace{-.2in}\footnotesize$1\mapsto 1', 2\mapsto 2',\\ 3\mapsto 1',4\mapsto 4',\\ 5\mapsto 5',6\mapsto5'$}}&\hsp&~}}
\parbox{.8in}{\fbox{\xymatrix{\LMO{\color{red}{1'}}\ar@<.5ex>[r]^w\ar[d]_y&\LMO{\color{ForestGreen}{2'}}\ar@<.5ex>[l]^x\\\LMO{4'}&\LMO{\color{blue}{5'}}\ar@(r,u)[]_z}}}
\end{align}
The colors indicate our choice of function $f_0\taking V\to V'$. Given that choice, condition (\ref{dia:graph hom}) imposes in this case that there is a unique choice of graph homomorphism $f\taking G\to G'$. 

\end{example}

\begin{exercise}~
\sexc Where are $a,b,c,d,e$ sent under $f_1\taking A\to A'$ in Diagram (\ref{dia:graph hom example})? 
\next Choose a couple elements of $A$ and check that they behave as specified by Diagram (\ref{dia:graph hom}).
\endsexc
\end{exercise}

\begin{exercise}
Let $G$ be a graph, let $n\in\NN$ be a natural number, and let $[n]$ be the chain graph of length $n$, as in Example \ref{ex:[n] as graph}. Is a path of length $n$ in $G$ the same thing as a graph homomorphism $[n]\to G$, or are there subtle differences? More precisely, is there always an isomorphism between the set of graph homomorphisms $[n]\to G$ and the set $\Path_G^{(n)}$ of length-$n$ paths in $G$?
\end{exercise}

\begin{exercise}
Given a morphism of graphs $f\taking G\to G'$, there an induced function $\Path(f)\taking\Path(G)\to\Path(G')$. 
\sexc Is it the case that for every $n\in\NN$, the function $\Path(f)$ carries $\Path^{(n)}(G)$ to $\Path^{(n)}(G')$, or can path lengths change in this process?
\next Suppose that $f_0$ and $f_1$ are injective (meaning no two distinct vertices in $G$ are sent to the same vertex (respectively for arrows) under $f$). Does this imply that $\Path(f)$ is also injective (meaning no two distinct paths are sent to the same path under $f$)?
\next Suppose that $f_0$ and $f_1$ are surjective (meaning every vertex in $G'$ and every arrow in $G'$ is in the image of $f$). Does this imply that $\Path(f)$ is also surjective? Hint: at least one of the answers to these three questions is ``no".
\endsexc
\end{exercise}

\begin{exercise}\label{exc:single condition for graph hom}

Given a graph $(V,A,src,tgt)$, let $i\taking A\to V\times V$ be function guaranteed by the universal property for products, as applied to $src,tgt\taking A\to V$. One might hope to summarize Condition (\ref{dia:graph hom}) for graph homomorphisms by the commutativity of the single square 
\begin{align}\label{dia:equiv graph hom}
\xymatrix{A\ar[r]^{f_1}\ar[d]_{i}&A'\ar[d]^{i'}\\V\times V\ar[r]_{f_0\times f_0}&V'\times V'.}
\end{align}
Is the commutativity of the diagram in (\ref{dia:equiv graph hom}) indeed equivalent to the commutativity of the diagrams in (\ref{dia:graph hom})?
\end{exercise}


\subsubsection{Binary relations and graphs}

\begin{definition}\label{def:binary relation}\index{relation!binary}

Let $X$ be a set. A {\em binary relation on $X$} is a subset $R\ss X\times X$. 

\end{definition}

If $X=\NN$ is the set of integers, then the usual $\leq$ defines a relation on $X$: given $(m,n)\in\NN\times\NN$, we put $(m,n)\in R$ iff $m\leq n$. As a table it might be written as to the left
\begin{align}\label{dia:3 relations}
\begin{tabular}{|p{.7cm}|p{.7cm}|}
\bhline
\multicolumn{2}{|c|}{$m\leq n$}\\\bhline
m&n\\\bbhline
0&0\\\hline
0&1\\\hline
1&1\\\hline
0&2\\\hline
1&2\\\hline
2&2\\\hline
0&3\\\bhline
$\vdots$&$\vdots$\\\hline
\end{tabular}
\hspace{1in}
\begin{tabular}{|p{.7cm}|p{.7cm}|}
\bhline
\multicolumn{2}{|c|}{$n=5m$}\\\bhline
m&n\\\bbhline
0&0\\\hline
1&5\\\hline
2&10\\\hline
3&15\\\hline
4&20\\\hline
5&25\\\hline
6&30\\\bhline
$\vdots$&$\vdots$\\\hline
\end{tabular}
\hspace{1in}
\begin{tabular}{|p{.7cm}|p{.7cm}|}
\bhline
\multicolumn{2}{|c|}{$|n-m|\leq 1$}\\\bhline
m&n\\\bbhline
0&0\\\hline
0&1\\\hline
1&0\\\hline
1&1\\\hline
1&2\\\hline
2&1\\\hline
2&2\\\hline
$\vdots$&$\vdots$\\\hline
\end{tabular}
\end{align}
The middle table is the relation $\{(m,n)\in\NN\times\NN\|n=5m\}\ss\NN\times\NN$ and the right-hand table is the relation $\{(m,n)\in\NN\times\NN\||n-m|\leq 1\}\ss\NN\times\NN$. 

\begin{exercise}
A relation on $\RR$ is a subset of $\RR\times\RR$, and one can indicate such a subset of the plane by shading. Choose an error bound $\epsilon>0$ and draw the relation one might refer to as ``$\epsilon$-approximation". To say it another way, draw the relation ``$x$ is within $\epsilon$ of $y$".
\end{exercise}

\begin{exercise}[Binary relations to graphs]\label{exc:rel to graph}\index{relation!graph of}

\sexc If $R\ss S\times S$ is a binary relation, find a natural way to make a graph out of it, having vertices $S$. 
\next What is the set $A$ of arrows? 
\next What are the source and target functions $src,tgt\taking A\to S$?
\next Take the left-hand table in (\ref{dia:3 relations}) and consider its first $7$ rows (i.e. forget the $\vdots$). Draw the corresponding graph (do you see a tetrahedron?). 
\next Do the same for the right-hand table.
\endsexc
\end{exercise}

\begin{exercise}[Graphs to binary relations]\label{ex:graph to rel}~
\sexc If $(V,A,src,tgt)$ is a graph, find a natural way to make a binary relation $R\ss V\times V$ out of it. 
\next Take the left-hand graph $G$ from (\ref{dia:graph hom example}) and write out the corresponding binary relation in table form.
\endsexc
\end{exercise}

\begin{exercise}[Going around the loops]
\sexc Given a binary relation $R\ss S\times S$, you know from Exercise \ref{exc:rel to graph} how to construct a graph out of it, and from Exercise \ref{ex:graph to rel} how to make a new binary relation out of that. How does the resulting relation compare with the original?
\next Given a graph $(V,A,src,tgt)$, you know from Exercise \ref{ex:graph to rel} how to make a new binary relation out of it, and from Exercise \ref{exc:rel to graph} how to construct a new graph out of that. How does the resulting graph compare with the original? 
\endsexc
\end{exercise}


\section{Orders}\label{sec:orders}

People usually think of certain sets as though they just {\em are} ordered, e.g. that an order is ordained by God. For example the natural numbers just {\em are} ordered. The letters in the alphabet just {\em are} ordered. But in fact we put orders on sets, and some are simply more commonly used in culture. One could order the letters in the alphabet by frequency of use and $e$ would come before $a$. Given different purposes, we can put different orders on the same set. For example in Exercise \ref{exc:divides as po} we will give a different ordering on the natural numbers that is useful in elementary number theory.

In science, we might order the set of materials in two different ways. In the first, we consider material $A$ to be ``before" material $B$ if $A$ is an ingredient or part of $B$, so water would be before concrete. But we could also order materials based on how electrically conductive they are, whereby concrete would be before water. This section is about different kinds of orders.


\subsection{Definitions of preorder, partial order, linear order}

\begin{definition}\label{def:orders}\index{order}

Let $S$ be a set and $R\ss S\times S$ a binary relation on $S$; if $(s,s')\in R$ we will write $s\leq s'$. Then we say that $R$ is a {\em preorder}\index{order!preorder} if, for all $s,s',s''\in S$ we have
\begin{description}
\item[Reflexivity:] $s\leq s$, and
\item[Transitivity:] if $s\leq s'$ and $s'\leq s''$, then $s\leq s''$.
\end{description}
We say that $R$ is a {\em partial order}\index{order!partial order} if it is a preorder and, in addition, for all $s,s'\in S$ we have
\begin{description}
\item[Antisymmetry:] If $s\leq s'$ and $s'\leq s$, then $s=s'$.
\end{description}
We say that $R$ is a {\em linear order}\index{order!linear order} if it is a partial order and, in addition, for all $s,s'\in S$ we have
\begin{description}
\item[Comparability:] Either $s\leq s'$ or $s'\leq s$.
\end{description}
We denote such a preorder (or partial order or linear order) by $(S,\leq)$.
\end{definition}

\begin{exercise}~
\sexc Decide whether the table to the left in Display (\ref{dia:3 relations}) constitutes a linear order. 
\next Show that neither of the other tables are even preorders.
\endsexc
\end{exercise}

\begin{example}[Partial order not linear order]\label{ex:pre not par}

We will draw an olog for playing cards. 
\begin{align}\label{dia:card olog}
\footnotesize
\xymatrixnocompile@=15pt{
\obox{}{.3in}{a diamond}\ar[dr]^{\tn{is}}&&\obox{}{.4in}{a heart}\ar[dl]_{\tn{is}}&&\obox{}{.35in}{a club}\ar[dr]^{\tn{is}}&&\obox{}{.4in}{a spade}\ar[dl]_{\tn{is}}\\
&\obox{}{.25in}{a red card}\ar[drr]^{\tn{is}}&&&&\obox{}{.4in}{a black card}\ar[dll]_{\tn{is}}\\
\obox{}{.45in}{a 4 of diamonds}\ar[d]^{\tn{is}}\ar[uu]_{\tn{is}}&&&\obox{}{.35in}{a card}&&&\obox{}{.4in}{a black queen}\ar[d]^{\tn{is}}\ar[ul]_{\tn{is}}\\
\obox{}{.2in}{a 4}\ar[rr]^{\tn{is}}&&\obox{}{.4in}{a numbered card}\ar[ur]^{\tn{is}}&&\obox{}{.3in}{a face card}\ar[ul]_{\tn{is}}&&\obox{}{.4in}{a queen}\ar[ll]_{\tn{is}}
}
\end{align}
We can put a binary relation on the set of boxes here by saying $A\leq B$ if there is a path $A\to B$. One can see immediately that this is a preorder because length=0 paths give reflexivity and concatenation of paths gives transitivity. To see that it is a partial order we only note that there are no loops. But this partial order is not a linear order because there is no path (in either direction) between, e.g., \fakebox{a 4 of diamonds} and \fakebox{a black queen}, so it violates the comparability condition.

\end{example}

\begin{remark}

Note that olog (\ref{dia:card olog}) in Example \ref{ex:pre not par} is a good olog in the sense that given any collection of cards (e.g. choose 45 cards at random from each of 7 decks and throw them in a pile), they can be classified according to the boxes of (\ref{dia:card olog}) such that every arrow indeed constitutes a function (which happens to be injective). For example the arrow $\fakebox{a heart}\Too{\tn{is}}\fakebox{a red card}$ is a function from the set of chosen hearts to the set of chosen red cards.

\end{remark}

\begin{example}[Preorder not partial order]

Every equivalence relation is a preorder but rarely are they partial orders. For example if $S=\{1,2\}$ and we put $R=S\times S$, then this is an equivalence relation. It is a preorder but not a partial order (because $1\leq 2$ and $2\leq 1$, but $1\neq 2$, so antisymmetry fails).

\end{example}

\begin{application}

Classically, we think of time as linearly ordered. A nice model is $(\RR,\leq)$, the usual linear order on the set of real numbers. But according to the \href{http://en.wikipedia.org/wiki/Relativity_of_simultaneity}{\text theory of relativity}, there is not actually a single order to the events in the universe. Different observers correctly observe different orders on the set of events, and so in some sense on time itself. 

\end{application}

\begin{example}[Finite linear orders]\label{ex:finite lo}\index{linear order!finite}

Let $n\in\NN$ be a natural number. Define a linear order on the set $\{0,1,2,\ldots,n\}$ in the standard way. Pictorially, 
$$
[n]:=\xymatrix{\LMO{0}\ar[r]&\LMO{1}\ar[r]&\LMO{2}\ar[r]&\cdots\ar[r]&\LMO{n}}
$$\index{a symbol!$[n]$}

Every finite linear order, i.e. linear order on a finite set, is of the above form. That is, though the labels might change, the picture would be the same. We can make this precise when we have a notion of morphism of orders (see Definition \ref{def:morphism of orders})

\end{example}

\begin{exercise}
Let $S=\{1,2,3,4\}$. 
\sexc Find a preorder $R\ss S\times S$ such that the set $R$ is as small as possible. Is it a partial order? Is it a linear order?
\next Find a preorder $R'\ss S\times S$ such that the set $R'$ is as large as possible. Is it a partial order? Is it a linear order?
\endsexc
\end{exercise}

\begin{exercise}~
\sexc List all the preorder relations possible on the set $\{1,2\}$.
\next For any $n\in\NN$, how many linear orders exist on the set $\{1,2,3,\ldots,n\}$. 
\next Does your formula work when $n=0$?
\endsexc
\end{exercise}

\begin{remark}\label{rem:preorder to graph}\index{preorder!converting to graph}

We can draw any preorder $(S,\leq)$ as a graph with vertices $S$ and with an arrow $a\to b$ if $a\leq b$. These are precisely the graphs with the following two properties for any vertices $a,b\in S$:
\begin{enumerate}[\hsp 1.]
\item there is at most one arrow $a\to b$, and
\item if there is a path from $a$ to $b$ then there is an arrow $a\to b$.
\end{enumerate}
If $(S,\leq)$ is a partial order then the associated graph has an additional ``no loops" property,
\begin{enumerate}[\hsp 3.]
\item if $n\in\NN$ is an integer with $n\geq 2$ then there are no paths of length $n$ that start at $a$ and end at $a$.
\end{enumerate}
If $(S,\leq)$ is a linear order then there is an additional ``comparability" property,
\begin{enumerate}[\hsp 4.]
\item for any two vertices $a,b$ there is an arrow $a\to b$ or an arrow $b\to a$.
\end{enumerate}

Given a graph $G$, we can create a binary relation $\leq$ on its set $S$ of vertices as follows. Say $a\leq b$ if there is a path in $G$ from $a$ to $b$. This relation will be reflexive and transitive, so it is a preorder. If the graph satisfies Property 3 then the preorder will be a partial order, and if the graph also satisfies Property 4 then the partial order will be a linear order. Thus graphs give us a nice way to visualize orders.\index{graph!converting to a preorder}

\end{remark}

\begin{slogan}
A graph generates a preorder: $v\leq w$ if there is a path $v\to w$. This is a great way to picture a preorder. 
\end{slogan}

\begin{exercise}
Let $G=(V,A,src,tgt)$ be the graph below. 
$$\fbox{\xymatrix{
\LMO{a}\ar[r]&\LMO{b}\ar@/^1pc/[r]&\LMO{c}\ar@/^1pc/[l]\ar[r]&\LMO{d}\\
\LMO{e}&\LMO{f}\ar[l]\ar[r]&\LMO{g}\ar[ur]}}
$$
In the corresponding pre-order which of the following are true: 
\sexc $a\leq b$?
\next $a\leq c$?
\next $c\leq b$?
\next $b=c$?
\next $e\leq f$?
\next $f\leq d$?
\endsexc
\end{exercise}

\begin{exercise}\label{exc:power poset}\index{power set!as poset}~
\sexc Let $S=\{1,2\}$. The subsets of $S$ form a partial order; draw the associated graph. 
\next Repeat this for $Q=\emptyset$, $R=\{1\}$, and $T=\{1,2,3\}$. 
\next Do you see $n$-dimensional cubes?
\endsexc
\end{exercise}

\begin{definition}\label{def:clique}\index{preorder!clique in}

Let $(S,\leq)$ be a preorder. A {\em clique} is a subset $S'\ss S$ such that for each $a,b\in S'$ one has $a\leq b$.

\end{definition}

\begin{exercise}
True or false: a partial order is a preorder that has no cliques. (If false, is there a ``nearby" true statement?)
\end{exercise}

\begin{example}\label{ex:preorder generated}\index{preorder!generated}

Let $X$ be a set and $R\ss X\times X$ a relation. For elements $x,y\in X$ we will say there is an {\em $R$-path} from $x$ to $y$ if there exists a natural number $n\in\NN$ and elements $x_0,x_1,\ldots,x_n$ such that
\begin{enumerate}
\item $x_0=x$,
\item $x_n=y$, and
\item for all $i\in\NN$, if $0\leq i\leq n-1$ then $(x_i,x_{i+1})\in R$.
\end{enumerate}
Let $\ol{R}$ denote the relation where $(x,y)\in\ol{R}$ if there exists an $R$-path from $x$ to $y$. We call $\ol{R}$ the {\em preorder generated by $R$.} We note some facts about $\ol{R}$.
\begin{description}
\item[Containment.] If $(x,y)\in R$ then $(x,y)\in\ol{R}$. That is $R\ss\ol{R}$.
\item[Reflexivity]. For all $x\in X$ we have $(x,x)\in\ol{R}$. 
\item[Transitivity.] For all $x,y,z\in X$, if $(x,y)\in\ol{R}$ and $(y,z)\in\ol{R}$ then $(x,z)\in\ol{R}$.
\end{description}
To check the containment claim, just use $n=1$ so $x_0=x$ and $x_n=y$. To check the reflexivity claim, use $n=0$ so $x_0=x=y$ and condition 3 is vacuously satisfied. To check transitivitiy, suppose given $R$-paths $x=x_0,x_1,\ldots,x_n=y$ and $y=y_0,y_1,\ldots,y_p=z$; then $x=x_0,x_1,\ldots x_n,y_1,\ldots,y_p=z$ will be an $R$-path from $x$ to $z$.

The point is that we can turn any relation into a preorder in a canonical way. Here is a concrete case of the above idea.

Let $X=\{a,b,c,d\}$ and suppose given the relation $\{(a,b),(b,c),(b,d),(d,c),(c,c)\}$. This is neither reflexive nor transitive, so it's not a preorder. To make it a preorder we follow the above prescription. Starting with $R$-paths of length $n=0$ we put  $\{(a,a), (b,b), (c,c), (d,d)\}$ into $\ol{R}$. The $R$-paths of length 1 add our original elements, $\{(a,b),(b,c),(b,d),(d,c),(c,c)\}$. We don't mind redundancy (e.g. $(c,c)$), but from now on in this example we will only write down the new elements. The $R$-paths of length 2 add $\{(a,c),(a,d)\}$ to $\ol{R}$. One can check that $R$-paths of length 3 and above do not add anything new to $\ol{R}$, so we are done. The relation $$\ol{R}=\{(a,a), (b,b), (c,c), (d,d), (a,b), (b,c), (b,d), (d,c), (a,c), (a,d)\}$$ is reflexive and transitive, hence a preorder.

\end{example}

\begin{exercise}

Let $X=\{a,b,c,d,e,f\}$ and let $R=\{(a,b),(b,c),(b,d),(d,e),(f,a)\}$. 
\sexc What is the preorder $\ol{R}$ generated by $R$?
\next Is it a partial order?
\endsexc
\end{exercise}

\begin{exercise}
Let $X$ be the set of people and let $R\ss X\times X$ be the relation with $(x,y)\in R$ if $x$ is the child of $y$. Describe the preorder generated by $R$.
\end{exercise}


\subsection{Meets and joins}\label{sec:meets and joins}

Let $X$ be any set. Recall from Definition \ref{def:subobject classifier} that the powerset of $X$, denoted $\PP(X)$ is the set of subsets of $X$. There is a natural order on $\PP(X)$ given by the subset relationship, as exemplified in Exercise \ref{exc:power poset}. Given two elements $a,b\in\PP(X)$ we can consider them as subsets of $X$ and take their intersection as an element of $\PP(X)$ which we denote $a\wedge b$. We can also consider them as subsets of $X$ and take their union as an element of $\PP(X)$ which we denote $a\vee b$. The intersection and union operations are generalized in the following definition.

\begin{definition}\label{def:meets and joins}\index{preorder!meet}\index{preorder!join}\index{meet}\index{join}

Let $(S,\leq)$ be a preorder and let $s,t\in S$ be elements. A {\em meet of $s$ and $t$} is an element $w\in S$ satisfying the following universal property: 
\begin{itemize}
\item $w\leq s$ and $w\leq t$ and, 
\item for any $x\in S$, if $x\leq s$ and $x\leq t$ then $x\leq w$.
\end{itemize}
If $w$ is a meet of $s$ and $t$, we write $w\iso s\wedge t$.

A {\em join of $s$ and $t$} is an element $w\in S$ satisfying the following universal property: 
\begin{itemize}
\item $s\leq w$ and $t\leq w$ and, 
\item for any $x\in S$, if $s\leq x$ and $t\leq x$ then $w\leq x$.
\end{itemize}
If $w$ is a join of $s$ and $t$, we write $w\iso s\vee t$.

\end{definition}

That is, the meet of $s$ and $t$ is the biggest thing smaller than both, i.e. a {\em greatest lower bound}, and the join of $s$ and $t$ is the smallest thing bigger than both, i.e. a {\em least upper bound}. Note that the meet of $s$ and $t$ might be $s$ or $t$ itself.  Note that $s$ and $t$ may have more than one meet (or more than one join). However, any two meets of $s$ and $t$ must be in the same clique, by the universal property (and the same for joins).

\begin{exercise}
Consider the partial order from Example \ref{ex:pre not par}. 
\sexc What is the join of \fakebox{a diamond} and \fakebox{a heart}? 
\next What is the meet of \fakebox{a black card} and \fakebox{a queen}? 
\next What is the meet of \fakebox{a diamond} and \fakebox{a card}?
\endsexc
\end{exercise}

Not every two elements in a preorder need have a meet, nor need they have a join. 

\begin{exercise}\label{exc:not all meets and joins}~
\sexc If possible, find two elements in the partial order from Example \ref{ex:pre not par} that do not have a meet.
\footnote{Use the displayed preorder, not any kind of ``completion of what's there".} 
\next If possible, find two elements that do not have a join (in that preorder).
\endsexc
\end{exercise}

\begin{exercise}
As mentioned in the introduction to this section, the power set $S:=\PP(X)$ of any set $X$ naturally has the structure of a partial order. Its elements $s\in S$ correspond to subsets $s\ss X$, and we put $s\leq t$ if and only if $s\ss t$ as subsets of $X$. The meet of two elements is their intersection as subsets of $X$, $s\wedge t= s\cap t$, and the join of two elements is their union as subsets of $X$, $s\vee t=s\cup t$.
\sexc Is it possible to put a monoid structure on the set $S$ in which the multiplication formula is given by meets? If so, what would the identity element be?
\next Is it possible to put a monoid structure on the set $S$ in which the multiplication formula is given by joins? If so, what would the identity element be?
\endsexc
\end{exercise}

\begin{example}[Trees]\label{ex:tree}

A {\em tree}\index{tree}\index{order!tree}, i.e. a system of nodes and branches, all of which emanate from a single node called the {\em root}\index{tree!root}, is a partial order, but generally not a linear order. A tree $(T,\leq)$ can either be oriented toward the root (so the root is the largest element) or away from the root (so the root is the smallest element); let's only consider the latter. 

Below is a tree, pictured as a graph. The root is labeled $e$.
\begin{align}\label{dia:tree}
\xymatrix@=10pt{
&&&&&&\LMO{a}\\
&&\LMO{b}\ar[rr]&&\LMO{c}\ar[urr]\ar[rr]\ar[drr]&&\LMO{d}\\
\LMO{e}\ar[urr]\ar[drr]&&&&&&\LMO{f}\\
&&\LMO{g}\ar[rr]\ar[drr]&&\LMO{h}\\
&&&&\LMO{i}
}
\end{align}

In a tree, every pair of elements $s, t\in T$ has a meet $s\wedge t$ (their closest mutual ancestor). On the other hand if $s$ and $t$ have a join $c=s\vee t$ then either $c=s$ or $c=t$. 

\end{example}

\begin{exercise}
Consider the tree drawn in (\ref{dia:tree}).
\sexc What is the meet $i\wedge h$?
\next What is the meet $h\wedge b$?
\next What is the join $b\vee a$?
\next What is the join $b\vee g$?
\endsexc
\end{exercise}

\subsection{Opposite order}

\begin{definition}\label{def:opposite order}\index{order!opposite}

Let $\mcS:=(S,\leq)$ be a preorder. The {\em opposite preorder}, denoted $\mcS\op$ is the preorder $(S,\leq\op)$ having the same set of elements but where $s\leq\op s'$ iff $s'\leq s$.

\end{definition}

\begin{example}

Recall the preorder $\mcN:=(\NN,{\tt divides})$ from Exercise \ref{exc:divides as po}. Then $\mcN\op$ is the set of natural numbers but where $m\leq n$ iff $m$ is a multiple of $n$. So $6\leq 2$ and $6\leq 3$.

\end{example}

\begin{exercise}
Suppose that $\mcS:=(S,\leq)$ is a preorder. 
\sexc If $\mcS$ is a partial order, is $\mcS\op$ also a partial order? 
\next If $\mcS$ is a linear order, is $\mcS\op$ a linear order?
\endsexc
\end{exercise}

\begin{exercise}
Suppose that $\mcS:=(S,\leq)$ is a preorder, and that $s_1,s_2\in S$ have join $t$ in $\mcS$. The preorder $\mcS\op$ has the same elements as $\mcS$. Is $t$ the join of $s_1$ and $s_2$ in $\mcS\op$, or is it their meet, or is it not necessarily their meet nor their join?
\end{exercise}


\subsection{Morphism of orders}

An order $(S,\leq)$, be it a preorder, a partial order, or a linear order, involves a set and a binary relations. For two orders to be comparable, their sets and their relations should be appropriately comparable.\index{appropriate comparison}

\begin{definition}\label{def:morphism of orders}\index{order!morphism}

Let $\mcS:=(S,\leq)$ and $\mcS':=(S',\leq')$ be preorders (respectively partial orders or linear orders). A {\em morphism of preorders} (resp. {\em of partial orders} or {\em of linear orders}) $f$ {\em from $\mcS$ to $\mcS'$}, denoted $f\taking\mcS\to\mcS'$, is a function $f\taking S\to S'$ such that, for every pair of elements $s_1,s_2\in S$, if $s_1\leq s_2$ then $f(s_1)\leq' f(s_2)$.

\end{definition}

\begin{example}

Let $X$ and $Y$ be sets, let $f\taking X\to Y$ be a function. Then for every subset $X'\ss X$, its image $f(X')\ss Y$ is a subset (see Section \ref{sec:functions}). Thus we have a function $F\taking\PP(X)\to\PP(Y)$, given by taking images. This is a morphism of partial orders $(\PP(X),\ss)\to(\PP(Y),\ss)$. Indeed, if $a\ss b$ in $\PP(X)$ then $f(a)\ss f(b)$ in $\PP(Y)$.

\end{example}

\begin{application}

It's often said that ``a team is only as strong as its weakest member". Is this true for materials? The hypothesis that a material is only as strong as its weakest constituent can be understood as follows. 

Recall from the introduction to this section (see \ref{sec:orders}, page \pageref{sec:orders}) that we can put several different orders on the set $M$ of materials. One example there was the order given by constituency ($m\leq_C m'$ if $m$ is an ingredient or constituent of $m'$). Another order is given by strength: $m\leq_S m'$ if $m'$ is stronger than $m$ (in some fixed setting). 

Is it true that if material $m$ is a constituent of material $m'$ then the strength of $m'$ is less than or equal to the strength of $m$? This is the substance of our quote above. Mathematically the question would be posed, ``is there a morphism of preorders $(M,\leq_C)\too(M,\leq_S\op)$?"

\end{application}

\begin{exercise}
Let $X$ and $Y$ be sets, let $f\taking X\to Y$ be a function. Then for every subset $Y'\ss Y$, its preimage $f^\m1(Y')\ss X$ is a subset (see Definition \ref{def:preimage}). Thus we have a function $F\taking\PP(Y)\to\PP(X)$, given by taking preimages. Is it a morphism of partial orders?
\end{exercise}

\begin{example}\label{ex:discrete and indiscrete}

Let $S$ be a set. The smallest preorder structure that can be put on $S$ is to say $a\leq b$ iff $a=b$. This is indeed reflexive and transitive, and it is called the {\em discrete preorder on $S$}.\index{preorder!discrete}

The largest preorder structure that can be put on $S$ is to say $a\leq b$ for all $a,b\in S$. This again is reflexive and transitive, and it is called the {\em indiscrete preorder on $S$}.\index{preorder!indiscrete}

\end{example}

\begin{exercise}
Let $S$ be a set and let $(T,\leq_T)$ be a preorder. Let $\leq_D$ be the discrete preorder on $S$. Given a morphism of preorders $(S,\leq_D)\to (T,\leq_T)$ we get a function $S\to T$. 
\sexc Which functions $S\to T$ arise in this way? 
\next Given a morphism of preorders $(T,\leq_T)\to(S,\leq_D)$, we get a function $T\to S$. In terms of $\leq_T$, which functions $T\to S$ arise in this way?
\endsexc
\end{exercise}

\begin{exercise}
Let $S$ be a set and let $(T,\leq_T)$ be a preorder. Let $\leq_I$ be the indiscrete preorder on $S$. Given a morphism of preorders $(S,\leq_I)\to (T,\leq_T)$ we get a function $S\to T$. 
\sexc In terms of $\leq_T$, which functions $S\to T$ arise in this way? 
\next Given a morphism of preorders $(T,\leq_T)\to(S,\leq_I)$, we get a function $T\to S$. In terms of $\leq_T$, which functions $T\to S$ arise in this way?
\endsexc
\end{exercise}


\subsection{Other applications}


\subsubsection{Biological classification}\index{biological classification}

\href{http://en.wikipedia.org/wiki/Biological_classification}{\text Biological classification} is a method for dividing the set of organisms into distinct classes, called taxa. In fact, it turns out that such a classification, say a phylogenetic tree, can be understood as a partial order $C$ on the set of taxa. The typical {\em ranking} of these taxa, including kingdom, phylum, etc., can be understood as morphism of orders $f\taking C\to [n]$, for some $n\in\NN$. 

For example we may have a tree (see Example \ref{ex:tree}) that looks like this 
$$
\xymatrix@=10pt{
&&\LTO{Archaea}\ar[rr]&&\LTO{Pyrodicticum}\\
&&&&\LTO{Spirochetes}\\
\LTO{Life}\ar[rr]\ar[ddrr]\ar[uurr]&&\LTO{Bacteria}\ar[rr]\ar[rru]&&\LTO{Aquifex}\\
&&&&\LTO{Fungi}\\
&&\LTO{Eukaryota}\ar[rr]\ar[urr]&&\LTO{Animals}\ar[rrr]&&&\LTO{Homo Sapien}}
$$

We also have a linear order that looks like this:
$$
\xymatrix{\LTO{Life}\ar[r]&\LTO{Domain}\ar[r]&\LTO{Kingdom}\ar[r]&\LTO{Phylum}\ar[r]&\cdots\ar[r]&\LTO{Genus}\ar[r]&\LTO{Species}}
$$
and the ranking system that puts Eukaryota at Domain and Hopo Sapien at Species is an order-preserving function from the dots upstairs to the dots downstairs; that is, it is a morphism of preorders.

\begin{exercise}
Since the phylogenetic tree is a tree, it has all meets.
\sexc Determine the meet of dogs and humans. 
\next If we did not require the phylogenetic partial order to be a tree, what would it mean if two taxa (nodes in the phylogenetic partial order), say $a$ and $b$, had join $c$ with $c\neq a$ and $c\neq b$?
\endsexc
\end{exercise}

\begin{exercise}~
\sexc In your favorite scientific realm, are there any interesting classification systems that are actually orders? 
\next Choose one; what would meets and joins mean in that setting?
\endsexc
\end{exercise}


\subsubsection{Security}\index{security}

Security, say of sensitive information, is based on two things: a security clearance and ``need to know." The former, security clearance might have levels like ``confidential", ``secret", ``top secret". But maybe we can throw in ``president" and some others too, like ``plebe". 

\begin{exercise}
Does it appear that security clearance is a preorder, a partial order, or a linear order?
\end{exercise}

Need-to-know is another classification of people. For each bit of information, we do not necessarily want everyone to know about it, even everyone of the specified clearance. It is only disseminated to those that need to know. 

\begin{exercise}
Let $P$ be the set of all people and let $\ol{I}$ be the set of all pieces of information known by the government. For each subset $I\ss\ol{I}$, let $K(I)\ss P$ be the set of people that need to know every piece of information in $I$. Let $S=\{K(I)\|I\ss\ol{I}\}$ be the set of all ``need-to-know groups", with the subset relation denoted $\leq$. 

\sexc Is $(S,\leq)$ a preorder? If not, find a nearby preorder. 
\next If $I_1\ss I_2$ do we always have $K(I_1)\ss K(I_2)$ or $K(I_2)\ss K(I_1)$ or possibly neither? 
\next Should the preorder $(S,\leq)$ have all meets? 
\next Should $(S,\leq)$ have all joins?
\endsexc
\end{exercise}


\subsubsection{Spaces, e.g. geography}\index{space}

Consider closed curves that can be drawn in the plane $\RR^2$, e.g. circles, ellipses, and kidney-bean shaped curves. The interiors of these closed curves (not including the boundary itself) are called {\em basic open sets in $\RR^2$}. The good thing about such an interior $U$ is that any point $p\in U$ is not on the boundary, so no matter how close $p$ is to the boundary of $U$, there will always be a tiny basic open set surrounding $p$ and completely contained in $U$. In fact, the union of any collection of basic open sets still has this property. An {\em open set in $\RR^2$} is any subset $U\ss \RR^2$ that can be formed as the union of a collection of basic open sets.

\begin{example}

Let $U=\{(x,y)\in\RR^2\|x>0\}$. To see that $U$ is open, define the following sets: for any $a,b\in\RR$, let $S(a,b)$ be the square parallel to the axes, with side length 1, where the upper left corner is $(a,b)$. Let $S'(a,b)$ be the interior of $S(a,b)$. Then each $S'(a,b)$ is open, and $U$ is the union of $S'(a,b)$ over the collection of all $a>0$ and all $b$,$$U=\bigcup_{\parbox{.35in}{\tiny$a,b\in\RR,\\~\;\;a>0$}}S'(a,b).$$ 

\end{example}\index{geography}

The idea of open sets extends to spaces beyond $\RR^2$. For example, on the earth one could define a basic open set to be the interior of any region one can ``draw a circle around" (with a metaphorical pen), and define open sets to be unions of basic open sets. 

\begin{exercise}
Let $S$ be the set of open subsets on earth, as defined in the above paragraph. 
\sexc If $\leq$ is the subset relation, is $(S,\leq)$ a preorder or a partial order? 
\next Does it have meets, does it have joins?
\endsexc
\end{exercise}

\begin{exercise}\label{exc:cosheaf of temps}

Let $S$ be the set of open subsets of earth as defined above. To each open subset of earth suppose we know the range of recorded temperature throughout $s$ (i.e. the low and high throughout the region). Thus to each element $s\in S$ we assign an interval $T(s):=\{x\in\RR\|a\leq x\leq b\}$. If we order the set $V$ of intervals of $\RR$ by the subset relation, it gives a partial order on $V$. 
\sexc Does our assignment $T\taking S\to V$ amount to a morphism of orders? 
\next Does it preserve meets or joins? (Hint: it doesn't preserve both.)
\endsexc
\end{exercise}

\begin{exercise}~
\sexc Can you think of a space relevant to your favorite area of science for which it makes sense to assign an interval of real numbers to each open set somehow, analogously to Exercise \ref{exc:cosheaf of temps}? For example for a sample of some material under stress, perhaps the strain on each open set is somehow an interval? 
\next Repeat the questions from Exercise \ref{exc:cosheaf of temps}.
\endsexc
\end{exercise}


\section{Databases: schemas and instances}\label{sec:databases}

The first three sections of this chapter were about classical objects from mathematics. The present section is about databases, which are classical objects from computer science. These are truly ``categories and functors, without admitting it" (see Theorem \ref{thm:equivalence of categories and schemas}).


\subsection{What are databases?}\label{sec:what are dbs}\index{database!tables}

Data, in particular the set of observations made during experiment, plays
\footnote{The word data is generally considered to be the plural form of the word datum. However, individual datum elements are only useful when they are organized into structures (e.g. if one were to shuffle the cells in a spreadsheet, most would consider the data to be destroyed). It is the whole organized structure that really houses the information; the data must be in formation in order to be useful. Thus I will use the word {\em data} as a collective noun (akin to the word ``sand"); it bridges the divide between the {\em individual datum elements} (akin to the grains of sand) and the {\em data set} (akin to a sand pile). In particular, I will often use the word data as a singular noun.\index{data}}
a primary role in science of any kind. To be useful data must be organized, often in a row-and-column display called a table. Columns existing in different tables can refer to the same data.

A database is a collection of tables, each table $T$ of which consists of a set of columns and a set of rows. We roughly explain the role of tables, columns, and rows as follows. The existence of table $T$ suggests the existence of a fixed methodology for observing objects or events of a certain type. Each column $c$ in $T$ prescribes a single kind or method of observation, so that the datum inhabiting any cell in column $c$ refers to an observation of that kind. Each row $r$ in $T$ has a fixed sourcing event or object, which can be observed using the methods prescribed by the columns. The cell $(r,c)$ refers to the observation of kind $c$ made on event $r$. All of the rows in $T$ should refer to uniquely identifiable objects or events of a single type, and the name of the table $T$ should refer to that type.

\begin{example}\label{ex:graphene}

When graphene is strained (lengthened by a factor of $x\geq 1$), it becomes stressed (carries a force in the direction of the lengthening). The following is a made-up set of data.

\begin{align}\label{ex:first tables}\footnotesize
\begin{tabular}{| l || l | l | l |}\bhline
\multicolumn{4}{|c|}{Graphene sample}\\\bhline
{\bf ID}&{\bf Source}&{\bf Stress}&{\bf Strain}\\\bbhline
A118-1&C Smkt&0&0\\\hline
A118-2&C Smkt&0.02&20\\\hline
A118-3&C Smkt&0.05&40\\\hline
A118-4&AC&0.04&37\\\hline
A118-5&AC&0.1&80\\\hline
A118-6&C Plat&0.1&82\\\bhline
\end{tabular}\hsp
\begin{tabular}{| l || l | l |}\bhline
\multicolumn{3}{|c|}{Supplier}\\\bhline
{\bf ID}&{\bf Full name}&{\bf Phone}\\\bbhline
C Smkt&Carbon Supermarket&(541)781-6611\\\hline
AC&Advanced Chemical&(410) 693-0818\\\hline
C Plat&Carbon Platform&(510) 719-2857\\\hline
McD&McDonard's Burgers&(617) 244-4400\\\hline
APP&Acme Pen and Paper&(617) 823-5603\\\bhline
\end{tabular}
\end{align}

In the first table, titled ``Graphene sample", the rows refer to graphene samples, and the table is so named. Each graphene sample can be observed according to the source supplier from which it came, the strain that it was subjected to, and the stress that it carried. These observations are the columns.  In the second table, the rows refer to suppliers of various things, and the table is so named. Each supplier can be observed according to its full name and its phone number; these are the columns.

In the left-hand table it appears either that each graphene sample was used only once, or that the person recording the data did not keep track of which samples were reused. If such details become important later, the lab may want to change the layout of the first table by adding on the appropriate column. This can be accomplished using morphisms of schemas, which will be discussed in Section \ref{sec:sch as category}.

\end{example}


\subsubsection{Primary keys, foreign keys, and data columns}\index{database!primary key}\index{database!foreign key}

There is a bit more structure in the above tables (Example \ref{ex:first tables}) then may first meet the eye. Each table has a {\em primary ID column}, found on the left, as well as some {\em data columns} and some {\em foreign key columns}. The primary key column is tasked with uniquely identifying different rows. Each data column houses elementary data of a certain sort. Perhaps most interesting from a structural point of view are the foreign key columns, because they link one table to another, creating a connection pattern between tables. Each foreign key column houses data that needs to be further unpacked. It thus refers us to another {\em foreign} table, in particular the primary ID column of that table. In Example \ref{ex:first tables} the {\tt Source} column was a foreign key to the {\tt Supplier} table.

Here is another example, lifted from \cite{Sp2}.

\begin{example}\label{ex:department store 1}

Consider the bookkeeping necessary to run a department store. We keep track of a set of employees and a set of departments. For each employee $e$, we keep track of
\begin{enumerate}[\hsp E.1\;]
\item the {\bf first} name of $e$, which is a {\tt FirstNameString},
\item the {\bf last} name of $e$, which is a {\tt LastNameString},
\item the {\bf manager} of $e$, which is an {\tt Employee}, and
\item the department that $e$ {\bf works in}, which is a {\tt Department}.
\end{enumerate}
For each department $d$, we keep track of
\begin{enumerate}[\hsp D.1\;]
\item the {\bf name} of $d$, which is a {\tt DepartmentNameString}, and
\item the {\bf secretary} of $d$, which is an {\tt Employee}.
\end{enumerate}

Above we can suppose that E.1, E.2, and D.1 are data columns (referring to names of various sorts), and E.3, E.4, and D.2 are foreign key columns (referring to managers, secretaries, etc.). 

Display (\ref{dia:instance on maincat}) shows how such a database might look at a particular moment in time. 
\begin{align}\label{dia:instance on maincat}\index{a schema!department store}
&\footnotesize
\begin{tabular}{| l || l | l | l | l |}\bhline
\multicolumn{5}{| c |}{{\tt Employee}}\\\bhline 
{\bf ID}&{\bf first}&{\bf last}&{\bf manager}&{\bf worksIn}\\\bbhline 101&David&Hilbert&103&q10\\\hline 102&Bertrand&Russell&102&x02\\\hline 103&Emmy&Noether&103&q10\\\bhline
\end{tabular}&\hsp\footnotesize
\begin{tabular}{| l || l | l |}\bhline
\multicolumn{3}{| c |}{{\tt Department}}\\
\bhline {\bf ID}&{\bf name}&{\bf secretary}\\\bbhline q10&Sales&101\\\hline x02&Production&102\\\bhline
\end{tabular}
\end{align}\vspace{.1in}

\end{example}


\subsubsection{Business rules}\index{database!business rules}

Looking at the tables from Example \ref{ex:department store 1}, one may notice a few patterns. First, every employee works in the same department as his or manager. Second, every department's secretary works in that department. Perhaps the business counts on these rules for the way it structures itself. In that case the database should enforce those rules, i.e. it should check that whenever the data is updated, it conforms to the rules: 

\begin{enumerate}[\hsp Rule 1\;]
\item For every employee $e$, the {\bf manager} of $e$ {\bf works in} the same department that $e$ {\bf works in}.
\item For every department $d$, the {\bf secretary} of $d$ {\bf works in} department $d$.
\end{enumerate}
\vspace{-.3in}\begin{align}\label{dia:rules}\end{align}\vspace{-.3in}

Together, the statements E.1, E.2, E.3, E.4, D.1, and D.2 from Example \ref{ex:department store 1} and Rule 1 and Rule 2, constitute what we will call the {\em schema} of the database.\index{schema!of a database}\index{database!schema} We will formalize this idea in Section \ref{sec:schemas}.


\subsubsection{Data columns as foreign keys}

To make everything consistent, we could even say that data columns are specific kinds of foreign keys. That is, each data column constitutes a foreign key to some non-branching {\em leaf table}\index{schema!leaf table}, which has no additional data. 

\begin{example}\label{ex:department store 2}

Consider again Example \ref{ex:department store 1}. Note that first names and last names had a particular type, which we all but ignored above. We could cease to ignore them by adding three tables, as follows.

\begin{align}\label{dia:instance on maincat 2}\footnotesize
\begin{tabular}{| l ||}\bhline
\multicolumn{1}{| c |}{{\tt FirstNameString}}\\\bhline
{\bf ID}\\\bbhline Alan\\\hline Alice\\\hline Bertrand\\\hline Carl\\\hline David\\\hline Emmy\\\hline\hspace{.25in}\vdots\\\bhline
\end{tabular}\hspace{.6in}\footnotesize
\begin{tabular}{| l ||}\bhline
\multicolumn{1}{| c |}{{\tt LastNameString}}\\\bhline
{\bf ID}\\\bbhline Arden\\\hline Hilbert\\\hline Jones\\\hline Noether\\\hline Russell\\\hline\hspace{.25in}\vdots\\\bhline
\end{tabular}\hspace{.6in}\footnotesize
\begin{tabular}{| l ||}\bhline
\multicolumn{1}{| c |}{{\tt DepartmentNameString}}\\\bhline
{\bf ID}\\\bbhline Marketing\\\hline Production\\\hline Sales\\\hline\hspace{.25in}\vdots\\\bhline
\end{tabular}
\end{align}

In combination, Displays (\ref{dia:instance on maincat}) and (\ref{dia:instance on maincat 2}) form a collection of tables with the property that every column is either a primary key or a foreign key. The notion of data column is now subsumed under the notion of foreign key column. Everything is either a primary key (one per table, labeled ID) or a foreign key column (everything else).

\end{example}


\subsection{Schemas}\label{sec:schemas}

The above section may all seem intuitive or reasonable in some ways, but also a bit difficult to fully grasp, perhaps. It would be nice to summarize what is happening in a picture. Such a picture, which will basically be a graph, should capture the {\em conceptual layout} to which the data conforms, without yet being concerned with the individual data that may populate the tables in this instant. We proceed at first by example, giving the precise definition in Definition \ref{def:schema}.

\begin{example}\label{ex:department store 3}

In Examples \ref{ex:department store 1} and \ref{ex:department store 2}, the conceptual layout for a department store was given, and some example tables were shown. We were instructed to keep track of employees, departments, and six types of data (E.1, E.2, E.3, E.4, D.1, and D.2), and we were instructed to follow two rules (Rule 1, Rule 2). All of this is summarized in the following picture:
\begin{align}\label{dia:maincat schema}
\MainCatLarge{\mcC:=\tn{ Schema for tables (\ref{dia:instance on maincat}) and (\ref{dia:instance on maincat 2}) conforming to (\ref{dia:rules})}}
\end{align}
The five tables from (\ref{dia:instance on maincat}) and (\ref{dia:instance on maincat 2}) are seen as five vertices; this is also the number of primary ID columns. The six foreign key columns from (\ref{dia:instance on maincat}) and (\ref{dia:instance on maincat 2}) are seen as six arrows; each points from a table to a foreign table. The two rules from (\ref{dia:rules}) are seen as statements at the top of Display (\ref{dia:maincat schema}).We will explain path equivalences in Definition \ref{def:congruence}.

\end{example}

\begin{exercise}\label{exc:schema for first tables}
Come up with a schema (consisting of dots and arrows) describing the conceptual layout of information presented in Example \ref{ex:graphene}. 
\end{exercise}

In order to define schemas, we must first define the notion of schematic equivalence relation, which is to hold on the set of paths of a graph $G$ (see Section \ref{sec:paths in graph}). Such an equivalence relation (in addition to being reflexive, symmetric, and transitive) has two sorts of additional properties: equivalent paths must have the same source and target, and the composition of equivalent paths with other equivalent paths must yield equivalent paths. Formally we have Definition \ref{def:congruence}.

\begin{definition}\label{def:congruence}\index{PED}\index{congruence}\index{schema!Path equivalence declaration (PED)}\index{schema!congruence}\

Let $G=(V,A,src,tgt)$ be a graph, and let $\Path_G$ denote the set of paths in $G$ (see Definition \ref{def:paths in graph}). A {\em path equivalence declaration} (or {\em PED}) is an expression of the form $p\simeq q$ where $p,q\in\Path_G$ have the same source and target, $src(p)=src(q)$ and $tgt(p)=tgt(q)$. 

A {\em congruence} on $G$ is a relation $\simeq$ on $\Path_G$ that has the following properties: 
\begin{enumerate}
\item The relation $\simeq$ is an equivalence relation.
\item If $p\simeq q$ then $src(p)=src(q)$.
\item If $p\simeq q$ then $tgt(p)=tgt(q)$.
\item Suppose $p,q\taking b\to c$ are paths, and $m\taking a\to b$ is an arrow. If $p\simeq q$ then $mp\simeq mq$. 
\item Suppose $p,q\taking a\to b$ are paths, and $n\taking b\to c$ is an arrow. If $p\simeq q$ then $pn\simeq qn$.
\end{enumerate}

\end{definition}

Any set of path equivalence declarations (PEDs) generates a congruence. We tend to elide the difference between a congruence and the set of PEDs that generates it.

\begin{exercise}\label{exc:generating congruence}
Consider the graph shown in (\ref{dia:maincat schema}), and the two declarations shown at the top. They generate a congruence. 
\sexc Is it true that the following PED is an element of this congruence? $$\tn{{\tt Employee} manager manager worksIn $\stackrel{?}{\simeq}$ {\tt Employee} worksIn}$$ \next What about this one? $$\tn{{\tt Employee} worksIn secretary $\stackrel{?}{\simeq}$ {\tt Employee}}$$ 
\next What about this one? $$\tn{{\tt Department} secretary manager worksIn name $\stackrel{?}{\simeq}$ {\tt Department} name}$$
\endsexc
\end{exercise}

\begin{lemma}\label{lemma:composing PEDs}

Suppose that $G$ is a graph and $\simeq$ is a congruence on $G$. Suppose $p\simeq q\taking a\to b$ and $r\simeq s\taking b\to c$. Then $pr\simeq qs$.

\end{lemma}

\begin{proof}

The picture to have in mind is this: $$\xymatrix@=13pt{&\bullet\ar[r]&\cdots\ar[r]&\bullet\ar[dr]&&\bullet\ar[r]&\cdots\ar[r]&\bullet\ar[dr]\\\LMO{a}\ar@{}[rrrr]|{\simeq}\ar[ur]\ar[dr]\ar@{-->}@/^1.5pc/[rrrr]_p\ar@{-->}@/_1.5pc/[rrrr]^q&&&&\LMO{b}\ar@{}[rrrr]|{\simeq}\ar[ur]\ar[dr]\ar@{-->}@/^1.5pc/[rrrr]_r\ar@{-->}@/_1.5pc/[rrrr]^s&&&&\LMO{c}\\&\bullet\ar[r]&\cdots\ar[r]&\bullet\ar[ur]&&\bullet\ar[r]&\cdots\ar[r]&\bullet\ar[ur]}$$ Applying condition (3) from Definition \ref{def:congruence} to each arrow in path $p$, it follows by induction that $pr\simeq ps$. Applying condition (4) to each arrow in path $s$, it follows similarly that $ps\simeq qs$. Because $\simeq$ is an equivalence relation, it follows that $pr\simeq qs$. 

\end{proof}

\begin{definition}\label{def:schema}\index{schema}\index{database!schema}

A {\em database schema} (or simply {\em schema}) $\mcC$ consists of a pair $\mcC:=(G,\simeq)$ where $G$ is a graph and $\simeq$ is a congruence on $G$. 

\end{definition}

\begin{example}

The picture drawn in (\ref{dia:maincat schema}) has the makings of a schema. Pictured is a graph with two PEDs; these generate a congruence, as discussed in Exercise \ref{exc:generating congruence}.  

\end{example}

A schema can be converted into a system of tables each with a primary key and some number of foreign keys referring to other tables, as discussed in Section \ref{sec:what are dbs}. Definition \ref{def:schema} gives a precise conceptual understanding of what a schema is, and the following rules describe how to convert such a thing into a table layout.

\begin{rules}\label{rules:schema to tables}\index{olog!rules}

Converting a schema $\mcC=(G,\simeq)$ into a table layout should be done as follows:
\begin{enumerate}[(i)]
\item There should be a table for every vertex in $G$ and if the vertex is named, the table should have that name;
\item Each table should have a left-most column called ID, set apart from the other columns by a double vertical line; and
\item To each arrow $a$ in $G$ having source vertex $s:=src(a)$ and target vertex $t:=tgt(a)$, there should be a foreign key column $a$ in table $s$, referring to table $t$; if the arrow $a$ is named, column $a$ should have that name.
\end{enumerate}

\end{rules}

\begin{example}[Discrete dynamical system]\label{ex:dds}\index{dynamical system!discrete}

Consider the schema 
\begin{align}\label{dia:loop}
\Loop:=\LoopSchema
\end{align}
in which the congruence is trivial (i.e. generated by the empty set of PEDs.) This schema is quite interesting. It encodes a set $s$ and a function $f\taking s\to s$. Such a thing is called a {\em discrete dynamical system}. One imagines $s$ as the set of states and, for any state $x\in s$, a notion of ``next state" $f(x)\in s$. For example
\begin{align}\label{dia:dds data}
\begin{tabular}{| l || c |}\bhline
\multicolumn{2}{| c |}{s}\\\bhline 
{\bf ID}&{\bf f}\\\bbhline
A & B\\\hline
B & C\\\hline
C & C\\\hline
D & B\\\hline
E & C\\\hline
F & G\\\hline
G & H\\\hline
H & G\\\hline
\end{tabular}
\hspace{.5in}\text{...pictured...}\hspace{.5in}
\parbox{1.4in}{\fbox{\xymatrix@=7pt{
\LMO{A}\ar[rr]&&\LMO{B}\ar[rr]&&\LMO{C}\ar@(u,r)[]^~\\
\LMO{D\ar[urr]}&&\LMO{E}\ar[urr]\\
\LMO{F}\ar[rr]&&\LMO{G}\ar@/^.5pc/[rr]&&\LMO{H}\ar@/^.5pc/[ll]^~
}}}
\end{align}

\end{example}

\begin{application}

Imagine a \href{http://en.wikipedia.org/wiki/Chronon}{\text quantum-time} universe in which there are discrete time steps. We model it as a discrete dynamical system, i.e. a table of the form (\ref{dia:dds data}). For every possible state of the universe we include a row in the table. The state in the next instant is recorded in the second column.

\end{application}

\begin{example}[Finite hierarchy]\label{ex:finite hierarchy}\index{hierarchy}

The schema $\Loop$ can also be used to encode hierarchies, such as the manager relation from Examples \ref{ex:department store 1} and \ref{ex:department store 3}, 
$$\fbox{\xymatrix{\LTO{E}\ar@(l,u)[]^{\tn{mgr}}}}$$
One problem with this, however, is if a schema has even one loop, then it can have infinitely many paths (corresponding, e.g. to an employees manager's manager's manager's ... manager). 

Sometimes we know that in a given company that process eventually ends, a famous example being that at Ben and Jerry's ice cream, there were only seven levels. In that case we know that an employee's 8th level manager is equal to his or her 7th level manager. This can be encoded by the PED $${\tt E} \tn{ mgr}\tn{ mgr}\tn{ mgr}\tn{ mgr}\tn{ mgr}\tn{ mgr}\tn{ mgr}\tn{ mgr}\simeq{\tt E}\tn{ mgr}\tn{ mgr}\tn{ mgr}\tn{ mgr}\tn{ mgr}\tn{ mgr}\tn{ mgr}$$ or more concisely, $\tn{mgr}^8=\tn{mgr}^7$.

\end{example}

\begin{exercise}
Is there any nontrivial PED on $\Loop$ that holds for the data in Example \ref{ex:dds}? If so, what is it and how many equivalence classes of paths in $\Loop$ are there after you impose that relation?
\end{exercise}

\begin{exercise}
Let $P$ be a chess-playing program. Given any position (including the history of the game and choice of whose turn it is), $P$ will make a move. 
\sexc Is this an example of a discrete dynamical system? 
\next How do the rules for ending the game in a win or draw play out in this model? (Look up online how chess games end if you don't know.)
\endsexc
\end{exercise}


\subsubsection{Ologging schemas}\label{sec:olog as db schema}\index{olog!as database schema}

It should be clear that a database schema is nothing but an olog in disguise. The difference is basically the readability requirements for ologs. There is an important new addition in this section, namely that we can fill out an olog with data. Conversely, we have seen that databases are not any harder to understand than ologs are.

\begin{example}\label{ex:orbits}

Consider the olog 
\begin{align}\label{dia:moon1}\obox{}{.5in}{a moon}\Too{\tn{orbits}}\obox{}{.5in}{a planet}\end{align}
We can document some instances of this relationship using the following tables: 
\begin{align}\label{dia:moon2}
\begin{tabular}{| c || c |}\bhline
\multicolumn{2}{| c |}{\bf orbits}\\\bhline
{\bf a moon}&{\bf a planet}\\\bbhline
The Moon&Earth\\\hline 
Phobos&Mars\\\hline 
Deimos&Mars\\\hline 
Ganymede & Jupiter\\\hline
Titan & Saturn\\\bhline
\end{tabular}
\end{align}  

Clearly, this table of instances can be updated as more moons are discovered by the author (be it by telescope, conversation, or research).

\end{example}

\begin{exercise}
In fact, Example \ref{ex:orbits} did not follow Rules \ref{rules:schema to tables}. Strictly following those rules, copy over the data from (\ref{dia:moon2}) into tables that are in accordance with schema (\ref{dia:moon1}).
\end{exercise}

\begin{exercise}~
\sexc Write down a schema, in terms of the boxes \fakebox{a thing I own} and \fakebox{a place} and one additional arrow, that might help one remember where they decided to put ``random" things. 
\next What is a good label for the arrow? 
\next Fill in some rows of the corresponding set of tables for your own case.
\endsexc
\end{exercise}

\begin{exercise}\label{exc:father and child}
Consider the olog 
$$
\xymatrix{\obox{C}{.4in}{a child}\LA{rr}{has}&&\obox{F}{.5in}{a father}\ar@/^1pc/[ll]^{\tn{has as first}}\ar@/_1pc/[ll]_{\tn{has as tallest}}}
$$
\sexc What path equivalence declarations would be appropriate for this olog? You can use $f\taking F\to C$, $t\taking F\to C$, and $h\taking C\to F$ if you prefer. 
\next How many PEDs are in the congruence?
\endsexc
\end{exercise}


\subsection{Instances}

Given a database schema $(G,\simeq)$, an instance of it is just a bunch of tables whose data conform to the specified layout. These can be seen throughout the previous section, most explicitly in the relationship between schema (\ref{dia:maincat schema}) and tables (\ref{dia:instance on maincat}) and (\ref{dia:instance on maincat 2}), and between schema (\ref{dia:loop}) and table (\ref{dia:dds data}). Below is the mathematical definition.

\begin{definition}\label{def:instance}\index{database!instance}\index{instance}

Let $\mcC=(G,\simeq)$ where $G=(V,A,src,tgt)$. An {\em instance on $\mcC$}, denoted $(\PK,\FK)\taking\mcC\to\Set$, is defined as follows: One announces some constituents (A. primary ID part, B. foreign key part) and asserts that they conform to a law (1. preservation of congruence). Specifically, one announces
\begin{enumerate}[\hsp A.]
\item a function $\PK\taking V\to \Set$; i.e. to each vertex $v\in V$ one provides a set $\PK(v)$;\footnote{The elements of $\PK(v)$ will be listed as the rows of table $v$, or more precisely as the leftmost cells of these rows.} and
\item for every arrow $a\in A$ with $v=src(a)$ and $w=tgt(a)$, a function $\FK(a)\taking\PK(v)\to\PK(w)$.
\footnote{The arrow $a$ will correspond to a column, and to each row $r\in\PK(v)$ the $(r,a)$ cell will contain the datum $\FK(a)(r)$.}
\end{enumerate}
One asserts that the following law holds for any vertices $v, w$ and paths $p=va_1a_2\ldots a_m$ and $q=va_1'a_2'\ldots a_n'$ from $v$ to $w$:
\begin{enumerate}[\hsp 1.]
\item If $p\simeq q$ then for all $x\in\PK(v)$, we have $$\FK(a_m)\circ\cdots\circ\FK(a_2)\circ\FK(a_1)(x)=\FK(a_n')\circ\cdots\circ\FK(a_2')\circ\FK(a_1')(x)$$ in $\PK(w).$
\end{enumerate}

\end{definition}

\begin{exercise}\label{ex:self email}
Consider the olog pictured below: 
$$\mcC:=
\fbox{\parbox{3.1in}{
\xymatrix{
\obox{}{.7in}{a self-email}\LA{r}{is}&\obox{}{.5in}{an email}\ar@/^1pc/[r]^{\tn{is sent by}}\ar@/_1pc/[r]_{\tn{is sent to}}&\obox{}{.5in}{a person}}~\\\\
\parbox{3.1in}{Given $x$, a self-email, consider the following. \\
We know that $x$ is a self-email, which is an email, which is sent by a person that we'll call $P(x)$.\\
We also know that $x$ is a self-email, which is an email, which is sent to a person that we'll call $Q(x)$.\\
Fact: whenever $x$ is a self-email, we will have $P(x)=Q(x)$}
}}
$$
\begin{align}\label{dia:self email}
\begin{tabular}{| l || l |}\bhline
\multicolumn{2}{| c |}{{\tt a self-email}}\\\bhline
{\bf ID}&{\bf is}\\\bbhline 
SEm1207&Em1207\\\hline 
SEm1210&Em1210\\\hline 
SEm1211&Em1211\\\bhline
\end{tabular}&\hsp
\begin{tabular}{| l || l | l |}\bhline
\multicolumn{3}{| c |}{{\tt an email}}\\\bhline 
{\bf ID}&{\bf is sent by}&{\bf is sent to}\\\bbhline 
Em1206&Bob&Sue\\\hline 
Em1207 &Carl&Carl\\\hline 
Em1208&Sue & Martha\\\hline 
Em1209&Chris&Bob\\\hline 
Em1210&Chris&Chris\\\hline 
Em1211&Julia&Julia\\\hline 
Em1212&Martha&Chris\\\bhline
\end{tabular}\hsp
\begin{tabular}{| l ||}\bhline
\multicolumn{1}{| c |}{{\tt a person}}\\\bhline 
{\bf ID}\\\bbhline 
Bob\\\hline 
Carl\\\hline 
Chris\\\hline 
Julia\\\hline 
Martha\\\hline 
Sue\\\bhline
\end{tabular}
\end{align}\normalsize 

\sexc What is the set $\PK(\fakebox{an email})$? 
\next What is the set $\PK(\fakebox{a person})$? 
\next What is the function $\FK(\tn{is sent by})\taking\PK(\fakebox{an email})\to\PK(\fakebox{a person})$?
\next Interpret the sentences at the bottom of $\mcC$ as the Englishification of a simple path equivalence declaration. Is it satisfied by the instance (\ref{dia:self email}); that is, does law 1. from Definition \ref{def:instance} hold?\index{Englishifiication}
\endsexc
\end{exercise}

\begin{example}[Monoid action table]\label{ex:monoid action table}

In Example \ref{ex:monoid as olog}, we saw how a monoid $\mcM$ could be captured as an olog with only one object. As a database schema, this means there is only one table. Every generator of $\mcM$ would be a column of the table. The notion of database instance for such a schema is precisely the notion of action table from Section \ref{sec:monoid action table}. Note that a monoid can act on itself, in which case this action table is the monoid's multiplication table as in Example \ref{ex:multiplication table}, but it can also act on any other set as in Example \ref{ex:action table}. If $\mcM$ acts on a set $S$, then the set of rows in the action table will be $S$.

\end{example}

\begin{exercise}
Draw (as a graph) the schema for which Table \ref{dia:action table for FSM} is an instance.
\end{exercise}

\begin{exercise}
Suppose that $\mcM$ is a monoid and some instance of it is written out in table form. It's possible that $\mcM$ is a group. What evidence in an instance table for $\mcM$ might suggest that $\mcM$ is a group? 
\end{exercise}


\subsubsection{Paths through a database}

Let $\mcC:=(G,\simeq)$ be a schema and let $(\PK,\FK)\taking\mcC\to\Set$ be an instance on $\mcC$. Then for every arrow $a\taking v\to w$ in $G$ we get a function $\FK(a)\taking\PK(v)\to\PK(w)$. Functions can be composed, so in fact for every path through $G$ we get a function. Namely, if $p=v_0a_1,a_2,\ldots,a_n$ is a path from $v_0$ to $v_n$ then the instance provides a function $$\FK(p):=\FK(a_n)\circ\cdots\FK(a_2)\circ\FK(a_1)\taking\PK(v_0)\to\PK(v_n),$$ which first made an appearance as part of Law 1 in Definition \ref{def:instance}.

\begin{example}\label{ex:paths as functions}

Consider the department store schema from Example \ref{ex:department store 3}, and in (\ref{dia:maincat schema}) the path $[\tn{worksIn, secretary, last}]$ which points from {\tt Employee} to {\tt LastNameString}. The instance will let us interpret this path as a function from the set of employees to the set of last names; this could be a useful function to have around. The instance from (\ref{dia:instance on maincat}) would yield the following function 

\begin{align*}
\begin{tabular}{| l || l |}\bhline
\multicolumn{2}{| c |}{{\tt Employee}}\\\bhline 
{\bf ID}&{\bf Secr. name}\\\bbhline 
101&Hilbert\\\hline 
102&Russell\\\hline 
103&Hilbert\\\bhline
\end{tabular}
\end{align*}

\end{example}

\begin{exercise}
Consider the path $p:=[f,f]$ on the $\Loop$ schema from (\ref{dia:loop}). Using the instance from (\ref{dia:dds data}), where $\PK(s)=\{A,B,C,D,E,F,G,H\}$, interpret $p$ as a function $\PK(s)\to\PK(s)$, and write this as a 2-column table, as above in Example \ref{ex:paths as functions}.
\end{exercise}

\begin{exercise}~
\sexc Given an instance $(\PK,\FK)$ on a schema $\mcC$, and given a trivial path $p$ (i.e. $p$ has length 0; it starts at some vertex but doesn't go anywhere), what function does $p$ yield?
\next What are the domain and codomain of $p$?
\endsexc
\end{exercise}


\chapter{Basic category theory}\label{chap:categories}

``{\it ...We know only a very few---and, therefore, very precious---schemes whose unifying powers cross many realms.}" -- Marvin Minsky.\footnote{\cite[Problems of disunity, p. 126]{Min}.}\\\\

Categories, or an equivalent notion, have already been secretly introduced as ologs. One can think of a category as a graph (as in Section \ref{sec:graphs}) in which certain paths have been declared equivalent. (Ologs demand an extra requirement that everything in sight be readable in natural language, and this cannot be part of the mathematical definition of category.) The formal definition of category is given in Definition \ref{def:category}, but it will not be obviously the same as the ``graph+path equivalences" notion; the latter was given in Definition \ref{def:schema} as the definition of a {\em schema}. Once we talk about how different categories can be compared using functors (Definition \ref{def:functor}), and how different schemas can be compared using schema mappings (Definition \ref{def:schema morphism}), we will prove that the two notions are equivalent (Theorem \ref{thm:equivalence of categories and schemas}).


\section{Categories and Functors}

In this section we give the standard definition of categories and functors. These, together with natural transformations (Section \ref{sec:nat trans}), form the backbone of category theory. We also give some examples.


\subsection{Categories}\label{sec:categories}

In everyday speech we think of a category as a kind of thing. A category consists of a collection of things, all of which are related in some way. In mathematics, a category can also be construed as a collection of things and a type of relationship between pairs of such things. For this kind of thing-relationship duo to count as a category, we need to check two rules, which have the following flavor: every thing must be related to itself by simply being itself, and if one thing is related to another and the second is related to a third, then the first is related to the third. In a category, the ``things" are called {\em objects} and the ``relationships" are called {\em morphisms}.

In various places throughout this book so far we have discussed things of various sorts, e.g. sets, monoids, graphs. In each case we discussed how such things should be appropriately compared\index{appropriate comparison}. In each case the ``things" will stand as the objects and the ``appropriate comparisons" will stand as the morphisms in the category. Here is the definition.

\begin{definition}\label{def:category}\index{category}\index{hom-set}\index{morphism}

A {\em category} $\mcC$ is defined as follows: One announces some constituents (A. objects, B. morphisms, C. identities, D. compositions) and asserts that they conform to some laws (1. identity law, 2. associativity law). Specifically, one announces:
\begin{enumerate}[\hsp A.]
\item a collection $\Ob(\mcC)$, elements of which are called {\em objects};\index{a symbol!$\Ob$}
\item for every pair $x,y\in\Ob(\mcC)$, a set $\Hom_\mcC(x,y)\in\Set$.\index{a symbol!$\Hom_\mcC$} It is called the {\em hom-set from $x$ to $y$}; its elements are called {\em morphisms from $x$ to $y$};
\footnote{The reason for the notation $\Hom$ and the word {\em hom-set} is that morphisms are often called {\em homomorphisms}, e.g. in group theory.}
\item for every object $x\in\Ob(\mcC)$, a specified morphism denoted $\id_x\in\Hom_\mcC(x,x)$ called {\em the identity morphism on $x$}; and
\item for every three objects $x,y,z\in\Ob(\mcC)$, a function $$\circ\taking\Hom_\mcC(y,z)\times\Hom_\mcC(x,y)\to\Hom_\mcC(x,z),$$ called {\em the composition formula}.\index{a symbol!$\circ$}\index{composition!of morphisms}
\end{enumerate}
Given objects $x,y\in\Ob(\mcC)$, we can denote a morphism $f\in\Hom_\mcC(x,y)$ by $f\taking x\to y$; we say that $x$ is the {\em domain} of $f$ and that $y$ is the {\em codomain} of $f$. Given also $g\taking y\to z$, the composition formula is written using infix notation, so $g\circ f\taking x\to z$ means $\circ(g,f)\in\Hom_\mcC(x,z)$.

One asserts that the following law holds:
\begin{enumerate}[\hsp 1.]
\item for every $x,y\in\Ob(\mcC)$ and every morphism $f\taking x\to y$, we have
$$f\circ\id_x=f\hsp\tn{and}\hsp\id_y\circ f=f;$$ \hsp{and};
\item if $w,x,y,z\in\Ob(\mcC)$ are any objects and $f\taking w\to x,\;\;g\taking x\to y,\;\;$ and $h\taking y\to z$ are any morphisms, then the two ways to compose are the same:$$(h\circ g)\circ f = h\circ(g\circ f) \in\Hom_\mcC(w,z).$$
\end{enumerate}
\end{definition}

\begin{remark}\label{rmk:small}\index{category!small}

There is perhaps much that is unfamiliar about Definition \ref{def:category} but there is also one thing that is strange about it. The objects $\Ob(\mcC)$ of $\mcC$ are said to be a ``collection" rather than a set. This is because we sometimes want to talk about the category of all sets, in which every possible set is an objects, and if we try to say that the collection of sets is itself, we run into \href{http://en.wikipedia.org/wiki/Russell's_paradox}{\text Russell's paradox}. Modeling this was a sticking point in the foundations of category theory, but it was eventually fixed by Grothendieck's notion of expanding universes.\index{Grothendieck!expanding universes} Roughly the idea is to choose some huge set $\kappa$ (with certain properties making it a {\em universe}), to work entirely inside of it when possible, and to call anything in that world {\em $\kappa$-small} (or just {\em small} if $\kappa$ is clear from context). When we need to look at $\kappa$ itself, we  choose an even bigger universe $\kappa'$ and work entirely within it.

A category in which the collection $\Ob(\mcC)$ is a set (or in the above language, a small set) is called a {\em small category}. From here on out we will not take care of the difference, referring to $\Ob(\mcC)$ as a set. We do not think this will do any harm to scientists using category theory, at least not in the beginning phases of their learning.\index{a warning!``set'' of objects in a category}

\end{remark}

\begin{example}[The category $\Set$ of sets]\index{a category!$\Set$}

Chapter \ref{chap:sets} was all about the category of sets, denoted $\Set$. The objects are the sets and the morphisms are the functions; we even used the current notation, referring to the set of functions $X\to Y$ as $\Hom_\Set(X,Y)$. The composition formula $\circ$ is given by function composition, and for every set $X$, the identity function $\id_X\taking X\to X$ serves as the identity morphism for $X\in\Ob(\Set)$. The two laws clearly hold, so $\Set$ is indeed a category. 

\end{example}

\begin{example}[The category $\Fin$ of finite sets]\index{a category!$\Fin$}\label{ex:Fin}

Inside the category $\Set$ is a {\em subcategory} $\Fin\ss\Set$, called the {\em category of finite sets}. Whereas an object $S\in\Ob(\Set)$ is a set that can have arbitrary cardinality, we define $\Fin$ such that its objects include all (and only) the sets $S$ with finitely many elements, i.e. $|S|=n$ for some natural number $n\in\NN$. Every object of $\Fin$ is an object of $\Set$, but not vice versa.

Although $\Fin$ and $\Set$ have a different collection of objects, their morphisms are in some sense ``the same". For any two finite sets $S,S'\in\Ob(\Fin)$, we can also think of $S,S'\in\Ob(\Set)$, and we have
$$\Hom_\Fin(S,S')=\Hom_\Set(S,S').$$
That is a morphism in $\Fin$ between finite sets $S$ and $S'$ is simply a function $f\taking S\to S'$.

\end{example}

\begin{example}[The category $\Mon$ of monoids]\label{ex:mon is cat}\index{a category!$\Mon$}

We defined monoids in Definition \ref{def:monoid} and monoid homomorphisms in Definition \ref{def:monoid hom}. Every monoid $\mcM:=(M,e,\star_M)$ has an identity homomorphism $\id_\mcM\taking\mcM\to\mcM$, given by the identity function $\id_M\taking M\to M$. To compose two monoid homomorphisms $f\taking\mcM\to\mcM'$ and $g\taking\mcM'\to\mcM''$, we compose their underlying functions $f\taking M\to M'$ and $g\taking M'\to M''$, and check that the result $g\circ f$ is a monoid homomorphism. Indeed,
$$g\circ f(e)=g(e')=e''$$
$$g\circ f(m_1\star_Mm_2)=g(f(m_1)\star_{M'}f(m_2))=g\circ f(m_1)\star_{M''}g\circ f(m_2).$$
It is clear that the two laws hold, so $\Mon$ is a category.

\end{example}

\begin{exercise}[The category $\Grp$ of groups]\index{a category!$\Grp$}
Suppose we set out to define a category $\Grp$, having groups as objects and group homomorphisms as morphisms, see Definition \ref{def:group homomorphism}. Show (to the level of detail of Example \ref{ex:mon is cat}) that the rest of the conditions for $\Grp$ to be a category are satisfied.
\end{exercise}

\begin{exercise}[The category $\PrO$ of preorders]\index{a category!$\PrO$}
Suppose we set out to define a category $\PrO$, having preorders as objects and preorder homomorphisms as morphisms (see Definition \ref{def:morphism of orders}). Show (to the level of detail of Example \ref{ex:mon is cat} that the rest of the conditions for $\PrO$ to be a category are satisfied.
\end{exercise}

\begin{example}[Non-category 1]

So what's not a category? Two things can go wrong: either one fails to specify all the relevant constituents (A, B, C, D from Definition \ref{def:category}, or the constituents do not obey the laws (1, 2).\index{category!non-example}

Let $G$ be the following graph,
$$G=\fbox{\xymatrix{\LMO{a}\ar[r]^f&\LMO{b}\ar[r]^g&\LMO{c}}}.$$
Suppose we try to define a category $\mcG$ by faithfully recording vertices as objects and arrows as morphisms. Will that be a category?

Following that scheme, we put $\Ob(\mcG)=\{a,b,c\}$. For all 9 pairs of objects we need a hom-set.
Say 
\begin{align*}
\begin{array}{lll}
\Hom_\mcG(a,a)=\emptyset&\hsp\Hom_\mcG(a,b)=\{f\}&\hsp\Hom_\mcG(a,c)=\emptyset\\
\Hom_\mcG(b,a)=\emptyset&\hsp\Hom_\mcG(b,b)=\emptyset&\hsp\Hom_\mcG(b,c)=\{g\}\\
\Hom_\mcG(c,a)=\emptyset&\hsp\Hom_\mcG(c,b)=\emptyset&\hsp\Hom_\mcG(c,c)=\emptyset
\end{array}
\end{align*}
If we say we are done, the listener should object that we have given neither identities nor a composition formula. In fact, it is impossible to give identities under our scheme, because e.g. $\Hom_\mcG(a,a)=\emptyset$.

Suppose we fix that problem, adding an element to each of our ``diagonals" so that 
$$\Hom_\mcG(a,a)=\{\id_a\},\hsp\Hom_\mcG(b,b)=\{\id_b\},\hsp\tn{and}\hsp\Hom_\mcG(c,c)=\{\id_c\}.$$ What about a composition formula? We need a function $\Hom_\mcG(a,b)\times\Hom_\mcG(b,c)\to\Hom_\mcG(a,c)$, but the domain is nonempty and the codomain is empty; there is no such function. 

Again, we must make a change, adding an element to make $$\Hom_\mcG(a,c)=\{h\}.$$ We would now say $g\circ f=h$. Finally, this does the trick and we have a category. A computer could check this quickly, as can someone with good intuition for categories; for everyone else, it may be a painstaking process involving determining whether there is a unique composition formula for each of the 27 pairs of hom-sets and whether the associative law holds in the 81 necessary cases. Luckily this computation is ``sparse" (lots of $\emptyset$'s), so it's not as bad as it first seems.

Redrawing all the morphisms as arrows, our graph has become:
$$G=\fbox{\xymatrix{\LMO{a}\ar@(ul,dl)[]_{\id_a}\ar[r]^f\ar@/_1pc/[rr]_h&\LMO{b}\ar@(ur,ul)[]_{\id_b}\ar[r]^g&\LMO{c}\ar@(ur,dr)[]^{\id_c}}}$$

\end{example}

\begin{example}[Non-category 2]\index{category!non-example}

In this example, we will make a faux-category $\mcF$ with one object and many morphisms. The problem here will be our composition formula. 

Define $\mcF$ to have one object $\Ob(\mcF)=\singleton$, and $\Hom_\mcF(\smiley,\smiley)=\NN$. Define $\id_{\smiley}=1\in\NN$. Define the composition formula $\circ\taking\NN\times\NN\to\NN$ by $m\circ n=m^n$. This is a perfectly cromulent function, but it does not work right as a composition formula. Indeed, for the identity law to hold, we would need $m^1=m=1^m$, and one side of this is false. For the associativity law to hold, we would need $(m^n)^p=m^{(n^p)}$, but this is also not the case.

To fix this problem we have to completely revamp our composition formula. It would work to use multiplication, $m\circ n=m*n$. Then the identity law would read $1*m=m=m*1$, and that holds; and the associativity law would read $(m*n)*p=m*(n*p)$, and that holds.

\end{example}

\begin{example}[The category of preorders with joins]\label{ex:preorders with joins}

Suppose that we are only interested in preorders $(X,\leq)$ for which every pair of elements has a join. We saw in Exercise \ref{exc:not all meets and joins} that not all preorders have this property. However we can create a category $\mcC$ in which every object does have this property. To begin we put $\Ob(\mcC)=\{(X,\leq)\in\Ob(\PrO)\| (X,\leq)\tn{ has all joins}\}.$ But what about morphisms?

One option would be to put in no morphisms (other than identities), and to just consider this collection of objects as having no structure other than a set.

Another option would be to put in exactly the same morphisms as in $\PrO$: for any objects $a,b\in\Ob(\mcC)$ we consider $a$ and $b$ as regular old preorders, and put $\Hom_\mcC(a,b):=\Hom_{\PrO}(a,b)$. The resulting category of preorders with joins is called the {\em full subcategory of $\PrO$ spanned by the preorders with joins}.\index{subcategory!full}\footnote{The definition of full subcategories will be given as Definition \ref{def:full subcategory}.}

A third option, and the one perhaps that would jump out to a category theorist, is to take the choice about how we define our objects as a clue to how we should define our morphisms. Namely, if we are so interested in joins, perhaps we want joins to be preserved under morphisms. That is, if $f\taking (X,\leq_X)\to (Y,\leq_Y)$ is a morphism of preorders then for any join $w=x\vee x'$ in $X$ we might want to enforce that $f(w)=f(x)\vee f(x')$ in $Y$. Thus a third possibility for the morphisms of $\mcC$ would be $$\Hom_\mcC(a,b):=\{f\in\Hom_{\PrO}(a,b)\|f \tn{ preserves joins}\}.$$ One can check easily that the identity morphisms preserve joins and that compositions of join-preserving morphisms are join-preserving, so this version of homomorphisms makes for a well-defined category.

\end{example}

\begin{example}[Category $\FLin$ of finite linear orders]\label{ex:FLin}\index{a category!$\FLin$}

We have a category $\PrO$ of preorders, and some of its objects are finite (nonempty) linear orders. Let $\FLin$ be the full subcategory of $\PrO$ spanned by the linear orders. That is, following Definition \ref{def:morphism of orders}, given linear orders $X,Y$, every morphism of preorders $X\to Y$ counts as a morphism in $\FLin$: $$\Hom_\FLin(X,Y)=\Hom_\PrO(X,Y).$$ 

\end{example}

\begin{exercise}
Let $\FLin$ be the category of finite linear orders, defined in Example \ref{ex:FLin}. For $n\in\NN$, let $[n]$ be the linear order defined in Example \ref{ex:finite lo}. What are the cardinalities of the following sets: 
\sexc $\Hom_{\FLin}([0],[3])$; 
\next $\Hom_\FLin([3],[0])$;
\next $\Hom_\FLin([2],[3])$;
\next $\Hom_\FLin([1],[n])$?
\next (Challenge) $\Hom_\FLin([m],[n])$?
\endsexc

It turns out that the category $\FLin$ of linear orders is sufficiently rich that much of algebraic topology (the study of arbitrary spaces, such as Mobius strips and $7$-dimensional spheres) can be understood in its terms. See Example \ref{ex:simplicial set}.
\end{exercise}

\begin{example}[Category of graphs]\index{a category!$\Grph$}

We defined graphs in Definition \ref{def:graph} and graph homomorphisms in Definition \ref{def:graph homomorphism}. To see that these are sufficient to form a category is considered routine to a seasoned category-theorist, so let's see why. 

Since a morphism from $\mcG=(V,A,src,tgt)$ to $\mcG'=(V',A',src',tgt')$ involves two functions $f_0\taking V\to V'$ and $f_1\taking A\to A'$, the identity and composition formulas will simply arise from the identity and composition formulas for sets. Associativity will follow similarly. The only thing that needs to be checked, really, is that the composition of two such things, each satisfying (\ref{dia:graph hom}), will itself satisfy (\ref{dia:graph hom}). Just for completeness, we check that now.

Suppose that $f=(f_0,f_1)\taking\mcG\to\mcG'$ and $g=(g_0,g_1)\taking\mcG'\to\mcG''$ are graph homomorphisms, where $\mcG''=(V'',A'',src'',tgt'')$. Then in each diagram below
\begin{align}\label{dia:graph hom comp}
\xymatrix{A\ar[r]^{f_1}\ar[d]^{src}&A'\ar[r]^{g_1}\ar[d]^{src'}&A''\ar[d]^{src''}\\V\ar[r]_{f_0}&V'\ar[r]_{g_0}&V''
}\hspace{1in}
\xymatrix{A\ar[r]^{f_1}\ar[d]^{tgt}&A'\ar[d]^{tgt'}\ar[r]^{g_1}&A''\ar[d]^{tgt''}\\V\ar[r]_{f_0}&V'\ar[r]_{g_0}&V''
}
\end{align}
the left-hand square commutes because $f$ is a graph homomorphism and the right-hand square commutes because $g$ is a graph homomorphism. Thus the whole rectangle commutes, meaning that $g\circ f$ is a graph homomorphism, as desired. 

We denote the category of graphs and graph homomorphisms by $\Grph$.

\end{example}

\begin{remark}

When one is struggling to understand basic definitions, notation, and style, a phase which naturally occurs when learning new mathematics (or any new language), the above example will probably appear long and tiring. I'd say you've mastered the basics when the above example really does feel straightforward. Around this time, I imagine you'll begin to get a sense of the remarkable organisational potential of the categorical way of thinking.

\end{remark} 

\begin{exercise}\label{exc:vector field 1}\index{vector field}
Let $F$ be a vector field on $\RR^2$. \href{http://en.wikipedia.org/wiki/Line_integral#Line_integral_of_a_vector_field}{Recall} that for two points $x,x'\in\RR^2$, any curve $C$ with endpoints $x$ and $x'$, and any parameterization $r\taking [a,b]\to C$, the line integral $\int_CF(r)\cdot dr$ returns a real number. It does not depend on $r$, except its orientation (direction). Therefore, if we think of $C$ has having an orientation, say going from $x$ to $x'$, then $\int_CF$ is a well-defined real number. If $C$ goes from $x$ to $x'$, let's suggestively write $C\taking x\to x'$. Define an equivalence relation $\sim$ on the set of oriented curves in $\RR^2$ by saying $C\sim C'$ if
\begin{itemize}
\item $C$ and $C'$ start at the same point,
\item $C$ and $C'$ end at the same point, and
\item $\int_CF=\int_{C'}F$.
\end{itemize}

Suppose we try to make a category $\mcC_F$ as follows. Put $\Ob(\mcC_F)=\RR^2$, and for every pair of points $x,x'\in\RR^2$, let $\Hom_{\mcC_F}(x,x')=\{C\taking x\to x'\}/\sim$, where $C\taking x\to x'$ is an oriented curve and $\sim$ means ``same line integral", as explained above. 

Is there an identity morphism and a composition formula that will make $\mcC_F$ into a category? 
\end{exercise}


\subsubsection{Isomorphisms}

In any category we have a notion of isomorphism between objects.

\begin{definition}\index{isomorphism}\index{morphism!inverse}

Let $\mcC$ be a category and let $X,Y\in\Ob(\mcC)$ be objects. An {\em isomorphism $f$ from $X$ to $Y$} is a morphism $f\taking X\to Y$ in $\mcC$, such that there exists a morphism $g\taking Y\to X$ in $\mcC$ such that $$g\circ f=\id_X\hsp\tn{and}\hsp f\circ g=\id_Y.$$ In this case we say that the morphism $f$ is {\em invertible} and that $g$ is the {\em inverse} of $f$. We may also say that the objects $X$ and $Y$ are {\em isomorphic}.

\end{definition}

\begin{example}

If $\mcC=\Set$ is the category of sets, then the above definition coincides precisely with the one given in Definition \ref{def:iso in set}.

\end{example}

\begin{exercise}
Suppose that $G=(V,A,src,tgt)$ and $G'=(V',A',src',tgt')$ are graphs and that $f=(f_0,f_1)\taking G\to G'$ is a graph homomorphism (as in Definition \ref{def:graph homomorphism}). 
\sexc If $f$ is an isomorphism in $\Grph$, does this imply that $f_0\taking V\to V'$ and $f_1\taking A\to A'$ are isomorphisms in $\Set$?
\next  If so, why; and if not, show a counterexample (where $f$ is an isomorphism but either $f_0$ or $f_1$ is not).
\endsexc
\end{exercise}

\begin{exercise}
Suppose that $G=(V,A,src,tgt)$ and $G'=(V',A',src',tgt')$ are graphs and that $f=(f_0,f_1)\taking G\to G'$ is a graph homomorphism (as in Definition \ref{def:graph homomorphism}). 
\sexc If $f_0\taking V\to V'$ and $f_1\taking A\to A'$ are isomorphisms in $\Set$, does this imply that $f$ is an isomorphism in $\Grph$?
\next If so, why; and if not, show a counterexample (where $f_0$ and $f_1$ are isomorphisms but $f$ is not).
\endsexc
\end{exercise}

\begin{lemma}\label{lemma:isomorphic ER}

Let $\mcC$ be a category and let $\sim$ be the relation on $\Ob(\mcC)$ given by saying $X\sim Y$ iff $X$ and $Y$ are isomorphic. Then $\sim$ is an equivalence relation.

\end{lemma}

\begin{proof}

The proof of Lemma \ref{lemma:isomorphic ER in Set} can be mimicked in this more general setting.

\end{proof}


\subsubsection{Another viewpoint on categories}

Here is an alternate definition of category, using the work we did in Chapter \ref{chap:sets}.

\begin{exercise}\label{exc:cat in set}

Suppose we begin our definition of category as follows. 

A {\em category}, $\mcC$ consists of a sequence $(\Ob(\mcC),\Hom_\mcC,dom,cod,\ids,\circ)$, where 
\begin{enumerate}
\item $\Ob(\mcC)$ is a set,\footnote{See Remark \ref{rmk:small}.}
\item $\Hom_\mcC$ is a set, and $dom,cod\taking\Hom_\mcC\to\Ob(\mcC)$ are functions, 
\item $\ids\taking\Ob(\mcC)\to\Hom_\mcC$ is a function, and 
\item $\circ$ is a function as depicted in the commutative diagram below
\begin{align}\label{dia:pullback version of cat}
\xymatrix{
\Hom_\mcC\ar@/^1pc/[drrr]^{cod}\ar@/_1pc/[dddr]_{dom}&&\\
&\Hom_\mcC\times_{\Ob(\mcC)}\Hom_\mcC\ar@{}[ur]|{\checkmark}\ar@{}[ddl]|(.4){\checkmark}\ar[ul]_\circ\ar[r]\ar[d]\ullimit&\Hom_\mcC\ar[r]_{cod}\ar[d]^{dom}&\Ob(\mcC)\\
&\Hom_\mcC\ar[r]_{cod}\ar[d]^{dom}&\Ob(\mcC)\\
&\Ob(\mcC)
}
\end{align}
\end{enumerate}

\sexc Express the fact that for any $x\in\Ob(\mcC)$ the morphism $\id_x$ points from $x$ to $x$ in terms of the functions $\id,dom,cod$. 
\next Express the condition that composing a morphism $f$ with an appropriate identity morphism yields $f$.
\next Express the associativity law in these terms (Hint: Proposition \ref{prop:pasting} may be useful).
\endsexc
\end{exercise}

\begin{example}[Partial olog for a category]

Below is an olog that captures some of the essential structures of a category.

\begin{align}\label{dia:olog for cats}
\xymatrixnocompile@=30pt{
\obox{}{.72in}{a morphism in $\mcC$}\ar@/^1pc/[drrr]^{\hsp\tn{has as codomain}}\ar@/_1pc/[dddr]_{\tn{has as domain}}&&&\\
&\obox{}{.85in}{a pair $(g,f)$ of composable morphisms}\ar@{}[l]|{\checkmark}\ar@{}[ur]|{\checkmark}\ar[ul]_{\;\;\tn{has as composition}}\ar[r]^{\parbox{.3in}{\scriptsize\rr\tn{yields as $g$}}}\ar[d]_{\tn{yields as $f$}}\ar@{}[rd]|(.35)*+{\lrcorner}&\obox{}{.72in}{a morphism in $\mcC$}\ar[r]_{\parbox{.4in}{\scriptsize\rr\tn{has as codomain}}}\ar[d]^{\tn{has as domain}}&\obox{}{.8in}{an object of $\mcC$}\\
&\obox{}{.72in}{a morphism in $\mcC$}\ar[r]_{\parbox{.4in}{\scriptsize\rr\tn{has as codomain}}}\ar[d]^{\tn{has as domain}}&\obox{}{.8in}{an object of $\mcC$}\\
&\obox{}{.8in}{an object of $\mcC$}
}
\end{align}

Missing from (\ref{dia:olog for cats}) is the notion of identity morphism (as an arrow from \fakebox{an object of $\mcC$} to \fakebox{a morphism in $\mcC$}) and the associated path equivalences, as well as the identity and associativity laws. All of these can be added to the olog, at the expense of some clutter.

\end{example}

\begin{remark}

Perhaps it is already clear that category theory is very interconnected. It may feel like everything relates to everything, and this feeling may intensify as you go on. However, the relationships between different notions are rigorously defined, and not random. Moreover, almost everything presented in this book can be formalized in a proof system like \href{http://en.wikipedia.org/wiki/Coq}{\text Coq} (the most obvious exceptions being things like the readability requirement of ologs and the modeling of scientific applications).

Whenever you feel cognitive vertigo, look to formal definitions as the ground of your understanding. It is good practice to make sure that the intuition you've developed actually ``touches down" on that ground, i.e. that your way of thinking can be built up solidly from the foundational definitions.

\end{remark}


\subsection{Functors}

A category $\mcC=(\Ob(\mcC),\Hom_\mcC,dom,cod,\ids,\circ)$, involves a set of objects, a set of morphisms, a notion of domains and codomains, a notion of identity morphisms, and a composition formula. For two categories to be comparable, these various components should be appropriately comparable.\index{appropriate comparison}

\begin{definition}\label{def:functor}\index{functor}

Let $\mcC$ and $\mcC'$ be categories. A {\em functor $F$ from $\mcC$ to $\mcC'$}, denoted $F\taking\mcC\to\mcC'$, is defined as follows: One announces some constituents (A. on-objects part, B. on-morphisms part) and asserts that they conform to some laws (1. preservation of identities, 2. preservation of composition). Specifically, one announces
\begin{enumerate}[\hsp A.]
\item a function $\Ob(F)\taking\Ob(\mcC)\to\Ob(\mcC')$, which we sometimes denote simply by $F\taking\Ob(\mcC)\to\Ob(\mcC')$; and
\item for every pair of objects $c,d\in\Ob(\mcC)$, a function $$\Hom_F(c,d)\taking\Hom_\mcC(c,d)\to\Hom_{\mcC'}(F(c),F(d)),$$ which we sometimes denote simply by $F\taking\Hom_\mcC(c,d)\to\Hom_{\mcC'}(F(c),F(d))$.
\end{enumerate}
One asserts that the following laws hold:
\begin{enumerate}[\hsp 1.]
\item Identities are preserved by $F$. That is, for any object $c\in\Ob(\mcC)$, we have $F(\id_c)=\id_{F(c)}$; and
\item Composition is preserved by $F$. That is, for any objects $b,c,d\in\Ob(\mcC)$ and morphisms $g\taking b\to c$ and $h\taking c\to d$, we have $F(h\circ g)=F(h)\circ F(g)$.
\end{enumerate}
\end{definition}

\begin{example}[Monoids have underlying sets]

Recall from Definition \ref{def:monoid} that if $\mcM=(M,e,\star)$ is a monoid, then $M$ is a set. And recall from Definition \ref{def:monoid hom} that if $f\taking\mcM\to\mcM'$ is a monoid homomorphism then $f\taking M\to M'$ is a function. Thus we have a functor $$U\taking\Mon\to\Set$$\index{a functor!$\Mon\to\Set$} that takes every monoid to its underlying set and every monoid homomorphism to its underlying function. 

Given two monoids $\mcM=(M,e,\star)$ and $\mcM'=(M',e',\star')$, there may be many functions from $M$ to $M'$ that do not arise from monoid homomorphisms. It is often useful to speak of such functions. For example, one could assign to every command in one video game $V$ a command in another video game $V'$, but this may not work in the ``monoidy way" when performing a sequence of commands. By being able to speak of $M$ as a set, or as $\mcM$ as a monoid, and understanding the relationship $U$ between them, we can be clear about where we stand at all times in our discussion.

\end{example}

\begin{example}[Groups have underlying monoids]\label{ex:grp to monoid}

Recall that a group is just a monoid $(M,e,\star)$ with the extra property that every element $m\in M$ has an inverse $m'\star m=e=m\star m'$. Thus to every group we can assign its {\em underlying monoid}. Similarly, a group homomorphism is just a monoid homomorphism of its underlying monoids. This means that there is a functor $$U\taking\Grp\to\Mon$$\index{a functor!$\Grp\to\Mon$} that sends every group or group homomorphism to its underlying monoid or monoid homomorphism. That identity and composition are preserved is obvious.

\end{example}

\begin{slogan}
Out of all our available actions, some are reversable. 
\end{slogan}

\begin{application}

Suppose you're a scientist working with symmetries. But then suppose that the symmetry breaks somewhere, or you add some extra observable which is not reversible under the symmetry. You want to seamlessly relax the requirement that every action be reversible without changing anything else. You want to know where you can go, or what's allowed. The answer is to simply pass from the category of groups (or group actions) to the category of monoids (or monoid actions). 

We can also reverse this change of perspective. Recall that in Example \ref{ex:monoid as olog} we discussed a monoid $M$ controlling the actions of a video game character. The character position ($P$) could be moved up ($u$), moved down ($d$), or moved right ($r$). The path equivalences $P.u.d=P$ and $P.d.u=P$ imply that these two actions are mutually inverse, whereas moving right has no inverse. This, plus equivalences $P.r.u=P.u.r$ and $P.r.d=P.d.r$, defined a monoid $M$. 

Inside $M$ is a submonoid $G$, which includes just upward and downward movement. It has one object, just like $M$, i.e. $\Ob(M)=\{P\}=\Ob(G)$. But it has fewer morphisms. In fact there is a monoid isomorphism $G\iso\ZZ$ because we can assign to any movement in $G$ the number of ups, e.g. $P.u.u.u.u.u$ is assigned the integer $5$, $P.d.d.d$ is assigned the integer $-3$, and $P.d.u.u.d.d.u$ is assigned the integer $0\in\ZZ$. But $\ZZ$ is a group, because every integer has an inverse.

Thus we can consider $G$ as a group $G_1\in\Ob(\Grp)$ or as a monoid $G_2\in\Ob(\Mon)$. It is better to consider $G$ as a group, because groups are more structured than monoids. It's as though putting $G$ in $\Grp$ gives it more ``potential energy" than putting it in $\Mon$ --- we can always ``drop it down" from $\Grp$ to $\Mon$, but not vice versa. The way to make this precise is that we can make use of the functor $U\taking\Grp\to\Mon$ from Example \ref{ex:grp to monoid} and find that $U(G_1)=G_2$. But to find a functor $F\taking\Mon\to\Grp$ such that $F(G_2)=G_1$ would be much more ad hoc. 

The upshot is that we can use functors to compare groups and monoids.

\end{application}

\begin{example}

Recall that we have a category $\Set$ of sets and a category $\Fin$ of finite sets. We said that $\Fin$ was a subcategory of $\Set$. In fact we can think of this ``subcategory" relationship in terms of functors, just like we thought of the ``subset" relationship in terms of functions in Example \ref{ex:subset as function}. That is, if we have a subset $S\ss S'$, then every element $s\in S$ is an element of $S'$, so we make a function $f\taking S\to S'$ such that $f(s)=s\in S'$. 

To give a functor $i\taking\Fin\to\Set$, we have to announce how it will work on objects and how it will work on morphisms. We begin by announcing a function $i\taking\Ob(\Fin)\to\Ob(\Set)$. But that's easy because $\Ob(\Fin)\ss\Ob(\Set)$, so we proceed as above: $i(S)=S$ for any $S\in\Ob(\Fin)$. We also have announce, for each pair of objects $S,S'\in\Ob(\Fin)$, a function $$i\taking\Hom_\Fin(S,S')\to\Hom_\Set(S,S').$$ But again, that's easy because we know by definition (see Example \ref{ex:Fin}) that these two sets are equal, $\Hom_\Fin(S,S')=\Hom_\Set(S,S')$. Hence we can simply take $i$ to be the identity function on morphisms. It is easy to see that identites and compositions are preserved by $i$. Therefore, we have defined a functor $i$.

\end{example}

\begin{exercise}[Forgetful functors between types of orders]
A partial order is just a preorder with a special property. A linear order is just a partial order with a special property.
\sexc Is there an ``obvious" functor $\FLin\to\PrO$\index{a functor!$\FLin\to\PrO$}?
\next Is there an ``obvious" functor $\PrO\to\FLin$?
\endsexc
\end{exercise}

\begin{proposition}[Preorders to graphs]\label{prop:pro to grph}

Let $\PrO$ be the category of preorders and $\Grph$ be the category of graphs. There is a functor $P\taking\PrO\to\Grph$\index{a functor!$\PrO\to\Grph$} such that for any preorder $\mcX=(X,\leq)$, the graph $P(\mcX)$ has vertices $X$.

\end{proposition}

\begin{proof}

Given a preorder $\mcX=(X,\leq_X)$, we can make a graph $F(\mcX)$ with vertices $X$ and an arrow $x\to x'$ whenever $x\leq_X x'$, as in Remark \ref{rem:preorder to graph}. More precisely, the preorder $\leq_X$ is a relation, i.e. a subset $R_\mcX\ss X\times X$, which we think of as a function $i\taking R_\mcX\to X\times X$. Composing with projections $\pi_1,\pi_2\taking X\times X\to X$ gives us $$src_\mcX:=\pi_1\circ i\taking R_\mcX\to X\hsp\tn{and}\hsp tgt_\mcX:=\pi_2\circ i\taking R_\mcX\to X.$$ Then we put $F(\mcX):=(X,R_\mcX,src_\mcX,tgt_\mcX)$. This gives us a function $F\taking\Ob(\PrO)\to\Ob(\Grph)$.

Suppose now that $f\taking\mcX\to\mcY$ is a preorder morphism (where $\mcY=(Y,\leq_Y)$). This is a function $f\taking X\to Y$ such that for any $(x,x')\in X\times X$, if $x\leq_X x'$ then $f(x)\leq f(x')$. But that's the same as saying that there exists a dotted arrow making the following diagram of sets commute
$$
\xymatrix{R_\mcX\ar[r]\ar@{..>}[d]&X\times X\ar[d]^{f\times f}\\R_\mcY\ar[r]&Y\times Y
}
$$
(Note that there cannot be two different dotted arrows making that diagram commute because $R_\mcY\to Y\times Y$ is a monomorphism.) 
Our commutative square is precisely what's needed for a graph homomorphism, as shown in Exercise \ref{exc:single condition for graph hom}. Thus, we have defined $F$ on objects and on morphisms. It is clear that $F$ preserves identity and composition.

\end{proof}

\begin{exercise}
In Proposition \ref{prop:pro to grph} we gave a functor $P\taking\PrO\to\Grph$.
\sexc  Is every graph $G\in\Ob(\Grph)$ in the image of $P$ (or more precisely, is the function $$\Ob(P)\taking\Ob(\PrO)\to\Ob(\Grph)$$ surjective)?
\next If so, why; if not, name a graph not in the image.
\next Suppose that $G, H\in\Ob(\Grph)$ are two graphs that are in the image of $P$. Is every graph homomorphism $f\taking G\to H$ in the image of $\Hom_P$? In other words, does every graph homomorphism between $G$ and $H$ come from a preorder homomorphism?
\endsexc
\end{exercise}

\begin{remark}

There is a functor $W\taking\PrO\to\Set$\index{a functor!$\PrO\to\Set$} sending $(X,\leq)$ to $X$. There is a functor $T\taking\Grph\to\Set$\index{a functor!$\Grph\to\Set$} sending $(V,A,src,tgt)$ to $V$. When we understand the category of categories (Section \ref{sec:cat of cats}), it will be clear that Proposition \ref{prop:pro to grph} can be summarized as a commutative triangle in $\Cat$, 
$$
\xymatrix@=15pt{\PrO\ar[rr]^P\ar[ddr]_W&&\Grph\ar[ddl]^T\\\\&\Set}
$$

\end{remark}

\begin{exercise}[Graphs to preorders]\label{exc:grph to pro}
Recall from (\ref{dia:image}) that every function $f\taking A\to B$ has an image, $\im_f(A)\ss B$. Use this idea and Example \ref{ex:preorder generated} to construct a functor $Im\taking\Grph\to\PrO$\index{a functor!$\Grph\to\PrO$} such that for any graph $G=(V,A,src,tgt)$, the preorder has elements given by the vertices of $G$ (i.e. we have $Im(G)=(V,\leq_G)$, for some ordering $\leq_G$).
\end{exercise}

\begin{exercise}
What is the preorder $Im(G)$ when $G\in\Ob(\Grph)$ is the following graph?
$$
G:=\parbox{2in}{\fbox{\xymatrix{\LMO{v}\ar[r]^f&\LMO{w}\ar@/_1pc/[r]_h\ar@/^1pc/[r]^g&\LMO{x}\\\LMO{y}\ar@(l,u)[]^i\ar@/^1pc/[r]^j&\LMO{z}\ar@/^1pc/[l]^k}}}
$$
\end{exercise}

\begin{exercise}
Consider the functor $Im\taking\Grph\to\PrO$ constructed in Exercise \ref{exc:grph to pro}.
\sexc Is every preorder $\mcX\in\Ob(\PrO)$ in the image of $Im$ (or more precisely in the image of $\Ob(Im)\taking\Ob(\Grph)\to\Ob(\PrO)$)?
\next If so, why; if not, name a preorder not in the image.
\next Suppose that $\mcX,\mcY\in\Ob(\PrO)$ are two preorders that are in the image of $Im$. Is every preorder morphism $f\taking\mcX\to\mcY$ in the image of $\Hom_{Im}$? In other words, does every preorder homomorphism between $\mcX$ and $\mcY$ come from a graph homomorphism?
\endsexc
\end{exercise}

\begin{exercise}
We have functors $P\taking\PrO\to\Grph$ and $Im\taking\Grph\to\PrO$.
\sexc What can you say about $Im\circ P\taking\PrO\to\PrO$?
\next What can you say about $P\circ Im\taking\Grph\to\Grph$?
\endsexc
\end{exercise}

\begin{exercise}
Consider the functors $P\taking\PrO\to\Grph$ and $Im\taking\Grph\to\PrO$. And consider the chain graph $[n]$ of length $n$ from Example \ref{ex:[n] as graph} and the linear order $[n]$ of length $n$ from Example \ref{ex:finite lo}. To differentiate the two, let's rename them for this exercise as $[n]_{\Grph}\in\Ob(\Grph)$ and $[n]_{\PrO}\in\Ob(\PrO)$. We see a similarity between $[n]_{\Grph}$ and $[n]_{\PrO}$, and we might hope that our functors help us formalize this similarity. That is, we might hope that one of the following hold: 
$$P([n]_{\PrO})\iso^? [n]_{\Grph}\hsp\tn{or}\hsp Im([n]_{\Grph})\iso^? [n]_{\PrO}.$$ 
Do either, both, or neither of these hold?
\end{exercise}

\begin{remark}

In the course announcement for 18-S996, I wrote the following:
\begin{quote}
It is often useful to focus ones study by viewing an individual thing, or a group of things, as though it exists in isolation. However, the ability to rigorously change our point of view, seeing our object of study in a different context, often yields unexpected insights. Moreover this ability to change perspective is indispensable for effectively communicating with and learning from others. It is the relationships between things, rather than the things in and by themselves, that are responsible for generating the rich variety of phenomena we observe in the physical, informational, and mathematical worlds.
\end{quote}
This holds at many different levels. For example, one can study a group (in the sense of Definition \ref{def:group}) in isolation, trying to understand its subgroups or its automorphisms, and this is mathematically interesting. But one can also view it as a quotient of something else, or as a subgroup of something else. One can view the group as a monoid and look at monoid homomorphisms to or from it. One can look at the group in the context of symmetries by seeing how it acts on sets. These changes of viewpoint are all clearly and formally expressible within category theory. We know how the different changes of viewpoint compose and how they fit together in a larger context. 

\end{remark}

\begin{exercise}~
\sexc Is the above quote also true in your scientific discipline of expertise? How so? 
\next Can you imagine a way that category theory can help catalogue the kinds of relationships or changes of viewpoint that exist in your discipline? 
\next What kinds of structures that you use often really deserve to be better formalized?
\endsexc
Keep this kind of question in mind for your final project.
\end{exercise}

\begin{example}[Free monoids]\label{ex:free monoid}\index{monoid!free}

Let $G$ be a set. We saw in \ref{def:free monoid} that $\List(G)$ is a monoid, called the free monoid on $G$. Given a function $f\taking G\to G'$, there is an induced function $\List(f)\taking\List(G)\to\List(G')$, and this preserves the identity element $[\;]$ and concatenation of lists, so $\List(f)$ is a monoid homomorphism. It is easy to check that $\List\taking\Set\to\Mon$\index{a functor!$\Set\to\Mon$} is a functor.

\end{example}

\begin{application}\label{app:polymerase}
In Application \ref{app:DNA RNA} we discussed an isomorphism $\tn{Nuc}_\tn{DNA}\iso\tn{Nuc}_\tn{RNA}$ given by RNA transcription. Applying the functor $\List$ we get a function $$\List(\tn{Nuc}_\tn{DNA})\To{\iso}\List(\tn{Nuc}_\tn{RNA}),$$ which will send sequences of DNA nucleotides to sequences of RNA nucleotides and vice versa. This is performed by polymerases.

\end{application}

\begin{exercise}\label{exc:list as functor}\index{list!as functor}
Let $G=\{1,2,3,4,5\}, G'=\{a,b,c\}$, and let $f\taking G\to G'$ be given by the sequence $(a,c,b,a,c)$.\footnote{See Exercise \ref{exc:sequence} in case there is any confusion with this.} Then if $L=[1,1,3,5,4,5,3,2,4,1]$, what is $\List(f)(L)$?
\end{exercise}

\begin{exercise}\label{exc:rephrase functors}
We can rephrase our notion of functor in terms compatible with Exercise \ref{exc:cat in set}. We would begin by saying that a functor $F\taking\mcC\to\mcC'$ consists of two functions, $$\Ob(F)\taking\Ob(\mcC)\to\Ob(\mcC')\hsp\tn{and}\hsp\Hom_F\taking\Hom_\mcC\to\Hom_{\mcC'},$$ which we call the {\em on-objects part} and the {\em on-morphisms part}, respectively. They must follow some rules, expressed by the commutativity of the following squares in $\Set$:
\begin{align}\label{dia:rephrase functors}
\xymatrix{
\Hom_\mcC\ar[r]^{dom}\ar[d]_{\Hom_F}&\Ob(\mcC)\ar[d]^{\Ob(F)}\\
\Hom_{\mcC'}\ar[r]_{dom}&\Ob(\mcC')
}
\hspace{1in}
\xymatrix{
\Hom_\mcC\ar[r]^{cod}\ar[d]_{\Hom_F}&\Ob(\mcC)\ar[d]^{\Ob(F)}\\
\Hom_{\mcC'}\ar[r]_{cod}&\Ob(\mcC')}
\end{align}
\begin{align}\label{dia:rephrase functors 2}
\xymatrix{
\Ob(\mcC)\ar[d]_{\Ob(F)}\ar[r]^{\id}&\Hom_\mcC\ar[d]^{\Hom_F}\\
\Ob(\mcC')\ar[r]_{\id}&\Hom_{\mcC'}
}
\hspace{1in}
\xymatrix{
\Hom_\mcC\times_{\Ob(\mcC)}\Hom_{\mcC}\ar[r]^-{\circ}\ar[d]_{}&\Hom_\mcC\ar[d]^{\Hom_F}\\
\Hom_{\mcC'}\times_{\Ob(\mcC')}\Hom_{\mcC'}\ar[r]_-{\circ}&\Hom_{\mcC'}}
\end{align}
Where does the (unlabeled) left-hand function in the bottom right diagram come from? Hint: use Exercise \ref{exc:pointwise map of fp}.

Consider Diagram (\ref{dia:pullback version of cat}) and imagine it as though contained in a pane of glass. Then imagine a parallel pane of glass involving $\mcC'$ in place of $\mcC$ everywhere. 
\sexc Draw arrows from the $\mcC$ pane to the $\mcC'$ pane, each labeled $\Ob(F)$ or $\Hom_F$ as seems appropriate.
\next If $F$ is a functor (i.e. satisfies (\ref{dia:rephrase functors}) and (\ref{dia:rephrase functors 2})), do all the squares in your drawing commute?
\next  Does the definition of functor involve anything not captured in this setup?
\endsexc
\end{exercise}

\begin{example}[Paths-graph]\label{ex:paths-graph}\index{graph!paths-graph}

Let $G=(V,A,src,tgt)$ be a graph. Then for any pair of vertices $v,w\in G$, there is a set $\Path_G(v,w)$ of paths from $v$ to $w$; see Definition \ref{def:paths in graph}. In fact there is a set $\Path_G$ and functions $\ol{src},\ol{tgt}\taking\Path_G\to V$. That information is enough to define a new graph, $$\Paths(G):=(V,\Path_G,\ol{src},\ol{tgt}).$$

Moreover, given a graph homomorphism $f\taking G\to G'$, every path in $G$ is sent under $f$ to a path in $G'$. So $\Paths\taking\Grph\to\Grph$\index{a functor!$\Paths\taking\Grph\to\Grph$} is a functor.

\end{example}

\begin{exercise}\label{exc:morphisms on paths-graphs}~
\sexc Consider the graph $G$ from Example \ref{ex:graph hom}. Draw the paths-graph $\Paths(G)$ for $G$. 
\next Repeating the above exercise for $G'$ from the same example would be hard, because the path graph $\Paths(G')$ has infinitely many arrows. However, the graph homomorphism $f\taking G\to G'$ does induce a morphism of paths-graphs $\Paths(f)\taking\Paths(G)\to\Paths(G')$, and it is possible to say how that acts on the vertices and arrows of $\Paths(G)$. Please do so.
\next Given a graph homomorphism $f\taking G\to G'$ and two paths $p\taking v\to w$ and $q\taking w\to x$ in $G$, is it true that $\Paths(f)$ preserves the concatenation? What does that even mean?
\endsexc
\end{exercise}

\begin{exercise}\label{exc:functors preserve isos}
Suppose that $\mcC$ and $\mcD$ are categories, $c,c'\in\Ob(\mcC)$ are objects, and $F\taking\mcC\to\mcD$ is a functor. Suppose that $c$ and $c'$ are isomorphic in $\mcC$. Show that this implies that $F(c)$ and $F(c')$ are isomorphic in $\mcD$.
\end{exercise}

\begin{example}\label{ex:non-isomorphism of graphs via functors}

For any graph $G$, we can assign its set of loops $Eq(G)$ as in Exercise \ref{exc:(co)equalizer of graph}. This assignment is functorial in that given a graph homomorphism $G\to G'$ there is an induced function $Eq(G)\to Eq(G')$. Similarly, we can functorially assign the set of connected components of the graph, $Coeq(G)$. In other words $Eq\taking\Grph\to\Set$ and $Coeq\taking\Grph\to\Set$ are functors. The assignment of vertex set and arrow set are two more functors $\Grph\to\Set$.

Suppose you want to decide whether two graphs $G$ and $G'$ are isomorphic. Supposing that the graphs have thousands of vertices and thousands of arrows, this could take a long time. However, the functors above, in combination with Exercise \ref{exc:functors preserve isos} give us some things to try.

The first thing to do is to count the number of loops of each, because these numbers are generally small. If the number of loops in $G$ is different than the number of loops in $G'$ then because functors preserve isomorphisms, $G$ and $G'$ cannot be isomorphic. Similarly one can count the number of connected components, again generally a small number; if the number of components in $G$ is different than the number of components in $G'$ then $G\not\iso G'$. Similarly, one can simply count the number of vertices or the number of arrows in $G$ and $G'$. These are all isomorphism invariants.  

All this is a bit like trying to decide if a number is prime by checking if it's even, if its digits add up to a multiple of 3, or it ends in a 5; these tests do not determine the answer, but they offer some level of discernment.

\end{example}

\begin{remark}

In the introduction I said that functors allow ideas in one domain to be rigorously imported to another. Example \ref{ex:non-isomorphism of graphs via functors} is a first taste. Because functors preserve isomorphisms, we can tell graphs apart by looking at them in a simpler category, $\Set$. There is relatively simple theorem in $\Set$ that says that for different natural numbers $m,n$ the sets $\ul{m}$ and $\ul{n}$ are never isomorphic. This theorem is transported via our four functors to four different theorems about telling graphs apart.

\end{remark}


\subsubsection{The category of categories}\label{sec:cat of cats}

Recall from Remark \ref{rmk:small} that a small category $\mcC$ is one in which $\Ob(\mcC)$ is a set. We have not really been paying attention to this issue, and everything we have said so far works whether $\mcC$ is small or not. In the following definition we really ought to be a little more careful, so we are. 

\begin{proposition}\index{a category!$\Cat$}

There exists a category, called {\em the category of small categories} and denoted $\Cat$, in which the objects are the small categories and the morphisms are the functors, $$\Hom_\Cat(\mcC,\mcD)=\{F\taking\mcC\to\mcD\| F \tn{ is a functor}\}.$$ That is, there are identity functors, functors can be composed, and the identity and associativity laws hold.
 
\end{proposition}

\begin{proof}

We follow Definition \ref{def:category}. We have specified $\Ob(\Cat)$ and $\Hom_\Cat$ already. Given a small category $\mcC$, there is an identity functor $\id_\mcC\taking\mcC\to\mcC$ that is identity on the set of objects and the set of morphisms. And given a functor $F\taking\mcC\to\mcD$ and a functor $G\taking\mcD\to\mcE$, it is easy to check that $G\circ F\taking \mcC\to\mcE$, defined by composition of functions $\Ob(G)\circ\Ob(F)\taking\Ob(\mcC)\to\Ob(\mcE)$ and $\Hom_G\circ\Hom_F\taking\Hom_\mcC\to\Hom_\mcE$ (see Exercise \ref{exc:rephrase functors}), is a functor. For the same reasons, it is easy to show that functors obey the identity law and the composition formula. Therefore this specification of $\Cat$ satisfies the definition of being a category. 

\end{proof}

\begin{example}[Categories have underlying graphs]\label{ex:underlying graph}\index{category!underlying graph of}

Let $\mcC=(\Ob(\mcC),\Hom_\mcC,dom,cod,\ids,\circ)$ be a category (see Exercise \ref{exc:cat in set}). Then $(\Ob(\mcC),\Hom_\mcC,dom,cod)$ is a graph, which we will call the {\em graph underlying $\mcC$} and denote by $U(\mcC)\in\Ob(\Grph)$. A functor $F\taking\mcC\to\mcD$ induces a graph morphism $U(F)\taking U(\mcC)\to U(\mcD)$, as seen in (\ref{dia:rephrase functors}). So we have a functor, $$U\taking\Cat\to\Grph.$$

\end{example}

\begin{example}[Free category on a graph]\label{ex:free category}\index{category!free category}\index{graph!free category on}

In Example \ref{ex:paths-graph}, we discussed a functor $\Paths\taking\Grph\to\Grph$\index{a functor!$\Paths\taking\Grph\to\Grph$} that considered all the paths in a graph $G$ as the arrows of a new graph $\Paths(G)$. In fact, $\Paths(G)$ could be construed as a category, which we will denote $F(G)\in\Ob(\Cat)$ and call {\em the free category generated by $G$}. 

Here, the objects of the category $F(G)$ are the vertices of $G$. For any two vertices $v,v'$ the hom-set $\Hom_{F(G)}(v,v')$ is the set of paths in $G$ from $v$ to $v'$. The identity elements are given by the trivial paths, and the composition formula is given by concatenation of paths. 

To see that $F$ is a functor, we need to see that a graph homomorphism $f\taking G\to G'$ induces a functor $F(f)\taking F(G)\to F(G')$. But this was shown in Exercise \ref{exc:morphisms on paths-graphs}. Thus we have a functor $$F\taking\Grph\to\Cat$$\index{a functor!$\Grph\to\Cat$} called {\em the free category} functor.

\end{example}

\begin{exercise}\label{exc:[1]}
Let $G$ be the graph depicted $$\LMO{v_0}\Too{\;\;e\;\;}\LMO{v_1},$$ and let $[1]\in\Ob(\Cat)$ denote the free category on $G$ (see Example \ref{ex:free category}). We call $[1]$ the {\em free arrow category}.
\sexc What are its objects?
\next For every pair of objects in $[1]$, write down the hom-set.
\endsexc
\end{exercise}

\begin{exercise}
Let $G$ be the graph whose vertices are all cities in the US and whose arrows are airplane flights connecting cities. What idea is captured by the free category on $G$?
\end{exercise}

\begin{exercise}\label{exc:free underlying cat grph}
Let $F\taking\Grph\to\Cat$ denote the free category functor from Example \ref{ex:free category}, and let $U\taking\Cat\to\Grph$\index{a functor!$\Cat\to\Grph$} denote the underlying graph functor from Example \ref{ex:underlying graph}. We have seen the composition $U\circ F\taking\Grph\to\Grph$ before; what was it called?
\end{exercise}

\begin{exercise}
Recall the graph $G$ from Example \ref{ex:graph}. Let $\mcC=F(G)$ be the free category on $G$.
\sexc What is $\Hom_\mcC(v,x)$?
\next What is $\Hom_\mcC(x,v)$?
\endsexc
\end{exercise}

\begin{example}[Discrete graphs, discrete categories]\label{ex:discrete graph discrete cat}\index{category!discrete}

There is a functor $Disc\taking\Set\to\Grph$\index{a functor!$Disc\taking\Set\to\Grph$} that sends a set $S$ to the graph $$Disc(S):=(S,\emptyset,!,!),$$ where $!\taking\emptyset\to S$ is the unique function. We call $Disc(S)$ the {\em discrete graph on the set $S$}. It is clear that a function $S\to S'$ induces a morphism of discrete graphs. Now applying the free category functor $F\taking\Grph\to\Cat$, we get the so-called {\em discrete category on the set $S$}, which we also might call $Disc\taking\Set\to\Cat$.\index{a functor!$Disc\taking\Set\to\Cat$} 

\end{example}

\begin{exercise}
Recall from (\ref{dia:underline n}) the definition of the set $\ul{n}$ for any natural number $n\in\NN$, and let $D_n:=Disc(\ul{n})\in\Ob(\Cat)$.
\sexc List all the morphisms in $D_4$. 
\next List all the functors $D_3\to D_2.$
\endsexc
\end{exercise}

\begin{exercise}[Terminal category]\label{exc:term cat}\index{a category!terminal}
Let $\mcC$ be a category. How many functors are there $\mcC\to D_1$, where $D_1:=Disc(\ul{1})$ is the discrete category on one element?
\end{exercise}

We sometimes refer to $Disc(\ul{1})$ as the {\em terminal category} (for reasons that will be made clear in Section \ref{sec:lims and colims in a cat}), and for simplicity denote it by $\ul{1}$.

\begin{exercise}\label{exc:Ob is a functor}
If someone said ``$\Ob$ is a functor from $\Cat$ to $\Set$," what might they mean? \index{a functor!$\Ob\taking\Cat\to\Set$}
\end{exercise}


\section{Categories and functors commonly arising in mathematics}


\subsection{Monoids, groups, preorders, and graphs}\label{sec:mon grp pro as cat}

We saw in Section \ref{sec:categories} that there is a category $\Mon$ of monoids, a category $\Grp$ of groups, a category $\PrO$ of preorders, and a category $\Grph$ of graphs. In this section we show that each monoid $\mcM$, each group $\mcG$, and each preorder $\mcP$ can be considered as its own category. If each object in $\Mon$ is a category, we might hope that each morphism in $\Mon$ is just a functor, and this is true. The same holds for $\Grp$ and $\PrO$. We will deal with graphs in Section \ref{sec:graphs as functors}.


\subsubsection{Monoids as categories}\label{sec:monoids as cats}\index{monoid!as category}

In Example \ref{ex:monoid as olog} we said that to olog a monoid, we should use only one box. And again in Example \ref{ex:monoid action table} we said that a monoid action could be captured by only one table. These ideas emanated from the understanding that a monoid is perfectly modeled as a category with one object. 

\paragraph{Each monoid as a category with one object}

Let $(M,e,\star)$ be a monoid. We consider it as a category $\mcM$ with one object, $\Ob(\mcM)=\{\monOb\}$, and $$\Hom_\mcM(\monOb,\monOb):=M.$$ The identity morphism $\id_\monOb$ serves as the monoid identity $e$, and the composition formula $$\circ\taking\Hom_\mcM(\monOb,\monOb)\times\Hom_\mcM(\monOb,\monOb)\to\Hom_\mcM(\monOb,\monOb)$$ is given by $\star\taking M\times M\to M$. The associativity and identity laws for the monoid match precisely with the associativity and identity laws for categories.

If monoids are categories with one object, is there any categorical way of phrasing the notion of monoid homomorphism? Suppose that $\mcM=(M,e,\star)$ and $\mcM'=(M',e',\star')$. We know that a monoid homomorphism is a function $f\taking M\to M'$ such that $f(e)=e'$ and such that for every pair $m_0,m_1\in M$ we have $f(m_0\star m_1)=f(m_0)\star' f(m_1)$. What is a functor $\mcM\to\mcM'$? 

\paragraph{Each monoid homomorphism as a functor between one-object categories}

Say that $\Ob(\mcM)=\{\monOb\}$ and $\Ob(\mcM')=\{\monOb'\}$; and we know that $\Hom_\mcM(\monOb,\monOb)=M$ and $\Hom_{\mcM'}(\monOb',\monOb')=M'$. A functor $F\taking\mcM\to\mcM'$ consists first of a function $\Ob(\mcM)\to\Ob(\mcM')$, but these sets have only one element each, so there is nothing to say on that front. It also consists of a function $\Hom_\mcM\to\hom_{\mcM'}$ but that is just a function $M\to M'$. The identity and composition formulas for functors match precisely with the identity and composition formula for monoid homomorphisms, as discussed above. Thus a monoid homomorphism is nothing more than a functor between one-object categories. 
\begin{slogan}
A monoid is a category $\mcG$ with one object. A monoid homomorphism is just a functor between one-object categories.
\end{slogan}

We formalize this as the following theorem.

\begin{theorem}\label{thm:mon to cat}

There is a functor $i\taking\Mon\to\Cat$\index{a functor!$\Mon\to\Cat$} with the following properties:
\begin{itemize}
\item for every monoid $\mcM\in\Ob(\Mon)$, the category $i(\mcM)\in\Ob(\Cat)$ itself has exactly one object, $$|\Ob(i(\mcM))|=1$$ 
\item for every pair of monoids $\mcM,\mcM'\in\Ob(\Mon)$ the function $$\Hom_\Mon(\mcM,\mcM')\To{\iso}\Hom_\Cat(i(\mcM),i(\mcM')),$$ induced by the functor $i$, is a bijection.
\end{itemize}

\end{theorem}

\begin{proof}

This is basically the content of the preceding paragraphs. The functor $i$ sends a monoid to the corresponding category with one object and $i$ sends a monoid homomorphism to the corresponding functor; it is not hard to check that $i$ preserves identities and compositions.

\end{proof}

Theorem \ref{thm:mon to cat} situates the theory of monoids very nicely within the world of categories. But we have other ways of thinking about monoids, namely their actions on sets. As such it would greatly strengthen the story if we could subsume monoid actions within category theory also, and we can.

\paragraph{Each monoid action as a set-valued functor}

Recall from Definition \ref{def:monoid action} that if $(M,e,\star)$ is a monoid, an action consists of a set $S$ and a function $\acts\taking M\times S\to S$ such that $e\acts s=s$ and $m_0\acts (m_1\acts s)=(m_0\star m_1)\acts s$ for all $s\in S$. How might we relate the notion of monoid actions to the notion of functors? One idea is to try asking what a functor $F\taking\mcM\to\Set$ is; this idea will work.

Since $\mcM$ has only one object, we obtain one set, $S:=F(\monOb)\in\Ob(\Set)$. We also obtain a function $\Hom_F\taking\Hom_\mcM(\monOb,\monOb)\to\Hom_\Set(F(\monOb),F(\monOb))$, or more concisely, a function $$H_F\taking M\to\Hom_\Set(S,S).$$ By currying (see Proposition \ref{prop:curry}), this is the same as a function $\acts\taking M\times S\to S$. The rule that $e\acts s=s$ becomes the rule that functors preserve identities, $\Hom_F(\id_\monOb)=\id_S$. The other rule is equivalent to the composition formula for functors. 



\subsubsection{Groups as categories}\index{group!as category}

A group is just a monoid $(M,e,\star)$ in which every element $m\in M$ is invertible, meaning there exists some $m'\in M$ with $m\star m'=e=m'\star m.$ If a monoid is the same thing as a category $\mcM$ with one object, then a group must be a category with one object and with an additional property having to do with invertibility. The elements of $M$ are the morphisms of the category $\mcM$, so we need a notion of invertibility for morphisms. Luckily we have such a notion already, namely isomorphism. We have the following:
\begin{slogan}
A group is a category $\mcG$ with one object, such that every morphism in $\mcG$ is an isomorphism. A group homomorphism is just a functor between such categories.
\end{slogan}

\begin{theorem}\label{thm:grp to cat}

There is a functor $i\taking\Grp\to\Cat$\index{a functor!$\Grp\to\Cat$} with the following properties:
\begin{itemize}
\item for every group $\mcG\in\Ob(\Grp)$, the category $i(\mcG)\in\Ob(\Cat)$ itself has exactly one object, and every morphism $m$ in $i(\mcG)$ is an isomorphism; and 
\item for every pair of groups $\mcG,\mcG'\in\Ob(\Grp)$ the function $$\Hom_\Grp(\mcG,\mcG')\To{\iso}\Hom_\Cat(i(\mcG),i(\mcG')),$$ induced by the functor $i$, is a bijection.
\end{itemize}

\end{theorem}

Just as with monoids, an action of some group $(G,e,\star)$ on a set $S\in\Ob(\Set)$ is the same thing as a functor $\mcG\to\Set$ sending the unique object of $\mcG$ to the set $S$. 


\subsubsection{Monoid and group stationed at each object in a category}

If a monoid is just a category with one object, we can locate monoids in any category $\mcC$ by narrowing our gaze to one object in $\mcC$. Similarly for groups.

\begin{example}[Endomorphism monoid]\index{monoid!of endomorphisms}

Let $\mcC$ be a category and $x\in\Ob(\mcC)$ an object. Let $M=\Hom_\mcC(x,x)$. Note that for any two elements $f,g\in M$ we have $f\circ g\taking x\to x$ in $M$. Let $\mcM=(M,\id_x,\circ)$. It is easy to check that $\mcM$ is a monoid; it is called the {\em endomorphism monoid of $x$ in $\mcC$}.

\end{example}

\begin{example}[Automorphism group]\index{group!of automorphisms}

Let $\mcC$ be a category and $x\in\Ob(\mcC)$ an object. Let $G=\{f\taking x\to x\|f\tn{ is an isomorphism}\}.$ Let $\mcG=(G,\id_x,\circ)$. It is easy to check that $\mcG$ is a group; it is called the {\em automorphism group of $x$ in $\mcC$}.

\end{example}

\begin{exercise}
Let $S=\{1,2,3,4\}\in\Ob(\Set)$.
\sexc What is the automorphism group of $S$ in $\Set$, and how many elements does this group have?
\next What is the endomorphism monoid of $S$ in $\Set$, and how many elements does this monoid have? 
\next Recall from Example \ref{ex:grp to monoid} that every group has an underlying monoid $U(G)$; is the endomorphism monoid of $S$ the underlying monoid of the automorphism group of $S$?
\endsexc
\end{exercise}

\begin{exercise}\label{exc:symmetric square}
Consider the graph $G$ depicted below. 
$$
\xymatrix@=40pt{
\LMO{1}\ar@/^.3pc/[r]^{12}\ar@/^.3pc/[d]^{13}&\LMO{2}\ar@/^.3pc/[d]^{24}\ar@/^.3pc/[l]^{21}\\
\LMO{3}\ar@/^.3pc/[r]^{34}\ar@/^.3pc/[u]^{31}&\LMO{4}\ar@/^.3pc/[u]^{42}\ar@/^.3pc/[l]^{43}
}
$$
What is its group of automorphisms? Hint: every automorphism of $G$ will induce an automorphism of the set $\{1,2,3,4\}$; which ones will preserve the arrows?
\end{exercise}


\subsubsection{Preorders as categories}\label{sec:preorder as cat}\index{preorder!as category}

A preorder $(X,\leq)$ consists of a set $X$ and a binary relation $\leq$ that is reflexive and transitive. We can make from $(X,\leq)\in\Ob(\PrO)$ a category $\mcX\in\Ob(\Cat)$ as follows. Define $\Ob(\mcX)=X$ and for every two objects $x,y\in X$ define 
$$\Hom_\mcX(x,y)=\begin{cases}\{``x\leq y"\}&\tn{ if } x\leq y\\\emptyset&\tn{ if } x\not\leq y\end{cases}$$
To clarify: if $x\leq y$, we assign $\Hom_\mcX(x,y)$ to be the set containing only one element, namely the string ``$x\leq y$".\footnote{The name of this morphism is completely unimportant. What matters is that $\Hom_\mcX(x,y)$ has exactly one element iff $x\leq y$.} If $(x,y)$ is not in relation $\leq$, then we assign $\Hom_\mcX(x,y)$ to be the empty set. The composition formula 
\begin{align}\label{dia:comp in preorder}
\circ\taking\Hom_\mcX(x,y)\times\Hom_\mcX(y,z)\to\Hom_\mcX(x,z)
\end{align}
is completely determined because either one of two possibilities occurs. One possibility is that the left-hand side is empty (if either $x\not\leq y$ or $y\not\leq z$; in this case there is a unique function $\circ$ as in (\ref{dia:comp in preorder}). The other possibility is that the left-hand side is not empty in case $x\leq y$ and $y\leq$, which implies $x\leq z$, so the right-hand side has exactly one element $``x\leq z"$ in which case again there is a unique function $\circ$ as in (\ref{dia:comp in preorder}).

On the other hand, if $\mcC$ is a category having the property that for every pair of objects $x,y\in\Ob(\mcC)$, the set $\Hom_\mcC(x,y)$ is either empty or has one element, then we can form a preorder out of $\mcC$. Namely, take $X=\Ob(\mcC)$ and say $x\leq y$ if there exists a morphism $x\to y$ in $\mcC$. 

\begin{exercise}
We have seen that a preorder can be considered as a category $\mcP$. Recall from Definition \ref{def:orders} that a partial order is a preorder with an additional property. Phrase the defining property for partial orders in terms of isomorphisms in the category $\mcP$.
\end{exercise}

\begin{exercise}
Suppose that $\mcC$ is a preorder (considered as a category). Let $x,y\in\Ob(\mcC)$ be objects such that $x\leq y$ and $y\leq x$. Prove that there is an isomorphism $x\to y$ in $\mcC$.
\end{exercise}

\begin{example}

The olog from Example \ref{ex:pre not par} depicted a partial order, say $\mcP$. In it we have $$\Hom_\mcP(\fakebox{a diamond},\fakebox{a red card})=\{\tn{is}\}$$ and we have $$\Hom_\mcP(\fakebox{a black queen},\fakebox{a card})\iso\{\tn{is}\circ\tn{is}\};$$ Both of these sets contain exactly one element, the name is not important. The set $\Hom_\mcP(\fakebox{a 4},\fakebox{a 4 of diamonds})=\emptyset$. 

\end{example}

\begin{exercise}
Every linear order is a partial order with a special property. Can you phrase this property in terms of hom-sets?
\end{exercise}

\begin{proposition}\label{prop:preorders to cats}

There is a functor $i\taking\PrO\to\Cat$\index{a functor!$\PrO\to\Cat$} with the following properties for every preorder $(X,\leq)$:
\begin{enumerate}
\item the category $\mcX:=i(X,\leq)$ has objects $\Ob(\mcX)=X$; and
\item \label{cond:hom set in preorder} for each pair of elements $x,x'\in\Ob(\mcX)$ the set $\Hom_\mcX(x,x')$ has at most one element.
\end{enumerate}
Moreover, any category with property \ref{cond:hom set in preorder} is in the image of the functor $i$.

\end{proposition}

\begin{proof}

To specify a functor $i\taking\PrO\to\Cat$, we need to say what it does on objects and on morphisms. To an object $(X,\leq)$ in $\PrO$, we assign the category $\mcX$ with objects $X$ and a unique morphism from $x\to x'$ if $x\leq x'$; this was discussed at the top of Section \ref{sec:preorder as cat}. To a morphism $f\taking(X,\leq_X)\to(Y,\leq_Y)$ of preorders, we must assign a functor $i(f)\taking\mcX\to\mcY$. Again, to specify a functor we need to say what it does on objects and morphisms of $\mcX$. To an object $x\in\Ob(\mcX)=X$, we assign the object $f(x)\in Y=\Ob(\mcY)$. Given a morphism $f\taking x\to x'$ in $\mcX$, we know that $x\leq x'$ so by Definition \ref{def:morphism of orders} we have that $f(x)\leq f(x')$, and we assign to $f$ the unique morphism $f(x)\to f(x')$ in $\mcY$. To check that the rules of functors (preservation of identities and composition) are obeyed is routine.

\end{proof}

\begin{slogan}
A preorder is a category in which every hom-set has either 0 elements or 1 element. A preorder morphism is just a functor between such categories.
\end{slogan}

\begin{exercise}
Recall the functor $P\taking\PrO\to\Grph$\index{a functor!$\PrO\to\Grph$} from Proposition \ref{prop:pro to grph}, the functors $F\taking\Grph\to\Cat$ and $U\taking\Cat\to\Grph$ from Example \ref{exc:free underlying cat grph}, and the functor $i\taking\PrO\to\Cat$\index{a functor!$\PrO\to\Cat$} from Proposition \ref{prop:preorders to cats}.
\sexc Do either of the following diagrams of categories commute?
$$
\xymatrix@=15pt{\PrO\ar[rr]^P\ar[rdd]_i&\ar@{}[dd]|(.4)?&\Grph\ar[ldd]^F\\\\&\Cat}
\hspace{.5in}
\xymatrix@=15pt{\PrO\ar[rr]^P\ar[rdd]_i&\ar@{}[dd]|(.4)?&\Grph\\\\&\Cat\ar[uur]_U}
$$
\next We also had a functor $\Grph\to\PrO$. Does the following diagram of categories commute?
$$
\xymatrix@=15pt{\Grph\ar[rr]\ar[rdd]_F&\ar@{}[dd]|(.4)?&\PrO\ar[ldd]^i\\\\&\Cat}
$$
\endsexc
\end{exercise}


\subsubsection{Graphs as functors}\label{sec:graphs as functors}\index{graph!as functor}

Let $\mcC$ denote the category depicted below 
\begin{align}\label{dia:graph index}
\GrIn:=\fbox{\GrInSchema}\index{a category!$\GrIn$}
\end{align}
Then a functor $G\taking\GrIn\to\Set$ is the same thing as two sets $G(Ar),G(V\!e)$ and two functions $G(src)\taking G(Ar)\to G(V\!e)$ and $G(tgt)\taking G(Ar)\to G(V\!e)$. This is precisely what is needed for a graph; see Definition \ref{def:graph}. We call $\GrIn$ the {\em graph indexing category}.

\begin{exercise}
Consider the terminal category, $\ul{1}$, also known as the discrete category on one element (see Exercise \ref{exc:term cat}). Let $\GrIn$ be as in (\ref{dia:graph index}) and consider the functor $i_0\taking\ul{1}\to\GrIn$ sending the object of $\ul{1}$ to the object $V\in\Ob(\GrIn)$. If $G\taking\GrIn\to\Set$ is a graph, what is the composite $G\circ i_0$? It consists of only one set; what set is it? For example, what set is it when $G$ is the graph from Example \ref{ex:graph hom}.
\end{exercise}

If a graph is a functor $\GrIn\to\Set$, what is a graph homomorphism? We will see later in Example \ref{ex:graph hom as NT} that graph homomorphisms are homomorphisms between functors, which are called natural transformations. (Natural transformations are the highest-``level" structure that occurs in ordinary category theory.)

\begin{example}\index{graph!symmetric}

Let $\mcD$ be the category depicted below
\begin{align}\label{dia:symmetric graph index}
\mcD:=\fbox{\xymatrix{\LMO{A}\ar@(lu,ld)[]_\rho\ar@<.5ex>[r]^{src}\ar@<-.5ex>[r]_{tgt}&\LMO{V}}}
\end{align}
with the following composition formula: 
$$\rho\circ\rho=\id_A;\hsp src\circ\rho=tgt;\hsp\tn{and}\hsp tgt\circ\rho=src.$$

The idea here is that the morphism $\rho\taking A\to A$ reverses arrows. The PED $\rho\circ\rho=\id_A$ forces the fact that the reverse of the reverse of an arrow yields the original arrow. The PEDs $src\circ\rho=tgt$ and $tgt\circ\rho=src$ force the fact that when we reverse an arrow, its source and target switch roles. 

This category $\mcD$ is the {\em symmetric graph indexing category}. Just like any graph can be understood as a functor $\GrIn\to\Set$, where $\GrIn$ is the graph indexing category displayed in (\ref{dia:graph index}), any symmetric graph can be understood as a functor $\mcD\to\Set$, where $\mcD$ is the category drawn above. Given a functor $G\taking\mcD\to\Set$, we will have a set of arrows, a set of vertices, a source operation, a target operation, and a ``reverse direction" operation that all behave as expected.

It is customary to draw the connections in a symmetric graph as line segments rather than arrows between vertices. However, a better heuristic is to think that each connection between vertices consists of two arrows, one pointing in each direction. 

\begin{slogan}
In a symmetric graph, every arrow has an equal and opposite arrow.
\end{slogan}

\end{example}

\begin{exercise}
Which of the following graphs are symmetric:
\sexc The graph $G$ from (\ref{dia:graph})?
\next The graph $G$ from Exercise \ref{exc:secret turing}?
\next The graph $G'$ from (\ref{dia:graph hom example})?
\next The graph $\Loop$ from (\ref{dia:loop}), i.e. the graph having exactly one vertex and one arrow?
\next The graph $G$ from Exercise \ref{exc:symmetric square}?
\endsexc
\end{exercise}

\begin{exercise}
Let $\GrIn$ be the graph indexing category shown in (\ref{dia:graph index}) and let $\mcD$ be the symmetric graph indexing category displayed in (\ref{dia:symmetric graph index}).
\sexc How many functors are there of the form $\GrIn\to\mcD$?
\next Is one more ``reasonable" than the others? 
\next Choose the one that seems most reasonable and call it $i\taking\GrIn\to\mcD$. If a symmetric graph is a functor $S\taking\mcD\to\Set$, you can compose with $i$ to get a functor $S\circ i\taking\GrIn\to\Set$. This is a graph; what graph is it? What has changed?
\endsexc
\end{exercise}


\subsection{Database schemas present categories}\label{sec:schemas and cats intro}\index{schema!as category presentation}

Recall from Definition \ref{def:schema} that a database schema (or schema, for short) consists of a graph together with a certain kind of equivalence relation on its paths. In Section \ref{sec:sch as category} we will define a category $\Sch$ that has schemas as objects and appropriately modified graph homomorphisms as morphisms. In Section \ref{sec:proof of cat=sch} we prove that the category of schemas is equivalent (in the sense of Definition \ref{def:equiv of cats}) to the category of categories, $$\Sch\simeq\Cat.$$

The difference between schemas and categories is like the difference between monoid presentations, given by generators and relations as in Definition \ref{def:presented monoid}, and the monoids themselves. The same monoid has (infinitely) many different presentations, and so it is for categories: many different schemas can {\em present} the same category. Computer scientists may think of the schema as {\em syntax}\index{schema!as syntax} and the category it presents as the corresponding {\em semantics}. A schema is a compact form, and can be specified in finite space and time while generating something infinite. 

\begin{slogan}
A database schema is a category presentation.
\end{slogan}

We will formally show in Section \ref{sec:proof of cat=sch} how to turn a schema into a category (the category it {\em presents}). For now, it seems pedagogically better not to be so formal, because the idea is fairly straightforward. Suppose given a schema $\mcS$, which consists of a graph $G=(V,A,src,tgt)$ equipped with a congruence $\sim$ (see Definition \ref{def:congruence}). It presents a category $\mcC$ defined as follows. The set of objects in $\mcC$ is defined to be the vertices $V$; the set of morphisms in $\mcC$ is defined to be the quotient $\Paths(G)/\sim$; and the composition law is concatenation of paths. The path equivalences making up $\sim$ become commutative diagrams in $\mcC$.\index{category!presentation}\index{schema!as category presentation}

\begin{example}

The schema $\Loop$, depicted below, has no path equivalence declarations. As a graph it has one vertex and one arrow.
$$\Loop:=\LoopSchema$$ 
The category it generates, however, is the free monoid on one generator, $\NN$. It has one object $\monOb$ but a morphism $f^n\taking\monOb\to\monOb$ for every natural number $n\in\NN$, thought of as ``how many times to go around the loop $f$". Clearly, the schema is more compact that the infinite category it generates.

\end{example}

\begin{exercise}
Consider the olog from Exercise \ref{exc:father and child}, which says that for any father $x$, his first child's father is $x$. It is redrawn below as a schema $\mcS$, and we include the desired path equivalence declaration, $F\;c\;f=F$,
$$
\xymatrix{\LMO{F}\ar[r]^c&\LMO{C}\ar@/^1pc/[l]^f}
$$ 
How many morphisms are there (total) in the category generated by $\mcS$?
\end{exercise}

\begin{exercise}
Suppose that $G$ is a graph and that $\mcG$ is the schema generated by $G$ with no PEDs. What is the relationship between the category generated by $\mcG$ and the free category $F(G)\in\Ob(\Cat)$ as defined in Example \ref{ex:free category}?
\end{exercise}


\subsubsection{Instances on a schema $\mcC$}\label{sec:instances}\index{database!instance}\index{instance}

If schemas are like categories, what are instances? Recall that an instance $I$ on a schema $\mcS=(G,\simeq)$ assigns to each vertex $v$ in $G$ a set of rows say $I(v)\in\Ob(\Set)$. And to every arrow $a\taking v\to v'$ in $G$ the instance assigns a function $I(a)\taking I(v)\to I(v')$. The rule is that given two equivalent paths, their compositions must give the same function. Concisely, an instance is a functor $I\taking\mcS\to\Set$. 

\begin{example}

We have now seen that a monoid is just a category $\mcM$ with one object and that a monoid action is a functor $\mcM\to\Set$. Under our understanding of database schemas as categories, $\mcM$ is a schema and so an action becomes an instance of that schema. The monoid action table from Example {ex:action table} was simply a manifestation of the database instance according to the Rules \ref{rules:schema to tables}.

\end{example}

\begin{exercise}
In Section \ref{sec:graphs as functors} we discuss how each graph is a functor $\GrIn\to\Set$ for the graph indexing category depicted below:
$$\GrIn:=\fbox{\GrInSchema}$$
But now we know that if a graph is a set-valued functor then we can consider $\GrIn$ as a database schema.
\sexc How many tables, and how many columns of each should there be (if unsure, consult Rules \ref{rules:schema to tables})?
\next Write out the table view of graph $G$ from Example \ref{ex:graph hom}. 
\endsexc
\end{exercise}


\subsection{Spaces}\index{space}

Category theory was invented for use in algebraic topology, and in particular to discuss natural transformations between certain functors. We will get to natural transformations more formally in Section \ref{sec:nat trans}. For now, they are ways of relating functors. In the original use, Eilenberg and Mac Lane were interested in functors that connect topological spaces (shapes like spheres, etc.) to algebraic systems (groups, etc.) 

For example, there is a functor that assigns to each space $X$ its group $\pi_1(X)$ of round-trip voyages (starting and ending at some chosen point $x\in X$), modulo some equivalence relation. There is another functor that assigns to every space its group $H_1(X,\ZZ)$ of ways to drop some (positive or negative) number of circles on $X$. These two functors are related, but they are not equal. 

There is a relationship between the functor $\pi_1$ and the functor $H_1$. For example when $X$ is the figure-$8$ space (two circles joined at a point) the group $\pi_1(X)$ is much bigger than the group $H_1(X)$. Indeed $\pi_1(X)$ includes information about the order and direction of loops traveled; whereas the group $H_1(X,\ZZ)$ includes only information about how many times one goes around each loop. However, there is a natural transformation of functors $\pi_1(-)\to H_1(-,\ZZ)$, called the Hurewicz transformation, which ``forgets" the extra information and thus yields a simplification. 

\begin{example}\label{ex:topological space}\index{space!topological}\index{topology}
Given a set $X$, recall that $\PP(X)$ denotes the set of subsets of $X$. A {\em topology} on $X$ is a choice of which subsets $U\in\PP(X)$ will be called {\em open sets}. The union of any number of open sets must be considered to be an open set, and the intersection of any finite number of open sets must be considered open. One could say succinctly that a topology on $X$ is a sub-order $\Op(X)\ss\PP(X)$ that is closed under taking finite meets and infinite joins.

A {\em topological space}\index{topological space} is a pair $(X,\Op(X))$, where $X$ is a set and $\Op(X)$ is a topology on $X$. The elements of the set $X$ are called {\em points}. A {\em morphism of topological spaces} (also called a {\em continuous map}) is a function $f\taking X\to Y$ such that for every $V\in\Op(Y)$ the preimage $f^\m1(V)\in\PP(X)$ is actually in $\Op(X)$. That is, such that there exists a dashed arrow making the diagram below commute:
$$\xymatrix{\Op(Y)\ar@{-->}[r]\ar[d]&\Op(X)\ar[d]\\\PP(Y)\ar[r]_{f^\m1}&\PP(X).}$$
The {\em category of topological spaces}, denoted $\Top$, is the category having objects and morphisms as above.\index{a category!$\Top$}

\end{example}

\begin{exercise}\label{exc:points and opens in Top}~
\sexc Explain how ``looking at points" gives a functor $\Top\to\Set$.\index{a functor!$\Top\to\Set$}
\next Does ``looking at open sets" give a functor $\Top\to\PrO$?\index{a functor!$\Top\to\PrO\op$}
\endsexc
\end{exercise}

\begin{example}[Continuous dynamical systems]\label{ex:continuous dynamical systems}\index{dynamical system!continuous}

The set $\RR$ can be given a topology in a standard way.\footnote{The topology is given by saying that $U\ss\RR$ is open iff for every $x\in U$ there exists $\epsilon>0$ such that $\{y\in \RR\| |y-x|<\epsilon\}\ss U\}$. One says, ``$U\ss\RR$ is open if every point in $U$ has an epsilon-neighborhood fully contained in $U$".} But $(\RR,0,+)$ is also a monoid. Moreover, for every $x\in\RR$ the monoid operation $+\taking\RR\times\RR\to\RR$ is continuous.
\footnote{The topology on $\RR\times\RR$ is similar; a subset $U\ss\RR\times\RR$ is open if every point $x\in U$ has an epsilon-neighborhood (a disk around $x$ of some positive radius) fully contained in $U$.}
So we say that $\mcR:=(\RR,0,+)$ is a {\em topological monoid}.

Recall from Section \ref{sec:monoids as cats} that a monoid action is a functor $\mcM\to\Set$, where $\mcM$ is a monoid. Instead imagine a functor $a\taking\mcR\to\Top$? Since $\mcR$ is a category with one object, this amounts to an object $X\in\Ob(\Top)$, a space. And to every real number $t\in\RR$ we obtain a continuous map $a(t)\taking X\to X$. If we consider $X$ as the set of states of some system and $\RR$ as the time line, we have captured what is called a {\em continuous dynamical system}.

\end{example}

\begin{example}\index{a category!$\Vect$}\index{vector space}

Recall (see \cite{Axl}) that a {\em real vector space} is a set $X$, elements of which are called {\em vectors}, which is closed under addition and scalar multiplication. For example $\RR^3$ is a vector space. A {\em linear transformation from $X$ to $Y$} is a function $f\taking X\to Y$ that appropriately preserves addition and scalar multiplication. The {\em category of real vector spaces}, denoted $\Vect_\RR$, has as objects the real vector spaces and as morphisms the linear transformations.

There is a functor $\Vect_\RR\to\Grp$\index{a functor!$\Vect_\RR\to\Grp$} sending a vector space to its underlying group of vectors, where the group operation is addition of vectors and the group identity is the 0-vector. 

\end{example}

\begin{exercise}
Every vector space has vector subspaces, ordered by inclusion (the origin is inside of any line which is inside of certain planes, etc., and all are inside of the whole space $V$). If you know about this topic, answer the following questions.
\sexc Does a linear transformation $V\to V'$ induce a morphism of these orders? In other words, is there a functor $\Vect_\RR\to\PrO$?\index{a functor!$\Vect_\RR\to\PrO$}
\next Would you guess that there is a nice functor $\Vect_\RR\to\Top$?\index{a functor!$\Vect_\RR\to\Top$} By a ``nice functor" I mean one that doesn't make people roll their eyes (for example, there is a functor $\Vect_\RR\to\Top$ that sends every vector space to the empty space, and that's not really a ``nice" one. If someone asked for a functor $\Vect_\RR\to\Top$ for their birthday, this functor would make them sad. We're looking for a functor $\Vect_\RR\to\Top$ that would make them happy.)
\endsexc
\end{exercise}


\subsubsection{Groupoids}\label{sec:groupoid}\index{groupoid}

Groupoids are like groups except a groupoid can have more than one object. 

\begin{definition}

A {\em groupoid} is a category $\mcC$ such that every morphism is an isomorphism. If $\mcC$ and $\mcD$ are groupoids, a {\em morphism of groupoids}, denoted $F\taking\mcC\to\mcD$, is simply a functor. The category of groupoids is denoted $\Grpd$.\index{a category!$\Grpd$}

\end{definition}

\begin{example}

There is a functor $\Grpd\to\Cat$\index{a functor!$\Grpd\to\Cat$}, sending a groupoid to its underlying category. There is also a functor $\Grp\to\Grpd$\index{a functor!$\Grp\to\Grpd$} sending a group to ``itself as a groupoid with one object." 

\end{example}

\begin{application}\index{groupoid!of material states}

Let $M$ be a material in some original state $s_0$.\footnote{This example may be a bit crude, in accordance with the crudeness of my understanding of materials science.} Construct a category $\mcS_M$ whose objects are the states of $M$, e.g. by pulling on $M$ in different ways, or by heating it up, etc. we obtain such states. Include a morphism from state $s$ to state $s'$ if there exists a physical transformation from $s$ to $s'$. Physical transformations can be performed one after another, so we can compose morphisms, and perhaps we can agree this composition is associative. Note that there exists a morphism $i_s\taking s_0\to s$ for any $s$. Note also that this category is a preorder because there either exists a physical transformation or there does not. 
\footnote{Someone may choose to beef this category up to include the set of physical processes between states as the hom-set. This gives a category that is not a preorder. But there would be a functor from their category to ours.}

The \href{http://en.wikipedia.org/wiki/Elastic_modulus}{\text elastic deformation region} of the material is the set of states $s$ such that there exists a morphism $s\to s_0$, because any such morphism will be the inverse of $i_s\taking s_0\to s$. A transformation is irreversible if there is no transformation back. If $s_1$ is not in the elastic deformation region, we can (inventing a term) still talk about the region that is ``elastically-equivalent" to $s_1$. It is all the objects in $\mcS_M$ that are isomorphic to $s_1$. If we consider only elastic equivalences, we are looking at a groupoid sitting inside the larger category $\mcS_M$.

\end{application}

\begin{example}

Alan Weinstein \href{http://www.ams.org/notices/199607/weinstein.pdf}{\text explains} groupoids in terms of tiling patterns on a bathroom floor, see \cite{WeA}.

\end{example}

\begin{example}\label{ex:fundamental groupoid}\index{groupoid!fundamental}

Let $I=\{x\in\RR\|0\leq x\leq 1\}$ denote the unit interval. It can be given a topology in a standard way, as a subset of $\RR$ (see Example \ref{ex:continuous dynamical systems})

For any space $X$, a {\em path in $X$} is a continuous map $I\to X$. Two paths are called {\em homotopic} if one can be continuously deformed to the other, where the deformation occurs completely within $X$.
\footnote{
Let $I^2=\{(x,y)\in\RR^2\|0\leq x\leq 1 \tn{ and } 0\leq y\leq 1\}$ denote the square. There are two inclusions $i_0,i_1\taking I\to S$ that put the interval inside the square at the left and right sides. Two paths $f_0,f_1\taking I\to X$ are homotopic if there exists a continuous map $f\taking I\times I\to X$ such that $f_0=f\circ i_0$ and $f_1=f\circ i_1$, 
$$\xymatrix{I\ar@<-.5ex>[r]_{i_1}\ar@<.5ex>[r]^{i_0}&I\times I\ar[r]^f&X}$$
} 
One can prove that being homotopic is an equivalence relation on paths. 

Paths in $X$ can be composed, one after the other, and the composition is associative (up to homotopy). Moreover, for any point $x\in X$ there is a trivial path (that stays at $x$). Finally every path is invertible (by traversing it backwards) up to homotopy. 

This all means that to any space $X\in\Ob(\Top)$ we can associate a groupoid, called the {\em fundamental groupoid of $X$} and denoted $\Pi_1(X)\in\Ob(\Grpd)$. The objects of $\Pi_1(X)$ are the points of $X$; the morphisms in $\Pi_1(X)$ are the paths in $X$ (up to homotopy). A continuous map $f\taking X\to Y$ can be composed with any path $I\to X$ to give a path $I\to Y$ and this preserves homotopy. So in fact $\Pi_1\taking\Top\to\Grpd$\index{a functor!$\Pi_1\taking\Top\to\Grpd$} is a functor.

\end{example}

\begin{exercise}
Let $T$ denote the surface of a donut, i.e. a torus. Choose two points $p,q\in T$. Since $\Pi_1(T)$ is a groupoid, it is also a category. What would the hom-set $\Hom_{\Pi_1(T)}(p,q)$ represent?
\end{exercise}

\begin{exercise}\index{vector field}
Let $U\ss\RR^2$ be an open subset of the plane, and let $F$ be an \href{http://en.wikipedia.org/wiki/Conservative_vector_field#Irrotational_vector_fields}{\text irrotational vector field} on $U$ (i.e. one with $\tn{curl}(F)=0$). Following Exercise \ref{exc:vector field 1}, we have a category $\mcC_F$. If two curves $C,C'$ in $U$ are homotopic then they have the same line integral, $\int_CF=\int_{C'}F$.

We also have a category $\Pi_1U$, given by the fundamental groupoid, as in Example \ref{ex:fundamental groupoid}. Both categories have the same objects, $\Ob(\mcC_F)=|U|=\Ob(\Pi_1U)$, the set of points in $U$.
\sexc Is there a functor $\mcC_F\to\Pi_1U$ or a functor $\Pi_1U\to\mcC_F$ that is identity on the underlying objects? 
\next What is $\mcC_F$ if $F$ is a conservative vector field?\index{vector field!conservative}
\endsexc
\end{exercise}

\begin{exercise}
Consider the set $A$ of all (well-formed) arithmetic expressions in the symbols $\{0,\ldots,9,+,-,*,(,)\}$. For example, here are some elements of $A$: $$52,\hsp 52-7,\hsp 50+3*(6-2).$$ We can say that an equivalence between two arithmetic expressions is a justification that they give the same ``final answer", e.g. $52+60$ is equivalent to $10*(5+6)+(2+0)$, which is equivalent to $10*11+2$. I've basically described a groupoid. What are its objects and what are its morphisms?
\end{exercise}


\subsection{Logic, set theory, and computer science}


\subsubsection{The category of propositions}\label{sec:propositions}\index{a category!$\Prop$}

Given a domain of discourse, a logical proposition is a statement that is evalued in any model of that domain as either true or ``not always true". For example, in the domain of real numbers we might have the proposition 
$$\tn{For all real numbers }x\in\RR\tn{ there exists a real number } y\in\RR\tn{ such that }y>3x.$$
We say that one logical proposition $P$ {\em implies} another proposition $Q$, denoted $P\Rightarrow Q$ if, for every model in which $P$ is true, so is $Q$. There is a category $\Prop$ whose objects are logical propositions and whose morphisms are proofs that one statement implies another. Crudely, one might say that {\em $B$ holds at least as often as $A$} if there is a morphism $A\to B$ (meaning whenever $A$ holds, so does $B$). So the proposition ``$x\neq x$" holds very seldom and ``$x=x$" always very often.

\begin{example}
We can repeat this idea for non-mathematical statements. Take all possible statements that are verifiable by experiment as objects of a category. Given two such statements, it may be that one implies the other (e.g. ``if the speed of light is fixed then there are relativistic effects"). Every statement implies itself (identity) and implication is transitive, so we have a category. 
\end{example}

Let's consider differences in proofs to be irrelevant, so the category $\Prop$ becomes a preorder: either $A$ implies $B$ or it does not. Then it makes sense to discuss meets and joins. It turns out that meets are ``and's" and joins are ``or's". That is, given propositions $A,B$ the meet $A\wedge B$ is defined to be a proposition that holds as often as possible subject to the constraint that it implies both $A$ and $B$; the proposition ``$A$ holds and $B$ holds" fits the bill. Similarly, the join $A\vee B$ is given by ``$A$ holds or $B$ holds".

\begin{exercise}\label{exc:juris 1}
Consider the set of possible laws (most likely an infinite set) that can be dictated to hold throughout a jurisdiction. Consider each law as a proposition (``such and such is (dictated to be) the case"), i.e as an object of our preorder $\Prop$. Given a jurisdiction $V$, and a set of laws $\{\ell_1,\ell_2,\ldots,\ell_n\}$ that are dictated to hold throughout $V$, we take their meet $L(V):=\ell_1\wedge\ell_2\wedge\cdots\wedge\ell_n$ and consider it to be the single law of the land $V$. Suppose that $V$ is a jurisdiction and $U$ is a sub-jurisdiction (e.g. $U$ is a county and $V$ is a state); write $U\leq V$. Then clearly any law dictated by the large jurisdiction (the state) must also hold throughout the small jurisdiction (the county).
\sexc What is the relation in $\Prop$ between $L(U)$ and $L(V)$?
\next Consider the preorder $J$ on jurisdictions given by $\leq$ as above. Is ``the law of the land" a morphism of preorders $J\to\Prop$? To be a bit more high-brow, considering both $J$ and $\Prop$ to be categories (by Proposition \ref{prop:preorders to cats}), we have a function $L\taking\Ob(J)\to\Ob(\Prop)$; this question is asking whether $L$ extends to a functor $J\to\Prop$.\footnote{Hint: Exercises \ref{exc:juris 1} and \ref{exc:juris 2} will ask similar yes/no questions and at least one of these is correctly answered ``no".}
\endsexc
\end{exercise}

\begin{exercise}\label{exc:juris 2}
Take again the preorder $J$ of jurisdictions from Exercise \ref{exc:juris 1} and the idea that laws are propositions. But this time, let $R(V)$ be the set of all possible laws (not just those dictated to hold) that are in actuality being respected, i.e. followed, by all people in $V$. This assigns to each jurisdiction a set.
\sexc Since preorders can be considered categories, does our ``the set of respected laws" function $R\taking\Ob(J)\to\Ob(\Set)$ extend to a functor $J\to\Set$? 
\next What about if instead we take the meet of all these laws and assign to each jurisdiction the maximal law respected throughout. Does this assignment $\Ob(J)\to\Ob(\Prop)$ extend to a functor $J\to\Prop$?$~^{\arabic{footnote}}$
\endsexc
\end{exercise}


\subsubsection{A categorical characterization of $\Set$}\index{set!Lawvere's description of}
The category $\Set$ of sets is fundamental in mathematics, but instead of thinking of it as something given or somehow special, it can be shown to merely be a category with certain properties, each of which can be phrased purely categorically. This was shown by Lawvere \cite{Law}. A very readable account is given in \cite{Le2}.


\subsubsection{Categories in computer science}\index{CCCs}\index{category!cartesian closed}

Computer science makes heavy use of trees, graphs, orders, lists, and monoids. We have seen that all of these are naturally viewed in the context of category theory, though it seems that such facts are rarely mentioned explicitly in computer science textbooks. However, categories are also used explicitly in the theory of programming languages (PL). Researchers in that field attempt to understand the connection between what programs are supposed to do (their denotation) and what they actually cause to occur (their operation). Category theory provides a useful mathematical formalism in which to study this.

The kind of category most often considered by a PL researcher is what is known as a {\em Cartesian closed category} or {\em CCC}, which means a category $\mcT$ that has products (like $A\times B$ in $\Set$) and exponential objects (like $B^A$ in $\Set$). $\Set$ is an example of a CCC, but there are others that are more appropriate for actual computation. The objects in a PL person's CCC represent the {\em types} of the language, types such as {\tt integers, strings, floats}. The morphisms represent computable functions, e.g. {\tt length: strings}$\too${\tt integers}. The products allow one to discuss pairs $(a,b)$ where $a$ is of one type and $b$ is of another type. Exponential objects allow one to consider computable functions as things that can be input to a function (e.g. given any computable function {\tt floats}$\to${\tt integers} one can consistently multiply its results by 2 and get a new computable function {\tt floats}$\to${\tt integers}. We will be getting to products in Section \ref{def:products in a cat} and exponential objects in Section \ref{sec:vert and hor}. 

But category theory did not only offer a language for thinking about programs, it offered an unexpected tool called monads. The above CCC model for types allows researchers only to discuss functions, leading to the notion of functional programming languages; however, not all things that a computer does are functions. For example, reading input and output, changing internal state, etc. are operations that can be performed that ruin the functional-ness of programs. Monads were found in 19?? by Moggi \cite{Mog} to provide a powerful abstraction that opens the doors to such non-functional operations without forcing the developer to leave the category-theoretic garden of eden. We will discuss monads in Section \ref{sec:monads}.

We have also seen in Section \ref{sec:schemas and cats intro} that databases are well captured by the language of categories. We will formalize this in Section \ref{sec:cat equiv sch}. Throughout the remainder of this book we will continue to use databases to bring clarity to concepts within standard category theory. 
 

\subsection{Categories applied in science} 

Categories are being used throughout mathematics to relate various subjects, as well as to draw out the essential structures within these subjects. For example, there is an active research for ``categorifying" classical theories like that of knots, links, and braids \cite{Kho}. It is similarly applied in science, to clarify complex subjects. Here are some very brief descriptions of scientific disciplines to which category theory is applied.

Quantum field theory is was categorified by Atiyah \cite{Ati} in the late 1980's, with much success (at least in producing interesting mathematics). In this domain, one takes a category in which an object is a reasonable space, called a manifold, and a morphism is a manifold connecting two manifolds, like a cylinder connects two circles. Such connecting manifolds are called cobordisms, and as such people refer to the category as $\Cob$. Topological quantum field theory is the study of functors $\Cob\to\Vect$ that assign a vector space to each manifold and a linear transformation of vector spaces to each cobordism. 

Information theory \index{information theory}
\footnote{To me, the subject of ``information theory" is badly named. That discipline is devoted to finding ideal compression schemes for messages to be sent quickly and accurately across a noisy channel. It deliberately does not pay any attention to what the messages mean. To my mind this should be called compression theory or redundancy theory. Information is inherently meaningful---that is its purpose---any theory that is unconcerned with the meaning is not really studying information per se. The people who decide on speed limits for roads and highways may care about human health, but a study limited to deciding ideal speed limits should not be called ``human health theory".} 
is the study of how to ideally compress messages so that they can be sent quickly and accurately across a noisy channel.\footnote{Despite what was said above, Information theory has been extremely important in a diverse array of fields, including computer science \cite{MacK}, but also in neuroscience \cite{Bar}, \cite{Lin} and physics \cite{Eve}. I'm not trying to denigrate the field; I am only frustrated with its name.} Invented in 1948 by Claude Shannon, its main quantity of interest is the number of bits necessary to encode a piece of information. For example, the amount of information in an English sentence can be greatly reduced. The fact that {\tt t}'s are often followed by {\tt h}'s, or that {\tt e}'s are much more common than {\tt z}'s, implies that letters are not being used as efficiently as possible. The amount of bits necessary to encode a message is called its {\em entropy} and has been linked to the commonly used notion of the same name in physics. 

In \cite{BFL}, Baez, Fritz, and Leinster show that entropy can be captured quite cleanly using category theory. They make a category {\tt FinProb} whose objects are finite sets equipped with a probability measure, and whose morphisms are probability preserving functions. They characterize {\em information loss} as a way to assign numbers to such morphisms, subject to certain explicit constraints. They then show that the entropy of an object in {\tt FinProb} is the amount of information lost under the unique map to the singleton set $\singleton$. This approach explicates (by way of the explicit constraints for information loss functions) the essential idea of Shannon's information theory, allowing it to be generalized to categories other than {\tt FinProb}. Thus Baez and Leinster effectively {\em categorified} information theory.

Robert Rosen proposed in the 1970s that category theory could play a major role in biology. That story is only now starting to be fleshed out. There is a categorical account of evolution and memory, called {\em Memory Evolutive Systems} \cite{EV}. There is also a paper \cite{BP2} by Brown and Porter with applications to neuroscience.


\section{Natural transformations}\label{sec:nat trans}

In this section we conclude our discussion of the {\bf Big 3}, by defining natural transformations. Category theory was originally invented to discuss natural transformations. These were sufficiently conceptually challenging that they required formalization and thus the invention of category theory. If we think of categories as domains (of discourse, interaction, comparability, etc.) and of functors as transformations between different domains, the natural transformations compare different transformations.

Natural transformations can seem a bit abstruse at first, but hopefully some examples and exercises will help.


\subsection{Definition and examples}

Let's begin with an example. There is a functor $\List\taking\Set\to\Set$, which sends a set $X$ to the set $\List(X)$ consisting of all lists whose entries are elements of $X$. Given a morphism $f\taking X\to Y$, we can transform a list with entries in $X$ into a list with entries in $Y$ by applying $f$ to each (this was worked out in Exercise \ref{exc:list as functor}).\index{a functor!$\List\taking\Set\to\Set$}. 

It may seem a strange thing to contemplate, but there is also a functor $\List\circ\List\taking\Set\to\Set$ that sends a set $X$ to the set of lists of lists in $X$. If $X=\{a,b,c\}$ then $\List\circ\List(X)$ contains elements like $\big[[a,b],[a,c,a,b,c],[c]\big]$ and $\big[[\;]\big]$ and $\big[[a],[\;],[a,a,a]\big]$. We can {\em naturally transform} a list of lists into a list by concatenation. In other words, for any set $X$ there is a function $\mu_X\taking\List\circ\List(X)\to\List(X)$ which sends our lists above to $[a,b,a,c,a,b,c,c]$ and $[\;]$ and $[a,a,a,a]$, respectively. In fact, even if we use a function $f\taking X\to Y$ to convert a list of $X$'s into a list of $Y$'s (or a list of lists of $X$'s into a list of lists of $Y$'s), the concatenation ``works right". Take a deep breath for the precise statement couched as a slogan.

\begin{slogan}
Naturality works like this: Using a function $f\taking X\to Y$ to convert a list of lists of $X$'s into a list of list of $Y$'s and then concatenating to get a simple list of $Y$'s {\bf does the same thing as} first concatenating our list of lists of $X$'s into a simple list of $X$'s and then using our function $f$ to convert it into a list of $Y$'s.
\end{slogan}

Let's make this concrete. Let $X=\{a,b,c\}$, let $Y=\{1,2,3\}$, and let $f\taking X\to Y$ assign $f(a)=1, f(b)=1, f(c)=2$. Our naturality condition says the following for any list of lists of $X$'s, in particular for $\big[[a,b],[a,c,a,b,c],[c]\big]$:
$$\xymatrix@=40pt{
\big[[a,b],[a,c,a,b,c],[c]\big]\ar@{|->}[r]^-{\mu_X}\ar@{|->}[d]_{\List\circ\List(f)}&[a,b,a,c,a,b,c,c]\ar@{|->}[d]^{\List(f)}\\
\big[[1,1],[1,2,1,1,2],[2]\big]\ar@{|->}[r]_-{\mu_Y}&[1,1,1,2,1,1,2,2]
}
$$

Keep these $\mu_X$ in mind in the following definition---they serve as the ``components" of a natural transformation $\List\circ\List\to\List$ of functors $\mcC\to\mcD$, where $\mcC=\mcD=\Set$.

\begin{definition}\label{def:natural transformation}\index{natural transformation}

Let $\mcC$ and $\mcD$ be categories and let $F\taking\mcC\to\mcD$ and $G\taking\mcC\to\mcD$ be functors. A {\em natural transformation $\alpha$ from $F$ to $G$}, denoted $\alpha\taking F\to G$, is defined as follows: one announces some constituents (A. components) and asserts that they conform to some laws (1. naturality squares). Specifically, one announces
\begin{enumerate}[\hsp A.]
\item for each object $c\in\Ob(\mcC)$ a morphism $\alpha_c\taking F(c)\to G(c)$ in $\mcD$, called {\em the $c$-component of $\alpha$}.\index{component}
\end{enumerate}
One asserts that the following law holds:
\begin{enumerate}[\hsp 1.]
\item For every morphism $h\taking c\to c'$ in $\mcC$, the following square, called the {\em naturality square for $h$}, must commute:
\begin{align}\label{dia:naturality square}
\xymatrix{F(c)\ar@{}[dr]|{\checkmark}\ar[d]_{F(h)}\ar[r]^{\alpha_c}&G(c)\ar[d]^{G(h)}\\F(c')\ar[r]_{\alpha_{c'}}&G(c')}
\end{align}
\end{enumerate}

\end{definition}

\begin{example}

Consider the categories $\mcC\iso[1]$ and $\mcD\iso[2]$ drawn below:
$$\mcC:=\fbox{\xymatrix{\LMO{0}\ar[r]^p&\LMO{1}
}}
\hspace{.5in}
\mcD:=\fbox{\xymatrix{\LMO{A}\ar[r]^f&\LMO{B}\ar[r]^g&\LMO{C}.
}}
$$
Consider the functors $F,G\taking[1]\to[2]$ where $F(0)=A$, $F(1)=B$, $G(0)=A$, and $G(1)=C$. The orange dots and arrows in the picture below represent the image of $\mcC$ under $F$ and $G$.

\begin{center}
\includegraphics[height=4in]{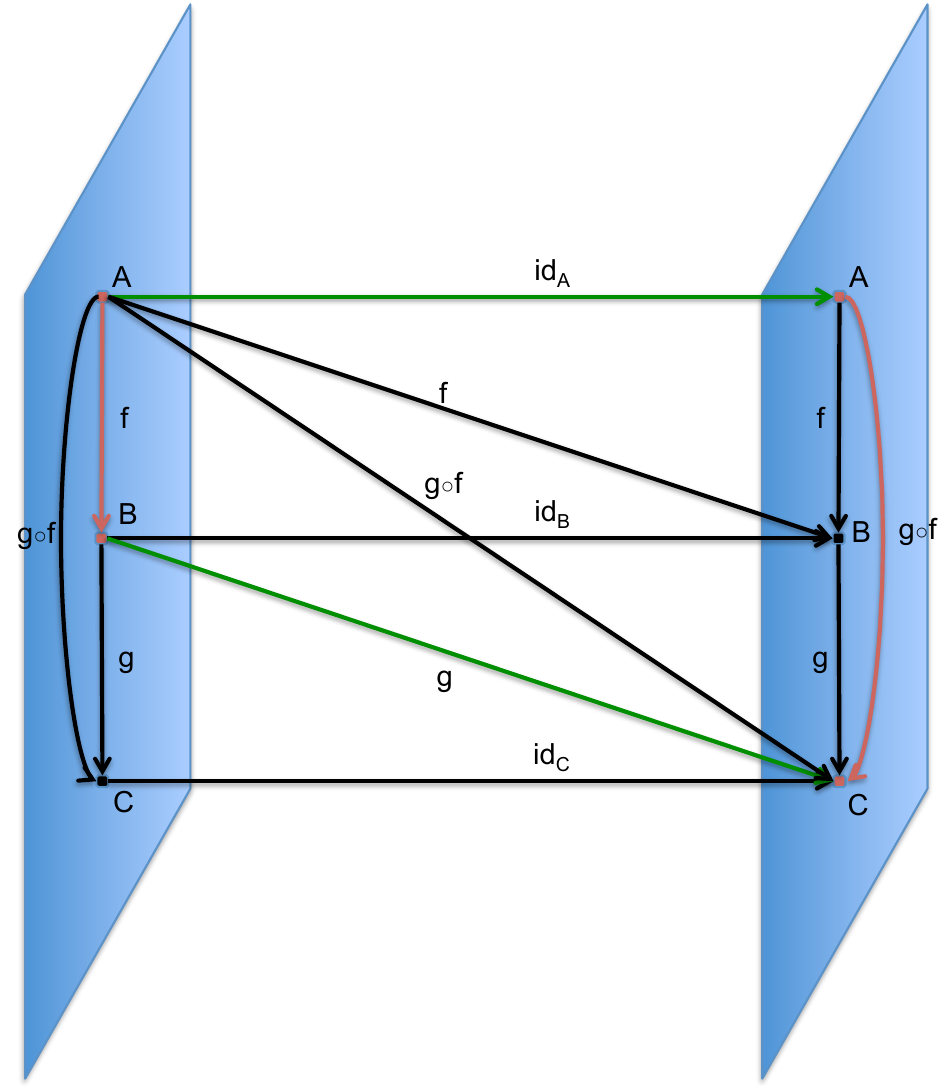}
\end{center}

It turns out that there is only one possible natural transformation $F\to G$; we call it $\alpha$ and explore its naturality square. We have drawn the components of $\alpha\taking F\to G$ in green. These components are $\alpha_0=\id_A\taking F(0)\to G(0)$ and $\alpha_1=g\taking F(1)\to G(1)$. The naturality square for $p\taking 0\to 1$ is written twice below, once with notation following that in (\ref{dia:naturality square}) and once in local notation.
$$
\xymatrix{
F(0)\ar[r]^{\alpha_0}\ar[d]_{F(p)}&G(0)\ar[d]^{G(p)}\\
F(1)\ar[r]_{\alpha_1}&G(1)}
\hspace{.6in}
\xymatrix{
A\ar[r]^{\id_A}\ar[d]_{f}&A\ar[d]^{g\circ f}\\
B\ar[r]_{g}&C
}
$$
It is clear that this diagram commutes, so our components $\alpha_0$ and $\alpha_1$ satisfy the law of Definition \ref{def:natural transformation}, making $\alpha$ a natural transformation.

\end{example}

\begin{lemma}\label{lemma:generators for nattrans}

Let $\mcC$ and $\mcD$ be categories, let $F,G\taking\mcC\to\mcD$ be functors, and for every object $c\in\Ob(\mcC)$, let $\alpha_c\taking F(c)\to G(c)$ be a morphism in $\mcD$. Suppose given a path $c_0\To{f_1}c_1\To{f_2}\cdots\To{f_n} c_n$ such that the naturality square 
$$
\xymatrix{F(c_{i-1})\ar[d]_{F(f_i)}\ar[r]^{\alpha_{c_{i-1}}}&G(c_{i-1})\ar[d]^{G(f_i)}\\F(c_i)\ar[r]_{\alpha_{c_i}}&G(c_i)
}
$$
commutes for each $1\leq i\leq n$. Then the naturality square for the composite $p:=f_n\circ\cdots\circ f_2\circ f_1\taking c_0\to c_n$ 
$$\xymatrix{F(c_0)\ar[r]^{\alpha_{c_0}}\ar[d]_{F(p)}&G(c_0)\ar[d]^{G(p)}\\F(c_n)\ar[r]_{\alpha_{c_n}}&G(c_n)}
$$
also commutes. In particular, the naturality square commutes for every identity morphism $\id_c$.

\end{lemma}

\begin{proof}

When $n=0$ we have a path of length 0 starting at each $c\in\Ob(\mcC)$. It vacuously satisfies the condition, so we need to see that its naturality square 
$$\xymatrix{F(c)\ar[r]^{\alpha_c}\ar[d]_{F(\id_c)}&G(c)\ar[d]^{G(\id_c)}\\F(c)\ar[r]_{\alpha_c}&G(c)}
$$
commutes. But this is clear because functors preserve identities. 

The rest of the proof follows by induction on $n$. Suppose $q=f_{n-1}\circ\cdots\circ f_2\circ f_1\taking c_0\to c_{n-1}$ and $p=f_n\circ q$ and that the naturality squares for $q$ and for $f_n$ commute; we need only show that the naturality square for $p$ commutes. That is, we assume the two small squares commute below; but it follows that the large rectangle does too, completing the proof.

$$
\xymatrix{
F(c_0)\ar[r]^{\alpha_{c_0}}\ar[d]_{F(q)}&G(c_0)\ar[d]^{G(q)}\\
F(c_{n-1})\ar[r]^{\alpha_{c_{n-1}}}\ar[d]_{F(f_n)}&G(c_{n-1})\ar[d]^{G(f_n)}\\
F(c_n)\ar[r]^{\alpha_{c_n}}&G(c_n)
}
$$

\end{proof}

\begin{example}\label{ex:nattrans [1]}

Let $\color{red}{\mcC}=\color{blue}{\mcD}=\color{black}{[1]}$ be the linear order of length 1, thought of as a category (by Proposition \ref{prop:preorders to cats}). There are three functors $\mcC\to\mcD$, which we can write as $(0,0), (0,1),$ and $(1,1)$; these are depicted left to right below.
$$\xymatrix{
\LMO{\color{red}{0}}\ar@{|->}[r]\ar@[red][d]_f&\LMO{\color{blue}{0}}\ar@[blue][d]^f&&\LMO{\color{red}{0}}\ar@{|->}[r]\ar@[red][d]_f&\LMO{\color{blue}{0}}\ar@[blue][d]^f&&\LMO{\color{red}{0}}\ar@{|->}[dr]\ar@[red][d]_f&\LMO{\color{blue}{0}}\ar@[blue][d]^f\\
\LMO{\color{red}{1}}\ar@{|->}[ur]&\LMO{\color{blue}{1}}&&\LMO{\color{red}{1}}\ar@{|->}[r]&\LMO{\color{blue}{1}}&&\LMO{\color{red}{1}}\ar@{|->}[r]&\LMO{\color{blue}{1}}
}
$$
These are just functors so far. What are the natural transformations say $\alpha\taking (0,0)\to(0,1)$? To specify a natural transformation, we must specify a component for each object in $\mcC$. In our case $\alpha_0\taking 0\to 0$ and $\alpha_1\taking 0\to 1$. There is only one possible choice: $\alpha_0=\id_0$ and $\alpha_1=f$. Now that we have chosen components we need to check the naturality squares. 

There are three morphisms in $\mcC$, namely $\id_0, f, \id_1$. By Lemma \ref{lemma:generators for nattrans}, we need only check the naturality square for $f$. We write it twice below, once in the abstract notation and once in concrete notation:
$$
\xymatrix{
F(0)\ar[r]^{\alpha_0}\ar[d]_{F(f)}&G(0)\ar[d]^{G(f)}\\
F(1)\ar[r]_{\alpha_1}&G(1)
}\hspace{.5in}
\xymatrix{
0\ar[r]^{\id_0}\ar[d]_{\id_0}&0\ar[d]^{f}\\
0\ar[r]_{f}&1
}
$$
This commutes, so $\alpha$ is indeed a natural transformation.

\end{example}

\begin{exercise}
With notation as in Example \ref{ex:nattrans [1]},
\sexc how many natural transformations are there $(0,0)\to (1,1)$?
\next how many natural transformations are there $(0,0)\to (0,0)$?
\next how many natural transformations are there $(0,1)\to (0,0)$?
\next how many natural transformations are there $(0,1)\to (1,1)$?
\endsexc
\end{exercise}

\begin{exercise}
Let $\List\taking\Set\to\Set$ be the functor sending a set $X$ to the set $\List(X)$ of lists with entries in $X$. We saw above that there is a natural transformation $\List\circ\List\to\List$ given by concatenation.
\sexc If someone said ``singleton lists give a natural transformation $\sigma$ from $\id_\Set$ to $\List$", what might they mean? That is, for a set $X$, what component $\sigma_X$ might they be suggesting?
\next Do these components satisfy the necessary naturality squares for functions $f\taking X\to Y$?
\endsexc
\end{exercise}

\begin{exercise}
Let $\mcC$ and $\mcD$ be categories, and suppose that $d\in\Ob(\mcD)$ is a terminal object. Consider the functor $\{d\}^\mcC\taking\mcC\to\mcD$ that sends each object $c\in\Ob(\mcC)$ to $d$ and each morphism in $\mcC$ to the identity morphism $\id_d$ on $d$. 
\sexc For any other functor $F\taking\mcC\to\mcD$, how many natural transformations are there $F\to\{d\}^\mcC$? 
\next Let $\mcD=\Set$ and let $d=\singleton$. If $\mcC=[1]$ is the linear order of length 1, and $F\taking\mcC\to\Set$ is any functor, what does it mean to give a natural transformation $\{d\}^\mcC\to F$?
\endsexc
\end{exercise}

\begin{application}\label{app:change of fsm}\index{natural transformation!as refinement of model}

In Figure \ref{fig:fsa} we drew a \href{http://en.wikipedia.org/wiki/Finite-state_machine}{finite state machine} on alphabet $\Sigma=\{a,b\}$, and in Example \ref{ex:action table} we showed the associated action table. It will be reproduced below. Imagine this was your model for understanding the behavior of some system when acted on by commands $a$ and $b$. And suppose that a collaborator tells you that she has a more refined notion that fits with the same data. Her notion has 6 states rather than 3, but it's ``compatible". What might that mean? 

Let's call the original state machine $X$ and the new model $Y$.

\begin{center}
\parbox{1.9in}{\boxtitle{$X$:=}\fbox{
\includegraphics[height=1.5in]{FSM1}
}}
\hsp
\parbox{2.6in}{\boxtitle{$Y$:=}\fbox{
\includegraphics[height=1.7in]{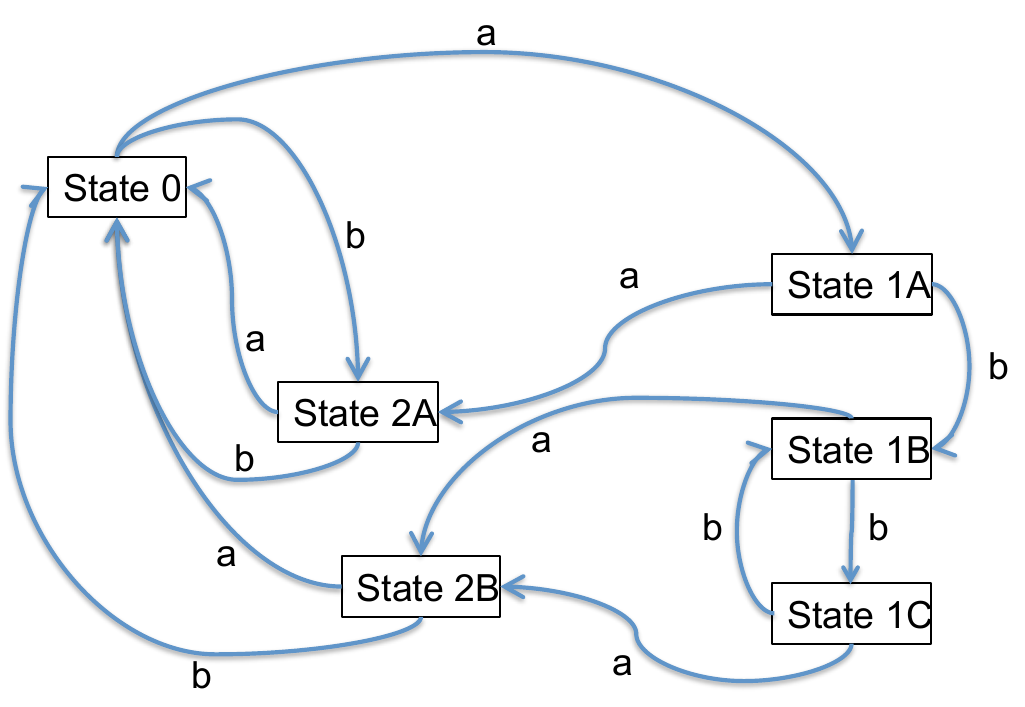}
}}
\end{center}

The action tables for these two machines are:
\begin{center}
\begin{tabular}{| l || l | l |}\bhline
\multicolumn{3}{|c|}{Original model $X$}\\\bhline
{\bf ID}&{\bf a}&{\bf b}\\\bbhline
State 0&State 1&State 2\\\hline
State 1& State 2& State 1\\\hline
State 2&State 0&State 0\\\bhline
\end{tabular}
\hspace{.5in}
\begin{tabular}{| l || l | l |}\bhline
\multicolumn{3}{|c|}{Proposed model $Y$}\\\bhline
{\bf ID}&{\bf a}&{\bf b}\\\bbhline
State 0&State 1A&State 2A\\\hline
State 1A& State 2A& State 1B\\\hline
State 1B& State 2B& State 1C\\\hline
State 1C&State 2B&State 1B\\\hline
State 2A&State 0&State 0\\\hline
State 2B&State 0&State 0\\\bhline
\end{tabular}
\end{center}

How are these models compatible? Looking at the table for $Y$, if one removes the distinction between States 1A, 1B, 1C and between States 2A and 2B, then one returns with the table for $X$. The table for $Y$ is more specific, but it is fully compatible with table $X$. The sense in which it is compatible is precisely the sense defined by there being a natural transformation.

Recall that $\mcM=(\List(\Sigma),[\;],\plpl)$ is a monoid, and that a monoid is simply a category with one object, say $\Ob(\mcM)=\{\monOb\}$ (see Section \ref{sec:mon grp pro as cat}). With $\Sigma=\{a,b\}$, the monoid $\mcM$ can be visualized as follows:
$$
\mcM=\fbox{\xymatrix{\LMO{\monOb}\ar@(ul,dl)[]_a\ar@(ur,dr)[]^b}}
$$
Recall also that a state machine on $\mcM$ is simply a functor $\mcM\to\Set$. We thus have two such functors, $X$ and $Y$. A natural transformation $\alpha\taking Y\to X$ would consist of a component $\alpha_m$ for every object $m\in\Ob(\mcM)$, such that certain diagrams commute. But $\mcM$ having only one object, we need only one function $\alpha_\monOb\taking Y(\monOb)\to X(\monOb)$, where $Y(\monOb)$ is the set of (6) states of $Y$ and $X(\monOb)$ is the set of (3) states of $X$.

The states of $Y$ have been named so as to make the function $\alpha_\monOb$ particularly easy to guess.\footnote{The function $\alpha_\monOb\taking Y(\monOb)\to X(\monOb)$ makes the following assignments: $\tn{State 0}\mapsto \tn{State 0}, \tn{State 1A}\mapsto \tn{State 1}, \tn{State 1B}\mapsto \tn{State 1}, \tn{State 1C}\mapsto \tn{State 1}, \tn{State 2A}\mapsto \tn{State 2}, \tn{State 2B}\mapsto \tn{State 2}.$} We need to check that two squares commute:
\begin{align}\label{dia:naturality squares for fsm}
\xymatrix{Y(\monOb)\ar[r]^{\alpha_\monOb}\ar[d]_{Y(a)}&X(\monOb)\ar[d]^{X(a)}\\Y(\monOb)\ar[r]_{\alpha_\monOb}&X(\monOb)
}
\hspace{.5in}
\xymatrix{Y(\monOb)\ar[r]^{\alpha_\monOb}\ar[d]_{Y(b)}&X(\monOb)\ar[d]^{X(b)}\\Y(\monOb)\ar[r]_{\alpha_\monOb}&X(\monOb)
}
\end{align}
This can only be checked by going through and making sure certain things match, as specified by (\ref{dia:naturality squares for fsm}); we spell it out in gory detail. The columns that should match are those whose entries are written in blue.

\begin{align}\label{dia:naturality for a}
\begin{tabular}{| l || l | l | l | l |}
\bhline
\multicolumn{5}{|c|}{Naturality square for $a\taking \monOb\to\monOb$}\\\bhline
{\bf $Y(\monOb)$\; [ID]}&{\bf $Y(a)$}&{\bf $\alpha_\monOb\circ Y(a)$}&{\bf $\alpha_\monOb$}&{\bf $X(a)\circ\alpha_\monOb$}\\\bbhline
State 0&State 1A&\color{blue}{State 1}&State 0&\color{blue}{State 1}\\\hline
State 1A& State 2A&\color{blue}{State 2}&State 1&\color{blue}{State 2}\\\hline
State 1B& State 2B&\color{blue}{State 2}&State 1&\color{blue}{State 2}\\\hline
State 1C&State 2B&\color{blue}{State 2}&State 1&\color{blue}{State 2}\\\hline
State 2A&State 0&\color{blue}{State 0}&State 2&\color{blue}{State 0}\\\hline
State 2B&State 0&\color{blue}{State 0}&State 2&\color{blue}{State 0}\\\bhline
\end{tabular}
\end{align}
\begin{align}\label{dia:naturality for b}
\begin{tabular}{| l || l | l | l | l |}
\bhline
\multicolumn{5}{|c|}{Naturality square for $b\taking\monOb\to\monOb$}\\\bhline
{\bf $Y(\monOb)$\; [ID]}&{\bf $Y(b)$}&{\bf $\alpha_\monOb\circ Y(b)$}&{\bf $\alpha_\monOb$}&{\bf $X(b)\circ\alpha_\monOb$}\\\bbhline
State 0&State 2A&\color{blue}{State 2}&State 0&\color{blue}{State 2}\\\hline
State 1A& State 1B&\color{blue}{State 1}&State 1&\color{blue}{State 1}\\\hline
State 1B& State 1C&\color{blue}{State 1}&State 1&\color{blue}{State 1}\\\hline
State 1C&State 1B&\color{blue}{State 1}&State 1&\color{blue}{State 1}\\\hline
State 2A&State 0&\color{blue}{State 0}&State 2&\color{blue}{State 0}\\\hline
State 2B&State 0&\color{blue}{State 0}&State 2&\color{blue}{State 0}\\\bhline
\end{tabular}
\end{align}

In reality we need to check that for {\em every} morphism in $\mcM$, such as $[a,a,b]$, a similar diagram commutes. But this holds automatically. For example (flipping the naturality square sideways for typographical reasons)
$$
\xymatrix{
Y(\monOb)\ar[r]^{Y(a)}\ar[d]^{\alpha_\monOb}&Y(\monOb)\ar[r]^{Y(a)}\ar[d]^{\alpha_\monOb}&Y(\monOb)\ar[r]^{Y(b)}\ar[d]^{\alpha_\monOb}&Y(\monOb)\ar[d]^{\alpha_\monOb}\\
X(\monOb)\ar[r]_{X(a)}&X(\monOb)\ar[r]_{X(a)}&X(\monOb)\ar[r]_{X(b)}&X(\monOb)
}
$$
Since each small square above commutes (as checked by tables \ref{dia:naturality for a} and \ref{dia:naturality for b}), the big outer rectangle commutes too.

To recap, the notion of compatibility between $Y$ and $X$ is one that can be checked and agreed upon by humans, but doing so it is left implicit, and it may be difficult to explain to an outsider what exactly was agreed to, especially in more complex situations. It is quite convenient to simply claim ``there is a natural transformation from $Y$ to $X$."

\end{application}

\begin{exercise}\label{exc:id nat trans}
Let $F\taking\mcC\to\mcD$ be a functor. Suppose someone said ``the identity on $F$ is a natural transformation from $F$ to itself." \sexc What might they mean?
\next If it is somehow true, what are the components of this natural transformation?
\endsexc
\end{exercise}

\begin{example}

Let $[1]\in\Ob(\Cat)$ be the free arrow category described in Exercise \ref{exc:[1]} and let $\mcD$ be any category. To specify a functor $F\taking[1]\to\mcD$ requires the specification of two objects, $F(v_1), F(v_2)\in\Ob(\mcD)$ and a morphism $F(e)\taking F(v_1)\to F(v_2)$ in $\mcD$. The identity and composition formulas are taken care of once that much is specified. To recap, a functor $F\taking[1]\to\mcD$ is the same thing as a morphism in $\mcD$.

Thus, choosing two functors $F,G\taking[1]\to\mcD$ is precisely the same thing as choosing two morphisms in $\mcD$. Let us call them $f\taking a_0\to a_1$ and $g\taking b_0\to b_1$, where to be clear we have $f=F(e), a_0=F(v_0), a_1=F(v_1)$ and $g=G(e), b_0=G(v_0), b_1=G(v_1)$. 

A natural transformation $\alpha\taking F\to G$ consists of two components, $h_0:=\alpha_{v_0}\taking a_0\to b_0$ and $h_1:=\alpha_{v_1}\taking a_1\to b_1$, drawn as dashed lines below:
$$\xymatrix{a_0\ar@{-->}[r]^{h_0}\ar[d]_f&b_0\ar[d]^{g}\\a_1\ar@{-->}[r]_{h_1}&b_1}$$
The condition for $\alpha$ to be a natural transformation is that the above square commutes. 

In other words, a functor $[1]\to\mcD$ is an arrow in $\mcD$ and a natural transformation between two such functors is just a commutative square in $\mcD$.

\end{example}

\begin{example}\label{ex:graph to paths}

Recall that to any graph $G$ we can associate the so-called paths-graph $\Paths(G)$, as described in Example \ref{ex:paths-graph}. This is a functor $\Paths\taking\Grph\to\Grph$.\index{a functor!$\Paths\taking\Grph\to\Grph$} There is also an identity functor $\id_{\Grph}\taking\Grph\to\Grph$. A natural transformation $\eta\taking\id_\Grph\to\Paths$ would consist of a graph homomorphism $\eta_G\taking \id_\Grph(G)\to\Paths(G)$ for every graph $G$. But $\id_\Grph(G)=G$ by definition, so we need $\eta_G\taking G\to\Paths(G)$. Recall that $\Paths(G)$ has the same vertices as $G$ and every arrow in $G$ counts as a path (of length 1). So there is an obvious graph homomorphism from $G$ to $\Paths(G)$. It is not hard to see that the necessary naturality squares commute.

\end{example}

\begin{example}\label{ex:concat paths of paths}

For any graph $G$ we can associate the paths-graph $\Paths(G)$, and nothing stops us from doing that twice to yield a new graph $\Paths(\Paths(G))$. Let's think through what a path of paths in $G$ is. It's a head-to-tail sequence of arrows in $\Paths(G)$, meaning a head-to-tail sequence of paths in $G$. These composable sequences of paths (or ``paths of paths") are the individual arrows in $\Paths(\Paths(G))$. (The vertices in $\Paths(G)$ and $\Paths(\Paths(G))$ are the same as those in $G$, and all source and target functions are as expected.)

Clearly, given such a sequence of paths in $G$, we could compose them to one big path in $G$ with the same endpoints. In other words, there is graph morphism $\mu_G\taking\Paths(\Paths(G))\to\Paths(G)$, that one might call ``concatenation". In fact, this concatenation extends to a natural transformation $$\mu\taking\Paths\circ\Paths\to\Paths$$ between functors $\Grph\to\Grph$. In Example \ref{ex:graph to paths}, we compared a graph to its paths-graph using a natural transformation $\id_{\Grph}\to\Paths$; here we are making a similar kind of comparison.

\end{example}

\begin{remark}

In Example \ref{ex:graph to paths} we saw that there is a natural transformation sending each graph into its paths-graph. There is a formal sense in which a category is nothing more than a kind of reverse mapping. That is, to specify a category is the same thing as to specify a graph $G$ together with a graph homomorphism $\Paths(G)\to G$. The formalities involve monads, which we will discuss in Section \ref{sec:monads}.

\end{remark}

\begin{exercise}
Let $X$ and $Y$ be sets, and let $f\taking X\to Y$. There is a functor $C_X\taking\Grph\to\Set$ that sends every graph to the set $X$ and sends every morphism of graphs to the identity morphism $\id_X\taking X\to X$. This functor is called {\em the constant functor at $X$}. Similarly there is a constant functor $C_Y\taking\Grph\to\Set$.
\sexc Use $f$ to construct a natural transformation $C_X\to C_Y$.
\next What are its components?
\endsexc
\end{exercise}

\begin{exercise}
For any graph $(V,A,src,tgt)$ we can extract the set of arrows or the set of vertices. Since each morphism of graphs includes a function between their arrow sets and a function between their vertex sets, we actually have functors $Ar\taking\Grph\to\Set$\index{a functor!$\Grph\to\Set$} and $V\!e\taking\Grph\to\Set$.
\sexc If someone said ``taking source vertices gives a natural transformation from $Ar$ to $V\!e$", what natural transfromation might they be referring to?
\next What are its components? 
\next If a different person, say from a totally different country, were to say ``taking target vertices also gives a natural transformation from $Ar$ to $V\!e$," would they also be correct?
\endsexc
\end{exercise}

\begin{example}[Graph homomorphisms are natural transformations]\label{ex:graph hom as NT}

As discussed above (see Diagram \ref{dia:graph index}), there is a category $\GrIn$ for which a functor $G\taking\GrIn\to\Set$ is the same thing as a graph. Namely, we have 
\begin{align*}
\GrIn:=\fbox{\GrInSchema}
\end{align*}
A natural transformation of two such functors $\alpha\taking G\to G'$ involves two components, $\alpha_{Ar}\taking G(Ar)\to G'(Ar)$ and $\alpha_{V\!e}\taking G(V\!e)\to G'(V\!e)$, and two naturality squares, one for $src$ and one for $tgt$. This is precisely the same thing as a graph homomorphism, as defined in Definition \ref{def:graph homomorphism}.

\end{example}


\subsection{Vertical and horizontal composition}\label{sec:vert and hor}

In this section we discuss two types of compositions for natural transformations. The terms vertical and horizontal are used to describe them; these terms come from the following pictures:

$$
\parbox{1in}{\xymatrix{&\ar@{}[d]|(.65){\alpha\Down}\\\mcC\ar@/^2pc/[rr]^F\ar[rr]|G\ar@/_2pc/[rr]_H&&\mcD\\&\ar@{}[u]|(.65){\beta\Down}}}
\hspace{1in}
\parbox{2in}{\xymatrix{\mcC\ar@/^1.5pc/[rr]^{F_1}\ar@{}[rr]|{\gamma_1\Down}\ar@/_1.5pc/[rr]_{G_1}&&\mcD\ar@/^1.5pc/[rr]^{F_2}\ar@{}[rr]|{\gamma_2\Down}\ar@/_1.5pc/[rr]_{G_2}&&\mcE}}
$$
We generally use $\circ$ to denote both kinds of composition, but if we want to be very clear we will differentiate as follows: $\beta\circ\alpha\taking F\to H$ for vertical composition, and $\gamma_2\diamond\gamma_1\taking F_2\circ F_1\too G_2\circ G_1$ for horizontal composition. Of course, the actual arrangement of things on a page of text does not correlate with verticality or horizontality---these are just names. We will define them more carefully below.


\subsubsection{Vertical composition of natural transformations}\index{natural transformation!vertical composition of}

The following proposition proves that functors and natural transformations (using vertical composition) form a category.

\begin{proposition}\label{prop:Fun(C,D)}

Let $\mcC$ and $\mcD$ be categories. There exists a category, called {\em the category of functors from $\mcC$ to $\mcD$} and denoted $\Fun(\mcC,\mcD)$\index{a symbol!$\Fun$}, whose objects are the functors $\mcC\to\mcD$ and whose morphisms are the natural transformations,
$$\Hom_{\Fun(\mcC,\mcD)}(F,G)=\{\alpha\taking F\to G\|\alpha\tn{ is a natural transformation}\}.$$That is, there are identity natural transformations, natural transformations can be composed, and the identity and associativity laws hold.

\end{proposition}

\begin{proof}

We showed in Exercise \ref{exc:id nat trans} that there for any functor $F\taking\mcC\to\mcD$, there is an identity natural transformation $\id_F\taking F\to F$ (its component at $c\in\Ob(\mcC)$ is $\id_{F(c)}\taking F(c)\to F(c)$). 

Given a natural transformation $\alpha\taking F\to G$ and a natural transformation $\beta\taking G\to H$, we propose for the composite $\beta\circ\alpha$ the transformation $\gamma\taking F\to H$ having components $\beta_c\circ\alpha_c$ for every $c\in\Ob(\mcC)$. To see that $\gamma$ is indeed a natural transformation, one simply puts together naturality squares for $\alpha$ and $\beta$ to get naturality squares for $\beta\circ\alpha$. 

The associativity and identity laws for $\Fun(\mcC,\mcD)$ follow from those holding for morphisms in $\mcD$.

\end{proof}

\begin{notation}
We sometimes denote the category $\Fun(\mcC,\mcD)$ by $\mcD^\mcC$. 
\end{notation}

\begin{example}

Recall from Exercise \ref{exc:Ob is a functor} that there is a functor $\Ob\taking\Cat\to\Set$\index{a functor!$\Ob\taking\Cat\to\Set$} sending a category to its set of objects. And recall from Example \ref{ex:discrete graph discrete cat} that there is a functor $Disc\taking\Set\to\Cat$\index{a functor!$Disc\taking\Set\to\Cat$} sending a set to the discrete category with that set of objects (all morphisms in $Disc(S)$ are identity morphisms). Let $P\taking\Cat\to\Cat$ be the composition $P=Disc\circ\Ob$. Then $P$ takes a category and makes a new category with the same objects but no morphisms. It's like crystal meth for categories.

Let $\id_\Cat\taking\Cat\to\Cat$ be the identity functor. There is a natural transformation $i\taking P\to\id_\Cat$. For any category $\mcC$, the component $i_\mcC\taking P(\mcC)\to\mcC$ is pretty easily understood. It is a morphism of categories, i.e. a functor. The two categories $P(\mcC)$ and $\mcC$ have the same set of objects, namely $\Ob(\mcC)$, so our functor is identity on objects; and $P(\mcC)$ has no non-identity morphisms, so nothing else needs be specified.

\end{example}

\begin{exercise}
Let $\mcC=\fbox{$\LMO{A}$}$ be the category with $\Ob(\mcC)=\{A\}$, and $\Hom_\mcC(A,A)=\{\id_A\}$. What is $\Fun(\mcC,\Set)$? In particular, characterize the objects and the morphisms.
\end{exercise}

\begin{exercise}
Let $n\in\NN$ and let $\ul{n}$ be the set with $n$ elements, considered as a discrete category.
\footnote{When we have a functor, such as $Disc\taking\Set\to\Cat$, we may sometimes say things like ``Let $S$ be a set, considered as a category" (or in general, given a functor $F\taking\mcC\to\mcD$, we may say ``consider $c\in\Ob(\mcC)$, taken as an object in $\mcD$"). What this means is that we want to take ideas and methods available in $\Cat$ and use them on our set $S$. Having our functor $Disc$ lying around, we use it to move $S$ into $\Cat$, as $Disc(S)\in\Ob(\Cat)$, upon which we can use our intended methods. However, our human minds get bogged down seeing $Disc(S)$ because it is bulky (e.g. $\Fun(Disc(\ul{3}),Disc(\ul{2}))$ is harder to read than $\Fun(\ul{3},\ul{2})$). So we abuse notation and write $S$ in place of  $Disc(S)$. To add insult to injury, we talk about $S$ as though it was still a set, e.g. discussing its elements rather than its objects. This kind of conceptual abbreviation is standard practice in mathematical discussion because it eases the mental burden for experts, but when one says ``Let $S$ be an $X$ considered as a $Y$" the other may always ask, ``How again are you considering $X$'s to be $Y$'s?" and expect a functor .}
In other words, we write $\ul{n}$ to mean what should really be called $Disc(\ul{n})$. Describe the category $\Fun(\ul{3},\ul{2})$.
\end{exercise}

\begin{exercise}
Let $\ul{1}$ denote the discrete category with one object, and let $\mcC$ be any category.
\sexc What are the objects of $\Fun(\ul{1},\mcC)$?
\next What are the morphisms of $\Fun(\ul{1},\mcC)$?
\endsexc
\end{exercise}

\begin{example}

Let $\ul{1}$ denote the discrete category with one object (also known as the trivial monoid). For any category $\mcC$, we investigate the category $\mcD:=\Fun(\mcC,\ul{1})$. Its objects are functors $\mcC\to\ul{1}$. Such a functor $F$ assigns to each object in $\mcC$ an object in $\ul{1}$ of which there is one; so there is no choice in what $F$ does on objects. And there is only one morphism in $\ul{1}$ so there is no choice in what $F$ does on morphisms. The upshot is that there is only one object in $\mcD$, let's call it $F$, in $\mcD$, so $\mcD$ is a monoid. What are its morphisms? 

A morphism $\alpha\taking F\to F$ in $\mcD$ is a natural transformation of functors. For every $c\in\Ob(\mcC)$ we need a component $\alpha_c\taking F(c)\to F(c)$, which is a morphism $1\to 1$ in $\ul{1}$. But there is only one morphism in $\ul{1}$, namely $\id_1$, so there is no choice about what these components should be: they are all $\id_1$. The necessary naturality squares commute, so $\alpha$ is indeed a natural transformation. Thus the monoid $\mcD$ is the trivial monoid; that is, $\Fun(\mcC,\ul{1})\iso\ul{1}$ for any category $\mcC$.

\end{example}

\begin{exercise}
Let $\ul{0}$ represent the discrete category on 0 objects; it has no objects and no morphisms. Let $\mcC$ be any category. What is $\Fun(\ul{0},\mcC)$?
\end{exercise}

\begin{exercise}
Let $[1]$ denote the free arrow category as in Exercise \ref{exc:[1]}, and let $\mcC$ be the graph indexing category from (\ref{dia:graph index}). Draw the underlying graph of the category $\Fun([1],\mcC)$, and then specify which pairs of paths in that graph correspond to commutative diagrams in $\Fun([1],\mcC)$.
\end{exercise}


\subsubsection{Natural isomorphisms}\index{natural isomorphism}

Let $\mcC$ and $\mcD$ be categories. We have defined a category $\Fun(\mcC,\mcD)$ whose objects are functors $\mcC\to\mcD$ and whose morphisms are natural transformations. What are the isomorphisms in this category? 

\begin{lemma}\label{lemma:natural iso}

Let $\mcC$ and $\mcD$ be categories and let $F,G\taking\mcC\to\mcD$ be functors. A natural transformation $\alpha\taking F\to G$ is an isomorphism in $\Fun(\mcC,\mcD)$ if and only if the component $\alpha_c\taking F(c)\to G(c)$ is an isomorphism for each object $c\in\Ob(\mcC)$. In this case $\alpha$ is called a {\em natural isomorphism}.

\end{lemma}

\begin{proof}

First suppose that $\alpha$ is an isomorphism with inverse $\beta\taking G\to F$, and let $\beta_c\taking G(c)\to F(c)$ denote its $c$ component. We know that $\alpha\circ\beta=\id_G$ and $\beta\circ\alpha=\id_F$. Using the definitions of composition and identity given in Proposition \ref{prop:Fun(C,D)}, this means that for every $c\in\Ob(\mcC)$ we have $\alpha_c\circ\beta_c=\id_{G(c)}$ and $\beta_c\circ\alpha_c=\id_{F(c)}$; in other words $\alpha_c$ is an isomorphism.

Second suppose that each $\alpha_c$ is an isomorphism with inverse $\beta_c\taking G(c)\to F(c)$. We need to see that these components assemble into a natural transformation; i.e. for every morphism $h\taking c\to c'$ in $\mcC$ the right-hand square 
$$
\xymatrix{F(c)\ar@{}[dr]|{\checkmark}\ar[d]_{F(h)}\ar[r]^{\alpha_c}&G(c)\ar[d]^{G(h)}\\F(c')\ar[r]_{\alpha_{c'}}&G(c')}\hspace{.5in}
\xymatrix{G(c)\ar@{}[dr]|{?}\ar[d]_{G(h)}\ar[r]^{\beta_c}&F(c)\ar[d]^{F(h)}\\G(c')\ar[r]_{\beta_{c'}}&F(c')}
$$
commutes. We know that the left-hand square commutes because $\alpha$ is a natural transformation; we have labeled each square with a ? or a $\checkmark$ accordingly. In the following diagram we want to show that the left-hand square commutes. We know that the middle square commutes.
$$
\xymatrix@=40pt{G(c)\ar@/^2pc/[rr]^{\id_{G(c)}}\ar@{}[dr]|{?}\ar[d]_{G(h)}\ar[r]^{\beta_c}&F(c)\ar@{}[dr]|{\checkmark}\ar[d]^{F(h)}\ar[r]^{\alpha_c}&G(c)\ar@{}[dr]|{?}\ar[d]^{G(h)}\ar[r]^{\beta_c}&F(c)\ar[d]^{F(h)}\\
G(c')\ar[r]_{\beta_{c'}}&F(c')\ar[r]_{\alpha_{c'}}\ar@/_2pc/[rr]_{\id_{F(c')}}&G(c')\ar[r]_{\beta_{c'}}&F(c')}
$$
To complete the proof we need only to show that $F(h)\circ\beta_c=\beta_{c'}\circ G(h)$. This can be shown by a ``diagram chase." We go through it symbolically, for demonstration.
\begin{align*}
F(h)\circ\beta_c=\beta_{c'}\circ\alpha_{c'}\circ F(h)\circ\beta_c=\beta_{c'}\circ G(h)\circ\alpha_c\circ\beta_c=\beta_{c'}\circ G(h).
\end{align*}

\end{proof}

\begin{exercise}
Recall from Application \ref{app:change of fsm} that a finite state machine on alphabet $\Sigma$ can be understood as a functor $\mcM\to\Set$, where $\mcM=\List(\Sigma)$ is the free monoid generated by $\Sigma$. In that example we also discussed how natural transformations provide a nice language for changing state machines. Describe what kinds of changes are made by natural isomorphisms.
\end{exercise}


\subsubsection{Horizontal composition of natural transformations}

\begin{example}[Whiskering]\label{ex:whiskering}\index{natural transformation!for adding functionality}

Suppose that $\mcM=\List(a,b)$ and $\mcM'=\List(m,n,p)$ are free monoids, and let $F\taking\mcM'\to\mcM$ be given by sending $[m]\mapsto[a], [n]\mapsto[b]$, and $[p]\mapsto[b,a,a]$. An application of this might be if the sequence $[b,a,a]$ was commonly used in practice and one wanted to add a new button just for that sequence.

Recall Application \ref{app:change of fsm}. Let $X\taking\mcM\to\Set$ and $Y\taking\mcM\to\Set$ be the functors, and let $\alpha\taking Y\to X$ be the natural transformation found there. We reproduce them here: 
\begin{align*}
\fbox{\includegraphics[height=.9in]{FSM1}}\hspace{.5in}
&\fbox{\includegraphics[height=1in]{FSM2}}\\\\
\tiny
\begin{tabular}{| l || l | l |}\bhline
\multicolumn{3}{|c|}{Original model $X\taking\mcM\to\Set$}\\\bhline
{\bf ID}&{\bf a}&{\bf b}\\\bbhline
State 0&State 1&State 2\\\hline
State 1& State 2& State 1\\\hline
State 2&State 0&State 0\\\bhline
\end{tabular}
\hspace{.5in}
&\tiny\begin{tabular}{| l || l | l |}\bhline
\multicolumn{3}{|c|}{Proposed model $Y\taking\mcM\to\Set$}\\\bhline
{\bf ID}&{\bf a}&{\bf b}\\\bbhline
State 0&State 1A&State 2A\\\hline
State 1A& State 2A& State 1B\\\hline
State 1B& State 2B& State 1C\\\hline
State 1C&State 2B&State 1B\\\hline
State 2A&State 0&State 0\\\hline
State 2B&State 0&State 0\\\bhline
\end{tabular}
\end{align*}

We can compose $X$ and $Y$ with $F$ as in the diagram below
$$
\xymatrix{\mcM'\ar[r]^F&\mcM\ar@/^1pc/[rr]^Y\ar@/_1pc/[rr]_X\ar@{}[rr]|{\alpha\Down}&&\Set}
$$
to get functors $Y\circ F$ and $X\circ F$, both of type $\mcM'\to\Set$. What would these be?
\footnote{The $p$-column comes from applying $b$ then $a$ then $a$, as specified above by $F$.}
\begin{center}\footnotesize
\begin{tabular}{| l || l | l | l |}\bhline
\multicolumn{4}{|c|}{$X\circ F$}\\\bhline
{\bf ID}&{\bf m}&{\bf n}&{\bf p}\\\bbhline
State 0&State 1&State 2&State 1\\\hline
State 1& State 2& State 1&State 0\\\hline
State 2&State 0&State 0&State 2\\\bhline
\end{tabular}
\hspace{.5in}
\begin{tabular}{| l || l | l | l |}\bhline
\multicolumn{4}{|c|}{$Y\circ F$}\\\bhline
{\bf ID}&{\bf m}&{\bf n}&{\bf p}\\\bbhline
State 0&State 1A&State 2A&State 1A\\\hline
State 1A& State 2A& State 1B&State 0\\\hline
State 1B& State 2B& State 1C&State 0\\\hline
State 1C&State 2B&State 1B&State 0\\\hline
State 2A&State 0&State 0&State 2A\\\hline
State 2B&State 0&State 0&State 2A\\\bhline
\end{tabular}
\end{center}

The map $\alpha$ is what sent both State 1A and State 1B in $Y$ to State 1 in $X$, and so on. We can see that ``the same $\alpha$ works now:" the $p$ column of the table respects that mapping. But $\alpha$ was a natural transformation $Y\to X$ where as we need a natural transformation $Y\circ F\to X\circ F$. This is called {\em whiskering}. It is a kind of horizontal composition of natural transformation.

\end{example}

\begin{definition}[Whiskering]\index{natural transformation!whiskering of}\label{def:whiskering}

Let $\mcB,\mcC,\mcD,$ and $\mcE$ be categories, let $G_1,G_2\taking\mcC\to\mcD$ be functors, and let $\alpha\taking G_1\to G_2$ a natural transformation. Suppose that $F\taking\mcB\to\mcC$ (respectively $H\taking\mcD\to\mcE$) is a functor, depicted below:
$$
\xymatrix{\mcB\ar[r]^F&\mcC\ar@{}[r]|{\alpha\Down}\ar@/^1pc/[r]^{G_1}\ar@/_1pc/[r]_{G_2}&\mcD}
\hspace{.7in}
\left(\tn{respectively,}\hsp\xymatrix{\mcC\ar@{}[r]|{\alpha\Down}\ar@/^1pc/[r]^{G_1}\ar@/_1pc/[r]_{G_2}&\mcD\ar[r]^H&\mcE}\right),
$$
Then the {\em pre-whiskering of $\alpha$ by $F$}, denoted $\alpha\diamond F\taking G_1\circ F\to G_2\circ F$ (respectively, the {\em post-whiskering of $\alpha$ by $H$}, denoted $H\diamond\alpha\taking H\circ G_1\to H\circ G_2$) is defined as follows.\index{a symbol!$\diamond$}

For each $b\in\Ob(\mcB)$ the component $(\alpha\diamond F)_b\taking G_1\circ F(b)\to G_2\circ F(b)$ is defined to be $\alpha_{F(b)}$. (Respectively, for each $c\in\Ob(\mcC)$ the component $(H\diamond\alpha)_c\taking H\circ G_1(c)\to H\circ G_2(c)$ is defined to be $H(\alpha_c)$.) Checking that the naturality squares (in each case) is straightforward.

\end{definition}

The rest of this section can safely be skipped; I include it only for my own sense of completeness.

\begin{definition}[Horizontal composition of natural transformations]\index{natural transformation!horizontal composition of}\label{def:horizontal comp of nt}

Let $\mcB,\mcC,$ and $\mcD$ be categories, let $F_1,F_2\taking\mcB\to\mcC$ and $G_1,G_2\taking\mcC\to\mcD$ be functors, and let $\alpha\taking F_1\to F_2$ and $\beta\taking G_1\to G_2$ be natural transformations, as depicted below:
$$
\xymatrix{\mcB\ar@{}[r]|{\alpha\Down}\ar@/^1pc/[r]^{F_1}\ar@/_1pc/[r]_{F_2}&\mcC\ar@{}[r]|{\beta\Down}\ar@/^1pc/[r]^{G_1}\ar@/_1pc/[r]_{G_2}&\mcD}
$$
By pre- and post-whiskering in one order or the other we get the following diagram
$$
\xymatrix{
G_1\circ F_1\ar[r]^{G_1\diamond\alpha}\ar[d]_{\beta\diamond F_1}&G_1\circ F_2\ar[d]^{\beta\diamond F_2}\\
G_2\circ F_1\ar[r]_{G_2\diamond\alpha}&G_2\circ F_2}
$$
It is straightforward to show that this diagram commutes, so we can take the composition to be our definition of the horizontal composition 
$$\beta\diamond\alpha\taking G_1\circ F_1\to G_2\circ F_2.$$

\end{definition}

\begin{remark}

Whiskering a natural transformation $\alpha$ with a functor $F$ is the same thing as horizontally composing $\alpha$ with the identity natural transformation $\id_F$. This is true for both pre- and post- whiskering. For example in the notation of Definition \ref{def:whiskering} we have 
$$\alpha\diamond F = \alpha\diamond\id_F\hsp\tn{and}\hsp H\diamond\alpha=\id_H\diamond\alpha.$$

\end{remark}

\begin{remark}

All of the above is somehow similar to the world of paths inside a database schema $\mcS$, as seen in Definition \ref{def:congruence}. Indeed, a congruence on the paths of $\mcS$ is an equivalence relation that is closed under composition. The equivalence relation part is analogous to the fact that natural transformations can be composed vertically. The closure under composition part (Properties (3) and (4) in Definition \ref{def:congruence}) is analogous to pre- and post whiskering. See also Lemma \ref{lemma:composing PEDs}. 

This is being mentioned only as a curiosity and a way for the reader to draw connections, not with any additional purpose at this time.

\end{remark}

\begin{theorem}\index{natural transformation!interchange}
$$
\xymatrix{
&\ar@{}[d]|(.65){\alpha_1\Down}&&\ar@{}[d]|(.65){\beta_1\Down}\\
\mcC\ar@/^2pc/[rr]^{F_1}\ar[rr]|{F_2}\ar@/_2pc/[rr]_{F_3}&&\mcD\ar@/^2pc/[rr]^{G_1}\ar[rr]|{G_2}\ar@/_2pc/[rr]_{G_3}&&\mcE\\
&\ar@{}[u]|(.65){\alpha_2\Down}&&\ar@{}[u]|(.65){\beta_2\Down}}
$$
Given a setup of categories, functors, and natural transformations as above, we have
$$(\beta_2\circ\beta_1)\diamond(\alpha_2\circ\alpha_1)\;=\;(\beta_2\diamond\alpha_2)\circ(\beta_1\diamond\alpha_1).$$

\end{theorem}

\begin{proof}

One need only observe that each square in the following diagram commutes, so following the outer path $(\beta_2\circ\beta_1)\diamond(\alpha_2\circ\alpha_1)$ yields the same morphism as following the diagonal path $;(\beta_2\diamond\alpha_2)\circ(\beta_1\diamond\alpha_1)$:
$$
\xymatrix{
G_1F_1\ar[r]^{G_1\diamond\alpha_1}\ar[d]_{\beta_1\diamond F_1}&G_1F_2\ar[r]^{G_1\diamond\alpha_2}\ar[d]_{\beta_1\diamond F_2}&G_1F_3\ar[d]^{\beta_1\diamond F_3}\\
G_2F_1\ar[r]^{G_2\diamond\alpha_1}\ar[d]_{\beta_2\diamond F_1}&G_2F_2\ar[r]^{G_2\diamond\alpha_2}\ar[d]_{\beta_2\diamond F_2}&G_2F_3\ar[d]^{\beta_2\diamond F_3}\\
G_3F_1\ar[r]_{G_3\diamond\alpha_1}&G_3F_2\ar[r]_{G_3\diamond\alpha_2}&G_3F_3
}
$$

\end{proof}


\subsection{The category of instances on a database schema}\index{a category!$\mcC\set$}\index{database!category of instances on}

In Section \ref{sec:schemas and cats intro} we showed that schemas are presentations of categories, and we will show in Section \ref{sec:cat equiv sch} that in fact the category of schemas is equivalent to the category of categories. In this section we therefore take license to blur the distinction between schemas and categories.

If $\mcC$ is a schema, i.e. a category, then as we discussed in Section \ref{sec:instances}, an instance on $\mcC$ is a functor $I\taking\mcC\to\Set$. But now we have a notion beyond categories and functors, namely that of natural transformations. So we make the following definition.

\begin{definition}\label{def:mcC-set}

Let $\mcC$ be a schema (or category). The {\em category of instances on $\mcC$}, denoted $\mcC\set$, is $\Fun(\mcC,\Set)$. Its objects are $\mcC$-instances (i.e. functors $\mcC\to\Set)$ and its morphisms are natural transformations.

\end{definition}

\begin{remark}

One might object to Definition \ref{def:mcC-set} on the grounds that database instances should not be infinite. This is a reasonable perspective, so it is a pleasant fact that the above definition can be modified easily to accomodate it. The subcategory $\Fin$ (see Example \ref{ex:Fin}) of finite sets can be substituted for $\Set$ in Definition \ref{def:mcC-set}. One could define the {\em category of finite instances on $\mcC$} as $\mcC-\Fin=\Fun(\mcC,\Fin)$. Almost all of the ideas in this book will make perfect sense in $\mcC-\Fin$.

\end{remark}

Natural transformations should serve as some kind of morphism between instances on the same schema. How are we to interpret a natural transformation $\alpha\taking I\to J$ between database instances $I,J\taking\mcC\to\Set$? 

Our first clue comes from Application \ref{app:change of fsm}. There we considered the case of a monoid $\mcM$, and we thought about a natural transformation between two functors $X,Y\taking\mcM\to\Set$, considered as different finite state machines. The notion of natural transformation captured the idea of one model being a refinement of another. This same kind of idea works for databases with more than one table (categories with more than one object), but the whole thing is a bit opaque. Let's work it through slowly.

\begin{example}\label{ex:nts on term}
Let us consider the terminal schema, $\ul{1}\iso\fbox{$\bullet^{\tn{Grapes}}$}$. An instance is a functor $\ul{1}\to\Set$ and it is easy to see that this is the same thing as just a set. A natural transformation $\alpha\taking I\to J$ is a function from set $I$ to set $J$. In the standard table view, we might have $I$ and $J$ as below:
\begin{center}
\begin{tabular}{| l ||}\bhline
\multicolumn{1}{| c |}{Grapes $(I)$}\\\bhline
{\bf ID}\\\bbhline
Grape 1\\\hline
Grape 3\\\hline
Grape 4\\\bhline
\end{tabular}
\hspace{1in}
\begin{tabular}{| l ||}\bhline
\multicolumn{1}{| c |}{Grapes $(J)$}\\\bhline
{\bf ID}\\\bbhline
Jan1-01\\\hline
Jan1-02\\\hline
Jan1-03\\\hline
Jan1-04\\\hline
Jan3-01\\\hline
Jan4-01\\\hline
Jan4-02\\\bhline
\end{tabular}
\end{center}

There are 343 natural transformations $I\to J$. Perhaps some of them make more sense than others; e.g. we could hope that the numbers in $I$ corresponded to the numbers after the dash in $J$, or perhaps to what seems to be the date in January. But it could be that the rows in $J$ correspond to batches, and all three grapes in $I$ are part of the first batch on Jan-1. The notion of natural transformation is a mathematical one.
\end{example}

\begin{exercise}\label{exc:indexed sets as functors}\index{indexed set!as functor}
Recall the notion of set-indexed sets from Definition \ref{def:indexed sets}. Let $A$ be a set, and come up with a schema $\mcA$ such that instances on $\mcA$ are $A$-indexed sets. Is our current notion of morphism between instances (i.e. natural transformations) well-aligned with the above definition of ``mapping of $A$-indexed sets"?
\end{exercise}

For a general schema (or category) $\mcC$, let us think through what a morphism $\alpha\taking I\to J$ between instances $I,J\taking\mcC\to\Set$ is. For each object $c\in\Ob(\mcC)$ there is a component $\alpha_c\taking I(c)\to J(c)$. This means that just like in Example \ref{ex:nts on term}, there is for each table $c$ a function from the rows in $I$'s manifestation of $c$ to the rows in $J$'s manifestation of $c$. So to make a natural transformation, such a function has to be specified table by table. But then we have to contend with naturality squares, one for every arrow in $\mcC$. Arrows in $\mcC$ correspond to foreign key columns in the database. The naturality requirement was already covered in Application \ref{app:change of fsm} (and see especially how (\ref{dia:naturality squares for fsm}) is checked in (\ref{dia:naturality for a}) and (\ref{dia:naturality for b})).

\begin{example}\label{ex:graph hom as NT done out}\index{graph homomorphism!as functor}

We saw in Section \ref{sec:graphs as functors} that graphs can be regarded as functors $\mcG\to\Set$, where $\mcG\iso\GrIn$ is the ``schema for graphs" shown here: 
$$\mcG:=\fbox{\xymatrix{\LTO{Arrow}\ar@<.5ex>[r]^{src}\ar@<-.5ex>[r]_{tgt}&\LTO{Vertex}}}$$\index{a schema!indexing graphs}

A database instance $I\taking\mcG\to\Set$ on $\mcG$ consists of two tables. Here is an example instance:
\begin{align*}
I:=\parbox{2in}{\fbox{\xymatrix{\bullet^v\ar[r]^f&\bullet^w\ar@/_1pc/[r]_h\ar@/^1pc/[r]^g&\bullet^x}}}
\hspace{.5in}
\begin{array}{| l || l | l |}\bhline
\multicolumn{3}{|c|}{{\tt Arrow}\;\; (I)}\\\bhline
{\bf ID}&{\bf src}&{\bf tgt}\\\bbhline
f&v&w\\\hline
g&w&x\\\hline
h&w&x\\\bhline
\end{array}
\hspace{.5in}
\begin{array}{| l ||}\bhline
\multicolumn{1}{|c|}{{\tt Vertex}\;\; (I)}\\\bhline
{\bf ID}\\\bbhline
v\\\hline
w\\\hline
x\\\bhline
\end{array}
\end{align*}
To discuss natural transformations, we need two instances. Here is another, $J\taking\mcG\to\Set$,
\begin{align*}
J:=\parbox{2in}{\fbox{\xymatrix{
\LMO{q}\ar[r]^i&\LMO{r}\ar@/^1pc/[r]^j&\LMO{s}\ar@/^1pc/[l]^k\ar[r]^\ell&\LMO{t}\\&&&\LMO{u}}}}
\hspace{.5in}
\begin{array}{| l || l | l |}\bhline
\multicolumn{3}{|c|}{{\tt Arrow}\;\; (J)}\\\bhline
{\bf ID}&{\bf src}&{\bf tgt}\\\bbhline
i&q&r\\\hline
j&r&s\\\hline
k&s&r\\\hline
\ell&s&t\\\bhline
\end{array}
\hspace{.5in}
\begin{array}{| l ||}\bhline
\multicolumn{1}{|c|}{{\tt Vertex}\;\; (J)}\\\bhline
{\bf ID}\\\bbhline
q\\\hline
r\\\hline
s\\\hline
t\\\hline
u\\\bhline
\end{array}
\end{align*}
To give a natural transformation $\alpha\taking I\to J$, we give two components: one for arrows and one for vertices. We need to say where each vertex in $I$ goes in $J$ and we need to say where each arrow in $I$ goes in $J$. The naturality squares insist that if we specify that $g\mapsto j$, for example, then we better specify that$w\mapsto r$ and that $x\mapsto s$. What a computer is very good at, but a human is fairly slow at, is checking that a given pair of components (arrows and vertices) really is natural. 

There are 8000 ways to come up with component functions $\alpha_{{\tt Arrow}}$ and $\alpha_{{\tt Vertex}}$, but precisely four natural transformations, i.e. four graph homomorphisms, $I\to J$; the other 7996 are haphazard flingings of arrows to arrows and vertices to vertices without any regard to sources and targets. We briefly describe the four now. 

First off, nothing can be sent to $u$ because arrows must go to arrows and $u$ touches no arrows. If we send $v\mapsto q$ then $f$ must map to $i$, and $w$ must map to $r$, and both $g$ and $h$ must map to $j$, and $x$ must map to $s$. If we send $v\mapsto r$ then there are two choices for $g$ and $h$. If we send $v\mapsto s$ then there's one way to obtain a graph morphism. If we try to send $v\mapsto^?t$, we fail. All of this can be seen by staring at the tables rather than at the pictorial representations of the graphs; the human eye understands these pictures better, but the computer understands the tables better.

\end{example}

\begin{exercise}
If $I,J\taking\mcG\to\Set$ are as in Example \ref{ex:graph hom as NT done out}, how many natural transformations are there $J\to I$?
\end{exercise}

\begin{exercise}
Let $Y_A\taking\mcG\to\Set$ denote the instance below:
\begin{align*}
\begin{array}{| l || l | l |}\bhline
\multicolumn{3}{|c|}{{\tt Arrow}\;\; (Y_A)}\\\bhline
{\bf ID}&{\bf src}&{\bf tgt}\\\bbhline
a&v_0&v_1\\\bhline
\end{array}
\hspace{.5in}
\begin{array}{| l ||}\bhline
\multicolumn{1}{|c|}{{\tt Vertex}\;\; (Y_A)}\\\bhline
{\bf ID}\\\bbhline
v_0\\\hline
v_1\\\bhline
\end{array}
\end{align*}
Let $I\taking\mcG\to\Set$ be as in Example \ref{ex:graph hom as NT done out}.
\sexc How many natural transformations are there $Y_A\to I$?
\next With $J$ as above, how many natural transformations are there $Y_A\to J$?
\next Do you have any conjecture about the way natural transformations $Y_A\to X$ behave for arbitrary graphs $X\taking\mcG\to\Set$?
\endsexc
\end{exercise}

In terms of databases, this notion of instance morphism $I\to J$ is fairly benign. For every table its a mapping from the set of rows in $I$'s version of the table to $J$'s version of the table, such that all the foreign keys are respected. We will see that this notion of morphism has excellent formal properties, so that projections, unions, and joins of tables (the typical database operations) would be predicted to be ``obviously interesting" by a category theorist who had no idea what a database was.
\footnote{More precisely, given a functor between schemas $F\taking\mcC\to\mcD$, the pullback $\Delta_F\taking\mcD\set\to\mcC\set$, its left $\Sigma_F$ and its right adjoint $\Pi_F$ constitute these important queries. See Section \ref{sec:data migration}.}

However, something is also missing from the natural transformation picture. A very important occurrence in the world of databases is the update. Everyone can understand this: a person makes a change in one of the tables, like changing your address from Cambridge, MA to Hereford, UK. Most such arbitrary changes of database instance are not ``natural", in that the new linking pattern is incompatible with the old.

It is interesting to consider how updates of $\mcC$-instances should be understood category theoretically. We might want a category $Upd_\mcC$ whose objects are $\mcC$-instances and whose morphisms are updates. But then what is the composition formula? Is there a unique morphism $I\to J$ whenever $J$ can be obtained as an update on $I$? Because in that case, we would be defining $Upd_\mcC$ to be the indiscrete category on the set of $\mcC$-instances (see Example \ref{ex:indiscrete cat equiv to terminal}).

\begin{exercise}
Research project: Can you come up with a satisfactory way to model database updates category-theoretically? Let $\NN$ be the category
$$[\NN]:=\xymatrix{\LMO{0}\ar[r]&\LMO{1}\ar[r]&\LMO{2}\ar[r]&\cdots}$$ 
representing a discrete timeline. A place to start might be to use something like the slice category $\Cat_{/[\NN]}$ where the fiber over each object in $\NN$ is a snapshot of the database in time. Can you make this work?
\end{exercise}


\subsection{Equivalence of categories}\label{sec:equivalence of cats}

We have a category $\Cat$ of categories, and in every category there is a notion of isomorphism between objects: one morphism each way, such that each round-trip composition is the identity. An isomorphism in $\Cat$, therefore, takes place between two categories, say $\mcC$ and $\mcD$: it is a functor $F\taking\mcC\to\mcD$ and a functor $G\taking\mcD\to\mcC$ such that $G\circ F=\id_\mcC$ and $F\circ G=\id_\mcD$. 

It turns out that categories are often similar enough to be considered equivalent without being isomorphic. For this reason, the notion of isomorphism is considered ``too strong" to be useful for categories. The feeling to a category theorist might be akin to saying that two material samples are the same if there is an atom-by-atom matching, or that two words are the same if they are written in the same font, of the same size, by the same person, in the same state of mind. 

As reasonable as isomorphism is as a notion {\em in} most categories, it fails to be the ``right notion" {\em about} categories. The reason is that {\em in} categories there are objects and morphisms, whereas when we talk {\em about} categories, we have categories and functors, plus natural transformations. These serve as mappings between mappings, and this is not part of the structure of an ordinary category. In cases where a category $\mcC$ does have such mappings between mappings, it is often a ``better notion" if we take that extra structure into account, like we will for categories. This whole subject leads us to the study of 2-categories (or $n$-categories, or $\infty$-categories), which we do not discuss in this book. See, for example, \cite{Le1} for an introduction.

Regardless, our purpose now is to explain this ``good notion" of sameness for categories, namely {\em equivalences of categories}, which appropriately take natural transformations into account. Instead of ``functors going both ways with round trips equal to identity", which is required in order to be an isomorphism of categories, equivalence of categories demands ``functors going both ways with round trips {\em isomorphic} to identity".

\begin{definition}[Equivalence of categories]\label{def:equiv of cats}\index{category!equivalence of}

Let $\mcC$ and $\mcC'$ be categories. A functor $F\taking\mcC\to\mcC'$ is called {\em an equivalence of categories}, and denoted $F\taking\mcC\To{\simeq}\mcC'$,\index{a symbol!$\simeq$}
\footnote{The notation $\simeq$ has already been used for equivalences of paths in a schema. We do not mean to equate these ideas; we are just reusing the symbol. Hopefully no confusion will arise.}
 if there exists a functor $F'\taking\mcC'\to\mcC$ and natural isomorphisms $\alpha\taking \id_\mcC\To{\iso}F'\circ F$ and $\alpha'\taking\id_{\mcC'}\To{\iso}F\circ F'$. In this case we say that $F$ and $F'$ are {\em mutually inverse equivalences}.

\end{definition}

Unpacking a bit, suppose we are given functors $F\taking\mcC\to\mcC'$ and $F'\taking\mcC'\to\mcC$. We want to know something about the roundtrips on $\mcC$ and on $\mcC'$; we want to know the same kind of information about each roundtrip, so let's concentrate on the $\mcC$ side. We want to know something about $F'\circ F\taking\mcC\to\mcC$, so let's name it $i\taking\mcC\to\mcC$; we want to know that $i$ is a natural isomorphism. That is, for every $c\in\Ob(\mcC)$ we want an isomorphism $\alpha_c\taking c\To{\iso} i(c)$, and we want to know that these isomorphisms are picked carefully enough that given $g\taking c\to c'$ in $\mcC$, the choice of isomorphisms for $c$ and $c'$ are compatible,
$$\xymatrix{c\ar[r]^{\alpha_c}\ar[d]_g&i(c)\ar[d]^{i(g)}\\c'\ar[r]_{\alpha_{c'}}&i(c').}$$
To be an equivalence, the same has to hold for the other roundtrip, $i'=F\circ F'\taking\mcC'\to\mcC'.$

\begin{exercise}
Let $\mcC$ and $\mcC'$ be categories. Suppose that $F\taking\mcC\to\mcC'$ is an isomorphism of categories.
\sexc Is it an equivalence of categories?
\next What are the components of $\alpha$ and $\alpha'$ (with notation as in Definition \ref{def:equiv of cats})?
\endsexc
\end{exercise}

\begin{example}\label{ex:indiscrete cat equiv to terminal}

Let $S$ be a set and let $S\times S\ss S\times S$ be the complete relation on $S$, which is a preorder $K_S$. Recall from Proposition \ref{prop:preorders to cats} that we have a functor $i\taking\PrO\to\Cat$,\index{a functor!$\PrO\to\Cat$} and the resulting category $i(K_S)$ is called the {\em indiscrete category on $S$}; it has objects $S$ and a single morphism between every pair of objects. Here is a picture of $K_{\{1,2,3\}}$:
$$\xymatrix@=15pt{\LMO{1}\ar@(l,u)[]\ar@/^.5pc/[rr]\ar@/^.5pc/[rdd]&&\LMO{2}\ar@(u,r)[]\ar@/^.5pc/[ll]\ar@/^.5pc/[ldd]\\\\&\LMO{3}\ar@(dr,dl)[]_~\ar@/^.5pc/[uur]\ar@/^.5pc/[uul]}$$

It is easy check that $K_{\ul{1}}$, the indiscrete category on one element, is isomorphic to $\ul{1}$, the discrete category on one object, also known as the terminal category (see Exercise \ref{exc:term cat}). The category $\ul{1}$ consists of one object, its identity morphism, and nothing else. 

The only way that $K_S$ can be isomorphic to $\ul{1}$ is if $S$ has one element.
\footnote{One way to see this is that by Exercise \ref{exc:Ob is a functor}, we have a functor $\Ob\taking\Cat\to\Set$, and we know by Exercise \ref{exc:functors preserve isos} that functors preserve isomorphisms, so an isomorphism between categories must restrict to an isomorphism between their sets of objects. The only sets that are isomorphic to $\ul{1}$ have one element.} 
On the other hand, there is an equivalence of categories $$K_S\simeq\ul{1}$$ for every set $S\neq\emptyset$. 

In fact, there are many such equivalences, one for each element of $S$. To see this, let $S$ be a nonempty set and choose an element $s_0\in S$. For every $s\in S$, there is a unique isomorphism $k_s\taking s\To{\iso}s_0$ in $K_S$. Let $F\taking K_S\to\ul{1}$ be the only possible functor (see Exercise \ref{exc:term cat}), and let $F'\taking\ul{1}\to K_S$ send the unique object in $\ul{1}$ to the object $s_0$. 

Note that $F'\circ F=\id_{\ul{1}}\taking\ul{1}\to\ul{1}$ is the identity, but that $F\circ F'\taking K_S\to K_S$ sends everything to $s_0$. Let $\alpha=\id_{\ul{1}}$ and define $\alpha'\taking\id_{K_S}\to F\circ F'$ by $\alpha'_s=k_s$. Note that $\alpha'_s$ is an isomorphism for each $s\in\Ob(K_S)$, and note that $\alpha'$ is a natural transformation (hence natural isomorphism) because every possible square commutes in $K_S$. This completes the proof, initiated in the paragraph above, that the category $K_S$ is equivalent to $\ul{1}$ for every nonempty set $S$, and that this fact can be witnessed by any element $s_0\in S$.

\end{example}

\begin{example}\label{ex:finite linear orders}

Consider the category $\FLin$, described in Example \ref{ex:FLin}, of finite nonempty linear orders. For every natural number $n\in\NN$, let $[n]\in\Ob(\FLin)$ denote the linear order shown in Example \ref{ex:finite lo}. Define a category $\bD$\index{a category!$\bD$} whose objects are given by $\Ob(\bD)=\{[n]\|n\in\NN\}$ and with $\Hom_{\bD}([m],[n])=\Hom_{\FLin}([m],[n])$. The difference between $\FLin$ and $\bD$ is only that objects in $\FLin$ may have ``funny labels", e.g. 
$$\xymatrix{\LMO{5}\ar[r]&\LMO{x}\ar[r]&\LMO{``Sam"}}$$ 
whereas objects in $\bD$ all have standard labels, e.g.
$$\xymatrix{\LMO{0}\ar[r]&\LMO{1}\ar[r]&\LMO{2}}$$
Clearly $\FLin$ is a much larger category, and yet feels like it is ``pretty much the same as" $\bD$. Justly, they are equivalent, $\FLin\simeq\bD$. 

The functor $F'\taking\bD\to\FLin$\index{a functor!$\bD\to\FLin$} is the inclusion; the functor $F\taking\FLin\to\bD$ sends every finite nonempty linear order $X\in\Ob(\FLin)$ to the object $F(X):=[n]\in\bD$, where $\Ob(X)\iso\{0,1,\ldots,n\}$. For each such $X$ there is a unique isomorphism $\alpha_X\taking X\To{\iso}[n]$, and these fit together into
\footnote{The phrase ``these fit together into" is suggestive shorthand for, and thus can be replaced with, the phrase ``the naturality squares commute for these components, so together they constitute".}
the required natural isomorphism $\id_{\FLin}\to F'\circ F$. The other natural isomorphism $\alpha'\taking\id_{\bD}\to F\circ F'$ is the identity.

\end{example}

\begin{exercise}
Recall from Definition \ref{def:cardinality} that a set $X$ is called finite if there exists a natural number $n\in\NN$ and an isomorphism of sets $X\to\ul{n}$. Let $\Fin$\index{a category!$\Fin$} denote the category whose objects are the finite sets and whose morphisms are the functions. Let $\mcS$ denote the category whose objects are the sets $\ul{n}$ and whose morphisms are again the functions. For every object $X\in\Ob(\Fin)$ there exists an isomorphism $p_X\taking X\to\ul{n}$ for some unique object $\ul{n}\in\Ob(\mcS)$. Find an equivalence of categories $\Fin\To{\simeq}\mcS$. 
\end{exercise}

\begin{exercise}
We say that two categories $\mcC$ and $\mcD$ are equivalent if there exists an equivalence of categories between them. Show that the relation of ``being equivalent" is an equivalence relation on $\Ob(\Cat)$.
\end{exercise}

\begin{example}\label{ex:Z1 not equiv Z2}

Consider the group $\ZZ_2:=(\{0,1\},0,+)$, where $1+1=0$. As a category, $\ZZ_2$ has one object $\monOb$ and two morphisms, namely $0,1$, such that $0$ is the identity. Since $\ZZ_2$ is a group, the morphism $1\taking\monOb\to\monOb$ must have an inverse $x$, meaning $1+x=0$, and $x=1$ is the only solution.

The point is that the morphism $1$ in $\ZZ_2$ is an isomorphism. Let $\mcC=\ul{1}$ be the terminal category as in Exercise \ref{exc:term cat}. One might accidentally believe that $\mcC$ is equivalent to $\ZZ_2$, but this is not the case! The argument in favor of the accidental belief is that we have unique functors $F\taking\ZZ_2\to\mcC$ and $F'\taking\mcC\to\ZZ_2$ (and this is true); the roundtrip $F\circ F'\taking\mcC\to\mcC$ is the identity (and this is true); and for the roundtrip $F'\circ F\taking\ZZ_2\to\ZZ_2$ both morphisms in $\ZZ_2$ are isomorphisms, so any choice of morphism $\alpha_\monOb\taking\monOb\to F'\circ F(\monOb)$ will be an isomorphism (and this is true). The problem is that no such $\alpha_\monOb$ will be a natural transformation.

When we roundtrip $F'\circ F\taking\ZZ_2\to\ZZ_2$, the image of $1\taking\monOb\to\monOb$ is $F'\circ F(1)=0=\id_{\monOb}$. So the naturality square for the morphism $1$ looks like this:
$$
\xymatrix{\monOb\ar[r]^{\alpha_\monOb}\ar[d]_{1}&\monOb\ar[d]^{0=F'\circ F(1)}\\\monOb\ar[r]_{\alpha_\monOb}&\monOb}
$$
where we still haven't decided whether we want $\alpha_\monOb$ to be $0$ or $1$. Unfortunately, neither choice works (i.e. for neither choice will the diagram commute) because $x+1\neq x+0$ in $\ZZ_2$.

\end{example}

\begin{definition}[Skeleton]\index{skeleton}

Let $\mcC$ be a category. We saw in Lemma \ref{lemma:isomorphic ER} that the relation of ``being isomorphic" is an equivalence relation $\cong$ on $\Ob(\mcC)$. An {\em election in $\mcC$} is a choice $E$ of the following sort:
\begin{itemize}
\item for each $\cong$-equivalence class $S\ss\Ob(\mcC)$ a choice of object $s_E\in S$, called the {\em elected object for $S$}, and
\item for each object $c\in\Ob(\mcC)$ a choice of isomorphism $i_c\taking s_E\to c$ and $j_c\taking c\to s_E$ with $i_c\circ j_c=\id_c$ and $j_c\circ i_c=\id_{s_E}$, where $s_E$ is an elected object (depending on $c$).
\end{itemize}
Given an election $E$ in $\mcC$, there is a category called the {\em $E$-elected skeleton of $\mcC$}, denoted $\Skel_E(\mcC)$, whose objects are the elected objects and whose morphisms $s\to t$ for any elected objects $s,t\in\Ob(\mcC)$ are given by $\Hom_{\Skel_E(\mcC)}(s,t)=\Hom_\mcC(s,t)$. Any object $c\in\Ob(\mcC)$ is isomorphic to a unique elected object $s_E$; we refer to $s_E$ as the {\em elected representative} of $c$; we refer to the isomorphisms $i_c$ and $j_c$ as the {\em representing isomorphisms} for $c$.

\end{definition}

\begin{proposition}

Let $\mcC$ be a category and let $E$ be an election in $\mcC$. There is an equivalence of categories $$\Skel_E(\mcC)\simeq\mcC.$$

\end{proposition}

\begin{proof}

The functor $F'\taking\Skel_E(\mcC)\to\mcC$ is the inclusion. The functor $F\taking\mcC\to\Skel_E(\mcC)$ sends each object in $\mcC$ to its elected representative. Given objects $c,c'\in\Ob(\mcC)$ with elected representatives $s,t$ respectively, and given a morphism $g\taking c\to c'$ in $\mcC$, let $i_c,j_c,i_{c'},$ and $j_{c'}$ be the representing isomorphisms, and define $F(g)\taking s\to t$ to be the composite 
$$\xymatrix{s\ar[r]^{i_c}&c\ar[r]^g&c'\ar[r]^{j_{c'}}&t.}$$
This is functorial because it sends the identity to the identity and $F(g\circ g')=F(g)\circ F(g')$.

The composite $F\circ F'\taking\Skel_E(\mcC)\to\Skel_E(\mcC)$ is the identity. For each $c\in\Ob(\mcC)$ define $\alpha_c\taking c\To{\iso} F'\circ F(c)$ by  $\alpha_c:=j_c$. Given $g\taking c\to c'$ the required naturality square is shown to the left below:
$$
\xymatrix@=40pt{
c\ar[r]^{j_c}\ar[d]_{g}\ar@{}[dr]|{?}&s\ar[r]^{i_c}\ar[d]^{F'\circ F(g)}&c\ar[d]^g\\
c'\ar[r]_{j_c'}&t&c'\ar[l]^{j_c'}}
$$
The right-hand part commutes by definition of $F$ and $F'$; i.e. $j'\circ g\circ i_c=F'\circ F(g)$. The left-hand square commutes because $i_c\circ j_c=\id_c$.

\end{proof}

\begin{definition}

A {\em skeleton of $\mcC$} is a category $\mcS$, equivalent to $\mcC$, such that for any two objects $s,s'\in\Ob(\mcS)$, if $s\iso s'$ then $s=s'$. 

\end{definition}

\begin{exercise}
Let $\mcP$ be a preorder (considered as a category).
\sexc If $\mcP'$ is a skeleton of $\mcP$, is it a partial order?
\next Is every partial order the skeleton of some preorder?
\endsexc
\end{exercise}

\begin{definition}[Full and faithful functors]\label{def:full faithful}\index{functor!full}\index{functor!faithful}

Let $\mcC$ and $\mcD$ be categories, and let $F\taking\mcC\to\mcD$ be a functor. For any two objects $c,c'\in\Ob(\mcC)$, we have a function $\Hom_F(c,c')\taking\Hom_\mcC(c,c')\to\Hom_\mcD(F(c),F(c'))$ guaranteed by the definition of functor.
We say that $F$ is {\em a full functor} if $\Hom_F(c,c')$ is surjective for every $c,c'$.
We say that $F$ is {\em a faithful functor} if $\Hom_F(c,c')$ is injective for every $c,c'$. We say that $F$ is {\em a fully faithful functor} if $\Hom_F(c,c')$ is bijective for every $c,c'$.

\end{definition}

\begin{exercise}
Let $\ul{1}$ and $\ul{2}$ be the discrete categories on one and two objects, respectively. There is only one functor $\ul{2}\to\ul{1}$.
\sexc Is it full?
\next Is it faithful?
\endsexc
\end{exercise}

\begin{exercise}\label{exc:empty fully faithful}
Let $\ul{0}$ denote the empty category, and let $\mcC$ be any category. There is a unique functor $F\taking \ul{0}\to\mcC$.
\sexc For general $\mcC$ will $F$ be full?
\next For general $\mcC$ will $F$ be faithful?
\next For general $\mcC$ will $F$ be an equivalence of categories?
\endsexc
\end{exercise}

\begin{proposition}

Let $\mcC$ and $\mcC'$ be categories and let $F\taking\mcC\to\mcC'$ be an equivalence of categories. Then $F$ is fully faithful.

\end{proposition}

\begin{proof}

Suppose $F$ is an equivalence, so we can find a functor $F'\taking\mcC'\to\mcC$ and natural isomorphisms $\alpha\taking\id_\mcC\To{\iso}F'\circ F$ and $\alpha'\taking\id_{\mcC'}\To{\iso} F\circ F'$. We need to know that for any objects $c,d\in\Ob(\mcC)$, the map $$\Hom_F(c,d)\taking\Hom_\mcC(c,d)\to\Hom_{\mcC'}(Fc,Fd)$$ is bijective. Consider the following diagram 
$$
\xymatrix{
\Hom_\mcC(c,d)\ar[rr]^{\Hom_F(c,d)}\ar[ddrr]_\alpha&&\Hom_{\mcC'}(Fc,Fd)\ar[ddrr]^{\alpha'}\ar[dd]^{\Hom_{F'}(Fc,Fd)}\\\\
&&\Hom_\mcC(F'Fc,F'Fd)\ar[rr]_{\Hom_{F}(F'Fc,F'Fd)}&&\Hom_{\mcC'}(FF'Fc,FF'Fd)
}
$$
The fact that $\alpha$ is bijective implies that the vertical function is surjective. The fact that $\alpha'$ is bijective implies that the vertical function is injective, so it is bijective. This implies that $\Hom_F(c,d)$ is bijective as well.

\end{proof}

\begin{exercise}
Let $\ZZ_2$ be the group (as category) from Example \ref{ex:Z1 not equiv Z2}. Are there any fully faithful functors $\ZZ_2\to\ul{1}$?
\end{exercise}


\section{Categories and schemas are equivalent, $\Cat\simeq\Sch$}\label{sec:cat equiv sch}

Perhaps it is intuitively clear that schemas are somehow equivalent to categories, and in this section we make that precise. The basic idea was already laid out in Section \ref{sec:schemas and cats intro}.


\subsection{The category $\Sch$ of schemas}\label{sec:sch as category}\index{category!as equivalent to schema}\index{schema!as equivalent to category}

Recall from Definition \ref{def:schema} that a schema consists of a pair $\mcC:=(G,\simeq)$, where $G=(V,A,src,tgt)$ is a graph and $\simeq$ is a congruence, meaning a kind of equivalence relation on the paths in $G$ (see Definition \ref{def:congruence}. If we think of a schema as being analogous to a category, what should fulfill the role of functors? That is, what are to be the morphisms in $\Sch$?

Unfortunately, ones first guess may give the wrong notion if we want an equivalence $\Sch\simeq\Cat$. Since objects in $\Sch$ are graphs with additional structure, one might imagine that a morphism $\mcC\to\mcC'$ in $\Sch$ should be a graph homomorphism (as in Definition \ref{def:graph homomorphism}) that preserves said structure. But graph homomorphisms require that arrows be sent to arrows, whereas we are more interested in paths than in individual arrows---the arrows are merely useful for presentation. 

If instead we define morphisms between schemas to be maps that send paths in $\mcC$ to paths in $\mcC'$, subject to the requirements that path endpoints, path concatenations, and path equivalences are preserved, this will turn out to give the correct notion. And since a path is a concatenation of its arrows, it suffices to give a function $F$ from the arrows of $\mcC$ to the paths of $\mcC'$, which automatically takes care of the first two requirements above; we must only take care that $F$ preserves path equivalences.

Recall from Examples \ref{ex:paths-graph} and \ref{ex:concat paths of paths} the paths-graph functor $\Paths\taking\Grph\to\Grph$,\index{a functor!$\Paths\taking\Grph\to\Grph$} the paths of paths functor $\Paths\circ\Paths\taking\Grph\to\Grph$, and the natural transformations for any graph $G$, 
\begin{align}\label{dia:proto paths monad}
\eta_G\taking G\to\Paths(G)\hsp\tn{and}\hsp\mu_G\taking\Paths(\Paths(G))\to\Paths(G).
\end{align}
The function $\eta_G$ spells out the fact that every arrow in $G$ counts as a path in $G$, and the function $\mu_G$ spells out the fact that a head-to-tail sequence of paths (a path of paths) in $G$ can be concatenated to a single path in $G$.

\begin{exercise}
Let $[2]$ denote the graph $\LMO{0}\to\LMO{1}\to\LMO{2}$, and let $\Loop$ denote the unique graph having one vertex and one arrow (pictured in Diagram (\ref{dia:loop})).
\sexc Find a graph homomorphism $f\taking[2]\to\Paths(\Loop)$ that is injective on arrows (i.e. such that no two arrows in the graph $[2]$ are sent by $f$ to the same arrow in $\Paths(\Loop)$).
\next The graph $[2]$ has 6 paths, so $\Paths([2])$ has 6 arrows. What are the images of these arrows under the graph homomorphism $\Paths(f)\taking\Paths([2])\to\Paths(\Paths(\Loop))$? 
\endsexc
\end{exercise}

We are almost ready to give the definition of schema morphism, but before we do, let's return to our original idea. Given graphs $G,G'$ (underlying schemas $\mcC,\mcC'$) we originally wanted a function from the paths in $G$ to the paths in $G'$, but we realized it was more concise to speak of a function from arrows in $G$ to paths in $G'$. How do we get back what we originally wanted from the concise version? Given a graph homomorphism $f\taking G\to\Paths(G')$, we use (\ref{dia:proto paths monad}) to form the following composition, which we denote simply by $\Paths_f\taking\Paths(G)\to\Paths(G')$:
\begin{align}\label{dia:kleisli comp in graph}
\xymatrix@=35pt{\Paths(G)\ar[r]^-{\Paths(f)}&\Paths(\Paths(G'))\ar[r]^-{\mu_{G'}}&\Paths(G')}
\end{align}
This says that given a function from arrows in $G$ to paths in $G'$, a path in $G$ becomes a path of paths in $G'$, which can be concatenated to a path in $G'$. This simply and precisely spells out our intuition.

\begin{definition}[Schema morphism]\label{def:schema morphism}\index{schema!morphism}

Let $G=(V,A,src,tgt)$ and $G'=(V',A',src',tgt')$ be graphs, and let $\mcC=(G,\simeq_G)$ and $\mcC'=(G',\simeq_{G'})$ be schemas. A {\em schema morphism $F$ from $\mcC$ to $\mcD$}, denoted $F\taking\mcC\to\mcD$ is a graph homomorphism 
\footnote{By Definition \ref{def:graph homomorphism}, a graph homomorphism $F\taking G\to\Paths(G')$ will consist of a vertex part $F_0\taking V\to V'$ and an arrows part $F_1\taking E\to\Path(G')$. See also Definition \ref{def:paths in graph}.}
$$F\taking G\to\Paths(G')$$ that satisfies the following condition for any paths $p$ and $q$ in $G$: 
\begin{align}\label{dia:path cond for schema morphism}
\tn{if \;\;$p\simeq_G q$ \;\;then\;\; $\Paths_F(p)\simeq_{G'}\Paths_F(q)$}.
\end{align}

Two schema morphisms $E,F\taking\mcC\to\mcC'$ are considered identical if they agree on vertices (i.e. $E_0=F_0$) and if, for every arrow $f$ in $G$, there is a path equivalence in $G'$ $$E_1(f)\simeq_{G'}F_1(f).$$

We now define the {\em category of schemas}, denoted $\Sch$,\index{a category!$\Sch$} to be the category whose objects are schemas as in Definition \ref{def:schema} and whose morphisms are schema morphisms defined as above. The identity morphism on schema $\mcC=(G,\simeq_G)$ is the schema morphism $\id_\mcC:=\eta_G\taking G\to\Paths(G)$ as defined in Equation (\ref{dia:proto paths monad}). We need only understand how to compose schema morphisms $F\taking\mcC\to\mcC'$ and $F'\taking\mcC'\to\mcC''$. On objects their composition is obvious. Given an arrow in $\mcC$, it is sent to a path in $\mcC'$; each arrow in that path is sent to a path in $\mcC''$. We then have a path of paths which we can concatenate (via $\mu_{G''}\taking\Paths(\Paths(G''))\to\Paths(G'')$ as in \ref{dia:proto paths monad}) to get a path in $\mcC''$ as desired.

\end{definition}

\begin{slogan}
A schema morphism sends vertices to vertices, arrows to paths, and path equivalences to path equivalences.
\end{slogan}

\begin{example}

Let $[2]$ be the linear order graph of length 2, pictured to the left, and let $\mcC$ denote the schema pictured to the right below:
$$
[2]:=\fbox{\xymatrix{\LMO{0}\ar[r]^{f_1}&\LMO{1}\ar[r]^{f_2}&\LMO{2}}}
\hspace{1in}
\mcC:=\parbox{.9in}{\fbox{
\xymatrix{\LMO{a}\ar[r]^g\ar[rd]_i&\LMO{b}\ar[d]^h\\&\LMO{c}
}}}
$$
We impose on $\mcC$ the path equivalence declaration $[g,h]\simeq[i]$ and show that in this case $\mcC$ and $[2]$ are isomorphic in $\Sch$. We have a schema morphism $F\taking[2]\to\mcC$ sending $0\mapsto a, 1\mapsto b, 2\mapsto c$, and sending each arrow in $[2]$ to an arrow in $\mcC$. And we have a schema morphism $F'\taking\mcC\to[2]$ which reverses this mapping on vertices; note that $F'$ must send the arrow $i$ in $\mcC$ to the path $[f_1,f_2]$ in $[2]$, which is ok! The roundtrip $F'\circ F\taking [2]\to[2]$ is identity. The roundtrip $F\circ F'\taking\mcC\to\mcC$ may look like it's not the identity; indeed it sends vertices to themselves but it sends $i$ to the path $[g,h]$. But according to Definition \ref{def:schema morphism}, this schema morphism is considered identical to $\id_\mcC$ because there is a path equivalence $\id_\mcC(i)=[i]\simeq[g,h]=F\circ F'(i).$

\end{example}

\begin{exercise}
Consider the schema $[2]$ and the schema $\mcC$ pictured above, except where this time we {\em do not} impose any path equivalence declarations on $\mcC$, so $[g,h]\not\simeq[i]$ in our current version of $\mcC$.
\sexc How many schema morphisms are there $[2]\to\mcC$ that send 0 to $a$?
\next How many schema morphisms are there $\mcC\to[2]$ that send $a$ to $0$?
\endsexc
\end{exercise}

\begin{exercise}\label{exc:finite hierarchies 1}
Consider the graph $\Loop$ pictured below $$\Loop:=\LoopSchema$$ and for any natural number $n$, let $\mcL_n$ denote the schema $(\Loop,\simeq_n)$ where $\simeq_n$ is the PED $f^{n+1}\simeq f^n$. This is the ``finite hierarchy" schema of Example \ref{ex:finite hierarchy}. Let $\ul{1}$ denote the graph with one vertex and no arrows; consider it as a schema.
\sexc Is $\ul{1}$ isomorphic to $\mcL_1$ in $\Sch$?
\next Is it isomorphic to any (other) $\mcL_n$?
\endsexc
\end{exercise}

\begin{exercise}
Let $\Loop$ and $\mcL_n$ be the schemas defined in Exercise \ref{exc:finite hierarchies 1}.
\sexc What is the cardinality of the set $\Hom_\Sch(\mcL_3,\mcL_5)$?
\next What is the cardinality of the set $\Hom_\Sch(\mcL_5,\mcL_3)$? Hint: the cardinality of the set $\Hom_\Sch(\mcL_4,\mcL_9)$ is 8.
\endsexc
\end{exercise}


\subsection{Proving the equivalence}\label{sec:proof of cat=sch}

\begin{construction}[From schema to category]

We will define a functor $L\taking\Sch\to\Cat$\index{a functor!$\Sch\to\Cat$}. Let $\mcC=(G,\simeq)$ be a categorical schema, where $G=(V,A,src,tgt)$. Define $L(\mcC)$ to be the category with $\Ob(L(\mcC))=V$, and with $\Hom_{L(\mcC)}(v_1,v_2):=\Path_G(v,w)/\simeq$, i.e. the set of paths in $G$, modulo the path equivalence relation for $\mcC$. The composition of morphisms is defined by concatenation of paths, and Lemma \ref{lemma:composing PEDs} ensures that such composition is well-defined. We have thus defined $L$ on objects of $\Sch$.

Given a schema morphism $F\taking\mcC\to\mcC'$, where $\mcC'=(G',\simeq')$, we need to produce a functor $L(F)\taking L(\mcC)\to L(\mcC')$. The objects of $L(\mcC)$ and $L(\mcC')$ are the vertices of $G$ and $G'$ respectively, and $F$ provides the necessary function on objects. Diagram (\ref{dia:kleisli comp in graph}) provides a function $\Paths_F\taking\Paths(G)\to\Paths(G')$ will provide the requisite function for morphisms. 

A morphism in $L(\mcC)$ is an equivalence class of paths in $\mcC$. For any representative path $p\in\Paths(G)$, we have $\Paths_F(p)\in\Paths(G')$, and if $p\simeq q$ then $\Paths_F(p)\simeq'\Paths_F(q)$ by condition \ref{dia:path cond for schema morphism}. Thus $\Paths_F$ indeed provides us with a function $\Hom_{L(\mcC)}\to\Hom_{L(\mcC')}$. This defines $L$ on morphisms in $\Sch$. It is clear that $L$ preserves composition and identities, so it is a functor.

\end{construction}

\begin{construction}[From category to schema]

We will define a functor $R\taking\Cat\to\Sch$.\index{a functor!$\Cat\to\Sch$} Let $\mcC=(\Ob(\mcC),\Hom_\mcC,dom,cod,\ids,\circ)$ be a category (see Exercise \ref{exc:cat in set}). Let $R(\mcC)=(G,\simeq)$ where $G$ is the graph $$G=(\Ob(\mcC),\Hom_\mcC,dom,cod),$$ and with $\simeq$ defined as the congruence generated by the following path equivalence declarations: for any composable sequence of morphisms $f_1,f_2,\ldots,f_n$ (with $dom(f_{i+1})=cod(f_i)$ for each $1\leq i\leq n-1$) we put 
\begin{align}\label{dia:composition formula}
[f_1,f_2,\ldots,f_n]\simeq [f_n\circ\cdots\circ f_2\circ f_1].
\end{align} 
This defines $R$ on objects of $\Cat$. 

A functor $F\taking\mcC\to\mcD$ induces a schema morphism $R(F)\taking R(\mcC)\to R(\mcD)$, because vertices are sent to vertices, arrows are sent to arrows (as paths of length 1), and path equivalence is preserved by (\ref{dia:composition formula}) and the fact that $F$ preserves the composition formula. This defines $R$ on morphisms in $\Cat$. It is clear that $R$ preserves compositions, so it is a functor.

\end{construction}

\begin{theorem}\label{thm:equivalence of categories and schemas}

The functors $$\xymatrix{L\taking\Sch\ar@<.5ex>[r]&\Cat\!:R\ar@<.5ex>[l]}$$ are mutually inverse equivalences of categories.

\end{theorem}

\begin{proof}[Sketch of proof]

It is clear that there is a natural isomorphism $\alpha\taking\id_\Cat\To{\iso} L\circ R$; i.e. for any category $\mcC$, there is an isomorphism $\mcC\iso L(R(\mcC))$. 

Before giving an isomorphism $\beta\taking\id_\Sch\To{\iso}R\circ L$, we briefly describe $R(L(\mcS))=:(G',\simeq')$ for a schema $\mcS=(G,\simeq)$. Write $G=(V,A,src,tgt)$ and $G'=(V',A',src',tgt')$. On vertices we have $V=V'$. On arrows we have $A'=\Path_G/\simeq$. The congruence $\simeq'$ for $R(L(\mcS))$ is imposed in (\ref{dia:composition formula}). Under $\simeq'$, every path of paths in $G$ is made equivalent to its concatenation, considered as a path of length 1 in $G'$. 

There is a natural transformation $\beta\taking\id_\Sch\to R\circ L$ whose $\mcS$-component sends each arrow in $G$ to a certain path of length 1 in $G'$. We need to see that $\beta_\mcS$ has an inverse. But this is straightforward: every arrow $f$ in $R\circ L(\mcS)$ is an  equivalence class of paths in $\mcS$; choose any one and send $f$ there; by Definition \ref{def:schema morphism} any other choice will give the identical morphism of schemas. It is easy to show that the roundtrips are identities (again up to the notion of identity given in Definition \ref{def:schema morphism}).

\end{proof}


\section{Limits and colimits}

Limits and colimits are universal constructions, meaning they represent certain ideals of behavior in a category. When it comes to sets that map to $A$ and $B$, the $(A\times B)$-grid is ideal---it projects on to both $A$ and $B$ as straightforwardly as possible. When it comes to sets that can interpret the elements of both $A$ and $B$, the disjoint union $A\sqcup B$ is ideal---it includes both $A$ and $B$ without confusion or superfluity. These are limits and colimits in $\Set$. Limits and colimits exist in other categories as well.

Limits in a preorder are meets, colimits in a preorder are joins. Limits and colimits also exist for database instances and monoid actions, allowing us to discuss for example the product or union of different state machines. Limits and colimits exist for spaces, giving rise to products and unions, as well as quotients.

Limits and colimits do not exist in every category; when $\mcC$ is complete with respect to limits (or colimits), these limits always seem to mean something valuable to human intuition. For example, when a subject has already been studied for a long time before category theory came around, it often turns out that classically interesting constructions in the subject correspond to limits and colimits in its categorification $\mcC$. For example products, unions, equivalence relations, etc. are classical ideas in set theory that are naturally captured by limits and colimits in $\Set$. 


\subsection{Products and coproducts in a category}

In Sections \ref{sec:prods and coprods in set}, we discussed products and coproducts in the category $\Set$ of sets. Now we discuss the same notions in an arbitrary category. For both products and coproducts we will begin with examples and then write down the general concept, but we'll work on products first.


\subsubsection{Products}\index{products}

The product of two sets is a grid, which projects down onto each of the two sets. This is good intuition for products in general.

\begin{example}\label{ex:product of preorders}

Given two preorders, $\mcX_1:=(X_1,\leq_1)$ and $\mcX_2:=(X_2,\leq_2)$, we can take their product and get a new preorder $\mcX_1\times\mcX_2$. Both $\mcX_1$ and $\mcX_2$ have underlying sets (namely $X_1$ and $X_2$), so we might hope that the underlying set of $\mcX_1\times\mcX_2$ is the set $X_1\times X_2$ of ordered pairs, and this turns out to be true. We have a notion of less-than on $\mcX_1$ and we have a notion of less-than on $\mcX_2$; we need to construct a notion of less-than on $\mcX_1\times\mcX_2$. So, given two ordered pairs $(x_1,x_2)$ and $(x_1',x_2')$, when should we say that $(x_1,x_2)\leq_{1,2}(x_1',x_2')$ holds? The obvious guess is to say that it holds iff both $x_1\leq_1x_1'$ and $x_2\leq_2x_2'$ hold, and this works:
$$\mcX_1\times\mcX_2:=(X_1\times X_2,\leq_{1,2})$$

Note that the projection functions $X_1\times X_2\to X_1$ and $X_1\times X_2\to X_2$ induce morphisms of preorders. That is, if $(x_1,x_2)\leq_{1,2}(x_1',x_2')$ then in particular $x_1\leq x_1'$. So we have preorder morphisms
$$\xymatrix@=15pt{&\mcX_1\times\mcX_2\ar[ldd]\ar[rdd]\\\\\mcX_1&&\mcX_2}$$

\end{example}

\begin{exercise}
Suppose that you have a partial order $(S,\leq_S)$ on songs (so you know some songs are preferable to others but sometimes you can't compare). And suppose you have a partial order $(A,\leq_A)$ on pieces of art. You're about to be given a pair $(s,a)$ including a song and a piece of art. Does the product partial order $\mcS\times\mcA$ provide a reasonable guess for your preferences on pairs?  
\end{exercise}

\begin{exercise}\label{exc:divides as po}
Consider the partial order $\leq$ on $\NN$ given by standard ``less-than-or-equal-to", so $5\leq 9$ etc. And consider another partial order, {\tt divides} on $\NN$, where $a\;{\tt divides}\;b$ if ``$a$ goes into $b$ evenly", i.e. if there exists $n\in\NN$ such that $a*n=b$, so $5\;{\tt divides}\;35$. If we call the product order $(X,\preceq):=(\NN,\leq)\times(\NN,{\tt divides})$, which of the following are true: 
$$(2,4)\preceq(3,4)? \hsp (2,4)\preceq(3,5)?\hsp (2,4)\preceq (8,0)?\hsp (2,4)\preceq(0,0)?$$
\end{exercise}

\begin{example}\label{ex:product of graphs}
Given two graphs $G_1=(V_1,A_1,src_1,tgt_1)$ and $G_2=(V_2,A_2,src_2,tgt_2)$, we can take their product and get a new graph $G_1\times G_2$. The vertices will be the grid of vertices $V_1\times V_2$, so each vertex in $G_1\times G_2$ is labeled by a pair of vertices, one from $G_1$ and one from $G_2$. When should an arrow connect $(v_1,v_2)$ to $(v_1',v_2')$? Whenever we can find an arrow in $G_1$ connecting $v_1$ to $v_1'$ and we can find an arrow in $G_2$ connecting $v_2$ to $v_2'$. It turns out there is a simple formula for the set of arrows in $G_1\times G_2$, namely $A_1\times A_2$.

Let's write $G:=G_1\times G_2$ and say $G=(V,A,src,tgt)$. We now know that $V=V_1\times V_2$ and $A=A_1\times A_2$. What should the source and target functions $A\to V$ be? Given a function $src_1\taking A_1\to V_1$ and a function $src_2\taking A_2\to V_2$, the universal property of products in $\Set$ (Lemma \ref{lemma:up for prod} or better Example \ref{ex:product to product}) provides a unique function 
$$src:=src_1\times src_2\taking A_1\times A_2\to V_1\times V_2$$ 
Namely the source of arrow $(a_1,a_2)$ will be the vertex $(src_1(a_1),src_2(a_2))$. Similarly we have a ready-made choice of target function $tgt=tgt_1\times tgt_2$. We have now defined the product graph.

Here's a concrete example. Let $I$ and $J$ be as drawn below:
\begin{align*}
&I:=\parbox{.8in}{\fbox{\xymatrix{\LMO{v}\ar[d]_f\\\LMO{w}\ar@/_1pc/[d]_g\ar@/^1pc/[d]^h\\\LMO{x}}}}\hspace{.6in}
&J:=\parbox{1.8in}{\fbox{\xymatrix{\LMO{q}\ar[r]^i&\LMO{r}\ar@/^1pc/[r]^j&\LMO{s}\ar@/^1pc/[l]^k\ar[r]^\ell&\LMO{t}}}}\\
&\small
\begin{array}{| l || l | l |}\bhline
\multicolumn{3}{|c|}{{\tt Arrow}\;\; (I)}\\\bhline
{\bf ID}&{\bf src}&{\bf tgt}\\\bbhline
f&v&w\\\hline
g&w&x\\\hline
h&w&x\\\bhline
\end{array}
\hsp
\begin{array}{| l ||}\bhline
\multicolumn{1}{|c|}{{\tt Vertex}\;\; (I)}\\\bhline
{\bf ID}\\\bbhline
v\\\hline
w\\\hline
x\\\bhline
\end{array}\hsp
&\small
\begin{array}{| l || l | l |}\bhline
\multicolumn{3}{|c|}{{\tt Arrow}\;\; (J)}\\\bhline
{\bf ID}&{\bf src}&{\bf tgt}\\\bbhline
i&q&r\\\hline
j&r&s\\\hline
k&s&r\\\hline
\ell&s&t\\\bhline
\end{array}
\hsp
\begin{array}{| l ||}\bhline
\multicolumn{1}{|c|}{{\tt Vertex}\;\; (J)}\\\bhline
{\bf ID}\\\bbhline
q\\\hline
r\\\hline
s\\\hline
t\\\bhline
\end{array}
\end{align*}
The product $I\times J$ drawn below has, as expected $3*4=12$ vertices and $3*4=12$ arrows: 
$$\parbox{2.4in}{\boxtitle{$I\times J:=$}\fbox{\xymatrix{
\LMO{(v,q)}\ar[rd]^{(f,i)}&\LMO{(v,r)}\ar[rd]&\LMO{(v,s)}\ar[ld]\ar[rd]&\LMO{(v,t)}\\
\LMO{(w,q)}\ar@/^1ex/[rd]\ar@/_1ex/[rd]&\LMO{(w,r)}\ar@/^1ex/[rd]\ar@/_1ex/[rd]&\LMO{(w,s)}\ar@/^1ex/[ld]\ar@/_1ex/[ld]\ar@/^1ex/[rd]\ar@/_1ex/[rd]&\LMO{(w,t)}\\
\LMO{(x,q)}&\LMO{(x,r)}&\LMO{(x,s)}&\LMO{(x,t)}
}}}
\hspace{.5in}\tiny
\begin{array}{| l || l | l |}\bhline
\multicolumn{3}{|c|}{{\tt Arrow}\;\; (I\times J)}\\\bhline
{\bf ID}&{\bf src}&{\bf tgt}\\\bbhline
(f,i)&(v,q)&(w,r)\\\hline
(f,j)&(v,r)&(w,s)\\\hline
(f,k)&(v,s)&(w,r)\\\hline
(f,\ell)&(v,s)&(w,t)\\\hline
(g,i)&(w,q)&(x,r)\\\hline
(g,j)&(w,r)&(x,s)\\\hline
(g,k)&(w,s)&(x,r)\\\hline
(g,\ell)&(w,s)&(x,t)\\\hline
(h,i)&(w,q)&(x,r)\\\hline
(h,j)&(w,r)&(x,s)\\\hline
(h,k)&(w,s)&(x,r)\\\hline
(h,\ell)&(w,s)&(x,t)\\\bhline
\end{array}
\hsp
\begin{array}{| l ||}\bhline
\multicolumn{1}{|c|}{{\tt Vertex}\;\; (I\times J)}\\\bhline
{\bf ID}\\\bbhline
(v,q)\\\hline
(v,r)\\\hline
(v,s)\\\hline
(v,t)\\\hline
(w,q)\\\hline
(w,r)\\\hline
(w,s)\\\hline
(w,t)\\\hline
(x,q)\\\hline
(x,r)\\\hline
(x,s)\\\hline
(x,t)\\\bhline
\end{array}
$$

Here is the most important thing to notice. Look at the {\tt Arrow} table for $I\times J$, and for each ordered pair, look only at the second entry in all three columns; you will see something that matches with the {\tt Arrow} table for $J$. Do the same for $I$, and again you'll see a perfect match. These ``matchings" are readily-visible graph homomorphisms $I\times J\to I$ and $I\times J\to J$ in $\Grph$. 

\end{example}

\begin{exercise}
Let $[1]=\fbox{$\LMO{0}\To{\;f\;}\LMO{1}$}$ be the linear order graph of length 1 and let $P=\Paths([1])$ be its paths-graph, as in Example \ref{ex:paths-graph} (so $P$ should have three arrows and two vertices). Draw the graph $P\times P$. 
\end{exercise}

\begin{exercise}
Recall from Example \ref{ex:dds} that a discrete dynamical system (DDS) is a set $s$ together with a function $f\taking s\to s$. By now it should be clear that if 
$$\Loop:=\LoopSchema$$\index{a schema!$\Loop$}
is the loop schema, then a DDS is simply an instance (a functor) $I\taking\Loop\to\Set$. We have not yet discussed products of DDS's, but perhaps you can guess how they should work.  For example, consider the instances $I,J\taking\Loop\to\Set$ tabulated below:
\begin{align*}
\begin{tabular}{| l || c |}\bhline
\multicolumn{2}{| c |}{s\;\; (I)}\\\bhline 
{\bf ID}&{\bf f}\\\bbhline
A & C\\\hline
B & C\\\hline
C & C\\\hline
\end{tabular}
\hspace{.7in}
\begin{tabular}{| l || c |}\bhline
\multicolumn{2}{| c |}{s\;\; (J)}\\\bhline 
{\bf ID}&{\bf f}\\\bbhline
x & y\\\hline
y & x\\\hline
z & z\\\hline
\end{tabular}
\end{align*}~
\sexc Make a guess and tabulate $I\times J$. Then draw it.\footnote{The result is not necessarily inspiring, but at least computing it is straightforward.}
\next Recall the notion of natural transformations between functors (see Example \ref{ex:graph hom as NT done out}), which in the case of functors $\Loop\to\Set$ are the morphisms of instances. Do you see clearly that there is a morphism of instances $I\times J\to I$ and $I\times J\to J$? Just check that if you look only at the left-hand coordinates in your $I\times J$, you see something compatible with $I$. 
\endsexc
\end{exercise}

In every case above, what's most important to recognize is that there are projection maps $I\times J\to I$ and $I\times J\to J$, and that the construction of $I\times J$ seems as straightforward as possible, subject to having these projections. It is time to give the definition.

\begin{definition}\label{def:products in a cat}\index{products}

Let $\mcC$ be a category and let $X,Y\in\Ob(\mcC)$ be objects. A {\em span on $X$ and $Y$} consists of three constituents $(Z,p,q)$, where $Z\in\Ob(\mcC)$ is an object, and where $p\taking Z\to X$ and $q\taking Z\to Y$ are morphisms in $\mcC$. 
$$\xymatrix@=15pt{&Z\ar[ldd]_p\ar[rdd]^q\\\\X&&Y}$$   

A {\em product of $X$ and $Y$} is a span $X\From{\pi_1}X\times Y\To{\pi_2}Y$, \footnote{The names $X\times Y$ and $\pi_1,\pi_2$ are not mathematically important, they are pedagogically suggestive.} such that for any other span $X\From{p}Z\To{q}Y$ there {\em exists a unique} morphism $t_{p,q}\taking Z\to X\times Y$ such that the diagram below commutes:
$$
\xymatrix@=15pt{&X\times Y\ar[ldd]_{\pi_1}\ar[rdd]^{\pi_2}\\\\X&&Y\\\\&Z\ar[uul]^{p}\ar[uur]_q\ar@{-->}[uuuu]^{t_{p,q}}}
$$
\end{definition}

\begin{remark}\label{rem:gateway}\index{gateway}\index{universal property}

Definition \ref{def:products in a cat} endows the product of two objects with something known as a {\em universal property}. It says that a product of two objects $X$ and $Y$ maps to those two objects, and serves as a gateway for all who do the same. ``None shall map to $X$ and $Y$ except through me!" This grandiose property is held by  products in all the various categories we have discussed so far. It is what I meant when I said things like ``$X\times Y$ maps to both $X$ and $Y$ and does so as straightforwardly as possible".  The grid of dots obtained as the product of two sets has such a property, as was shown in Example \ref{ex:grid2}.

\end{remark}

\begin{example}

In Example \ref{ex:product of preorders} we discussed products of preorders. In this example we will discuss products in an individual preorder. That is, by Proposition \ref{prop:preorders to cats}, there is a functor $\PrO\to\Cat$\index{a functor!$\PrO\to\Cat$} that realizes every preorder as a category. If $\mcP=(P,\leq)$ is a preorder, what are products in $\mcP$? Given two objects $a,b\in\Ob(\mcP)$ we first consider spans on $a$ and $b$, i.e. $a\from z\to b$. That would be some $z$ such that $z\leq a$ and $z\leq b$. The product will be such a span $a\geq a\times b\leq b$, but such that every other spanning object $z$ is less than or equal to $a\times b$. In other words $a\times b$ is as big as possible subject to the condition of being less than $a$ and less than $b$. This is precisely the meet of $a$ and $b$ (see Definition \ref{def:meets and joins}). 

\end{example}

\begin{example}\label{ex:products dont exist}\index{products!as not always existing}

Note that the product of two objects in a category $\mcC$ may not exist. Let's return to preorders to see this phenomenon.

Consider the set $\RR^2$, and say that $(x_1,y_1)\leq (x_2,y_2)$ if there exists $\ell\geq 1$ such that $x_1\ell=x_2$ and $y_1\ell=y_2$; in other words, point $p$ is less than point $q$ if, in order to travel from $q$ to the origin along a straight line, one must pass through $p$ along the way. 
\footnote{Note that $(0,0)$ is not related to anything else.} 
We have given a perfectly good partial order, but $p:=(1,0)$ and $q:=(0,1)$ do not have a product. Indeed, it would have to be a non-zero point that was on the same line-through-the origin as $p$ and the same line-through-the-origin as $q$, of which there are none.

\end{example}

\begin{example}

Note that there can be more than one product of two objects in a category $\mcC$, but that any two choices will be canonically isomorphic. Let's return once more to preorders to see this phenomenon.

Consider the set $\RR^2$ and say that $(x_1,y_1)\leq (x_2,y_2)$ if $x_1^2+y_1^2\leq x_2^2+y_2^2$, in other words if the former is on a smaller 0-circle (by which I mean ``circle centered at the origin") than the latter is. 

For any two points $p,q$ there will be lots of points that serve as products: anything on the smaller of their two 0-circles will suffice. Given any two points $a,b$ on this smaller circle, we will have a unique isomorphism $a\iso b$ because $a\leq b$ and $b\leq a$ and all morphisms are unique in a preorder.

\end{example}

\begin{exercise}
Consider the preorder $\mcP$ of cards in a deck, shown in Example \ref{ex:pre not par}; it is not the entire story of cards in a deck, but take it to be so. In other words, be like a computer and take what's there at face value. Consider the preorder $\mcP$ as a category (by way of the functor $\PrO\to\Cat$\index{a functor!$\PrO\to\Cat$}).
\sexc For each of the following pairs, what is their product in $\mcP$ (if it exists)?
\begin{align*}
&\fakebox{a diamond}\times\fakebox{a heart}\;? \hsp &\fakebox{a queen}\times\fakebox{a black card}\;?\\
& \fakebox{a card}\times\fakebox{a red card}\;?\hsp&\fakebox{a face card}\times\fakebox{a black card}\;?
\end{align*}
\next How would these answers differ if $\mcP$ was completed to the ``whole story" partial order classifying cards in a deck?
\endsexc
\end{exercise}

\begin{exercise}
Let $X$ be a set, and consider it as a discrete category. Given two objects $x,y\in\Ob(X)$, under what conditions will there exist a product $x\times y$?
\end{exercise}

\begin{exercise}
Let $f\taking\RR\to\RR$ be a function, like you would see in 6th grade (maybe $f(x)=x+7$). A typical thing to do is to graph $f$ as a curve running through the plane $\RR^2:=\RR\times\RR$. This curve can be understood as a function $F\taking\RR\to\RR^2$.
\sexc Given some $x\in\RR$, what are the coordinates of $F(x)\in\RR^2$? 
\next Obtain $F\taking\RR\to\RR^2$ using the universal property given in Definition \ref{def:products in a cat}. 
\endsexc
\end{exercise}

\begin{exercise}
Consider the preorder $(\NN,{\tt divides})$, discussed in Exercise \ref{exc:divides as po}, where e.g. $5\leq 15$ but $5\not\leq 6$. \sexc What is the product of $9$ and $12$ in this category?
\next Is there a standard name for products in this category?
\endsexc
\end{exercise}

\begin{example}\label{ex:[1]x[1]}

All products exist in the category $\Cat$. Given two categories $\mcC$ and $\mcD$, there is a product category $\mcC\times\mcD$. We have $\Ob(\mcC\times\mcD)=\Ob(\mcC)\times\Ob(\mcD)$ and for any two objects $(c,d)$ and $(c',d')$, we have $$\Hom_{\mcC\times\mcD}((c,d),(c',d'))=\Hom_\mcC(c,c')\times\Hom_\mcC(d,d').$$ The composition formula is ``obvious".

Let $[1]\in\Ob(\Cat)$ denote the linear order category of length 1, drawn $$[1]:=\fbox{\xymatrix{\LMO{0}\ar[r]^f&\LMO{1}}}$$ As a schema it has one arrow, but as a category it has three morphisms. So we expect $[1]\times[1]$ to have 9 morphisms, and that's true. In fact, $[1]\times[1]$ looks like a commutative square:
\begin{align}\label{dia:comm square}
\xymatrix@=40pt{
\LMO{(0,0)}\ar[r]^{\id_0\times f}\ar[d]_{f\times\id_0}&\LMO{(0,1)}\ar[d]^{f\times\id_1}\\
\LMO{(1,0)}\ar[r]_{\id_1\times f}&\LMO{(1,1)}}
\end{align}
We see only four morphisms here, but there are also four identities and one morphism $(0,0)\to(1,1)$ given by composition of either direction. It is a minor miracle that the categorical product somehow ``knows" that this square should commute; however, this is not the mere preference of man but instead the dictate of God! By which I mean, this follows rigorously from the definitions we already gave of $\Cat$ and products.

\end{example}


\subsubsection{Coproducts}\index{coproducts}

The coproduct of two sets is their disjoint union, which includes non-overlapping copies of each of the two sets. This is good intuition for coproducts in general.

\begin{example}

Given two preorders, $\mcX_1:=(X_1,\leq_1)$ and $\mcX_2:=(X_2,\leq_2)$, we can take their coproduct and get a new preorder $\mcX_1\sqcup\mcX_2$. Both $\mcX_1$ and $\mcX_2$ have underlying sets (namely $X_1$ and $X_2$), so we might hope that the underlying set of $\mcX_1\times\mcX_2$ is the disjoint union $X_1\sqcup X_2$, and that turns out to be true. We have a notion of less-than on $\mcX_1$ and we have a notion of less-than on $\mcX_2$. 

Given an element $x\in X_1\sqcup X_2$ and an element $x'\in X_1\sqcup X_2$, how can we use $\leq_1$ and $\leq_2$ to compare $x_1$ and $x_2$? The relation $\leq_1$ only knows how to compare elements of $X_1$ and the relation $\leq_2$ only knows how to compare elements of $X_2$. But $x$ and $x'$ may come from different homes; e.g. $x\in X_1$ and $x'\in X_2$, in which case neither $\leq_1$ nor $\leq_2$ gives any clue about which should be bigger. 

So when should we say that $x\leq_{1\sqcup 2} x'$ holds? The obvious guess is to say that $x$ is less than $x'$ iff somebody says it is; that is, if both $x$ and $x'$ are from the same home and the local ordering has $x\leq x'$. To be precise, we say $x\leq_{1\sqcup 2}x'$ if and only if either one of the following conditions hold:
\begin{itemize}
\item $x\in X_1$ and $x'\in X_1$ and $x\leq_1x'$, or
\item $x\in X_2$ and $x'\in X_2$ and $x\leq_2x'$.
\end{itemize}
With $\leq_{1\sqcup 2}$ so defined, one checks that it is not only a preorder, but that it serves as a coproduct of $\mcX_1$ and $\mcX_2$, 
$$\mcX_1\sqcup\mcX_2:=(X_1\sqcup X_2,\leq_{1\sqcup 2}).$$

Note that the inclusion functions $X_1\to X_1\sqcup X_2$ and $X_2\to X_1\sqcup X_2$ induce morphisms of preorders. That is, if $x,x'\in X_1$ are elements such that $x\leq_1x'$ in $\mcX_1$ then the same will hold in $\mcX_1\sqcup\mcX_2$. So we have preorder morphisms
$$\xymatrix@=15pt{&\mcX_1\sqcup\mcX_2\\\\\mcX_1\ar[ruu]&&\mcX_2\ar[luu]}$$

\end{example}

\begin{exercise}
Suppose that you have a partial order $\mcA:=(A,\leq_A)$ on apples (so you know some apples are preferable to others but sometimes you can't compare). And suppose you have a partial order $\mcO:=(O,\leq_O)$ on oranges. You're about to be given two pieces of fruit from a basket of apples and oranges. Is the coproduct partial order $\mcA\sqcup\mcO$ a reasonable guess for your preferences, or does it seem biased?
\end{exercise}

\begin{example}\label{ex:coproduct of graphs}
Given two graphs $G_1=(V_1,A_1,src_1,tgt_1)$ and $G_2=(V_2,A_2,src_2,tgt_2)$, we can take their coproduct and get a new graph $G_1\sqcup G_2$. The vertices will be the disjoint union of vertices $V_1\sqcup V_2$, so each vertex in $G_1\sqcup G_2$ is labeled either by a vertex in $G_1$ or by one in $G_2$ (and if any labels are shared, then something must be done to differentiate them). When should an arrow connect $v$ to $v'$? Whenever both are from the same component (i.e. either $v,v'\in V_1$ or $v,v'\in V_2$) and we can find an arrow connecting them in that component. It turns out there is a simple formula for the set of arrows in $G_1\sqcup G_2$, namely $A_1\sqcup A_2$.

Let's write $G:=G_1\sqcup G_2$ and say $G=(V,A,src,tgt)$. We now know that $V=V_1\sqcup V_2$ and $A=A_1\sqcup A_2$. What should the source and target functions $A\to V$ be? Given a function $src_1\taking A_1\to V_1$ and a function $src_2\taking A_2\to V_2$, the universal property of coproducts in $\Set$ can be used to specify a unique function 
$$src:=src_1\sqcup src_2\taking A_1\sqcup A_2\to V_1\sqcup V_2.$$ 
Namely for any arrow $a\in A$, we know either $a\in A_1$ or $a\in A_2$ (and not both), so the source of $a$ will be the vertex $src_1(a)$ if $a\in A_1$ and $src_2(a)$ if $a\in A_2$. Similarly we have a ready-made choice of target function $tgt=tgt_1\sqcup tgt_2$. We have now defined the coproduct graph.

Here's a real example. Let $I$ and $J$ be as in Example \ref{ex:graph hom as NT done out}, drawn below:
\begin{align*}
&I:=\parbox{.8in}{\fbox{\xymatrix{\LMO{v}\ar[d]_f\\\LMO{w}\ar@/_1pc/[d]_g\ar@/^1pc/[d]^h\\\LMO{x}}}}\hspace{.6in}
&J:=\parbox{1.8in}{\fbox{\xymatrix{\LMO{q}\ar[r]^i&\LMO{r}\ar@/^1pc/[r]^j&\LMO{s}\ar@/^1pc/[l]^k\ar[r]^\ell&\LMO{t}\\&&&\LMO{u}}}}\\
&\small
\begin{array}{| l || l | l |}\bhline
\multicolumn{3}{|c|}{{\tt Arrow}\;\; (I)}\\\bhline
{\bf ID}&{\bf src}&{\bf tgt}\\\bbhline
f&v&w\\\hline
g&w&x\\\hline
h&w&x\\\bhline
\end{array}
\hsp
\begin{array}{| l ||}\bhline
\multicolumn{1}{|c|}{{\tt Vertex}\;\; (I)}\\\bhline
{\bf ID}\\\bbhline
v\\\hline
w\\\hline
x\\\bhline
\end{array}\hsp
&\small
\begin{array}{| l || l | l |}\bhline
\multicolumn{3}{|c|}{{\tt Arrow}\;\; (J)}\\\bhline
{\bf ID}&{\bf src}&{\bf tgt}\\\bbhline
i&q&r\\\hline
j&r&s\\\hline
k&s&r\\\hline
\ell&s&t\\\bhline
\end{array}
\hsp
\begin{array}{| l ||}\bhline
\multicolumn{1}{|c|}{{\tt Vertex}\;\; (J)}\\\bhline
{\bf ID}\\\bbhline
q\\\hline
r\\\hline
s\\\hline
t\\\hline
u\\\bhline
\end{array}
\end{align*}
The coproduct $I\sqcup J$ drawn below has, as expected $3+5=8$ vertices and $3+4=7$ arrows: 
$$\parbox{2.4in}{\boxtitle{$I\sqcup J:=$}\fbox{\xymatrix{
\LMO{v}\ar[d]_f\\
\LMO{w}\ar@/_1pc/[d]_g\ar@/^1pc/[d]^h&\LMO{q}\ar[r]^i&\LMO{r}\ar@/^1pc/[r]^j&\LMO{s}\ar@/^1pc/[l]^k\ar[r]^\ell&\LMO{t}\\
\LMO{x}&&&&\LMO{u}
}}}
\hspace{.5in}\small
\begin{array}{| l || l | l |}\bhline
\multicolumn{3}{|c|}{{\tt Arrow}\;\; (I\sqcup J)}\\\bhline
{\bf ID}&{\bf src}&{\bf tgt}\\\bbhline
f&v&w\\\hline
g&w&x\\\hline
h&w&x\\\hline
i&q&r\\\hline
j&r&s\\\hline
k&s&r\\\hline
\ell&s&t\\\bhline
\end{array}
\hsp
\begin{array}{| l ||}\bhline
\multicolumn{1}{|c|}{{\tt Vertex}\;\; (I\sqcup J)}\\\bhline
{\bf ID}\\\bbhline
v\\\hline
w\\\hline
x\\\hline
q\\\hline
r\\\hline
s\\\hline
t\\\hline
u\\\bhline
\end{array}
$$

Here is the most important thing to notice. Look at the {\tt Arrow} table $I$ and notice that there is a way to send each row to a row in $I\sqcup J$, such that all the foreign keys match. Similarly in the arrow table and the two vertex tables for $J$. These ``matchings" are readily-visible graph homomorphisms $I\to I\sqcup J$ and $J\to I\sqcup J$ in $\Grph$. 

\end{example}

\begin{exercise}
Recall from Example \ref{ex:dds} that a discrete dynamical system (DDS) is a set $s$ together with a function $f\taking s\to s$; if 
$$\Loop:=\LoopSchema$$
is the loop schema, then a DDS is simply an instance (a functor) $I\taking\Loop\to\Set$. We have not yet discussed coproducts of DDS's, but perhaps you can guess how they should work.  For example, consider the instances $I,J\taking\Loop\to\Set$ tabulated below:
\begin{align*}
\begin{tabular}{| l || c |}\bhline
\multicolumn{2}{| c |}{s\;\; (I)}\\\bhline 
{\bf ID}&{\bf f}\\\bbhline
A & C\\\hline
B & C\\\hline
C & C\\\hline
\end{tabular}
\hspace{.7in}
\begin{tabular}{| l || c |}\bhline
\multicolumn{2}{| c |}{s\;\; (J)}\\\bhline 
{\bf ID}&{\bf f}\\\bbhline
x & y\\\hline
y & x\\\hline
z & z\\\hline
\end{tabular}
\end{align*}
Make a guess and tabulate $I\sqcup J$. Then draw it.
\end{exercise}

In every case above (preorders, graphs, DDSs), what's most important to recognize is that there are inclusion maps $I\to I\sqcup J$ and $J\to I\sqcup J$, and that the construction of $I\sqcup J$ seems as straightforward as possible, subject to having these inclusions. It is time to give the definition.

\begin{definition}\label{def:coproducts in a cat}

Let $\mcC$ be a category and let $X,Y\in\Ob(\mcC)$ be objects. A {\em cospan on $X$ and $Y$}\index{cospan} consists of three constituents $(Z,i,j)$, where $Z\in\Ob(\mcC)$ is an object, and where $i\taking X\to Z$ and $j\taking Y\to Z$ are morphisms in $\mcC$. 
$$\xymatrix@=15pt{&Z\\\\X\ar[ruu]^i&&Y\ar[luu]_j}$$   

A {\em coproduct of $X$ and $Y$} is a cospan $X\To{\iota_1}X\sqcup Y\From{\iota_2}Y$, \footnote{The names $X\sqcup Y$ and $\iota_1,\iota_2$ are not mathematically important, they are pedagogically suggestive.} such that for any other cospan $X\To{i}Z\From{j}Y$ there {\em exists a unique} morphism $s_{i,j}\taking X\sqcup Y\to Z$ such that the diagram below commutes:
$$
\xymatrix@=15pt{&X\sqcup Y\ar@{-->}[dddd]_{s_{i,j}}\\\\X\ar[uur]^{\iota_1}\ar[ddr]_i&&Y\ar[uul]_{\iota_2}\ar[ddl]^j\\\\&Z}
$$
\end{definition}

\begin{remark}

Definition \ref{def:products in a cat} endows the coproduct of two objects with a {\em universal property}. It says that a coproduct of two objects $X$ and $Y$ receives maps from those two objects, and serves as a gateway for all who do the same. ``None shall receive maps from $X$ and $Y$ except through me!" This grandiose property is held by all the coproducts we have discussed so far. It is what I meant when I said things like ``$X\sqcup Y$ receives maps from both $X$ and $Y$ and does so as straightforwardly as possible".  The disjoint union of dots obtained as the coproduct of two sets has such a property, as can be seen by thinking about Example \ref{ex:coprod of dots}.

\end{remark}

\begin{example}

By Proposition \ref{prop:preorders to cats}, there is a functor $\PrO\to\Cat$\index{a functor!$\PrO\to\Cat$} that realizes every preorder as a category. If $\mcP=(P,\leq)$ is a preorder, what are coproducts in $\mcP$? Given two objects $a,b\in\Ob(\mcP)$ we first consider cospans on $a$ and $b$, i.e. $a\to z\from b$. A cospan of $a$ and $b$ is any $z$ such that $a\leq z$ and $b\leq z$. The coproduct will be such a cospan $a\leq a\sqcup b\geq b$, but such that every other cospanning object $z$ is greater than or equal to $a\sqcup b$. In other words $a\sqcup b$ is as small as possible subject to the condition of being bigger than $a$ and bigger than $b$. This is precisely the join of $a$ and $b$ (see Definition \ref{def:meets and joins}).

\end{example}

Just as for products, the coproduct of two objects in a category $\mcC$ may not exist, or it may not be unique. The non-uniqueness is much less ``bad" because given two candidate coproducts, they will be canonically isomorphic. They may not be equal, but they are isomorphic. But coproducts might not exist at all in certain categories. We will explore that a bit below.

\begin{example}

Consider the set $\RR^2$ and partial order from Example \ref{ex:products dont exist} where $(x_1,y_1)\leq (x_2,y_2)$ if there exists $\ell\geq 1$ such that $x_1\ell=x_2$ and $y_1\ell=y_2$. Again the points $p:=(1,0)$ and $q:=(0,1)$ do not have a coproduct. Indeed, it would have to be a non-zero point that was on the same line-through-the origin as $p$ and the same line-through-the-origin as $q$, of which there are none.

\end{example}

\begin{exercise}
Consider the preorder $\mcP$ of cards in a deck, shown in Example \ref{ex:pre not par}; it is not the entire story of cards in a deck, but take it to be so. In other words, be like a computer and take what's there at face value. Consider the preorder $\mcP$ as a category (by way of the functor $\PrO\to\Cat$). For each of the following pairs, what is their coproduct in $\mcP$ (if it exists)?
\sexc 
\begin{tabbing}
\hspace{.5in}\= \fakebox{a diamond}$\sqcup$\fakebox{a heart}\;?\hspace{.5in} \=\fakebox{a queen}$\sqcup$\fakebox{a black card}\;?\\\\
\> \fakebox{a card}$\sqcup$\fakebox{a red card}\;?\>\fakebox{a face card}$\sqcup$\fakebox{a black card}\;?
\end{tabbing}
\next How would these answers differ if $\mcP$ was completed to the ``whole story" partial order classifying cards in a deck?
\endsexc
\end{exercise}

\begin{exercise}
Let $X$ be a set, and consider it as a discrete category. Given two objects $x,y\in\Ob(X)$, under what conditions will there exist a coproduct $x\sqcup y$?
\end{exercise}

\begin{exercise}
Consider the preorder $(\NN,{\tt divides})$, discussed in Exercise \ref{exc:divides as po}, where e.g. $5\leq 15$ but $5\not\leq 6$. \sexc What is the coproduct of $9$ and $12$ in that category?
\next Is there a standard name for coproducts in that category?
\endsexc
\end{exercise}


\subsection{Diagrams in a category}\label{sec:diagrams in a category}\index{diagram}

We have been drawing diagrams since the beginning of the book. What is it that we have been drawing pictures {\em of}? The answer is that we have been drawing functors.

\begin{definition}\index{diagram}\index{indexing category}

Let $\mcC$ and $I$ be categories.
\footnote{In fact, the indexing category $I$ is usually assumed to be small in the sense of Remark \ref{rmk:small}, meaning that its collection of objects is a set.}
An {\em $I$-shaped diagram in $\mcC$} is simply a functor $d\taking I\to\mcC$. In this case $I$ is called the {\em indexing category} for the diagram.

\end{definition}

Suppose given an indexing category $I$ and an $I$-shaped diagram $X\taking I\to\mcC$. One draws this as follows. For each object in $q\in I$, draw a dot labeled by $X(q)$; if several objects in $I$ point to the same object in $\mcC$, then several dots will be labeled the same way. Draw the images of morphisms $f\taking q\to q'$ in $I$ by drawing arrows between dots $X(q)$ and $X(q')$, and label each arrow by the image morphism $X(f)$ in $\mcC$. Again, if several morphisms in $I$ are sent to the same morphism in $\mcC$, then several arrows will be labeled the same way. One can abbreviate this process by not drawing {\em every} morphism in $I$, so long as every morphism in $I$ is represented by a unique path in $\mcC$, i.e. as long as the drawing is sufficiently unambiguous as a depiction of $X\taking I\to\mcC$.

\begin{example}\label{ex:comm vs noncomm diags}

Consider the commutative diagram in $\Set$ drawn below:
\begin{align}\label{dia:comm diag of nats in set}
\xymatrix{\NN\ar[r]^{+1}\ar[d]_{*2}&\NN\ar[d]^{*2}\\\NN\ar[r]_{+2}&\ZZ}
\end{align}
This is the drawing of a functor $d\taking[1]\times[1]\to\Set$ (see Example \ref{ex:[1]x[1]}). With notation for the objects and morphisms of $[1]\times[1]$ as shown in Diagram (\ref{dia:comm square}), we have $d(0,0)=d(0,1)=d(1,0)=\NN$ and $d(1,1)=\ZZ$ (for some reason..) and $d(\id_0,f)\taking\NN\to\NN$ given by $n\mapsto n+1$, etc. 

The fact that $d$ is a functor means it must respect composition formulas, which implies that Diagram (\ref{dia:comm diag of nats in set}) commutes. Recall from Section \ref{sec:comm diag} that not all diagrams one can draw will commute; one must specify that a given diagram commutes if he or she wishes to communicate this fact. But then how is a {\em non-commuting diagram} to be understood as a functor?

Let $G\in\Ob(\Grph)$ denote the following graph 
$$\xymatrix{\LMO{(0,0)}\ar[r]^f\ar[d]_h&\LMO{(0,1)}\ar[d]^{g}\\\LMO{(1,0)}\ar[r]_{i}&\LMO{(1,1)}}$$
Recall the free category functor $F\taking\Grph\to\Cat$ from Example \ref{ex:free category}. The free category $F(G)\in\Ob(\Cat)$ on $G$ looks almost like $[1]\times[1]$ except that since $[f,g]$ is a different path in $G$ than is $[h,i]$, they become different morphisms in $F(G)$. A functor $F(G)\to\Set$ might be drawn the same way that (\ref{dia:comm diag of nats in set}) is, but it would be a diagram that would {\em not} be said to commute.

We call $[1]\times [1]$ the {\em commutative square indexing category}. 
\footnote{We might call what is here denoted by $F(G)$ the {\em noncommutative square indexing category}.}

\end{example}

\begin{exercise}
Consider $[2]$, the linear order category of length 2.
\sexc Is $[2]$ the appropriate indexing category for commutative triangles?
\next If not, what is?
\endsexc
\end{exercise}

\begin{example}

Recall that an equalizer in $\Set$ was a diagram of sets that looked like this:
\begin{align}\label{dia:equalizer diag}
\xymatrix{\LMO{E}\ar[r]^f&\LMO{A}\ar@<.5ex>[r]^{g_1}\ar@<-.5ex>[r]_{g_2}&\LMO{B}}
\end{align}
where $g_1\circ f=g_2\circ f$. What is the indexing category for such a diagram? It is the schema (\ref{dia:equalizer diag}) with the PED $[f,g_1]\simeq[f,g_2]$. That is, in some sense you're seeing the indexing category, but the PED needs to be declared.

\end{example}

\begin{exercise}\label{exc:coincidence}
Let $\mcC$ be a category, $A\in\Ob(\mcC)$ an object, and $f\taking A\to A$ a morphism in $\mcC$. Consider the two diagrams in $\mcC$ drawn below:
$$\fbox{\xymatrix{\LMO{A}\ar[r]^f&\LMO{A}\ar[r]^f&\LMO{A}\ar[r]^f&\cdots}}\hspace{1in}\fbox{\xymatrix{\LMO{A}\ar@(ul,dl)[]_f}}$$
\sexc Should these two diagrams have the same indexing category?
\next If they should have the same indexing category, what is causing or allowing the pictures to appear different?
\next If they should not have the same indexing category, what coincidence makes the two pictures have so much in common?
\endsexc
\end{exercise}

\begin{definition}\label{def:lcone}\index{cone!left}

Let $I\in\Ob(\Cat)$ be a category. The {\em left cone on $I$}, denoted $I\lcone$\index{a symbol!$\lcone$}, is the category defined as follows. On objects we put $\Ob(I\lcone)=\{-\infty\}\sqcup\Ob(I)$, and we call the new object $-\infty$ the {\em cone point of $I\lcone$}. On morphisms we add a single new morphism $s_b\taking-\infty\to b$ for every object $b\in\Ob(I)$; more precisely,
$$\Hom_{I\lcone}(a,b)=
\begin{cases}
\Hom_I(a,b)&\tn{ if }a,b\in\Ob(I)\\
\{s_b\}&\tn{ if }a=-\infty, b\in\Ob(I)\\
\{\id_{-\infty}\}&\tn{ if } a=b=-\infty\\
\emptyset&\tn{ if } a\in\Ob(I), b=-\infty.
\end{cases}$$
The composition formula is in some sense obvious. To compose two morphisms both in $I$, compose as dictated by $I$; if one has $-\infty$ as source then there will be a unique choice of composite.

There is an obvious inclusion of categories,
\begin{align}\label{dia:inclusion into cone}
I\to I\lcone.
\end{align}

\end{definition}

\begin{remark}\label{rem:schemas are cats!}

Note that the specification of $I\lcone$ given in Definition \ref{def:lcone} works just as well if $I$ is considered a schema and we are constructing a schema $I\lcone$: add the new object $-\infty$ and the new arrows $s_b\taking-\infty\to b$ for each $b\in\Ob(I)$, and for every morphism $f\taking b\to b'$ in $I$ add a PED $[s_{b'}]\simeq[s_b,f]$. We generally will not distinguish between categories and schemas, since they are equivalent.

\end{remark}

\begin{example}\label{ex:stars}\index{a category!$\Star_n$}

For a natural number $n\in\NN$, we define the {\em $n$-leaf star schema}, denoted $\Star_n$, to be the category (or schema, see Remark \ref{rem:schemas are cats!}) $\ul{n}\lcone$, where $\ul{n}$ is the discrete category on $n$ objects. Below we draw $\Star_0, \Star_1, \Star_2$, and $\Star_3$.
$$
\parbox{.3in}{\boxtitle{$\Star_0$}\fbox{$\LMO{-\infty}$}}
\hspace{.5in}
\parbox{.4in}{\boxtitle{$\Star_1$}\fbox{\xymatrix@=15pt{\LMO{-\infty}\ar[dd]_{s_1}\\\\\LMO{1}}}}
\hspace{.5in}
\parbox{1.1in}{\boxtitle{$\Star_2$}\fbox{\xymatrix@=15pt{&\LMO{-\infty}\ar[ddl]_{s_1}\ar[ddr]^{s_2}\\\\\LMO{1}&&\LMO{2}}}}
\hspace{.5in}
\parbox{1.1in}{\boxtitle{$\Star_3$}\fbox{\xymatrix@=15pt{&\LMO{-\infty}\ar[ddl]_{s_1}\ar[dd]_{s_2}\ar[ddr]^{s_3}\\\\\LMO{1}&\LMO{2}&\LMO{3}}}}
$$

\end{example}

\begin{exercise}
Let $\mcC_0:=\ul{0}$ denote the empty category and for any natural number $n\in\NN$, let $\mcC_{n+1}=(\mcC_n)\lcone.$ Draw $\mcC_4$.  
\end{exercise}

\begin{exercise}
Let $\mcC$ be the graph indexing schema as in (\ref{dia:graph index}). What is $\mcC\lcone$ and how does it compare to (\ref{dia:equalizer diag})? 
\end{exercise}

\begin{definition}\label{def:rcone}\index{cone!right}

Let $I\in\Ob(\Cat)$ be a category. The {\em right cone on $I$}, denoted $I\rcone$,\index{a symbol!$\rcone$} is the category defined as follows. On objects we put $\Ob(I\rcone)=\Ob(I)\sqcup\{\infty\}$, and we call the new object $\infty$ the {\em cone point of $I\rcone$}. On morphisms we add a single new morphism $t_b\taking b\to\infty$ for every object $b\in \Ob(I)$; more precisely,
$$\Hom_{I\rcone}(a,b)=
\begin{cases}
\Hom_I(a,b)&\tn{ if }a,b\in\Ob(I)\\
\{t_b\}&\tn{ if }a\in\Ob(I), b=\infty\\
\{\id_{\infty}\}&\tn{ if } a=b=\infty\\
\emptyset&\tn{ if } a=\infty,b\in\Ob(I).
\end{cases}$$
The composition formula is in some sense obvious. To compose two morphisms both in $I$, compose as dictated by $I$; if one has $\infty$ as target then there will be a unique choice of composite.

There is an obvious inclusion of categories $I\to I\rcone$.

\end{definition}

\begin{exercise}
Let $\mcC$ be the category $(\ul{2}\lcone)\rcone$, where $\ul{2}$ is the discrete category on two objects. Then $\mcC$ is somehow square-shaped, but what category is it exactly? Looking at Example \ref{ex:comm vs noncomm diags}, is $\mcC$ the commutative diagram indexing category $[1]\times[1]$, is it the non-commutative diagram indexing category $F(G)$, or is it something else?
\end{exercise}


\subsection{Limits and colimits in a category}\label{sec:lims and colims in a cat}

Let $\mcC$ be a category, let $I$ be an indexing category (which just means that $I$ is a category that we're about to use as the indexing category for a diagram), and let $D\taking I\to\mcC$ an $I$-shaped diagram (which just means a functor). It is in relation to this setup that we can discuss the limit or colimit. In general the limit of a diagram $D\taking I\to\mcC$ will be a $I\lcone$ shaped diagram $\lim D\taking I\lcone\to\mcC$. In the case of products $I=\ul{2}$ and $I\lcone=\Star_2$ looks like a span (see Example \ref{ex:stars}). But out of all the $I\lcone$-shaped diagrams, which is the limit of $D$? Answer: the one with the universal ``gateway" property, see Remark \ref{rem:gateway}.


\subsubsection{Universal objects}

\begin{definition}\index{initial object}\index{terminal object}

Let $\mcC$ be a category. An object $a\in\Ob(\mcC)$ is called {\em initial} if, for all objects $c\in\Ob(\mcC)$ there exists a unique morphism $a\to c$, i.e. $|\Hom_\mcC(a,c)|=1$. An object $z\in\Ob(\mcC)$ is called {\em terminal} if, for all objects $c\in\Ob(\mcC)$ there is exists a unique morphism $c\to z$, i.e. $|\Hom_\mcC(c,z)|=1$. 

\end{definition}

An object in a category is called {\em universal} if it is either initial or terminal, but we rarely use that term in practice, preferring to be specific about whether the object is initial or terminal. The word {\em final} is synonymous with the word terminal, but we'll try to constantly use terminal. 

Colimits will end up being defined as initial things of a certain sort, and limits will end up being defined as terminal things of a certain sort. But we will get to that in Section \ref{sec:examples of limits}.

\begin{warning}\index{a warning!misuse of {\em the}}

A category $\mcC$ may have more than one initial object; similarly a category $\mcC$ may have more than one terminal object. We will see in Example \ref{ex:universal obs in set} that any set with one element, e.g. $\{*\}$ or $\singleton$, is a terminal object in $\Set$. These terminal sets have the same number of elements, but they are not the exact-same set; two sets having the same cardinality means precisely that there exists an isomorphism between them.

In fact, Proposition \ref{prop:initials are isomorphic} below shows that in any category $\mcC$, any two terminal objects in $\mcC$ are isomorphic (similarly, any two initial objects in $\mcC$ are isomorphic). While there are many isomorphisms in $\Set$ between $\{1,2,3\}$ and $\{a,b,c\}$, there is only one isomorphism between $\{*\}$ and $\smiley$. This is always the case for universal objects: there is a unique isomorphism between any two terminal (respectively initial) objects in any category.

As a result, people often speak of {\em the} initial object in $\mcC$ or {\em the} terminal object in $\mcC$, as though there was only one. ``It's unique up to unique ismorphism!" is the justification for this use of the so-called definite article {\em the} rather than the indefinite article {\em a}. This is not a very misleading way of speaking, because just like the president today does not contain exactly the same atoms as the president yesterday, the difference is unimportant. But we still mention this as a warning: if $\mcC$ has a terminal object, we may speak of it as though it were unique, calling it {\em the terminal object}, and similarly for initial objects.

We will use the definite article throughout this document, e.g. in Example \ref{ex:universal obs in set} we will discuss the initial object in $\Set$ and the terminal object in $\Set$. This is common throughout mathematical literature as well.

\end{warning}

\begin{proposition}\label{prop:initials are isomorphic}

Let $\mcC$ be a category and let $a_1,a_2\in\Ob(\mcC)$ both be initial objects. Then there is a unique isomorphism $a_1\To{\iso}a_2$. (Similarly, for any two terminal objects in $\mcC$ there is a unique isomorphism between them.) 

\end{proposition}

\begin{proof}

Suppose $a_1$ and $a_2$ are initial. Since $a_1$ is initial there is a unique morphism $f\taking a_1\to a_2$; there is also a unique morphism $a_1\to a_1$, which must be $\id_{a_1}$. Since $a_2$ is initial there is a unique morphism $g\taking a_2\to a_1$; there is also a unique morphism $a_2\to a_2$, which must be $\id_{a_2}$. So $g\circ f=\id_{a_1}$ and $f\circ g=\id_{a_2}$, which means that $f$ is the desired (unique) isomorphism.

The proof for terminal objects is appropriately ``dual".

\end{proof}

\begin{example}\label{ex:universal obs in set}

The initial object in $\Set$ is the set $a$ for which there is always one way to map from $a$ to anything else. Given $c\in\Ob(\Set)$ there is exactly one function $\emptyset\to c$, because there are no choices to be made, so the empty set $\emptyset$ is the initial object in $\Set$.

The terminal object in $\Set$ is the set $z$ for which there is always one way to map to $z$ from anything else. Given $c\in\Ob(\Set)$ there is exactly one function $c\to\singleton$, where $\singleton$ is any set with one element, because there are no choices to be made: everything in $c$ must be sent to the single element in $\singleton$. There are lots of terminal objects in $\Set$, and they are all isomorphic to $\ul{1}$.

\end{example}

\begin{example}

The initial object in $\Grph$ is the graph $a$ for which there is always one way to map from $a$ to anything else. Given $c\in\Ob(\Grph)$, there is exactly one function $\emptyset\to c$, where $\emptyset\in\Grph$ is the empty graph; so $\emptyset$ is the initial object.

The terminal object in $\Grph$ is more interesting. It is $\Loop$, the graph with one vertex and one arrow. In fact there are infinitely many terminal objects in $\Grph$, but all of them are isomorphic to $\Loop$. 

\end{example}

\begin{exercise}
Let $X$ be a set, let $\PP(X)$ be the set of subsets of $X$ (see Definition \ref{def:subobject classifier}). We can regard $\PP(X)$ as a preorder under inclusion of subsets (see for example Section \ref{sec:meets and joins}). And we can regard preorders as categories using a functor $\PrO\to\Cat$ (see Proposition \ref{prop:preorders to cats}).
\sexc What is the initial object in $\PP(X)$?
\next What is the terminal object in $\PP(X)$? 
\endsexc
\end{exercise}

\begin{example}\label{ex:initial monoid terminal monoid}\index{monoid!initial}\index{monoid!terminal}

The initial object in the category $\Mon$ of monoids is the trivial monoid, $\ul{1}$. For any monoid $M$, a morphism of monoids $\ul{1}\to M$ is a functor between 1-object categories and these are determined by where they send morphisms. Since $\ul{1}$ has only the identity morphism and functors must preserve identities, there is no choice involved in finding a monoid morphism $\ul{1}\to M$.

Similarly, the terminal object in $\Mon$ is also the trivial monoid, $\ul{1}$. For any monoid $M$, a morphism of monoids $M\to\ul{1}$ sends everything to the identity; there is no choice.

\end{example}

\begin{exercise}~
\sexc What is the initial object in $\Grp$, the category of groups?
\next What is the terminal object in $\Grp$?
\endsexc
\end{exercise}
\begin{example}

Recall the preorder $\Prop$ of logical propositions from Section \ref{sec:propositions}. The initial object is a proposition that implies all others. It turns out that ``FALSE" is such a proposition. The proposition ``FALSE" is like ``$1\neq1$"; in logical formalism it can be shown that if ``FALSE" is true then everything is true.

The terminal object in $\Prop$ is a proposition that is implied by all others. It turns out that ``TRUE" is such a proposition. In logical formalism, everything implies that ``TRUE" is true.

\end{example}

\begin{example}

The discrete category $\ul{2}$ has no initial object and no terminal object. The reason is that it has two objects $1,2$, but no maps from one to the other, so $\Hom_{\ul{2}}(1,2)=\Hom_{\ul{2}}(2,1)=\emptyset$.

\end{example}

\begin{exercise}
Recall the {\tt divides} preorder from Exercise \ref{exc:divides as po}, where $5\;{\tt divides}\;15$.
\sexc Considering this preorder as a category, does it have an initial object?
\next Does it have a terminal object?
\endsexc
\end{exercise}

\begin{exercise}
Let $\mcM=(\List(\{a,b\}),[\;],\plpl)$ denote the free monoid on $\{a,b\}$ (see Definition \ref{def:free monoid}), considered as a category (via Theorem \ref{thm:mon to cat}).
\sexc Does it have an initial object?
\next Does it have a terminal object?
\next Which monoids have initial (respectively terminal) objects?
\endsexc
\end{exercise}

\begin{exercise}
Let $S$ be a set and consider the indiscrete category $K_S\in\Ob(\Cat)$ on objects $S$ (see Example \ref{ex:indiscrete cat equiv to terminal}).
\sexc For what $S$ does $K_S$ have an initial object?
\next For what $S$ does $K_S$ have a terminal object?
\endsexc
\end{exercise}


\subsubsection{Examples of limits}\label{sec:examples of limits}

Let $\mcC$ be a category and let $X,Y\in\Ob(\mcC)$ be objects. Definition \ref{def:products in a cat} defines a product  of $X$ and $Y$ to be a span $X\From{\pi_1}X\times Y\To{\pi_2}Y$ such that for every other span $X\From{p}Z\To{q}Y$ there exists a unique morphism $Z\to X\times Y$ making the triangles commute. It turns out that we can enunciate this in our newly formed language of universal objects by saying that the span $X\From{\pi_1}X\times Y\To{\pi_2}Y$ is itself a terminal object in the category of spans on $X$ and $Y$. Phrasing the definition of products in this way will be generalizable to defining arbitrary limits.

\begin{construction}[Products]\index{products}

Let $\mcC$ be a category and let $X_1,X_2$ be objects. We can consider this setup as a diagram $X\taking\ul{2}\to\mcC$, where $X(1)=X_1$ and $X(2)=X_2$. Consider the category $\ul{2}\lcone=\Star_2$, which is drawn in Example \ref{ex:stars}; the inclusion $i\taking\ul{2}\to\ul{2}\lcone$, as in (\ref{dia:inclusion into cone}); and the category of functors $\Fun(\ul{2}\lcone,\mcC)$. The objects in $\Fun(\ul{2}\lcone,\mcC)$ are spans in $\mcC$ and the morphisms are natural transformations between them. Given a functor $S\taking\ul{2}\lcone\to\mcC$ we can compose with $i\taking\ul{2}\to\ul{2}\lcone$ to get a functor $\ul{2}\to\mcC$. We want that to be $X$.
$$\xymatrix@=30pt{\ul{2}\ar[r]^{X}\ar[d]_i&\mcC\\\ul{2}\lcone\ar[ur]_S}$$
So we are ready to define the category of spans on $X_1$ and $X_2$.

Define the {\em category of spans on $X$}, denoted $\mcC_{/X}$, to be the category whose objects and morphisms are as follows:
\begin{align}\label{dia:slice for products}
\Ob(\mcC_{/X})&=\{S\taking\ul{2}\lcone\to\mcC\|S\circ i=X\}\\
\nonumber\Hom_{\mcC_{/X}}(S,S')&=\{\alpha\taking S\to S'\|\alpha\circ i=\id_X\}.
\end{align}
The product of $X_1$ and $X_2$ was defined in Definition \ref{def:products in a cat}; we can now recast $X_1\times X_2$ as the terminal object in $\mcC_{/X}$.

To bring this down to earth, an object in $\mcC_{/X}$ can be pictured as a diagram in $\mcC$ of the following form:
$$\xymatrix@=15pt{&Z\ar[ldd]_p\ar[rdd]^q\\\\X_1&&X_2}$$   
In other words, the objects of $\mcC_{/X}$ are spans, each of which we might write in-line as $X_1\From{p}Z\To{q}X_2$. A morphism in $\mcC_{/X}$ from object $X_1\From{p}Z\To{q}X_2$ to object $X_1\From{p'}Z'\To{q'}X_2$ consists of a morphism $\ell\taking Z\to Z'$, such that $p'\circ\ell=p$ and $q'\circ\ell=q$. So the set of such morphisms in $\mcC_{/X}$ are all the $\ell$'s that make the right-hand diagram commute:
\footnote{To be completely pedantic, according to (\ref{dia:slice for products}), the morphisms in $\mcC_{/X}$ should be drawn like this:
\begin{align*}
\Hom_{\mcC_{/X}}\normalsize\left(\parbox{1in}{\xymatrix@=8pt{&Z\ar[ldd]_p\ar[rdd]^q\\\\X_1&&X_2}}\hspace{.2in},\hspace{.2in}\parbox{1in}{\xymatrix@=8pt{&Z'\ar[ldd]_{p'}\ar[rdd]^{q'}\\\\X_1&&X_2}}\right)\hspace{.2in}=\hspace{.2in}\left\{\;\;\parbox{1in}{\xymatrix@=15pt{&Z\ar[ldd]_{p}\ar[rdd]^{q}\ar@{-->}[ddddd]^{\alpha^{}_{-\infty}}\\\\X_1\ar@{=}[d]_{\alpha_1}&&X_2\ar@{=}[d]^{\alpha_2}\\X_1&&X_2\\\\&Z'\ar[uul]^{p'}\ar[uur]_{q'}}}\;\;\right\}
\end{align*}
But this is going a bit overboard. The point is, the set $\Hom_{\mcC_{/X}}$ is the set of morphisms serving the role of $\alpha_{-\infty}\taking Z\to Z'$.}
\begin{align}\label{dia:morphism of spans}
\Hom_{\mcC_{/X}}\normalsize\left(\parbox{1in}{\xymatrix@=8pt{&Z\ar[ldd]_p\ar[rdd]^q\\\\X_1&&X_2}}\hspace{.2in},\hspace{.2in}\parbox{1in}{\xymatrix@=8pt{&Z'\ar[ldd]_{p'}\ar[rdd]^{q'}\\\\X_1&&X_2}}\right)\hspace{.2in}=\hspace{.2in}\left\{\;\;\parbox{1in}{\xymatrix@=15pt{&Z\ar[ldd]_{p}\ar[rdd]^{q}\ar@{-->}[dddd]^\ell\\\\X_1&&X_2\\\\&Z'\ar[uul]^{p'}\ar[uur]_{q'}}}\;\;\right\}
\end{align}

Each object in $\mcC_{/X}$ is a span on $X_1$ and $X_2$, and each morphism in $\mcC_{/X}$ is a ``morphism of cone points in $\mcC$ making everything in sight commute". The terminal object in $\mcC_{/X}$ is the product of $X_1$ and $X_2$; see Definition \ref{def:products in a cat}.

\end{construction}

It may be strange to have a category in which the objects are spans in another category. But once you admit this possibility, the notion of morphism between spans is totally sensible. Or if it isn't, then stare at (\ref{dia:morphism of spans}) for 30 seconds and say to yourself ``When in Rome..!" These are the aqueducts of category theory, and they work wonders.

\begin{example}\label{ex:category of spans}

Consider the arbitrary 6-object category $\mcC$ drawn below, in which the three diagrams that can commute do:
$$\mcC:=\parbox{3in}{\fbox{\xymatrix@=39pt{&&\LMO{X_1}&\\\LMO{A}\ar@/^1pc/[urr]^a&\LMO{B}\ar[l]_f\ar@{}[u]|(.4){\checkmark}\ar[ur]_{b_1}\ar[dr]^{b_2}&&\LMO{C}\ar@{}[u]|(.4){\checkmark}\ar@{}[d]|(.4){\checkmark}\ar[ul]^{c_1}\ar[dl]_{c_2}\ar[r]^g&\LMO{D}\ar@/_1pc/[llu]_{d_1}\ar@/^1pc/[lld]^{d_2}\\&&\LMO{X_2}&}}}$$
Let $X\taking\ul{2}\to\mcC$ be given by $X(1)=X_1$ and $X(2)=X_2$. Then the category of spans on $X$ might be drawn
$$\mcC_{/X}\iso\fbox{\xymatrix{&\LMO{(B,b_1,b_2)}&&\LMO{(C,c_1,c_2)}\ar[r]^g&\LMO{(D,d_1,d_2)}}}$$

\end{example}


\subsubsection{Definition of limit}

\begin{definition}\label{def:slice and limit}\index{category!slice}\index{slice}\index{limit}

Let $\mcC$ be a category, let $I$ be a category; let $I\lcone$ be the left cone on $I$, and let $i\taking I\to I\lcone$ be the inclusion. Suppose that $X\taking I\to\mcC$ is an $I$-shaped diagram in $\mcC$. The {\em slice category of $\mcC$ over $X$} denoted $\mcC_{/X}$\index{a symbol!$\mcC_{/X}$} is the category whose objects and morphisms are as follows:
\begin{align*}
\Ob(\mcC_{/X})&=\{S\taking I\lcone\to\mcC\|S\circ i=X\}\\
\Hom_{\mcC_{/X}}(S,S')&=\{\alpha\taking S\to S'\|\alpha\circ i=\id_X\}.
\end{align*}

A {\em limit of $X$}, denoted $\lim_IX$ or $\lim X$,\index{a symbol!$\lim$} is a terminal object in $\mcC_{/X}$.

\end{definition}

\paragraph{Pullbacks}\index{pullback}\index{universal property!pullback}

The relevant indexing category for pullbacks is the cospan, $I=\ul{2}\rcone$ drawn as to the left below: 
$$
\parbox{1.2in}{\boxtitle{$I$}\fbox{\xymatrix{\LMO{0}\ar[rd]&&\LMO{1}\ar[ld]\\&\LMO{2}}}}
\hspace{1in}
\parbox{1.5in}{\boxtitle{$X\taking I\to\mcC$}\dbox{\xymatrix{\LMO{X_0}\ar[rd]&&\LMO{X_1}\ar[ld]\\&\LMO{X_2}}}}
\;\;\footnote{We use a dash box here because we're not drawing the whole category but merely a diagram existing inside $\mcC$.}
$$
A $I$-shaped diagram in $\mcC$ is a functor $X\taking I\to\mcC$, which we might draw as to the right above (e.g. $X_0\in\Ob(\mcC)$).

An object $S$ in the slice category $\mcC_{/X}$ is a commutative diagram $S\taking I\lcone\to\mcC$ over $X$, which looks like the box to the left below: 
$$
\parbox{1.5in}{\boxtitle{$S\in\Ob(\mcC_{/X})$}\dbox{\xymatrix{&S_{-\infty}\ar[rd]\ar[ld]\\\LMO{X_0}\ar[rd]&&\LMO{X_1}\ar[ld]\\&\LMO{X_2}}}}
\hspace{1in}
\parbox{1.5in}{\boxtitle{$f\taking S\to S'$}\dbox{\xymatrix{&S_{-\infty}\ar@/^1pc/[rdd]\ar@/_1pc/[ldd]\ar[d]^f\\&S'_{-\infty}\ar[rd]\ar[ld]\\\LMO{X_0}\ar[rd]&&\LMO{X_1}\ar[ld]\\&\LMO{X_2}}}}
$$
A morphism in $\mcC_{/X}$ is drawn in the dashbox to the right above. A terminal object in $\mcC_{/X}$ is precisely the ``gateway" we want, i.e. the limit of $X$ is the pullback $X_0\times_{X_2}X_1$.

\begin{exercise}\index{equalizer}
Let $I$ be the graph indexing category (see \ref{dia:graph index}).
\sexc What is $I\lcone$?
\next Now let $G\taking I\to\Set$ be the graph from Example \ref{ex:graph}. Give an example of an object in $\Set_{/G}$. 
\next We have already given a name to the limit of $G\taking I\to\Set$; what is it?
\endsexc
\end{exercise}

\begin{exercise}\label{exc:terminal as limit}
Let $\mcC$ be a category and let $I=\emptyset$ be the empty category. There is a unique functor $X\taking\emptyset\to\mcC$.
\sexc What is the slice category $\mcC_{/X}$?
\next What is the limit of $X$?
\endsexc
\end{exercise}

\begin{example}

Often one wants to take the limit of some strange diagram. We have now constructed the limit for any shape diagram. For example, if we want to take the product of more than two, say $n$, objects, we could use the diagram shape $I=\ul{n}$ whose cone is $\Star_n$ from Example \ref{ex:stars}.

\end{example}

\begin{example}\label{ex:product version of nat trans}\index{natural transformation!as functor}

We have now defined limits in any category, so we have defined limits in $\Cat$. Let $[1]$ denote the category depicted 
$$\xymatrix{\LMO{0}\ar[r]^e&\LMO{1}}$$
and let $\mcC$ be a category. Naming two categories is the same thing as naming a functor $X\taking\ul{2}\to\Cat$, so we now have such a functor. Its limit is denoted $[1]\times\mcC$. It turns out that $[1]\times\mcC$ looks like a ``$\mcC$-shaped prism". It consists of two panes, front and back say, each having the precise shape as $\mcC$ (same objects, same arrows, same composition), and morphisms from the front pane to the back pane making all front-to-back squares commute. For example, if $\mcC$ looked was the category generated by the schema to the left below, then $\mcC\times[1]$ would be the category generated by the schema to the right below:
$$
\xymatrix{
\LMO{A}\ar[rr]^f\ar[dd]_g&&\LMO{B}\ar[dd]^h\\\\\LMO{C}&&\LMO{D}
}
\hspace{1in}
\xymatrix@=20pt{
&\LMO{A1}\ar[rr]^{f1}\ar'[d][dd]_(.6){g1}&&\LMO{B1}\ar[dd]^{h1}\\
\LMO{A0}\ar[ur]^{Ae}\ar[rr]^(.6){f0}\ar[dd]_{g0}&&\LMO{B0}\ar[ur]^{Be}\ar[dd]^{h0}\\
&\LMO{C1}&&\LMO{D1}\\
\LMO{C0}\ar[ur]_{Ce}&&\LMO{D0}\ar[ur]_{De}&
}
$$

It turns out that a natural transformation $\alpha\taking F\to G$ between functors $F,G\taking\mcC\to\mcD$ is the same thing as a functor $\mcC\times[1]\to\mcD$ such that the front pane is sent via $F$ and the back pane is sent via $G$. The components are captured by the front-to-back morphisms, and the naturality is captured by the commutativity of the front-to-back squares in $\mcC\times[1]$.

\end{example}

\begin{remark}\index{relative set!as slice category}

Recall in Section \ref{sec:relative sets} we described relative sets. In fact, Definition \ref{def:relative sets} basically defines a category of relative sets over any fixed set $B$. Let $\ul{1}$ denote the discrete category on one object, and note that providing a functor $\ul{1}\to\Set$ is the same as simply providing a set, so consider $B\taking\ul{1}\to\Set$. Then the slice category $\Set_{/B}$, as defined in Definition \ref{def:slice and limit} is precisely the category of relative sets over $B$: it has the same objects and morphisms as was described in Definition \ref{def:relative sets}.

\end{remark}


\subsubsection{Definition of colimit}

The definition of colimits is appropriately ``dual" to the definition of limits. Instead of looking at left cones, we look at right cones; instead of being interested in terminal objects, we are interested in initial objects.

\begin{definition}\label{def:coslice and colimit}\index{coslice}\index{category!coslice}\index{colimit}

Let $\mcC$ be a category, let $I$ be a category; let $I\rcone$ be the right cone on $I$, and let $i\taking I\to I\rcone$ be the inclusion. Suppose that $X\taking I\to\mcC$ is an $I$-shaped diagram in $\mcC$. The {\em coslice category of $\mcC$ over $X$} denoted $\mcC_{X/}$\index{a symbol!$\mcC_{X/}$} is the category whose objects and morphisms are as follows:
\begin{align*}
\Ob(\mcC_{X/})&=\{S\taking I\rcone\to\mcC\|S\circ i=X\}\\
\Hom_{\mcC_{X/}}(S,S')&=\{\alpha\taking S\to S'\|\alpha\circ i=\id_X\}.
\end{align*}

A {\em colimit of $X$}, denoted $\colim_IX$ or $\colim X$,\index{a symbol!$\colim$} is an initial object in $\mcC_{X/}$.

\end{definition}

\paragraph{Pushouts}\index{pushout}

The relevant indexing category for pushouts is the span, $I=\ul{2}\lcone$ drawn as to the left below: 
$$
\parbox{1.2in}{\boxtitle{$I$}\fbox{\xymatrix{\LMO{1}&&\LMO{2}\\&\LMO{0}\ar[ul]\ar[ur]}}}
\hspace{1in}
\parbox{1.5in}{\boxtitle{$X\taking I\to\mcC$}\dbox{\xymatrix{\LMO{X_1}&&\LMO{X_2}\\&\LMO{X_0}\ar[ul]\ar[ur]}}}
$$
An $I$-shaped diagram in $\mcC$ is a functor $X\taking I\to\mcC$, which we might draw as to the right above (e.g. $X_0\in\Ob(\mcC)$).

An object $S$ in the coslice category $\mcC_{X/}$ is a commutative diagram $S\taking I\rcone\to\mcC$ over $X$, which looks like the box to the left below: 
$$
\parbox{1.5in}{\boxtitle{$S\in\Ob(\mcC_{X/})$}\dbox{\xymatrix{&S_{\infty}\\\LMO{X_1}\ar[ru]&&\LMO{X_2}\ar[lu]\\&\LMO{X_0}\ar[ul]\ar[ur]}}}
\hspace{1in}
\parbox{1.5in}{\boxtitle{$f\taking S\to S'$}\dbox{\xymatrix{&S'_{\infty}\\&S_{\infty}\ar[u]_f\\
\LMO{X_1}\ar[ur]\ar@/^1pc/[ruu]&&\LMO{X_2}\ar[lu]\ar@/_1pc/[luu]\\\
&\LMO{X_0}\ar[ur]\ar[ul]}}}
$$
A morphism in $\mcC_{X/}$ is drawn in the dashbox to the right above. An initial object in $\mcC_{X/}$ is precisely the ``gateway" we want; i.e. the colimit of $X$ is the pushout, $X_1\sqcup_{X_0}X_2$.

\begin{exercise}
Let $I$ be the graph indexing category (see \ref{dia:graph index}).
\sexc What is $I\rcone$?
\next Now let $G\taking I\to\Set$ be the graph from Example \ref{ex:graph}. Give an example of an object in $\Set_{G/}$. 
\next We have already given a name to the colimit of $G\taking I\to\Set$; what is it?
\endsexc
\end{exercise}

\begin{exercise}\label{exc:initial as colimit}
Let $\mcC$ be a category and let $I=\emptyset$ be the empty category. There is a unique functor $X\taking\emptyset\to\mcC$.
\sexc What is the coslice category $\mcC_{X/}$?
\next What is the colimit of $X$ (assuming it exists)?
\endsexc
\end{exercise}

\begin{example}[Cone as colimit]

We have now defined colimits in any category, so we have defined colimits in $\Cat$. Let $\mcC$ be a category and recall from Example \ref{ex:product version of nat trans} the category $\mcC\times[1]$. The inclusion of the front pane is a functor $i_0\taking\mcC\to\mcC\times[1]$ (similarly, the inclusion of the back pane is a functor $i_1\taking\mcC\to\mcC\times[1]$). Finally let $t\taking\mcC\to\ul{1}$ be the unique functor to the terminal category (see Exercise \ref{exc:term cat}). We now have a diagram in $\Cat$ of the form 
$$\xymatrix{\mcC\ar[r]^{i_0}\ar[d]_{t}&\mcC\times[1]\\\ul{1}}$$
The colimit (i.e. the pushout) of this diagram in $\Cat$ slurps down the entire front pane of $\mcC\times[1]$ to a point, and the resulting category is isomorphic to $\mcC\lcone$. Figure \ref{fig:left cone} is a drawing of this phenomenon.
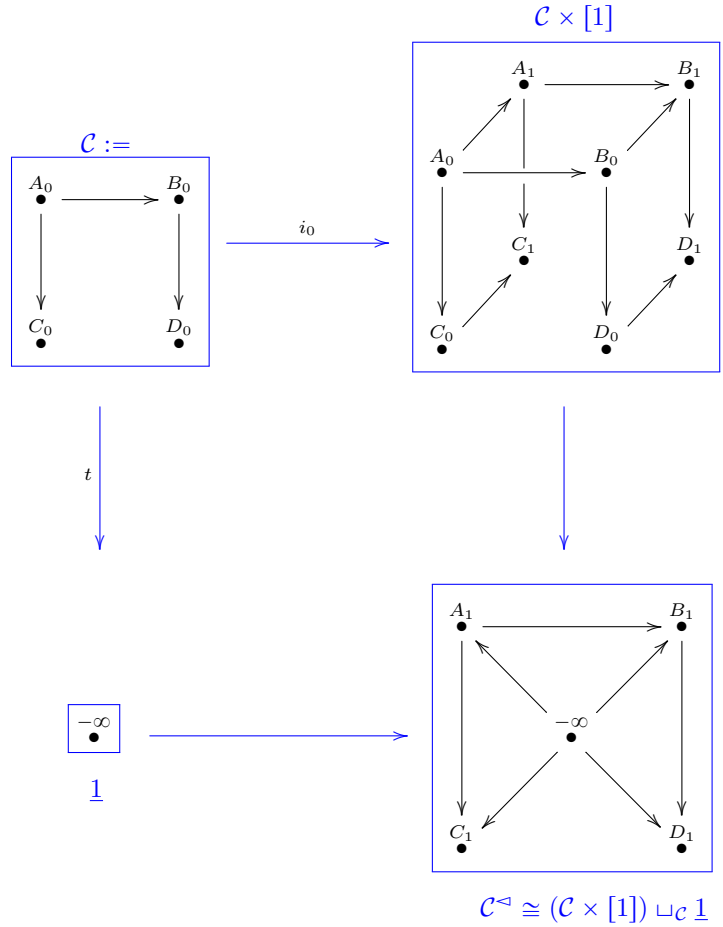
\begin{figure}[H]
$$
\parbox{1in}{~\vspace{.6in}\boxtitle{\color{blue}{$\mcC:=$}}\cfbox{blue}{\parbox{.7in}{\xymatrix@=15pt{
\LMO{A_0}\ar[rr]\ar[dd]&&\LMO{B_0}\ar[dd]\\\\\LMO{C_0}&&\LMO{D_0}
}}}}
\parbox{1.1in}{~\vspace{.6in}\\\xymatrix{~\ar@[blue][rr]^{i_0}&\hspace{.1in}&~}}
\parbox{1.7in}{\boxtitle{\color{blue}{$\mcC\times[1]$}}\cfbox{blue}{\parbox{1.4in}{\xymatrix@=15pt{
&\LMO{A_1}\ar[rr]\ar'[d][dd]&&\LMO{B_1}\ar[dd]\\
\LMO{A_0}\ar[ur]\ar[rr]\ar[dd]&&\LMO{B_0}\ar[ur]\ar[dd]\\
&\LMO{C_1}&&\LMO{D_1}\\
\LMO{C_0}\ar[ur]&&\LMO{D_0}\ar[ur]&
}}}}
$$
$$
\hspace{-.5in}\xymatrix{~\ar@[blue][dd]_t\\\\~}\hspace{2.3in}\xymatrix{~\ar@[blue][dd]\\\\~}
$$
$$
\hspace{.4in}\parbox{.3in}{\cfbox{blue}{$\LMO{-\infty}$}\begin{center}\color{blue}{$\ul{1}$}\end{center}}
\parbox{1.6in}{\xymatrix{~\ar@[blue][rr]&\hspace{.6in}&~}\vspace{.15in}~}
\parbox{1.7in}{\cfbox{blue}{\parbox{1.7in}{
\xymatrix{
\LMO{A_1}\ar[rr]\ar[dd]&&\LMO{B_1}\ar[dd]\\
&\LMO{-\infty}\ar[ur]\ar[dr]\ar[ul]\ar[dl]\\
\LMO{C_1}&&\LMO{D_1}}}}\begin{center}\color{blue}{$\mcC\lcone\iso(\mcC\times[1])\sqcup_{\mcC}\ul{1}$}\end{center}}
$$
\caption{Let $\mcC$ be the category drawn in the upper left corner. The left cone $\mcC\lcone$ on $\mcC$ is obtained as a pushout in $\Cat$. We first make a prism $\mcC\times[1]$, and then identify the front pane with a point.}
\label{fig:left cone}
\end{figure}
(Similarly, the pushout of the analogous diagram for $i_1$ would give $\mcC\rcone$.)

\end{example}

\begin{example}\label{ex:pushout in Top}\index{pushout!of topological spaces}

Consider the category $\Top$ of topological spaces. The (hollow) circle is a topological space which people often denote $S^1$ (for ``1-dimensional sphere"). The filled-in circle, also called a 2-dimensional disk, is denoted $D^2$. The inclusion of the circle into the disk is continuous so we have a morphism in $\Top$ of the form $i\taking S^1\to D^2$. The terminal object in $\Top$ is the one-point space $\singleton$, and so there is a unique morphism $t\taking S^1\to\singleton$. The pushout of the diagram $D^2\From{i}S^1\To{t}\singleton$ is isomorphic to the 2-dimensional sphere (the exterior of a tennis ball), $S^2$. The reason is that we have slurped the entire bounding circle to a point, and the category of topological spaces has the right morphisms to ensure that the resulting space really is a sphere. 

\end{example}

\begin{application}\index{subway}

Consider the symmetric graph $G_n$ consisting of a chain of $n$ vertices, 
$$\xymatrix{\LMO{1}\ar@{-}[r]&\LMO{2}\ar@{-}[r]&\cdots\ar@{-}[r]&\LMO{n}}$$
Think of this as modeling a subway line. There are $n$-many graph homomorphisms $G_1\to G_n$ given by the various vertices. One can create \href{http://en.wikipedia.org/wiki/Transit_map}{\text transit maps} using colimits. For example, the colimit of the diagram to the left is the symmetric graph drawn to the right below.
$$
\colim\left(\parbox{1.2in}{
\xymatrix{
G_1\ar[r]^4\ar[d]_4&\color{orange}{G_7}&G_1\ar[l]_6\ar[d]^1\\
\color{purple}{G_5}&&\color{ForestGreen}{G_3}\\
G_1\ar[u]^2\ar[r]_3&\color{blue}{G_7}&G_1\ar[u]_2\ar[l]^5
}}\right)
\hspace{.3in}\tn{can be drawn}\hspace{-.4in}
\parbox{3in}{\xymatrix@=12pt{
&&&\bullet_{\color{purple}{5}}\ar@{-}[d]\\
\LMO{{\color{orange}{1}}}\ar@{-}[r]&\LMO{{\color{orange}{2}}}\ar@{-}[r]&\LMO{{\color{orange}{3}}}\ar@{-}[r]&\LMO{{\color{orange}{4}}}_{\color{purple}{4}}\ar@{-}[r]\ar@{-}[d]&\LMO{{\color{orange}{5}}}\ar@{-}[r]&\LMO{{\color{orange}{6}}}_{\color{ForestGreen}{1}}\ar@{-}[r]\ar@{-}[dd]&\LMO{{\color{orange}{7}}}\\
&&&\bullet_{\color{purple}{3}}\ar@{-}[d]\\
&\LMO{{\color{blue}{1}}}\ar@{-}[r]&\LMO{{\color{blue}{2}}}\ar@{-}[r]&\LMO{{\color{blue}{3}}}_{\color{purple}{2}}\ar@{-}[r]\ar@{-}[d]&\LMO{{\color{blue}{4}}}\ar@{-}[r]&\LMO{{\color{blue}{5}}}_{\color{ForestGreen}{2}}\ar@{-}[d]\ar@{-}[r]&\LMO{{\color{blue}{6}}}\ar@{-}[r]&\LMO{{\color{blue}{7}}}\\
&&&\bullet_{\color{purple}{1}}&&\bullet_{\color{ForestGreen}{3}}}}
$$ 

\end{application}


\section{Other notions in $\Cat$}

In this section we discuss some leftover notions about categories. For example in Section \ref{sec:opposite} we explain a kind of duality for categories, in which arrows are flipped. For example reversing the order in a preorder is an example of this duality, as is the similarity between limits and colimits. In Section \ref{sec:grothendieck construction} we discuss the so-called Grothendieck construction which in some sense graphs functors, and we show that it is useful for transforming databases into the kind of format (RDF) used in scraping data off webpages. We define a general construction for creating categories in Section \ref{sec:comma}. Finally, in Section \ref{sec:arithmetic of categories} we show that precisely the same arithmetic statements that held for sets in Section \ref{sec:arithmetic of sets} hold for categories. 


\subsection{Opposite categories}\label{sec:opposite}\index{functor!contravariant}\index{functor!covariant}

People used to discuss two different kinds of functors between categories: the so-called {\em covariant functors} and the so-called {\em contravariant functors}. Covariant functors are what we have been calling functors. The reader may have come across the idea of contravariance when considering Exercise \ref{exc:points and opens in Top}.\footnote{Similarly, see Exercise \ref{exc:juris 2}.} There we saw that a continuous mapping of topological spaces $f\taking X\to Y$ does not induce a morphism of orders on their open sets $\Op(X)\to\Op(Y)$; that is not required by the notion of continuity. Instead, a morphism of topological spaces $f\taking X\to Y$ induces a morphism of orders $\Op(Y)\to\Op(X)$, going backwards. So we do not have a functor $\Top\to\PrO$ in this way, but it's quite close. One used to say that $\Op$ is a {\em contravariant functor} $\Top\to\PrO$.

As important and common as contravariance is, people found that keeping track of which functors were covariant and which were contravariant was a big hassle. Luckily, there is a simple work-around, which simplifies everything: the notion of opposite categories.

\begin{definition}\index{category!opposite}

Let $\mcC$ be a category. The {\em opposite category} of $\mcC$, denoted $\mcC\op$,\index{a symbol!$\mcC\op$} has the same objects as $\mcC$, i.e. $\Ob(\mcC\op)=\Ob(\mcC)$, and for any two objects $c,c'$, one defines
$$\Hom_{\mcC\op}(c,c'):=\Hom_\mcC(c',c).$$

\end{definition}

\begin{example}

If $n\in\NN$ is a natural number and $\ul{n}$ the corresponding discrete category, then $\ul{n}\op=\ul{n}$. Recall  the span category $I=\ul{2}\lcone$ from Definition \ref{def:products in a cat}. Its opposite is the cospan category $I\op=\ul{2}\rcone$, from Definition \ref{def:coproducts in a cat}.

\end{example}

\begin{exercise}
Let $\mcC$ be the category from Example \ref{ex:category of spans}. Draw $\mcC\op$.
\end{exercise}

\begin{lemma}

Let $\mcC$ and $\mcD$ be categories. One has $(\mcC\op)\op=\mcC$. Also we have $\Fun(\mcC,\mcD)\iso\Fun(\mcC\op,\mcD\op)$. This implies that a functor $\mcC\op\to\mcD$ can be identified with a functor $\mcC\to\mcD\op$.

\end{lemma}

\begin{proof}

This follows straightforwardly from the definitions.

\end{proof}

\begin{exercise}
In Exercises \ref{exc:points and opens in Top}, \ref{exc:juris 1}, and \ref{exc:juris 2} there were questions about whether a certain function $\Ob(\mcC)\to\Ob(\mcD)$ extended to a functor $\mcC\to\mcD$. In each case, see if the proposed function would extend to a ``contravariant functor" i.e. to a functor $\mcC\op\to\mcD$.
\end{exercise}

\begin{example}[Simplicial sets]\label{ex:simplicial set}\index{simplicial set}

Recall from Example \ref{ex:finite linear orders} the category $\bD$ of linear orders $[n]$. For example, $[1]$ is the linear order $0\leq 1$ and $[2]$ is the linear order $0\leq 1\leq2$. Both $[1]$ and $[2]$ are objects of $\bD$.\index{a category!$\bD$} There are 6 morphisms from $[1]$ to $[2]$, which we could denote $$\Hom_{\bD}([1],[2])=\{(0,0), (0,1), (0,2), (1,1), (1,2), (2,2)\}.$$

It may seem strange, but the category $\bD\op$ turns out to be quite useful in algebraic topology. It is the indexing category for a combinatorial approach to the homotopy theory of spaces. That is, we can represent something like the category of spaces and continuous maps using the functor category $\sSet:=\Fun(\bD\op,\Set)$,\index{a category!$\sSet$} which is called the {\em category of simplicial sets}. 

This may seem very complicated compared to something we did earlier, namely simplicial complexes. But simplicial sets have excellent formal properties that simplicial complexes do not. We will not go further with this here, but through the work of Dan Kan, Andr\'{e} Joyal, Jacob Lurie, and many others, simplicial sets have allowed category theory to pierce deeply into the realm of topology and vice versa.

\end{example}


\subsection{Grothendieck construction}\label{sec:grothendieck construction}\index{Grothendieck!construction}

Let $\mcC$ be a database schema (or category) and let $J\taking\mcC\to\Set$ be an instance. We have been drawing this in table form, but there is another standard way of laying out the data in $J$, called the \href{http://en.wikipedia.org/wiki/Resource_Description_Framework}{\em resource descriptive framework} or RDF. Developed for the web, RDF is a useful format when one does not have a schema in hand, e.g. when scraping information off of a website, one does not know what schema will be best. In these cases, information is stored in so-called RDF triples, which are of the form $$\la\tn{Subject, Predicate, Object}\ra$$\index{RDF}

For example, one might see something like 
\begin{align}\label{dia:Obama yells at congress}
\begin{tabular}{| lll |}
\bhline
{\bf Subject}&{\bf Predicate}&{\bf Object}\\\bbhline
A01&occurredOn&D13114\\
A01&performedBy&P44\\
A01&actionDescription&Told congress to raise debt ceiling\\
D13114&hasYear&2013\\
D13114&hasMonth&January\\
D13114&hasDay&14\\
P44&FirstName&Barack\\
P44&LastName&Obama\\\bhline
\end{tabular}
\end{align}

Category-theoretically, it is quite simple to convert a database instance $J\taking\mcC\to\Set$ into an RDF triple store. To do so, we use the {\em Grothendieck construction}, which is more aptly named the category of elements construction, defined below.\footnote{Apparently, Alexander Grothendieck\index{Grothendieck} did not invent this construction, it was discussed prior to Grothendieck's use of it, e.g. by Mac Lane. But more to the point, the term Grothendieck construction is not grammatically suited in the sense that both the following are awkward in English: ``the Grothendieck construction of $J$ is ..." (awkward because $J$ is not being constructed but used in a construction) and ``the Grothendieck construct for $J$ is..." (awkward because it just is). The term {\em category of elements} is more descriptive and easier to use grammatically.}

\begin{definition}\label{def:grothendieck}\index{category!of elements}

Let $\mcC$ be a category and let $J\taking\mcC\to\Set$ be a functor. The {\em category of elements of $J$}, denoted $\int_\mcC J$,\index{a symbol!$\int$} is defined as follows:
\begin{align*}
\Ob(\varint_\mcC J):=&\;\;\{(C,x)\|C\in\Ob(\mcC), x\in J(C)\}.\\
\Hom_{\int_\mcC J}((C,x),(C',x')):=&\;\;\{f\taking C\to C'\|J(f)(x)=x'\}.
\end{align*}

There is a natural functor $\pi_J\taking\int_\mcC J\too\mcC$. It sends each object $(C,x)\in\Ob(\int_\mcC J)$ to the object $C\in\Ob(\mcC)$. And it sends each morphism $f\taking (C,x)\to (C',x')$ to the morphism $f\taking C\to C'$. We call $\pi_J$ the {\em projection functor}.

\end{definition}

\begin{example}

Let $A$ be a set, and consider it as a discrete category. We saw in Exercise \ref{exc:indexed sets as functors} that a functor $S\taking A\to\Set$ is the same thing as an $A$-indexed set, as discussed in Section \ref{sec:indexed sets}. We will follow Definition \ref{def:indexed sets} and for each $a\in A$ write $S_a:=S(a)$.

What is the category of elements of a functor $S\taking A\to\Set$? The objects of $\int_AS$ are pairs $(a,s)$ where $a\in A$ and $s\in S(a)$. Since $A$ has nothing but identity morphisms, $\int_AS$ has nothing but identity morphisms; i.e. it is the discrete category on a set. In fact that set is the disjoint union $$\varint_AS=\bigsqcup_{a\in A}S_a.$$ The functor $\pi_S\taking\int_AS\to A$ sends each element in $S_a$ to the element $a\in A$. 

One can see this as a kind of histogram. For example, let $A=\{{\tt BOS, NYC, LA, DC}\}$ and let $S\taking A\to\Set$ assign 
\begin{align*}
S_{{\tt BOS}}&=\{{\tt Abby, Bob, Casandra}\},\\
S_{\tt NYC}&=\emptyset,\\
S_{\tt LA}&=\{{\tt John, Jim}\}, \tn{and}\\
S_{\tt DC}&=\{{\tt Abby,Carla}\}.
\end{align*}
Then the category of elements of $S$ would look like the (discrete) category at the top: 
\begin{align}\label{dia:elements for cities}
\varint_AS=\parbox{2.7in}{\fbox{\xymatrix@=10pt{
\LTO{(BOS,Abby)}\\
\LTO{(BOS,Bob)}&\hspace{.3in}&\LTO{(LA,John)}&\LTO{(DC,Abby)}\\
\LTO{(BOS,Casandra)}&&\LTO{(LA,Jim)}&\LTO{(DC,Carla)}
}}}
\end{align}
$$
\hsp\xymatrix{~\ar[d]_{\pi_S}\\~}
$$
$$
\;\;A=\fbox{\xymatrix@=33pt{\hspace{.25in}\LTO{BOS}&\LTO{NYC}&\LTO{LA}&\LTO{DC}\hspace{.2in}}}
$$

We also see that the category of elements construction has converted an $A$-indexed set into a relative set over $A$, as in Definition \ref{def:relative sets}.

\end{example}

The above example does not show at all how the Grothendieck construction transforms a database instance into an RDF triple store. The reason is that our database schema was $A$, a discrete category that specifies no connections between data (it simply collects the data into bins). 
So lets examine a more interesting database schema and instance. This is taken from \cite{Sp2}.

\begin{application}\index{RDF!as category of elements}

Consider the schema below, which we first encountered in Example \ref{ex:department store 3}:
\begin{align}\label{dia:basic cat}\mcC:=\MainCatLarge{}\end{align}
And consider the instance $J\taking\mcC\to\Set$, which we first encountered in (\ref{dia:instance on maincat}) and (\ref{dia:instance on maincat 2})

\begin{align*}
&\footnotesize
\begin{tabular}{| l || l | l | l | l |}\bhline
\multicolumn{5}{| c |}{{\tt Employee}}\\\bhline 
{\bf ID}&{\bf first}&{\bf last}&{\bf manager}&{\bf worksIn}\\\bbhline 101&David&Hilbert&103&q10\\\hline 102&Bertrand&Russell&102&x02\\\hline 103&Emmy&Noether&103&q10\\\bhline
\end{tabular}&\hsp\footnotesize
\begin{tabular}{| l || l | l |}\bhline
\multicolumn{3}{| c |}{{\tt Department}}\\
\bhline {\bf ID}&{\bf name}&{\bf secretary}\\\bbhline q10&Sales&101\\\hline x02&Production&102\\\bhline
\end{tabular}
\end{align*}\vspace{.1in}

\begin{align*}\footnotesize
\begin{tabular}{| l ||}\bhline
\multicolumn{1}{| c |}{{\tt FirstNameString}}\\\bhline
{\bf ID}\\\bbhline Alan\\\hline Bertrand\\\hline Carl\\\hline David\\\hline Emmy\\\bhline
\end{tabular}\hspace{.6in}\footnotesize
\begin{tabular}{| l ||}\bhline
\multicolumn{1}{| c |}{{\tt LastNameString}}\\\bhline
{\bf ID}\\\bbhline Arden\\\hline Hilbert\\\hline Jones\\\hline Noether\\\hline Russell\\\bhline
\end{tabular}\hspace{.6in}\footnotesize
\begin{tabular}{| l ||}\bhline
\multicolumn{1}{| c |}{{\tt DepartmentNameString}}\\\bhline
{\bf ID}\\\bbhline Marketing\\\hline Production\\\hline Sales\\\bhline
\end{tabular}
\end{align*}  

The category of elements of $J\taking\mcC\to\Set$ looks like this:

\begin{align}\label{dia:Grothendieck}
\hspace{0in}\varint_\mcC J=\parbox{3.9in}{
\fbox{
\xymatrix@=.1pt{
&\LTO{101}\ar@/_2.5pc/[ddddl]+<-3pt,3pt>_{\tin{first}}\ar@/_1.5pc/[ddrrr]+<-4pt,0pt>^{\tin{last}}\ar@/_1pc/[rr]+<-3pt,-3pt>_-{\tin{manager}}\ar@/^1.8pc/[rrrrrr]^{\tin{worksIn}}&\LTO{\;\;102}&\LTO{103}&&&&\LTO{q10}&\LTO{x02}\ar@/^1.6pc/[llllll]+<4pt,-2pt>_{\tin{secretary}}\ar@/^1.2pc/[lddd]+<6pt,2pt>^(.4){\tin{name}}\\\parbox{.1in}{~\\\vspace{.6in}~}\\\LTO{Alan}&&&&\LTO{Hilbert}&\hspace{.4in}&\hspace{.4in}&\LTO{Production}\\\LTO{\;\;Bertrand}&&&&\LTO{Russell}&&&\LTO{\hspace{-.1in}Sales}\\\LTO{\hspace{.2in}David}&&&&\LTO{Noether}&&&\LTO{Marketing}\\\LTO{Emmy}&&&&\LTO{Arden}\\\LTO{Carl}&&&&\LTO{Jones}}}\\
\xymatrix{\hspace{1.6in}&\ar[d]^{\pi_J}\\&~}}\\
\nonumber \mcC=\parbox{3.9in}{\hspace{.1in}\fbox{
			\xymatrix@=9pt{&\LTO{Employee}\ar@<.5ex>[rrrrr]^{\tn{worksIn}}\ar@(l,u)[]+<5pt,10pt>^{\tn{manager}}\ar[dddl]_{\tn{first}}\ar[dddr]^{\tn{last}}&&&\hspace{0in}&&\LTO{Department}\ar@<.5ex>[lllll]^{\tn{secretary}}\ar[ddd]^{\tn{name}}\\\\\\\LTO{\parbox{.3in}{\tt \scriptsize FirstNameString}}&&\LTO{LastNameString}&~&~&~&\LTO{DepartmentNameString}
			}}}
\end{align}~\\

In the above drawing (\ref{dia:Grothendieck}) of $\int_\mcC J$, we left out 10 arrows for ease of readability, for example, we left out an arrow $\LTO{102}\Too{\tt first}\LTO{Bertrand}$.

For the punchline, how do we see the category of elements $\int_\mcC J$ as an RDF triple store? For each arrow in $\int_\mcC J$, we take the triple consisting of the source vertex, the arrow name, and the target vertex. So our triple store would include triples such as $\la{\tt 102\;\; first\;\; Bertrand}\ra$ and $\la{\tt 101\;\; manager\;\; 103}\ra$.

\end{application}

\begin{exercise}
Come up with a schema and instance whose category of elements contains (at least) the data from (\ref{dia:Obama yells at congress}).
\end{exercise}

\begin{slogan}
The Grothendieck construction takes structured, boxed-up data and flattens it by throwing it all into one big space. The projection functor is then tasked with remembering which box each datum originally came from.
\end{slogan}

\begin{exercise}\label{exc:FSM as elements of monoid action}\index{finite state machine}
Recall from Section \ref{sec:FSMs} that a finite state machine is a free monoid $(\List(\Sigma),[\;],\plpl)$ acting on a set $X$. Recall also that we can consider a monoid as a category $\mcM$ with one object and a monoid action as a set-valued functor $F\taking\mcM\to\Set$, (see Section \ref{sec:monoids as cats}). In the case of Figure \ref{fig:fsa} the monoid in question is $\List(a,b)$, which can be drawn as the schema
$$\fbox{\xymatrix{\monOb\ar@(ul,dl)[]_(.3)a\ar@(ul,dl)[]_(.7){~}\ar@(ur,dr)[]^(.3)b\ar@(ur,dr)[]^(.7){~}}}$$
and the functor $F\taking\mcM\to\Set$ is recorded in an action table in Example \ref{ex:action table}. What is $\int_\mcM F$? How does it relate to the picture in Figure \ref{fig:fsa}?
\end{exercise}


\subsection{Full subcategory}\index{subcategory!full}

\begin{definition}\label{def:full subcategory}

Let $\mcC$ be a category and let $X\ss\Ob(\mcC)$ be a set of objects in $\mcC$. The {\em full subcategory of $\mcC$ spanned by $X$} is the category, which we denote by $\mcC_{\Ob=X}$, with objects $\Ob(\mcC_{\Ob=X}):=X$ and with morphisms $\Hom_{\mcC_{\Ob=X}}(x,x'):=\Hom_\mcC(x,x')$.

\end{definition}

\begin{example}

The following are examples of full subcategories. We will name them in the form ``$X$ inside of $Y$", and each time we mean that $X$ and $Y$ are names of categories, the category $X$ can be considered as a subcategory of the category $Y$ in some sense, and it is full. In other words, all morphisms in $Y$ ``count" as morphisms in $X$.
\begin{itemize}
\item Finite sets inside of sets, $\Fin\ss\Set$;
\item Finite sets of the form $\ul{n}$ inside of $\Fin$;
\item Linear orders of the form $[n]$ inside of all finite linear orders, $\bD\ss\FLin$;
\item Groups inside of monoids, $\Grp\ss\Mon$;
\item Monoids inside of categories, $\Mon\ss\Cat$;
\item Sets inside of graphs, $\Set\ss\Grph$;
\item Partial orders (resp. linear orders) inside of $\PrO$;
\item Discrete categories (resp. indiscrete categories) inside of $\Cat$;
\end{itemize}

\end{example}

\begin{remark}

A subcategory $\mcC\ss\mcD$ is (up to isomorphism) just a functor $i\taking\mcC\to\mcD$ that happens to be injective on objects and arrows. The subcategory is full if and only if $i$ is a full functor in the sense of Definition \ref{def:full faithful}.

\end{remark}

%
%

\begin{example}

Let $\mcC$ be a category, let $X\ss\Ob(\mcC)$ be a set of objects, and let $\mcC_{\Ob=X}$ denote the full subcategory of $\mcC$ spanned by $X$. We can realize this as a fiber product of categories. Indeed, recall that for any set, we can form the indiscrete category on that set; see Example \ref{ex:indiscrete cat equiv to terminal}. In fact, we have a functor $Ind\taking\Set\to\Cat$.\index{a functor!$Ind\taking\Set\to\Cat$} Thus our function $X\to\Ob(\mcC)$ can be converted into a functor between indiscrete categories $Ind(X)\to Ind(\Ob(\mcC))$. There is also a functor $\mcC\to Ind(\Ob(\mcC))$ sending each object to itself. Then the full subcategory of $\mcC$ spanned by $X$ is the fiber product of categories,
$$\xymatrix{\mcC_{\Ob=X}\ar[r]\ar[d]&\mcC\ar[d]\\Ind(X)\ar[r]&Ind(\Ob(\mcC))}$$

\end{example}

\begin{exercise}
Including all identities and all compositions, how many morphisms are there in the full subcategory of $\Set$ spanned by the objects $\{\ul{0},\ul{1},\ul{2}\}$? Write them out.
\end{exercise}


\subsection{Comma categories}\label{sec:comma}\index{category!comma}

Category theory includes a highly developed and interoperable catalogue of materials and production techniques. One such is the comma category.

\begin{definition}\label{def:comma category}

Let $\mcA,\mcB,$ and $\mcC$ be categories and let $F\taking\mcA\to\mcC$ and $G\taking\mcB\to\mcC$ be functors. The {\em comma category of $\mcC$ morphisms from $F$ to $G$}, denoted $(F\down_\mcC G)$ or simply $(F\down G)$,\index{a symbol!$(F\down G)$} is the category with objects $$\Ob(F\down G)=\{(a,b,f)\|a\in\Ob(\mcA), b\in\Ob(\mcB), f\taking F(a)\to G(b)\tn{ in }\mcC\}$$ and for any two objects $(a,b,f)$ and $(a',b',f')$ the set $\Hom_{(F\down G)}((a,b,f),(a',b',f'))$ of morphisms $(a,b,f)\too(a',b',f')$ is 
$$\{(q,r)\|q\taking a\to a'\tn{ in }\mcA,\;\; r\taking b\to b'\tn{ in }\mcB,\tn{ such that } f'\circ F(q)=G(r)\circ f\}.$$
In pictures,
$$\Hom_{(F\down G)}((a,b,f),(a',b',f')):=\left\{\parbox{2in}{\xymatrix{
a\ar[d]_q&F(a)\ar@{}[dr]|{\checkmark}\ar[r]^f\ar[d]_{F(q)}&G(b)\ar[d]^{G(r)}&b\ar[d]^r\\
a'&F(a')\ar[r]_{f'}&G(b')&b'
}}\right\}$$
We refer to the diagram $\mcA\To{F}\mcC\From{G}\mcB$ (in $\Cat$) as the {\em setup} for the comma category $(F\down G)$.

There is a canonical functor $(F\down G)\to\mcA$ called {\em left projecton}, sending $(a,b,f)$ to $a$, and a canonical functor $(F\down G)\to\mcB$ called {\em right projection}, sending $(a,b,f)$ to $b$. 

\end{definition}

A setup $\mcA\To{F}\mcC\From{G}\mcB$ is reversable; i.e. we can flip it to obtain $\mcB\To{G}\mcC\From{F}\mcA$. However, note that $(F\down G)$ is different than (i.e. almost never equivalent to) $(G\down F)$, unless every arrow in $\mcC$ is an isomorphism.

\begin{slogan}
When two categories $\mcA,\mcB$ can be interpreted in a common setting $\mcC$, the comma category integrates them by recording how to move from $\mcA$ to $\mcB$ inside $\mcC$.
\end{slogan}

\begin{example}

Let $\mcC$ be a category and $I\taking\mcC\to\Set$ a functor. In this example we show that the comma category construction captures the notion of taking the category of elements $\int_\mcC I$; see Definition \ref{def:grothendieck}. 

Consider the set $\ul{1}$, the category $Disc(\ul{1})$, and the functor $F\taking Disc(\ul{1})\to\Set$ sending the unique object to the set $\ul{1}$. We use the comma category setup $\ul{1}\Too{F}\Set\Fromm{I}\mcC$. There is an isomorphism of categories 
$$\int_\mcC I\iso (F\down I).$$
Indeed, an object in $(F\down I)$ is a triple $(a,b,f)$ where $a\in\Ob(\ul{1}), b\in\Ob(\mcC)$, and $f\taking F(a)\to I(b)$ is a morphism in $\Set$. There is only one object in $\ul{1}$, so this reduces to a pair $(b,f)$ where $b\in\Ob(\mcC)$ and $f\taking \singleton\to I(b)$. The set of functions $\singleton\to I(b)$ is isomorphic to $I(b)$, as we saw in Exercise \ref{exc:generator for set}. So we have reduced $\Ob(F\down I)$ to the set of pairs $(b,x)$ where $b\in\Ob(\mcC)$ and $x\in I(b)$; this is $\Ob(\int_\mcC I)$. Because there is only one function $\ul{1}\to\ul{1}$, a morphism $(b,x)\to(b',x')$ in $(F\down I)$ boils down to a morphism $r\taking b\to b'$ such that the diagram 
$$\xymatrix{\ul{1}\ar[r]^x\ar@{=}[d]&I(b)\ar[d]^{I(r)}\\\ul{1}\ar[r]_{x'}&I(b')}$$
commutes. But such diagrams are in one-to-one correspondence with the diagrams needed for morphisms in $\int_\mcC I$.

\end{example}

\begin{exercise}
Let $\mcC$ be a category and let $c,c'\in\Ob(\mcC)$ be objects. Consider them as functors $c,c'\taking\ul{1}\to\mcC$, and consider the setup $\ul{1}\Too{c}\mcC\Fromm{c'}\ul{1}$. What is the comma category $(c\down c')$?
\end{exercise}


\subsection{Arithmetic of categories}\label{sec:arithmetic of categories}

In Section \ref{sec:arithmetic of sets}, we summarized some of the properties of products, coproducts, and exponentials for sets, attempting to show that they lined up precisely with familiar arithmetic properties of natural numbers. Astoundingly, we can do the same for categories.

In the following proposition, we denote the coproduct of two categories $\mcA$ and $\mcB$ by the notation $\mcA+\mcB$ rather than $\mcA\sqcup\mcB$. We also denote the functor category $\Fun(\mcA,\mcB)$ by $\mcB^\mcA$. Finally, we use $\ul{0}$ and $\ul{1}$ to refer to the discrete category on 0 and on 1 object, respectively.

\begin{proposition}\label{prop:arithmetic of cats}\index{category!arithmetic of}

The following isomorphisms exist for any small categories $\mcA,\mcB,$ and $\mcC$.

\begin{itemize}
\item $\mcA+\ul{0}\iso \mcA$
\item $\mcA + \mcB\iso \mcB + \mcA$
\item $(\mcA + \mcB) + \mcC \iso \mcA + (\mcB + \mcC)$
\item $\mcA\times\ul{0}\iso\ul{0}$
\item $\mcA\times\ul{1}\iso \mcA$
\item $\mcA\times \mcB\iso \mcB\times \mcA$
\item $(\mcA\times\mcB)\times\mcC\iso\mcA\times(\mcB\times\mcC)$
\item $\mcA\times(\mcB+\mcC)\iso (\mcA\times \mcB)+(\mcA\times \mcC)$
\item $\mcA^{\ul{0}}\iso \ul{1}$
\item $\mcA^{\ul{1}}\iso \mcA$
\item $\ul{0}^\mcA\iso\ul{0}$,\;\; if $\mcA\neq\ul{0}$
\item $\ul{1}^\mcA\iso\ul{1}$
\item $\mcA^{\mcB+\mcC}\iso \mcA^\mcB\times \mcA^\mcC$
\item $(\mcA^\mcB)^\mcC\iso \mcA^{\mcB\times \mcC}$
\end{itemize}

\end{proposition}

\begin{proof}

These are standard results; see \cite{Mac}.

\end{proof}


\chapter{Categories at work}

We have now set up an understanding of the basic notions of category theory: categories, functors, natural transformations, and universal properties. We have discussed many sources of examples: orders, graphs, monoids, and databases. We begin this chapter with the notion of {\em adjoint functors} (also known as {\em adjunctions}), which are like dictionaries that translate back and forth between different categories. 


\section{Adjoint functors}\index{adjoint functors}

Just above, in the introduction to this chapter, I said that adjoint functors are like dictionaries that translate back and forth between different categories. How far can we take that analogy?

In the common understanding of dictionaries, we assume that the two languages (say French and English) are equally expressive, and that a good dictionary will be an even exchange of ideas. But in category theory we often have two categories that are not on the same conceptual level. This is most clear in the case of so-called {\em free-forgetful adjunctions}. In Section \ref{sec:adjoints discuss and define} we will explore the sense in which each adjunction provides a dictionary between two categories that are not necessarily on an equal footing, so to speak.


\subsection{Discussion and definition}\label{sec:adjoints discuss and define}

Consider the category of monoids and the category of sets. A monoid $(M,e,\star)$ is a set with an identity element and a multiplication formula that is associative. A set is just a set. A dictionary between $\Mon$ and $\Set$ should not be required to set up an even exchange, but instead an exchange that is appropriate to the structures at hand. It will be in the form of two functors, one we'll denote by $L\taking\Set\to\Mon$, and one we'll denote by $R\taking\Mon\to\Set$. But to say what ``appropriate" means requires more work.\index{a functor!$\Set\to\Mon$}\index{a functor!$\Mon\to\Set$}

Let's bring it down to earth with an analogy.\index{adjunction!analogy: babies and adults} A one-year-old can make repeatable noises and an adult can make repeatable noises. One might say ``after all, talking is nothing but making repeatable noises." But the adult's repeatable noises are called words, they form sentences, and these sentences can cause nuclear wars. There is something more in adult language than there is simply in repeatable sounds. In the same vein, a tennis match can be viewed as physics, but you won't see the match. So we have something analogous to two categories here: ((repeated noises)) and ((meaningful words)). We are looking for adjoint functors going back and forth, serving as the appropriate sort of dictionary.

To translate baby talk into adult language we would make every repeated noise a kind of word, thereby granting it meaning. We don't know what a given repeated noise should mean, but we give it a slot in our conceptual space, always pondering ``I wonder what she means by Konnen.." On the other hand, to translate from meaningful words to repeatable noises is easy. We just hear the word as a repeated noise, which is how the baby probably hears it.

Adjoint functors often come in the form of ``free" and ``forgetful". Here we freely add Konnen to our conceptual space without having any idea how it adheres to the rest of the child's noises or feelings. But it doesn't act like a sound to us, it acts like a word; we don't know what it means but we figure it means something. Conversely, the translation going the other way is ``forgetful", forgetting the meaning of our words and just hearing them as sounds. The baby hears our words and accepts them as mere sounds, not knowing that there is anything extra to get.

Back to sets and monoids, the sets are like the babies from our story: they are simple objects full of unconnected dots. The monoids are like adults, forming words and performing actions. In the monoid, each element means something and combines with other elements in some way. There are lots of different sets and lots of different monoids, just as there are many babies and many adults, but there are patterns to the behavior of each kind and we put them in different categories.

Applying free functor $L\taking\Set\to\Mon$ to a set $X$ makes every element $x\in X$ a word, and these words can be strung together to form more complex words. (We discussed the free functor in Section \ref{sec:free monoid}.) Since a set such as $X$ carries no information about the meaning or structure of its various elements, the free monoid $F(X)$ does not relate different words in any way. To apply the forgetful functor $R\taking\Mon\to\Set$ to a monoid, even a structured one, is to simply forget that its elements are anything but mere elements of a set. It sends a monoid $(M,1,\star)$ to the set $M$. 

The analogy is complete. However, this is all just ideas. Let's give a definition, then return to our sets, monoids, sounds, and words.

\begin{definition}\label{def:adjunction}\index{adjunction}\index{functor!adjoint}

Let $\mcB$ and $\mcA$ be categories. \footnote{Throughout this definition, notice that $B$'s come before $A$'s, especially in (\ref{dia:adjunction isomorphism}), which might be confusing. It was a stylistic choice to match with the {\bf B}abies and {\bf A}dults discussion above and below this definition.}
An {\em adjunction between $\mcB$ and $\mcA$} is a pair of functors 
$$L\taking\mcB\to\mcA\hsp\tn{and}\hsp R\taking\mcA\to\mcB$$ 
together with a natural isomorphism
\footnote{The natural isomorphism $\alpha$ (see Lemma \ref{lemma:natural iso}) is between two functors $\mcB\op\times\mcA\to\Set$, namely the functor $(B,A)\mapsto\Hom_\mcA(L(B),A)$ and the functor $(B,A)\mapsto\Hom_\mcB(B,R(A))$.} 
whose component for any objects $A\in\Ob(\mcA)$ and $B\in\Ob(\mcB)$ is: 
\begin{align}\label{dia:adjunction isomorphism}
\alpha_{B,A}\taking\Hom_\mcA(L(B),A)\Too{\iso}\Hom_\mcB(B,R(A)).
\end{align}
This isomorphism is called the {\em adjunction isomorphism} for the $(L,R)$ adjunction, and for any morphism $f\taking L(B)\to A$ in $\mcA$, we refer to $\alpha_{B,A}(f)\taking B\to R(A)$ as {\em the adjunct} of $f$.
\footnote{Conversely, for any $g\taking B\to R(A)$ in $\mcB$ we refer to $\alpha_{B,A}^\m1(g)\taking L(B)\to A$ as {\em the adjunct} of $g$.}\index{adjunct}\index{adjunction!adjunction isomorphism}

The functor $L$ is called the {\em left adjoint} and the functor $R$ is called the {\em right adjoint}. We may say that {\em $L$ is the left adjoint of $R$} or that {\em $R$ is the right adjoint of $L$}. 
\footnote{The left adjoint does not have to be called $L$, nor does the right adjoint have to be called $R$, of course. This is suggestive.}
We often denote this setup by 
$$\adjoint{L}{\mcB}{\mcA}{R}$$

\end{definition}

\begin{proposition}\label{prop:free forgetful monoid}

Let $L\taking\Set\to\Mon$ be the functor sending $X\in\Ob(\Set)$ to the free monoid $L(X):=(\List(X),[\;],\plpl)$, as in Definition \ref{def:free monoid}. Let $R\taking\Mon\to\Set$ be the functor sending each monoid $\mcM:=(M,1,\star)$ to its underlying set $R(\mcM):=M$. Then $L$ is left adjoint to $R$.

\end{proposition}

\begin{proof}

If we can find a natural isomorphism of sets 
$$\alpha_{X,\mcM}\taking\Hom_\Mon(L(X),\mcM)\to\Hom_\Set(X,R(\mcM))$$
we will have succeeded in showing that these functors are adjoint.

Suppose given an element $f\in\Hom_\Mon(L(X),\mcM)$, i.e. a monoid homomorphism $f\taking\List(X)\to M$ (sending $[\;]$ to $1$ and list concatenation to $\star$). Then in particular we can apply $f$ to the singleton list $[x]$ for any $x\in X$. This gives a function $X\to M$ by $x\mapsto f([x])$, and this is $\alpha_{X,\mcM}(f)\taking X\to M=R(\mcM)$. We need only to supply an inverse $\beta_{X,\mcM}\taking\Hom_\Set(X,R(\mcM))\to\Hom_\Mon(L(X),\mcM).$

Suppose given an element $g\in\Hom_\Set(X,R(\mcM))$, i.e. a function $g\taking X\to M$. Then to any list $\ell=[x_1,x_2,\ldots,x_n]\in\List(X)$ we can assign $\beta_{X,\mcM}(\ell):=g(x_1)\star g(x_2)\star\cdots\star g(x_n)$ (if $\ell=[\;]$ is the empty list, assign $\beta_{X,\mcM}([\;]):=1$). We now have a function $\List(X)\to M$. It is a monoid homomorphism because it respects identity and composition. It is easy to check that $\beta$ and $\alpha$ are mutually inverse, completing the proof.

\end{proof}

\begin{example}

We need to ground our discussion in some concrete mathematics. In Proposition \ref{prop:free forgetful monoid} we provided our long-awaited adjunction between sets and monoids. A set $X$ gets transformed into a monoid by considering lists in $X$; a monoid $\mcM$ gets transformed into a set by forgetting the multiplication law. So we have a functor going one way and the other, 
$$L\taking\Set\to\Mon,\hspace{1in}R\taking\Mon\to\Set,$$
but an adjunction is more than that: it includes a guarantee about the relationship between these two functors. What is the relationship between $L$ and $R$? Consider an arbitrary monoid $\mcM=(M,1,*)$.

If I want to pick out 3 elements of the set $M$, that's the same thing as giving a function $\{a,b,c\}\to M$. But that function exists in the category of sets; in fact it is an element of $\Hom_\Set(\{a,b,c\},M)$. But since $M=R(\mcM)$ is the underlying set of our monoid, we can view the current paragraph in the light of our adjunction Equation (\ref{dia:adjunction isomorphism}) by saying it has been about the set
$$\Hom_\Set(\{a,b,c\},R(\mcM)).$$
This set classifies all the ways to pick three elements out of the underlying set of our monoid $\mcM$. It was constructed completely from within the category $\Set$.

Now we ask what Equation (\ref{dia:adjunction isomorphism}) means. The equation
$$\Hom_\Mon(L(\{a,b,c\}),\mcM)\iso\Hom_\Set(\{a,b,c\},R(\mcM)).$$
tells us that somehow we can answer the same question completely from within the category of monoids. In fact it tells us how to do so, namely as $\Hom_\Mon(\List(\{1,2,3\},\mcM)$.  Exercise \ref{ex:monoid adjunction} looks at how that should go. The answer is ``hidden" in the proof of Proposition \ref{prop:free forgetful monoid}.

\end{example}

\begin{exercise}\label{ex:monoid adjunction}
Let $X=\{a,b,c\}$ and let $\mcM=(\NN,1,*)$ be the multiplicative monoid of natural numbers (see Example \ref{ex:multiplication table}). Let $f\taking X\to\NN$ be the function given by $f(a)=7, f(b)=2, f(c)=2$, and let $\beta_{X,\mcM}\taking\Hom_\Set(X,R(\mcM))\to\Hom_\Mon(L(X),\mcM)$ be as in the proof of Proposition \ref{prop:free forgetful monoid}. What is $\beta_{X,\mcM}(f)([b,b,a,c])$?
\end{exercise}

Let us look once more at the adjunction between adults and babies. Using the notation of Definition \ref{def:adjunction} $\mcA$ is the ``adult category" of meaningful words and $\mcB$ is the ``baby category" of repeated noises. The left adjoint turns every repeated sound into a meaningful word (having ``free" meaning) and the right adjoint ``forgets" the meaning of any word and considers it merely as a sound. 

At the risk of taking this simple analogy too far, let's have a go at the heart of the issue: how to conceive of the isomorphism (\ref{dia:adjunction isomorphism}) of $\Hom$'s. Once we have freely given a slot to each of baby's repeated sounds, we try to find a mapping from the lexicon $L(B)$ of these new words to our own lexicon $A$ of meaningful words; these are mappings in the adult category $\mcA$ of the form $L(B)\to A.$ And (stretching it) the baby tries to find a mapping (which we might see as emulation) from her set $B$ of repeatable sounds to the set $R(A)$ of the sounds the adult seems to repeat. If there was a global system for making these transformations that would establish  (\ref{dia:adjunction isomorphism}) and hence the adjunction.

Note that the directionality of the adjunction makes a difference. If $L\taking\mcB\to\mcA$ is left adjoint to $R\taking\mcA\to\mcB$ we rarely have an isomorphism $\Hom_\mcA(A,L(B))\iso\Hom_\mcB(R(A),B)$. In the case of babies and adults, we see that it would make little sense to look for a mapping in the category of meaningful words from the adult lexicon to the wordifications of baby-sounds, because there is unlikely to be a good candidate for most of our words. That is, to which of our child's repeated noises would we assign the concept ``weekday"? 

Again, the above is simply an analogy, and almost certainly not formalizable. The next example shows mathematically the point we tried to make in the previous paragraph, that the directionality of an adjunction is not arbitrary.

\begin{example}\label{ex:adjunction monoids and sets}

Let $L\taking\Set\to\Mon$ and $R\taking\Mon\to\Set$ be the free and forgetful functors from Proposition \ref{prop:free forgetful monoid}. We know that $L$ is left adjoint to $R$; however $L$ is {\em not} right adjoint to $R$. In other words, we can show that the necessary natural isomorphism cannot exist.

Let $X=\{a,b\}$ and let $\mcM=(\{1\},1,!)$ be the trivial monoid. Then the necessary natural isomorphism would need to give us a bijection 
$$\Hom_\Mon(\mcM,L(X))\iso^?\Hom_\Set(\{1\},X).$$ 
But the left-hand side has one element, because $\mcM$ is the initial object in $\Mon$ (see Example \ref{ex:initial monoid terminal monoid}), whereas the right-hand side has two elements. Therefore no isomorphism can exist.

\end{example}

\begin{example}

Preorders have underlying sets, giving rise to a functor $U\taking\PrO\to\Set$.\index{a functor!$\PrO\to\Set$} The functor $U$ has both a left adjoint and a right adjoint. The left adjoint of $U$ is $D\taking\Set\to\PrO$,\index{a functor!$\Set\to\PrO$} sending a set $X$ to the discrete preorder on $X$ (the preorder with underlying set $X$, having the fewest possible $\leq$'s). The right adjoint of $U$ is $I\taking\Set\to\PrO$, sending a set $X$ to the indiscrete preorder on $X$ (the preorder with underlying set $X$, having the most possible $\leq$'s). See Example \ref{ex:discrete and indiscrete}. 

\end{example}

\begin{exercise}
Let $U\taking\Grph\to\Set$\index{a functor!$\Grph\to\Set$} denote the functor sending a graph to its underlying set of vertices. This functor has both a left and a right adjoint. 
\sexc What functor $\Set\to\Grph$ is the left adjoint of $U$?
\next What functor $\Set\to\Grph$ is the right adjoint of $U$?
\endsexc
\end{exercise}

\begin{example}\label{ex:other adjunctions}\index{currying!as adjunction}

Here are some other adjunctions:

\begin{itemize}
\item $\Ob\taking\Cat\to\Set$\index{a functor!$\Ob\taking\Cat\to\Set$} has a left adjoint $\Set\to\Cat$ given by the discrete category.
\item $\Ob\taking\Cat\to\Set$ has a right adjoint $\Set\to\Cat$ given by the indiscrete category.
\item The underlying graph functor $\Cat\to\Grph$\index{a functor!$\Cat\to\Grph$}\index{a functor!$\Grph\to\Cat$} has a left adjoint $\Grph\to\Cat$ given by the free category.
\item The functor $\PrO\to\Grph$, \index{a functor!$\PrO\to\Grph$} given by drawing edges for $\leq$'s, has a left adjoint given by existence of paths.
\item The forgetful functor from posets to preorders has a left adjoint given by quotient by isomorphism relation.
\item Given a set $A$, the functor $(-\times A)\taking\Set\to\Set$ has a right adjoint $\Hom(A,-)$ (this was called currying in Section \ref{sec:currying}). 
\end{itemize}

\end{example}

\begin{exercise}
Let $F\taking\mcC\to\mcD$ and $G\taking\mcD\to\mcC$ be mutually inverse equivalences of categories (see Definition \ref{def:equiv of cats}). Are they adjoint in one direction or the other?
\end{exercise}

\begin{exercise}
The discrete category functor $Disc\taking\Set\to\Cat$ has a left adjoint $p\taking\Cat\to\Set$. 
\sexc For an arbitrary object $X\in\Ob(\Set)$ and an arbitrary object $\mcC\in\Ob(\Cat)$, write down the adjunction isomorphism.
\next Let $\mcC$ be the free category on the graph $G$:
$$
G:=\parbox{2in}{\fbox{\xymatrix{\LMO{v}\ar[r]^f&\LMO{w}\ar@/_1pc/[r]_h\ar@/^1pc/[r]^g&\LMO{x}\\\LMO{y}\ar@(l,u)[]^i\ar@/^1pc/[r]^j&\LMO{z}\ar@/^1pc/[l]^k}}}
$$
and let $X=\{1,2,3\}$. How many elements does the set $\Hom_\Set(\mcC,Disc(X))$ have?
\next What can you do to an arbitrary category $\mcC$ to make a set $p(\mcC)$ such that the adjunction isomorphism holds? That is, how does the functor $p$ behave on objects?
\endsexc
\end{exercise}

The following proposition says that all adjoints to a given functor are isomorphic to each other. 

\begin{proposition}\label{prop:unicity of adjoints}

Let $\mcC$ and $\mcD$ be categories, let $F\taking\mcC\to\mcD$ be a functor, and let $G,G'\taking\mcD\to\mcC$ also be functors. If both $G$ and $G'$ are right adjoint (respectively left adjoint) to $F$ then there is a natural isomorphism $\phi\taking G\to G'$.

\end{proposition}

\begin{proof}

Suppose that both $G$ and $G'$ are right adjoint to $F$ (the case of $G$ and $G'$ being left adjoint is similarly proved). We first give a formula for the components of $\phi\taking G\to G'$ and its inverse $\psi\taking G'\to G$. Given an object $d\in\Ob(\mcD)$, we use $c=G(d)$ to obtain two natural isomorphisms, one from each adjunction: 
$$\Hom_\mcC(G(d),G(d))\iso\Hom_\mcD(F(G(d)),d)\iso\Hom_\mcC(G(d),G'(d)).$$
The identity component $\id_{G(d)}$ is then sent to some morphism $G(d)\to G'(d)$, which we take to be $\phi_d$. Similarly, we use $c'=G'(d)$ to obtain two natural isomorphisms, one from each adjunction:
$$\Hom_\mcC(G'(d),G'(d))\iso\Hom_\mcD(F(G'(d)),d)\iso\Hom_\mcC(G'(d),G(d)).$$
Again, the identity component $\id_{G'(d)}$ is sent to some morphism $G'(d)\to G(d)$, which we take to be $\psi_d$. The naturality of the isomorphisms implies that $\phi$ and $\psi$ are natural transformations, and it is straightforward to check that they are mutually inverse.

\end{proof}


\subsubsection{Quantifiers as adjoints}

One of the simplest but neatest places that adjoints show up is between preimages and the logical quantifiers $\exists$ and $\forall$, which we first discussed in Notation \ref{not:basic math notation}. \index{preimage}\index{a symbol!$\forall$}\index{a symbol!$\exists$} The setting in which to discuss this is that of sets and their power preorders. That is, if $X$ is a set then recall from Section \ref{sec:meets and joins} that the power set $\PP(X)$ has a natural ordering by inclusion of subsets. 

Given a function $f\taking X\to Y$ and a subset $V\ss Y$ the preimage is $f^\m1(V):=\{x\in X\|f(x)\in V\}$. If $V'\ss V$ then $f^\m1(V')\ss f^\m1(V)$, so in fact $f^\m1\taking\PP(Y)\to\PP(X)$ can be considered a functor (where of course we are thinking of preorders as categories). The quantifiers appear as adjoints of $f^\m1$.

Let's begin with the left adjoint of $f^\m1\taking\PP(Y)\to\PP(X)$. It is a functor $L_f\taking\PP(X)\to\PP(Y)$. Choose an object $U\ss X$ in $\PP(X)$. It turns out that
$$L_f(U)=\{y\in Y\|\exists x\in f^\m1(y)\tn{ such that }x\in U\}.$$
And the right adjoint $R_f\taking\PP(X)\to\PP(Y)$, when applied to $U$ is 
$$R_f(U)=\{y\in Y\|\forall x\in f^\m1(y), x\in U\}.$$
In fact, the functor $L_f$ is generally denoted $\exists_f\taking\PP(X)\to\PP(Y)$, and $R_f$ is generally denoted $\forall_f\taking\PP(X)\to\PP(Y)$. 
$$
\xymatrix{\PP(X)\ar@/^1.2pc/[rr]^{\exists_f}\ar@/_1.2pc/[rr]^{\forall_f}&&\PP(Y)\ar[ll]_{f^\m1}.}
$$
We will see in the next example why this notation is apt.

\begin{example}

In logic or computer science, the quantifiers $\exists$ and $\forall$ are used to ask whether any or all elements of a set have a certain property. For example, one may have a set of natural numbers and want to know whether any or all are even or odd.
Let $Y=\{{\tt even,odd}\}$, and let $p\taking\NN\to Y$ be the function that takes assigns to each natural number its parity (even or odd). Because the elements of $\PP(\NN)$ and $\PP(Y)$ are ordered by ``inclusion of subsets", we can construe these orders as categories (by Proposition \ref{prop:preorders to cats}). That's all old; what's new is that we have adjunctions between these categories
$$
\xymatrix{\PP(\NN)\ar@/^1.2pc/[rr]^{\exists_p}\ar@/_1.2pc/[rr]^{\forall_p}&&\PP(Y)\ar[ll]_{p^\m1}.}
$$
Given a subset $U\ss\NN$, i.e. an object $U\in\Ob(\PP(\NN))$, we investigate the objects $\exists_p(U),\forall_p(U)$. These are both subsets of $\{{\tt even,odd}\}$. The set $\exists_p(U)$ includes the element {\tt even} if there exists an even number in $U$; it includes the element {\tt odd} if there exists an odd number in $U$. Similarly, the set $\forall_p(U)$ includes the element {\tt even} if every even number is in $U$ and it includes {\tt odd} if every odd number is in $U$.
\footnote{It may not be clear that by this point we have also handled the question, ``is every element of $U$ even?" One simply checks that {\tt odd} is not an element of $\exists_pU$.}

We explain just one of these in terms of the definitions. Let $V=\{{\tt even}\}\ss Y$. Then $f^\m1(V)\ss\NN$ is the set of even numbers, and there is a morphism $f^\m1(V)\to U$ in $\PP(\NN)$ if and only if $U$ contains all the even numbers. Therefore, the adjunction isomorphism $\Hom_{\PP(\NN)}(f^\m1(V),U)\iso\Hom_{\PP(Y)}(V,\forall_pU)$ says that $V\ss\forall_pU$, i.e. $\forall_p(U)$ includes the element {\tt even} if and only if $U$ contains all the even numbers, as we said above.

\end{example}

\begin{exercise}
The national Scout jamboree is a gathering of Boy Scouts from troops across the US. Let $X$ be the set of Boy Scouts in the US, and let $Y$ be the set of Boy Scout troops in the US. Let $t\taking X\to Y$ be the function that assigns to each Boy Scout his troop. Let $U\ss X$ be the set of Boy Scouts in attendance at this years jamboree. What is the meaning of the objects $\exists_tU$ and $\forall_tU$?
\end{exercise}

\begin{exercise}
Let $X$ be a set and $U\ss X$ a subset. Find a set $Y$ and a function $f\taking X\to Y$ such that $\exists_f(U)$ somehow tells you whether $U$ is non-empty, and such that $\forall_f(U)$ somehow tells you whether $U=X$.
\end{exercise}

In fact, ``quantifiers as adjoints" is part of a larger story. Suppose we think of elements of a set $X$ as bins, or storage areas. An element of $\PP(X)$ can be construed as an injection $U\inj X$, i.e. an assignment of a bin to each element of $U$, with at most one element of $U$ in each bin. Relaxing that restriction, we may consider arbitrary sets $U$ and assignments $U\to X$ of a bin to each element $u\in U$. Given a function $f\taking X\to Y$, we can generalize $\exists_f$ and $\forall_f$ to functors denoted $\Sigma_f$ and $\Pi_f$, which will parameterize disjoint unions and products (respectively) over $y\in Y$. This will be discussed in Section \ref{sec:data migration}.


\subsection{Universal concepts in terms of adjoints}\label{sec:universal concepts}

In this section we discuss how universal concepts, i.e. initial objects and terminal objects, colimits and limits, are easily phrased in the language of adjoint functors. We will say that a functor $F\taking\mcC\to\mcD$ {\em is a left adjoint} if there exists a functor $G\taking\mcD\to\mcC$ such that $F$ is a left adjoint of $G$. We showed in Proposition \ref{prop:unicity of adjoints} that if $F$ is a left adjoint of some functor $G$, then it is isomorphic to every other left adjoint of $G$, and $G$ is isomorphic to every other right adjoint of $F$.

\begin{example}

Let $\mcC$ be a category and $t\taking\mcC\to\ul{1}$ the unique functor to the terminal category. Then $t$ is a left adjoint if and only if $\mcC$ has a terminal object, and $t$ is a right adjoint if and only if $\mcC$ has an initial object. The proofs are dual, so let's focus on the first.

The functor $t$ has a right adjoint $R\taking\ul{1}\to\mcC$ if and only if there is an isomorphism $$\Hom_\mcC(c,r)\iso\Hom_{\ul{1}}(t(c),1),$$ where $r=R(1)$. But $\Hom_{\ul{1}}(t(c),1)$ has one element. Thus $t$ has a right adjoint iff there is a unique morphism $c\to r$ in $\mcC$. This is the definition of $r$ being a terminal object.

\end{example}

When we defined colimits and limits in Definitions \ref{def:coslice and colimit} and \ref{def:slice and limit} we did so for individual $I$-shaped diagrams $X\taking I\to\mcC$. Using adjoints we can define the limit of every $I$-shaped diagram in $\mcC$ at once.

Let $t\taking\mcC\to\ul{1}$ denote the unique functor to the terminal category. Given an object $c\in\Ob(\mcC)$, consider it as a functor $c\taking\ul{1}\to\mcC$. Then $c\circ t\taking I\to\mcC$ is the {\em constant functor at $c$}\index{functor!constant}, sending each object in $I$ to the same $\mcC$-object $c$, and every morphism in $I$ to $\id_c$. This induces a functor that we denote by $\Delta_t\taking\mcC\to\Fun(I,\mcC)$.

Suppose we want to take the colimit or limit of $X$. We are given an object $X$ of $\Fun(I,\mcC)$ and we want back an object of $\mcC$. We could hope, and it turns out to be true, that the adjoints of $\Delta_t$ are the limit and colimit. Indeed let $\Sigma_t\taking\Fun(I,\mcC)\to\mcC$ be the left adjoint of $\Delta_t$, and let $\Pi_t\taking\Fun(I,\mcC)\to\mcC$ be the right adjoint of $\Delta_t$. Then $\Sigma_t$ is the functor that takes colimits, and $\Pi_t$ is the functor that takes limits.

We will work with a generalization of colimits and limits in Section \ref{sec:data migration}. But for now, let's bring this down to earth with a concrete example.

\begin{example}

Let $\mcC=\Set$, and let $I=\ul{3}$. The category $\Fun(I,\Set)$ is the category of $\{1,2,3\}$-indexed sets, e.g. $(\ZZ,\NN,\ZZ)\in\Ob(\Fun(I,\Set))$ is an object of it. The functor $\Delta_t\taking\Set\to\Fun(I,\Set)$ acts as follows. Given a set $c\in\Ob(\Set)$, consider it as a functor $c\taking\ul{1}\to\Set$, and the composite $c\circ t\taking I\to\Set$ is the constant functor. That is, $\Delta_t(c)\taking I\to\Set$ is the $\{1,2,3\}$--indexed set $(c,c,c)$.

To say that $\Delta_t$ has a right adjoint called $\Pi_t\taking\Fun(I,\Set)\to\Set$ and that it ``takes limits" should mean that if we look through the definition of right adjoint, we will see that the formula will somehow yield the appropriate limit. Fix a functor $D\taking I\to\Set$, so $D(1),D(2),$ and $D(3)$ are sets. The limit $\lim D$ of $D$ is the product $D(1)\times D(2)\times D(3)$. For example, if $D=(\ZZ,\NN,\ZZ)$ then $\lim D=\ZZ\times\NN\times\ZZ$. How does this fact arise in the definition of adjoint?

The definition of $\Pi_t$ being the right adjoint to $\Delta_t$ says that there is a natural isomorphism of sets, 
\begin{align}\label{dia:limit as adjoint}
\Hom_{\Fun(I,\Set)}(\Delta_t(c),D)\iso\Hom_\Set(c,\Pi_t(D)).
\end{align}
The left-hand side has elements $f\in\Hom_{\Fun(I,\Set)}(\Delta_t(c),D)$ that look like the left below, but having these three maps is equivalent to having the diagram to the right below:
$$\xymatrix{
c\ar[dd]^{f(1)}&c\ar[dd]^{f(2)}&c\ar[dd]^{f(3)}\\\\
D(1)&D(2)&D(3)
}
\hspace{1in}
\xymatrix{
&c\ar[ddl]_{f(1)}\ar[dd]^{f(2)}\ar[ddr]^{f(3)}\\\\
D(1)&D(2)&D(3)
}$$
The isomorphism in (\ref{dia:limit as adjoint}) says that choosing the three maps $f(1),f(2),f(3)$ is the same thing as choosing a function $c\to\Pi_t(D)$. But this is very close to the universal property of limits: there is a unique map $\ell\taking c\to D(1)\times D(2)\times D(3)$, so this product serves well as $\Pi_t$ as we have said. We're not giving a formal proof here, but what is missing at this point is the fact that certain diagrams have to commute. This comes down to the naturality of the isomorphism (\ref{dia:limit as adjoint}). The map $\ell$ induces a naturality square
$$\xymatrix{\Delta_t(c)\ar[r]^{\Delta_t(\ell)}\ar[d]_f&\Delta_t\Pi_tD\ar[d]^\pi\\D\ar@{=}[r]&D}$$
which says that the following diagram commutes:
$$\xymatrix{&c\ar[ddl]_{f(1)}\ar[dd]^{f(2)}\ar[ddr]^{f(3)}\ar@/_2pc/[dddd]_\ell\\\\
D(1)&D(2)&D(3)\\\\
&D(1)\times D(2)\times D(3)\ar[uul]^{\pi_1}\ar[uu]_{\pi_2}\ar[uur]_{\pi_3}
}$$

\end{example}

It is not hard to show that the composition of left adjoints is a left adjoint, and the composition of right adjoints is a right adjoint. In the following example we show how currying (as in Sections \ref{sec:currying} and \ref{ex:other adjunctions}) arises out of a certain combination of data migration functors. 

\begin{example}[Currying via $\Delta,\Sigma,\Pi$]\index{currying!via data migration functors}

Let $A,B,$ and $C$ be sets. Consider the unique functor $a\taking A\to\ul{1}$ and consider $B$ and $C$ as functors $\ul{1}\Too{B}\Set$ and $\ul{1}\Too{C}\Set$ respectively. 
$$\xymatrix{A\ar[r]^a&\ul{1}\ar@/^1pc/[r]^B\ar@/_1pc/[r]_C&\Set}$$
Note that $\ul{1}\set\iso\Set$, and we will elide the difference. Our goal is to see currying arise out of the adjunction between $\Sigma_a\circ\Delta_a$ and $\Pi_a\circ\Delta_a$, which tells us that there is an isomorphism
\begin{align}\label{dia:migration for currying}
\Hom_\Set(\Sigma_a\Delta_a(B),C)\iso\Hom_\Set(B,\Pi_a\Delta_a(C)).
\end{align}

By definition, $\Delta_a(B)\taking A\to\Set$ assigns to each element $a\in A$ the set $B$. Since $\Sigma_A$ takes disjoint unions, we have a bijection
$$\Sigma_a(\Delta_a(B))=\left(\coprod_{a\in A}B\right)\iso A\times B.$$ 
Similarly $\Delta_a(C)\taking A\to\Set$ assigns to each element $a\in A$ the set $C$. Since $\Pi_A$ takes products, we have a bijection
$$\Pi_a(\Delta_a(C))=\left(\prod_{a\in A}C\right)\iso C^A.$$
The currying isomorphism $\Hom_\Set(A\times B,C)\iso\Hom_\Set(B,C^A)$ falls out of (\ref{dia:migration for currying}).

\end{example}


\subsection{Preservation of colimits or limits}

One useful fact about adjunctions is that left adjoints preserve all colimits and right adjoints preserve all limits. 

\begin{proposition}

Let $\adjoint{L}{\mcB}{\mcA}{R}$ be an adjunction. For any indexing category $I$ and functor $D\taking I\to\mcB$, if $D$ has a colimit in $\mcB$ then there is a unique isomorphism 
$$L(\colim D)\iso \colim (L\circ D).$$

Similarly, for any $I\in\Ob(\Cat)$ and functor $D\taking I\to\mcA$, if $D$ has a limit in $\mcA$ then there is a unique isomorphism 
$$R(\lim D)\iso \lim (R\circ D).$$

\end{proposition}

\begin{proof}

The proof is simple if one knows the Yoneda lemma (Section \ref{sec:yoneda}). I have decided to skip it to keep the book shorter. See \cite{Mac}.

\end{proof}

\begin{example}

Since $\Ob\taking\Cat\to\Set$ is both a left adjoint and a right adjoint, it must preserve both limits and colimits. This means that if you want to know the set of objects in the fiber product of some categories, you can simply take the fiber product of the set of objects in those categories, $$\Ob(\mcA\times_\mcC\mcB)\iso\Ob(\mcA)\times_{\Ob(\mcC)}\Ob(\mcB).$$ While the right-hand side might look daunting, it is just a fiber product in $\Set$ which is quite understandable.

This is greatly simplifying. If one thinks through what defines a limit in $\Cat$, one is dragged through notions of slice categories and terminal objects in them. These slice categories are in $\Cat$ so they involve several categories and functors, and it gets hairy or even hopeless to a beginner. Knowing that the objects are given by a simple fiber product makes the search for limits in $\Cat$ much simpler. 

For example, if $[n]$ is the linear order category of length $n$ then $[n]\times[m]$ has $nm+n+m+1$ objects because $[n]$ has $n+1$ objects and $[m]$ has $m+1$ objects. 

\end{example}

\begin{example}

The ``path poset" functor $L\taking\Grph\to\PrO$ given by existence of paths (see Exercise \ref{exc:grph to pro}) is left adjoint to the functor $R\taking\PrO\to\Grph$ given by replacing $\leq$'s by arrows. This means that $L$ preserves colimits. So taking the union of graphs $G$ and $H$ results in a graph whose path poset  $L(G\sqcup H)$ is the union of the path posets of $G$ and $H$. But this is not so for products. 

Let $G=H=\fbox{\xymatrix{\LMO{a}\ar[r]^f&\LMO{b}}}$. Then $L(G)=L(H)=[1]$, the linear order of length 1. But the product $G\times H$ in $\Grph$ looks like the graph 
$$\xymatrix{\LMO{(a,a)}\ar[rd]&\LMO{(a,b)}\\\LMO{(b,a)}&\LMO{(b,b)}}$$
Its preorder $L(G\times H)$ does not have $(a,a)\leq(a,b)$, whereas this is the case in $L(G)\times L(H)$.

\end{example}


\subsection{Data migration}\label{sec:data migration}\index{data migration}

As we saw in Sections \ref{sec:schemas and cats intro} and \ref{sec:instances}, a database schema is a category $\mcC$ and an instance is a functor $I\taking\mcC\to\Set$.  

\begin{notation}

Let $\mcC$ be a category. Throughout this section we denote by $\mcC\set$ the category $\Fun(\mcC,\Set)$ of functors from $\mcC$ to $\Set$, i.e. the category of instances on $\mcC$. 

\end{notation}

In this section we discuss what happens to the resulting instances when different schemas are connected by a functor, say $F\taking\mcC\to\mcD$. It turns out that three adjoint functors emerge: $\Delta_F\taking\mcD\set\to\mcC\set$, $\Sigma_F\taking\mcC\set\to\mcD\set$, and $\Pi_F\taking\mcC\set\to\mcD\set$, where $\Delta_F$ is adjoint to both, 
$$
\adjoint{\Sigma_F}{\mcC\set}{\mcD\set}{\Delta_F}
\hspace{.7in}
\adjoint{\Delta_F}{\mcD\set}{\mcC\set}{\Pi_F.}
$$
It turns out that almost all the basic database operations are captured by these three functors. For example, $\Delta_F$ handles the job of duplicating or deleting tables, as well as duplicating or deleting columns in a single table. The functor $\Sigma_F$ handles taking unions, and the functor $\Pi_F$ handles joining tables together, matching columns, or selecting the rows with certain properties (e.g. everyone whose first name is Mary).


\subsubsection{Pullback: $\Delta$}\index{data migration!pullback $\Delta$}

Given a functor $F\taking\mcC\to\mcD$ and a functor $I\taking\mcD\to\Set$, we can compose them to get a functor $I\circ F\taking\mcC\to\Set$. In other words, the presence of $F$ provides a way to convert $\mcD$-instances into $\mcC$-instances. In fact this conversion is functorial, meaning that morphisms of $\mcD$-instances are sent to morphisms of $\mcC$-instances. We denote the resulting functor by $\Delta_F\taking\mcD\set\to\mcC\set$ and call it {\em pullback along $F$}.

We have seen an example of this before in Example \ref{ex:whiskering}, where we showed how a monoid homomorphism $F\taking\mcM'\to\mcM$ could add functionality to a finite state machine. More generally, we can use pullbacks to reorganize data, copying and deleting tables and columns. 

\begin{remark}

Given a functor $F\taking\mcC\to\mcD$, which we think of as a schema translation, the functor $\Delta_F\taking\mcD\set\to\mcC\set$ ``goes the opposite way". The reasoning is simple to any explain (composition of functors) but something about it is often very strange to people, at first. The rough idea of this ``contravariance" is captured by the role-reversal in the following slogan:

\begin{slogan} 
If I get my information from you, then your information becomes my information. 
\end{slogan}

\end{remark}

Consider the following functor $F\taking\mcC\to\mcD$: 
\footnote{This example was taken from \cite{Sp1}, \url{http://arxiv.org/abs/1009.1166}.}
\begin{align}\label{dia:translation}
\mcC:=\parbox{1.2in}{\fbox{\xymatrix@=10pt{&\LTO{SSN}\\&\LTO{First}\\\color{red}{\LTO{T1}}\ar[uur]\ar[ur]\ar[dr]&&\color{red}{\LTO{T2}}\ar[ul]\ar[dl]\ar[ddl]\\&\LTO{Last}\\&\LTO{Salary}}}}\Too{F}\parbox{.8in}{\fbox{\xymatrix@=10pt{&\LTO{SSN}\\&\LTO{First}\\\color{red}{\LTO{T}}\ar[uur]\ar[ur]\ar[dr]\ar[ddr]\\&\LTO{Last}\\&\LTO{Salary}}}}=:\mcD
\end{align}

Let's spend a moment recalling how to ``read" schemas. In schema $\mcC$ there are leaf tables {\tt SSN, First, Last, Salary}, which represent different kinds of basic data. More interestingly, there are two {\em fact tables}\index{schema!leaf table}\index{schema!fact table}. The first is called {\tt T1} and it relates {\tt SSN, First,} and ${\tt Last}$. The second is called {\tt T2} and it relates {\tt First, Last,} and {\tt Salary}.

The functor $F\taking\mcC\to\mcD$ relates $\mcC$ to a schema with a single fact table relating all four attributes: {\tt SSN, First, Last,} and {\tt Salary}. We are interested in $\Delta_F\taking\mcD\set\to\mcC\set$. Suppose given the following database instance $I\taking\mcD\to\Set$ on $\mcD$:
$$
\begin{tabular}{| l || l | l | l | l |}\bhline\multicolumn{5}{| c |}{{\tt T}}\\\bhline {\bf ID}&{\bf SSN}&{\bf First}&{\bf Last}&{\bf Salary}\\\bbhline XF667&115-234&Bob&Smith&\$250\\\hline XF891&122-988&Sue&Smith&\$300\\\hline XF221&198-877&Alice&Jones&\$100\\\bhline
\end{tabular}
$$
$$
\begin{tabular}{| l ||}\bhline\multicolumn{1}{| c |}{{\tt SSN}}\\\bhline {\bf ID}\\\bbhline 115-234\\\hline 118-334\\\hline 122-988\\\hline 198-877\\\hline 342-164\\\bhline
\end{tabular}\hspace{.5in}
\begin{tabular}{| l ||}\bhline\multicolumn{1}{| c |}{{\tt First}}\\\bhline {\bf ID}\\\bbhline Adam\\\hline Alice\\\hline Bob\\\hline Carl\\\hline Sam\\\hline Sue\\\bhline
\end{tabular}
\hspace{.5in}
\begin{tabular}{| l ||}\bhline\multicolumn{1}{| c |}{{\tt Last}}\\\bhline {\bf ID}\\\bbhline Jones\\\hline Miller\\\hline Pratt\\\hline Richards\\\hline Smith\\\bhline
\end{tabular}
\hspace{.5in}
\begin{tabular}{| l ||}\bhline\multicolumn{1}{| c |}{{\tt Salary}}\\\bhline {\bf ID}\\\bbhline \$100\\\hline \$150\\\hline \$200\\\hline \$250\\\hline \$300\\\bhline
\end{tabular}
$$

How do you get the instance $\Delta_F(I)\taking\mcC\to\Set$? The formula was given above: compose $I$ with $F$. In terms of tables, it feels like duplicating table {\tt T} as {\tt T1} and {\tt T2}, but deleting a column from each in accordance with the definition of $\mcC$ in (\ref{dia:translation}). Here is the result, $\Delta_F(I)$, in table form:

$$\begin{tabular}{| l || l | l | l |}\bhline\multicolumn{4}{| c |}{{\tt T1}}\\\bhline {\bf ID}&{\bf SSN}&{\bf First}&{\bf Last}\\\bbhline XF667&115-234&Bob&Smith\\\hline XF891&122-988&Sue&Smith\\\hline XF221&198-877&Alice&Jones\\\bhline
\end{tabular}
\hspace{.5in}
\begin{tabular}{| l || l | l | l |}\bhline\multicolumn{4}{| c |}{{\tt T2}}\\\bhline {\bf ID}&{\bf First}&{\bf Last}&{\bf Salary}\\\bbhline XF221&Alice&Jones&\$100\\\hline XF667&Bob&Smith&\$250\\\hline XF891&Sue&Smith&\$300\\\hline 
\end{tabular}
$$
$$
\begin{tabular}{| l ||}\bhline\multicolumn{1}{| c |}{{\tt SSN}}\\\bhline {\bf ID}\\\bbhline 115-234\\\hline 118-334\\\hline 122-988\\\hline 198-877\\\hline 342-164\\\bhline
\end{tabular}\hspace{.5in}
\begin{tabular}{| l ||}\bhline\multicolumn{1}{| c |}{{\tt First}}\\\bhline {\bf ID}\\\bbhline Adam\\\hline Alice\\\hline Bob\\\hline Carl\\\hline Sam\\\hline Sue\\\bhline
\end{tabular}
\hspace{.5in}
\begin{tabular}{| l ||}\bhline\multicolumn{1}{| c |}{{\tt Last}}\\\bhline {\bf ID}\\\bbhline Jones\\\hline Miller\\\hline Pratt\\\hline Richards\\\hline Smith\\\bhline
\end{tabular}
\hspace{.5in}
\begin{tabular}{| l ||}\bhline\multicolumn{1}{| c |}{{\tt Salary}}\\\bhline {\bf ID}\\\bbhline \$100\\\hline \$150\\\hline \$200\\\hline \$250\\\hline \$300\\\bhline
\end{tabular}
$$

\begin{exercise}\index{leaf table}
Let $\mcC=(G,\simeq)$ be a schema. A leaf table is an object $c\in\Ob(\mcC)$ with no outgoing arrows.
\sexc Write the condition of being a ``leaf table" mathematically in three different languages: that of graphs (using symbols $V,A,src,tgt$), that of categories (using $\Hom_\mcC$, etc.), and that of tables (in terms of columns, tables, rows, etc.).
\next In the language of categories, is there a difference between a terminal object and a leaf table? Explain.
\endsexc
\end{exercise}

\begin{exercise}
Consider the schemas $$[1]=\fbox{$\LMO{\color{blue}{0}}\Too{f}\LMO{\color{red}{1}}$}\hsp\tn{and}\hsp [2]=\fbox{$\LMO{\color{blue}{0}}\Too{g}\LMO{1}\Too{h}\LMO{\color{red}{2}}$},$$ and the functor $F\taking [1]\to[2]$ given by sending $0\mapsto 0$ and $1\mapsto 2$. 
\sexc How many possibilities are there for $F(f)$?
\next Now suppose $I\taking[2]\to\Set$ is given by the following tables. 
$$
\begin{tabular}{| l || l |}
\bhline
\multicolumn{2}{|c|}{0}\\\bhline
{\bf ID}&{\bf g}\\\bbhline
Am&To be verb\\\hline
Baltimore&Place\\\hline
Carla&Person\\\hline
Develop&Action verb\\\hline
Edward&Person\\\hline
Foolish&Adjective\\\hline
Green&Adjective\\\bhline
\end{tabular}
\hspace{.4in}
\begin{tabular}{| l || l |}
\bhline
\multicolumn{2}{|c|}{1}\\\bhline
{\bf ID}&{\bf h}\\\bbhline
Action verb&Verb\\\hline
Adjective&Adjective\\\hline
Place&Noun\\\hline
Person&Noun\\\hline
To be verb&Verb\\\bhline
\end{tabular}
\hspace{.4in}
\begin{tabular}{| l ||}
\bhline
\multicolumn{1}{| c |}{2}\\\bhline
{\bf ID}\\\bbhline
Adjective\\\hline
Noun\\\hline
Verb\\\bhline
\end{tabular}
$$
Write out the two tables associated to the $[1]$-instance $\Delta_F(I)\taking[1]\to\Set$.
\endsexc
\end{exercise}


\subsubsection{Left pushforward: $\Sigma$}\label{sec:left push}\index{data migration!left pushforward $\Sigma$}

Let $F\taking\mcC\to\mcD$ be a functor. The functor $\Delta_F\taking\mcD\set\to\mcC\set$ has a left adjoint, $\Sigma_F\taking\mcC\set\to\mcD\set$. The rough idea is that $\Sigma_F$ performs parameterized colimits. Given an instance $I\taking\mcC\to\Set$, we get an instance on $\mcD$ that acts as follows. For each object $d\in\Ob(\mcD)$, the set $\Sigma_F(I)(d)$ is the colimit (think, union) of some diagram back home in $\mcC$. 

Left pushforwards (also known as left Kan extensions) are discussed at length in \cite{Sp1}; here we begin with some examples from that paper.\index{Kan extension!left}

\begin{example}\label{ex:left pushforward and skolem}

We again use the functor $F\taking\mcC\to\mcD$ drawn below
\begin{align}\tag{\ref{dia:translation}}\mcC:=\parbox{1.2in}{\fbox{\xymatrix@=10pt{&\LTO{SSN}\\&\LTO{First}\\\color{red}{\LTO{T1}}\ar[uur]\ar[ur]\ar[dr]&&\color{red}{\LTO{T2}}\ar[ul]\ar[dl]\ar[ddl]\\&\LTO{Last}\\&\LTO{Salary}}}}\Too{F}\parbox{.8in}{\fbox{\xymatrix@=10pt{&\LTO{SSN}\\&\LTO{First}\\\color{red}{\LTO{T}}\ar[uur]\ar[ur]\ar[dr]\ar[ddr]\\&\LTO{Last}\\&\LTO{Salary}}}}=:\mcD
\end{align}
We will be applying the left pushforward $\Sigma_F\taking\mcC\set\to\mcD\set$ to the following instance $I\taking\mcC\to\Set$: 
$$
\begin{tabular}{| l || l | l | l |}\bhline\multicolumn{4}{| c |}{{\tt T1}}\\\bhline {\bf ID}&{\bf SSN}&{\bf First}&{\bf Last}\\\bbhline T1-001&115-234&Bob&Smith\\\hline T1-002&122-988&Sue&Smith\\\hline T1-003&198-877&Alice&Jones\\\bhline
\end{tabular}
\hsp
\begin{tabular}{| l || l | l | l |}\bhline\multicolumn{4}{| c |}{{\tt T2}}\\\bhline {\bf ID}&{\bf First}&{\bf Last}&{\bf Salary}\\\bbhline T2-001&Alice&Jones&\$100\\\hline T2-002&Sam&Miller&\$150\\\hline T2-004&Sue&Smith&\$300\\\hline T2-010&Carl&Pratt&\$200 \\\bhline
\end{tabular}
$$
$$
\begin{tabular}{| l ||}\bhline\multicolumn{1}{| c |}{{\tt SSN}}\\\bhline {\bf ID}\\\bbhline 115-234\\\hline 118-334\\\hline 122-988\\\hline 198-877\\\hline 342-164\\\bhline
\end{tabular}
\hsp\hsp
\begin{tabular}{| l ||}\bhline\multicolumn{1}{| c |}{{\tt First}}\\\bhline {\bf ID}\\\bbhline Adam\\\hline Alice\\\hline Bob\\\hline Carl\\\hline Sam\\\hline Sue\\\bhline
\end{tabular}
\hsp\hsp
\begin{tabular}{| l ||}\bhline\multicolumn{1}{| c |}{{\tt Last}}\\\bhline {\bf ID}\\\bbhline Jones\\\hline Miller\\\hline Pratt\\\hline Richards\\\hline Smith\\\bhline
\end{tabular}
\hsp\hsp
\begin{tabular}{| l ||}\bhline\multicolumn{1}{| c |}{{\tt Salary}}\\\bhline {\bf ID}\\\bbhline \$100\\\hline \$150\\\hline \$200\\\hline \$250\\\hline \$300\\\bhline
\end{tabular}
$$

The functor $F\taking\mcC\to\mcD$ sent both tables {\tt T1} and {\tt T2} to table {\tt T}. Applying $\Sigma_F$ will take the what was in {\tt T1} and {\tt T2} and put the union in {\tt T}. The result $\Sigma_FI\taking\mcD\to\Set$ is as follows:
\begin{center}
\begin{tabular}{| l || l | l | l | l |}\bhline\multicolumn{5}{| c |}{{\tt T}}\\\bhline {\bf ID}&{\bf SSN}&{\bf First}&{\bf Last}&{\bf Salary}\\\bbhline  T1-001&115-234&Bob&Smith&T1-001.Salary\\\hline T1-002&122-988&Sue&Smith&T1-002.Salary\\\hline T1-003&198-877&Alice&Jones&T1-003.Salary\\\hline T2-001&T2-A101.SSN&Alice&Jones&\$100\\\hline T2-002&T2-A102.SSN&Sam&Miller&\$150\\\hline T2-004&T2-004.SSN&Sue&Smith&\$300\\\hline T2-010&T2-A110.SSN&Carl&Pratt&\$200 \\\bhline
\end{tabular}
\end{center}
$$
\begin{tabular}{| l ||}\bhline
\multicolumn{1}{| c |}{{\tt SSN}}\\\bhline 
{\bf ID}\\\bbhline 
115-234\\\hline 
118-334\\\hline 
122-988\\\hline 
198-877\\\hline 
342-164\\\hline
T2-001.SSN\\\hline
T2-002.SSN\\\hline
T2-004.SSN\\\hline
T2-010.SSN\\\bhline
\end{tabular}
\hspace{.5in}
\begin{tabular}{| l ||}\bhline
\multicolumn{1}{| c |}{{\tt First}}\\\bhline 
{\bf ID}\\\bbhline 
Adam\\\hline 
Alice\\\hline 
Bob\\\hline 
Carl\\\hline 
Sam\\\hline 
Sue\\\bhline
\end{tabular}
\hspace{.5in}
\begin{tabular}{| l ||}\bhline
\multicolumn{1}{| c |}{{\tt Last}}\\\bhline 
{\bf ID}\\\bbhline 
Jones\\\hline 
Miller\\\hline 
Pratt\\\hline 
Richards\\\hline 
Smith\\\bhline
\end{tabular}
\hspace{.5in}
\begin{tabular}{| l ||}\bhline
\multicolumn{1}{| c |}{{\tt Salary}}\\\bhline 
{\bf ID}\\\bbhline 
\$100\\\hline 
\$150\\\hline 
\$200\\\hline 
\$250\\\hline 
\$300\\\hline
T1-001.Salary\\\hline
T1-002-Salary\\\hline
T1-003-Salary\\\bhline
\end{tabular}
$$

As you can see, there was no set salary information for any data coming from table {\tt T1} nor any set SSN information for any data coming form table {\tt T2}. But the definition of adjoint, given in Definition \ref{def:adjunction}, yielded the universal response: freely add new variables that take the place of missing information. It turns out that this idea already has a name in logic, {\em Skolem variables}, and a name in database theory, {\em labeled nulls}.\index{labeled null}\index{Skolem variable}

\end{example}

\begin{exercise}
Consider the functor $F\taking\ul{3}\to\ul{2}$ sending $1\mapsto 1, 2\mapsto 2, 3\mapsto 2$.
\sexc Write down an instance $I\taking\ul{3}\to\Set$.
\next Given the description that ``$\Sigma_F$ performs a parameterized colimit", make an educated guess about what $\Sigma_F(I)$ will be. Give your answer in the form of two sets that are made up from the three sets you already wrote down.
\endsexc
\end{exercise}

We now briefly give the actual formula for computing left pushforwards. Suppose that $F\taking\mcC\to\mcD$ is a functor and let $I\taking\mcC\to\Set$ be a set-valued functor on $\mcC$. Then $\Sigma_F(I)\taking\mcD\to\Set$ is defined as follows. Given an object $d\in\Ob(\mcD)$ we first form the comma category (see Definition \ref{def:comma category}) for the setup
$$\mcC\To{F}\mcD\From{d}\ul{1}$$
and denote it by $(F\down d)$. There is a canonical projection functor $\pi\taking(F\down d)\to\mcC$, which we can compose with $I\taking\mcC\to\Set$ to obtain a functor $(F\down d)\to\Set$. We are ready to define $\Sigma_F(I)(d)$ to be its colimit,
$$\Sigma_F(I)(d):=\colim_{(F\down d)}I\circ\pi.$$
We have defined $\Sigma_F(I)\taking\mcD\to\Set$ on objects $d\in\Ob(\mcD)$. As for morphisms we will be even more brief, but one can see \cite{Sp1} for details. Given a morphism $g\taking d\to d'$ one notes that there is an induced functor $(F\down g)\taking (F\down d)\to(F\down d')$ and a commutative diagram of categories:
$$
\xymatrix@=15pt{
(F\down d)\ar[rr]^{(F\down g)}\ar[ddr]^{\pi}\ar[ddddr]_{I\circ\pi}&&(F\down d')\ar[ddl]_{\pi'}\ar[ddddl]^{I\circ\pi'}\\\\
&\mcC\ar[dd]^I\\\\
&\Set
}
$$
By the universal property of colimits, this induces the required function $$\colim_{(F\down d)}I\circ\pi\Too{\Sigma_F(I)(g)}\colim_{(F\down d')}I\circ\pi'.$$


\subsubsection{Right pushforward: $\Pi$}\index{data migration!right pushforward $\Pi$}

Let $F\taking\mcC\to\mcD$ be a functor. We heard in Section \ref{sec:left push} that the functor $\Delta_F\taking\mcD\set\to\mcC\set$ has a left adjoint. Here we explain that it has a right adjoint, $\Pi_F\taking\mcC\set\to\mcD\set$ as well. The rough idea is that $\Pi_F$ performs parameterized limits. Given an instance $I\taking\mcC\to\Set$, we get an instance on $\mcD$ that acts as follows. For each object $d\in\Ob(\mcD)$, the set $\Pi_F(I)(d)$ is the limit (think, fiber product) of some diagram back home in $\mcC$. 

Right pushforwards (also known as right Kan extensions) are discussed at length in \cite{Sp1}; here we begin with some examples from that paper.\index{Kan extension!right}

\begin{example}

We once again use the functor $F\taking\mcC\to\mcD$ from Example \ref{ex:left pushforward and skolem}. We will apply the right pushforward $\Pi_F$ to instance $I\taking\mcC\to\Set$ from that example.
\footnote{To repeat for convenience,
\begin{align}\tag{\ref{dia:translation}}\mcC:=\parbox{1.2in}{\fbox{\xymatrix@=10pt{&\LTO{SSN}\\&\LTO{First}\\\color{red}{\LTO{T1}}\ar[uur]\ar[ur]\ar[dr]&&\color{red}{\LTO{T2}}\ar[ul]\ar[dl]\ar[ddl]\\&\LTO{Last}\\&\LTO{Salary}}}}\Too{F}\parbox{.8in}{\fbox{\xymatrix@=10pt{&\LTO{SSN}\\&\LTO{First}\\\color{red}{\LTO{T}}\ar[uur]\ar[ur]\ar[dr]\ar[ddr]\\&\LTO{Last}\\&\LTO{Salary}}}}=:\mcD
\end{align}
$I\taking\mcC\to\Set$ is 
$$
\begin{tabular}{| l || l | l | l |}\bhline\multicolumn{4}{| c |}{{\tt T1}}\\\bhline {\bf ID}&{\bf SSN}&{\bf First}&{\bf Last}\\\bbhline T1-001&115-234&Bob&Smith\\\hline T1-002&122-988&Sue&Smith\\\hline T1-003&198-877&Alice&Jones\\\bhline
\end{tabular}
\hsp
\begin{tabular}{| l || l | l | l |}\bhline\multicolumn{4}{| c |}{{\tt T2}}\\\bhline {\bf ID}&{\bf First}&{\bf Last}&{\bf Salary}\\\bbhline T2-001&Alice&Jones&\$100\\\hline T2-002&Sam&Miller&\$150\\\hline T2-004&Sue&Smith&\$300\\\hline T2-010&Carl&Pratt&\$200 \\\bhline
\end{tabular}
$$
$$
\begin{tabular}{| l ||}\bhline\multicolumn{1}{| c |}{{\tt SSN}}\\\bhline {\bf ID}\\\bbhline 115-234\\\hline 118-334\\\hline 122-988\\\hline 198-877\\\hline 342-164\\\bhline
\end{tabular}
\hsp\hsp
\begin{tabular}{| l ||}\bhline\multicolumn{1}{| c |}{{\tt First}}\\\bhline {\bf ID}\\\bbhline Adam\\\hline Alice\\\hline Bob\\\hline Carl\\\hline Sam\\\hline Sue\\\bhline
\end{tabular}
\hsp\hsp
\begin{tabular}{| l ||}\bhline\multicolumn{1}{| c |}{{\tt Last}}\\\bhline {\bf ID}\\\bbhline Jones\\\hline Miller\\\hline Pratt\\\hline Richards\\\hline Smith\\\bhline
\end{tabular}
\hsp\hsp
\begin{tabular}{| l ||}\bhline\multicolumn{1}{| c |}{{\tt Salary}}\\\bhline {\bf ID}\\\bbhline \$100\\\hline \$150\\\hline \$200\\\hline \$250\\\hline \$300\\\bhline
\end{tabular}
$$
}

The instance $\Pi_F(I)$ will put data in all 5 tables in $\mcD$. In {\tt T} it will put pairs $(t_1,t_2)$ where $t_1$ is a row in {\tt T1} and $t_2$ is a row in {\tt T2} for which the first and last names agree. It will copy the leaf tables exactly, so we do not display them here; the following is the table {\tt T} for $\Pi_F(I)$:
\begin{center}
\begin{tabular}{| l || l | l | l | l |}\bhline\multicolumn{5}{| c |}{{\tt T}}\\\bhline {\bf ID}&{\bf SSN}&{\bf First}&{\bf Last}&{\bf Salary}\\\bbhline  T1-002T2-A104&122-988&Sue&Smith&\$300\\\hline T1-003T2-A101&198-877&Alice&Jones&\$100\\\bhline
\end{tabular}
\end{center}
Looking at {\tt T1} and {\tt T2}, there were only two ways to match first and last names.
\end{example}

\begin{exercise}
Consider the functor $F\taking\ul{3}\to\ul{2}$ sending $1\mapsto 1, 2\mapsto 2, 3\mapsto 2$.
\sexc Write down an instance $I\taking\ul{3}\to\Set$.
\next Given the description that ``$\Pi_F$ performs a parameterized limit", make an educated guess about what $\Pi_F(I)$ will be. Give your answer in the form of two sets that are made up from the three sets you already wrote down.
\endsexc
\end{exercise}

We now briefly give the actual formula for computing right pushforwards. Suppose that $F\taking\mcC\to\mcD$ is a functor and let $I\taking\mcC\to\Set$ be a set-valued functor on $\mcC$. Then $\Pi_F(I)\taking\mcD\to\Set$ is defined as follows. Given an object $d\in\Ob(\mcD)$ we first form the comma category (see Definition \ref{def:comma category}) for the setup
$$\ul{1}\To{d}\mcD\From{F}\mcC$$
and denote it by $(d\down F)$. There is a canonical projection functor $\pi\taking(d\down F)\to\mcC$, which we can compose with $I\taking\mcC\to\Set$ to obtain a functor $(d\down F)\to\Set$. We are ready to define $\Pi_F(I)(d)$ to be its limit,
$$\Pi_F(I)(d):=\lim_{(d\down F)}I\circ\pi.$$
We have defined $\Pi_F(I)\taking\mcD\to\Set$ on objects $d\in\Ob(\mcD)$. As for morphisms we will be even more brief, but one can see \cite{Sp1} for details. Given a morphism $g\taking d\to d'$ one notes that there is an induced functor $(g\down F)\taking (d'\down F)\to(d\down F)$ and a commutative diagram of categories:
$$
\xymatrix@=15pt{
(d'\down F)\ar[rr]^{(g\down F)}\ar[ddr]^{\pi'}\ar[ddddr]_{I\circ\pi'}&&(d\down F)\ar[ddl]_{\pi}\ar[ddddl]^{I\circ\pi}\\\\
&\mcC\ar[dd]^I\\\\
&\Set
}
$$
By the universal property of limits, this induces the required function $$\lim_{(d\down F)}I\circ\pi\Too{\Pi_F(I)(g)}\lim_{(d'\down F)}I\circ\pi'.$$


\section{Categories of functors}

For any two categories $\mcC$ and $\mcD$,
\footnote{Technically $\mcC$ has to be small (see Remark \ref{rmk:small}), but as we said there, we are not worrying about that distinction in this book.}
we discussed the category $\Fun(\mcC,\mcD)$ of functors and natural transformations between them. In this section we discuss functor categories a bit more and give some important applications within mathematics (sheaves) that extend to the real world.


\subsection{Set-valued functors}

Let $\mcC$ be a category. Then we have been writing $\mcC\set$ to denote the functor category $\Fun(\mcC,\Set)$. Here is a nice result about these categories.

\begin{proposition}\label{prop:inst closed under colim lim}

Let $\mcC$ be a category. The category $\mcC\set$ is closed under colimits and limits.

\end{proposition}

\begin{proof}[Sketch of proof]

Let $J$ be an indexing category and $D\taking J\to\mcC\set$ a functor. For each object $c\in\Ob(\mcC)$, we have a functor $D_c\taking J\to\Set$ defined by $D_c(j)=D(j)(c)$. Define a functor $L\taking\mcC\to\Set$ by $L(c)=\lim_J D_c$, and note that for each $f\taking c\to c'$ in $\mcC$ there is an induced function $L(f)\taking L(c)\to L(c')$. One can check that $L$ is a limit of $J$, because it satisfies the relevant universal property. 

The dual proof holds for colimits.

\end{proof}

\begin{application}

When taking in data about a scientific subject, one often finds that the way one thinks about the problem changes over time. We understand this phenomenon in the language of databases in terms of a \href{http://en.wikipedia.org/wiki/Schema_evolution}{\text series of schemas} $\mcC_1,\mcC_2,\ldots,\mcC_n$, perhaps indexed chronologically. The problem is that old data is held in old schemas and we want to see it in our current understanding. The first step is to transfer all the old data to our new schema in the freest possible way, that is, making no assumptions about how to fill in the new fields. If one creates functors $F_i\taking\mcC_i\to\mcC_{i+1}$ from each of these schemas to the next, then we can push the data forward using $\Sigma_{F_i}$. 

Doing this we will have $n$ datasets on $\mcD:=\mcC_n$, namely one for each ``epoch of understanding". Since the category $\mcD\set$ has all colimits, we can take the union of these datasets and get one. It will have many Skolem variables (see Example \ref{ex:left pushforward and skolem}), and these need to be handled in a coherent way. However, the universality of left adjoints could be interpreted as saying that any reasonable formula for handling this old data can be applied to our results.

\end{application}

\begin{exercise}\label{exc:universal objects in C-set}\index{initial object!in $\mcC\set$}\index{terminal object!in $\mcC\set$}
By Proposition \ref{prop:inst closed under colim lim}, the category $\mcC\set$ is closed under taking limits. By Exercises \ref{exc:terminal as limit} and \ref{exc:initial as colimit}, this means in particular that $\mcC\set$ has an initial object and a terminal object. 
\sexc Let $A\in\Ob(\mcC\set)$ be the initial object, considered as a functor $A\taking\mcC\to\Set$. For any $c\in\Ob(\mcC)$, what is the set $A(c)$?
\next Let $Z\in\Ob(\mcC\set)$ be the terminal object, considered as a functor $Z\taking\mcC\to\Set$. For any $c\in\Ob(\mcC)$, what is the set $Z(c)$?
\endsexc
\end{exercise}

Proposition \ref{prop:inst closed under colim lim} says that we can add or multiply database states together. In fact, database states on $\mcC$ form what is called a {\em topos} which means that just about every consideration we made for sets holds for instances on any schema. Perhaps the simplest schema is $\mcC=\fbox{$\bullet$}$, on which the relevant topos is indeed $\Set$. But schemas can be arbitrarily complex, and it is impressive that all of these considerations make sense in such generality. Here is a table that makes a comparison between these domains.
\begin{center}
\begin{tabular}{| l | l |}\bhline
\multicolumn{2}{| c |}{Dictionary between $\Set$ and $\mcC\set$}\\\hline
{\bf Concept in $\Set$}&{\bf Concept in $\mcC\set$}\\\bbhline
Set & Object in $\mcC\set$\\\hline
Function & Morphism in $\mcC\set$\\\hline
Element&Representable functor\\\hline
Empty set & Initial object\\\hline
Natural numbers&Natural numbers object\\\hline
Image&Image\\\hline
(Co)limits&(Co)limits\\\hline
Exponential objects&Exponential objects\\\hline
``Familiar" arithmetic&``Familiar" arithmetic\\\hline
Power sets $2^X$&Power objects $\Omega^X$\\\hline
Characteristic functions&Characteristic morphisms\\\hline
Surjections, injections&Epimorphisms, monomorphisms\\\bhline
\end{tabular}
\end{center}

In the above table we said that elements of a set are akin to representable functors in $\mcC\set$, but we have not yet defined those; we do so in Section \ref{sec:representable functors}. First we briefly discuss monomorphisms and epimorphisms in general (Definition \ref{def:mono, epi}) and then in $\mcC\set$ (Proposition \ref{prop:epi mono in c-set}). 

\begin{definition}[Monomorphism, Epimorphism]\label{def:mono, epi}\index{epimorphism}\index{monomorphism}

Let $\mcS$ be a category and let $f\taking X\to Y$ be a morphism. We say that $f$ is a {\em monomorphism} if it has the following property. For all objects $A\in\Ob(\mcS)$ and morphisms $g,g'\taking A\to X$ in $\mcS$, 
$$
\xymatrix{A\ar@/^1pc/[r]^g\ar@/_1pc/[r]_{g'}&X\ar[r]^f&Y}
$$
if $f\circ g=f\circ g'$ then $g=g'$.

We say that $f\taking X\to Y$ is an {\em epimorphism} if it has the following property. For all objects $B\in\Ob(\mcS)$ and morphisms $h,h'\taking Y\to B$ in $\mcS$,
$$
\xymatrix{X\ar[r]^f&Y\ar@/^1pc/[r]^h\ar@/_1pc/[r]_{h'}&B}
$$
if $h\circ f=h'\circ f$ then $h=h'$.

\end{definition}

In the category of sets, monomorphisms are the same as injections and epimorphisms are the same as surjections (see Proposition \ref{prop:inj and surj}). The same is true in $\mcC\set$: one can check ``table by table" that a morphism of instances is mono or epi.

\begin{proposition}\label{prop:epi mono in c-set}

Let $\mcC$ be a category and let $X,Y\taking\mcC\to\Set$ be objects in $\mcC\set$ and let $f\taking X\to Y$ be a morphism in $\mcC\set$. Then $f$ is a monomorphism (respectively an epimorphism) if and only if, for every object $c\in\Ob(\mcC)$, the function $f(c)\taking X(c)\to Y(c)$ is injective (respectively surjective). 

\end{proposition}

\begin{proof}[Sketch of proof]

We first show that if $f$ is mono (respectively epi) then so is $f(c)$ for all $c\in\Ob(\mcC)$. Considering $c$ as a functor $c\taking\ul{1}\to\mcC$, this result follows from the fact that $\Delta_c$ preserves limits and colimits, hence monos and epis. 

We now check that if $f(c)$ is mono for all $c\in\Ob(\mcC)$ then $f$ is mono. Suppose that $g,g'\taking A\to X$ are morphisms in $\mcC\set$ such that $f\circ g=f\circ g'$. Then for every $c$ we have $f\circ g(c)=f\circ g'(c)$ which implies by hypothesis that $g(c)=g'(c)$. But the morphisms in $\mcC\set$ are natural transformations, and if two natural transformations $g,g'$ have the same components then they are the same. 

A similar argument works to show the analogous result for epimorphisms.

\end{proof}


\subsubsection{Representable functors}\label{sec:representable functors}

Given a category $\mcC$, there are certain functors $\mcC\to\Set$ that come with the package, one for every object in $\mcC$. So if $\mcC$ is a database schema, then for every table $c\in\Ob(\mcC)$ there is a certain database instance associated to it. These instances, i.e. set-valued functors, are called representable functors, and they'll be defined in Definition \ref{dec:representable functor}. The idea is that if a database schema represents a conceptual layout of types (e.g. as an olog), then each type $T$ has an instance associated to it, standing for ``the generic thing of type $T$ with all its generic attributes".

\begin{definition}\label{def:representable functor}\index{functor!representable}\index{representable functor}

Let $\mcC$ be a category and let $c\in\Ob(\mcC)$ be an object. The functor $\Hom_\mcC(c,-)\taking\mcC\to\Set$, sending $d\in\Ob(\mcC)$ to the set $\Hom_\mcC(c,d)$ and acting similarly on morphisms $d\to d'$, is said to be {\em represented by $c$}. If a functor $F\taking\mcC\to\Set$ is isomorphic to $\Hom_\mcC(c,-)$, we say that $F$ is {\em a representable functor}. We sometimes write $Y_c:=\Hom_\mcC(c,-)$ for short.

\end{definition}

\begin{example}

Given a category $\mcC$ and an object $c\in\Ob(\mcC)$, we get a representable functor. If we think of $\mcC$ as a database schema and $c$ as a table, then what does the representable functor $Y_c\taking\mcC\to\Set$ look like in terms of databases? It turns out that the following procedure will generate it. 

Begin by writing a new row, say ``$\smiley$", in the ID column of table $c$. For each foreign key column $f\taking c\to c'$, add a row in the ID column of table $c'$ called $``f(\smiley)"$ and record that result (i.e. ``$f(\smiley)$") in the $f$ column of table $c$. Repeat as follows: for each table $d$, identify all rows $r$ that have blank cell in column $g\taking d\to e$. Add a new row called $``g(r)"$ to table $e$ and record that result in the $(r,g)$ cell of table $d$.

Here is a concrete example. Let $\mcC$ be the following schema: 
$$\xymatrix{\LMO{A}\ar[r]^f&\LMO{B}\ar@<.5ex>[r]^{g_1}\ar@<-.5ex>[r]_{g_2}\ar[d]_h&\LMO{C}\ar[r]^i&\LMO{D}\\&\LMO{E}}$$
Then $Y_B\taking\mcC\to\Set$ is the following instance
\begin{center}
\begin{tabular}{| l || l |}\bhline
\multicolumn{2}{|c|}{A}\\\bhline
{\bf ID}&{\bf $f$}\\\bbhline
\end{tabular}
\hsp
\begin{tabular}{| l || l | l | l |}\bhline
\multicolumn{4}{|c|}{B}\\\bhline
{\bf ID}&{\bf $g_1$}&{\bf $g_2$}&{\bf $h$}\\\bbhline
$\smiley$&$g_1(\smiley)$&$g_2(\smiley)$&$h(\smiley)$\\\bhline
\end{tabular}
\hsp
\begin{tabular}{| l || l |}\bhline
\multicolumn{2}{|c|}{C}\\\bhline
{\bf ID}&{\bf $i$}\\\bbhline
$g_1(\smiley)$&$i(g_1(\smiley))$\\\hline
$g_2(\smiley)$&$i(g_2(\smiley))$\\\bhline
\end{tabular}
\end{center}
\begin{center}
\begin{tabular}{| l ||}\bhline
\multicolumn{1}{| c |}{D}\\\bhline
{\bf ID}\\\bbhline
$i(g_1(\smiley))$\\\hline
$i(g_2(\smiley))$\\\bhline
\end{tabular}
\hsp
\begin{tabular}{| l ||}\bhline
\multicolumn{1}{| c |}{E}\\\bhline
{\bf ID}\\\bbhline
$h(\smiley)$\\\bhline
\end{tabular}
\end{center}

We began with a single element in table $B$ and followed the arrows, putting new entries wherever they were required. One might call this the {\em schematically implied reference spread} or {\em SIRS}\index{schematically implied reference spread} of the element $\smiley$ in table $B$. Notice that the table at $A$ is empty, because there are no morphisms $B\to A$.

\end{example}

Representable functors $Y_c$ yield databases states that are as free as possible, subject to having the initial row $\smiley$ in table $c$. We have seen things like this before (by the name of Skolem variables)\index{Skolem} when studying the left pushforward $\Sigma$. Indeed, if $c\in\Ob(\mcC)$ is an object, we can consider it as a functor $c\taking\ul{1}\to\mcC$. A database instance on $\ul{1}$ is the same thing as a set $X$. The left pushforward $\Sigma_c(X)$ has the same kinds of Skolem variables. If $X=\{\smiley\}$ is a one element set, then we get the representable functor $\Sigma_c(\singleton)\iso Y_c$.

\begin{exercise}\label{exc:representables on graph}
Consider the schema for graphs, 
$$\GrIn:=\fbox{\parbox{1in}{\GrInSchema}}$$
\sexc Write down the representable functor $Y_{Ar}\taking\GrIn\to\Set$ as two tables.
\next Write down the representable functor $Y_{V\!e}$ as two tables.
\endsexc
\end{exercise}

\begin{exercise}
Consider the loop schema $$\Loop:=\LoopSchema.$$ What is the representable functor $Y_s\taking\Loop\to\Set$?
\end{exercise}

Let $B$ be a box in an olog, say \fakebox{a person}, and recall that an aspect of $B$ is an outgoing arrow, such as $\fakebox{a person}\To{\tn{has as height in inches}}\fakebox{an integer}$. The following slogan explains representable functors in those terms.

\begin{slogan}
The functor represented by \fakebox{a person} simply leaves a placeholder, like $\la$person's name here$\ra$ or $\la$person's height here$\ra$, for every aspect of \fakebox{a person}. 

In general, there is a representable functor for every type in an olog. The representable functor for type $T$ simply encapsulates the most generic or abstract example of type $T$, by leaving a placeholder for each of its attributes.
\end{slogan}


\subsubsection{Yoneda's lemma}\label{sec:yoneda}

One of the most powerful tools in category theory is Yoneda's lemma. It is often considered by new students to be quite abstract, but grounding it in databases may help.

The idea is this. Suppose that $I\taking\mcC\to\Set$ is a database instance, and let $c\in\Ob(\mcC)$ be an object. Because $I$ is a functor, we know that for every row $r\in I(c)$ in table $c$ a value has been recorded in the $f$-column, where $f\taking c\to c'$ is any outgoing arrow. The value in the $(r,f)$-cell refers to some row in table $c'$. What we're saying is that each row in table $c$ induces SIRS throughout the database. They may not be ``Skolem", or in any sense ``freely generated", but they are there nonetheless. The point is that to each row in $c$ there is a unique mapping $Y_c\to I$. 

\begin{lemma}[Yoneda's lemma, part 1.]\index{Yoneda's lemma}\label{lemma:Yoneda}

Let $\mcC$ be a category, $c\in\Ob(\mcC)$ an object, and $I\taking\mcC\to\Set$ a set-valued functor. There is a natural bijection $$\Hom_{\mcC\set}(Y_c,I)\Too{\iso}I(c).$$

\end{lemma}

\begin{proof}

See \cite{Mac}.

\end{proof}

\begin{example}\label{ex:yoneda}

Consider the category $\mcC$ drawn below:
$$
\mcC:=\parbox{2.1in}{\fbox{\parbox{2.1in}{\begin{center}\small mother\;$\circ$\;firstChild\;=\;$\id_{\tn{Mother}}$\normalsize\end{center}$$\xymatrix{\LTO{Child}\ar[rr]^{\tn{mother}}&&\LTO{Mother}\ar@/^1pc/[ll]^{\tn{firstChild}}}$$}}}
$$
There are two representable functors, $Y_{\tt Child}$ and $Y_{\tt Mother}$. The latter, when written as a database instance, will consist of a single row in each table. The former, $Y_{\tt Child}\taking\mcC\to\Set$ is shown here:
\begin{center}
\begin{tabular}{| l || l |}\bhline
\multicolumn{2}{|c|}{Child}\\\bhline
{\bf ID}&{\bf mother}\\\hline
$\smiley$&mother($\smiley$)\\\hline
firstChild(mother($\smiley$))&mother($\smiley$)\\\bbhline
\end{tabular}
\hsp
\begin{tabular}{| l || l |}\bhline
\multicolumn{2}{|c|}{Mother}\\\bhline
{\bf ID}&{\bf firstChild}\\\bbhline
mother($\smiley$)&firstChild(mother($\smiley$))\\\bhline
\end{tabular}
\end{center}
The representable functor $Y_{\tt Child}$ is the freest instance possible, starting with one element in the Child table and satisfying the constraints. 

Here is another instance $I\taking\mcC\to\Set$:
\begin{center}
\begin{tabular}{| l || l |}\bhline
\multicolumn{2}{|c|}{\tt Child}\\\bhline
{\bf ID}&{\bf mother}\\\hline
Amy&Ms. Adams\\\hline
Bob&Ms. Adams\\\hline
Carl&Ms. Jones\\\hline
Deb&Ms. Smith\\\bhline
\end{tabular}
\hsp
\begin{tabular}{| l || l |}\bhline
\multicolumn{2}{|c|}{\tt Mother}\\\bhline
{\bf ID}&{\bf firstChild}\\\bbhline
Ms. Adams&Bob\\\hline
Ms. Jones&Carl\\\hline
Ms. Smith&Deb\\\bhline
\end{tabular}
\end{center}

\end{example}

Yoneda's lemma (\ref{lemma:Yoneda}) is about the set of natural transformations $Y_{\tt Child}\to I$. Recall from Definition \ref{def:natural transformation} that a search for natural transformations can get a bit tedious. Yoneda's lemma makes the calculation quite trivial. In our case there are exactly four such natural transformations, and they are completely determined by where $\smiley$ goes. In some sense the symbol $\smiley$ {\em represents} child-ness in our database. 

\begin{exercise}
Consider the schema $\mcC$ and instance $I\taking\mcC\to\Set$ from Example \ref{ex:yoneda}. Let $Y_{\tt Child}$ be the representable functor as above. 
\sexc Let $\alpha\taking Y_{\tt Child}\to I$ be the natural transformation sending $\smiley$ to Amy. What is $\alpha_{\tn{Child}}(\tn{firstChild(mother}(\smiley)))$?
\footnote{There is a lot of clutter, perhaps. Note that ``firstChild(mother($\smiley$))" is a row in the {\tt Child} table. Assuming that the math follows the meaning, if $\smiley$ points to Amy, where should firstChild(Mother($\smiley$)) point?}
\next Let $\alpha\taking Y_{\tt Child}\to I$ be the natural transformation sending $\smiley$ to Bob. What is $\alpha_{\tt Child}(\tn{firstChild(mother}(\smiley)))$?
\next Let $\alpha\taking Y_{\tt Child}\to I$ be the natural transformation sending $\smiley$ to Carl. What is $\alpha_{\tt Child}(\tn{firstChild(mother}(\smiley)))$?
\next Let $\alpha\taking Y_{\tt Child}\to I$ be the natural transformation sending $\smiley$ to Deb. What is $\alpha_{\tt Child}(\tn{firstChild(mother}(\smiley)))$?
\next Let $\alpha\taking Y_{\tt Child}\to I$ be the natural transformation sending $\smiley$ to Amy. What is $\alpha_{\tt Mother}(\tn{mother}(\smiley))$?
\endsexc
\end{exercise}

We saw in Section \ref{sec:representable functors} that a representable functor is a mathematically-generated database instance for an abstract thing of type $T$. It creates placeholders for every attribute that things of type $T$ are supposed to have.

\begin{slogan}
Yoneda's lemma says the following. Specifying an actual thing of type $T$ is the same as filling in all placeholders found in the generic thing of type $T$.
\end{slogan}

Yoneda's lemma is considered by many category theory lovers to be the most important tool in the subject. While its power is probably unclear to students whose sole background in category theory comes from this book, Yoneda's lemma is indeed extremely useful for reasoning. It allows us to move the notion of functor application into the realm of morphisms between functors (i.e. morphisms in $\mcC\set$, which are natural transformations). This keeps everything in one place --- it's all in the morphisms --- and thus more interoperable.

\begin{example}\label{ex:yoneda for cyclic monoid}

In Example \ref{ex:cyclic monoid (7,4)}, we discussed the cyclic monoid $\mcM$ generated by the symbol $Q$ and subject to the relation $Q^7=Q^4$. We drew a picture like this: 
\begin{align}\label{dia:grothendieck yoneda (4,7)}
\xymatrix@=15pt{
\LMO{Q^0}\ar[rr]&&\LMO{Q^1}\ar[rr]&&\LMO{Q^2}\ar[rr]&&\LMO{Q^3}\ar[rr]&&\LMO{Q^4}\ar[dr]\\
&&&&&&&\LMO{Q^6}\ar[ur]&&\LMO{Q^5}\ar[ll]
}
\end{align}
We are finally ready to give the mathematical foundation for this picture. Since $\mcM$ is a category with one object, $\monOb$, there is a unique representable functor (up to isomorphism) $Y:=Y_\monOb\taking\mcM\to\Set$. A functor $\mcM\to\Set$ can be thought of as a set with an $\mcM$-action, as discussed in Section \ref{sec:monoids as cats}. Here the required set is 
$$Y(\monOb)=\Hom_\mcM(\monOb,\monOb)\iso\{Q^0,Q^1,Q^2,Q^3,Q^4,Q^5,Q^6\}$$ 
and the action is pretty straightforward (it is called the {\em principal action}). We might say that (\ref{dia:grothendieck yoneda (4,7)}) is a picture of this principal action of $\mcM$. 

However, we can go one step further. Given a functor $Y\taking\mcM\to\Set$, we can take its category of elements, $\int_\mcM Y$ as in Section \ref{sec:grothendieck construction}. The category $\int_\mcM Y$ has objects $Y(\monOb)\in\Ob(\Set)$, i.e. the set of dots in (\ref{dia:grothendieck yoneda (4,7)}), and it has a unique morphism $Q^i\to Q^j$ for every path of length $\leq 6$ from $Q^i$ to $Q^j$ in that picture.

\end{example}

\begin{exercise}

Let $c\in\Ob(\mcC)$ be an object and let $I\in\Ob(\mcC\set)$ be another object. Consider $c$ also as a functor $c\taking\ul{1}\to\mcC$ and recall the pullback functor $\Delta_c\taking\mcC\set\to\Set$ and its left adjoint $\Sigma_c\taking\Set\to\mcC\set$ from Section \ref{sec:data migration}.
\sexc What is the set $\Delta_c(I)$?
\next What is $\Hom_\Set(\singleton,\Delta_c(I))$?
\next What is $\Hom_{\mcC\set}(\Sigma_c(\singleton),I)$?
\next How does $\Sigma_c(\singleton)$ compare to $Y_c$, the functor represented by $c$, as objects in $\mcC\set$?
\endsexc
\end{exercise}

\begin{lemma}[Yoneda's lemma, part 2]

Let $\mcC$ be a category. The assignment $c\mapsto Y_c$ from Lemma \ref{lemma:Yoneda} extends to a functor $Y\taking\mcC\op\to\mcC\set$, and this functor is fully faithful. 

In particular, if $c,c'\in\Ob(\mcC)$ are objects and there is an isomorphism $Y_c\iso Y_{c'}$ in $\mcC\set$, then there is an isomorphism $c\iso c'$ in $\mcC$.

\end{lemma}
\begin{proof}
See \cite{Mac}.
\end{proof}

\begin{exercise}
The distributive law for addition of natural numbers says $(a+b)\times c=a\times c+b\times c$. Below we will give a proof of the distributive law, using category-theoretic reasoning. Annotate anything in {\color{red}red} ink  with a justification for why it is true.
\begin{proposition}
For any natural numbers $a,b,c\in\NN$, the distributive law 
$$(a+b)c=ac+bc$$ 
holds.
\end{proposition}
\begin{proof}[Sketch of proof. To finish, justify {\color{red}red stuff}]
~\\
Let $A,B,C$ be finite sets and let $X$ be another finite set.
\begin{align*}
\Hom_\Set((A+B)\times C,X)
&{\color{red}\iso}\Hom_\Set(A+B,X^C)\\
&{\color{red}\iso}\Hom_\Set(A,X^C)\times\Hom_\Set(B,X^C)\\
&{\color{red}\iso}\Hom_\Set(A\times C,X)\times\Hom_\Set(B\times C,X)\\
&{\color{red}\iso}\Hom_\Set((A\times C)+(B\times C),X).
\end{align*}
By {\color{red} the appropriate application} of Yoneda's lemma, we see that there is an isomorphism
$$(A+B)\times C\iso(A\times C)+(B\times C)$$
in $\Fin$. The result about natural numbers {\color{red}follows}.
\end{proof}
\end{exercise}


\subsubsection{The subobject classifier $\Omega\in\Ob(\mcC\set)$}\index{subobject classifier!in $\mcC\set$}

If $\mcC$ is a category then the functor category $\mcC\set$ is a very nice kind of category, called a {\em topos}.\index{topos} Note that when $\mcC=\ul{1}$ is the terminal category, then we have an isomorphism $\mcC\set\iso\Set$, so the category of sets is a special case of a topos. What is so interesting about toposes (or topoi) is that they so nicely generalize many properties of $\Set$. In this short section we investigate only one such property, namely that $\mcC\set$ has a subobject classifier, denoted $\Omega\in\Ob(\mcC\set)$. In the case $\mcC=\ul{1}$, we saw back in Section \ref{def:subobject classifier} that the subobject classifier is $\{True,False\}\in\Ob(\Set)$. 

As usual, we consider the matter of subobject classifiers by grounding the discussion in terms of databases.

\begin{definition}

Let $\mcC$ be a category, let $\mcC\set$ denote its category of instances, and let $1\in\Ob(\mcC\set)$ denote the terminal object. A {\em subobject classifier for $\mcC\set$} is an object $\Omega_\mcC\in\Ob(\mcC\set)$ and a morphism $t\taking 1\to\Omega_\mcC$ with the following property. For any monomorphism $f\taking X\to Y$ in $\mcC\set$, there exists a unique morphism $char(f)\taking Y\to\Omega_\mcC$ such that the following diagram is a pullback in $\mcC\set$:
$$
\xymatrix@=25pt{X\ar[r]^{!}\ar[d]_f\ullimit&1\ar[d]^t\\Y\ar[r]_{char(f)}&\Omega_\mcC
}
$$

\end{definition}

In terms of databases, what this means is that for every schema $\mcC$ there is some special instance $\Omega_\mcC\in\Ob(\mcC\set)$ that somehow classifies sub-instances. When our schema is the terminal category, $\mcC=\ul{1}$, instances are sets and we saw in Definition \ref{def:subobject classifier} that the subobject classifier is $\Omega_{\ul{1}}=\{True, False\}$. One might think that the subobject classifier for $\mcC\set$ should just consist of a two-element set table-by-table, i.e. that for every $c\in\Ob(\mcC)$ we should have $\Omega_{\mcC}=^?\{True,False\}$, but this is not correct. 

In fact, for any object $c\in\Ob(\mcC)$, it is easy to say what $\Omega_\mcC(c)$ should be. We know by Yoneda's lemma (Lemma \ref{lemma:Yoneda}) that $\Omega_\mcC(c)=\Hom_{\mcC\set}(Y_c,\Omega_\mcC)$, where $Y_c$ is the functor represented by $c$. There is a bijection between $\Hom_{\mcC\set}(Y_c,\Omega_\mcC)$ and the set of sub-instances of $Y_c$. Each morphism $f\taking c\to d$ in $\mcC$ induces a morphism $Y_f\taking Y_d\to Y_c$, and the map $\Omega_\mcC(f)\taking\Omega_\mcC(c)\to\Omega_\mcC(d)$ sends a sub-instance $A\ss Y_c$ to the pullback 
$$
\xymatrix{Y_f^\m1(A)\ar[r]\ar[d]\ullimit&A\ar[d]\\Y_d\ar[r]_{Y_f}&Y_c}
$$

But this is all very abstract. We now give an example of a subobject classifier.  

\begin{example}

Consider the category $\mcC\iso[3]$ depicted below
$$\mcC:=\parbox{3in}{\fbox{
\xymatrix{&&&\ar@{}[d]|(.4){\checkmark}&&&\\
\LTO{0}\ar[rr]^{\tn{after\_1}}\ar@/_1.5pc/[rrrr]_{\tn{after\_2}}\ar@/^3pc/[rrrrrr]^{\tn{after\_3}}&&\LTO{1}\ar[rr]^{\tn{after\_1}}\ar@/_1.5pc/[rrrr]_{\tn{after\_2}}&&\LTO{2}\ar[rr]^{\tn{after\_1}}&&\LTO{3}\\
&&&\ar@{}[ull]|(.8){\checkmark}\ar@{}[urr]|(.8){\checkmark}&&
}}}
$$
To write down $\Omega_\mcC$ we need to understand the representable functors $Y_c\in\Ob(\mcC\set)$, for $c={\tt 0},{\tt 1},{\tt 2},{\tt 3}$, as well as their subobjects. Here is $Y_{\tt 0}$ as an instance:

\begin{center}\small
\begin{tabular}{| l || l | l | l |}
\bhline
\multicolumn{4}{|c|}{$Y_{\tt 0}({\tt 0})$}\\\bhline
{\bf ID}&{\bf after\_1}&{\bf after\_2}&{\bf after\_3}\\\bbhline
$\smiley$&after\_1($\smiley$)&after\_2($\smiley$)&after\_3($\smiley$)\\\bhline
\end{tabular}
\hsp
\begin{tabular}{| l || l | l |}
\bhline
\multicolumn{3}{|c|}{$Y_{\tt 0}({\tt 1})$}\\\bhline
{\bf ID}&{\bf after\_1}&{\bf after\_2}\\\bbhline
after\_1$(\smiley)$&after\_2$(\smiley)$&after\_3$(\smiley)$\\\bhline
\end{tabular}\\\vspace{.2in}
\begin{tabular}{| l || l |}
\bhline
\multicolumn{2}{|c|}{$Y_{\tt 0}({\tt 2})$}\\\bhline
{\bf ID}&{\bf after\_1}\\\bbhline
after\_2$(\smiley)$&after\_3$(\smiley)$\\\bhline
\end{tabular}
\hsp
\begin{tabular}{| l ||}
\bhline
\multicolumn{1}{|c|}{$Y_{\tt 0}({\tt 3})$}\\\bhline
{\bf ID}\\\bbhline
after\_3$(\smiley)$\\\bhline
\end{tabular}

\end{center}

What are the sub-instances of this? There is the empty sub-instance $\emptyset\ss Y_{\tt 0}$ and the identity sub-instance $Y_{\tt 0}\ss Y_{\tt 0}$. But there are three more as well. Note that if we want to keep the $\smiley$ row of table {\tt 0} then we have to keep everything. But if we throw away the $\smiley$ row of table {\tt 0} we can still keep the rest and get a sub-instance. If we want to keep the after\_1$(\smiley)$ row of table {\tt 1} then we have to keep its images in tables {\tt 2} and {\tt 3}. But we could throw away both the $\smiley$ row of table {\tt 0} and the after\_1$(\smiley)$ row of table {\tt 1} and still keep the rest. And so on. In other words, the subobjects of $Y_{\tt 0}$ are in bijection with the set $\Omega_\mcC({\tt 0}):=\{\tn{{\it yes}, {\it in 1}, {\it in 2}, {\it in 3}, {\it never}}\}$. 

The same analysis holds for the other tables of $\Omega_\mcC$. It looks like this:
\begin{center}
\begin{tabular}{| l || l | l | l |}
\bhline
\multicolumn{4}{|c|}{$\Omega_\mcC({\tt 0})$}\\\bhline
{\bf ID}&{\bf after\_1}&{\bf after\_2}&{\bf after\_3}\\\bbhline
{\it yes}&{\it yes}&{\it yes}&{\it yes}\\\hline
{\it in 1}&{\it yes}&{\it yes}&{\it yes}\\\hline
{\it in 2}&{\it in 1}&{\it yes}&{\it yes}\\\hline
{\it in 3}&{\it in 2}&{\it in 1}&{\it yes}\\\hline
{\it never}&{\it never}&{\it never}&{\it never}\\\bhline
\end{tabular}
\hsp
\begin{tabular}{| l || l | l |}
\bhline
\multicolumn{3}{|c|}{$\Omega_\mcC({\tt 1})$}\\\bhline
{\bf ID}&{\bf after\_1}&{\bf after\_2}\\\bbhline
{\it yes}&{\it yes}&{\it yes}\\\hline
{\it in 1}&{\it yes}&{\it yes}\\\hline
{\it in 2}&{\it in 1}&{\it yes}\\\hline
{\it never}&{\it never}&{\it never}\\\bhline
\end{tabular}\\\vspace{.2in}
\begin{tabular}{| l || l |}
\bhline
\multicolumn{2}{|c|}{$\Omega_\mcC({\tt 2})$}\\\bhline
{\bf ID}&{\bf after\_1}\\\bbhline
{\it yes}&{\it yes}\\\hline
{\it in 1}&{\it yes}\\\hline
{\it never}&{\it never}\\\bhline
\end{tabular}
\hsp
\begin{tabular}{| l ||}
\bhline
\multicolumn{1}{|c|}{$\Omega_\mcC({\tt 3})$}\\\bhline
{\bf ID}\\\bbhline
{\it yes}\\\hline
{\it never}\\\bhline
\end{tabular}

\end{center}
The morphism $1\to\Omega_\mcC$ picks out the {\it yes} row of every table.

Now that we have constructed $\Omega_\mcC\in\Ob(\mcC\set)$, we are ready to see it in action. What makes $\Omega_\mcC$ special is that for any instance $X\taking\mcC\to\Set$, the subinstances if $X$ are in one-to-one correspondence with the morphisms $X\to\Omega_\mcC$. Consider the following arbitrary instance $X$, where the blue rows denote a sub-instance $A\ss X$.

\begin{align}\label{dia:instance for omega}\footnotesize
\begin{array}{| l || l | l | l |}
\bhline
\multicolumn{4}{|c|}{X({\tt 0})}\\\bhline
{\bf ID}&{\bf after\ 1}&{\bf after\ 2}&{\bf after\ 3}\\\bbhline
a_1&b_1&\color{blue}{c_1}&\color{blue}{d_1}\\\hline
a_2&\color{blue}{b_2}&\color{blue}{c_1}&\color{blue}{d_1}\\\hline
a_3&\color{blue}{b_2}&\color{blue}{c_1}&\color{blue}{d_1}\\\hline
a_4&b_3&c_2&d_2\\\hline
a_5&b_5&c_3&\color{blue}{d_1}\\\bhline
\end{array}
\hspace{.2in}
\begin{array}{| l || l | l |}
\bhline
\multicolumn{3}{|c|}{X({\tt 1})}\\\bhline
{\bf ID}&{\bf after\ 1}&{\bf after\ 2}\\\bbhline
b_1&\color{blue}{c_1}&\color{blue}{d_1}\\\hline
\color{blue}{b_2}&\color{blue}{c_1}&\color{blue}{d_1}\\\hline
b_3&c_2&d_2\\\hline
\color{blue}{b_4}&\color{blue}{c_1}&\color{blue}{d_1}\\\hline
b_5&c_3&\color{blue}{d_1}\\\bhline
\end{array}\hspace{.2in}
\begin{array}{| l || l |}
\bhline
\multicolumn{2}{|c|}{X({\tt 2})}\\\bhline
{\bf ID}&{\bf after\ 1}\\\bbhline
\color{blue}{c_1}&\color{blue}{d_1}\\\hline
c_2&d_2\\\hline
c_3&\color{blue}{d_1}\\\bhline
\end{array}
\hspace{.2in}
\begin{array}{| l ||}
\bhline
\multicolumn{1}{|c|}{X({\tt 3})}\\\bhline
{\bf ID}\\\bbhline
\color{blue}{d_1}\\\hline
d_2\\\bhline
\end{array}
\end{align}

This blue sub-instance $A\ss X$ corresponds to a map $char(A)\taking X\to\Omega_\mcC$. That is for each $c\in\Ob(\mcC)$ the rows in the $c$-table of $X$ are sent to the rows in the $c$-table of $\Omega_\mcC$. The way $char(A)$ works is as follows. For each table $i$ and row $x\in X(i)$, find the first column $f$ in which the entry is blue (i.e. $f(x)\in A$), and send $x$ to the corresponding element of $\Omega_\mcC(i)$. For example, $char(A)({\tt 0})$ sends $a_1$ to {\it in 2} and sends $a_4$ to {\it never}, and $char(A)({\tt 2})$ sends $c_1$ to {\it yes} and sends $c_2$ to {\it never}.

\end{example}

\begin{exercise}
\sexc Write out the blue subinstance $A\ss X$ shown in (\ref{dia:instance for omega}) as an instance of $\mcC$, i.e. as four tables. 
\next This subinstance $A\ss X$ corresponds to a map $\ell:=char(A)\taking X\to\Omega_\mcC$. For all $c\in\Ob(\mcC)$ we have a function $\ell(c)\taking X(c)\to\Omega_\mcC(c)$. With $c={\tt 1}$, write out $\ell({\tt 1})\taking X({\tt 1})\to\Omega_\mcC({\tt 1})$.
\endsexc
\end{exercise}

\begin{exercise}
Let $\Loop$ be the loop schema 
$$\Loop=\LoopSchema.$$ 
\sexc What is the subobject classifier $\Omega_\Loop\in\Ob(\Loop\set)$?
\next How does $\Omega_\Loop$ compare to the representable functor $Y_s$?
\endsexc
\end{exercise}

\begin{exercise}   
Let $\GrIn=\fbox{\GrInSchema}$ be the indexing category for graphs. 
\sexc Write down the subobject classifier $\Omega_\GrIn\in\Ob(\GrIn\set)$ in tabular form, i.e. as two tables.
\next Draw $\Omega_\GrIn$ as a graph.
\next Let $G$ be the graph below and $G'\ss G$ the blue part.
$$\xymatrix{
\LMO{\color{blue}{w}}\ar@/^1pc/[r]^f\ar@[blue][r]_{\color{blue}{g}}\ar[d]_h&\LMO{\color{blue}{x}}\\
\LMO{y}\ar@(l,d)[]_j\ar[r]_i&\LMO{\color{blue}{z}}
}
$$
Write down $G\in\Ob(\GrIn\set)$ in tabular form.
\next Write down the components of the natural transformation $char(G')\taking G\to\Omega_\GrIn$.
\endsexc
\end{exercise}


\subsection{Database instances in other categories}


\subsubsection{Representations of groups}\label{ex:reps of groups}

The classical mathematical subject of {\em representation theory}\index{representation theory} is the study of $\Fun(G,\Vect)$ where $G$ is a group and $\Vect$ is the category of vector spaces (over say $\RR$).\index{vector space}\index{a category!$\Vect$} Every such functor $F\taking G\to\Vect$ is called a {\em representation of $G$}. Since $G$ is a category with one object $\monOb$, $F$ consists of a single vector space $V=F(\monOb)$ together with an action of $G$ on it. 

We can think of this in terms of databases if we have a presentation of $G$ in terms of generators and relations. The schema corresponding to $G$ has one table and this table has a column for each generator. Giving a representation $F$ is the same as giving an instance on our schema, with some properties that stem from the fact that our target category is $\Vect$ rather than $\Set$. There are many possibilities for expressing
\footnote{We would use the term ``representing" or "presenting", but they are both taken in the context of our narrative!}
such data.

One possibility is if we could somehow draw $V$, say if $V$ is 1-, 2-, or 3-dimensional. If so, let $P$ be our chosen picture of $V$, e.g. $P$ is the standard drawing of a Cartesian coordinate plane. Then every column of our table would consist entirely of the picture $P$ instead of a set of rows. Drawing a point in the ID-column picture would result in a point being drawn in each other column's picture, in accordance with the $G$-action. Each column would of course respect addition and scalar multiplication.

Another possibility is to use the fact that there is a functor $U\taking\Vect\to\Set$, so our instance $F\taking G\to\Vect$ can be converted to an ordinary instance $U\circ F\taking G\to\Set$. We would have an ordinary set of rows. This set would generally be infinite, but it would be structured by addition and scalar multiplication. For example, assuming $V$ is finite dimensional, one could find a few rows that generated the rest. 

A third possibility is to use monads, which allow the table to have only as many rows as $V$ has dimensions. This is a considerable savings of space. See Section \ref{sec:monads}.


\subsubsection{Representations of quivers}

Representation theory also studies representations of quivers. A {\em quiver} is just the free category (see Example \ref{ex:free category}\index{category!free category}\index{graph!free category on}) on a graph. If $P$ is a graph with free category $\mcP$ then a representation of the quiver $\mcP$ is a functor $F\taking\mcP\to\Vect$. Such a representation consists of a vector space at every vertex of $P$ and a linear transformation for every arrow. All of the discussion from Section \ref{ex:reps of groups} works in this setting, except that there is more than one table.


\subsubsection{Other target categories}\label{sec:other targets}

One can imagine the value of using target categories other than $\Set$ or $\Vect$ for databases. 

\begin{application}\index{geography}

\href{http://en.wikipedia.org/wiki/Geographic_data}{\text Geographic data} consists of maps of the earth together with various functions on it. For example for any point on the earth one may want to know the average temperature recorded in the past 10 years, or the precise temperature at this moment. Earth can be considered as a topological space, $E$. Similarly, temperatures on earth reside on a continuum, say the space $T$ of real numbers $[-100,200]$. Thus the temperature record is a function $E\to T$. 

Other records such as precipitation, population density, elevation, etc. can all be considered as continuous functions from $E$ to some space. Agencies like the US Geological Survey hold databases of such information. By modeling them on functors $\mcC\to\Top$, they may be able to employ mathematical tools such as persistent homology \cite{WeS} to find interesting invariants of the data.

\end{application}

\begin{application}

Many other scientific disciplines could use the same kind of tool. For example, in studying the \href{http://en.wikipedia.org/wiki/Strength_of_materials}{\text mechanics of materials}, one may want to consider the material as a topological space $M$ and measure values such as energy as a continuous $M\to E$. Such observations could be modeled by databases with target category $\Top$ or $\Vect$ rather than $\Set$.

\end{application}


\subsection{Sheaves}\label{sec:sheaves}\index{sheaves}

Let $X$ be a topological space (see Example \ref{ex:topological space}), such as a sphere. In Section \ref{sec:other targets} we discussed continuous functions out of $X$, and their use in science (e.g. recording temperatures on the earth as a continuous map $X\to[-100,200]$). Sheaves allow us to consider the local-global nature of such maps, taking into account reparable discrepancies in data gathering tools. 

\begin{application}\label{app:sheaves of temperature}

Suppose that $X$ is the topological space corresponding to the earth; by a {\em region} we mean an open subset $U\ss X$. Suppose that we cover $X$ with 10,000 regions $U_1,U_2,\ldots,U_{10000}$, such that some of the regions overlap in a non-empty subregion (e.g. perhaps $U_5\cap U_9\neq\emptyset)$. For each $i,j$ let $U_{i,j}=U_i\cap U_j$. 

For each region $U_i\ss X$ we have a temperature recording device, which gives a function $T_i\taking U_i\to[-100,200]$. If $U_i\cap U_j\neq\emptyset$ then two different recording devices give us temperature data for the intersection $U_{i,j}$. Suppose we find that they do not give precisely the same data, but that there is a translation formula between their results. For example, $T_i$ might register $3^\circ$ warmer than $T_j$ registers, throughout the region $U_i\cap U_j$.

A consistent system of translation formulas is called a {\em sheaf}. It does not demand a universal ``true" temperature function, but only a consistent translation system between them. 

\end{application}

The following definitions (Definitions \ref{def:presheaf}, \ref{def:sheaf}) make the notion of sheaf precise, but we must go slowly (because it will already feel quick to the novice). For every region $U$, we can record the value of some function (say temperature) throughout $U$; although this record might consist of a mountain of data (a temperature for each point in $U$!), we think of it as one thing. That is, it is one element in the set of value-assignments throughout $U$. A sheaf holds the set of possible values-assignments-throughout-$U$'s for all the different regions $U$, as well as how a value-assignment-throughout-$U$ restricts to a value-assignment-throughout-$V$ for any subset $V\ss U$.

\begin{definition}\label{def:presheaf}\index{presheaf}

Let $X$ be a topological space, let $\Op(X)$ denote its partial order of open sets, and let $\Op(X)\op$ be the opposite category. A {\em presheaf on $X$} is a functor $\mcO\taking\Op(X)\op\to\Set$. For every open set $U\ss X$ we refer to the set $\mcO(U)$ as the {\em set of values-assignments throughout $U$ of $\mcO$}. If $V\ss U$ is an open subset, it corresponds to an arrow in $\Op(X)$ and applying the functor $\mcO$ yields a function called the {\em restriction map from $U$ to $V$} and denoted $\rho_{V,U}\taking\mcO(U)\to\mcO(V)$. Given $a\in\mcO(U)$, we may denote $\rho_{V,U}(a)$ by $a|_V$; it is called {\em the restriction of $a$ to $V$}.

The {\em category of presheaves on $X$} is simply $\Op(X)\op\set$; see Definition \ref{def:mcC-set}.

\end{definition}

\begin{exercise}~
\sexc Come up with $4$ overlapping open subsets that cover the square $X:=[0,3]\times[0,3]\ss\RR^2$. Write down a label for each open set as well as a label for each overlap (2-fold, 3-fold, etc.); you now have labeled $n$ open sets. For each of these open sets, draw a dot with the appropriate label, and then draw an arrow from one dot to another when the first refers to an open subset of the second. This is a preorder; call it $\Op(X)$. Now make up and write down formulas $R_1\taking X\to\RR$ and $R_2\taking X\to\RR$ with $R_1\leq R_2$, expressing a range of temperatures $R_1(p)\leq x\leq R_2(p)$ that an imaginary experiment shows can exist at each point $p$ in the square. 
\next Suppose we now tried to make our presheaf $\mcO\taking\Op(X)\op\to\Set$ as follows. For each of your open sets, say $A$, we could put $$\mcO(A):=\{f\taking A\to\RR\|R_1(a)\leq f(a)\leq R_2(a)\}.$$ What are the restriction maps? Do you like the name ``value-assignment throughout $A$" for elements of $\mcO(A)$? 
\next We can now make another presheaf $\mcO'$ given the same experiment. For each of your open sets, say $A$, we could put $$\mcO'(A):=\{f\taking A\to\RR\|f\tn{ is continuous, and }R_1(a)\leq f(a)\leq R_2(a)\}.$$ Are you comfortable with the idea that there is a morphism of presheaves $\mcO'\to\mcO$?
\endsexc
\end{exercise}

Before we define sheaves, we need to clarify the notion of covering. Suppose that $U$ is a region and that $V_1,\ldots,V_n$ are subregions (i.e. for each $1\leq i\leq n$ we have $V_i\ss U$). Then we say that the $V_i$ {\em cover} $U$ if every point in $U$ is in $V_i$ for some $i$. Another way to say this is that the natural function $\sqcup_iV_i\to U$ is surjective.

\begin{example}\label{ex:open cover}\index{open cover}

Let $X=\RR$ be the space of real numbers, and define the following open subsets: $U=(5,10), V_1=(5,7), V_2=(6,9), V_3=(7,10)$.
\footnote{We use parentheses to denote open intervals of real numbers. For example $(6,9)$ denotes the set $\{x\in\RR\|6<x<9\}$.} 
Then $V_1, V_2, V_3$ is a cover of $U$. It has overlaps $V_{12}=V_1\cap V_2=(6,7)$, $V_{13}=V_1\cap V_3=\emptyset$, $V_{23}=V_2\cap V_3=(7,9)$. 

Given a presheaf $\mcO\taking\Op(X)\op\to\Set$, we have sets and functions as in the following (incomplete) diagram
$$
\xymatrix@=15pt{
&&\mcO(V_1)\ar[drr]\\
&&&&\mcO(V_1\cap V_2)\\
\mcO(U)\ar[uurr]\ar[rr]\ar[ddrr]&&\mcO(V_2)\ar[urr]\ar[drr]\\
&&&&\mcO(V_2\cap V_3)\\
&&\mcO(V_3)\ar[urr]
}
$$

\end{example}

A presheaf $\mcO$ on $X$ tells us what value-assignments throughout $U$ can exist for each $U$. Suppose we have a value-assignment $a\in\mcO(U)$ throughout $U$ and another value-assignment $a'\in\mcO(U')$ throughout $U'$, and suppose that they agree as value-assignments throughout $U\cap U'$, i.e. $a|_{U\cap U'}=a'|_{U\cap U'}$. In this case we should have a unique value-assignment $b\in\mcO(U\cup U')$ throughout $U\cup U'$ that agrees on the $U$-part with $a$ and agrees on the $U'$-part with $a'$; i.e. $b|_U=a$ and $b|_{U'}=a'$. This is the sheaf condition. 

\begin{definition}\label{def:sheaf}

Let $X$ be a topological space, let $\Op(X)$ be its partial order of open sets, and let $\mcO\taking\Op(X)\op\to\Set$ be a presheaf. Given an open set $U\ss X$ and a cover $V_1,\ldots, V_n$ of $U$, the following condition is called the {\em sheaf condition}\index{sheaf!condition} for that cover. 
\begin{description}
\item [Sheaf condition] Given a sequence $a_1,\ldots,a_n$ where each is a value-assignment $a_i\in\mcO(V_i)$ throughout $V_i$, suppose that for all $i,j$ we have $a_i|_{V_i\cap V_j}=a_j|_{V_i\cap V_j}$; then there is a unique value-assignment $b\in\mcO(U)$ such that $b|_{V_i}=a_i$.
\end{description}
The presheaf $\mcO$ is called a {\em sheaf} if it satisfies the sheaf condition for every cover.

\end{definition}

\begin{example}

Let $X=\RR$ and let $U, V_1,V_2,V_3$ be the open cover given in Example \ref{ex:open cover}. Given a measurement taken throughout $V_1$, a measurement taken throughout $V_2$, and a measurement taken throughout $V_3$, we have elements $a_1\in\mcO(V_1), a_2\in\mcO(V_2),$ and $a_3\in\mcO(V_3)$. If they are in agreement on the overlap intervals, we can {\em glue} \index{sheaf!glueing} them to give a measurement throughout $U$.

\end{example}

\begin{remark}

In Application \ref{app:sheaves of temperature}, we said that sheaves would help us patch together information from different sources. Even if different temperature-recording devices $T_i$ and $T_j$ registered different temperatures on an overlapping region $U_i\cap U_j$, we said they could be patched together if there was a consistent translation system between their results. What is actually needed is a set of isomorphisms 
$$p_{i,j}\taking T_i|_{U_{i,j}}\To{\iso} T_j|_{U_{i,j}}$$ 
that translate between them, and that these $p_{i,j}$'s act in concert with one another. This (when precisely defined,) is called \href{http://en.wikipedia.org/wiki/Descent_theory}{\em descent data}.\index{descent data}\index{sheaf!descent data}. The way it interacts with our definition of sheaf given in Definitions \ref{def:presheaf} and \ref{def:sheaf} is buried in the restriction maps $\rho$ for the overlaps as subsets $U_{i,j}\ss U_i$ and $U_{i,j}\ss U_j$. We will not explain further here. One can see \cite{Gro}.

\end{remark}

\begin{application}

Consider outer space as a topological space $X$. Different astronomers record observations. Let $C=[390,700]$ denote the set of wavelengths in the visible light spectrum (written in nanometers). Given an open subset $U\ss X$ let $\mcO(U)$ denote the set of functions $U\to C$. The presheaf $\mcO$ satisfies the sheaf condition; this is the taken-for-granted fact that we can patch together different \href{http://en.wikipedia.org/wiki/Astrophotography}{\text observations of space}.

Below are three views of the night sky. Given a telescope position to obtain the first view, one moves the telescope right and a little down to obtain the second and one moves it down and left to obtain the third.
\footnote{Image credit: NASA, ESA, Digitized Sky Survey Consortium.}
\begin{center}\parbox{5.5in}{\begin{center}
\includegraphics[height=6cm]{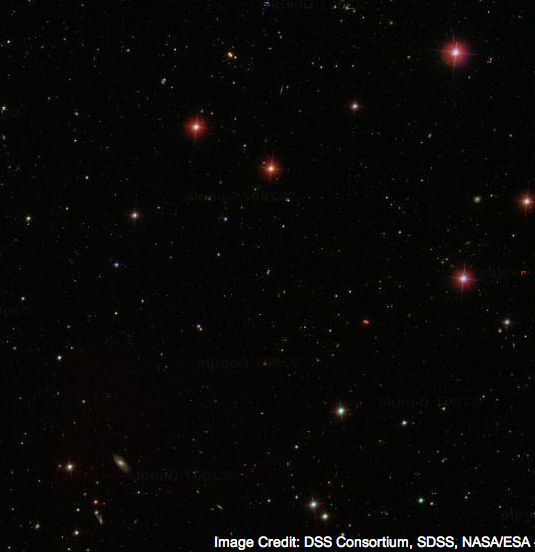}\hsp
\includegraphics[height=6cm]{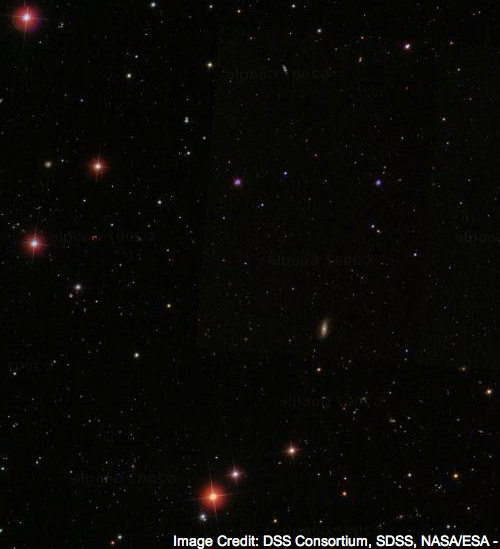}\end{center}
\hspace{1.8in}
\includegraphics[height=6cm]{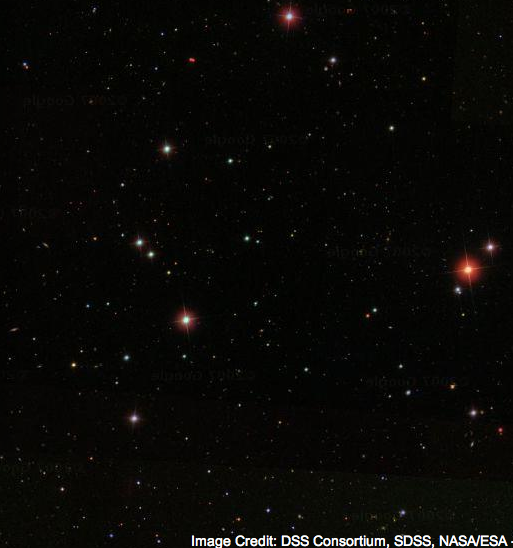}}\end{center}
These are value-assignments $a_1\in\mcO(V_1), a_2\in\mcO(V_2),$ and $a_3\in\mcO(V_3)$ throughout subsets $V_1,V_2,V_3\ss X$ (respectively). These subsets $V_1,V_2,V_3$ cover some (strangely-shaped) subset $U\ss X$. The sheaf condition says that these three value-assignments glue together to form a single value-assignment throughout $U$:
\begin{center}
\includegraphics[height=7.5cm]{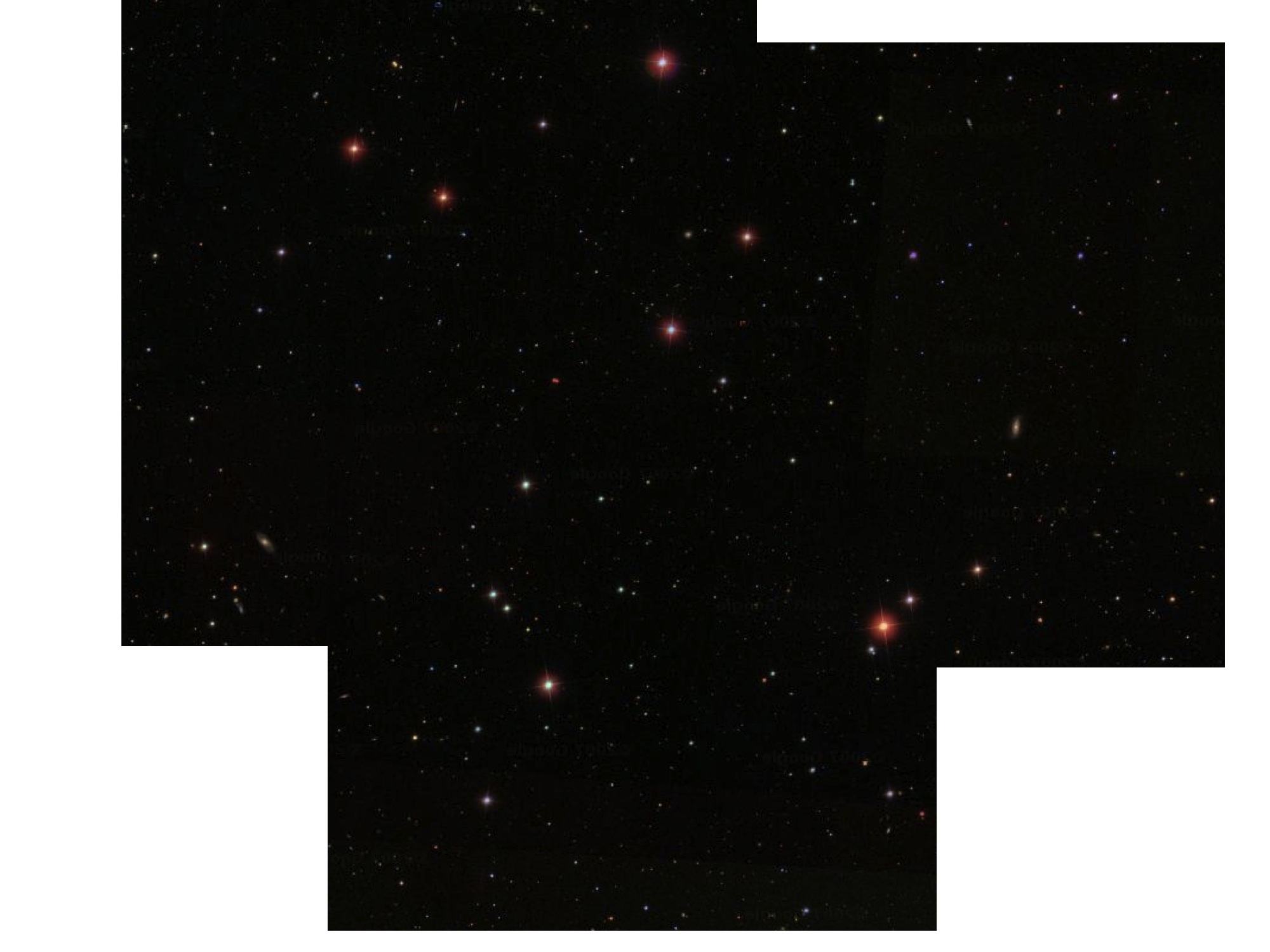}
\end{center}

\end{application}

\begin{exercise}
Find an application of sheaves in your own domain of expertise.
\end{exercise}

\begin{application}
Suppose we have a sheaf for temperatures on earth. For every region $U$ we have a set of theoretically possible temperature-assignments throughout $U$. For example we may know that if it is warm in Texas, warm in Arkansas, and warm in Kansas, then it cannot be cold in Oklahoma. With such a sheaf $\mcO$ in hand, one can use facts about the temperature in one region $U$ to predict the temperature in another region $V$. 

The mathematics is as follows. Suppose given regions $U,V\ss X$ and a subset $A\ss\mcO(U)$ corresponding to what we know about the temperature assignment throughout $U$. We take the following fiber product
$$
\xymatrix{(\rho_{U,X})^{\m1}(A)\ullimit\ar[r]\ar[d]&\mcO(X)\ar[d]^{\rho_{U,X}}\ar[r]^{\rho_{V,X}}&\mcO(V)\\
A\ar[r]&\mcO(U)}
$$
The image of the top map is a subset of $\mcO(V)$ telling us which temperature-assignments are possible throughout $V$ given our knowledge $A$ about the temperature throughout $U$.

We can imagine the same type of prediction systems for other domains as well, such as the energy of various parts of a material.
\end{application}

\begin{example}

In Exercises \ref{exc:juris 1} and \ref{exc:juris 2} we discussed the idea of laws being dictated or respected throughout a jurisdiction. If $X$ is earth, to every jurisdiction $U\ss X$ we assign the set $\mcO(U)$ of laws that are dictated to hold throughout $U$. Given a law on $U$ and a law on $V$, we can see if they amount to the same law on $U\cap V$. For example, on $U$ a law  might say ``no hunting near rivers" and on $V$ a law might say ``no hunting in public areas". It just so happens that on $U\cap V$ all public areas are near rivers and vice versa, so the laws agree there. These laws patch together to form a single rule about hunting that is enforced throughout the union $U\cup V$, respected by all jurisdictions within it.

\end{example}


\subsubsection{Sheaf of ologged concepts}\index{olog!sheaf of}

Definition \ref{def:sheaf} defines what should be called a sheaf of sets. We can discuss sheaves of groups or even sheaves of categories. Here is an application of the latter.

Recall the notion of simplicial complexes discussed in Section \ref{sec:simplicial complex}\index{simplicial complex}. They look like this: 
\begin{align}\label{dia:olog network}
\includegraphics[height=3in]{OlogNetwork5}
\end{align} 
Given such a simplicial complex $X$, we can imagine each vertex $v\in X_0$ as an entity with a worldview (e.g. a person) and each simplex as the common worldview shared by its vertices. To model this, we will assign to each vertex $v\in X$ an olog $\mcO(v)$, corresponding to the worldview held by that entity, and to each simplex $u\in X_n$, we assign an olog $\mcO(u)$ corresponding to a {\em common ground} worldview.\index{common ground}. Recall that $X$ is a subset of $\PP(X_0)$; it is a preorder and its elements (the simplices) are ordered by inclusion. If $u,v$ are simplices with $u\ss v$ then we want a map of ologs (i.e. a schema morphism) $\mcO(v)\to\mcO(u)$ corresponding to how any idea that is shared among the people in $v$ is shared among the people in $u$. Thus we have a functor $\mcO\taking X\to\Sch$ (where we are forgetting the distinction between ologs and databases for notational convenience).

To every simplicial complex (indeed every ordered set) one can associate a topological space; in fact we have a functor $Alx\taking\PrO\to\Top$,\index{a functor!$\PrO\to\Top$} called the \href{http://en.wikipedia.org/wiki/Alexandrov_topology}{\text Alexandrov} functor. Applying $Alx(X\op)$ we have a space which we denote by $\mcX$. One can visualize $\mcX$ as $X$, but the open sets include unions of simplices. There is a unique sheaf of categories on $\mcX$ that behaves like $X$ on simplices.

How does this work in the case of our sheaf $\mcO$ of worldviews? For simplices such as $(A)$ or $(CI)$, the sheaf returns the olog corresponding to that person or shared worldview. But for open sets like the union of $(CIJ)$ and $(IJK)$, what we get is the olog consisting of the types shared by $C, I$, and $J$ for which $I$ and $J$ affirm agreement with types shared by $I, J$, and $K$.

\begin{example}

Imagine two groups of people $G_1$ and $G_2$ each making observations about the world. Suppose that there is some overlap $H=G_1\cap G_2$. Then it may happen that there is a conversation including $G_1$ and $G_2$ and both groups are talking about something and, although using different words, $H$ says ``you guys are talking about the same things, you just use different words." In this case there is an object-assignment throughout $G_1\cup G_2$ that agrees with both those on $G_1$ and those on $G_2$.

\end{example}


\subsubsection{Time}

One can use sheaves to model objects in time; Goguen gave an approach to this in \cite{Gog}. For another approach, let $\mcC$ be a database schema. The lifespan of information about the world is generally finite; that is, what was true yesterday is not always the case today. Thus we can associate to each interval $U$ of time the information that we deem to hold throughout $U$. This is sometimes called the {\em valid time}\index{data!valid time} of the data.

If something is the case throughout $U$ and we have a subset $V\ss U$ then of course it is the case throughout $V$. And the sheaf condition holds too: if some information holds throughout $U$ and some other information holds throughout $U'$, and if these two things restrict to the same information on the overlap $U\cap V$, then they can be glued to information that holds throughout the union $U\cup V$.

So we can model information-change over time by using a sheaf of $\mcC$-sets on the topological space $\RR$. One way to think of this is simply as an instance on the schema $\mcC\times\Op(\RR)\op$. The sheaf condition is just an added property that our instances have to obey.

\begin{example}

Consider a hospital in which babies are born. In our scenario, mothers enter the hospital, babies are born, mothers and babies leave the hospital. Let $\mcC$ be the schema 
$$\fbox{\xymatrix{\obox{c}{.5in}{a baby}\LA{rr}{was birthed by}&\hspace{.2in}&\obox{m}{.6in}{a mother}}}$$
Consider the 8-hour intervals 
\begin{align*}
\tn{Shift}_1&:=(\tn{Jan }1 - 00:00,\tn{ Jan }1 - 08:00),\\
\tn{Shift}_2&:=(\tn{Jan }1 - 04:00,\tn{ Jan }1 - 12:00),\\
\tn{Shift}_3&:=(\tn{Jan }1 - 8:00,\tn{ Jan }1 - 16:00).
\end{align*}
The nurses take shifts of 8 hours, overlapping with their predecessors by 4 hours, and they record in the database only patients that were there throughout their shift or throughout any overlapping shift. A mother might be in the hospital throughout shift 1, arriving before the new year. A baby is born at 05:00 on Jan 1, and thus does not make it into the $\tn{Shift}_1$-table, but does make it into the $(\tn{Shift}_1\cap\tn{Shift}_2)$-table. The two are there until 17:00 on Jan 1, and so they are recorded in the $\tn{Shift}_2$ and $\tn{Shift}_3$ tables. 

\end{example}

Whether or not this implementation of the sheaf semantics is most useful in practice is certainly debatable. But something like this could easily be useful as a semantics, i.e. a way of thinking about, the temporal nature of data.


\section{Monads}\label{sec:monads}\index{monad}

Monads would probably not have been invented without category theory, but they have been quite useful in formalizing algebra, calculating invariants of topological spaces, and imbedding non-functional operations into functional programming languages. We will mainly discuss monads in terms of how they can help us make modeling contexts explicit, and in so doing allow us to simplify the language we use in the model.

Much of the following material on monads is taken from \cite{Sp3}.


\subsection{Monads formalize context}\index{monad!formalizing context}\index{context}

Monads can formalize assumptions about the way one will do business throughout a domain. For example, suppose that we want to consider functions that do not have to return a value for all inputs. Such {\em partial functions}\index{partial functions} can be composed. Indeed, given a partial function $f\taking A\to B$ and a partial function $g\taking B\to C$, one gets a partial function $g\circ f\taking A\to C$ in an obvious way.

Here we are drawing arrows as though we are talking about functions, but there is an implicit context in which we are actually talking about partial functions. Monads allow us to write things in the ``functional" way while holding the underlying context. What makes them useful is that the notion of {\em context} we are using here is made formal.

\begin{example}[Partial functions]\label{ex:partial function monad}\index{a monad!partial functions}\index{partial function}\index{a monad!maybe}

Partial functions can be modeled by ordinary functions, if we add a special ``no answer" element to the codomain. That is, the set of partial functions $A\to B$ is in one-to-one correspondence with the set of ordinary functions $A\to B\sqcup\singleton$. For example, suppose we want to model the partial function $f(x):=\frac{1}{x^2-1}\taking\RR\to\RR$ in this way, we would use the function 
$$f(x):=\begin{cases}
\frac{1}{x^2-1}&\tn{if } x\neq -1 \tn{ and } x\neq 1,\\
\smiley&\tn{if } x=-1,\\
\smiley&\tn{if } x= 1.
\end{cases}
$$
An ordinary function $f\taking A\to B$ can be considered a partial function because we can compose with the inclusion 
\begin{align}\label{dia:first eta}
B\to B\sqcup\singleton
\end{align}

But how do we compose two partial functions written in this way? Suppose $f\taking A\to B\sqcup\singleton$ and $g\taking B\to C\sqcup\singleton$ are functions. First form a new function 
\begin{align}\label{dia:first monad}
g':=g\sqcup\singleton\taking B\sqcup\singleton\to C\sqcup\singleton\sqcup\singleton
\end{align}
then compose to get $(g'\circ f)\taking A\to C\sqcup\singleton\sqcup\singleton$, and finally send both $\smiley$'s to the same element by composing with 
\begin{align}\label{dia:first mu}
C\sqcup\singleton\sqcup\singleton\to C\sqcup\singleton.
\end{align}

What does this mean? Every element $a\in A$ is sent by $f$ to either an element $b\in B$ or ``no answer". If it has an answer $f(a)\in B$, this is either sent by $g$ to an element $g(f(a))\in C$ or to ``no answer". We get a partial function $A\to C$ by sending $a$ to $g(f(a))$ if possible or to ``no answer" if it gets stopped along the way.

This monad is sometimes called the {\em maybe monad} in computer science, because a partial function $f\taking A\to B$ takes every element of $A$ and either outputs just an element of $B$ or outputs nothing; more succinctly, it outputs a ``maybe $B$".

\end{example}

\begin{application}\label{app:experimenter matters}

\href{http://en.wikipedia.org/wiki/Observer-expectancy_effect}{\text Experiments are supposed to be performed objectively}, but suppose we imagine that changing the person who performs the experiment, say in psychology, may change the outcome. Let $A$ be the set of experimenters, let $X$ be the parameter space for the experimental variables (e.g. $X=\tn{Age}\times\tn{Income}$) and let $Y$ be the observation space (e.g. $Y=\tn{propensity for violence}$). Then whereas we want to think of such an experiment as telling us about a function $f\taking X\to Y$, we may want to make some of the context explicit by including information about who performed the experiment. That is, we are really finding a function $f\taking X\times A\to Y$. 

However, it may be the case that even ascertaining someones age or income, which is done by asking that person, is subject to who in $A$ is doing the asking, and so we again want to consider the experimenter as part of the equation. In this case, we can use a monad to hide the fact that everything in sight is assumed to be influenced by $A$. In other words, we want to announce once and for all our modeling context---that every observable is possibly influenced by the observer---so that it can recede into the background.

We will return to this in Examples \ref{ex:experimenter matters 2} and \ref{ex:experimenter matters 3}.

\end{application}


\subsection{Definition and examples}

What aspects of Example \ref{ex:partial function monad} are really about monads, and what aspects are just about partial functions in particular? It is a functor and a pair of natural transformations that showed up in (\ref{dia:first monad}), (\ref{dia:first eta}), and (\ref{dia:first mu}). In this section we will give the definition and a few examples. We will return to our story about how monads formalize context in Section \ref{sec:kleisli}.

\begin{definition}[Monad]\label{def:monad}\index{monad}\index{monad!on $\Set$}

A {\em monad on $\Set$} is defined as follows: One announces some constituents (A. functor, B. unit map, C. multiplication map) and asserts that they conform to some laws (1. unit laws, 2. associativity law). Specifically, one announces
\begin{enumerate}[A.]
\item a functor $T\taking\Set\to\Set$,
\item a natural transformation $\eta\taking\id_{\Set}\to T$, and 
\item a natural transformation $\mu\taking T\circ T\to T$
\end{enumerate}
We sometimes refer to the functor $T$ as though it were the whole monad; we call $\eta$ the {\em unit map} and we call $\mu$ the {\em multiplication map}. One asserts that the following laws hold:
\begin{enumerate}[1.]
\item The following diagrams of functors $\Set\to\Set$ commute:
\begin{align*}
\xymatrix@=30pt{T\circ\id_\Set\ar[r]^-{\id_T\diamond\eta}\ar[dr]_{=}&T\circ T\ar[d]^\mu\\&T}\hspace{1in}
\xymatrix@=30pt{\id_\Set\circ T\ar[r]^-{\eta\diamond\id_T}\ar[dr]_=&T\circ T\ar[d]^\mu\\&T}
\end{align*}
\item The following diagram of functors $\Set\to\Set$ commutes:
\begin{align*}
\xymatrix@=30pt{T\circ T\circ T\ar[r]^-{\mu\diamond\id_T}\ar[d]_-{\id_T\diamond\mu}&T\circ T\ar[d]^\mu\\T\circ T\ar[r]_\mu&T}\end{align*}
\end{enumerate}

\end{definition}

\begin{example}[List monad]\label{ex:monad}\index{a monad!List}

We now go through Definition \ref{def:monad} using what is called the $\List$ monad.\index{list} The first step is to give a functor $\List\taking\Set\to\Set$, which we did in Example \ref{ex:free monoid}. Recall that if $X=\{p,q,r\}$ then $\List(X)$ includes the empty list $[\;]$, singleton lists, such as $[p]$, and any other list of elements in $X$, such as $[p,p,r,q,p]$. Given a function $f\taking X\to Y$, one obtains a function $\List(f)\taking\List(X)\to\List(Y)$ by entry-wise application of $f$.

As a monad, the functor $\List$ comes with two natural transformations, a unit map $\eta$ and a multiplication map $\mu$. Given a set $X$, the unit map $\eta_X\taking X\to\List(X)$ returns singleton lists as follows
$$\xymatrix@=.5pt{
&X\ar[rr]^{\eta_X}&\hspace{1.4in}&\List(X)\\
\ar@{..}[rrr]+<.3in,0pt>&&&\\\\\\
&p\ar@{|->}[rr]&&[p]\\
&q\ar@{|->}[rr]&&[q]\\
&r\ar@{|->}[rr]&&[r]}$$
Given a set $X$, the multiplication map $\mu_X\taking\List(\List(X))\to\List(X)$ flattens lists of lists as follows.
$$\xymatrix@=.5pt{
&\List(\List(X))\ar[rr]^{\mu_X}&\hspace{.7in}&\List(X)\\
\ar@{..}[rrr]+<.6in,0pt>&&&\\\\\\
&\big[[q, p, r], [], [q, r, p, r], [r]\big]\ar@{|->}[rr]&&[q, p, r, q, r, p, r, r]}$$
The naturality of $\eta$ and $\mu$ just mean that these maps work appropriately well under term-by-term replacement by a function $f\taking X\to Y$. Finally the three monad laws from Definition \ref{def:monad} can be exemplified as follows:
$$\xymatrix@=30pt{[p, q, q]\ar@{|->}[r]^-{\id_\List\circ\eta}\ar@{=}[rd]&\big[[p], [q], [q]\big]\ar@{|->}[d]^\mu\\&[p, q, q]}\hspace{.8in}
\xymatrix@=30pt{[p, q, q]\ar@{|->}[r]^-{\eta\circ\id_\List}\ar@{=}[rd]&\big[[p, q, q]\big]\ar@{|->}[d]^\mu\\&[p,q,q]}$$
\vspace{.1in}
$$\xymatrix@=30pt{\Big[\big[[p, q], [r]\big], \big[[], [r, q, q]\big]\Big]\ar@{|->}[r]^-{\mu\circ\id_\List}\ar@{|->}[d]_{\id_\List\circ\mu}&\big[[p, q], [r], [], [r, q, q]\big]\ar@{|->}[d]^\mu\\\big[[p, q, r], [r, q, q]\big]\ar@{|->}[r]_\mu&[p, q, r, r, q,q]}$$

\end{example}

\begin{exercise}\label{exc:power set monad}
Let $\PP\taking\Set\to\Set$ be the powerset functor, so that given a function $f\taking X\to Y$ the function $\PP(f)\taking\PP(X)\to\PP(Y)$ is given by taking images.
\sexc Make sense of the following statement: ``with $\eta$ defined by singleton subsets and with $\mu$ defined by union, $\top:=(\PP,\eta,\mu)$ is a monad".
\next  With $X=\{a,b\}$, write down the function $\eta_X$ as a 2-row, 2-column table, and write down the function $\mu_X$ as a 16-row, 2-column table (you can stop after 5 rows if you fully get it).
\next Check that you believe the monad laws from Definition \ref{def:monad}.
\endsexc
\end{exercise}

\begin{example}[Partial functions as a monad]\label{ex:partial functions as monad}

Here is the monad for partial functions. The functor $T\taking\Set\to\Set$ sends a set $X$ to the set $X\sqcup\singleton$. Clearly, given a function $f\taking X\to Y$ there is an induced function $f\sqcup\singleton\taking X\sqcup\singleton\to Y\sqcup\singleton$, so this is a functor. The natural transformation $\eta\taking\id\to T$ is given on a set $X$ by the component function $$\eta_X\taking X\to X\sqcup\singleton$$ that includes $X\inj X\sqcup\singleton$. Finally, the natural transformation $\mu\taking T\circ T\to T$ is given on a set $X$ by the component function $$\mu_X\taking X\sqcup\singleton\sqcup\singleton\too X\sqcup\singleton$$ that collapses both copies of $\smiley$.

\end{example}

\begin{exercise}\label{exc:exceptions}
Let $E$ be a set, elements we will refer to as {\em exceptions}\index{exceptions}\index{a monad!exceptions}. We imagine that a function $f\taking X\to Y$ either outputs a value or one of these exceptions, which might be things like ``overflow!" or ``division by zero!", etc. Let $T\taking\Set\to\Set$ be the functor $X\mapsto X\sqcup E$. Follow Example \ref{ex:partial functions as monad} and come up with a unit map $\eta$ and a multiplication map $\mu$ for which $(T,\eta,\mu)$ is a monad.
\end{exercise}

\begin{example}\label{ex:experimenter matters 2}

Fix a set $A$. Let $T\taking\Set\to\Set$ be given by $T(X)=X^A=\Hom_\Set(A,X)$; this is a functor. For a set $X$, let $\eta_X\taking X\to T(X)$ be given by the constant function, $x\mapsto c_x\taking A\to X$ where $c_x(a)=x$ for all $a\in A$. To specify a function
$$\mu_X\taking\Hom_\Set(A,T(X))\to\Hom_\Set(A,X),$$ we curry and need a function $A\times\Hom_\Set(A,T(X))\to X$. We have an evaluation function (see Exercise \ref{exc:evaluation}) $ev\taking A\times\Hom_\Set(A,T(X))\to T(X)$, and we have an identity function $\id_A\taking A\to A$, so we have a function $(\id_A\times ev)\taking A\times\Hom_\Set(A,T(X))\too A\times T(X)$. Composing that with another evaluation function $A\times\Hom_\Set(A,X)\to X$ yields our desired $\mu_X$. Namely, for all $b\in A$ and $f\in\Hom(A,T(X))$ we have
$$\mu_X(f)(b)=f(b)(b).$$

\end{example}

\begin{remark}\index{monad!on $\Grph$}\index{a monad!$\Paths$}

Monads can be defined on categories other than $\Set$. In fact, for any category $\mcC$ one can take Definition \ref{def:monad} and replace every occurrence of $\Set$ with $\mcC$ and obtain the definition for monads on $\mcC$. We have actually seen a monad $(\Paths,\eta,\mu)$ on the category $\Grph$ of graphs before, namely in Examples \ref{ex:graph to paths} and \ref{ex:concat paths of paths}. That is, $\Paths\taking\Grph\to\Grph$, which sends a graph to its paths-graph\index{graph!paths-graph} is the functor part. The unit map $\eta$ includes a graph into its paths-graph using the observation that every arrow is a path of length 1. And the multiplication map $\mu$ concatenates paths of paths. The Kleisli category of this monad (see Definition \ref{def:kleisli}) is used, e.g. in (\ref{dia:kleisli comp in graph}) to define morphisms of database schemas.

\end{remark}


\subsection{Kleisli category of a monad}\label{sec:kleisli}

Given a monad $\top:=(T,\eta,\mu)$, we can form a new category $\Kls(\top)$.

\begin{definition}\label{def:kleisli}\index{category!Kleisli}\index{Kleisli category}\index{monad!Kleisli category of}

Let $\top=(T,\eta,\mu)$ be a monad on $\Set$. Form a new category, called the {\em Kleisli category for $\top$}, denoted $\Kls(\top)$, with sets as objects, $\Ob(\Kls(\top)):=\Ob(\Set)$, and with $$\Hom_{\Kls(\top)}(X,Y):=\Hom_\Set(X,T(Y))$$ for sets $X,Y$. The identity morphism $\id_X\taking X\to X$ in $\Kls(\top)$ is given by $\eta\taking X\to T(X)$ in $\Set$. The composition of morphisms $f\taking X\to Y$ and $g\taking Y\to Z$ in $\Kls(\top)$ is given as follows. Writing them as functions, we have $f\taking X\to T(Y)$ and $g\taking Y\to T(Z)$. The first step is to apply the functor $T$ to $g$, giving $T(g)\taking T(Y)\to T(T(Z))$. Then compose with $f$ to get $T(g)\circ f\taking X\to T(T(Z))$. Finally, compose with $\mu_Z\taking T(T(Z))\to T(Z)$ to get the required function $X\to T(Z)$. The associativity of this composition formula follows from the associativity law for monads.

\end{definition}

\begin{example}

Recall the monad $\top$ for partial functions, $T(X)=X\sqcup\singleton$, from Example \ref{ex:partial functions as monad}. The Kleisli category $\Kls(\top)$ has sets as objects, but a morphism $f\taking X\to Y$ means a function $X\to Y\sqcup\singleton$, i.e a partial function. Given another morphism $g\taking Y\to Z$, the composition formula in $\Kls(\top)$ ensures that $g\circ f\taking X\to Z$ has the appropriate behavior.

Note how this monad allows us to make explicit our assumption that all functions are partial, and then hide it away from our notation.

\end{example}

\begin{remark}\label{rem:ordinary are kleisli}

For any monad $\top=(T,\eta,\mu)$ on $\Set$, there is a functor $i\taking \Set\to\Kls(\top)$ given as follows. On objects we have $\Ob(\Kls(\top))=\Ob(\Set)$, so take $i=\id_{\Ob(\Set)}$. Given a morphism $f\taking X\to Y$ in $\Set$, we need a morphism $i(f)\taking X\to Y$ in $\Kls(\top)$, i.e. a function $i(f)\taking X\to T(Y)$. We assign $i(f)$ to be the composite $X\To{f}Y\To{\eta}T(Y)$. The functoriality of this mapping follows from the unit law for monads.

The point is that any ordinary function (morphism in $\Set$) has an interpretation as a morphism in the Kleisli category of any monad. More categorically, there is a functor $\Set\to\Kls(\top)$.

\end{remark}

\begin{example}\label{ex:experimenter matters 3}

In this example we return to the setting laid out by Application \ref{app:experimenter matters} where we had a set $A$ of experimenters and assumed that the person doing the experiment may affect the outcome. We use the monad $\top=(T,\eta,\mu)$ from Example \ref{ex:experimenter matters 2} and hope that $\Kls(\top)$ will conform to our understanding of how to manage the affect of the experimenter on data.

The objects of $\Kls(\top)$ are ordinary sets, but a map $f\taking X\to Y$ in $\Kls(\top)$ is a function $X\to Y^A$. By currying this is the same as a function $X\times A\to Y$, as desired. To compose $f$ with $g\taking Y\to Z$ in $\Kls(\top)$, we follow the formula. It turns out to be equivalent to the following. We have a function $X\times A\to Y$ and a function $Y\times A\to Z$. Modifying the first slightly, we have a function $X\times A\to Y\times A$, by identity on $A$, and we can now compose to get $X\times A\to Z$.

What does this say in terms of experimenters affecting data gathering? It says that if we work within $\Kls(\top)$ then we will be able to assume that the experimenter is being taken into account; all proposed functions $X\to Y$ are actually functions $A\times X\to Y$. The natural way to compose these experiments is that we only consider the data from one experiment to feed into another if the experimenter is the same in both experiments.
\footnote{This requirement seems a bit stringent, but it can be mitigated in a variety of ways. One such way is to notice that by Remark \ref{rem:ordinary are kleisli} that we have not added any requirement, because any old way of doing business yields a valid new way of doing business (we just say ``every experimenter would get the same result"). Another way would be to hand off the experiment results to another person, who could carry it forward (see Example \ref{ex:preorder monad}).}

\end{example}

\begin{exercise}\label{exc:kleisli powerset relations}
In Exercise \ref{exc:power set monad} we discussed the power set monad $\top=(\PP,\eta,\mu)$.
\sexc Can you find a way to relate the morphisms in $\Kls(\top)$ to relations? That is, given a morphism $f\taking A\to B$ in $\Kls(\top)$, is there a natural way to associate to it a relation $R\ss A\times B$?
\next How does the composition formula in $\Kls(\top)$ relate to the composition of relations given in Definition \ref{def:composite span}?
\footnote{Actually, Definition \ref{def:composite span} is about composing spans, but a relation $R\ss A\times B$ is a kind of span, $R\to A\times B$.}
\endsexc
\end{exercise}

\begin{exercise}
Let $\top=(\PP,\eta,\mu)$ be the power set monad. The category $\Kls(\top)$ is closed under binary products, i.e. every pair of objects $A,B\in\Ob(\Kls(\top))$ have a product in $\Kls(\top)$. What is the product of $A=\{1,2,3\}$ and $B=\{a,b\}$?
\end{exercise}

\begin{exercise}
Let $\top=(\PP,\eta,\mu)$ be the power set monad. The category $\Kls(\top)$ is closed under binary coproducts, i.e. every pair of objects $A,B\in\Ob(\Kls(\top))$ have a coproduct in $\Kls(\top)$. What is the coproduct of $A=\{1,2,3\}$ and $B=\{a,b\}$?
\end{exercise}

\begin{example}\label{ex:preorder monad}

Let $A$ be any preorder. We speak of $A$ throughout this example as though it was the linear order given by time because this is a nice case, however the mathematics works for any $A\in\Ob(\PrO)$. 

There is a monad $\top=(T,\eta,\mu)$ that captures the idea that a function $f\taking X\to Y$ occurs in the context of time in the following sense: The output of $f$ is determined not only by the element $x\in X$ on which it is applied but also by the time at which it was applied to $x$; and the output of $f$ occurs at another time, which is not before the time of input.

The functor part of the monad is given on $X\in\Ob(\Set)$ by
$$T(X)=\{p\taking A\to A\times X\|\tn{ if } p(a)=(a',x)\tn{ then } a'\geq a\}.$$
The unit $\eta_X\taking X\to T(X)$ sends $x$ to the function $a\mapsto (a,x)$. The multiplication map $\mu_X\taking T(T(X))\to T(X)$ is roughly described as follows. If for every $a\in A$ you have a later element $a'\geq a$ and a function $p\taking A\to A\times X$ that takes elements of $A$ to later elements of $A$ and values of $X$, then $p(a')$ is a still later element of $A$ and a value of $X$, as desired.

Morphisms in the Kleisli category $\Kls(\top)$ can be curried to be functions $f\taking A\times X\to A\times Y$ such that if $f(a,x)=(a',y)$ then $a'\geq a$. 

\end{example}

\begin{remark}\label{rem:state monad}

One of the most important monads in computer science is the so-called {\em state monad}. It is used when one wants to allow a program to mutate state variables (e.g. in the program 
\begin{quote}if $x>4$ then $x:=x+1$ else Print ``done")\end{quote}
$x$ is a state variable. The state monad is a special case of the monad discussed in Example \ref{ex:preorder monad}. Given any set $A$, the usual {\em state monad of type $A$} is obtained by giving $A$ the indiscrete preorder (see Example \ref{ex:discrete and indiscrete}). More explicitly it is a monad with functor part $$X\mapsto (A\times X)^X,$$ and it will be briefly discussed in Example \ref{ex:currying gives state}.

\end{remark}

\begin{example}\label{ex:scientific method}

Here we reconsider the image from the front cover of this book, reproduced here.
\begin{center}
\includegraphics[width=.8\textwidth]{ScientificMethod}
\end{center}

It looks like an olog, and all ologs are database schemas (see Section \ref{sec:olog as db schema}). But how is ``analyzed by a person yields" a function from observations to hypotheses? The very name belies the fact that it is an invalid aspect in the sense of Section \ref{sec:invalid aspect}, because given an observation there may be more than one hypothesis yielded, corresponding to which person is doing the observing. In fact, all of the arrows in this diagram correspond to some hidden context involving people: the prediction is dependent on who analyzes the hypothesis, the specification of an experiment is dependent on who is motivated to specify it, and experiments may result in different observations by different observers. 

Without monads, the model of science proposed by this olog would be difficult to believe in. But by choosing a monad we can make explicit (and then hide from discourse) our implicit assumption that ``of course this is all dependent on which human is doing the science". The choice of monad is an additional modeling choice. Do we want to incorporate the partial order of time? Do we want the scientist to be modified by each function (i.e. the person is changed when analyzing an observation to yield a hypothesis)? These are all interesting possibilities. 

One reasonable choice would be to use the state monad of type $A$, where $A$ is the set of scientific models. This implies the following context: every morphism $f\taking X\to Y$ in the Kleisli category of this monad is really a morphism $f\taking X\times A\to Y\times A$; while ostensibly giving a map from $X$ to $Y$, it is influenced by the scientific model under which it is performed, and its outcome yields a new scientific model. 

Reading the olog in this context might look like this:

\begin{quote}
A hypothesis (in the presence of a scientific model) analyzed by a person produces a prediction (in the presence of a scientific model), which motivates the specification of an experiment (in the presence of a scientific model), which when executed results in an observation (in the presence of a scientific model), which analyzed by a person yields a hypothesis (in the presence of a scientific model).
\end{quote}

The parenthetical statements can be removed if we assume them to always be around, which can be done using the monad above.

\end{example}


\subsubsection{Relaxing functionality constraint for ologs}\label{sec:relaxing ologs}

In Section \ref{sec:aspects} we said that every arrow in an olog has to be English-readable as a sentence, and it has to correspond to a function. For example, the arrow 
\begin{align}\label{dia:non-functional but english}
\xymatrix{\fbox{a person}\LA{r}{has}&\fbox{a child}}
\end{align}
comprises an readable sentence, but does not correspond to a function because a person may have no children or more than one child. 
We'll call olog in which every arrow corresponds to a function (the only option proposed so far in the book) a {\em functional olog}. Requiring that ologs be functional as we have been doing, comes with advantages and disadvantages. The main advantage is that creating a functional olog requires more conceptual clarity about the situation, and this has benefits for the olog-creator as well as for anyone to whom he or she tries to explain the situation. The main disadvantage is that creating a functional olog takes more time, and the olog takes up more space on the page.

In the context of the power set monad (see Exercise \ref{exc:power set monad}), a morphism $f\taking X\to Y$ between sets $X$ and $Y$ becomes a binary relation on $X$ and $Y$, rather than a function, as seen in Exercise \ref{exc:kleisli powerset relations}. So in that context, the arrow in (\ref{dia:non-functional but english}) becomes valid. An olog in which arrows correspond to mere binary relations rather than functions might be called a {\em relational olog}.\index{olog!relational}


\subsection{Monads in databases}\label{sec:monads in db}\index{database!Kleisli}\index{instance!Kleisli}

In this section we discuss how to record data in the presence of a monad. The idea is quite simple. Given a schema (category) $\mcC$, an ordinary instance is a functor $I\taking\mcC\to\Set$. But if $\top=(T,\eta,\mu)$ is a monad, then a {\em Kleisli $\top$-instance on $\mcC$} is a functor $J\taking\mcC\to\Kls(\top)$. Such a functor associates to every object $c\in\Ob(\mcC)$ a set $J(c)$, and to every arrow $f\taking c\to c'$ in $\mcC$ a morphism $J(f)\taking J(c)\to J(c')$ in $\Kls(\top)$. How does this look in terms of tables?

Recall that to represent an ordinary database instance $I\taking\mcC\to\Set$, we use a tabular format in which every object $c\in\Ob(\mcC)$ is displayed as a table including one ID column and an additional column for every arrow emanating from $c$. In the ID column of table $c$ were elements of the set $I(c)$ and in the column assigned to some arrow $f\taking c\to c'$ the cells were elements of the set $I(c')$. 

To represent a {\em Kleisli} database instance $J\taking\mcC\to\Kls{\top}$ is similar; we again use a tabular format in which every object $c\in\Ob(\mcC)$ is displayed as a table including one ID column and an additional column for every arrow emanating from $c$. In the ID column of table $c$ are again elements of the set $J(c)$; however in the column assigned to some arrow $f\taking c\to c'$ are not elements of $J(c')$ but $T$-values in $J(c')$, i.e. elements of $T(J(c'))$. 

\begin{example}

Let $\top=(T,\eta,\mu)$ be the monad for partial functions, as discussed in Example \ref{ex:partial function monad}. Given any schema $\mcC$, we can represent a Kleisli $\top$-instance $I\taking\mcC\to\Kls(\top)$ in tabular format. To every object $c\in\Ob(\mcC)$ we'll have a set $I(c)$ of rows, and given a column $c\to c'$ every row will produce either a value in $I(c')$ or fail to produce a value; this is the essence of partial functions. We might denote the absence of a value using $\smiley$.

Consider the schema indexing graphs 
$$\mcC:=\fbox{\xymatrix{\LTO{Arrow}\ar@<.5ex>[r]^{src}\ar@<-.5ex>[r]_{tgt}&\LTO{Vertex}}}$$
As we discussed in Section \ref{sec:graphs as functors}, an ordinary instance on $\mcC$ represents a graph. 
\begin{align*}
I:=\parbox{2in}{\fbox{\xymatrix{\bullet^v\ar[r]^f&\bullet^w\ar@/_1pc/[r]_h\ar@/^1pc/[r]^g&\bullet^x}}}
\hspace{.5in}
\begin{array}{| l || l | l |}\bhline
\multicolumn{3}{|c|}{{\tt Arrow}\;\; (I)}\\\bhline
{\bf ID}&{\bf src}&{\bf tgt}\\\bbhline
f&v&w\\\hline
g&w&x\\\hline
h&w&x\\\bhline
\end{array}
\hspace{.5in}
\begin{array}{| l ||}\bhline
\multicolumn{1}{|c|}{{\tt Vertex}\;\; (I)}\\\bhline
{\bf ID}\\\bbhline
v\\\hline
w\\\hline
x\\\bhline
\end{array}
\end{align*}
A Kleisli $\top$-instance on $\mcC$ represents graphs in which edges can fail to have a source vertex, fail to have a target vertex, or both. 
\begin{align*}
J:=\parbox{2in}{\fbox{\xymatrix{\bullet^v\ar[d]_i\ar[r]^f&\bullet^w\ar@/_1pc/[r]_h\ar@/^1pc/[r]^g&\bullet^x\\&\ar[r]_j&}}}
\hspace{.5in}
\begin{array}{| l || l | l |}\bhline
\multicolumn{3}{|c|}{{\tt Arrow}\;\; (J)}\\\bhline
{\bf ID}&{\bf src}&{\bf tgt}\\\bbhline
f&v&w\\\hline
g&w&x\\\hline
h&w&x\\\hline
i&v&\smiley\\\hline
j&\smiley&\smiley\\\bhline
\end{array}
\hspace{.5in}
\begin{array}{| l ||}\bhline
\multicolumn{1}{|c|}{{\tt Vertex}\;\; (J)}\\\bhline
{\bf ID}\\\bbhline
v\\\hline
w\\\hline
x\\\bhline
\end{array}
\end{align*}
The context of these tables is that of partial functions, so we do not need a reference for $\smiley$ in the vertex table. Mathematically, the morphism $J(src)\taking J({\tt Arrow})\to J({\tt Vertex})$ needs to be a function $J({\tt Arrow})\to J({\tt Vertex})\sqcup\singleton$, and it is.

\end{example}

\subsubsection{Probability distributions}

Let $[0,1]\ss\RR$ denote the set of real numbers between $0$ and $1$. Let $X$ be a set and $p\taking X\to[0,1]$ a function. We say that $p$ is a {\em finitary probability distribution on $X$} if there exists a finite subset $W\ss X$ such that 
\begin{align}\label{dia:sum to 1}
\sum_{w\in W}p(w)=1,
\end{align} and such that $p(x)>0$ if and only if $x\in W$. Note that $W$ is unique if it exists; we call it {\em the support of $p$} and denote it $\Supp(p)$. Note also that if $X$ is a finite set then every function $p$ satisfying (\ref{dia:sum to 1}) is a finitary probability distribution on $X$.

For any set $X$, let $\Dist(X)$ denote the set of finitary probability distributions on $X$. It is easy to check that given a function $f\taking X\to Y$ one obtains a function $\Dist(f)\taking\Dist(X)\to\Dist(Y)$ by $\Dist(f)(y)=\sum_{f(x)=y}p(x)$. Thus we can consider $\Dist\taking\Set\to\Set$ as a functor, and in fact the functor part of a monad. Its unit $\eta\taking X\to\Dist(X)$ is given by the Kronecker delta function $x\mapsto \delta_x$ where $\delta_x(x)=1$ and $\delta_x(x')=0$ for $x'\neq x$. Its multiplication $\mu\taking\Dist(\Dist(X))\to\Dist(X)$ is given by weighted sum: given a finitary probability distribution $w\taking\Dist(X)\to[0,1]$ and $x\in X$, put $\mu(w)(x)=\sum_{p\in\Supp(w)}w(p)p(x).$ 

\begin{example}[Markov chains]\label{ex:markov}\index{Markov chain}\index{a schema!$\Loop$}

Let $\Loop$ be the loop schema, $$\Loop:=\LoopSchema$$ as in Example \ref{ex:dds}. A $\Dist$-instance on $\Loop$ is equivalent to a time-homogeneous Markov chain. To be explicit, a functor $\delta\taking\Loop\to\Kls{\Dist}$ assigns to the unique object $s\in\Ob(\Loop)$ a set $S=\delta(s)$, which we call the state space, and to $f\taking s\to s$ a function $\delta(f)\taking S\to\Dist(S)$, which sends each element $x\in S$ to some probability distribution on elements of $S$. For example, the table $\delta$ on the left corresponds to the Markov matrix $M$ on the right below:
\begin{align}
\delta:=
\begin{tabular}{| l || l |}\bhline
\multicolumn{2}{| c |}{\tt{s}}\\\bhline 
{\bf ID}&{\bf f}\\\bbhline
1 & .5(1)+.5(2)\\\hline
2 & 1(2)\\\hline
3 & .7(1)+.3(3)\\\hline
4 & .4(1)+.3(2)+.3(4)\\\bhline
\end{tabular}
\hspace{.5in}
M:=\left(
\begin{array}{cccc}
0.5 & 0.5 & 0 & 0\\
0 & 1 & 0 & 0\\
0.7 & 0 & 0.3 & 0\\
0.4 & 0.3 & 0 &0.3
\end{array}
\right)
\end{align}

As one might hope, for any natural number $n\in\NN$ the map $f^n\taking S\to\Dist(S)$ corresponds to the matrix $M^n$, which sends an element in $S$ to its probable location after $n$ iterations of the transition map.

\end{example}

\begin{application}

Every star \href{http://cas.sdss.org/dr6/en/proj/basic/color/fromstars.asp}{emits a spectrum of light}, which can be understood as a distribution on the electromagnetic spectrum. Given an object $B$ on earth, different parts of $B$ will \href{http://en.wikipedia.org/wiki/Absorption_spectroscopy}{absorb radiation} at different rates. Thus $B$ produces a function from the electromagnetic spectrum to distributions of energy absorption. In the context of the probability distributions monad, we can record data on the schema 
$$\xymatrix{\LTO{star}\LA{rr}{emits}&&\LTO{wavelengths}\LA{rr}{absorbed by $B$}&\hspace{.2in}&\LTO{energies}}$$
The composition formula for Kleisli categories is the desired one: to each star we associate the weighted sum of energy absorption rates over the set of wavelengths emitted by the star. 

\end{application}


\subsection{Monads and adjunctions}

There is a strong connection between monads and adjunctions: every adjunction creates a monad, and every monad ``comes from" an adjunction. For example, the $\List$ monad (Example \ref{ex:monad}) comes from the free-forgetful adjunction between sets and monoids
$$\Adjoint{F}{\Set}{\Mon}{U}$$
(see Proposition \ref{prop:free forgetful monoid}). That is, for any set $X$, the free monoid on $X$ is $$F(X)=(\List(X),[\;],\plpl),$$ and the underlying set of that monoid is $U(F(X))=\List(X)$. Now it may seem like there was no reason to use monoids at all---the set $\List(X)$ was needed in order to discuss $F(X)$---but it will turn out that the unit $\eta$ and multiplication $\mu$ will come drop out of the adjunction too. First, we discuss the unit and counit of an adjunction.

\begin{definition}\label{def:unit and counit of adjunction}\index{adjunction!unit}\index{adjunction!counit}

Let $\mcC$ and $\mcD$ be categories, and let $L\taking\mcC\to\mcD$ and $R\taking\mcD\to\mcC$ be functors with adjunction isomorphism 
$$\alpha_{c,d}\taking\Hom_\mcD(L(c),d)\Too{\iso}\Hom_\mcC(c,R(d))$$
for any objects $c\in\Ob(\mcC)$ and $d\in\Ob(\mcD)$. The {\em unit} $\eta\taking\id_\mcC\to R\circ L$ (respectively the {\em counit} $\epsilon\taking L\circ R\to\id_\mcD$) are natural transformations defined as follows.

Given an object $c\in\Ob(\mcC)$, we apply $\alpha$ to $\id_{L(c)}\taking L(c)\to L(c)$ to get 
$$\eta_c\taking c\to R\circ L(c);$$ 
similarly given an object $d\in\Ob(\mcD)$ we apply $\alpha^\m1$ to $\id_{R(d)}\taking R(d)\to R(d)$ to get 
$$\epsilon_d\taking L\circ R(d)\to d.$$ 

\end{definition}

Below we will show how to use the unit and counit of any adjunction to make a monad. We first walk through the process in Example \ref{ex:list adjunction makes monad}.

\begin{example}\label{ex:list adjunction makes monad}

Consider the adjunction $\Adjoint{F}{\Set}{\Mon}{U}$ between sets and monoids. Let $T=U\circ F\taking\Set\to\Set$; this will be the functor part of our monad, and we have $T=\List$. Then the unit of the adjunction, $\eta\taking\id_\Set\to U\circ F$ is precisely the unit of the monad: for any set $X\in\Ob(\Set)$ the component $\eta_X\taking X\to\List(X)$ is the function that takes $x\in X$ to the singleton
list $[x]\in\List(X)$. The monad also has a multiplication map $\mu_X\taking T(T(X))\to T(X)$, which amounts to flattening a list of lists. This function comes about using the counit $\epsilon$, as follows 
$$T\circ T=U\circ F\circ U\circ F\Too{\id_U\diamond\epsilon\diamond \id_F}U\circ F=T.$$

\end{example}

The general procedure for extracting a monad from an adjunction is analogous to that shown in Example \ref{ex:list adjunction makes monad}. Given any adjunction 
$$\Adjoint{L}{\mcC}{\mcD}{R}$$
We define $\top=R\circ L\taking\mcC\to\mcC$, we define $\eta\taking\id_\mcC\to \top$ to be the unit of the adjunction (as in Definition \ref{def:unit and counit of adjunction}), and we define $\mu\taking \top\circ \top\to \top$ to be the natural transformation $\id_R\diamond\epsilon\diamond\id_L\taking RLRL\to RL,$ obtained by applying the counit $\epsilon\taking LR\to\id_\mcD$.

The above procedure produces monads on arbitrary categories $\mcC$, whereas our definition of monad (Definition \ref{def:monad}) considers only the case $\mcC=\Set$. However, this definition can be generalized to arbitrary categories $\mcC$ by simply replacing every occurrence of the string $\Set$ with the string $\mcC$.\index{monad!on arbitrary category} Similarly, our definition of Kleisli categories (Definition \ref{def:kleisli}) considers only the case $\mcC=\Set$, but again the generalization to arbitrary categories $\mcC$ is straightforward. In Proposition \ref{prop:monad to adjunction}, it may be helpful to again put $\mcC=\Set$ if one is at all disoriented.

\begin{proposition}\label{prop:monad to adjunction}

Let $\mcC$ be a category, let $(\top,\eta,\mu)$ be a monad on $\mcC$, and let $\mcK:=\Kls_\mcC(\top)$ be the Kleisli category. Then there is an adjunction 
$$\Adjoint{L}{\mcC}{\mcK}{R}$$
such that the monad $(\top,\eta,\mu)$ is obtained (up to isomorphism) by the above procedure.

\end{proposition}

\begin{proof}[Sketch of proof]

The functor $L\taking\mcC\to\mcK$ was discussed in Remark \ref{rem:ordinary are kleisli}. We define it to be identity on objects (recall that $\Ob(\mcK)=\Ob(\mcC)$). Given objects $c,c'\in\Ob(\mcC)$ the function
$$\Hom_\mcC(c,c')\Too{L}\Hom_\mcK(c,c')=\Hom_\mcC(c,\top(c'))$$
is given by $f\mapsto \eta_{c'}\circ f$. The fact that this is a functor (i.e. that it preserves composition) follows from a monad axiom.

The functor $R\taking\mcK\to\mcC$ acts on objects by sending $c\in\Ob(\mcK)=\Ob(\mcC)$ to $\top(c)\in\Ob(\mcC)$. For objects $c,c'\in\Ob(\mcK)$ the function
$$\Hom_\mcC(c,\top(c'))=\Hom_\mcK(c,c')\Too{R}\Hom_\mcC(\top(c),\top(c'))$$
is given by sending the $\mcC$-morphism $f\taking c\to \top(c')$ to the composite 
$$\top(c)\Too{\top(f)}\top\top(c')\Too{\mu_{c'}}\top(c').$$
Again, the functoriality follows from monad axioms.

We will not continue on to show that these are adjoint or that they produce the monad $(\top,\eta,\mu)$, but see \cite[VI.5.1]{Mac} for the remainder of the proof.

\end{proof}

\begin{example}\label{ex:currying gives state}

Let $A\in\Ob(\Set)$ be a set, and recall the currying adjunction 
$$\Adjoint{A\times-}{\Set}{\Set}{-^A}$$
discussed briefly in Example \ref{ex:other adjunctions}. The corresponding monad $St_A$ is typically called the {\em state monad of type $A$} in programming language theory. Given a set $X$, we have $$St_A(X)=(A\times X)^A.$$ In the Kleisli category $\Kls(St_A)$ a morphism from $X$ to $Y$ is a function of the form $X\to (A\times Y)^A$, but this can be curried to a function $A\times X\to A\times Y$. 

This monad is related to holding on to an internal state variable of type $A$. Every morphism ostensibly from $X$ to $Y$ actually takes as input not only an element of $X$ but also the current state $a\in A$, and it produces as output not only an element of $Y$ but an updated state as well.

\end{example}

Computer scientists in programming language theory have found monads to be very useful (\cite{Mog}). In much the same way, monads on $\Set$ can be useful in databases, as discussed in Section \ref{sec:monads in db}. Another, totally different way to use monads in databases is by using a mapping between schemas to produce in each one an internal model of the other. That is, for any functor $F\taking\mcC\to\mcD$, i.e. mapping of database schemas, the adjunction $(\Sigma_F,\Delta_F)$ produces a monad on $\mcC\set$, and the adjunction $(\Delta_F,\Pi_F)$ produces a monad on $\mcD\set$. If one interprets the $\List$ monad as producing in $\Set$ an internal model of the category $\Mon$ of monoids, one can similarly interpret the above monads on $\mcC\set$ and $\mcD\set$ as producing internal models of each within the other.


\section{Operads}\label{sec:operad}

In this section we briefly introduce operads, which are generalizations of categories. They often are useful for speaking about self-similarity of structure. For example, we will use them to model agents made up of smaller agents, or materials made up of smaller materials. This association with self-similarity is not really inherent in the definition, but it tends to emerge in our thinking about many operads used in practice. 

Let me begin with a warning.

\begin{warning}\index{a warning!operads vs. multicategories}

My use of the term operad is not entirely standard and conflicts with widespread usage. The more common term for what I am calling an operad is {\em symmetric colored operad} or a {\em symmetric multicategory}\index{multicategory}\index{operad!colored}. An operad classically is a multicategory with one object, and a colored operad is a multicategory. The analogy is that ``operad is to multicategory as monoid is to category". The term multicategory stems from the fact that the morphisms in a multicategory have many, rather than one, input. But there is nothing really ``multi" about the multicategory itself, only its morphisms. Probably the real reason though is that I find the term multicategory to be clunky and the term operad to be sleek, clocking in at half the syllables. I apologize if my break with standard terminology causes any confusion.  

\end{warning}

This introduction to operads is quite short. One should see \cite{Le1} for an excellent treatment.


\subsection{Definition and classical examples}

An operad is like a category in that it has objects, morphisms, and a composition formula, and it follows an identity law and an associativity law. The difference is that each morphism has many inputs (and one output).
\begin{center}
\includegraphics[height=1in]{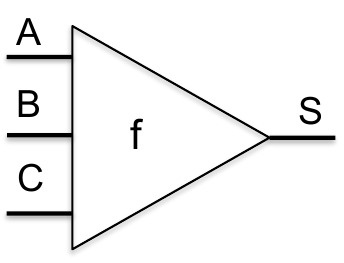}
\end{center}
The description of composition in an operad is a bit heavier than it is in a category, but the idea fairly straightforward. Here is a picture of morphisms being composed.
\begin{center}
\includegraphics[width=\textwidth]{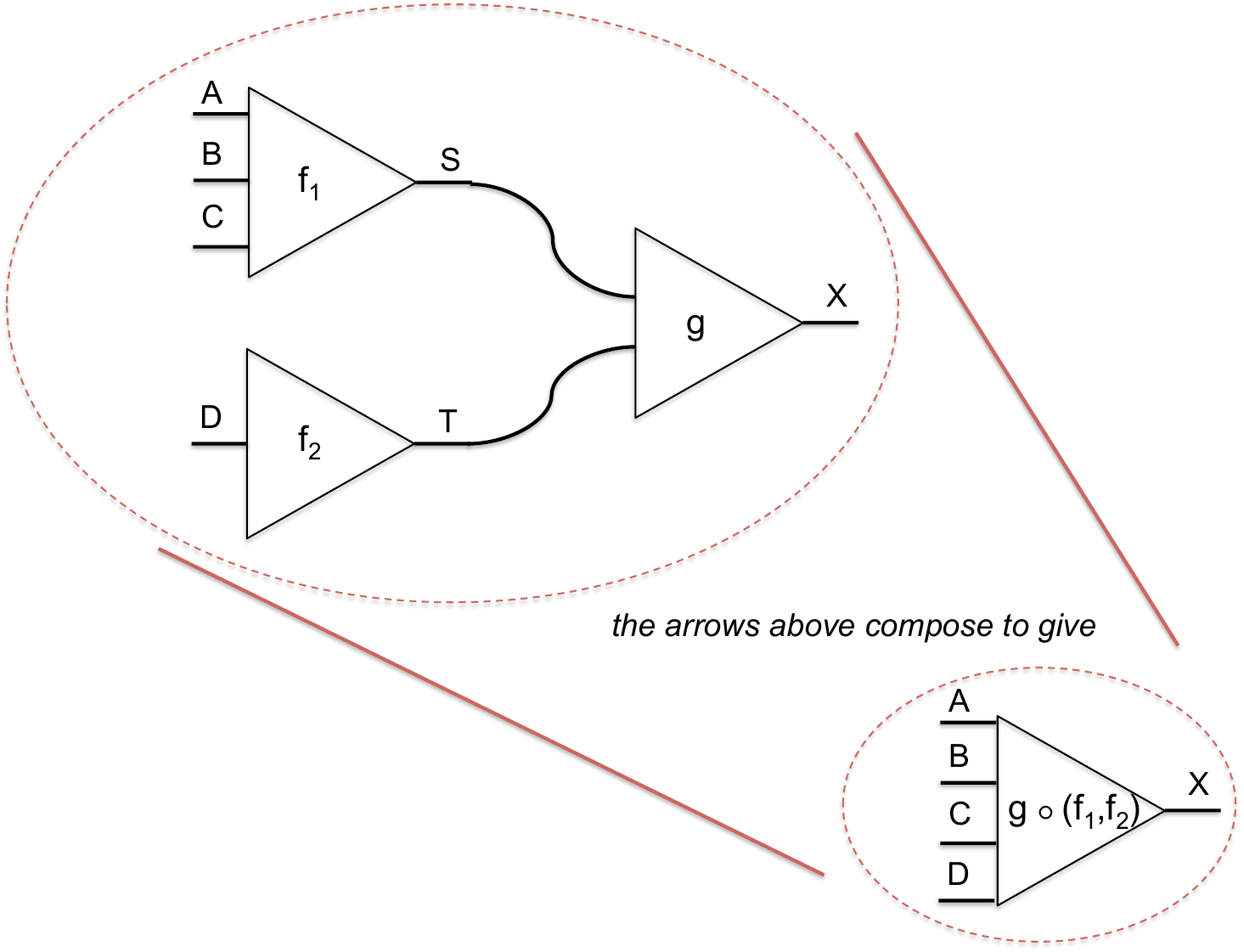}
\end{center}
Note that $S$ and $T$ disappear from the composition, but this is analogous to the way the middle object disappears from the composition of morphisms in a category
$$\color{Maroon}\dashbox{\color{black}$A\Too{f}S\Too{g} X$} \color{black}\hsp\tn{\em the arrows to the left compose to give}\hsp \color{Maroon}\dashbox{\color{black}$A\Too{g\circ f}X$}$$
Here is the definition, which we take directly from \cite{Sp4}.

\begin{definition}\label{def:operad}

An {\em operad} $\mcO$ is defined as follows: One announces some constituents (A. objects, B. morphisms, C. identities, D. compositions) and asserts that they conform to some laws (1. identity law, 2. associativity law). Specifically, 
\begin{enumerate}[\hsp A.]
\item one announces a collection $\Ob(\mcO)$, each element of which is called an {\em object} of $\mcO$.
\item for each object $y\in\Ob(\mcO)$, finite set $n\in\Ob(\Fin)$, and $n$-indexed set of objects $x\taking n\to\Ob(\mcO)$, one announces a set $\mcO_n(x;y)\in\Ob(\Set)$. Its elements are called {\em morphisms from $x$ to $y$} in $\mcO$. 
\item for every object $x\in\Ob(\mcO)$, one announces a specified morphism denoted $\id_x\in\mcO_1(x;x)$ called {\em the identity morphism on $x$}.
\item Let $s\taking m\to n$ be a morphism in $\Fin$. Let $z\in\Ob(\mcO)$ be an object, let $y\taking n\to\Ob(\mcO)$ be an $n$-indexed set of objects, and let $x\taking m\to\Ob(\mcO)$ be an $m$-indexed set of objects. For each element $i\in n$, write $m_i:=s^\m1(i)$ for the pre-image of $s$ under $i$, and write $x_i=x|_{m_i}\taking m_i\to\Ob(\mcO)$ for the restriction of $x$ to $m_i$. Then one announces a function 
\begin{align}\label{dia:composition formula}
\circ\taking\mcO_n(y;z)\times\prod_{i\in n}\mcO_{m_i}(x_i;y(i))\too\mcO_{m}(x;z),
\end{align} 
called {\em the composition formula}.
\end{enumerate}
Given an $n$-indexed set of objects $x\taking n\to\Ob(\mcO)$ and an object $y\in\Ob(\mcO)$, we sometimes abuse notation and denote the set of morphisms from $x$ to $y$ by $\mcO(x_1,\ldots,x_n;y)$.
\footnote{There are three abuses of notation when writing $\mcO(x_1,\ldots,x_n;y)$, which we will fix one by one. First, it confuses the set $n\in\Ob(\Fin)$ with its cardinality $|n|\in\NN$. But rather than writing $\mcO(x_1,\ldots,x_{|n|};y)$, it would be more consistent to write $\mcO(x(1),\ldots,x(|n|);y)$, because we have assigned subscripts another meaning in part D. But even this notation unfoundedly suggests that the set $n$ has been endowed with a linear ordering, which it has not. This may be seen as a more serious abuse, but see Remark \ref{rem:symmetry}.}
We may write $\Hom_\mcO(x_1,\ldots,x_n;y)$, in place of $\mcO(x_1,\ldots,x_n;y)$, when convenient. We can denote a morphism $\phi\in\mcO_n(x;y)$ by $\phi\taking x\to y$ or by $\phi\taking (x_1,\ldots,x_n)\to y$; we say that each $x_i$ is a {\em domain object} of $\phi$ and that $y$ is the {\em codomain object} of $\phi$. We use infix notation for the composition formula, e.g. writing $\psi\circ(\phi_1,\ldots,\phi_n)$.

One asserts that the following laws hold:
\begin{enumerate}[\hsp 1.]
\item for every $x_1,\ldots,x_n,y\in\Ob(\mcO)$ and every morphism $\phi\taking(x_1,\ldots,x_n)\to y$, we have
$$\phi\circ(\id_{x_1},\ldots,\id_{x_n})=\phi\hsp\tn{and}\hsp\id_y\circ\phi=\phi;$$
\item Let $m\To{s}n\To{t}p$ be composable morphisms in $\Fin$. Let $z\in\Ob(\mcO)$ be an object, let $y\taking p\to\Ob(\mcO)$, $x\taking n\to\Ob(\mcO)$, and $w\taking m\to\Ob(\mcO)$ respectively be a $p$-indexed, $n$-indexed, and $m$-indexed set of objects. For each $i\in p$, write $n_i=t^\m1(i)$ for the pre-image and $x_i\taking n_i\to\Ob(\mcO)$ for the restriction. Similarly, for each $k\in n$ write $m_k=s^\m1(k)$ and $w_k\taking m_k\to\Ob(\mcO)$; for each $i\in p$, write $m_{i,-}=(t\circ s)^\m1(i)$ and $w_{i,-}\taking m_{i,-}\to\Ob(\mcO)$; for each $j\in n_i$, write $m_{i,j}:=s^\m1(j)$ and $w_{i,j}\taking m_{i,j}\to\Ob(\mcO)$. Then the diagram below commutes:
$$\hspace{-1.4in}\xymatrix@=18pt{
&
{\hspace{.9in}\color{white}\prod\color{black}}
\save[]+<0cm,0cm>*\txt<30pc>{$
\mcO_p(y;z)\times\prod_{i\in p}\mcO_{n_i}(x_i;y(i))\times\prod_{i\in p,\ j\in n_i}\mcO_{m_{i,j}}(w_{i,j};x_i(j))
$}
\ar[rd]\ar[ld]
\restore\\
{\hspace{1in}\color{white}\prod\color{black}}
\save[]+<.4cm,0cm>*\txt<30pc>{$
\mcO_n(x;z)\times\prod_{k\in n}\mcO_{m_k}(w_k;x(k))
$}
\ar[dr]
\restore&&
{\hspace{1in}\color{white}\prod\color{black}}
\save[]+<-.3cm,0cm>*\txt<30pc>{$
\mcO_p(y;z)\times\prod_{i\in p}\mcO_{m_{i,-}}(w_{i,-};y(i))
$}
\ar[dl]
\restore\\
&\mcO_m(w;z)
}
$$

\end{enumerate}

\end{definition}

\begin{remark}\label{rem:symmetry}

In this remark we will discuss the abuse of notation in Definition \ref{def:operad} and how it relates to an action of a symmetric group on each morphism set in our definition of operad. We follow the notation of Definition \ref{def:operad}, especially following the use of subscripts in the composition formula.

Suppose that $\mcO$ is an operad, $z\in\Ob(\mcO)$ is an object, $y\taking n\to\Ob(\mcO)$ is an $n$-indexed set of objects, and $\phi\taking y\to z$ is a morphism. If we linearly order $n$, enabling us to write $\phi\taking (y(1),\ldots,y(|n|))\to z$, then changing the linear ordering amounts to finding an isomorphism of finite sets $\sigma\taking m\To{\iso} n$, where $|m|=|n|$. Let $x=y\circ\sigma$ and for each $i\in n$, note that $m_i=\sigma^\m1(\{i\})=\{\sigma^\m1(i)\}$, so $x_i=x|_{\sigma^\m1(i)}=y(i)$. Taking $\id_{x_i}\in\mcO_{m_i}(x_i;y(i))$ for each $i\in n$, and using the identity law, we find that the composition formula induces a bijection $\mcO_n(y;z)\To{\iso}\mcO_m(x;z)$, which we might denote by 
$$\sigma\taking\mcO(y(1),y(2),\ldots,y(n);z)\iso\mcO\big(y(\sigma(1)),y(\sigma(2)),\ldots,y(\sigma(n));z\big).$$
In other words, there is an induced group action of $\Aut(n)$ on $\mcO_n(x;z)$, where $\Aut(n)$ is the group of permutations of an $n$-element set.

Throughout this book, we will permit ourselves to abuse notation and speak of morphisms $\phi\taking (x_1,x_2,\ldots,x_n)\to y$ for a natural number $n\in\NN$, without mentioning the abuse inherent in choosing an order, so long as it is clear that permuting the order of indices would not change anything up to canonical isomorphism.

\end{remark}

\begin{example}\index{an operad!$\Sets$}

Let $\Sets$ denote the operad defined as follows. For objects we put $\Ob(\Sets)=\Ob(\Set)$. For a natural number $n\in\NN$ and sets $X_1,\ldots,X_n,Y$, put 
$$\Hom_\Sets(X_1,\ldots,X_n;Y):=\Hom_\Set(X_1\times\cdots\times X_n,Y).$$
Given functions $f_1\taking(X_{1,1}\times\cdots\times X_{1,m_1})\to Y_1$ through $f_n\taking (X_{n,1}\times\cdots\times X_{n,m_n})\to Y_n$ and a function $Y_1\times\cdots\times Y_n\to Z$, the universal property provides us a unique function of the form $(X_{1,1}\times\cdots\times X_{n,m_n})\too Z$, giving rise to our composition formula.

\end{example}

\begin{example}[Little squares operad]\label{ex:little squares}\index{an operad!little squares}

An operad commonly used in mathematics is called the {\em little $n$-cubes operad}\index{an operad!little $n$-cubes}. We'll focus on $n=2$ and talk about the little squares operad $\mcO$. Here the set of objects has only one element, which we denote by a square, $\Ob(\mcO)=\{\square\}$. For a natural number $n\in\NN$, a morphism $f\taking(\square,\square,\ldots,\square)\too\square$ is a positioning of $n$ non-overlapping squares inside of a square. Here is a picture of a morphism $(X_1,X_2,X_3)\to Y$, where $X_1=X_2=X_3=Y=\square$.
\begin{center}
\includegraphics[height=2in]{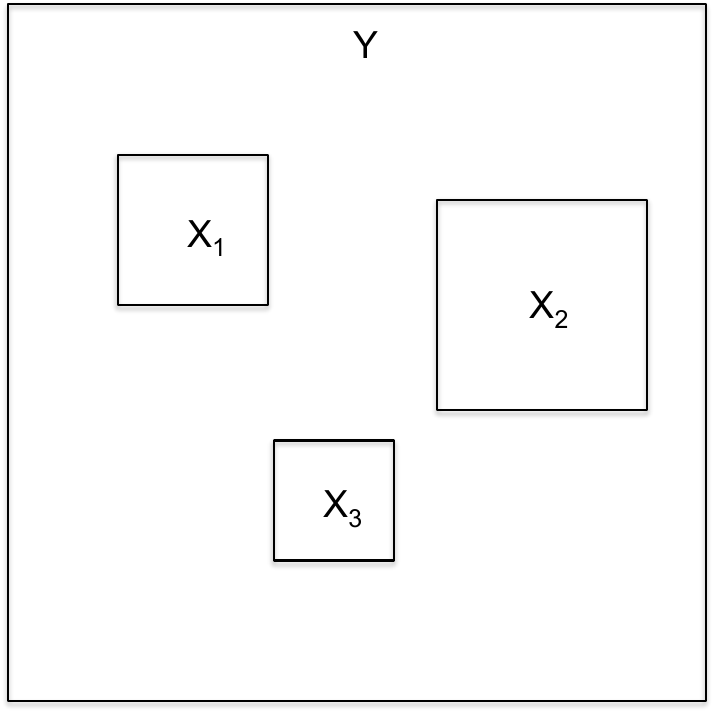}
\end{center}
The composition law says that given a positioning of small squares inside a large square, and given a positioning of tiny squares inside each of those small squares, we get a positioning of tiny squares inside a large square. A picture is shown in Figure \ref{fig:composition law for squares}.
\begin{figure}[H]
\begin{center}
\includegraphics[width=\textwidth]{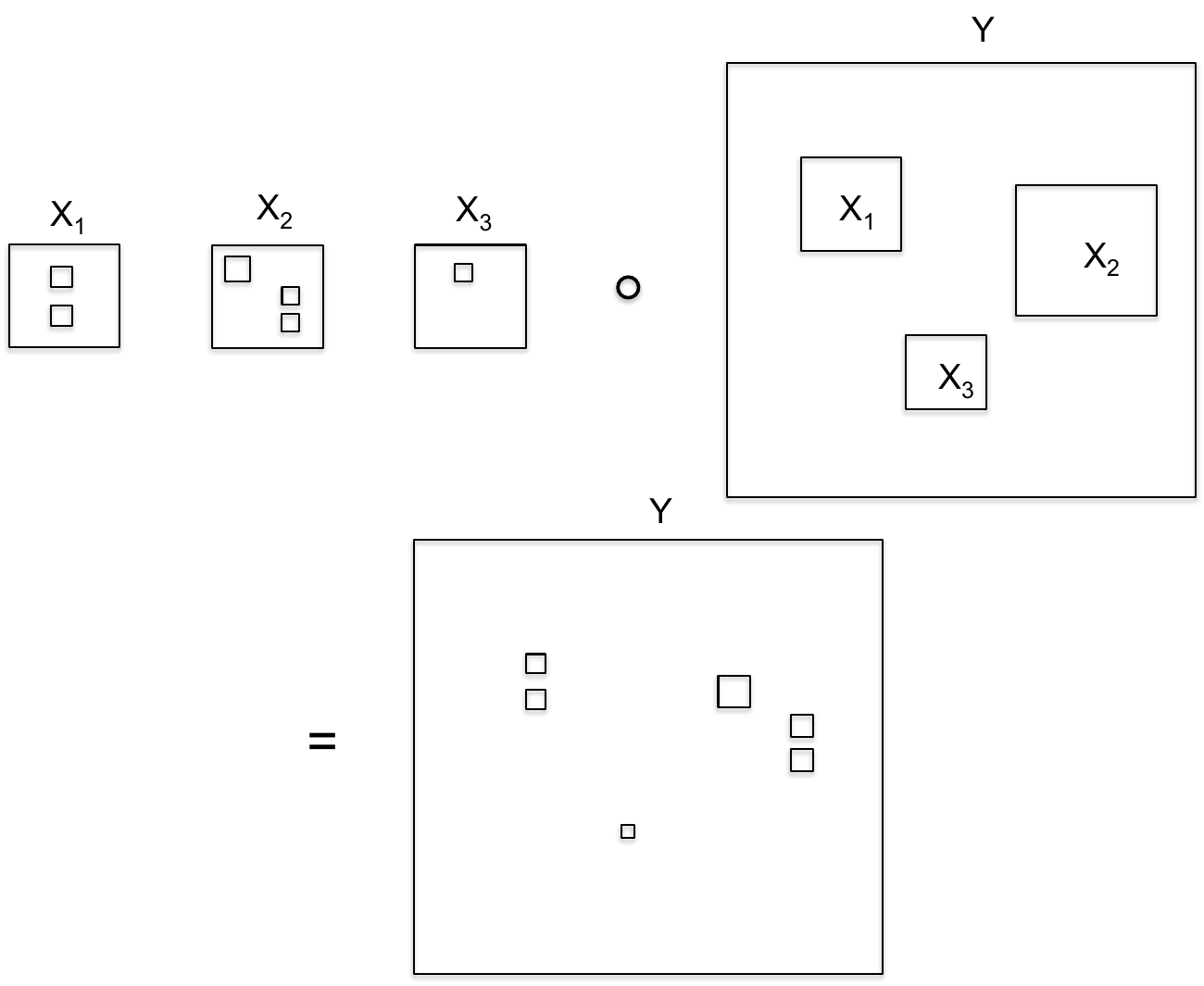}
\end{center}
\caption{Here we show a morphism $(X_1,X_2,X_3)\to Y$ and morphisms $(W_{1,1},W_{1,2})\to X_1$, $(W_{2,1},W_{2,2},W_{2,3})\to X_2$, and $(W_{3,1})\to X_3$, each of which is a positioning of squares inside a square. The composition law scales and positions the squares in the ``obvious" way.}
\label{fig:composition law for squares}
\end{figure}

Hopefully, what we meant by ``self-similarity" in the introduction to this section (see page \pageref{sec:operad}) is becoming clear.

\end{example}

\begin{exercise}\label{exc:little shapes}
Consider an operad $\mcO$ like the little squares operad from Example \ref{ex:little squares}, except with three objects: square, circle, equilateral triangle. A morphism is again a non-overlapping positioning of shapes inside of a shape. 
\sexc Draw an example of a morphism $f$ from two circles and a square to a triangle.
\next Find three other morphisms that compose into $f$, and draw the composite.
\endsexc
\end{exercise}


\subsubsection{Operads: functors and algebras}

If operads are like categories, then we can define things like functors and call them {\em operad functors}. Before giving the definition, we give a warning.

\begin{warning}\index{a warning!operad functors}

What we call operad functors in Definition \ref{def:operad morphism} are usually (if not always) called {\em operad morphisms}. We thought that the terminology clash between morphisms {\em of} operads and morphisms {\em in} an operad was too confusing. It is similar to what would occur in regular category theory (e.g. Chapter \ref{chap:categories}) if we replaced the term ``functor" with the term ``category morphism". 

\end{warning}

\begin{definition}\label{def:operad morphism}\index{operad!morphism of}

Let $\mcO$ and $\mcO'$ be operads. An {\em operad functor from $\mcO$ to $\mcO'$}, denoted $F\taking\mcO\to\mcO'$ consists of some constituents (A. on-objects part, B. on-morphisms part) conforming to some laws (1. preservation of identities, 2. preservation of composition), as follows:
\begin{enumerate}[\hsp A.]
\item There is a function $\Ob(F)\taking\Ob(\mcO)\to\Ob(\mcO')$.
\item For each object $y\in\Ob(\mcO)$, finite set $n\in\Ob(\Fin)$, and $n$-indexed set of objects $x\taking n\to\Ob(\mcO)$, there is a function $$F_n\taking\mcO_n(x;y)\to\mcO'_n(Fx;Fy).$$
\end{enumerate}
As in B. above, we often denote $\Ob(F)$, and also each $F_n$, simply by $F$. The laws that govern these constituents are as follows:
\begin{enumerate}[\hsp 1.]
\item For each object $x\in\Ob(\mcO)$, the equation $F(\id_x)=\id_{Fx}$ holds.
\item Let $s\taking m\to n$ be a morphism in $\Fin$. Let $z\in\Ob(\mcO)$ be an object, let $y\taking n\to\Ob(\mcO)$ be an $n$-indexed set of objects, and let $x\taking m\to\Ob(\mcO)$ be an $m$-indexed set of objects. Then, with notation as in Definition \ref{def:operad}, the following diagram of sets commutes:
\begin{align}\label{dia:operad functor on composition}
\xymatrix{
\mcO_n(y;z)\times\prod_{i\in n}\mcO_{m_i}(x_i;y(i))\ar[r]^-F\ar[d]_\circ&
\mcO'_n(Fy;Fz)\times\prod_{i\in n}\mcO'_{m_i}(Fx_i;Fy(i))\ar[d]^\circ\\
\mcO_m(x;z)\ar[r]_F&\mcO'_m(Fx;Fz)
}
\end{align}
\end{enumerate}

We denote the category of operads and operad functors by $\Oprd$.

\end{definition}

\begin{exercise}
Let $\mcO$ denote the little squares operad from Example \ref{ex:little squares} and let $\mcO'$ denote the operad you constructed in Exercise \ref{exc:little shapes}.
\sexc Can you come up with an operad functor $\mcO\to\mcO'$?
\next Is it possible to find an operad functor $\mcO'\to\mcO$? 
\endsexc
\end{exercise}

\begin{definition}[Operad algebra]\label{def:operad algebra}\index{algebra!operad}\index{operad!algebra of}

Let $\mcO$ be an operad. An {\em algebra on $\mcO$} is an operad functor $A\taking\mcO\to\Sets$.

\end{definition}

\begin{remark}

Every category can be construed as an operad (yes, there is a functor $\Cat\to\Oprd$), by simply not including non-unary morphisms. That is, given a category $\mcC$, one makes an operad $\mcO$ with $\Ob(\mcO):=\Ob(\mcC)$ and with 
$$
\Hom_\mcO(x_1,\ldots,x_n;y)=
\begin{cases}
\Hom_\mcC(x_1,y)&\tn{ if }n=1;\\
\emptyset&\tn{ if }n\neq 1
\end{cases}
$$
Just like a schema is a category presentation, it is possible to discuss operad presentations by generators and relations. Under this analogy, an algebra on an operad corresponds to an instance on a schema.

\end{remark}


\subsection{Applications of operads and their algebras}

Hierarchical structures may be well-modeled by operads. Describing such structures using operads and their algebras allows one to make appropriate distinctions between different types of thinking. For example, the allowable formations are encoded in the operad, whereas the elements that will fit into those formations are encoded in the algebra. Morphisms of algebras are high-level understandings of how elements of very different types (such as materials vs. numbers) can occupy the same place in the structure and be compared. We will give examples below.

\begin{application}

\href{http://en.wikipedia.org/wiki/Composite_material}{\text Every material is composed of constituent materials}, arranged in certain patterns. (In case the material is ``pure", we consider the material to consist of itself as the sole constituent.) Each of these constituent materials each is itself an arrangement of constituent materials. Thus we see a kind of self-similarity which we can model with operads. 

\begin{align}\label{dia:material comp}
\includegraphics[width=\textwidth]{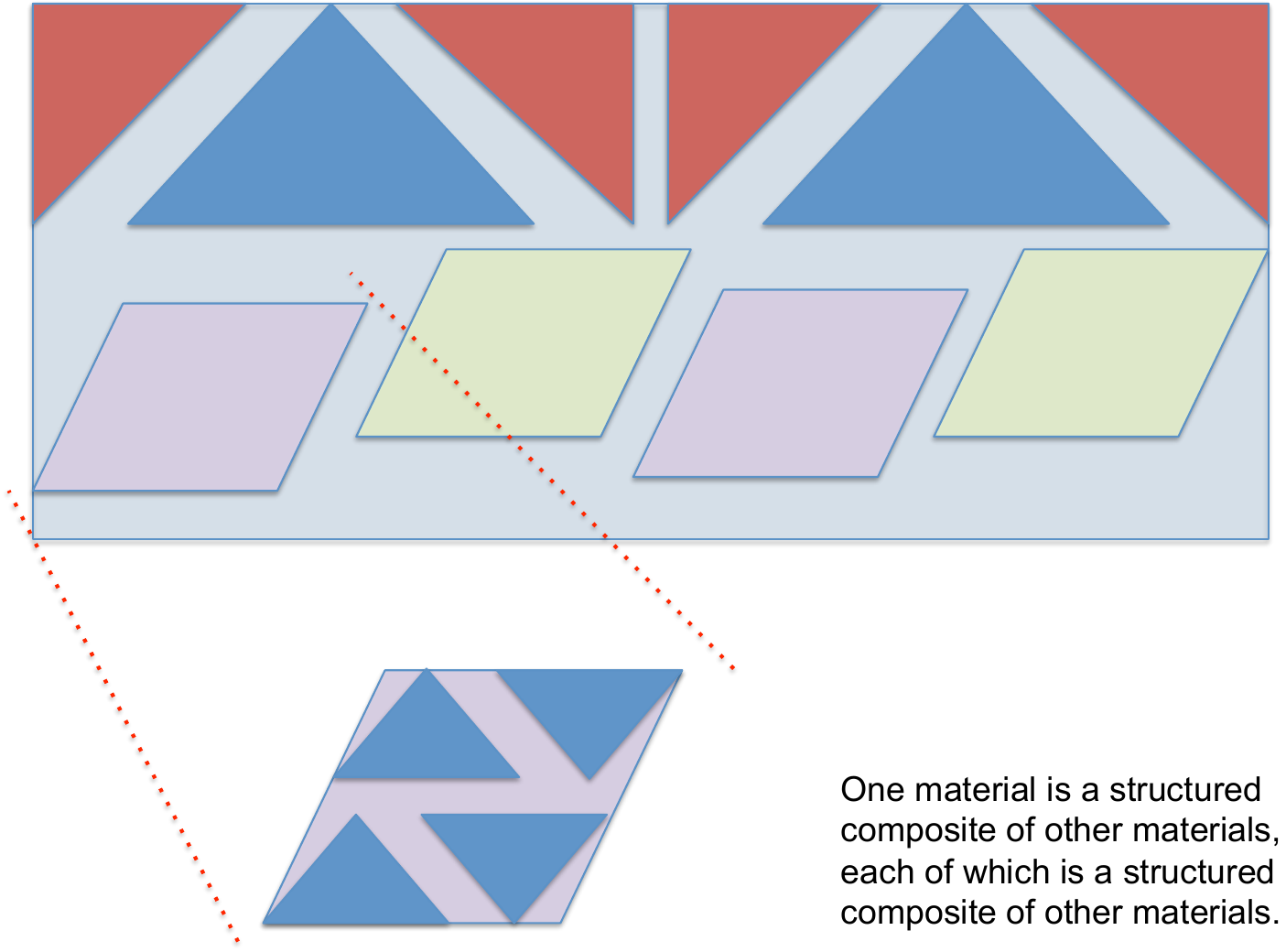}
\end{align}

For example, a tendon is made of collagen fibers that are assembled in series and then in parallel, in a specific way. Each collagen fibre is made of collagen fibrils that are again assembled in series and then in parallel, with slightly different specifications. We can continue down, perhaps indefinitely, though our resolution fails at some point. A collagen fibril is made up of tropocollagen collagen molecules, which are twisted ropes of collagen molecules, etc.\footnote{Thanks to Professor Sandra Shefelbine for explaining the hierarchical nature of collagen to me. Any errors are my own.}

Here is how operads might be employed. We want the same operad to model both actual materials, theoretical materials, and functional properties; that is we want more than one algebra on the same operad. 

The operad $\mcO$ should abstractly model the structure, but not the substance being structured. Imagine that each of the shapes (including the background ``shape") in Diagram (\ref{dia:material comp}) is a place-holder, saying something like ``{\em your material here}". Each morphism (that's what (\ref{dia:material comp}) is a picture of) represents a construction of a material out of parts. In our picture, it appears we are only concerned with the spacial arrangements, but there is far more flexibility than that. Whether we want to allow for additional details beyond spacial arrangements is the kinds of choice we make in a meeting called ``what operad should we use?" 

\end{application}

\begin{application}

Suppose we have chosen an operad $\mcO$ to model the structure of materials. Each object of $\mcO$ might correspond to a certain quality of material, and each morphism corresponds to an arrangement of various qualities to form a new quality. An algebra $A\taking\mcO\to\Sets$ on $\mcO$ forces us to choose what substances will fill in for these qualities. For every object $x\in\Ob(\mcO)$, we want a set $A(x)$ which will be the set of materials with that quality. For every arrangement, i.e. morphism, $f\taking (x_1,\ldots,x_n)\to y$, and every choice $a_1\in A(x_1), \ldots, a_n\in A(x_n)$ of materials, we need to understand what material $a'=A(f)(a_1,\ldots,a_n)\in A(y)$ will emerge when these materials are arranged in accordance with $f$. We are really pinning ourselves down here.

But there may be more than one interesting algebra on $\mcO$. Suppose that $B\taking\mcO\to\Sets$ is an algebra of strengths rather than materials. For each object $x\in\Ob(\mcO)$, which represents some quality, we let $B(x)$ be the set of possible strengths that something of quality $x$ can have. Then for each arrangement, i.e. morphism, $f\taking (x_1,\ldots,x_n)\to y$, and every choice $b_1\in B(x_1), \ldots, b_n\in B(x_n)$ of strengths, we need to understand what strength $b'=B(f)(b_1,\ldots,b_n)\in B(y)$ will emerge when these strengths are arranged in accordance with $f$. Certainly an impressive achievement!

Finally, a morphism of algebras $S\taking A\to B$ would consist of a coherent system for assigning to each material $a\in A(X)$ of a given quality $x$ a specific strength $S(a)\in B(X)$, in such a way that morphisms behaved appropriately. In this language we have stated a very precise goal for the field of material mechanics.

\end{application}

\begin{exercise}
Consider again the little squares operad $\mcO$ from Example \ref{ex:little squares}. Suppose we wanted to use this operad to describe those \href{http://en.wikipedia.org/wiki/Photographic_mosaic}{\text photographic mosaics}. 
\sexc Come up with an algebra $P\taking\mcO\to\Sets$ that sends the square to the set of all photos that can be pasted into that square. What does $P$ do on morphisms in $\mcO$?
\next Come up with an algebra $C\taking\mcO\to\Sets$ that sends each square to the set of all colors (visible frequencies of light). In other words, $C(\square)$ is the set of colors, not the set of ways to color the square. What does $C$ do on morphisms in $\mcO$. Hint: use some kind of averaging scheme for the morphisms.
\next Guess: if someone were to appropriately define morphisms of $\mcO$-algebras (something akin to natural transformations between functors $\mcO\to\Sets$), do you think there would some a morphism of algebras $P\to C$?
\endsexc
\end{exercise}


\subsubsection{Wiring diagrams}

\begin{example}\label{ex:operad of relations}

Here we describe an {\em operad of relations}, which we will denote by $\mcR$. The objects are sets, $\Ob(\mcR)=\Ob(\Set)$. A morphism $f\taking (x_1,x_2,\ldots,x_n)\too x'$ in $\mcR$ is a diagram in $\Set$ of the form 
\begin{align}\label{dia:operad of relations}
\xymatrix@=15pt{&&&R\ar[llldd]_(.7){f_1}\ar[lldd]^{f_2}\ar@{}[ldd]^(.6)\cdots\ar[dd]^(.6){f_n}\ar[rrdd]^(.7){f'}\\\\x_1&x_2&\cdots&x_n&&x'}
\end{align} 
such that the induced function $R\too(x_1\times x_2\times\cdots\times x_n\times x')$ is an injection.

We use a composition formula similar to that in Definition \ref{def:composite span}. Namely, we form a fiber product
$$\xymatrix{&&FP\ar[rd]\ar[ld]\\&\prod_{i\in\ul{n}}R_i\ar[ld]\ar[rd]&&S\ar[ld]\ar[rd]\\\prod_{i\in\ul{n}}\prod_{j\in\ul{m_i}}x_{i,j}&&\prod_{i\in\ul{n}}y_i&&z}$$
One can show that the induces function $FP\too\left(\prod_{i\in\ul{n}}\prod_{j\in\ul{m_i}}x_i\right)\times y$ is an injection, so we have a valid composition formula. Finally, the associativity and identity laws hold.
\footnote{Technically we need to use isomorphism classes of cone points, but we don't worry about this here.}

\end{example}

\begin{application}\label{app:entity by experience}
Suppose we are trying to model \href{http://en.wikipedia.org/wiki/Life}{\text life} in the following way. We define an entity as a set of phenomena, but in order to use colloquial language we say the entity {\em is able to experience} that set of phenomena. We also want to be able to put entities together to form a super-entity, so we have a notion of morphism $f\taking(e_1,\ldots,e_n)\too e'$ defined as a relation as in (\ref{dia:operad of relations}). The idea is that the morphism $f$ is a way of translating between the phenomena that may be experienced by the sub-entities and the phenomena that may be experienced by the super-entity. 

The operad $\mcR$ from Example \ref{ex:operad of relations} becomes useful as a language for discussing issues in this domain.
\end{application}

\begin{example}\label{ex:algebra on operad of rels}

Let $\mcR$ be the operad of relations from Example \ref{ex:operad of relations}. Consider the algebra $S\taking\mcR\to\Sets$ given by $S(x)=\PP(x)$. Given a morphism $\prod_ix_i\from R\to y$ and subsets $x_i'\ss x_i$, we have a subset $\prod_ix_i'\ss\prod_ix_i$. We take the fiber product
$$\xymatrix@=15pt{&FP\ar[rr]\ar[ld]&&R\ar[ld]\ar[rd]\\\prod_ix_i'\ar[rr]&&\prod_ix_i&&y}$$
and the image of $FP\to y$ is a subset of $y$. 

\end{example}

\begin{application}\label{app:desire}

Following Application \ref{app:entity by experience} we can use Example \ref{ex:algebra on operad of rels} as a model of survival. Each entity survives only for a subset of the phenomena that it can experience. Under this interpretation, the algebra from Example \ref{ex:algebra on operad of rels} defines survival as the survival of all parts. That is, suppose that we understand how a super-entity is composed of sub-entities in the sense that we have a translation between the set of phenomena that may be experienced across the sub-entities and the set of phenomena that may be experienced by the super-entity. Then the super-entity will survive exactly those phenomena which translate to phenomena for which each sub-entity desires. 

Perhaps a better term than survival would be ``allowance". A bureaucracy consists of a set of smaller bureaucracies, each of which allows certain phenomena to pass; the whole bureaucracy allows something to pass if and only if, when translated to the perspective of each sub-bureaucracy, it is allowed to pass there.

\end{application}

\begin{example}

In this example we discuss wiring diagrams that look like this:\index{wiring diagram}
\begin{center}
\includegraphics[width=\textwidth]{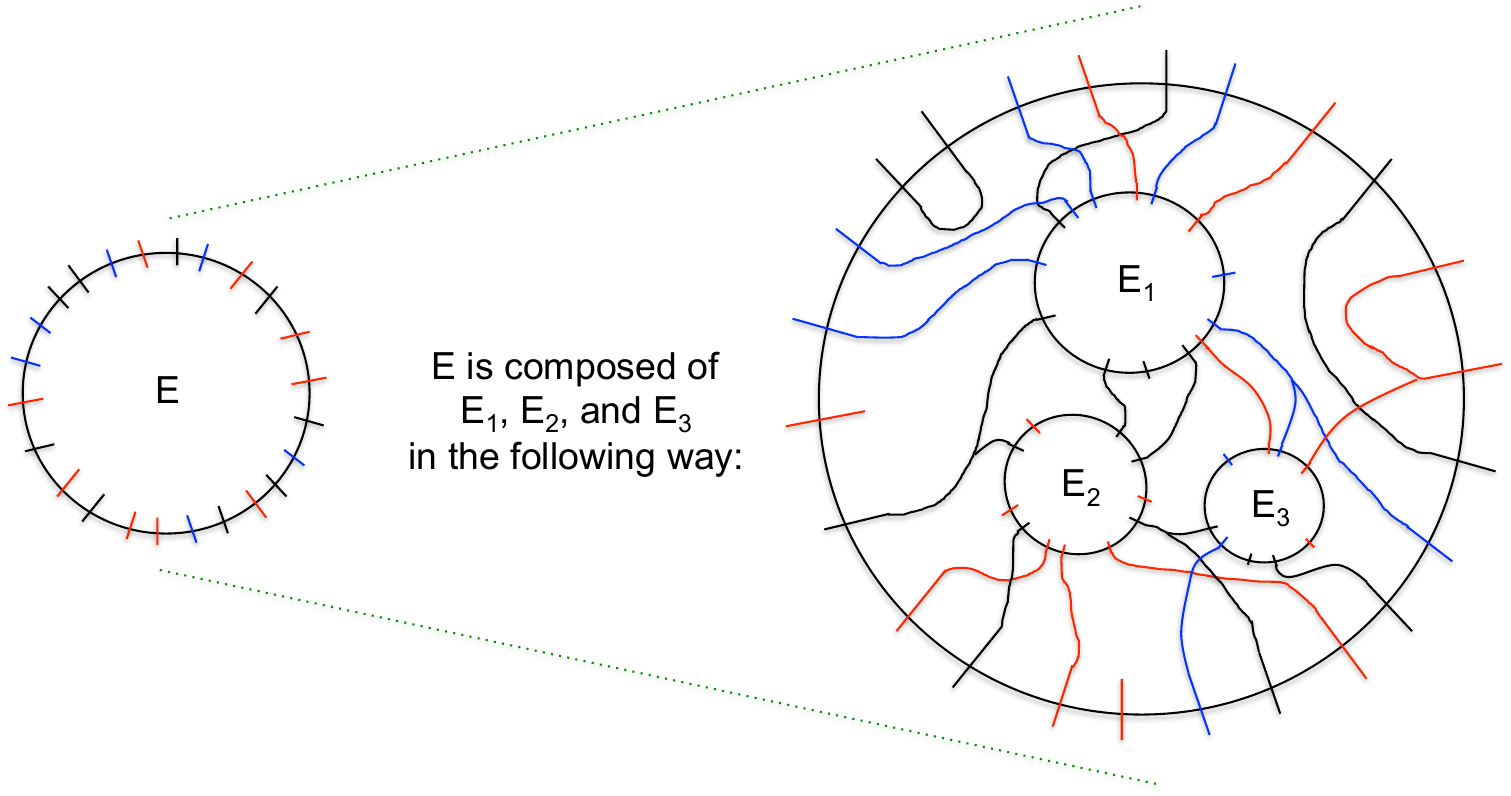}
\end{center}
The operad in question will be denoted $\mcW$; it is discussed in greater detail in \cite{Sp4}. The objects of $\mcW$ are pairs $(C,s)$ where $C$ is a finite set and $v\taking C\to\Ob(\Set)$ is a function. Think of such an object as a circle with $C$-many cables sticking out of it; each cable $c$ is assigned a set $v(c)$ corresponding to the set of values that can be carried on that cable. For example $E_2=(C,v)$ where $|C|=11$ and we consider $v$ to be specified by declaring that black wires carry $\ZZ$ and red wires carry $\{\tn{sweet, sour, salty, bitter, umami}\}$. 

The morphisms in $\mcW$ will be pictures as above, formalized as follows. Given objects $(C_1,v_1),\ldots,(C_n,v_n), (D,w)$, a morphism $F\taking((C_1,v_1),\ldots,(C_n,v_n))\too (D,w)$ is a commutative diagram of sets
\footnote{If one is concerned with cardinality issues, fix a cardinality $\kappa$ and replace $\Ob(\Set)$ everywhere with $\Ob(\Set_{<\kappa})$.} 
$$
\xymatrix{\bigsqcup_{i\in\ul{n}}C_i\ar[rd]_{\sqcup_iv_i}\ar[r]^i&G\ar[d]^x&D\ar[l]_j\ar[ld]^w\\&\Ob(\Set)}
$$
such that $i$ and $j$ are jointly surjective.

Composition of morphisms is easily understood in pictures: given wiring diagrams inside of wiring diagrams, we can throw away the intermediary circles. In terms of sets, we perform a pushout.

There is an operad functor $\mcW\to\mcS$ given by sending $(C,v)$ to $\prod_{c\in C}v(c)$. The idea is that to an entity defined as having a bunch of cables carrying variables, a phenomenon is the same thing as a choice of value on each cable. A wiring diagram translates between values experienced locally and values experienced globally. 

\end{example}

\begin{application}

In cognitive neuroscience or in industrial economics, it may be that we want to understand the behavior of an entity such as a mind, a society, or a business in terms of its structure. Knowing the connection pattern (\href{http://en.wikipedia.org/wiki/Connectome}{connectome}, \href{http://en.wikipedia.org/wiki/Supply_chain}{supply chain}) of sub-entities should help us understand how big changes are generated from small ones.

Under the functor $\mcW\to\mcS$ the algebra $\mcS\to\Sets$ from Application \ref{app:desire} becomes an algebra $\mcW\to\Sets$. To each entity we now associate some subset of the value-assignments it can carry. 
\end{application}

\begin{application}

In \cite{RS}, \href{http://dspace.mit.edu/bitstream/handle/1721.1/44215/MIT-CSAIL-TR-2009-002.pdf?sequence=1}{Radul and Sussman} discuss propagator networks. These can presumably be understood in terms of wiring diagrams and their algebra of relations.

\end{application}
 

\printindex

\bibliographystyle{amsalpha}

\end{document}